\documentclass[10pt]{amsart}
\usepackage{amsmath, amsrefs, accents, amsfonts, amscd, amssymb, amsthm,
  verbatim, psfrag, dsfont,enumerate,url, graphicx}
\usepackage{color, pdfsync, hyperref}
\usepackage{marginnote}

\newtheorem{conjecture}{Conjecture}

\newtheorem{lemma}{Lemma}[section]
\newtheorem{smalltheorem}[lemma]{Theorem}
\newtheorem{proposition}[lemma]{Proposition}
\newtheorem{theorem}{Theorem}
\newtheorem{corollary}[lemma]{Corollary}
\newtheorem{definition}[lemma]{Definition}
\newtheorem{remark}[lemma]{Remark}
\newtheorem{restatementlemma}[lemma]{Restatement of Lemma}
\newtheorem{restatementproposition}[lemma]{Restatement of Proposition}

\let\oldmarginpar\marginpar
\renewcommand\marginpar[1]{\-\oldmarginpar[\raggedright\footnotesize #1]%
{\raggedright\footnotesize #1}}

\begin{document}
\title{Feral Curves and Minimal Sets}
\author[J.W.~Fish]{Joel W.~Fish}
\thanks{The first author's research in development of this manuscript was
supported in part by the Ellentuck Fund, the Fund for Math at the
Institute for Advanced Study, and NSF-DMS Standard Research Grant Award
1610453 }
\author[H.~Hofer]{Helmut Hofer}
\address{
	Joel W.~Fish\\
	Department of Mathematics\\
        University of Massachusetts Boston
	}

\email{joel.fish@umb.edu}
\address{
	Helmut Hofer\\
        School of Mathematics
	Institute for Advanced Study
	}
\email{hofer@math.ias.edu}
\maketitle

\begin{abstract}
Here we prove that for each Hamiltonian function \(H\in
  \mathcal{C}^\infty(\mathbb{R}^4, \mathbb{R})\) defined on the standard
  symplectic \((\mathbb{R}^4, \omega_0)\), for which \(M:=H^{-1}(0)\) is a
  non-empty compact regular energy level, the Hamiltonian flow on \(M\) is
  not minimal.
That is, we prove there exists a closed invariant subset of the
  Hamiltonian flow in \(M\) that is neither \(\emptyset\) nor all of \(M\).
This answers the four dimensional case of a twenty year old question of
  Michel Herman, part of which can be regarded as a special case of the
  Gottschalk Conjecture. 

Our principal technique is the introduction and development of a
  new class of pseudoholomorphic curve in the ``symplectization''
  \(\mathbb{R} \times M\) of framed Hamiltonian manifolds \((M, \lambda,
  \omega)\).
We call these \emph{feral} curves because they are allowed to have
  infinite (so-called) Hofer energy, and hence may limit to invariant sets
  more general than the finite union of periodic orbits.
Standard pseudoholomorphic curve analysis is inapplicable without
  energy bounds, and thus much of this manuscript is devoted to establishing
  properties of feral curves, such as area and curvature estimates,
  energy thresholds, compactness, asymptotic properties, etc.
\end{abstract}

\tableofcontents
\allowdisplaybreaks
\newcounter{CurrentSection}
\newcounter{CurrentLemma}
\newcounter{CurrentTheorem}
\newcounter{CounterSectionFeralLimitCurves}
\newcounter{CounterLemmaFeralLimitCurves}
\newcounter{CounterSectionBoundedIntersections1}
\newcounter{CounterLemmaBoundedIntersections1}
\newcounter{CounterSectionAsymptoticLocalLocalAreaBound}
\newcounter{CounterTheoremAsymptoticLocalLocalAreaBound}
\newcounter{CounterSectionAsymptoticCurvatureBound}
\newcounter{CounterTheoremAsymptoticCurvatureBound}
\newcounter{CounterSectionAreaBounds}
\newcounter{CounterLemmaAreaBounds}
\newcounter{CounterSectionExtensionStructures}
\newcounter{CounterLemmaExtensionStructures}
\newcounter{CounterSectionExtensionExistence}
\newcounter{CounterLemmaExtensionExistence}
\newcounter{CounterSectionTheoremOne}
\newcounter{CounterTheoremTheoremOne}
\newcounter{CounterSectionEnergyThreshold}
\newcounter{CounterTheoremEnergyThreshold}
\newcounter{CounterSectionExistenceWorkhorse}
\newcounter{CounterTheoremExistenceWorkhorse}
\newcounter{CounterSectionAreaHomotopy}
\newcounter{CounterTheoremAreaHomotopy}
\newcounter{CounterSectionTheoremTwo}
\newcounter{CounterTheoremTheoremTwo}

\section{Introduction and Results}

Almost since their inception, pseudoholomorphic curves have been the
  common thread by which symplectic geometry, topology, and Hamiltonian
  dynamics have been intertwined.
Specifically, these curves generalize the notion of holomorphic
  curves in a complex manifold to curves in an \emph{almost complex} 
  symplectic manifold; moreover, they do so while preserving a variety of
  robust properties which detect subtle geometric aspects, dynamical
  features, and relationships between the two.
At its core, this manuscript is about the discovery of a
  new class of pseudoholomorphic curve (with reasonable properties) and
  their application toward answering a twenty year old question of Michel
  Herman \cite{Herman} raised at the 1998 ICM.
We should mention, this new class of potentially infinite energy curve
  seems to have been very difficult to predict, particularly as a natural
  extension of finite energy curves.
We elaborate further on this in Section \ref{SEC_overview_of_results}.
What follows are two separate but inextricably linked results,
  each of which is of notable interest to a separate camp of
  mathematician: the dynamicist and symplectic topologist.
We begin with an easy to state variant of our main dynamical theorem and
  provide some brief discussion of the significance of the result and its
  proof.

\setcounter{CounterSectionTheoremOne}{\value{section}}
\setcounter{CounterTheoremTheoremOne}{\value{theorem}}
\begin{theorem}(Main dynamical result)
  \label{THM_main_result}
  \hfill\\
Consider \(\mathbb{R}^4\) equipped with the standard symplectic structure
  and a Hamiltonian \(H\in \mathcal{C}^\infty(\mathbb{R}^4, \mathbb{R})\)
  for which \(M:=H^{-1}(0)\) is a non-empty compact regular energy level.
Then the Hamiltonian flow on \(M\) is not minimal.
\end{theorem}

Recall that a flow is minimal provided that every trajectory is dense; or,
  in other words, if there exist no closed invariant subsets other than the
  empty set and the total space.
Specialists in dynamical systems may regard the above result as a  
  proof of a Hamiltonian version of the Gottschalk conjecture; for
  additional details, see Section \ref{SEC_into_dyn_sys} below.
For symplectic topologists, the  dynamical result itself is perhaps less
  important than the proof, which heavily uses a new class of
  pseudoholomorphic curve, so called \emph{feral curves}.
This new class of pseudoholomorphic curves opens the door  for further
  studies of symplectic cobordisms without the usual requirement that the
  boundaries are of contact or stable Hamiltonian type.
For the current theory of pseudoholomorphic curves, such requirements have
  been technical necessities.

To sketch the proof idea of Theorem \ref{THM_main_result}, we note that
  the key technique we employ is both venerably old and radically new: Given
  our smooth hypersurface \(M:=H^{-1}(0)\subset \mathbb{R}^4\), we
  symplectically embed a neighborhood of \(M\) into \(\mathbb{C}P^2\),
  stretch the neck along this hypersurface, use Gromov's existence result
  for degree one pseudoholomorphic spheres, show these curves stretch as
  they fall into the negative symplectization end, and then establish a
  compactness result which yields a non-compact pseudoholomorphic curve in
  \(\mathbb{R}\times M\) which limits to the desired closed invariant
  subset.

The novelty here is not so much the simple geometric idea underlying the
  proof, but rather that the proof can be made to work at all.
Specifically, because the hypersurface \(H^{-1}(0)\) is neither contact
  type nor stable-Hamiltonian type, we do not have a priori Hofer energy
  bounds as we stretch the neck.
The result of our analysis is then to find a potentially \emph{infinite
  energy} pseudoholomorphic curve, which has surprisingly nice
  properties.
For example, the ends detect the desired closed invariant subset.
In either case, these curves have the interesting feature that each can
  interpolate between finite and infinite energy ends, and in families these
  curves can interpolate between finite and infinite energy curves.
Precisely because of this ability to transition between the tame (finite
  energy) and the wild (infinite energy), we have picked the name
  \emph{feral curves}.

Those familiar with pseudoholomorphic curves in symplectizations
  should readily be aware of the fact that finiteness of, and a priori
  bounds on, Hofer energy is an absolute bedrock assumption upon which a
  tremendous number of additional properties are built.
By removing this assumption, we must return to basics, and it should be no
  surprise then that this takes considerable effort.
In particular the widely used domain-centric approach of predominantly
  regarding curves as maps, must be replaced by a more target-centric
  approach which treats curves more like submanifolds.
The origin of this more target-centric approach was likely Taubes' work in
  \cite{Taubes1998} which regarded curves as integral currents, however
  the techniques therein are  too coarse for our needs here.
Instead we build on a mixture of ideas initiated in the alternate approach
  to compactness in Symplectic Field Theory\footnote{Symplectic
  Field Theory was introduced by Eliashberg, Givental, and Hofer in
  \cite{EGH}, and more recently a very nice overview and background was
  provided by Chris Wendl in \cite{WendlSFT}.} provided by Kai Cieliebak
  and Klaus Mohnke in \cite{CieM2005}, and then heavily generalized by the
  first author of this manuscript in \cite{Fish2}.

In Section \ref{SEC_context_pseudoholomorphic_curves} below, we elaborate
  on the difficulties involved with analyzing curves of \emph{infinite}
  energy, but note that the end result is the establishment of a decidedly
  novel class of pseudoholomorphic curve equipped with many properties
  which are not dissimilar from \emph{finite} energy curves, and moreover
  which strongly suggest a rich avenue of future research.  
Indeed, one natural direction would be to explore whether there exists a
  homology theory akin to ECH\footnote{For a nice introduction to Embedded
  Contact Homology, see \cite{HutchLec}, and for some nice dynamical
  applications thereof, see Daniel Cristofaro-Gardiner, Michael Hutchings,
  and Vinicius Ramos in \cite{CHR}, as well as Masayuki  Asaoka and Kei
  Irie in \cite{AI}.  },
  or SFT but which has generators which are
  dynamical structures other than (weighted) sets of periodic
  orbits.
An alternate direction would be to explore the possibility that for
  a \emph{generic} framework\footnote{Perhaps using abstract
  perturbations.}, feral curves actually have finite energy, and
  hence Symplectic Field Theory has extension to symplectic manifolds with
  generic boundary rather than contact-type or stable Hamiltonian boundary.

At present we outline the remainder of the manuscript.
First, in Section \ref{SEC_into_dyn_sys}, we provide some historical
  context for Theorem \ref{THM_main_result} from the perspective of
  dynamical systems.
In Section \ref{SEC_context_pseudoholomorphic_curves}, we elaborate on the
  historical context from the pseudoholomorphic curve perspective, and we
  highlight some of the potential difficulties that must be resolved in
  order to prove Theorem \ref{THM_main_result}.
We finish this introduction with Section \ref{SEC_overview_of_results}
  which provides an overview of the additional theorems proved in this
  manuscript.
Then, in Section \ref{SEC_background}, we provide background definitions
  and state some known results which will be used throughout later proofs.
Most of the material in this section is likely to be familiar to those
  comfortable with pseudoholomorphic curve analysis, however there are a
  number of definitions which may be novel.
In Section \ref{SEC_existence_minimal_subsets} we provide the main
  argument which establishes Theorem \ref{THM_main_result}.
This proof relies on several technical supporting results, and these are
  restated and proved in Section \ref{SEC_supporting_proofs}.

\subsection{Context: Dynamical Systems}\label{SEC_into_dyn_sys} 
In the 1950s, the following two important conjectures  about autonomous flows
  on \(S^3\) were stated:\\

  \noindent {\bf Seifert Conjecture:} \emph{Every non-singular flow
 on \(S^3\) has a periodic orbit.}\\
 
\noindent {\bf Gottschalk Conjecture:}  \emph{ $S^3$ does not support  a minimal
flow.}\\

Regarding the history of these two conjectures, we begin with the Seifert
  Conjecture.
After being posed in the early 1950s, it stood as an open problem for over
  twenty years, until 1974 when Paul Schweitzer \cite{Schw} proved the
  existence of a \(\mathcal{C}^1\) vector field on \(S^3\) with no closed
  orbits.
The existence of such a vector field then disproved the Seifert
  Conjecture, and hence Schweitzer's vector field was regarded as a
  \(\mathcal{C}^1\) counterexample.
Of course, \(\mathcal{C}^1\) vector fields are of rather low regularity,
  and hence comprise a rather broad class of vector fields, so it is
  natural to ask if there are more restrictive classes of vector fields,
  say of higher regularity, for which the Seifert Conjecture is true.
And indeed, over the next thirty years, this question was raised and
  answered in the negative for flows of increasing regularity.
For example, in 1988 Jenny Harrison \cite{Harr} adapted Schweitzer's
  argument to find a \(\mathcal{C}^{2+\delta}\) counterexample.
Using very different techniques, in 1994 Krystyna Kuperberg \cite{KK2} found
  a \(\mathcal{C}^\infty\) smooth counterexample to the Seifert
  Conjecture, and in 1996 Greg Kuperberg and Krystyna Kuperberg \cite{KK1}
  established an analytic counterexample.
Also in 1996, Greg Kuperberg \cite{Kup} found a volume preserving
  \(\mathcal{C}^1\) counterexample to the Seifert conjecture.

With so many counterexamples established, the Seifert Conjecture seemed
  definitively disproved, with one notable exception: Reeb flows.
Indeed, in 1993 Helmut Hofer \cite{H93} proved that every
  \(\mathcal{C}^\infty\) Reeb vector field on \(S^3\) generates a periodic
  orbit.  
A corollary of this result is the following. 
Let $\Omega$ be a smooth volume form on $S^3$ and $X$ a nonsingular volume
  preserving vector field.
Then  it holds  $d(i_X\Omega)=0$ and since $H^2(S^3,{\mathbb R})=0$ we can
  find a $1$-form $\lambda$ satisfying $d\lambda= i_X\Omega$.
Since $H^1(S^3,{\mathbb R})=0$ any primitive $\lambda'$  of $i_X\Omega$
  differs from $\lambda$ by the differential of a smooth map
  $h:S^3\rightarrow {\mathbb R}$, i.e. $\lambda' =\lambda+dh$. 
Hofer's theorem implies that in the case where a primitive $\lambda'$ of
  $i_X\Omega$ can be found satisfying  $\lambda'(X(x)))\neq 0$  for all
  $x\in S^3$, there exists a periodic orbit.
  
There are several points of note regarding Hofer's 1993 result.
Of particular interest is how heavily it relied on deep results from
  contact topology, like Eliashberg's classification of overtwisted contact
  three-manifolds as either tight or overtwisted, see \cite{E1}; Bennequin's
  proof that the standard contact structure on \(S^3\) is tight, see
  \cite{B}; and Eliashberg's complete classification of contact \(S^3\), see
  \cite{E2}.
To establish existence of periodic Reeb orbits, Hofer built on the theory
  of pseudoholomorphic curves introduced by Gromov in \cite{Gr}, and on
  Floer's idea to use them to find periodic orbits of Hamiltonian vector
  fields as in \cite{Floer}.
It is worth noting that Hofer's techniques were quite robust, and although
  not explicitly used to do so in \cite{H93}, they were capable of
  recovering Rabinowitz's results in \cite{Rab} which guarantee the
  existence of periodic Reeb orbits on the boundary of star-shaped domains
  in \(\mathbb{R}^4\).
Hofer's approach is relevant, since it is the principle idea behind the
  proof of Theorem \ref{THM_main_result} above.

Before proceeding, it is important to highlight a result which should be
  kept in mind, and held in contrast to Theorem \ref{THM_main_result}
  namely:

\begin{smalltheorem}[2003, Ginzburg-G\"{u}rel \cite{GG}] 
  \label{THM_ginzburg_gurel}
  \hfill \\
There exists a proper \(\mathcal{C}^2\)-smooth function
  \(H:\mathbb{R}^4\to \mathbb{R}\), for which \(H^{-1}(0) \simeq S^3\) is a
  regular level set on which the Hamiltonian flow has no periodic orbits.
\end{smalltheorem}

The focus here should not be on the relatively low regularity of
  the Hamiltonian, but rather on the non-existence of a periodic orbit.
Indeed, the relevance is that while the above result guarantees non-existence of
  any periodic orbits, the principle result of this manuscript guarantees
  the existence of a closed flow-invariant proper subset as a type of
  limit set of a pseudoholomorphic curve; this is discussed further in
  Section \ref{SEC_context_pseudoholomorphic_curves} below.
In particular then, this suggests that the new class of curves explored
  below indeed find closed invariant subsets more general than periodic
  orbits.
We note one slight caveat: Our analysis here is done in regularity
  \(\mathcal{C}^\infty\), while the Ginzburg-G\"{u}rel result holds in
  \(\mathcal{C}^2\).
Nevertheless we believe both results can be generalized to reach the
  desired conclusion. Indeed, the constructions in the present paper
  should be doable in a $\mathcal{C}^{2+\alpha}$-frame
  work.
Also, in \cite{GG}, the authors remark: `` It is quite likely that our
  construction gives an embedding  $S^3\rightarrow {\mathbb R}^4$ without
  closed characteristics, which is $\mathcal{C}^{2+\alpha}$-smooth."

With these results established, the answer to the Seifert Conjecture is
  well understood and essentially complete: It is false for vector fields
  as regular as one likes, and false for volume preserving flows, but true
  for Reeb flows.
At this point we turn our attention to Question 2 and the Gottschalk
  Conjecture, and we begin by noting that the \emph{lack} of progress on
  this problem stands in stark contrast to the nearly complete
  understanding of the Seifert Conjecture.
Indeed, despite more than a half century worth of attempts, no essential
  progress has been made on the Gottschalk Conjecture.
We make two important qualifications to that statement.
First, strictly speaking, results stated above which guarantee existence
  of periodic orbits, for example \cite{H93} and \cite{Rab}, are progress
  on the Gottschalk Conjecture for the class of Reeb vector fields, however
  because the closed invariant sets are always periodic orbits, this is more
  a result about the Seifert Conjecture than the Gottschalk Conjecture.
Second, although there has been no direct progress on the Gottschalk
  Conjecture, there have been a variety of results on related problems.
For example, in 2009 Clifford Taubes \cite{Taubes-Erg} proved that a
  volume preserving vector field on a compact 3-manifold  whose dual 2-form
  is exact (such as \(S^3\)) can not generate uniquely ergodic dynamics
  unless its asymptotic linking number is zero; in 2014 Bassam Fayad and
  Anatole Katok \cite{FK} construct analytic uniquely ergodic (hence
  minimal) volume preserving \emph{maps} (but not flows) on odd dimensional
  spheres; and in 2015 Ginzburg and Niche \cite{GN} showed that the
  autonomous Hamiltonian flow on a compact regular energy level in
  \(\mathbb{R}^{2n}\) (and somewhat more generally) cannot be uniquely
  ergodic.

In short, results in the direction of the Gottschalk Conjecture have been
  one of two types, namely either establishing the existence of periodic
  orbits as in the Reeb case, or else making definitive progress on a
  related problem.
As such, we note that it is somewhat surprising that more direct progress
  has not been made given the importance of this problem.
For example, if the Gottschalk conjecture is true, then in all likelihood a
  method to prove it will need to develop a global theory for
  finding closed invariant subsets, which in turn will touch on long-standing
  questions in dynamical systems, particularly in cases in which flows are
  volume-preserving.
It is also worth noting that Gottschalk's question has been well
  established as historically significant.
Indeed, it was raised in 1974 during the American Mathematical Society's
  special symposium on the mathematical consequences of Hilbert's problems
  \cite{AMS}.
It made another appearance in \cite{Smale} when mentioned by Steven Smale
  in his list of the most important problems for the twenty-first century.  
And it appeared again in 1998 at the International Congress of
  Mathematics during Michael Herman's talk \cite{Herman}, in which he
  raised the following related question.\\

\noindent{\bf Question:} (1998, Herman) \emph{When \(n\geq 2\), can one find
  a \(\mathcal{C}^\infty\) compact, connected, regular hypersurface in
  \(\mathbb{R}^{2n}\) on which the characteristic flow is minimal?}\\

Recall that the characteristic flow is just the Hamiltonian flow associated
  to any smooth Hamiltonian for which the hypersurface is a regular
  energy level.
Consequently, Herman's question might be regarded as the Hamiltonian
  Gottschalk conjecture for compact energy levels in \(\mathbb{R}^{2n}\),
  and the principle result of this manuscript is to answer his question in
  the negative when \(n=2\), i.e. a version of the  Gottschalk conjecture holds 
  for compact regular Hamiltonian energy surfaces in ${\mathbb R}^4$. 
We complete this section by stating a conjecture, which seems plausible
  given the developments in this paper. 
It combines a question about almost existence of periodic orbits, a
  well-studied problem, with the existence question of proper closed
  invariant subsets.

\begin{conjecture}[minimal sets in energy piles]
  \label{CON_minimal_sets_energy_piles}
  \hfill\\
Assume that $\Omega$ is a symplectic form on
  $[-1,1]\times S^3$ and denote by $H:[-1,1]\times S^3\rightarrow {\mathbb
  R}$ the Hamiltonian defined by $H(t,m)=t$.
Denote by $\Sigma_t$  the regular compact energy surface $H^{-1}(t)$ and
  define the subset $S\subset [-1,1]$ to consist of all $t$ for which the
  energy surface $\Sigma_t$ carries a periodic orbit. Then the following
  holds:
\begin{enumerate}                                                         
  \item \(\text{measure}(S)=2\)
  \item For \(t\in [-1,1]\setminus S\) there exists a closed proper
    invariant subset for the Hamiltonian flow on \(\Sigma_t\).
  \end{enumerate}
\end{conjecture}

\subsection{Context: Pseudoholomorphic
  Curves}\label{SEC_context_pseudoholomorphic_curves}
In 1985 Mikhail Gromov \cite{Gr} introduced the notion of 
  pseudoholomorphic curves in almost complex manifolds.
Such curves were a generalization from  holomorphic curves in complex
  manifolds, to curves in real manifolds equipped by a preferred rotation by
  \(90\) degrees in the tangent bundle (determined by an almost complex
  structure; see Definition \ref{DEF_acm} below).
Roughly speaking then, a pseudoholomorphic curve is a map from a Riemann
  surface into a manifold equipped with an almost complex structure with the
  property that the derivative of the map intertwines the complex structure
  on the Riemann surface with the almost complex manifold on the target.

These curves solve an elliptic partial differential equation and they form
  the zero set of a non-linear Fredholm operator and thus tend to live in
  smooth families.
A crucial observation by Gromov was that if the almost complex
  structure \(J\) is tamed by a symplectic form, then curves in a fixed
  homology class will have a priori bounded energy and area, and hence
  they degenerate in a manner which is essentially indistinguishable from
  the manner in which algebraic curves degenerate in smooth projective
  varieties; from a geometric analysis perspective, this is also
  essentially the same manner in which minimal surfaces degenerate in
  Riemannian manifolds.
Put another way, modulo the formation of nodal or cusp curves, families of
  pseudoholomorphic curves of a fixed homology class are compact; this is
  the celebrated Gromov compactness theorem for pseudoholomorphic curves.
Moreover, algebraic counts of curves have yielded the so-called
  Gromov-Witten invariants. 

In 1986, shortly after Gromov's seminal paper, Andreas Floer \cite{Floer}
  discovered that an inhomogeneous version of the pseudoholomorphic curve
  equation could be used to study the Morse homology of the loop space of a
  closed symplectic manifold.
Here the Morse function was the symplectic action functional associated to
  a one-periodic Hamiltonian function.
In turn, this action functional had one-periodic orbits of a Hamiltonian
  flow as critical points, and with such orbits as generators, the
  differential was determined by counting perturbed pseudoholomorphic
  cylinders (the so called Floer trajectories) between such orbits.
The resulting theory has become known as Hamiltonian Floer homology.

Then in 1993, Helmut Hofer \cite{H93} considered a sort of
  hybrid case: pseudoholomorphic curves in symplectizations of contact
  manifolds. 
Here the interesting feature was that the curves had infinite area, but
  had finite Hofer-energy; or equivalently, uniformly bounded local-area.
It turned out that such curves were asymptotic to cylinders over periodic
  Reeb orbits.
Furthermore, these curves were either positively or negatively asymptotic
  to such orbit cylinders, and hence under certain hypotheses one could
  construct a variety of flavors of contact homology (cylindrical,
  linearized, full, rational, embedded, etc), in which the generators are
  certain sets of (sometimes weighted) periodic Reeb orbits, and with the
  differential determined by counting certain finite energy
  pseudoholomorphic curves which positively limit to one orbit set and
  negatively limit to another orbit set.

It was eventually discovered that each of these theories (Gromov-Witten
  invariants, Hamiltonian Floer homology, contact homology, etc) was
  subsumed in a larger Symplectic Field Theory (SFT) proposed by
  Eliashberg, Givental, and Hofer in \cite{EGH}.
More precisely, the moduli spaces of pseudoholomorphic curves that
  generate each of these theories is contained in the collection of moduli
  spaces studied in SFT.

An absolutely crucial feature in each of these theories is that the curves
  in question have an a priori energy bound, which should be deduced from
  representing a fixed homology class, and which in turn guarantees a
  (local) area bound.
Indeed, without such energy control, pseudoholomorphic curves have
  notably wild behavior.
For example, in the symplectization of a contact manifold,
  \(\mathbb{R}\times M\), for any admissible almost complex structure and
  any Reeb trajectory \(\gamma:\mathbb{R}\to M\), the map \((s,t)\mapsto (s,
  \gamma(t))\in \mathbb{R}\times M\) is pseudoholomorphic and of infinite
  energy and which may have an image which is dense in \(\mathbb{R}\times
  M\); we call such curves pseudoholomorphic \emph{sheets}.
As a consequence of the apparent wild behavior of infinite energy curves,
  both popular and expert belief has been that there exists a dichotomy
  among pseudoholomorphic curves: those with energy bounds and those
  without.
Moreover, the former are tame and well understood while the latter have
  such wild behavior that one cannot feasibly hope study them in a
  meaningful way.

To illustrate this idea, we draw an analogy with holomorphic functions on
  the punctured complex plane.
Here, of course, there \emph{is} a dichotomy, namely functions with poles
  versus functions with essential singularities.
The former are meromorphic functions and are algebraic in nature, while
  the latter are especially unmanageable, particularly in light of Picard's
  Great Theorem, which states that in each neighborhood of an essential
  singularity, a holomorphic function takes on every complex value (except
  possibly one) infinitely many times.
This clear division of holomorphic functions has long been assumed to
  carry over into the realm of Symplectic Field Theory: curves either have
  bounded energy, are tame, and are well understood, or else they have
  unbounded energy, are wild, and are unmanageable.
One of the main thrusts of this manuscript is to defy conventional wisdom,
  and illuminate an intermediate class of infinite energy curves.
Or, perhaps more accurately, identify a class of curves which appears to
  \emph{interpolate} between tame and wild curves, which we designate as
  \emph{feral} curves.
We give a precise formulation of feral curves in Definition
  \ref{DEF_feral_J_curve} below, but roughly speaking they are proper
  pseudoholomorphic maps \(u:S\to \mathbb{R}\times M\) into
  symplectizations of framed Hamiltonian manifolds\footnote{
    For a precise formulation of a framed Hamiltonian manifold see
    Definition \ref{DEF_framed_Hamiltonian_structure} below, however at
    present we note that it is more general than both contact and stable
    Hamiltonian.
    } 
  for which \(S\) has finite topology (genus, connected components, etc.)
  and \(\int_S u^*\omega < \infty\).
We note that on one hand, the properness condition rules out the
  aforementioned pseudoholomorphic sheets \((s,t)\mapsto (s, \gamma(t))\),
  and the finite \(\omega\)-energy condition tends to prevent such curves
  becoming too wild, however, by not requiring the Hamiltonian structure
  to be stable allows feral curves to have infinite Hofer energy, and
  indeed we expect that some definitely do.

Before proceeding, we aim to give some idea of how difficult it is to
  study pseudoholomorphic curves without a priori bounded energy, so we
  take a moment to step through some potential issues.
As a model starting point, one might consider a sequence of finite energy
  planes, all asymptotic to the same simply covered orbit cylinder, and
  study what might happen as one progresses through the sequence while
  assuming the Hofer energy tends to infinity.
First, the SFT compactness theorem for pseudoholomorphic curves
  \cite{BEHWZ} does not apply directly, since energy is unbounded.
Nevertheless, one might mimic the argument to see where it breaks down.
In this model case, the conformal structures on the domain Riemann
  surfaces do not change, so the key issue is whether or not the gradient
  is bounded; if boundedness fails, we attempt bubbling analysis.
This is where difficulties start to arise.

In Gromov-Witten theory, if the gradient blows up, then rescaling analysis
  extracts a sphere-bubble, which captures a threshold amount of energy.
In SFT compactness something similar occurs, except that rescaling
  analysis extracts  either a sphere-bubble or else a finite energy plane,
  and either object captures a threshold amount of
  \(\omega/d\lambda\)-energy, so the process terminates after finitely
  many iterations.  
But without energy bounds, we cannot guarantee that a \emph{finite} energy
  plane bubbles off -- instead one might only be able to extract something
  akin to an infinite energy sheet, which has arbitrarily small
  \(\omega/d\lambda\)-energy.
Worse still, without some threshold amount of energy being captured via
  rescaling analysis, one can no longer guarantee that the gradient blows up
  only in a neighborhood of finitely many points.
Indeed, a priori the gradient could blow up everywhere.

Still, maybe by some alternate methods, or by considering a model
  example, one could perhaps extract something like an infinite energy plane
  which has finite \(\omega\)-energy.
However, even in such a case, two possibilities complicate matters
  further.
First, a priori, it need not be the case that the domain Riemann surface of
  such a curve is conformally equivalent to the complex plane; it could be
  an open disk instead.
In the SFT setting, it is usually assumed that the domains of curves are
  conformally equivalent to punctured Riemann surfaces, however this is an
  assumption which can be removed and then easily deduced from other
  standard analysis.
However, for infinite energy curves it is a possibility which needs to be
  more seriously considered.
Second, given a single proper infinite energy plane (or disk, as the case
  may be), it need not be the case that the gradient is globally bounded.
Again, in the usual SFT setting, this can be deduced in a variety of ways
  which depend on asymptotic analysis or finiteness of energy, but in
  the infinite energy case it is a possibility that must again be
  considered.

To summarize the difficulties, we see that once we remove a priori
  energy bounds, SFT compactness does not apply, there is no local area
  bound, there is no energy threshold, there is potentially dense gradient
  blowup, a single curve can have unbounded gradient, and even something
  simple like an infinite energy ``plane'' might in fact be holomorphically
  parametrized by an open disk, or its image may be dense in the target
  manifold.
In short, without energy bounds our arsenal of standard pseudoholomorphic
  techniques becomes largely ineffectual, and curve analysis rapidly
  appears unmanageable.
Those somewhat familiar with pseudoholomorphic curves can then perhaps
  see the difficulty faced at the outset: With so many basic tools rendered
  inapplicable, it becomes exceedingly difficult to formulate what
  properties to expect, let alone prove them.

Nevertheless, despite these obstacles, analysis is still possible, and it
  should not be surprising that a bulk of this manuscript is dedicated to
  establishing sufficient properties to prove the main dynamical result.
An overview of these results is provided in Section
  \ref{SEC_overview_of_results} below, but at present we provide an
  alternate characterization of feral curves which may be less amenable to
  analysis but which is better for providing a conceptual framework.

To that end, we first back up and re-characterize finite energy curves
  inside symplectizations of contact manifolds, where \(\omega =
  d\lambda\),  as follows.
Outside a large compact set, say \( [-n, n]\times M\) for \(n \gg 1\), a
  finite Hofer-energy curve is immersed, and the tangent planes are nearly
  vertical; that is, they are nearly tangent to the two-plane distribution
  \({\rm ker}\; \omega \subset T(\mathbb{R}\times M)\).
Consequently, outside a large compact set, one can project the
  asymptotic ends of a curve into the manifold \(M\) and regard this as a
  path of loops parameterized by level sets of the symplectization
  coordinate \(\mathbb{R}\).
Of interest here is the fact that such a path of loops is in fact an
  integral curve of a gradient-like vector field on the loop space of \(M\)
  which has periodic Reeb orbits as rest points.
Keeping this in mind, one can then regard finite energy pseudoholomorphic
  curves as submanifolds which can be geometrically or topologically
  interesting in some large compact sets of \(\mathbb{R}\times M\), like
  inside \([-n, n]\times M\), but outside of this compact set they can
  morally be thought of as gradient flow lines converging to critical
  points of a functional on the loop space of \(M\).
The surprising feature of feral curves is that they can be thought of in
  nearly the same way.
Indeed, as we make clear below, outside a large compact set a feral
  pseudoholomorphic curve is immersed with tangent planes nearly vertical.
Again, the result is that the ends of a feral curve can be regarded as
  path of loops in \(M\), and this path is in fact an integral curve of a
  gradient-like vector field.
The key difference however, which stands in stark contrast with the
  contact and stable Hamiltonian case, is that in the general framed
  Hamiltonian case the action functional  is not Palais-Smale.
More specifically, feral curves have ends which are ``gradient'' flow
  lines along which the action is bounded but the trajectory escapes to
  infinity.

The above characterization of feral curves is then both a boon and a curse.
The upside is that despite the fact that curves without energy bounds seem
  wildly unmanageable, we show that feral curves nevertheless have a
  surprising number of properties which make their study tractable and
  somewhat familiar, if non-standard.
Moreover, feral curves still lie in the general heuristic framework in
  which pseudoholomorphic curves are of type of generalized gradient flow
  line, and hence could be used to define some generalized version of Morse
  homology or a more complicated algebraic invariant like Symplectic Field
  Theory.
The great downside though, is that Morse theory for a general
  non-Palais-Smale functional is an ill conceived notion, and at best it
  is unlikely to be an invariant, and at worst it simply cannot be
  defined.
Indeed, in some sense, the general action functional in the framed
  Hamiltonian case appears to have ``critical points at infinity,'' which,
  at present, defy direct analysis, and hence preclude a complete SFT
  compactness theorem for feral curves, as well as a Fredholm theory, a
  gluing theory, and a reasonable hope of an algebraic invariant.

It is possible that the above characterization, and the potential problems
  it brings, may give the impression of casting a dark shadow over the
  landscape of possibilities for feral curves.
We take a moment then to highlight certain glimmers of hope.
First, we note that in examples, feral curves tend to have ends with a
  rather nice property: They tend to limit to a finite collection of
  hyperbolic minimal sets connected by families of heteroclinic
  trajectories.
Or, more geometrically then, while we have become used to
  pseudoholomorphic curves bubbling or breaking (as in Floer homology,
  contact homology, etc) and limiting to periodic orbits, now it seems
  possible that periodic orbits can themselves bubble or break and that
  feral curves detect this and limit to the broken orbit.
This raises a question: If one can analytically understand the violent
  breaking and gluing phenomena in Morse-like homology theories, then why
  can one not adapt the analysis to understand curves limiting to broken
  periodic orbits as well?
Perhaps one can.
Or perhaps one must regularize the space of periodic orbits, broken or
  not, in a fashion similar to regularizing moduli spaces of
  pseudoholomorphic curves before defining a differential or more
  complicated algebraic invariant.
In either case, these possibilities warrant investigation.

Finally, we raise an important, and perhaps deeper, question.\\

\noindent{\bf Question:} \emph{Are feral pseudoholomorphic curves
  essential or inessential?}\\

We elaborate.  
One perspective is that rather fundamentally, pseudoholomorphic curves
  detect topology of a symplectic nature.
Thus when pseudoholomorphic curves behave unexpectedly, there are roughly
  two possibilities.
The first is that the odd behavior is somehow non-generic and therefore is
  likely to be inconsequential.
The second is that the pseudoholomorphic curves in question are actually
  detecting an unexpected topological feature, and thus such curves, and
  the detected phenomena, are important and essential.
It is hopefully clear that answering the above question -- in either
  direction -- is an important avenue of research.

To close this section, we bring the discussion back, almost full circle,
  to the tame/wild dichotomy, and how feral curves fit comfortably in
  neither class, but rather share properties of each.
A consequence is that they provide a definitive opportunity to push
  pseudoholomorphic curves beyond their conventional limitations and
  possibly discover remarkably novel phenomena.
In order to proceed, the only price to pay is a willingness to give up a
  large body of conventional tools and intuition in favor for building new
  techniques from the ground up. 
The task is arduous, but in the end appears fruitful, as our principle
  dynamical result suggests.\\
 
\subsection{Overview of Results}\label{SEC_overview_of_results}
The purpose of this section is to provide an overview of the most
  important results proved in this manuscript.
The first result, Theorem \ref{THM_main_result}, has already been stated,
  but we restate it here for completeness.
The second, Theorem \ref{THM_second_main_result}, is an immediate
  generalization.
Each of these results are proved in Section
  \ref{SEC_existence_minimal_subsets}, however they each rely on some rather
  non-trivial properties of pseudoholomorphic curves which are then proved
  in Section \ref{SEC_supporting_proofs}.
Indeed, these results regarding properties of the so-called feral curves
  appear to be quite fundamental not just to our results, but for many
  future results as well.
Indeed, they appear to form the basic foundational analysis for the
  extension of pseudoholomorphic curve theory beyond symplectizations of
  contact and stable Hamiltonian manifolds, and into the realm of only
  framed Hamiltonian manifolds and symplectic cobordisms with simply generic
  boundary.
As such, we designate these results as theorems and highlight them below.
  In order to understand the statement of some of these results, we also
  provide some basic definitions, including the namesake of this manuscript,
  the \emph{feral curve}.
For each such result we provide a brief description to highlight its
  utility.

\setcounter{CurrentSection}{\value{section}}
\setcounter{CurrentTheorem}{\value{theorem}}
\setcounter{section}{\value{CounterSectionTheoremTwo}}
\setcounter{theorem}{\value{CounterTheoremTheoremTwo}}
\begin{theorem}[main dynamical result]\hfill\\
Consider \(\mathbb{R}^4\) equipped with the standard symplectic structure
  and a Hamiltonian \(H\in \mathcal{C}^\infty(\mathbb{R}^4, \mathbb{R})\)
  for which \(M:=H^{-1}(0)\) is a non-empty compact regular energy level.
Then the Hamiltonian flow on \(M\) is not minimal.
\end{theorem}
\setcounter{section}{\value{CurrentSection}}
\setcounter{theorem}{\value{CurrentTheorem}}

This of course is the main dynamical result of this manuscript.  
It is worth noting that it is crucial that the energy level be compact. 

\setcounter{CounterSectionTheoremTwo}{\value{section}}
\setcounter{CounterTheoremTheoremTwo}{\value{theorem}}
\begin{theorem}[second main dynamical result]
  \label{THM_second_main_result}
  \hfill\\
Let \((M^\pm, \eta^\pm)\) be a pair of compact three-dimensional framed
  Hamiltonian manifolds, and let \((\widetilde{W}, \tilde{\omega})\) be a
  symplectic cobordism from \((M^+, \eta^+)\) to \((M^-, \eta^-)\) in the
  sense of Definition \ref{DEF_symplectic_cobordism}.
Suppose that \((\widetilde{W}, \tilde{\omega})\) is exact,  \(M^-\) is
  connected, and that \((M^+, \eta^+)\) is contact type and has a
  connected component \(M'\) which is either \(S^3\), overtwisted, or
  there exists an embedded \(S^2\) in \(M'\subset \partial \widetilde{W}\)
  which is homotopically nontrivial in \(\widetilde{W}\).
Then the flow of the Hamiltonian vector field \(X_{\eta^-}\) on \(M^-\) is
  not minimal.
\end{theorem}
This is the second main dynamical result of this manuscript.  
It is perhaps surprising that this generalization can be obtained with so
  little modifications from the proof of Theorem \ref{THM_main_result}.

\begin{remark}[removing the exactness condition]
  \label{REM_removing_exactness}
  \hfill \\
It should be straightforward to generalize the argument of the proof to
  the case where exactness is replaced by the assumption that
  $\widetilde{\omega}$ vanishes on $\pi_2$.
This assumption would prevent a certain type of bubbling. 
Possibly, using polyfold technology, one might even get away without any
  assumption on the symplectic form $\widetilde{\omega}$.
\end{remark}
%

We now turn our attention to providing some definitions, which will in
  turn allow us to state a number of properties of the pseudoholomorphic
  curves to be studied.
We begin with the notion of a generalized puncture, which is necessary to
  define since a priori our curves may be non-compact but their domains need
  not be conformally equivalent to a finitely punctured Riemann surface.
\begin{definition}[generalized punctures]
  \label{DEF_generalized_punctures}
  \hfill\\
Let \(S\) and \(W\) each be smooth finite dimensional manifolds, each
  possibly non-compact, and each possibly with smooth compact boundary.
Let \(u:(S, \partial S)\to (W, \partial W)\) be a smooth proper map.
Let \(W_k\subset W\) be a sequence of open sets each with compact closure
  which satisfy
  \begin{enumerate}                                                      
      \item \(W_k \subset W_{k+1}\) for all \(k\in \mathbb{N}\)
      \item \(W = \cup_{k\in\mathbb{N}} W_k\).
     \end{enumerate}
Define \({\rm Punct}^{W_k}(S)\) to be the number of non-compact
  path-connected components of the set \(S\setminus u^{-1}(W_k)\).
Define 
\begin{equation*}                                                         
    {\rm Punct}(S):= \lim_{k\to \infty} {\rm Punct}^{W_k}(S).
  \end{equation*}
\end{definition}
%

\begin{remark}[monotonicity of $\rm{Punct}$]
  \label{REM_punct}
  \hfill\\
Regarding Definition \ref{DEF_generalized_punctures}, we note that if
  \(W'\subset W''\) are open subsets of \(W\), each with compact closure,
 then it straightforward to show that
  \begin{equation*}                                                      
      {\rm Punct}^{W'}(S) \leq {\rm Punct}^{W''}(S)
    \end{equation*}
  and hence \({\rm Punct}(S)\) is well-defined, and defined independent
  of the choice of exhausting sequence \(\{W_k\}_{k\in \mathbb{N}}\).
\end{remark}
%

Next we aim to provide the primary novel definition of this manuscript,
  however it relies on a number of standard notions which some readers may
  not be familiar with, but which are provided later in Section
  \ref{SEC_background}.
As such, we note that it will be helpful to be familiar with the notion of a
  framed Hamiltonian manifold (Definition
  \ref{DEF_framed_Hamiltonian_structure}), an \(\eta\)-adapted almost
  Hermitian structure (Definition \ref{DEF_adapted}), and a proper marked
  nodal pseudoholomorphic curve (Definition
  \ref{DEF_pseudholomorphic_curve}).
With these understood, we can then define a feral curve.

\begin{definition}[feral curves]
  \label{DEF_feral_J_curve}
  \hfill\\
Let \((M, \eta)\) be a framed Hamiltonian manifold, and let \((J, g)\) be an
  \(\eta\)-adapted almost Hermitian structure on \(\mathbb{R}\times
  M\).
Let \(\mathbf{u}=(u, S, j, W, J, \mu, D)\) be a proper marked nodal
  pseudoholomorphic curve (possibly with compact boundary) in
  \(\mathbb{R}\times M\).
We say  \(\mathbf{u}\) is a \emph{feral pseudoholomorphic
  curve}, or simply a \emph{feral curve}, provided
  \begin{enumerate}                                                       
    \item 
    \(\int_{S}u^*\omega <\infty\)
    \item 
    \({\rm Genus}(S)<\infty\)
    \item 
    \({\rm Punct}(S)<\infty\); that is, \((u, S, j)\) has a finite
    number of generalized punctures.
    \item 
    \(\#\mu <\infty\)
    \item 
    \(\# D < \infty\)
    \item 
    \(\# \pi_0(S)<\infty\)
    \end{enumerate}
\end{definition}
%

The above is the namesake definition of this manuscript. 
It may be helpful to think of such a curve simply as being proper
  pseudoholomorphic map, with finite \(\omega\)-energy, and finite topology.
It is also worth noting that in the more usual case that \(\eta = (\lambda
  , d\lambda)\) is a contact manifold, a feral curve is nothing other than a
  finite energy pseudoholomorphic curve.

We are now prepared to state the main properties of feral curves.
\setcounter{CounterSectionAreaBounds}{\value{section}}
\setcounter{CounterLemmaAreaBounds}{\value{lemma}}
\begin{theorem}[area bounds]
  \label{THM_area_bounds}
  \hfill\\
Let \((M, \eta)\) be a compact framed Hamiltonian manifold, let \((J,
  g)\) be an \(\eta\)-adapted almost Hermitian structure on
  \(\mathbb{R}\times M \), and fix positive constants \(r>0\) and
  \(E_0>0\).
Then there exists a constant \(C=C(J, g, \omega, \lambda, r, E_0)\)
  with the following property.
For each proper pseudoholomorphic map \(u:S\to \mathbb{R}\times M\)
  without boundary which satisfies
  \begin{equation*}                                                       
    \int_S u^*\omega \leq E_0 <\infty,
    \end{equation*}
  and for which there exists there exists \(a_0\in \mathbb{R}\) such
  that \((a\circ u)^{-1}(a_0) = \emptyset\) (e.g. if \(a\circ u\subset
  [0, \infty)\times M\)), the following holds. 
  \begin{equation*}                                                       
    \int_{\widetilde{S}} u^* (da \wedge \lambda + \omega) \leq C,
    \end{equation*} 
  where 
  \begin{equation*}                                                       
    \widetilde{S}:=\{\zeta\in S: a_0 - r <a\circ u (\zeta)< a_0 + r\}.
    \end{equation*}
To be clear: \(C\) depends on ambient geometry in \(\mathbb{R}\times
  M\), \(r\), and the \(\omega\)-energy bound \(E_0\), but \emph{not}
  the map \(u\).
\end{theorem}
%

The above estimate, as well as the generalizations provided in Section
  \ref{SEC_proof_of_exp_area_bounds}, are rather interesting.
Roughly the above states that if a feral curve has a local maximum or a
  local minimum, then the area cannot be arbitrarily large in a bounded
  neighborhood of that extremal point.
This is important because in general feral curves definitely can develop
  unbounded local area, but in some sense this must occur very far away from
  the absolute minimum or maximum.
In Section \ref{SEC_proof_of_exp_area_bounds} we shall prove a more general result
  about area bounds in a neighborhood of a level set of a proper curve with
  finite \(\omega\)-energy.
Very roughly, we show that for 
  \begin{equation*}                                                       
    \widetilde{S}_r:=\{\zeta\in S: a_0 - r <a\circ u (\zeta)< a_0 + r\}.
    \end{equation*}
  we have
  \begin{align*}                                                          
    {\rm Area}_{u^*g}(\widetilde{S}_r) \leq A e^{B r} 
    \end{align*}
  where \(A\) depends on \(\int_{(a\circ u)^{-1}(0)}u^*\lambda\) and \(\int_S
  u^*\omega\), and \(B\) depends only on the geometry of the ambient
  manifold.

\setcounter{CounterSectionEnergyThreshold}{\value{section}}
\setcounter{CounterTheoremEnergyThreshold}{\value{theorem}}
\begin{theorem}[$\omega$-energy threshold]
  \label{THM_energy_threshold}
  \hfill\\
Let \((M, \eta=(\lambda, \omega))\) be a compact framed Hamiltonian
  manifold, and let \((J, g)\) be an \(\eta\)-adapted almost Hermitian
  structure on \(\mathbb{R}\times M\).
Also, fix positive constants \(r>0\), and \(C_g>0\). 
Then there exists a positive constant \(0<\hbar=\hbar(M, \eta, J, g,
  r, C_g)\) with the following significance.
Let \(\{\mathbf{h}_k\}_{k\in\mathbb{N}}\) be a sequence of quadruples
  \((J_k, g_k, \lambda_k, \omega_k)\) with the property that each
  \(\eta_k=(\lambda_k, \omega_k)\) is a Hamiltonian structure on \(M\),
  and each \((J_k, g_k)\) is an \(\eta_k\)-adapted almost Hermitian
  structure on \(\mathbb{R}\times M\), and suppose that
  \begin{align*}                                                          
    (J_k, g_k, \lambda_k, \omega_k) \to (J, g, \lambda, \omega)
    \qquad\text{in }\mathcal{C}^\infty\text{ as }k\to \infty.
    \end{align*}
Furthermore, fix \(a_0\in \mathbb{R}\), and let \(u_k\colon S_k\to
  \mathbb{R}\times M\) be a sequence of compact connected generally
  immersed pseudoholomorphic maps which satisfy the following conditions:
  \begin{enumerate}[($\hbar$1)]                                           
      \item\label{EN_hbar1} either \(a\circ u_k(S_k)\subset [a_0,
        \infty)\) or  \(a\circ u_k(S_k)\subset (-\infty, a_0]\) for all
        \(k\in \mathbb{N}\)
      \item\label{EN_hbar2} \({\rm Genus}(S_k)\leq C_g\)
      \item\label{EN_hbar3} \(a\circ u_k(\partial S_k)\cap [a_0-r, a_0+r]
	= \emptyset\)
      \item\label{EN_hbar4} \(a_0\in a\circ u_k (S_k)\).
    \end{enumerate}
Then for all sufficiently large \(k\in \mathbb{N}\) we have
  \begin{equation*}                                                       
      \int_{S_k} u_k^*\omega_k \geq \hbar.
    \end{equation*}
\end{theorem}
%

Whereas Theorem \ref{THM_area_bounds} is concerned with showing that the
  area near an absolute maximum or minimum of a feral curve cannot be too
  large, Theorem \ref{THM_energy_threshold} shows that it cannot be to
  small either; or more precisely that the \(\omega\)-energy cannot be too
  small.
This result is one of the easiest to obtain, and follows essentially from
  a compactness theorem.
However, we note that such a compactness theorem \emph{requires} an area
  bound which one only has as an application of Theorem
  \ref{THM_area_bounds}.
We also note that the bound on genus can almost certainly be removed.
Indeed, whereas our proof employs target-local Gromov compactness, which
  requires the genus bound, one could probably replace our argument with
  Taubes's convergence as integral currents which does not require a genus
  bound.
Because some of our later results do require such a genus bound, such a
  (potentially) superfluous condition is not a hindrance and creates a more
  self-contained presentation.

\setcounter{CounterSectionAsymptoticLocalLocalAreaBound}{\value{section}}
\setcounter{CounterTheoremAsymptoticLocalLocalAreaBound}{\value{theorem}}
\begin{theorem}[asymptotic connected-local area bound]
  \label{THM_local_local_area_bound}
  \hfill\\
Let \((M, \eta)\) be a compact framed Hamiltonian manifold, and let \((J,
  g)\) be an \(\eta\)-adapted almost Hermitian structure on
  \(\mathbb{R}\times M\).
Then there exists a positive constant \(r_1=r_1(M, \eta, J, g)\) with
  the following significance.
For each generally immersed feral pseudoholomorphic curve \((u, S, j)\) in
  \(\mathbb{R}\times M\), there exists a compact set of the form
  \(K:=[-a_0, a_0]\times M\) with the property that for each \(\zeta\in
  S\) such that \(u(\zeta)\notin K\) we have
  \begin{equation*}                                                       
      {\rm Area}_{u^*g}\big(S_{r_1}(\zeta)\big)\leq 1;
    \end{equation*}
  here \(S_{r_1}(\zeta)\) is defined to be the connected component of
  \(u^{-1}(\mathcal{B}_{r_1}(u(\zeta)))\) containing \(\zeta\), and
  \(\mathcal{B}_{r_1}(p)\) is the open metric ball of radius \(r_1\)
  centered at the point \(p\in \mathbb{R}\times M\).
\end{theorem}

Arguably, Theorem \ref{THM_local_local_area_bound} is the most important
  property of feral curves developed here.
The difficulty is that in general the Hofer-energy of a feral curve may be
  infinite, which is to say that in general, we have
  \begin{align*}                                                            
    \sup_{z\in S} {\rm Area}_{u^*g}\Big(
    u^{-1}\big(\mathcal{B}_\epsilon(u(z))\big)\Big) = \infty
    \end{align*}
  for each \(\epsilon>0\); here \(\mathcal{B}_\epsilon(p)\) is a ball of
  radius \(\epsilon\) centered at \(p\in\mathbb{R}\times M\).
In contrast, Theorem \ref{THM_local_local_area_bound} states that if we
  replace \(u^{-1}(\mathcal{B}_\epsilon(u(z)))\) with the connected
  component in this set containing \(z\), then the associated supremum is
  in fact finite.
Establishing this estimate is a rather technical process, and it is
  perhaps worth noting that our proof crucially relies on the fact that the
  genus of a feral curve is finite and that the number of generalized
  punctures is also finite.
Indeed, deducing these area bounds in part from genus bounds is notably
  delicate.

With such a ``connected-local'' area bound established, a variety of
  asymptotic properties of feral curves can be established essentially via
  target-local Gromov compactness.
One such important result is the following.

\setcounter{CounterSectionAsymptoticCurvatureBound}{\value{section}}
\setcounter{CounterTheoremAsymptoticCurvatureBound}{\value{theorem}}
\begin{theorem}[asymptotic curvature bound]
  \label{THM_curv_bound}
  \hfill\\
Let \((M, \eta)\) be a compact framed Hamiltonian manifold, and let \((J,
  g)\) be an \(\eta\)-adapted almost Hermitian structure on
  \(\mathbb{R}\times M\).
For each feral pseudoholomorphic curve \(\mathbf{u}= (u, S, j,
  \mathbb{R}\times M, J, \mu, D)\), there exists a compact set of the form
  \(K:=[-a_2, a_2]\times M\), and positive constant \(C_\kappa=C_\kappa(M,
  \eta, J, g)\) with the following significance.
First, the restricted map
  \begin{equation*}                                                       
      u:S\setminus u^{-1}(K)\to \mathbb{R}\times M 
    \end{equation*}
  is an immersion. 
Second, for each \(\zeta\in S\setminus u^{-1}(K)\) we have 
  \begin{equation*}                                                       
      \|B_u(\zeta)\| \leq C_\kappa
    \end{equation*}
  where \(B_u(\zeta)\) is the second fundamental form of the immersion
  \(u\) evaluated at the point \(\zeta\).
\end{theorem}
%

It may be useful to paraphrase the above result as saying that outside a
  large compact set, a feral curve is immersed with a uniform point-wise
  curvature bound.
This result, together with Theorem \ref{THM_local_local_area_bound}
  guarantees that our curves have the nicest possible asymptotic behavior
  given that the Hofer-energy can be infinite.
Indeed, given that a curve with infinite Hofer-energy is generally thought
  to be too wild to analyze, the above two results provide a tremendous
  amount of structure.

\setcounter{CounterSectionExistenceWorkhorse}{\value{section}}
\setcounter{CounterTheoremExistenceWorkhorse}{\value{theorem}}
\begin{theorem}[existence workhorse]
  \label{THM_existence}
  \hfill\\
Let \((M, \eta)\) be a compact framed Hamiltonian manifold with \({\rm
  dim}(M) = 3\).
Let \(\{a_k\}_{k\in \mathbb{N}}\subset \mathbb{R}^-\) be a sequence for
  which \(a_k\to -\infty\) monotonically.
For each \(k\in \mathbb{N}\), let \((J_k, g_k)\) be a \(\eta\)-adapted
  almost complex structure on \(\mathbb{R}\times M\).
Suppose that there exists a positive constant \(C\geq 1\), and suppose that
  for each \(k\in \mathbb{N}\) and each \(b\in [a_k, 0] \) there exists a
  stable\footnote{Here we mean stable in the sense described just after 
  Definition \ref{DEF_pseudholomorphic_curve}.} unmarked but
  possibly nodal pseudoholomorphic curve
  \begin{equation*}                                                       
    \mathbf{u}_k^b=\big(u_k^b, S_k^b, j_k^b, (-\infty, 1)\times M, J_k,
    \emptyset, D_k^b \big)
    \end{equation*}
  with the following properties.\footnote{By definition, the
  \(\mathbf{u}_k^b\) are boundary-immersed.}
  \begin{enumerate}[(P1)]                                                 
    \item
    the topological space \(|S_k^b|\) is connected, which implies (P\ref{EN_P4}) below,
    \item 
    \(\mathbf{u}_k^b\) is compact and \(u_k^b(\partial S_k^b)\subset
    (0,1)\times M \),
    \item 
    \(\inf_{\zeta\in S_k^b} a\circ u_k^b(\zeta) = b\),
    \item \label{EN_P4}
    there exists a continuous path \(\alpha:[0,1]\to |S_k^b|\) satisfying 
    \begin{equation*}                                                     
      a\circ u_k^b\circ \alpha(0) = b \qquad\text{and}\qquad \alpha(1)\in
      \partial S_k^b
      \end{equation*},
    \item 
    \({\rm Genus}(S_k^b)\leq C\),
    \item 
    \(\int_{S_k^b}(u_k^b)^*\omega \leq C\),
    \item 
    \(\#D_k^b\leq C\),
    \item 
    the number of connected components of \(\partial S_k^b\) is bounded
    above by \(C\).
    \end{enumerate}
Furthermore, suppose that \(J_k\to J\) in \(\mathcal{C}^\infty\), and for
  each fixed \(k\), and each pair \(b, b'\in [a_k, 0]\) with \(b\neq b'\)
  we have
  \begin{equation*}                                                       
    \#\big(u_k^b(S_k^b)\cap u_k^{b'}(S_k^{b'})\big)\leq C.
    \end{equation*}
\emph{Then} there exists a closed set \(\Xi\subset M\)  satisfying
  \(\emptyset \neq \Xi \neq M\) which is invariant under the flow of the
  Hamiltonian vector field \(X_{\eta}\).
\end{theorem}
%
Our final main result regarding feral curves is the above workhorse
  theorem, which perhaps requires some explanation.
First, we must note how much more complicated this result is than the
  finite energy case.
Indeed, in the contact case it is sufficient to know that a finite energy
  curve exists in order to deduce that a periodic orbit exists.
In contrast, we cannot guarantee a similar dynamics result in the case of
  having a feral curve.
The trouble is that although a feral curve does indeed have a notion of a
  limit set, which is both closed and invariant, in general it may be the
  entire framed Hamiltonian manifold \(M\).
That is, the main difficulty is to establish that the limit (or rather, a
  limit) of a feral curve is not all of \(M\), and this is why we need
  Theorem \ref{THM_existence}.
While of course it would be more appealing to have conditions on a single
  feral curve which guaranteed that the associated limit set was not the
  total space, Theorem \ref{THM_existence} is general enough to apply to a
  rather large number of cases in which one constructs feral curves, and
  hence establishes non-minimal dynamics in a good number of cases.
  
As we close out this section, we provide a brief synthesis of the above
  results, in order to illustrate the new class of curves that have been
  found.  
Specifically, we put feral curves in the context of some historical curves
  as well as a more general class which we introduce as \(F\)-dominated
  curves.
As we shall see, these \(F\)-dominated curves have a weak but useful
  notion of a compactness result.
To define them, it will be convenient to make the following preliminary
  definition for a proper map \(u:S\to \mathbb{R}\times M\).
Indeed, for each \(c\in \mathbb{R}\) and \(r>0\), we define
\begin{align*}                                                            
  S_r(c):= u^{-1}\big( \mathcal{I}_r(c)\times
  M\big)\qquad\text{where}\qquad\mathcal{I}_r(c) := [c-r, c+r] \subset
  \mathbb{R}.
  \end{align*}

\begin{definition}[$F$-dominated pseudoholomorphic curves]
  \label{DEF_dominated_curves}
  \hfill \\
Let \((M, \eta)\) be an almost Hermitian manifold\footnote{See Definition
  \ref{DEF_almost_hermitian}.} equipped with an
  \(\eta\)-adapted\footnote{See Definition \ref{DEF_adapted}.} almost
  Hermitian structure.
Let \(\mathbf{u}=(u, S, j, \mathbb{R}\times M, J, \mu, D)\) be a marked
  nodal pseudoholomorphic curve\footnote{See Definition
  \ref{DEF_pseudholomorphic_curve}.}.
Suppose further that \(u:S\to \mathbb{R}\times M\) is proper, and for each
  connected component \(S'\subset S\) on which the restriction \(u:S'\to
  \mathbb{R}\times M\) is constant, we have
  \begin{equation*}                                                       
    \chi(S')- \#\big(S'\cap (\mu \cup D) )  <0.
    \end{equation*}
Let \(F:\mathbb{R}\to [0, \infty)\) be a continuous function. 
We say \(\mathbf{u}\) is \(F\)-dominated, provided there exists an \(c\in
  \mathbb{R}\) such that the following holds for every \(r\in \mathbb{R}\):
  \begin{align*}                                                          
    {\rm Points}\big(S_r(c)\big)+{\rm Genus} \big(S_r(c)\big) +
    {\rm Area}_{u^*g} \big(S_r(c)\big) \leq F(r)
    \end{align*}
  where
  \begin{align*}                                                          
    {\rm Points}\big(S_r(c)\big) := \#\big((\mu \cup D) \cap S_r(c)\big).
    \end{align*}
\end{definition}
%

Geometrically, we note that any slightly reasonable proper
  pseudoholomorphic map \(u:S\to \mathbb{R}\times M\) is \(F\)-dominated for
  some \(F\).
In contrast, if first given an \(F\), one next finds a curve which is in
  fact \(F\)-dominated, then this condition guarantees a certain maximal
  growth rate of the area as one moves away from a reference level (say
  \(\{0\}\times M\)). 
Similarly, we have bounds on the growth rate of genus and special points
  etc.
Next, let us recall that when Gromov introduced pseudoholomorphic
  curves in closed symplectic manifolds, a taming condition guaranteed
  that curves in a fixed homology class had uniformly bounded total area.
Moreover, a uniform \emph{total} area bound was precisely the analytic
  condition needed to prove compactness of a family of curves.
Then, in the symplectization case for SFT, specifically with an
  \(\mathbb{R}\)-invariant Riemannian metric, it turned out that curves in a
  fixed (relative) homology class could develop infinite area, however they
  still had a uniform \emph{local} area bound.
Again, this uniform local area bound was precisely the condition needed to
  prove SFT compactness.
Feral curves then go one step further in this progression, and need not
  even have uniform local area bounds.
Instead they have uniform \emph{connected} local area bounds, as in
  Theorem \ref{THM_local_local_area_bound}, and moreover fit within the
  class of \(F\)-dominated pseudoholomorphic curves, and hence one can
  prove a sort weak one-level SFT compactness theorem, or more
  specifically an \emph{exhaustive} Gromov compactness theorem; see
  Definition \ref{DEF_exhaustive_gromov_convergence} and Theorem
  \ref{THM_exhaustive_gromov_compactness} below.

We can then summarize as follows.
The results in our paper show that feral curves belong to the
  distinguished class of $F$-dominated pseudoholomorphic curves for which a
  version of an exhaustive Gromov compactness theorem exists, see \cite{FH1}.
However, feral curves are somewhat more special:
\begin{enumerate}                                                         
  \item 
  Due to the assumption of finiteness of the $\omega$-energy and a bound
  on genus and the number of ends,   it follows that the behavior of the
  ends of these curves reflects some of the underlying dynamical features
  of the Hamiltonian flow on \(M\).
  \item 
  In general, the uniform connected-local area bound can be obtained from
  topological bounds, which is an important feature in their construction.
  \end{enumerate}
We note that the methods in this paper can be used to establish the
  existence of nontrivial feral curves in quite general contexts.
We now take a moment to describe this process in a fairly general setting,
  thereby establishing truly feral curves.
Recall Proposition \ref{PROP_energy_levels_framed_ham} which associates to
  a smooth compact regular energy surface  $M$ of a Hamiltonian function
  $H$ on a symplectic manifold $(W,\Omega)$ the data $(\lambda,\omega,J)$
  which equips ${\mathbb R}\times M$ with a canonical almost complex
  structure.
For simplicity assume that $M$ is connected and denote the closures of the
  two components of $W\setminus M$ by by $A$ and $B$  so that $A\cap B=M$.
In favorable circumstances, for example a sufficiently rich Gromov-Witten
  theory, one can use a stretching construction around $M$ as described in
  this paper to obtain feral curves in ${\mathbb R}\times M$ with image
  in $[0,\infty)\times M$ or $(-\infty,0]\times M$.
Specifically carrying out this idea in the following  example leads to
  true feral curves which are not of finite energy.
Pick a smooth compact regular hypersurface $M$ in the standard symplectic
  vector space ${\mathbb R}^{2n}$, with $n\geq 3$, which does not admit a
  periodic  orbit.
Such surfaces exist by the results of M. Herman and Ginzburg, see
  \cites{Ginzburg:1995, Ginzburg1997, Herman:1999}.
One can view $M$ as lying in ${\mathbb C}{\mathbb P}^n$ in the complement
  of the divisor at infinity.
Then the fact that there are plenty of complex lines allows one to carry
  out the previously described deformation argument in such a way that the
  obtained curve is feral, and even better has a ${\mathbb R}$-projection
  which  has a minimum.\\

As a final remark, we would like to forewarn the reader that we
  have meticulously kept track of constants in our estimates, and that for
  the first ninety pages it would seem that the sole purpose of this is to
  {\it torture} the reader.
However, in later proofs we use rather sophisticated arguments which only
  work because of our careful book keeping.\\

\noindent{\bf Acknowledgements:}
The first author would like to thank Professors Kai Cieliebak, William
  Minicozzi, and Chris Wendl for a number of helpful conversations.
The first author would also like to thank the Institute for Advanced
  Study, the University of Massachusetts Boston, and the National Science
  Foundation for their generous accommodation and support of this
  research.

\section{Background}\label{SEC_background}

In this section, we will recall some basic notions which will be used
  throughout later sections.
All of these definitions should be either well known or readily absorbed
  by specialists of pseudoholomorphic curves.
We note that there are two notions presented below which may nevertheless
  be unfamiliar to such a reader, namely namely so-called target-local
  Gromov compactness (see Theorem \ref{THM_target_local_gromov_compactness})
  and exhaustive Gromov convergence (see Definition
  \ref{DEF_exhaustive_gromov_convergence}) and compactness (see Theorem
  \ref{THM_exhaustive_gromov_compactness}).
Before progressing to those rather technical concepts, we begin with more
  elementary notions.

\begin{remark}[on smoothness]
  \label{REM_on_smoothness}
  \hfill \\
Throughout this article, when referring to the regularity of
  differentiable objects (functions, forms, manifolds, etc.) the term
  \emph{smooth} will always refer to \(\mathcal{C}^\infty\)-smooth; any less
  regularity, for example \(\mathcal{C}^1\), will be mentioned explicitly.
\end{remark}
%

\subsection{Ambient Geometric Structures}\label{SEC_ambient_structures}

Here we begin by considering geometric structures on certain manifolds
  which will serve as the target space for our later defined
  pseudoholomorphic maps.

\begin{definition}[almost complex manifold]
  \label{DEF_acm}
  \hfill \\
Let \(W\) be a smooth manifold not necessarily closed, possibly with
  boundary, and let \(J\in \Gamma(\rm{End}(TW))\) be a smooth
  section for which \(J\circ J=-\mathds{1}\).
We call \(J\) an almost complex structure for \(W\), and the pair \((W,
  J)\) an almost complex manifold.
\end{definition}
%

\begin{definition}[almost Hermitian manifold]
  \label{DEF_almost_hermitian}
  \hfill \\
Let \(W\) be a smooth finite dimensional manifold equipped with an
  almost complex structure \(J\) and a Riemannian metric \(g\).
We say the pair \((J, g)\) is an almost Hermitian structure on \(W\)
  provided that \(J\) is an isometry for \(g\).
That is, \(g(x,y) = g(Jx,Jy)\) for all \(x,y\in TW\).
\end{definition}
%

We pause for a moment to comment on almost Hermitian manifolds, since they
  may at first seem needlessly tangential to the more natural objects of
  study, namely symplectic manifolds with compatible or tame almost complex
  structures.
First we point out that any almost complex manifold \((W, J)\) can be
  given an almost Hermitian structure \((J, g)\) by choosing an arbitrary
  Riemannian metric \(\tilde{g}\) and defining
  \(g(x,y):=\frac{1}{2}\big(\tilde{g}(x,y)+ \tilde{g}(Jx,Jy)\big).\)
Second, for a symplectic manifold \((W, \Omega)\), an
  \(\Omega\)-compatible almost complex structure \(J\) satisfies, by
  definition, the property that \(g(x,y):= \Omega(x, Jy)\) is a Riemannian
  metric.
For this metric, we immediately see that \(J\) is a \(g\)-isometry.
Third, in the case that \(J\) is \(\Omega\)-tame, we have \(\Omega(x, Jx)
  > 0\) for all \(x\in TW\) with \(x\neq 0\).
An associated Riemannian metric is then given by \(g(x, y) =
  \frac{1}{2}\big(\Omega(x,Jy) + \Omega(y, Jx)\big)\), for which again
  \(J\) a \(g\)-isometry.

At this point, one may still question the utility of moving from analysis
  in symplectic manifolds to almost Hermitian manifolds, and the answer is
  fairly simple: The manifolds that occupy our primary interest are not,
strictly speaking, symplectic.
Moreover, the principle role a symplectic form typically plays is to
  guarantee that pseudoholomorphic curves (defined below) in a fixed
  homology class have uniformly bounded energy, or area; however, this a
  priori bound fails in the almost Hermitian manifolds in which we are
  interested.
Nevertheless, our analysis will require the aid of Riemannian metric for
  which the almost complex structure is an isometry.
This inevitably leads the definition of an almost Hermitian manifold,
  and hence motivates our generalization.

In order to more precisely specify the manifolds of interest, we will
  require two more definitions.

\begin{definition}[framed Hamiltonian structure]
  \label{DEF_framed_Hamiltonian_structure}
  \hfill\\
Let \(M\) be a \(2n+1\) dimensional closed manifold, and let \(\lambda\)
  and \(\omega\) respectively be a smooth one-form and smooth two-form on
  \(M\).
We say \(\eta:=(\lambda, \omega)\) is a Hamiltonian structure for \(M\)
  provided \(d\omega =0\) and \(\lambda\wedge \omega^n\) is a volume form on
  \(M\).
We call \((M, \eta)\) a \emph{framed Hamiltonian manifold}.
We call \((M, \eta)\) an \emph{exact} framed Hamiltonian manifold, and
  \(\eta =(\lambda,\omega)\) an exact Hamiltonian structure, provided there
  is a one-form \(\tau\) on \(M\), for which \(\omega=d\tau\).
\end{definition}
%

Note that in the special case that a framed Hamiltonian structure
  \((\lambda, \omega)\) satisfies the additional condition that \({\rm
  ker}\, \omega\subset {\rm ker}\, d\lambda\), we call \((\lambda,
  \omega)\) a \emph{stable} Hamiltonian structure.
Throughout this article we will not make this additional assumption,
  however it will often be useful to make comparisons between analysis in
  the framed versus stable case.
We also note that there exists a vector field \(X_{\eta}\) associated to a
  Hamiltonian structure \(\eta=(\lambda, \omega)\) uniquely determined by
  the equations
  \begin{equation*}
    \lambda(X_{\eta})\equiv 1\qquad\text{and}\qquad \omega(X_{\eta},
    \cdot) \equiv 0.
    \end{equation*}
We call \(X_{\eta}\) the \emph{Hamiltonian vector field} associated to
  \(\eta\). Observe that by definition of \(X_{\eta}\) and Cartan's formula,
  we have:
  \begin{equation*}
    \mathcal{L}_{X_{\eta}} \omega = d(i_{X_{\eta}} \omega) + i_{X_{\eta}}
    d\omega = 0.
    \end{equation*}

\begin{definition}[$\eta$-adapted almost Hermitian structures]
  \label{DEF_adapted}
  \hfill\\
Given a manifold \(M\) with a framed Hamiltonian structure \(\eta=(\lambda,
  \omega)\), we consider \(\mathbb{R}\times M\); we
  henceforth equip \(\mathbb{R}\) with the coordinate \(a\).
Furthermore, we say the pair \((J, g)\) is an \(\eta\)-adapted almost
  Hermitian structure on \(\mathbb{R}\times M\) provided it satisfies the
  following conditions.
\begin{enumerate}[(J1)]
  \item\label{EN_J1} 
  \(J\) is an \(\mathbb{R}\)-invariant almost complex structure
  \item\label{EN_J2} 
  \(J\partial_a = X_{\eta}\)
  \item\label{EN_J3} 
  \(g=(da\wedge \lambda + \omega)(\cdot,J \cdot)\) is a Riemannian metric.
  \item\label{EN_J4} 
  \(J:{\rm ker}\, \lambda \cap {\rm ker}\, da\to {\rm ker}\, \lambda \cap
  {\rm ker}\, da\)
  \end{enumerate}
  where we have abused notation by writing \(\lambda\) and \(\omega\)
  instead of \({\rm pr}^*\lambda\) and \({\rm pr}^*\omega\) where \({\rm
  pr}: \mathbb{R}\times M\to M\) is the canonical projection.  
We shall refer to ${\mathbb R}\times M$ colloquially as the {\it
  `symplectization'} of $M$ even if this is admittedly not a good name.
\end{definition}
%

We will need to verify that our definition of \(\eta\)-adapted almost
  Hermitian structure is aptly named.
Specifically, we will need to verify that \(J\) is indeed an isometry for
  \(g\).
This will be accomplished momentarily, see Lemma
  \ref{LEM_eta_adapted_are_almost_Hermitian} below, however first we need
  the following.

\begin{lemma}[property of $\eta$-adapted $J$]
  \label{LEM_J_diff}
  \hfill\\
Let \(J\) be an adapted almost complex structure on the symplectization of
  the manifold \(M\) with Hamiltonian structure \((\lambda, \omega)\).
Then 
  \begin{equation*}
  -da\circ J = \lambda \quad\text{and}\quad \omega(Y, JY)\geq
  0\quad\text{for all}\quad Y\in T(\mathbb{R}\times M).
  \end{equation*}
\end{lemma}
%
\begin{proof}
Observe that any tangent vector \(Y\in T(\mathbb{R}\times M)\) can be
  uniquely written as \(Y=c_1 \partial_a + c_2 X_{\eta} + Y_\xi\) with
  \(Y_\xi\in {\rm ker}\, da \cap  {\rm ker}\,\lambda\).
In this case we have
  \begin{equation*}
    -da\circ J (Y) = -da(c_1 J \partial_a + c_2 J X_{\eta} +J Y_\xi) =
    -da( c_1 X_{\eta} -c_2 \partial_a + J Y_\xi) = c_2 = \lambda(Y).
    \end{equation*}

To prove the second part, we compute as follows.
\begin{align*}
  \omega\big(c_1 \partial_a + c_2 X_\eta + Y_\xi, J(c_1 \partial_a + c_2
  X_\eta + Y_\xi)\big)&= \omega(Y_\xi, JY_\xi)
  \\
  &=(da\wedge \lambda + \omega)(Y_\xi, JY_\xi)
  \\
  &=\|Y_\xi\|_g^2
  \\
  &\geq 0.
  \end{align*}
\end{proof}
%

\begin{lemma}[$\eta$-adapted $(J,g)$ are indeed almost Hermitian]
  \label{LEM_eta_adapted_are_almost_Hermitian}
  \hfill\\
Let \((M, \eta)\) be a framed Hamiltonian manifold with \(\eta= (\lambda,
  \omega)\), and let \((J, g)\) be an \(\eta\)-adapted almost Hermitian
  structure on \(\mathbb{R}\times M\) in the sense of Definition
  \ref{DEF_adapted}.
\emph{Then} \((J, g) \) is an almost Hermitian structure for
  \(\mathbb{R}\times M\).
That is, \(J\) is an isometry for \(g\). 
\end{lemma}
%
\begin{proof}
By Definition \ref{DEF_adapted}, we know that \(g = (da\wedge \lambda
  +\omega) (\cdot, J \cdot)\) is a Riemannian metric, so we must show that
  \(g(X,Y) = g(JX, JY)\).
Observe that
\begin{align*}                                                            
  g(X, Y) &= (da\wedge \lambda +\omega) (X, J Y)
  \\
  &=da(X)\lambda(JY) - \lambda(X) da(JY) + \omega(X, JY),
  \end{align*}
  and by Lemma \ref{LEM_J_diff}, we have \(-da(J\cdot) = \lambda(\cdot )\)
  so \(\lambda(J\cdot ) = da(\cdot)\), from which it immediately follows
  that
  \begin{equation}\label{EQ_formula_for_g}                                
    g = da \otimes da + \lambda \otimes \lambda + \omega(\cdot, J\cdot) 
    \end{equation}
  and
  \begin{align*}                                                          
    g(JX, JY) &=da(JX)\lambda(JJY) - \lambda(JX) da(JJY) + \omega(JX, JJY),
    \\
    &=-da(JX)\lambda(Y) + \lambda(JX) da(Y) - \omega(JX, Y),
    \\
    &=   da(Y)\lambda(JX)-\lambda(Y)da(JX) + \omega( Y,  JX),
    \\
    &=g(Y, X)
    \\
    &=g(X, Y)
    \end{align*}
  so indeed, \(J\) is an isometry, and hence \((J, g)\) is an almost
  Hermitian structure.
\end{proof}

We take a moment to give an example of how a framed Hamiltonian structure
  and adapted almost complex structure might arise in practice.
Indeed, we first consider a \((2n+2)\)-dimensional symplectic manifold
  \((W, \Omega)\) equipped with a smooth function \(H:W\to \mathbb{R}\) for
  which \(0\) is a regular value.
One may allow that \(W\) is a manifold with boundary, however in this case
  we require that \(\{H=0\} \cap \partial W = \emptyset\)
Assume further that \(J\) is an almost complex structure on \(W\) for
  which \(\Omega(\cdot, J\cdot)\) is a Riemannian metric.
Define \(M:=H^{-1}(0)\), and consider \(\xi:= TM \cap JTM\) as a subset of
  \(TM\subset TW\).
A bit of linear algebra shows that \(\xi\) is a hyperplane distribution in
  \(TM\), and by construction \(J:\xi\to \xi\).
Next, define the vector field \(X_H\in \Gamma(TM)\) by the following
  \begin{equation*}
    \Omega(X_H, \cdot) = -dH.  
    \end{equation*}
Note that \(X_H\) never vanishes since \(0\) is a regular value of \(H\).
We then define \(\lambda\) on \(M\) to be the unique one-form for which
  \begin{equation*}
    \lambda(X_H)\equiv 1\qquad\text{and}\qquad {\rm ker}\, \lambda = \xi.
    \end{equation*}
Regarding \(M\) as a closed manifold with \(i: M\hookrightarrow W\) the
  canonical inclusion, we define \(\omega\) to be the closed two-form given
  by \(\omega:= i^*\Omega\).
To show that \((\lambda, \omega)\) is a framed Hamiltonian structure for
  \(M\), it is then sufficient to show that \(\lambda\wedge \omega^n\) is a
  volume form on \(M\).
Since \(J:\xi\to \xi\), and \(\Omega(\cdot, J\cdot) \) is a Riemannian
  metric, and since \(\Omega\big|_\xi = \omega\),  it follows that there
  exists a symplectic basis of \(\xi\) of the form \(\{e_1, Je_1, \ldots,
  e_n, J e_n\}\).
However, we then have
  \begin{equation*}
    \lambda\wedge \omega^n (X_H, e_1, J e_1, \ldots, e_n, J e_n) =
    \lambda(X_H)\cdot \omega^n(e_1, J e_1, \ldots, e_n, J e_n)>0;
    \end{equation*}
  here we have made use of the fact that \(\lambda\big|_\xi\equiv 0\), and
  for any \(v\in \xi\) we have \(\omega(X_H, v)=\Omega(X_H, v) = -dH(v)
  =0\) since \(\xi\subset TM\) and \(dH\big|_{TM}\equiv 0\).
Thus \((\lambda, \omega)\) is indeed a framed Hamiltonian structure for
  \(M\).
Using the fact that \(0= -dH\big|_{TM} = \Omega(X_H,
  \cdot)\big|_{TM}=\omega(X_H, \cdot)\) and the definition of \(\lambda\),
  we find that \(X_H = X_{\eta}\).

Continuing our construction, we obtain an adapted almost complex structure
  on \(\mathbb{R}\times M\) in the following way.
We demand \(\mathbb{R}\)-translation invariance, so it is sufficient to
  define \(J\) along \(\{0\}\times M\subset \mathbb{R}\times M\).
Along this hypersurface, we can identify \(\xi\subset TM\) with \({\rm
  ker}\, da \cap {\rm ker}\, \lambda \subset T(\{0\}\times M)\).
Using this identification, \(J: {\rm ker}\, da \cap {\rm ker}\, \lambda
  \to {\rm ker}\, da \cap {\rm ker}\, \lambda\), and we then define
  \(J\partial_a = X_H=X_{\eta}\).

We summarize the previous discussion as follows.

\begin{proposition}[energy levels are framed Hamiltonian]
\label{PROP_energy_levels_framed_ham}
\hfill \\
Consider a symplectic manifold \((W, \Omega)\) equipped with a
  compatible almost complex structure \(J\), and a smooth function \(H:W\to
  \mathbb{R}\) for which \(0\) is a regular value and    \(M:= H^{-1}(0)\)
  is compact and disjoint from \(\partial W\).
Then $M$ naturally carries a framed Hamiltonian structure \(\eta=(\lambda,
  \omega)\) defined by \( {\rm ker}(\lambda) = TM\cap J(TM)\) and
  \(\lambda(X_H)\equiv 1$, where $\omega$ is the pull-back of $\Omega$ to
  $M$.
With the restriction of \(J\)  to \(TM\cap J(TM)\)  denoted again by \(J\), 
  we obtain \((\lambda,\omega, J)\) which defines a natural
  \(\eta\)-adapted almost Hermitian structure \((J,g)\) on the
  symplectization \(\mathbb{R}\times M\).
\end{proposition}
%

We now provide several more general target manifolds with adapted
  structures, the first of which we call a \emph{realized Hamiltonian
  homotopy}, and which is made precise in Definition
  \ref{DEF_hamiltonian_homotopy} below.
For clarity, we first provide the following motivation for the definition.
Consider a closed symplectic manifold \((W, \Omega)\) with compatible
  \(J\) as in Proposition \ref{PROP_energy_levels_framed_ham}, and consider
  a smooth Hamiltonian \(H:W\to \mathbb{R}\) for which \(0\) is a regular
  value.
Let us define \(M:= H^{-1}(0)\), and observe that \(M\) is diffeomorphic
  to \(\{H=t\}\) for all \(t\) sufficiently close to \(0\), and an
  explicit diffeomorphism is obtained by the flow of the vector field
  \(Y:= \frac {\nabla H}{\|\nabla H\|^{2}}\). 
We denote this diffeomorphism by \(\psi_t: M\to \{H=t\}\).
As a consequence of Proposition \ref{PROP_energy_levels_framed_ham}, this
  gives rise to a family of framed Hamiltonian structures \(\eta_t:=
  (\lambda_t, \omega_t)\) on \(M\) with \(\omega_t = \psi_t^*\Omega\).
Moreover, for all \(t\) and \(t_0\) sufficiently close to \(0\) we in fact
  have that \((\lambda_{t_0}, \omega_t)\) is a framed Hamiltonian
  structure on \(M\).
With this in mind, we then consider \(\mathcal{I} \subset \mathbb{R}\) to
  be any open interval, and we let \(f:\mathcal{I}\to \mathbb{R}\) be any
  smooth function mapping into a neighborhood of \(0\) for which \(f'\geq
  0\).
We can then equip \(\mathcal{I}\times M\) with the one-form
  \(\hat{\lambda}:= \pi^* \lambda_{t_0}\) and with the two-form
  \(\hat{\omega}= \Psi^* \Omega \) where \(\Psi(t, p) = \psi_{f(t)}(p)\).
In this way, a homotopy of two-forms \(t\mapsto \omega_t\) on \(M\)
  arising from framed Hamiltonian structures gives rise to a single two-form
  \(\hat{\omega}\) on \(\mathcal{I}\times M\), and somewhat similarly for
  \(\hat{\lambda}\).
It is for this reason that we call \((\mathcal{I}\times M, (\hat{\lambda},
  \hat{\omega}))\) a ``realized Hamiltonian homotopy.''
We make this idea both more precise and more general with the following
  definition.

\begin{definition}[realized Hamiltonian homotopy]
  \label{DEF_hamiltonian_homotopy}
  \hfill\\
Let \(M\) be a smooth (odd-dimensional) closed manifold, let
  \(\mathcal{I}\subset \mathbb{R}\) be an interval equipped with the
  coordinate \(t\), and let \(\hat{\lambda}\) and \(\hat{\omega}\)
  respectively be a one-form and two-form on \(\mathcal{I}\times M\).
We say \((\mathcal{I}\times M, (\hat{\lambda}, \hat{\omega}))\) is a
  realized Hamiltonian homotopy provided the following hold.
\begin{enumerate}                                                        
  \item \(\hat{\lambda}(\partial_t)=0 \).
  \item \(i_{\partial_t} \hat{\omega}=0\).
  \item \(d \hat{\omega}\big|_{\{t={\rm const}\}} = 0\)
  \item \(dt\wedge \hat{\lambda}\wedge \hat{\omega}\wedge \cdots \wedge
    \hat{\omega} >0\).
  \item \(\hat{\lambda}\) is invariant under the flow of \(\partial_t\)
  \item if \(\mathcal{I}\) is unbounded, then there exists a neighborhood
    of \(\{\pm\infty\}\times M\) on  which
    \(\hat{\omega}\) is invariant under the flow of \(\partial_t\).
  \end{enumerate}
\end{definition}
%
We note that the properties of a realized Hamiltonian homotopy imply that,
  near $\{\pm\infty\}\times M$, $\hat{\lambda}$ and $\hat{\omega}$ are
  pull-backs of forms $\lambda$ and $\omega$ on $M$, where in addition
  $\omega$ is closed and further $\lambda\wedge\omega^n>0$ holds with
  $2n+1=\text{dim}(M)$.
Observe that a realized Hamiltonian homotopy gives rise to two additional
  structures, the first of which is a vector field \(\widehat{X}\) on
  \(\mathcal{I}\times M\) which is uniquely determined by the equations
  \begin{align*}                                                          
    dt(\widehat{X})=0, \qquad \hat{\lambda}(\widehat{X}) = 1, \qquad
    i_{\widehat{X}}\hat{\omega} = 0.
    \end{align*}
The second structure is the codimension-two plane field distribution given
  by
  \begin{align*}                                                          
    \hat{\xi} = {\rm ker}\; dt \cap {\rm ker}\; \hat{\lambda}.
    \end{align*}
With this in mind, we now provide the notion of an almost Hermitian
  structure adapted to a realized Hamiltonian homotopy.

\begin{definition}[adapted structures for a realized Hamiltonian
  homotopy]
  \label{DEF_adapted_structures_Ham_homotopy}
  \hfill\\
Let \((\mathcal{I}\times M , (\hat{\lambda}, \hat{\omega}))\) be a
  realized Hamiltonian homotopy.
We say an almost Hermitian structure \((\widehat{J}, \hat{g})\) on
  \(\mathcal{I}\times M\) is adapted to this realized Hamiltonian homotopy
  provided the following hold.
\begin{enumerate}                                                         
  \item \(\widehat{J}\partial_t = \widehat{X}\).
  \item \(\widehat{J}\colon \hat{\xi}\to \hat{\xi}\).
  \item \(\hat{g} = (dt \wedge \hat{\lambda} + \hat{\omega})(\cdot,
    \widehat{J} \cdot) \).
  \item if \(\mathcal{I}\) is unbounded, then there exists a neighborhood
    of \(\{\pm\infty\}\times M\) on which the restriction
    \(\widehat{J}\big|_{\hat{\xi}}\) is invariant under
    the flow of \(\partial_t\).
  \end{enumerate}
\end{definition}
%

\begin{definition}[symplectic cobordism]
  \label{DEF_symplectic_cobordism}
  \hfill\\
Let \(M^+\) and \(M^-\) be smooth closed manifolds, and let \(\eta^\pm =
  (\lambda^\pm, \omega^\pm)\) denote framed Hamiltonian structures for each.
Suppose that each \(M^\pm\) is oriented by the volume form 
  \begin{align*}                                                          
    \lambda^\pm \wedge \omega^\pm \wedge \cdots \wedge \omega^\pm.
    \end{align*}
We say a compact symplectic manifold \((\widetilde{W}, \tilde{\omega})\)
  is a symplectic cobordism from
  \((M^+, \eta^+)\) to \((M^-, \eta^-)\) provided 
  \begin{enumerate}                                                       
    \item \(\partial \widetilde{W} = M^+ - M^-\)
    \item \(i_\pm^* \tilde{\omega} = \omega^\pm\), where \(i_\pm: M^\pm
      \to \widetilde{W}\) is the canonical inclusion.
    \end{enumerate}
\end{definition}
%
The basic concept we would like to introduce next is that of an
  extended symplectic cobordism and it is quite involved,
  however the overall notion should be mostly familiar to
  those experienced with symplectic manifolds with cylindrical ends.
The complexity of the definition arises in large part because the ends are
  not (symplectizations of) contact or stable Hamiltonian manifolds, and
  hence later when we derive new energy/area estimates for
  pseudoholomorphic curves we must rely on a very carefully arranged
  structure in the target manifold, specifically on the region which
  transitions from symplectic part to the cylindrically ended part.
To help digest these notions, 
  we break the definition apart and introduce auxiliary
  objects in several preliminary definitions.
In the first definition we disregard symplectic and almost complex
  considerations and target the notion of an
  \emph{extended} cobordism in general.

\begin{figure}
\label{FIG2}
\includegraphics[scale=0.3]{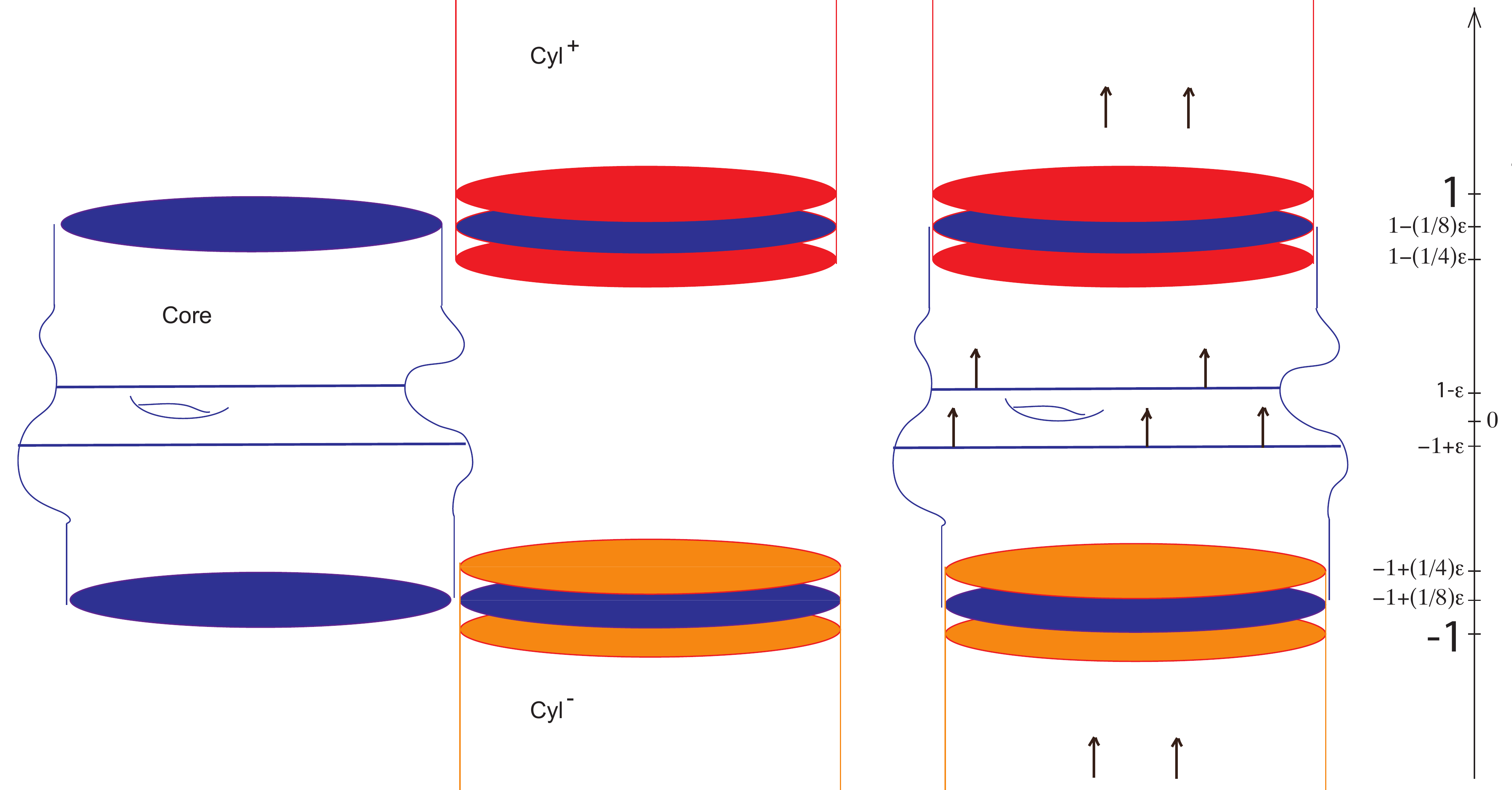}
\caption{
  The figure shows the core, \({\rm Cyl}^+\) and \({\rm Cyl}^-\)
  regions and indicates the region where
  \(d\bar{a}(\partial_{\bar{a}})=1\).
  }
\end{figure}
%

\begin{definition}[extended cobordism]
  \label{DEF_extended_cobordism}
  \hfill\\
Let \(M^\pm\) be two closed oriented manifolds. 
An extended cobordism from \(M^+\) to \(M^-\) is given by a tuple
  \((\bar{W},\bar{a},\partial_{\bar{a}},\varepsilon)\) and embeddings
  \(\phi^\pm:M^\pm\rightarrow \bar{W}\) with the following properties.
\begin{enumerate}                                                         
  \item 
  \(\overline{W}\) is a smooth oriented manifold.
  \item 
  \(\bar{a}:\overline{W}\rightarrow {\mathbb R}\) is a proper smooth
  function.
  \item 
  \(\partial_{\bar{a}}\in\Gamma(T\overline{W})\) is a smooth complete
  vector field.
  \item 
  \(\varepsilon>0\).
  \end{enumerate}
This data has the following additional properties:
\begin{enumerate}                                                         
  \item 
  \(d\bar{a}(\partial_{\bar{a}})=1\) on the domain \(\{|\bar{a}|>
  1-\varepsilon\}\).
  \item 
  The images of the embeddings \(\phi^\pm\) are the codimension-one
  submanifolds \(\bar{a}^{-1}(\pm 1)\).
  \item 
  The orientation induced on \(\bar{a}^{-1}( 1)\cup \bar{a}^{-1}(-1)\) as
    the boundary of the subdomain  with smooth boundary
    \(\bar{a}^{-1}([-1,1])\) of \(\overline{W}\) has the property that
    \(\phi^+:M^+\rightarrow \bar{a}^{-1}(1)\) is orientation-preserving and
    \(\phi^-:M^-\rightarrow \bar{a}^{-1}(-1)\) is orientation-reversing.
  \end{enumerate}
The compact sub-domain with smooth boundary defined by 
\begin{align}\label{EQ_core_W}                                           
  {\rm Core}(\overline{W}):=\big\{|\bar{a}|\leq
  1-\textstyle{\frac{1}{8}}\varepsilon\big\}
  \end{align}
  is called  the  ``core of \(\overline{W}\).'' 
The following domains are called the ``positive cylinder'' and ``negative
  cylinder'' respectively.
\begin{align}\label{EQ_cyl_W}                                            
  {\rm Cyl}^\pm(\overline{W}):=
  \big\{\pm\bar{a}>1-\textstyle{\frac{1}{4}}\varepsilon\big\}.
  \end{align}
\end{definition}
%

We illustrate the above definition by Figure \ref{FIG2}. 
We also note that in the case that \(\varepsilon> 1\) it holds that
  \(d\bar{a}(\partial_{\bar{a}})=1\) on all of \(\overline{W}\), and hence
  in this case \(\overline{W}\) is diffeomorphic to the oriented product
  \({\mathbb R}\times M^+\) via \((s,m)\rightarrow s\cdot \phi^+(m)\),
  where \((s,w)\rightarrow s\cdot w\) is notation for the flow associated
  to \(\partial_{\bar{a}}\).

Definition \ref{DEF_extended_cobordism} above has introduced the ``smooth
  aspects'' of the target manifolds which we shall need for further
  analysis, and so our next step is to introduce the additional
  symplectic, Hermitian, and framed Hamiltonian features on such
  manifolds.
In this case we shall start with \((M^\pm,\eta^\pm)\), where
  \(\eta^\pm=(\lambda^\pm,\omega^\pm)\) are framed Hamiltonian structures,
  and we will define an extended symplectic cobordism between \((M^+,
  \eta^+)\) and \((M^-, \eta^-)\) which carries additional compatible
  data.
For this new extension, with underlying extended cobordism
  \((\overline{W},\bar{a},\partial_{\bar{a}},\varepsilon)\), it will be
  important to identify certain subdomains distinguished by
  ranges of values of the function \(\bar{a}\).
While the explicit definition of these subdomains will seem pedantic, it
  is important to realize that they are necessary for the delicate analysis
  performed later.
In order to aid comprehension of these technical elements, we will
  elaborate on some of the more important features after providing the
  definition.

As a final point, we mention that our notion of an extended symplectic
  cobordism actually relies on not just a symplectic structure, but also
  framed Hamiltonian structures and an adapted almost Hermitian structure.
  As such, it would be more accurate to give it a more cumbersome name like
  ``extended almost K\"{a}hler cobordism,'' however here and throughout we
  opt for that brevity which emphasizes the most important structure.

\begin{definition}[extended symplectic cobordism]
  \label{DEF_symplectic_extended_cobordism}
  \hfill\\
Consider a pair of framed Hamiltonian structures \((M^\pm, \eta^\pm)\),
  with \(\eta^\pm = (\lambda^\pm, \omega^\pm)\) and view \(M^\pm\)
  equipped with the induced orientations.
Consider also the tuple \(\mathbf{W}=(\overline{W}, \bar{\omega},
  \overline{J}, \bar{g}, \bar{a}, \partial_{\bar{a}}, \epsilon)\) and
  embeddings \(\phi^\pm:M^\pm \to \overline{W}\) consisting of the following
  data.
\begin{enumerate}                                                         
  \item 
  \((\overline{W},\bar{a},\partial_{\bar{a}},\varepsilon)\) together with
  the $\phi^\pm$ is an extended cobordism between the oriented manifolds
  $M^+$ and $M^-$.
  \item 
  \(\bar{\omega}\) is a smooth closed two-form on \(\overline{W}\).
  \item 
  \((\overline{J}, \bar{g})\) is an almost Hermitian structure on
  \(\overline{W}\).
  \end{enumerate}
This data has the following additional properties. There exists 
  a smooth function \(\beta\colon \mathbb{R}\to
  \mathbb{R}\) satisfying 
\begin{align*}                                                            
  \beta(\bar{a}) = 0\; \; &\text{on}\; \; \{|\bar{a}|\geq 1\}\\
  \beta'(\bar{a})> 0 \; \; &\text{on}\; \;\{1-\epsilon <|\bar{a}| < 1\}
  \end{align*}
 such that the following hold:
\begin{enumerate}                                                         
  \item on the region \(\{|\bar{a}| < 1\}\) we have
  \begin{enumerate}                                                       
    \item 
    \(\bar{\omega}\) is non-degenerate.
    \item 
    \(\overline{J}\) is \(\bar{\omega}\)-compatible; that is,
    \(\bar{\omega}(\cdot , \overline{J}\cdot)\) is a Riemannian metric.
    \end{enumerate}
  \item 
  on the region \(\{|\bar{a}| < 1- \frac{1}{2}\epsilon \}\) we have
  \(\bar{g} = \bar{\omega}(\cdot , \overline{J} \cdot )\)
  \item 
  on the region \(\{|\bar{a}| > 1-\epsilon\}\), where we recall that
  \(d\bar{a}(\partial_{\bar{a}}) = 1\), we have 
  \begin{enumerate}                                                       
    \item 
    \(\bar{\omega} = \omega^\pm + d(\beta \lambda^\pm) =
    \Big(\beta'(d\bar{a}\wedge \lambda^\pm)\Big)+\Big(\omega^\pm + \beta
    d\lambda^\pm\Big)\)  where \(\beta = \beta(\bar{a})\) and we have
    abused notation by writing \(\lambda^\pm\) and \(\omega^\pm\) rather
    than the more accurate \((\phi_\pm^{-1} \circ {\rm pr}_\pm)^*
    \lambda^\pm \) and \((\phi_\pm^{-1} \circ {\rm pr}_\pm)^* \omega^\pm
    \) where
    \begin{align*}                                                        
      {\rm pr}_\pm \colon \{\pm \bar{a} > 1-\epsilon\} \to
      \{\bar{a}=\pm1\}
      \end{align*}
    is the smooth projection along the trajectories of
    \(\partial_{\bar{a}}\).  
    \item 
    Make the following definitions:
    \begin{enumerate}                                                     
      \item 
      the two form \(\hat{\omega}^\pm := \omega^\pm + \beta d\lambda^\pm\)
      \item 
      the co-dimension two plane field \(\xi:= {\rm ker}\; d\bar{a} \cap
      {\rm ker}\; \lambda^\pm\)
      \item 
      smooth vector field \(\overline{X}\) determined by 
      \begin{align*}                                                      
	d\bar{a}(\overline{X}) = 0\qquad \lambda^\pm(\overline{X}) =
	1\qquad {\rm ker}\; \hat{\omega}^\pm = {\rm
	Span}(\partial_{\bar{a}}, \overline{X});
        \end{align*}
      \end{enumerate}
    then \(T\overline{W} = \xi \oplus {\rm ker}\; \hat{\omega}^\pm\), and
    \(\overline{J}\) preserves this splitting; moreover we have
    \(\overline{J} \partial_{\bar{a}} = \overline{X}\), and
    \(\hat{\omega}^\pm(\cdot , \overline{J}\cdot)\big|_{\xi}\) is a bundle
    metric
    \item
    there exists a smooth positive function \(\theta\colon \mathbb{R}\to
    \mathbb{R}\) satisfying \(\theta(\bar{a}) = 1\) for
    \(|\bar{a}|>1-\frac{1}{4}\epsilon\) for which
    \begin{align*}                                                        
      \bar{g} = \big(\theta(\bar{a}) (d\bar{a}\wedge \lambda^\pm )+
      \hat{\omega}^\pm\big)(\cdot, \overline{J}\cdot)
      \end{align*}
    \end{enumerate}
  \item 
  on the region \(\{|\bar{a}| \geq 1\}\) we have \(\overline{J}\) is
  invariant under the flow of \(\partial_{\bar{a}}\)
  \end{enumerate}
Then we call the tuple \((\overline{W}, \bar{\omega}, \overline{J},
  \bar{g}, \bar{a}, \partial_{\bar{a}}, \epsilon)\) an extended symplectic
  cobordism from \((M^+, \eta^+)\) to \((M^-, \eta^-)\) equipped with
  adapted almost Hermitian structure.
\end{definition}
%

We need a quick definition before we provide some comments on the previous definition.                                               
\begin{definition}[compact region]
  \label{DEF_compact_region}
  \hfill \\
Let \(W\) be a manifold.  
Suppose \(\mathcal{U}\subset W\) is an open set for which its closure
  \({\rm cl}(\mathcal{U})\) inherits from \(W\) the structure of a
  smooth compact manifold possibly with boundary.
Then we call \({\rm cl}(\mathcal{U})\) a \emph{compact region} in \(W\).
\end{definition}
%

\begin{remark}[structural observations]
  \label{REM_structureal_observations}
  \hfill\\
In light of the careful definition of an extended symplectic cobordism
  denoted by \(\mathbf{W}=(\overline{W}, \bar{\omega}, \overline{J},
  \bar{g}, \bar{a}, \partial_{\bar{a}}, \epsilon)\), it is possible to
  identify a number of structures which we will make use of later.
The first of these is the set \({\rm Core}(\overline{W})\), as defined in
  equation (\ref{EQ_core_W}), and which also is a compact region in the
  sense of Definition \ref{DEF_compact_region}.
Likewise, we also make use of the open regions \({\rm
  Cyl}^\pm(\overline{W})\), as defined in equation (\ref{EQ_cyl_W}).
We also specifically note that 
\begin{align*}                                                            
  \overline{W} = {\rm Cyl}^-(\overline{W}) \cup {\rm Core}(\overline{W})
  \cup {\rm Cyl}^+(\overline{W}).
  \end{align*}
Second, we note there exist diffeomorphisms 
  \begin{align*}                                                          
    &\Phi^\pm\colon \mathcal{I}^\pm\times M^\pm \to {\rm
    Cyl}^\pm(\overline{W})
    \\
    &\Phi^\pm(t,p) = \varphi_{\partial_{\bar{a}}}^{t\mp
    1}\big(\phi_\pm(p)\big)
    \end{align*}
    where 
  \begin{align*}                                                          
    \mathcal{I}^+ = \big(1-{\textstyle \frac{1}{4}}\epsilon,
    \infty\big)\qquad\text{and}\qquad \mathcal{I}^- = \big(-\infty,
    -1+{\textstyle\frac{1}{4}\epsilon}\big),
    \end{align*}
    the \(\phi_\pm\colon M^\pm \to \{\bar{a} = \pm 1\}\subset
    \overline{W}\) are as in Definition
    \ref{DEF_symplectic_extended_cobordism}, and
    \(\varphi_{\partial_{\bar{a}}}^t\) is the time \(t\) flow of the
    vector field \(\partial_{\bar{a}}\).

As a consequence of these structures, we see that the manifolds
  \(\mathcal{I}^\pm\times M^\pm\) equipped with the pair of differential
  forms \((\Phi_\pm^* \lambda^\pm, \Phi_\pm^* \hat{\omega}^\pm)\) are each
  a realized Hamiltonian  homotopy in the sense of Definition
  \ref{DEF_hamiltonian_homotopy}; here recall that
  \begin{align*}                                                          
    \hat{\omega}^\pm = \omega^\pm +\beta d\lambda^\pm
    \qquad\text{and}\qquad \bar{\omega} = \omega^\pm + d(\beta\lambda^\pm)
    \end{align*}
  for \(\beta:\mathbb{R}\to \mathbb{R}\) satisfying \(\beta(\bar{a}) = 0\)
  on \(\{|\bar{a}| \geq 1-\frac{1}{4}\epsilon \}\) and \(\beta'(\bar{a})
  >0 \) on the region \(\{1-\epsilon < |\bar{a}| < 1\}\).
Moreover, the pair \((\Phi_\pm^*\bar{g}, \Phi_\pm^* \overline{J})\) is an
  adapted structure for the realized Hamiltonian homotopy, in the sense of
  Definition \ref{DEF_adapted_structures_Ham_homotopy}.

Finally, regarding \({\rm Core}(\overline{W})\), we note that there exists
  a positive constant, denoted \(C_\theta =
  C_\theta(\mathbf{W})\), for which
  \begin{align}\label{EQ_omega_coercive}                                  
    C_\theta^{-1} 
    \leq 
    \inf_{\bar{q}\in {\rm Core}(\overline{W})} \inf_{\substack{v\in
    T_{\bar{q}} \overline{W}\\ \|v\|_{\bar{g}}=1}} \bar{\omega}(v, Jv)
    \leq
    \sup_{\bar{q}\in {\rm Core}(\overline{W})} \sup_{\substack{v\in
    T_{\bar{q}} \overline{W}\\ \|v\|_{\bar{g}}=1}} \bar{\omega}(v, Jv)
    \leq 
    C_\theta.
    \end{align}
\end{remark}
%

The next lemma will give a means to construct an extended symplectic
  cobordism \(\boldsymbol{W}=(\overline{W}, \bar{\omega}, \overline{J},
  \bar{g}, \bar{a}, \partial_{\bar{a}},\epsilon)\) from a symplectic
  cobordism \(\widetilde{W}\) from \(M^+\) to \(M^-\) provided the
  \(M^\pm\) are equipped with framed Hamiltonian structures.
The proof of the lemma will also make the content and utility of
  Definition \ref{DEF_symplectic_extended_cobordism} more transparent.

\begin{lemma}[cobordism to extended cobordism]
  \label{LEM_cobordism_to_extended_corbordism}
  \hfill\\
Let \((M^\pm, \eta^\pm)\) be a pair of framed Hamiltonian manifolds with
  \(\eta^\pm = (\lambda^\pm, \omega^\pm)\).
Let \((\widetilde{W}, \tilde{\omega})\) be a symplectic cobordism from
  \((M^+, \eta^+)\) to \((M^-, \eta^-)\).
Then there exists an extended symplectic cobordism from \((M^+,
  \eta^+)\) to \((M^-, \eta^-)\) equipped with an adapted almost
  Hermitian structure, which we denote \(\mathbf{W}=(\overline{W},
  \bar{\omega}, \overline{J}, \bar{g}, \bar{a}, \partial_{\bar{a}},
  \epsilon)\).
Moreover, there exists a smooth surjection \(\Psi\colon \overline{W}\to
  \widetilde{W}\) such that the following hold.
  \begin{enumerate}                                                       
    \item 
    the domain restricted map \(\Psi\colon \{|\bar{a}| < 1\}\to
    \widetilde{W}\setminus \partial \widetilde{W}\) is a diffeomorphism
    \item 
    \(\Psi( \{|\bar{a}|\geq 1\}) = \partial \widetilde{W}\)
    \item \(\bar{\omega} = \Psi^* \widetilde{\omega}\)
    \end{enumerate}
\end{lemma}
%
\begin{proof}
We begin by noting that a version of Darboux's theorem guarantees that
  there exists an \(\epsilon>0\), disjoint neighborhoods
  \(\widetilde{\mathcal{O}}^\pm\) of the \(M^\pm\subset \partial
  \widetilde{W}\) in $\widetilde{W}$, and diffeomorphisms 
  \begin{align*}                                                          
    \psi_+\colon(-\epsilon, 0]\times M^+\to  \mathcal{O}^+
    \qquad\text{and}\qquad   
    \psi_-\colon [0, \epsilon)\times M^- \to \mathcal{O}^-
    \end{align*}
  for which \( \psi_\pm^* \tilde{\omega} = \omega^\pm + d(t\lambda^\pm) \)
  where \(t\) is the coordinate on \((-\epsilon, 0]\) and \([0,
  \epsilon)\).
This is the desired positive number \(\epsilon>0\).

Next, for notational convenience, we define the following intervals.
  \begin{align*}                                                          
    &\mathcal{I}^+:= (1-\epsilon, \infty) 
    &&\mathcal{I}^-:= (-\infty, \epsilon-1)  
    \\
    &\check{\mathcal{I}}^+:= (1-\textstyle{\frac{1}{4}}\epsilon, \infty)
    &&\check{\mathcal{I}}^-:=
      (-\infty,\textstyle{\frac{1}{4}} \epsilon -1)
    \end{align*}
We can then define the desired manifold \(\overline{W}\) as follows.
\begin{align}
  \overline{W} = \Big(\big( \mathcal{I}^- \times M^-\big)\; \; \sqcup\;\;
  \widetilde{W} \; \;
  \sqcup \; \; \big(\mathcal{I}^+\times  M^+\big)
  \Big)/\sim
  \end{align}
  where \(\tilde{q} \sim (t, p)\) provided one of the following two holds:
  \begin{enumerate}                                                       
    \item \(\tilde{q}\in \widetilde{W}\)
      and \((t, p)\in (1-\epsilon, 1]\times M^+\) and
      \(\psi_+(t-1, p)=\tilde{q}\)
    \item \(\tilde{q}\in \widetilde{W}\)
      and \((t, p)\in [ -1, \epsilon-1)\times M^-\) and
      \(\psi_-(t+1, p)=\tilde{q}\)
    \end{enumerate}
This defines the desired manifold \(\overline{W}\).
In light of the definition of \(\overline{W}\), we then let  \(I_0\),
  \(I_-\), and \(I_+\) denote the natural embeddings:
  \begin{align*}                                                          
    I_\pm\colon \mathcal{I}^\pm\times M^\pm \hookrightarrow
    \overline{W}\qquad\text{and}\qquad I_0\colon \widetilde{W}
    \hookrightarrow \overline{W}.
    \end{align*}
In what follows, it will be convenient to have the following sets
  established.
For each \(\delta \in [0, \epsilon]\) we define the positive and negative
  end regions
  \begin{align*}                                                          
    \overline{W}_\delta^+ = I_+\big((1-\delta, \infty)\times M^+\big) 
    \qquad\text{and}\qquad
    \overline{W}_\delta^- = I_-\big(( -\infty, \delta-1)\times M^-\big) 
  \end{align*}
With these established, we now define a smooth function \(\bar{a}\colon
  \overline{W}\to \mathbb{R}\)  satisfying
  \begin{align*}                                                          
    \bar{a}(\bar{q})= 
    \begin{cases}
    {\rm pr}_1\circ I_+^{-1}(\bar{q}) & \text{ if } \bar{q}\in
    \overline{W}_\epsilon^+
    \\
    {\rm pr}_1\circ I_-^{-1}(\bar{q}) & \text{ if } \bar{q}\in
    \overline{W}_\epsilon^-
    \end{cases}
    \end{align*}
  where \({\rm pr}_1\) is the canonical projection to the first factor,
  and
  \begin{align*}                                                          
    |\bar{a}(\bar{q})|\leq 1-\epsilon \qquad\text{for all }\; \bar{q}\in
    \overline{W}\setminus \big(\overline{W}_\epsilon^- \cup
    \overline{W}_\epsilon^+\big).
    \end{align*}
This defines the desired function \(\bar{a}\).
Letting \(t\) denote the coordinate on \(\mathcal{I}^\pm\) as appropriate,
  we define \(\partial_{a}:= (I_\pm)_*\partial_t\), and smoothly extend it
  on the rest of \(\overline{W}\).
This defines the smooth vector field \(\partial_{\bar{a}}\), and also
  establishes property (3a).
It also allows us to define the desired diffeomorphisms:
  \begin{align*}                                                          
    &\phi_\pm \colon M^\pm \to \{\bar{a}=\pm 1\}
    \\
    &\phi_\pm(p) = \psi_\pm(0,p).
    \end{align*}
It is perhaps worth noting that we have
\begin{align}\label{EQ_tilde_w}                                           
  \tilde{\omega} =
  \begin{cases}
  \omega^+ + d\bar{a} \wedge \lambda^+ + (\bar{a}-1)d\lambda^+ &\text{ on
  }\{ 1 - \epsilon < \bar{a} \leq 1\}
  \\
  \omega^- + d\bar{a} \wedge \lambda^- + (\bar{a}+1)d\lambda^- &\text{ on
  }\{ -1 \leq  \bar{a} < \epsilon- 1\}
  \end{cases}
  \end{align}

Next we aim to define the two-form \(\bar{\omega}\). 
To that end, we first define the desired smooth function \(\beta\colon
  \mathbb{R}\to \mathbb{R}\) which satisfies the following conditions:
  \begin{enumerate}                                                       
    \item 
    \(\beta(\bar{a}) = 0\) for \(|\bar{a}| \geq 1\)
    \item 
    \(\beta'(\bar{a}) > 0\) for \(1-\frac{1}{2}\epsilon < |\bar{a}| < 1\)
    \item 
    \(\beta(\bar{a}) = \bar{a} -1\) for \(1- \epsilon < \bar{a} <
    1-\frac{1}{2}\epsilon\)
    \item 
    \(\beta(\bar{a}) = \bar{a} +1\) for \(\frac{1}{2}\epsilon - 1 <
    \bar{a} < \epsilon - 1\).
    \end{enumerate}
With this function defined, we can then define \(\bar{\omega}\) on
  \(\{1-\epsilon <|\bar{a}|\}\) by the following:
  \begin{align*}                                                          
    \bar{\omega} 
    &= 
    \omega^\pm + d\big(\beta(\bar{a})\lambda^\pm\big)
    \\
    &= 
    \omega^\pm +\beta'(\bar{a})d\bar{a} \wedge \lambda^\pm +
    \beta(\bar{a})d\lambda^\pm
    \\
    &= 
    \beta'(\bar{a}) d\bar{a} \wedge \lambda^\pm + \hat{\omega}^\pm
    \end{align*}
  for 
  \begin{align}\label{DEF_hat_w}                                          
    \hat{\omega}^\pm = \omega^\pm + \beta(\bar{a})d\lambda^\pm.
    \end{align}

As a consequence of the definition of \(\beta\) together with
  \(\tilde{\omega}\) expressed as in equation (\ref{EQ_tilde_w}), we see
  that on \(\{1-\epsilon < |\bar{a}| < 1- \frac{1}{2}\epsilon\}\) we have
  \(\tilde{\omega} = \bar{\omega}\), and hence we extend the definition of
  \(\bar{\omega}\) so that \(\bar{\omega}\big|_{\{|\bar{a}| \leq 1 -
  \epsilon\}}:= \widetilde{\omega}\), which yields a smooth two-form
  \(\bar{\omega}\) defined on all of \(\overline{W}\).
This is the desired \(\bar{\omega}\) and establishes property (3b).
It also immediately follows that on \(\{|\bar{a}| < 1\}\) the two-form
  \(\bar{\omega}\) is non-degenerate which establishes property (1a).
This latter fact is established by observing that
  \begin{align*}                                                          
    &\beta\times Id\colon (1-\epsilon, 1)\times M^+ \to (-\epsilon,
    0)\times M^+
    \\
    &\beta\times Id\colon (-1, \epsilon-1)\times M^- \to (0,
    \epsilon)\times M^-
    \end{align*}
  are diffeomorphisms which pull back \(\tilde{\omega}\) to
  \(\bar{\omega}\).

At this point, we are prepared to define the smooth map \(\Psi\colon
  \overline{W}\to \widetilde{W}\).
As a first step, we define the smooth map
  \begin{align*}                                                          
    &\Psi \colon \overline{W}_\epsilon^+\cup \overline{W}_\epsilon^- \to
    \widetilde{W}
    \\
    &\Psi =  \psi_\pm \circ (\beta\times Id) \circ I_\pm^{-1}
    \end{align*}
We then observe that on the collar regions
  \(\overline{W}_\epsilon^\pm\setminus
  \overline{W}_{\frac{3}{4}\epsilon}^\pm\) we have \(I_0\circ \Psi = Id\),
  so the following extension yields the desired smooth surjection.
  \begin{align*}                                                          
    &\Psi \colon \overline{W}\to \widetilde{W}
    \\
    &\Psi\colon = 
    \begin{cases}
      \psi_\pm \circ (\beta\times Id) \circ I_\pm^{-1}(\bar{q}) &\text{ if
      }|\bar{a}(\bar{q})| > 1-\epsilon
      \\ 
      I_0^{-1}(\bar{q}) &\text{ if }|\bar{a}(\bar{q})| < 1 -
      \frac{3}{4}\epsilon.
    \end{cases}
    \end{align*}
From this, and our above constructions, we can immediately see that
  \(\Psi\) satisfies the desired properties, including \(\Psi^*\tilde{w} =
  \bar{\omega}\).

We now turn our attention to defining the almost complex structure
  \(\overline{J}\) on \(\overline{W}\).
To that end, we first observe that on the ends
  \(\overline{W}_\epsilon^\pm\) we have the splitting
  \begin{align}\label{EQ_split_ends}                                      
    T\overline{W} = {\rm ker}\; (d\bar{a}\wedge \lambda^\pm) \oplus {\rm
    ker}\; \hat{\omega}^\pm
    \end{align}
  where \(\hat{\omega}^\pm\) is defined in equation (\ref{DEF_hat_w}), and
  we observe that by construction we have \(\bar{\omega}\big|_{{\rm ker}\;
  (d\bar{a}\wedge \lambda^\pm)} = \hat{\omega}^\pm\).
We also define the vector field \(\overline{X}_\pm\) by
\begin{align*}                                                            
  d\bar{a}(\overline{X}_\pm)=0\qquad \lambda^\pm(\overline{X}_\pm) = 1
  \qquad i_{\overline{X}_\pm}\hat{\omega}^\pm = 0.
  \end{align*}  
From this we define \(\overline{J}\) so that the following conditions hold
  \begin{align*}                                                          
    \overline{J}\colon {\rm ker}\; (d\bar{a}\wedge \lambda^\pm) \to {\rm
    ker}\; (d\bar{a}\wedge \lambda^\pm) \qquad\text{and}\qquad
    \overline{J}\colon {\rm ker}\; \hat{\omega}^\pm \to {\rm ker}\;
    \hat{\omega}^\pm,
    \end{align*}
  and moreover, \(\overline{J}\partial_{\bar{a}} = \overline{X}\) and
  \(\hat{\omega}^\pm(\cdot , \overline{J}\cdot )\big|_{{\rm ker}\;
  (d\bar{a}\wedge \lambda^\pm)} \) is symmetric and positive definite.
Now, recall that on \(\{1-\epsilon< |\bar{a}|\}\) we have \(\bar{\omega} =
  \beta'(\bar{a}) d\bar{a}\wedge \lambda^\pm + \hat{\omega}^\pm\), so that
  this \(\overline{J}\) is \(\bar{\omega}\)-compatible on \(\{1-\epsilon <
  |\bar{a}| < 1\}\), and because \(\bar{\omega}\) is non-degenerate on
  \(\{|\bar{a}| < 1\}\) we may then smoothly extend \(\overline{J}\) to be
  an almost complex structure which is \(\bar{\omega}\)-compatible on
  \(\{|\bar{a}|< 1\}\).
This defines the desired \(\overline{J}\); in particular, this establishes
  properties (1b), (3c), and (4)

At this point, we note that it only remains to define the Riemannian
  metric \(\bar{g}\), show that \(\overline{J}\) is a \(\bar{g}\)-isometry,
  and  properties (2) and (3d).

To that end, we must first define the metric \(\bar{g}\), and to do that,
  we will first fix a smooth function \(\chi\) which satisfies
  \begin{align*}                                                          
    &\chi\colon \mathbb{R}\to [0, 1]
    \\
    &\chi(\bar{a}) = 
    \begin{cases}
    0 &\text{if } |\bar{a}| -1 < -\frac{1}{2}\epsilon\\
    1 &\text{if } |\bar{a}| -1 > -\frac{1}{4}\epsilon.
    \end{cases} 
    \end{align*}
Then, working on the ends, we define
  \begin{align*}                                                          
    \bar{g} =\Big( \theta(\bar{a})(d\bar{a} \wedge \lambda^\pm)  +
    \hat{\omega}^\pm\Big)(\cdot, \overline{J}\cdot)
    \end{align*}
  where
  \begin{align*}                                                          
    \theta(\bar{a}) = \chi(\bar{a})
    +\big(1-\chi(\bar{a})\big)\beta'(\bar{a}) \qquad\text{and}\qquad
    \hat{\omega}^\pm = \omega^\pm + \beta(\bar{a}) d{\lambda}^\pm .
    \end{align*}
We see immediately from this definition that on the set \(\{|\bar{a}| -1 >
  -\epsilon\}\) that \(\bar{g}\) is a Riemannian metric for which
  \(\overline{J}\) is an isometry.
Moreover, on the region \(\{ 1 - \epsilon < |\bar{a}| < 1 -
  \frac{1}{2}\epsilon\} \) we have \(\chi=0\) and hence on this region we
  have
  \begin{align*}                                                          
    \bar{g} &= \Big(\beta'(\bar{a}) (d\bar{a}\wedge \lambda^\pm) +
    \omega^\pm +\beta(\bar{a}) d\lambda^\pm\Big)(\cdot , \overline{J}\cdot)
    \\
    &=\bar{\omega}(\cdot, \overline{J} \cdot)
    \end{align*}
  so that we may smoothly extend \(\bar{g}\) by requiring that on \(\{|a|
  < 1-\frac{1}{2}\epsilon\}\), we have \(\bar{g} = \bar{\omega}(\cdot,
  \overline{J}\cdot)\).
Consequently, \(\bar{g}\) is indeed a Riemannian metric on
  \(\overline{W}\), for which \(\overline{J}\) is always an isometry.
This, in turn, establishes properties (2) and (3d), which then completes
  the proof of Lemma \ref{LEM_cobordism_to_extended_corbordism}.
\end{proof}
%

\begin{remark}[adjustment of $\overline{J}$]
  \label{REM_adjustment_of_J}
  \hfill\\
Recall (see for example Proposition 2.63 of \cite{MS2}) that if
  \(\widetilde{W}\) is a finite dimensional manifold, and \(\widetilde{E}\to
  \widetilde{W}\) is a rank \(2n\) bundle with a symplectic bilinear form
  \(\tilde{\omega}\), then on \(\widetilde{E}\) there exists an almost
  complex structure \(\widetilde{J}\) compatible with \(\tilde{\omega}\),
  and moreover the space of such almost complex structures is contractible.
As a consequence, if a symplectic cobordism
  \((\widetilde{W},\tilde{\omega})\) is equipped with an almost complex
  structure \(\widetilde{J}\), and \(\mathcal{K}\subset
  \widetilde{W}\setminus \partial \widetilde{W}\) is a compact set, then
  after shrinking \(\epsilon>0\) if necessary, the associated extended
  symplectic cobordism \(\mathbf{W} = (\overline{W}, \bar{\omega},
  \overline{J}, \bar{g}, \bar{a}, \partial_{\bar{a}}, \epsilon)\) can be
  arranged so that there exists a neighborhood \(\mathcal{O}\) of
  \(\mathcal{K}\) such that on \(\Psi^{-1}(\mathcal{O})\) we have
  \(\Psi^*\widetilde{J} = \overline{J}\).
\end{remark}
%

\subsection{Pseudoholomorphic Curves} \label{SEC_pseudoholomorphic_curves}

We now turn to pseudoholomorphic maps and properties thereof.

\begin{definition}[pseudoholomorphic map]
  \label{DEF_pseudoholomorphic_map}
  \hfill\\
Let \((S, j)\) and \((W, J)\) be smooth almost complex manifolds with
  \({\rm dim}(S)=2\), each possibly with boundary.
A \(\mathcal{C}^\infty\)-smooth map \(u:S\to W\) is said to be
  \emph{pseudoholomorphic} provided \(J\cdot Tu = Tu\cdot j\). 
That is, the tangent map of \(u\) intertwines the almost complex
  structures on domain and target.
Unless otherwise specified, we allow \(S\) to be disconnected.

We say such a map is \emph{proper} provided the preimage of any compact
  set is compact.
We say such a map is \emph{boundary-immersed} provided either \(u:
  \partial S \to W\) is an immersion, or else if \(\partial S =
  \emptyset\).
\end{definition}
%

Given a proper pseudoholomorphic map \(u:S\to W\), it will be convenient to
  denote the set of critical points by \(\mathcal{Z}_u:=\{\zeta\in S:
  T_\zeta u = 0\}\).
Recall that any connected component \(S_0\) of \(S\) for which
  \(\mathcal{Z}_u\cap S_0\) contains an interior accumulation point, we
  must have \(u\big|_{S_0}\equiv p\in W \); for details, see Lemma 2.4.1
  in \cite{MS}.
As such, for a pseudoholomorphic map \(u:S\to W\), the restriction of
  \(u\) to each connected component \(S_0\subset S\) is either a constant
  map, or else it is generally immersed in the following sense.

\begin{definition}[generally immersed]
  \label{DEF_generally_immersed}
  \hfill\\
We say a pseudoholomorphic map is \emph{generally immersed} provided the
  set of critical points has no interior accumulation points.
\end{definition}
%

\begin{definition}[marked nodal Riemann surface]
  \label{DEF_nodal_riemann_surface}
  \hfill\\
A \emph{nodal Riemann surface} is a triple \((S, j, D)\), with entries as
  follows.
The first entry, \(S\), is a real two-dimensional manifold, which may have
  smooth boundary, but we require that each connected component of
  \(\partial S\) is compact.
The second entry, \(j\), is a smooth almost complex structure on \(S\).
Finally, the third entry, \(D\subset S\setminus \partial S\), is an
  unordered closed discrete set of pairs \(D=
  \{\overline{d}_1,\underline{d}_1, \overline{d}_2,\underline{d}_2,\ldots
  \} \) which we call nodal points, and the pairs \(\{\underline{d}_i,
  \overline{d}_i\}\) we call nodal pairs.

A \emph{marked nodal Riemann surface} is the four-tuple \((S, j, \mu, D)\)
  where \((S, j, D)\) is a nodal Riemann surface, and where \(\mu\subset
  S\setminus (D\cup \partial S) \) is a discrete closed set of points.
\end{definition}
%

\begin{remark}[nodal notation]
  \label{REM_nodal_noation}
  \hfill\\
A careful reader may notice that in a nodal Riemann surface, the
  structure which determines which nodal points are paired with which
  other nodal points to form a nodal pair is implied by the notation but
  not explicitly provided in the tuple \((S, j, D)\).
Although this ambiguity is standard in the literature, it can be made
  precise by letting \(D = \{d_1, d_2, \ldots\}\) be a closed discrete
  set of points, and letting \(\iota: D\to D\) denote an involution which
  sends each nodal point \(d\in D \) to the unique point \(d'\in D\) (with
  \(d\neq d'\)) with the property that \(\{d, d'\}\) is a nodal pair.
A nodal Riemann surface would then be given by the tuple \((S, j,
  D, \iota)\).
In this way, \(\iota(\overline{d}_i) = \underline{d}_i\) and
  \(\iota(\underline{d}_i) = \overline{d}_i\).
Here and throughout, we shall follow the more ambiguous but less
  cumbersome notation of Definition \ref{DEF_nodal_riemann_surface}, and
  leave the obvious precisification to the reader.
\end{remark}
%

Associated to a nodal Riemann surface is the topological space \(|S|\)
  defined by identifying a nodal point with the other point in its nodal
  pair; in other words, the space \(S/ (\underline{d}_i \sim
  \overline{d}_i)\).

As in Section 4.4 of \cite{BEHWZ}, we define \(S^D\) to
  be the oriented blow-up of \(S\) at the points $D$, and we
  let $\overline{\Gamma}_i:=\big( T_{\overline{d}_i}(S)\setminus
  \{0\}\big)/\mathbb{R}_+^*\subset S^D$ and \( \underline{\Gamma}_i:=\big(
  T_{\underline{d}_i}(S)\setminus \{0\}\big)/ \mathbb{R}_+^*\subset S^D\)
  denote the newly created boundary circles over the \(d_i\). 

\begin{definition}[decorated marked nodal Riemann surface]
  \label{DEF_decorated_nodal_riemann_surface}
  \hfill\\
A \emph{decorated marked nodal Riemann surface} is a tuple \((S, j, \mu,
  D, r)\) where \((S, j, \mu, D)\) is a marked nodal Riemann surface, and
  \(r\) is a set of orientation reversing orthogonal maps
  \(\bar{r}_\nu:\overline{\Gamma}_\nu\to \underline{\Gamma}_\nu\) and 
  \(\underline{r}_\nu:\underline{\Gamma}_\nu\to \overline{\Gamma}_\nu\),
  which we call \emph{decorations}; here by orthogonal orientation
  reversing, we mean that \(r_\nu(e^{i\theta}z) = e^{-i\theta}r_\nu(z)\)
  for each \(z\in \Gamma_\nu \).
We also define \(S^{D,r}\) to be the smooth surface obtained
  by gluing the components of \(S^D\) along the boundary circles
  \(\{\overline{\Gamma}_1,\underline{\Gamma}_1,
  \overline{\Gamma}_2,\underline{\Gamma}_2, \ldots  \}\) via the
  decorations \(\bar{r}_\nu\) and \(\underline{r}_\nu\).
We will let \(\Gamma_\nu\) denote the special circles
  \(\overline{\Gamma}_\nu=\underline{\Gamma}_\nu\subset S^{D,r}\).
\end{definition}
%

We will also need the following definition.

\begin{definition}[arithmetic genus] 
  \label{DEF_arithmetic_genus}
  \hfill\\
Let \(\mathbf{S}= (S, j, \mu, D)\) be a marked nodal nodal
  Riemann surface.
As above, let \(S^D\) be the oriented blow-up of \(S\) at the points
  \(D\), and let \(S^{D,r}\) denote the surface obtained by gluing \(S^D\)
  together along pairs of circles associated to pairs of nodal points.
We define the arithmetic genus of \(\mathbf{S}\) to be the genus of
  \(S^{D,r}\).
That is,
  \begin{equation*}                                                       
    {\rm Genus}_{arith}(\mathbf{S})  = {\rm Genus}(S^{D, r}).
    \end{equation*}
\end{definition}
%

We note that it is more standard to define the arithmetic genus in terms
  of a formula involving the genera of connected components, number of
  marked points, number of nodal points, etc.
It will be convenient for later applications to have the above definition
  at our disposal, however it is equivalent to the more standard formulaic
  definition; see the Appendix of \cite{FH1} for details.

\begin{definition}[stable Riemann surface]
  \label{DEF_stable_Riemann_surface}
  \hfill\\
We say a compact marked nodal Riemann surface, \((S,j, \mu, D)\), is
  \emph{stable} if and only if for each connected component
  \(\widetilde{S}\subset S\) we have
  \begin{equation*}                                                       
    \chi(\widetilde{S})-\#(\widetilde{S}\cap (\mu\cup D)) < 0.
    \end{equation*}
\end{definition}
%

\begin{lemma}[uniformization]
  \label{LEM_uniformization}
  \hfill \\
Let \((S,j, \mu, D)\) be a stable compact marked nodal Riemann surface,
  possibly with boundary.
Then there exists a unique smooth geodesically complete metric \(h\) on
  \(\dot{S}:=S\setminus(\mu\cup D)\) in the conformal class of \(j\) such
  that \({\rm Area}_h(\dot{S}) <\infty\), the Gauss curvature of \(h\) is
  identically \(-1\), and the boundary components of \(S\) are all
  \(h\)-geodesics.
\end{lemma}
%
\begin{proof}
This is the well known uniformization theorem.
A proof via variational partial differential equation methods in the case
  that \(\mu\cup D=\emptyset = \partial S\) case can be found in
  \cite{Tromba}.  
The case with boundary can be treated by modifying the argument in
  \cite{Tromba} to consider an associated Neumann boundary value problem.
The case with punctures can be treated by removing disks of arbitrarily
  small radius centered at points in \(\Gamma\) and taking limits.
\end{proof}

We call \(h\) the Poincar\'{e} metric associated to \((S, j, \mu, D)\), and
  will often denote it \(h^{j, \mu\cup D}\) to denote the dependence upon
  both the conformal structure \(j\) and the special points \(\mu\cup D\);
  for example, see the notion of Gromov convergence given below in
  Definition \ref{DEF_gromov_convergence}.

\begin{remark}[Orientations on Riemann surfaces]
  \label{REM_orientations_surfaces}
  \hfill\\
Any Riemann surface is oriented by the almost complex structure so that
  \((v, jv)\) is a positively oriented frame whenever \(v\neq 0\).
Furthermore, if a Riemann surface \((S, j)\) has boundary, then the
  boundary will be oriented by letting \(\nu\) be an \emph{outward pointing}
  unit normal, and defining \(j\nu\) to be a positively oriented basis of
  \(\partial S\).
\end{remark}
%

\begin{definition}[Genus]
  \label{DEF_genus}
  \hfill\\
Let \(S\) be a two dimensional oriented manifold, possibly with boundary,
  with at most countably many connected components, and with the property
  that each connected component of \(\partial S\) is compact.
Then
\begin{enumerate}                                                         
    \item If \(S\) is closed and connected, then define \({\rm
      Genus}(S):=g\) where \(\chi(S)=2-2g\) is the Euler characteristic
      of \(S\).
    \item If \(S\) is compact and connected with \(n\) boundary
      components, define \(\widetilde{S}=\big(S \sqcup (\sqcup_{k=1}^n
      D^2)\big)/\sim\) to be the closed surface capped off by \(n\) disks,
      and define \({\rm Genus}(S):={\rm Genus}(\widetilde{S})\).
    \item If \(S\) is compact (possibly with boundary), then \({\rm
      Genus}(S)\) is defined to be the sum of the genera of each connected
      component.
    \item If \(S\) is not compact, then \({\rm Genus}(S)\)
      is defined by taking any nested sequence \(S_1\subset S_2\subset
      S_3\subset\cdots\) of compact surfaces (possibly with boundary) such
      that \(S_k\subset S\) for all \(k\in \mathbb{N}\) and such that \(S
      =\cup_k S_k\); then we define \({\rm Genus}(S):=\lim_{k\to \infty}
      {\rm Genus}(S_k)\).
  \end{enumerate}
\end{definition}
%

\begin{remark}[Genus monotonicity]
  \label{REM_genus_monotonicity}
  \hfill \\
Note that for compact surfaces with boundary, \({\rm Genus}(\cdot)\),
  thought of as a function, satisfies a notion of super-additivity made
  precise in Lemma \ref{LEM_genus_addition} below.
As a consequence of this lemma, it immediately follows that if \(S'\)
  and \(S''\) are compact surfaces with boundary and satisfy \(S' \subset
  S''\),  then
  \begin{equation*}                                                       
      {\rm Genus}(S')\leq {\rm Genus}(S''),
    \end{equation*}
  and hence for a non-compact surface \(S\), \({\rm Genus}(S)\) is well
  defined by Definition \ref{DEF_genus}.
\end{remark}
%

\begin{lemma}[Genus super-additivity]
  \label{LEM_genus_addition}
  \hfill\\
Suppose \(S\) is a smooth compact oriented two-dimensional manifold,
  possibly with boundary.
Suppose further that \(S=S_1\cup S_2\), for which the intersection
  \(S_1\cap S_2\) consists of a finite union of pairwise disjoint smooth
  embedded loops for which \(S_1\cap S_2 = (\partial S_1)\cap(\partial
  S_2)\)
Then 
\begin{equation*}                                                         
  {\rm Genus}(S_1 \cup S_2) \geq {\rm Genus}(S_1)+{\rm Genus}(S_2).
  \end{equation*}
\end{lemma}
%

\begin{proof}
We begin by resolving a related problem. 
Indeed, suppose \(S'\) is a smooth compact oriented two-dimensional
  manifold, possibly with boundary, and further suppose there exists an
  smooth orientation reversing diffeomorphism from one connected component
  of \(\partial S'\) to another.
Define \(S:=S'/\sim\) where \(x\sim \phi(x)\). 
Then we claim 
  \begin{equation}\label{EQ_genus_inequality}                             
    {\rm Genus}(S)\geq {\rm Genus}(S').
    \end{equation}
To prove inequality (\ref{EQ_genus_inequality}),  we recall the Euler
  characteristic of the surface \(S\) is given by
  \begin{equation*}                                                       
    \chi(S)= 2\#\pi_0(S) - 2{\rm Genus}(S) - \#\pi_0(\partial S),
    \end{equation*}
  where \(\#\pi_0(X)\) denotes the number of connected components of the
  space \(X\).
However, by the Gauss-Bonnet theorem we also have 
  \begin{equation*}                                                       
    \chi(S) = \frac{1}{2\pi}\int_S K_g\; d A +
    \frac{1}{2\pi}\int_{\partial S} \kappa_g ds
    \end{equation*}
   where \(g\) is a Riemannian metric on \(S\), \(K_g\) is the Gaussian
   curvature of \(g\), and \(\kappa_g\) is the associated geodesic
   curvature.
By choosing a metric on \(S'\) for which the \(\partial S'\) consists
  of geodesics and which smoothly descends to \(S\), we see that
  \begin{equation*}                                                       
    \chi(S) = \chi(S'). 
    \end{equation*}
Next we observe that 
  \begin{equation*}                                                       
    \#\pi_0(\partial S) = \#\pi_0(\partial S') - 2
    \end{equation*}
  and 
  \begin{equation*}                                                       
    \#\pi_0(S) = \#\pi_0(S') - e
    \end{equation*}
   where \(e\in \{0, 1\}\).
We now take the difference of the two following equations
  \begin{align*}                                                          
    \chi(S)&= 2\#\pi_0(S) - 2{\rm Genus}(S) - \#\pi_0(\partial S)\\
    \chi(S')&= 2\#\pi_0(S') - 2{\rm Genus}(S') - \#\pi_0(\partial S'),
    \end{align*}
  and make use of the three above equations to find that
  \begin{equation*}                                                       
    {\rm Genus}(S)-{\rm Genus}(S') = 1-e\geq 0.
    \end{equation*}
We conclude that 
  \begin{equation*}                                                       
    {\rm Genus}(S)\geq {\rm Genus}(S').
    \end{equation*}
Next observe that if \(\phi\) is a smooth orientation reversing
  diffeomorphism from the union of several connected components of
  \(\partial S'\) to the union of several other connected components of
  \(\partial S'\) then again \({\rm Genus}(S)\geq {\rm Genus}(S')\)
  because the above argument can simply be iterated.
However,  Lemma \ref{LEM_genus_addition} then follows immediately
  by letting
  \begin{equation*}                                                       
    S = S_1\cup S_2\qquad\text{and}\qquad S'=S_1 \sqcup S_2,
    \end{equation*}
   since 
  \begin{equation*}                                                       
    {\rm Genus}(S_1 \sqcup S_2) = {\rm Genus}(S_1) + {\rm Genus}(S_2).
    \end{equation*}
This completes the proof of Lemma \ref{LEM_genus_addition}.
\end{proof}
%

\begin{definition}[marked nodal pseudoholomorphic curve]
  \label{DEF_pseudholomorphic_curve}
  \hfill\\
A \emph{marked nodal pseudoholomorphic curve} is a tuple \(\mathbf{u}=(u,
  S, j, W, J, \mu, D)\) with entries as follows.  
The triple \((S, j, \mu, D)\) is a marked nodal Riemann surface. 
The pair \((W, J)\) is a smooth real \(2n\)-dimensional almost complex
  manifold, and \(u:S\to W\) is a smooth map for which $J\cdot Tu =
  Tu\cdot j$.
Finally, we require that \(u(\overline{d}_i) = u(\underline{d}_i)\) for
  each nodal pair \(\{\overline{d}_i, \underline{d}_i\}\subset D\).
\end{definition}
%

Unless otherwise specified, we will allow \(S\), the domain of a
  pseudoholomorphic curve to be non-compact, to have smooth boundary, and
  to have unbounded topology (i.e. countably infinite connected
  components, boundary components, and genus).

\begin{definition}[stability and common types of pseudoholomorphic curves]
  \label{DEF_stable}
  \hfill\\
We will say that a pseudoholomorphic curve \(\mathbf{u}=(u, S, j, W, J,
  \mu, D)\) is 
\begin{enumerate}                                                         
  \item 
  \emph{compact} provided \(S\) is a compact manifold with smooth boundary,
  \item 
  \emph{closed} provided \(S\) is a compact manifold without boundary,
  \item 
  \emph{connected} provided that \(|S|\) is connected, 
  \item 
  \emph{proper and boundary-immersed} provided the map \(u:S\to W\) is
  proper and boundary-immersed respectively.
  \end{enumerate}
Lastly, we say that a boundary-immersed curve \(\mathbf{u}\) is
  \emph{stable} provided that for each connected component \(S_0\subset
  S\) on which the restriction \(u:S_0\to W\) is constant, we have
  \begin{equation}\label{EQ_stability_condition}                          
    \chi(S_0)- \#\big(S_0\cap (\mu \cup D) )  <0.
    \end{equation}
\end{definition}
%

\begin{remark}[on our notion of stability]
  \label{REM_stability}
  \hfill\\
We begin by observing that our above definition of stability is notably
  different from the more precise definition established in Symplectic
  Field Theory, and instead has more in common with notion from
  Gromov-Witten theory. 
Indeed, in SFT the target manifold may be \(\mathbb{R}\times M\) with an
  \(\mathbb{R}\)-action given by translation in the first factor,
  and therefore stable pseudoholomorphic curves (or buildings thereof) are
  equivalence classes defined via this \(\mathbb{R}\)-action.
More specifically, a pseudoholomorphic building is stable only if it has
  the property that on each floor there is a connected component which is
  not a trivial cylinder, and moreover that each constant component \(S_0\)
  satisfies the inequality of equation (\ref{EQ_stability_condition}).
In this way, an orbit cylinder by itself is not stable in an SFT sense,
  but is stable in the sense of Definition \ref{DEF_stable} above.
Because of this discrepancy, and because of the importance of the notion
  of stability, some discussion is warranted.

To that end, we first recall that in the study of pseudoholomorphic
  curves, the notion of stability arises predominantly so that a sequence of
  \emph{stable} curves of fixed topological type\footnote{
    For example, a sequence of connected curves of some specified genus and
    in a specific (relative) homology class.} 
  necessarily has a subsequence which has a \emph{unique} stable limit.
Moreover, the topology associated to this limit must be such that there
  exists a gluing theorem which (at the very least) finds the tail of the
  subsequence given only a transverse limit curve.
For a historical example, one can consider Gromov's original definition of
  compactness in \cite{Gr}, which led to non-unique limits (due to
  arbitrarily complicated trees of constant spheres) and hence
  non-Hausdorff topologies on the associated moduli spaces.
This was then remedied by Kontsevich, who proposed the notion of stability
  which yielded the desired unique limits.
Similarly, in Symplectic Field Theory, unless one declares levels
  consisting only of trivial cylinders to be unstable (which is \emph{not}
  done in Definition \ref{DEF_stable}), then limit buildings obtained via
  compactness will not be unique.
We now explain why.

The issue is that at present we do not have a full compactness theorem. 
Indeed, for the bulk of the argument below, we really only consider
  curves whose symplectization coordinate has an absolute minimum, and
  moreover if given a sequence of such curves with suitable bounds (though
  not necessarily energy bounds) we extract a very weak notion of a limit
  curve.  
Indeed, in an SFT sense, what we find is only the bottom-most level in any
  naturally arising limit building.
The reason for this is two-fold. 
The first is simply to obtain a feral pseudoholomorphic curve (namely the
  bottom-most level), which may have infinite energy, and we use this to
  find the desired closed invariant subset of the Hamiltonian flow.
This is rather analogous to how Hofer first used a preliminary
  compactness/bubbling argument to establish the existence of a finite
  energy plane in the symplectization of contact \(S^3\), and hence
  proving that Reeb flows on \(S^3\) must have a periodic orbit; see
  \cite{H93}.
Only later was a full compactness theorem established; see \cite{BEHWZ}.
Consequently, stability for the purposes of this manuscript need only take
  into account a single level of whatever building structure the (eventually
  understood) full limit has.

This raises the question: Given an understanding of SFT compactness, and
  an argument to extract a single level of the (supposed) feral limit
  building, why have we not proved a full SFT compactness theorem for
  feral curves?
Indeed, it is in fact not difficult to build on the ideas here and in
  \cite{FH1} to find many levels of a limit building.
However, at present there still remain several complications. 
The first is that the limit may have infinitely many levels, each with
  positive \(\omega\)-energy.
Indeed, unlike feral planes, feral cylinders do not have an
  \(\omega\)-energy threshold.
Second, even if one extracts all levels of a feral limit building which
  have positive \(\omega\)-energy, it is not yet clear if the sum of the
  \(\omega\)-energy of each of the limit levels will equal the
  \(\omega\)-energy of the curves in the approximating sequence.
Without such knowledge, it is not even clear the ``full'' limit has been
  found, since one should expect the \(\omega\)-energy to be preserved in
  the compactification process.
And third, even if a limit curve is understood which captures all the
  \(\omega\)-energy there is to capture, at present it is not understood
  how to glue two properly feral ends together.
There are a variety of special cases in which this seems possible, however
  it is not clear if these cases are exceptional or generic.
As a consequence of all of these issues, one should regard the notion of
  stability provided in Definition \ref{DEF_stable} as preliminary,
  proprietary, and restricted to the specific needs of this manuscript.
We expect an updated and more precise notion to naturally arise in future
  work.

\end{remark}
%

\begin{definition}[decorated marked nodal pseudoholomorphic curve]
  \label{DEF_decorated}
  \hfill \\
A \emph{decorated} marked nodal pseudoholomorphic curve \((\mathbf{u},
  r)\) is a pair for which \(\mathbf{u}=(u, S, j, W, J, \mu, D)\) is a
  marked nodal pseudoholomorphic curve, and \((S, j, \mu, D, r)\) is a
  decorated marked nodal Riemann surface as in Definition
  \ref{DEF_decorated_nodal_riemann_surface}.
As above, we let \(S^{D,r}\) be the smooth surface obtained by taking
  the oriented blow up of \(S\) at the points in \(D\) and then gluing the
  components of the result together along the boundary circles
  \(\overline{\Gamma}_\nu\) and \(\underline{\Gamma}_\nu\).
Consequently, see that the smooth map \(u\colon S\to W\) then lifts to a
  continuous map \(u\colon S^{D,r}\to W\).
\end{definition}
%

\begin{definition}[area of pseudoholomorphic curves]
  \label{DEF_area}
  \hfill \\
Let \mbox{\(\mathbf{u}=(u, S, j, W, J, \mu, D)\)} be a marked nodal
  pseudoholomorphic curve which is proper and boundary immersed.
Assume further that the almost Hermitian manifold \((W, J, g)\) has
  no boundary.
Let \(S_{\rm const}\subset S\) denote the union of connected components
  of \(S\) on which \(u\) is a constant map.
As noted in the discussion following Definition
  \ref{DEF_pseudoholomorphic_map} (pseudoholomorphic map), the map
  \(u:S\setminus S_{\rm const}\to W\) is generally immersed in the sense
  of Definition \ref{DEF_generally_immersed}.
Consequently on \(S\setminus S_{\rm const}\) we  can define the
  following metric
  \begin{equation*}                                                       
    {\rm dist}_{u^*g}(\zeta_0,\zeta_1):=\inf
    \Big\{{\textstyle \int_0^1} \langle
    \dot{\gamma}(t),\dot{\gamma}(t)\rangle_{u^*g}^{\frac{1}{2}}dt:
    \gamma\in \mathcal{C}^1\big([0,1],S\big)\text{ and }
    \gamma(i)=\zeta_i\Big\},
    \end{equation*}
  where our convention will be that if $\zeta_0$ and
  $\zeta_1$ lie in different connected components, then ${\rm
  dist}_{u^*g}(\zeta_0,\zeta_1):=\infty$.
Thus we may regard $(S\setminus S_{\rm const},{\rm dist}_{u^*g})$
  as a metric space, in which case it can be equipped with Hausdorff
  measures $d\mathcal{H}^k$.
Note that if $\mathcal{O}\subset S\setminus S_{\rm const}$ is an open set
  on which $u$ is an immersion, then $d\mathcal{H}^{2}(\mathcal{O})={\rm
  Area}_{u^*g}(\mathcal{O})$.
As such, our convention will be to simply define the area
  of an arbitrary open set \(\mathcal{U}\subset S\setminus S_{\rm
  const}\) to be \({\rm Area}_{u^*g}(\mathcal{U}):=
  d\mathcal{H}^{2}(\mathcal{U})\).
Finally, for an arbitrary open set \(\mathcal{U}\subset S\) we define
  \begin{equation*}                                                       
    {\rm Area}_{u^*g}(\mathcal{U}):=d\mathcal{H}^{2}(\mathcal{U}\setminus
    S_{\rm const}).
    \end{equation*}
\end{definition}
%

Again, in the absence of a symplectic form, the above definition may seem
  foreign, so we pause for a moment to show that in the perhaps more
  familiar setup in which \(\Omega\) is a symplectic form, and \(J\) is an
  \(\Omega\)-compatible almost complex structure so that \(g:= \Omega\circ
  ({\rm Id}\times J)\) is a Riemannian metric, the above definition of
  metric area of a pseudoholomorphic curve agrees with the symplectic area as
  expected.
To that end, we suppose \(u:(S, j)\to (W, J)\) is pseudoholomorphic map,
  and \(z\in S \) for which \(T_z u\neq 0\).
We then let \(\mathcal{O}(z)\subset S\) be an open neighborhood of \(z\) on
  which \(u\) is an immersion and on which there exist conformal
  coordinates \((s, t)\) for which \(j\partial_s = \partial_t\).
Consequently \(J u_s =  u_t\), and
  \begin{equation*}                                                       
    \Omega(u_s, u_t)= \|u_s\|_g^2 = \|J u_t\|_g^2 = \|u_t\|_g^2  
    \end{equation*}
  and
  \begin{equation*}                                                       
    \langle u_s, u_t\rangle  =  \Omega(u_s, Ju_t) =  -\Omega(u_s, u_s) =0,
    \end{equation*}
  from which it follows that
  \begin{align*}
    d\mathcal{H}^2(\mathcal{O})&= {\rm Area}_{u^*g}(\mathcal{O})
    \\
    &=\int_{\mathcal{O}}\big(\|u_s\|_g^2 \|u_t\|_g^2 - \langle
      u_s,u_t\rangle_g^2\big)^{\frac{1}{2}}ds\wedge dt \\
    &=\int_{\mathcal{O}}\|u_s\|_g^2ds\wedge dt
    \\
    &=\int_{\mathcal{O}} u^*\Omega.
    \end{align*}
We conclude that indeed, in the case that \(J\) is \(\Omega\)-compatible,
  metric area as defined above agree with symplectic area of
  pseudoholomorphic curves.

We now turn our attention to issues of convergence of pseudoholomorphic
  curves.
In what follows it will be important to recall that given a compact stable
  marked nodal Riemann surface \((S, j, \mu, D)\), the associated
  Poincar\'{e} metric, as provided in Lemma \ref{LEM_uniformization}, on
  \(S\setminus (\mu\cup D)\) is denoted by \(h^{j, \mu\cup D}\).
We begin with our principle notion of convergence of pseudoholomorphic
  curves.

\begin{definition}[Gromov convergence]
  \label{DEF_gromov_convergence}
  \hfill \\
A sequence \(\mathbf{u}_k = (u_k,S_k,j_k,W, J_k, \mu_k, D_k)\) of
  compact marked nodal stable boundary-immersed pseudoholomorphic curves
  is said to converge in a Gromov-sense to a compact marked nodal stable
  boundary-immersed pseudoholomorphic curve \(\mathbf{u}=(u,S,j,W,J,\mu,
  D)\) provided the following are true for all sufficiently large
  \(k\in \mathbb{N}\).
\begin{enumerate}                                                         
  \item 
  \(J_k\to J$ in $C^\infty\).
  \item 
  There exist sets of marked points 
  \begin{equation*}                                                       
    \mu_k'\subset S_k\setminus (\partial S_k\cup \mu_k\cup D_k) 
    \qquad\text{and}\qquad
    \mu'\subset S\setminus (\partial S \cup \mu \cup D)
    \end{equation*}
  with the property that \(\#\mu' = \# \mu_k'\), and with the property
  that for each connected component \(\widetilde{S}_k\) of \(S_k\)
  we have
  \begin{equation*}                                                       
    \chi(\widetilde{S}_k) - \#\big(\widetilde{S}_k\cap (\mu_k\cup \mu_k'
    \cup D_k)\big) < 0
    \end{equation*}
  and for each connected component \(\widetilde{S}\) of \(S\) we have
  \begin{equation*}                                                       
    \chi(\widetilde{S}) - \#\big(\widetilde{S}\cap (\mu\cup \mu' \cup
    D)\big) < 0.
    \end{equation*}
  \item 
  There exists a decoration \(r\) for \(\mathbf{u}\), a sequence of
  decorations \(r_k\) for the \(\mathbf{u}_k\), and sequences of
  diffeomorphisms \(\phi_k: S^{D,r}\to S_k^{D_k,r_k}\) such that the
  following hold
  \begin{enumerate}                                                       
    \item 
    \(\phi_k(\mu) = \mu_k\)
    \item 
    \(\phi_k(\mu') = \mu_k'\)
    \item 
    for each \(i=1,\ldots,\delta\) the curve \(\phi_k(\Gamma_i)\) is
    a \(h^{j_k,\mu_k\cup \mu_k'\cup D_k}\)-geodesic in the punctured
    surface \(S_k':= S_k\setminus (\mu_k\cup \mu_k'\cup D_k)\).
    \end{enumerate}
  \item 
  \(\phi_k^* h^{j_k,\mu_k\cup\mu_k'\cup D_k} \to h^{j,\mu\cup\mu'\cup
    D}\) in \(C_{loc}^\infty\big(S^{D,r}\setminus (\mu\cup\mu'\cup_i
    \Gamma_i) \big)\); here we have abused notation by letting
    \(h^{j,\mu\cup\mu'\cup D}\) also denote its lift to \(S^{D,r}\).
  \item 
  \(\phi_k^*u_k \to u$ in $C^0(S^{D,r})\).
  \item 
  \(\phi_k^*u_k \to u\) in \(C_{loc}^\infty(S^{D,r}\setminus
    \cup_i\Gamma_i)\).
  \item 
  For each connected component \(\Lambda\) of \(\partial \overline{S}\),
  the \(\phi_k^* h^{j_k,\mu_k\cup\mu_k'\cup D_k}\)-length of \(\Lambda\)
  is uniformly bounded away from \(0\) and \(\infty\).
  \end{enumerate}
\end{definition}
With this notion of convergence established, we can now provide the target-localized version of Gromov's compactness
  theorem for pseudoholomorphic curves.
\begin{smalltheorem}[Target-local Gromov compactness]
  \label{THM_target_local_gromov_compactness}
  \hfill\\
Let $(W,J,g)$ be an almost Hermitian manifold, possibly with boundary,
  and let $(J_k,g_k)$ be a sequence of almost Hermitian structures which
  converge in $C^\infty$ to $(J,g)$.
Also let $\mathcal{K}_1, \mathcal{K}_2\subset {\rm Int}(W)$ be compact
  regions, satisfying \(\mathcal{K}_1\subset {\rm Int}(\mathcal{K}_2)\), and
  let $\mathbf{u}_k=(u_k, S_k, j_k, W, J_k, \mu_k, D_k)$ be a sequence of
  stable compact marked nodal pseudoholomorphic curves satisfying
  \(u_k(\partial S_k) \cap \mathcal{K}_2 = \emptyset\) and suppose there
  exists a large positive constant \(C>0\) for which
  \begin{enumerate}
    \item 
    \({\rm Area}_{u_k^*g_k}(S_k)\leq C\),
    \item 
    \({\rm Genus}(S_k)\leq C\),
    \item
    \(\# \big(\mu_k \cup D_k\big) \leq C\)
    \end{enumerate}
Then, after passing to a subsequence (still denoted with subscripts
  \(k\)), there exist compact surfaces with boundary
  \(\widetilde{S}_k\subset S_k\) with the following properties
\begin{enumerate}                                                         
  \item the following are compact pseudoholomorphic curves
    \begin{equation*}                                                     
      (u_k, \widetilde{S}_k, j_k, \mu_k \cap \widetilde{S}_k, D_k
      \cap \widetilde{S}_k)
      \end{equation*}
  \item these domain-restricted converge in a Gromov sense to a compact
    stable marked nodal boundary immersed pseudoholomorphic curve.
  \item \(u_k(S_k\setminus \widetilde{S}_k) \subset W\setminus
    \mathcal{K}_1\).
  \end{enumerate}
\end{smalltheorem}
%
\begin{proof}
This is essentially a restatement of Corollary 3.1 from \cite{FH1}, and a
  slight generalization of Theorem 3.1 from \cite{Fish2}.
\end{proof}

Our final task of this section is to provide a notion of Gromov
  compactness for pseudoholomorphic curves in an exhaustive sense.
The idea is best illustrated with an example. 
Consider \(\mathbb{R}\times M\) equipped with an almost Hermitian
  structure \((J, g)\).
Now let \(\mathbb{D}\) denote the compact unit disk in \(\mathbb{C}\), and
  consider a sequence of pseudoholomorphic curves \(u_k : \mathbb{D}\to
  \mathbb{R}\times M\) with the following properties.
  \begin{enumerate}                                                       
    \item \(\inf \{a\circ u_k(\mathbb{D})\} = 0\) for each
      \(k\in\mathbb{N}\)
    \item \(\sup \{a\circ u_k(\mathbb{D}) \}= k = a\circ u_k (\partial
      \mathbb{D}) \) for each \(k\in\mathbb{N}\)
    \item \({\rm Area}_{u_k^*g} \big((a\circ u_k)^{-1}([-n, n])\big) \leq
      C_n\) for each \(k\in\mathbb{N}\) and each \(n\in \mathbb{N}\).
    \end{enumerate}
Geometrically then, we have a sequence of disks, with a minimum in
  \(\{0\}\times M\), both a maximum and boundary in \(\{k\}\times M\), and a
  sort of locally bounded area.
The question then becomes: Is there a notion of convergence for such
  curves which can be guaranteed after passing to a subsequence, and which
  yields a proper curve without boundary in \(\mathbb{R}\times M\)?
As it turns out, the answer is yes, and we make the notion and the result
  precise with Definition  \ref{DEF_exhaustive_gromov_convergence} and
  Theorem \ref{THM_exhaustive_gromov_compactness} respectively below.
First however, we will need the following definition so that the desired
  exhaustive compactness result can be stated in sufficient generality.

\begin{definition}[properly exhausting regions]
  \label{DEF_properly_exhausting_regions}
  \hfill\\
Let \((\overline{W}, \overline{J}, \bar{g})\) be an almost Hermitian
  manifold, which need not be compact.
We say a sequence of almost Hermitian manifolds \((W_k, J_k, g_k)\)
  \emph{properly exhaust} \((\overline{W}, \overline{J}, \bar{g})\)
  provided the following hold.
  \begin{enumerate}                                                       
    \item 
    For each \(k\in \mathbb{N}\) we have \(W_k\subset W_{k+1}\), and
    moreover \(W_k\) is an open subset of \(W_{k+1}\) in the \(W_{k+1}\)
    topology.
    \item 
    \(\overline{W} = \bigcup_{k\in\mathbb{N}} W_k\)
    \item 
    The smooth structure on \(W_k\) equals the smooth structure induced
    from \(W_{k+1}\).
    \item 
    The set \({\rm cl}(W_k) \subset W_{k+1}\) is a compact manifold with
    smooth boundary.
    \item 
    Regarding \((J_k, g_k)\) as almost Hermitian structures on
    \(\overline{W}\), we require \((J_k, g_k)\to (\overline{J}, \bar{g})\)
    in \(\mathcal{C}_{loc}^\infty\).
    \end{enumerate}
\end{definition}
%

\begin{definition}[convergence in an exhaustive Gromov sense]
  \label{DEF_exhaustive_gromov_convergence}
  \hfill \\
Let \((\overline{W}, \overline{J}, \bar{g})\) be a smooth almost Hermitian
  manifold, not necessarily compact, and let \((W_k, J_k, g_k)\) be a
  sequence which properly exhausts \((\overline{W}, \overline{J},
  \bar{g})\), in the sense of Definition
  \ref{DEF_properly_exhausting_regions}.
Suppose further that the tuples \(\bar{\mathbf{u}}=(\bar{u},
  \overline{S}, \bar{j}, \overline{W}, \overline{J}, \bar{\mu},
  \overline{D})\) and, for each \(k\in \mathbb{N}\), \(\mathbf{u}_k=(u_k,
  S_k, j_k, W_k, J_k, \mu_k, D_k)\), are each marked nodal proper stable
  pseudoholomorphic curves without boundary.
We say the sequence \(\{\mathbf{u}_k\}_{k\in \mathbb{N}}\)
  converges to \(\bar{\mathbf{u}}\) in an \emph{exhaustive Gromov
  sense} provided there exists a collection of compact smooth two
  dimensional manifolds with boundary \(\{\overline{S}^\ell\}_{\ell \in
  \mathbb{N}}\) with \(\overline{S}^\ell\subset \overline{S}\) for each
  \(\ell\in \mathbb{N}\), and there exists a collection of compact smooth
  two dimensional manifolds with boundary \(\{S_k^\ell\}_{\substack{\ell
  \in \mathbb{N}\\ k\geq \ell}}\) with \(S_k^\ell \subset S_k\) for
  all \(k, \ell\in \mathbb{N}\) with \(k\geq \ell\) for which the
  following hold.
\begin{enumerate}                                                         
  \item 
  \(\overline{S}^\ell\subset \overline{S}^{\ell+1}\setminus \partial
  \overline{S}^{\ell+1}\) for all \(\ell\in \mathbb{N}\)
  \item 
  \(\overline{S} = \bigcup_{\ell \in \mathbb{N}} \overline{S}^\ell\) 
  \item
  for each fixed \(k\in \mathbb{N}\) and each \(0\leq \ell \leq k-1\) we
    have \(S_k^\ell\subset S_k^{\ell+1}\setminus \partial S_k^{\ell+1}\)
  \item
  for each \(k\geq  \ell\in \mathbb{N}\) we have
  \begin{equation*}                                                       
    u_{k} ^{-1}(W_{\ell})\subset S_k^{\ell},
    \end{equation*}
  \item 
  for each fixed \(\ell\in \mathbb{N}\), the sequence 
    \begin{equation*}                                                     
      \big\{
      \big( u_k,\; 
      S_k^{\ell},\; 
      j_k,\;  
      \overline{W}, \;
      J_k, \;
      S_k^{\ell}\cap\mu_k,\; 
      S_k^{\ell}\cap D_k\big)
      \big\}_{k \geq \ell}
      \end{equation*}
    is a sequence of compact marked nodal stable boundary-immersed
    pseudoholomorphic curves which converges in a Gromov sense to the
    proper marked nodal stable boundary-immersed pseudoholomorphic curve
    \begin{equation*}                                                       
      \big(\bar{u},\; 
      \overline{S}^{\ell},\; 
      \bar{j},\;  
      \overline{W}, \;
      \overline{J}, \;
      \overline{S}^{\ell}\cap\bar{\mu},\; 
      \overline{S}^{\ell}\cap \overline{D}\big).
      \end{equation*}
  \end{enumerate}
\end{definition}
%

\begin{smalltheorem}[exhaustive Gromov compactness]
  \label{THM_exhaustive_gromov_compactness}
  \hfill \\
Let \((\overline{W}, \overline{J}, \bar{g})\) be a smooth almost Hermitian
  manifold, not necessarily compact, and let \((W_k, J_k, g_k)\) be a
  sequence which properly exhausts \((\overline{W}, \overline{J},
  \bar{g})\), in the sense of Definition
  \ref{DEF_properly_exhausting_regions}.
Suppose further that the sequence denoted by 
\begin{align*}                                                            
  \{\mathbf{u}_k\}_{k\in \mathbb{N}}=\{(u_k, S_k, j_k, W_k, J_k, \mu_k,
  D_k)\}_{k\in \mathbb{N}}
  \end{align*}
  is a sequence of proper stable marked nodal pseudoholomorphic curves
  without boundary for which there also exists a sequence of large
  constants \(C_k\) with the property that for each fixed \(k\in
  \mathbb{N}\) the following hold 
  \begin{enumerate}[(C1)]
    \item 
    \( \displaystyle{\sup_{\ell \geq k}} \; {\rm Area}_{u_\ell^*g_\ell}
    (\widehat{S}_\ell^k)\leq C_k \)
    \item 
    \(\displaystyle{\sup_{\ell \geq k}} \; {\rm
    Genus}(\widehat{S}_\ell^k)\leq C_k \)
    \item 
    \(\displaystyle{\sup_{\ell \geq k}}\; \# \big((\mu_\ell \cup
    D_\ell)\cap \widehat{S}_\ell^k\big)\leq C_k \)
    \end{enumerate}
  where \(\widehat{S}_\ell^k:=u_\ell^{-1}(W_k)\). 
\emph{Then} a subsequence converges in an exhaustive Gromov sense to
  \((\bar{u}, \overline{S}, \bar{j}, \overline{W}, \overline{J}, \bar{\mu},
  \overline{D})\) which is a proper stable marked nodal pseudoholomorphic
  curve without boundary.
\end{smalltheorem}
%
\begin{proof}
This is a restatement of Theorem 1 from \cite{FH1}. 
\end{proof}

\section{Existence of Minimal Subsets}\label{SEC_existence_minimal_subsets}
The primary purpose of this section is to prove our main dynamical
  result, Theorem \ref{THM_main_result}, as well as an almost
  immediate generalization, Theorem \ref{THM_second_main_result}.
It is useful to note that substantial work is required to prove the first
  result, however the second result will follow from combining the
  foundational results about feral curves developed in later sections
  together with some well established techniques from \cite{H93}.

Due to the length of the proof of Theorem \ref{THM_main_result}, we take a
  moment to sketch the main ideas in order to outline it.
The steps are as follows.\\

\noindent{\bf Step 1:} \emph{Geometric and dynamical setup.} 
Here we embed our dynamical problem into \(\mathbb{C}P^2\) and build an
  extended symplectic cobordism from the empty set to a framed Hamiltonian
  manifold with dynamics conjugated to those on the given energy level.  
We also establish the existence of an embedded pseudoholomorphic sphere
  which we later show has some nice properties.\vspace{5pt}\\
\noindent{\bf Step 2:} \emph{Automatic transversality and an abundance of
  curves.}
Here we define a moduli space of curves containing the previously found
  curve, and show these curves are very nice: they are each embedded,
  pairwise intersect exactly once, are cut out transversely, and locally fill
  out an open set.\vspace{5pt}\\
\noindent{\bf Step 3:} \emph{The moduli space \(\mathcal{M}\) extends into
  the negative end of \(\overline{W}\)}.
Here we show that the curves in the moduli space above have images which
  descend arbitrarily deep into the negative cylindrical end of the extended
  cobordism.\vspace{5pt}\\
\noindent{\bf Step 4:} \emph{An area estimate}.
Here we prove an area estimate, which can roughly be regarded as showing
  that within a bounded distance from the non \(\mathbb{R}\)-invariant
  region the areas of the curves must be uniformly bounded.
This estimate is not particularly difficult to obtain, but it is necessary
  in order to apply Theorem \ref{THM_existence}.\vspace{5pt}\\
\noindent{\bf Step 5:} \emph{Trimming curves and applying the workhorse
  theorem.}
Finally we use the area bound from the previous step to trim away portions
  of the curve in the non \(\mathbb{R}\)-invariant regions in the extended
  cobordism, and we show that the resulting family of curves satisfy the
  hypotheses of Theorem \ref{THM_existence}, which then guarantees the
  existence of a non-trivial closed invariant subset.
The desired result is then immediate.

\subsection{Proof of Theorem \ref{THM_main_result}}\hfill \\
We now prove our first main dynamical result regarding the non-minimality
  of the Hamiltonian flow on compact hypersurfaces in \(\mathbb{R}^4\).

\setcounter{CurrentSection}{\value{section}}
\setcounter{CurrentTheorem}{\value{theorem}}
\setcounter{section}{\value{CounterSectionTheoremOne}}
\setcounter{theorem}{\value{CounterTheoremTheoremOne}}
\begin{theorem}[First main dynamical result]\hfill\\
Consider \(\mathbb{R}^4\) equipped with the standard symplectic structure
  and a Hamiltonian \(H\in \mathcal{C}^\infty(\mathbb{R}^4, \mathbb{R})\)
  for which \(M:=H^{-1}(0)\) is a non-empty compact regular energy level.
Then the Hamiltonian flow on \(M\) is not minimal.
\end{theorem}
%
\begin{proof}\hfill\\
\setcounter{section}{\value{CurrentSection}}
\setcounter{theorem}{\value{CurrentTheorem}}

\noindent{\bf Step 1:} \emph{Geometric and dynamical setup.}\\

Here we build an extended symplectic cobordism, which has a negative end
  \((M^-, \eta^-)\) which is a framed Hamiltonian manifold with dynamics
  conjugated to those on \(H^{-1}(0)\).
We begin by letting \(\omega_0\) denote the standard symplectic form on
  \(\mathbb{R}^4\) given by
  \begin{align*}                                                          
    \omega_0 = dx_1\wedge dy_1 + dx_2\wedge dy_2. 
    \end{align*}
For each \(\epsilon>0\) define the linear isomorphism
  \begin{align*}                                                          
    &L_\epsilon \colon \mathbb{R}^4\to \mathbb{R}^4
    \\
    &L_\epsilon(x_1, y_1, x_2, y_2) =
    \Big(\frac{x_1}{\epsilon},\frac{y_1}{\epsilon},\frac{x_2}{\epsilon},
    \frac{y_2}{\epsilon}\Big).
    \end{align*}
We also define the smooth family of functions \(H_\epsilon\colon
  \mathbb{R}^4\to \mathbb{R}\) by \(H_\epsilon:= \epsilon^2 H\circ
  L_\epsilon\), and we define manifolds \(M_\epsilon:=
  H_\epsilon^{-1}(0)\).
By definition, it follows that
  \begin{align*}                                                          
    L_\epsilon \colon M_\epsilon \to M
    \end{align*}
  is a diffeomorphism for each \(\epsilon>0\).
Moreover, along \(M\) and \(M_\epsilon\) respectively we have \(X_H \in
  \Gamma(TM)\) and \(X_{H_\epsilon} \in \Gamma(TM_\epsilon)\).
Further still, we have

\begin{align*}                                                            
  {\omega}_0(X_{H_\epsilon}, v ) 
  &= - d H_\epsilon( v)\\
  &=  -\epsilon^2 d H(TL_\epsilon \cdot v) \\
  &= \epsilon^2 {\omega}_0 (X_H, TL_\epsilon \cdot v)\\
  &= \epsilon^2 {\omega}_0 (TL_\epsilon \cdot (TL_\epsilon)^{-1} \cdot
    X_H, TL_\epsilon \cdot v)\\
  &= \epsilon^2 (L_\epsilon^*{\omega}_0) \big( L_\epsilon^* X_H, v)\\
  &=  {\omega}_0(L_\epsilon^* X_H , v),
  \end{align*}
   from which we see that \(X_{H_\epsilon} = L_\epsilon^* X_H\).
Consequently, the flow of the Hamiltonian vector field \(X_H\) on \(M\) is
  not minimal if and only if the flow of the Hamiltonian vector field
  \(X_{H_\epsilon}\) on \(M_\epsilon\) is not minimal.

Next let \(\tilde{\omega}\) denote the Fubini-Study metric on
  \(\mathbb{C}P^2\).
Recall that there exists a holomorphic embedding \(\iota\colon
  \mathbb{C}^2\to \mathbb{C}P^2\) so that \(\Sigma:= \mathbb{C}P^2\setminus
  \iota(\mathbb{C}^2) \) is an embedded complex submanifold holomorphically
  diffeomorphic to \(\mathbb{C}P^1\).
Next, we regard \(\mathbb{R}^4 \simeq \mathbb{C}^2\) so that \(\iota\colon
  \mathbb{R}^4\to \mathbb{C}P^2\) is an embedding.
Then \(\iota^* \tilde{\omega}\) is a symplectic form on \(\mathbb{R}^4\),
  so there exists a open neighborhood \(\mathcal{O}\) of \(0\in
  \mathbb{R}^4\) and a diffeomorphism \(\phi\colon \mathcal{O}\to
  \phi(\mathcal{O}) \subset \mathbb{R}^4\) for which \(\phi^* \omega_0 =
  \iota^* \tilde{\omega}\).
Next, we fix \(\epsilon>0\) sufficiently small so that
  \(H_\epsilon^{-1}(0) \subset \phi(\mathcal{O})\), and we define the open
  set \(\mathcal{U}:= \{p\in \mathbb{R}^4: H_\epsilon(p)< 0\}\subset
  \phi(\mathcal{O})\).
However, we then have \(H_\epsilon^{-1}(0) = \partial\big({\rm
  cl}(\mathcal{U})\big)\), and that \(\phi(\mathcal{O})\) is an open
  neighborhood of \(\mathcal{U}\cup H_\epsilon^{-1}(0) = {\rm
  cl}(\mathcal{U})\) and
  \begin{align*}                                                          
    \iota \circ\phi^{-1}\colon \phi(\mathcal{O}) \to \mathbb{C}P^2
    \end{align*}
  is an embedding for which \((\iota\circ \phi^{-1})^* \tilde{\omega} =
  \omega_0 \).
We then define the manifold \(\widetilde{W}\) with smooth boundary by
  \(\widetilde{W}:= \mathbb{C}P^2 \setminus \iota \circ \phi^{-1}
  (\mathcal{U})\) and equip it with the symplectic form
  \(\tilde{\omega}\).
We then define the (not yet oriented) manifold \(M^-:= \partial
  \widetilde{W}\), and equip it with the two-form \(\omega: =
  i^*\tilde{\omega}\) where \(i:M^-\hookrightarrow \widetilde{W}\) is the
  canonical inclusion.
Letting \(X:= X_{H_\epsilon}\big|_{H_\epsilon^{-1}(0)}\), we see that by
  construction we have
  \begin{align*}                                                            
    \omega \big( (\iota \circ \phi^{-1})_* X ,\cdot \big) = 0
    \end{align*}
  and \(X\) never vanishes, so that we can define the one-form \(\lambda^-\) via 

  \begin{align*}                                                            
    \lambda^-(X) = 1 \qquad\text{and}\qquad {\rm ker}\; \lambda^- = TM^-
    \cap (J TM^-)
    \end{align*}
  where \(\widetilde{J} \in \Gamma({\rm End} (T\widetilde{W}))\) is the
  (almost) complex structure induced on the \emph{real} manifold
  \(\widetilde{W}\) from multiplication by \(i=\sqrt{-1}\) on
  \(\widetilde{W} \subset \mathbb{C}P^2\) when regarded as a
  \emph{complex} manifold.
We note that \(\widetilde{J}\) is \(\tilde{\omega}\)-compatible in the
  sense that \(\tilde{\omega}( \cdot, \widetilde{J}\cdot)\) is a Riemannian
  metric.
Defining \(\eta^-:= (\lambda^-, \omega^-)\), we see that \((M^-, \eta^-)\)
  is a framed Hamiltonian manifold, with \(X_{\eta^-} = (\iota \circ
  \phi^{-1})_*X\), and hence the flow of \(X_{\eta^-}\) on \(M^-\) is
  minimal if and only if the flow of \(X_H\) on \(H^{-1}(0)\) is minimal.

We pause to sum up the salient features of our geometric
  construction.
Given our compact regular energy surface \(H^{-1}(0)\subset \mathbb{R}^4\)
  inside \((\mathbb{R}^4, \omega_0)\), we have constructed a connected
  symplectic cobordism \((\widetilde{W}, \tilde{\omega})\), in the sense
  of Definition \ref{DEF_symplectic_cobordism}, from the empty framed
  Hamiltonian manifold to \((M^-, \eta^-)\), where the flow of
  \(X_{\eta^-}\) is conjugated to that of \(X_H\) on \(H^{-1}(0)\).
Consequently, the flow of \(X_H\) on \(H^{-1}(0)\) is minimal if and only
  if the flow of \(X_{\eta^-}\) is minimal on \(M^-\).
Moreover, \(\widetilde{W}\setminus \partial \widetilde{W}\) contains an
  embedded degree one sphere, \(\widetilde{\Sigma}\), with the property
  that its tangent planes are \(\widetilde{J}\) invariant.
Or in other words, there exists an embedded pseudoholomorphic curve
  \begin{align*}                                                          
    \tilde{\mathbf{u}} = (\tilde{u}, S^2, j, \widetilde{W}, \widetilde{J},
    \emptyset, \emptyset)
    \end{align*}
  with the property that \(\tilde{u}(S^2) =
  \widetilde{\Sigma}\subset \widetilde{W} \subset \mathbb{C}P^2\).
Finally, we note that  \(H_2(\mathbb{C}P^2, \mathbb{Z}) = \mathbb{Z}\) and
  is generated by \(A_{\widetilde{\Sigma}}\), where
  \(A_{\widetilde{\Sigma}}\) is the homology class associated to
  \(\tilde{u}\).
Moreover,
\begin{align*}                                                            
  \frac{1}{\pi} \int_{S^2} \tilde{u}^*\tilde{\omega} =1.
  \end{align*}

At this point, we apply Lemma \ref{LEM_cobordism_to_extended_corbordism},
  which yields an associated symplectic cobordism \(\mathbf{W} =
  (\overline{W}, \bar{\omega}, \overline{J}, \bar{g}, \bar{a},
  \partial_{\bar{a}}, \epsilon)\), in the sense of Definition
  \ref{DEF_symplectic_extended_cobordism}, from the empty framed
  Hamiltonian manifold to \((M^-, \eta^-)\).
Moreover, in light of Remark \ref{REM_adjustment_of_J}, we may assume that
  \(\bar{\mathbf{u}}_0 = (\bar{u}_0, S^2, j, \overline{W}, \overline{J},
  \emptyset, \emptyset) \) is a pseudoholomorphic curve, where
  \(\bar{u}_0:=\Psi^{-1}\circ \tilde{u}\).
We also define \(\overline{\Sigma} = \Psi^{-1}(\widetilde{W}) \subset
  \overline{W}\), and we let \(A_{\overline{\Sigma}}\) denote the homology
  class associated to \(\bar{u}_0\).\\

\noindent{\bf Step 2:} \emph{Automatic transversality and an abundance of
curves.}\\

Here we define moduli spaces of interest and establish properties
  thereof.
Essentially we are interested in the moduli space of degree one spheres in
  \(\overline{W}\), and a certain path connected component thereof, and we
  show these curves are embedded, cut out transversely, and pairwise
  intersect exactly once and do so transversely.
These ideas were all essentially introduced by Gromov in \cite{Gr} and
  have been extensively employed since, however we provide details for
  completeness, heavily referencing Hofer-Lizan-Sikorav \cite{HLS} and
  McDuff-Salamon \cite{MS} for detailed proofs.

We begin by defining \(\widetilde{\mathcal{M}}\) to be the following
  set of pseudoholomorphic curves
  \begin{align*}                                                            
    \widetilde{\mathcal{M}} = \big\{\mathbf{v} &= (v, S^2, j,
    \overline{W}, \overline{J}, \emptyset, \emptyset) : \mathbf{v} \text{
    is a pseudoholomorphic curve, and }\\ &\qquad \mathbf{v} \text{ and
    }\bar{\mathbf{u}}_0\text{ are homologous.}\big\}
    \end{align*}
  where here, as before, \(j\) denotes the standard (almost) complex
  structure on \(S^2\).
We equip \(\widetilde{\mathcal{M}}\) with the \(\mathcal{C}^\infty\)
  topology, and we then let \(\mathcal{M}\) denoted the path connected
  component of \(\widetilde{\mathcal{M}}\) which contains
  \(\bar{\mathbf{u}}_0\).

Recall that a closed pseudoholomorphic map \(u\colon (S, j) \to (W, J)\)
  is a multiple cover provided there exists another Riemann surface
  \((S',j')\), a holomorphic branched covering \(\phi\colon (S, j)\to (S',
  j')\) with degree strictly greater than one, and a pseudoholomorphic map
  \(u'\colon (S', j')\to (W, J)\) for which \(u = u' \circ \phi\).
Also recall that a closed curve is said to be simple whenever it is not
  multiply covered.

\begin{proposition}[embeddedness and transvserse intersections]
  \label{PROP_embeddedness}
  \hfill \\
Let \(\mathbf{W}\) and \(\widetilde{\mathcal{M}}\) be as above. 
Then the following hold.
\begin{enumerate}                                                         
  \item If \(\mathbf{u}\in \widetilde{\mathcal{M}}\) then \(u:S^2\to
    \overline{W}\) is an embedding.
  \item If \(\mathbf{u},\mathbf{v}\in \widetilde{\mathcal{M}}\) with
    \(u(S^2)\neq v(S^2)\) then
    \begin{align*}                                                        
      1= \#\{(\zeta_0, \zeta_1)\in S^2\times S^2 : u(\zeta_0) =
      v(\zeta_1)\},
      \end{align*}
    and these intersections are transverse.   
  \end{enumerate}
\end{proposition}
%
\begin{proof}
As a first step, we recall Theorem 2.6.3 from \cite{MS} which is often
  called ``positivity of intersections.''
Roughly it states that if \((W, J)\) is an almost complex four-manifold,
  and \(A_1, A_2\in H_2(W; \mathbb{Z})\) are homology classes represented by
  simple pseudoholomorphic curves \((u_1, S_1, j_1, W, J, \emptyset,
  \emptyset )\) and \((u_2, S_2, j_2, W, J, \emptyset, \emptyset) \)
  respectively, which have the property that there do not exist non-empty
  open sets \(\mathcal{U}_1\subset S_1\) and \(\mathcal{U}_2\subset S_2\)
  with the property that \(u_1(\mathcal{U}_1) = u_2(\mathcal{U}_2)\), then
  \begin{align}\label{EQ_intersection_inequality}                         
    \delta(\mathbf{u}_1, \mathbf{u}_2) \leq A_1\cdot A_2,
    \end{align}
  where 
  \begin{align*}                                                          
    \delta(\mathbf{u}_1, \mathbf{u}_2) = \#\{(z_1, z_2)\in S_1 \times S_2
    : u_1(z_1) = u_2(z_2)\}.
    \end{align*}
Moreover, we have equality in equation (\ref{EQ_intersection_inequality})
  if and only if the intersections are transverse.
From this, we may immediately draw the following conclusion: If
  \(\mathbf{u}\in \widetilde{\mathcal{M}}\), then \(\mathbf{u}\) is
  simple.
Indeed, if \(\bar{\mathbf{u}}_0\) represents the homology class \(A\), and
  \(\mathbf{u}\in \widetilde{\mathcal{M}}\) were not simple, then there
  would exist a homology class \(B\) represented by a simple
  pseudoholomorphic curve \(\mathbf{u}'\) which would necessarily satisfy
  \begin{align*}                                                          
    0 < \delta(\mathbf{u}', \bar{\mathbf{u}}_0) \leq  B\cdot A < A\cdot A
    = 1,
    \end{align*}
  which is impossible.
As a consequence of this fact, together with unique continuation (see
  Theorem 2.3.2 and Corollary 2.3.3 in Section 2.3 of \cite{MS}), the
  second part of Proposition \ref{PROP_embeddedness} follows immediately.

To prove the first part of Proposition \ref{PROP_embeddedness}, we first
  recall Theorem 2.6.4 from \cite{MS}, namely the adjunction inequality.
It states that if \((W, J)\) is an almost complex four-manifold, and
  \(A\in H_2(M;\mathbb{Z})\) is a homology class represented by a simple
  pseudoholomorphic curve \(\mathbf{u}\), then
  \begin{align*}                                                          
    2\delta(\mathbf{u}) - \chi(S) \leq A\cdot A - c_1(A)
    \end{align*}
  where  
  \begin{align*}                                                          
    \delta(\mathbf{u}) = {\textstyle \frac{1}{2}} \#\{(z_1, z_2)\in S
    \times S :  z_1\neq z_2,\ u(z_1) = u(z_2)\}.
    \end{align*}
    Moreover, equality holds if and only if \(\mathbf{u}\) is an immersion
  with only transverse self-intersections.
For curves \(\mathbf{u}\in \widetilde{\mathcal{M}}\) we have \(\chi(S) =
  \chi(S^2) = 2\), \(A\cdot A = 1\), and \(c_1(A) = 3\).
By definition we must have \(\delta(\mathbf{u})\geq 0\), and by the
  adjunction inequality we must then also have \(\delta(\mathbf{u}) \leq
  0\).
It immediately follows that \(\mathbf{u}\) is an embedded
  pseudoholomorphic curve.
This establishes the first part, and hence completes the proof of
  Proposition \ref{PROP_embeddedness}.
\end{proof}

As identified by Gromov in \cite{Gr} and detailed by Hofer-Lizan-Sikorav
  in \cite{HLS}, for a given pseudoholomorphic curve
  \(\mathbf{u} = (u, S^2, j, W, J, \emptyset, \emptyset)\), there exists a
  non-linear partial differential operator denoted \(\bar{\partial}_\nu\)
  called the normal Cauchy-Riemann operator, which is defined on suitably
  small sections of the normal bundle over \(\mathbf{u}\); that is, the
  sub-bundle of \(u^*TW\) consisting of those planes orthogonal to the
  tangent sub-bundle \(TS\subset u^* TW\).
We call the space \(\widetilde{\mathcal{M}}/{\rm Aut}(S^2)\) the space of
  \emph{non-parameterized} curves homologous to \(\bar{\mathbf{u}}_0\), and
  note that a neighborhood of \([\mathbf{u}]\in \widetilde{\mathcal{M}}/{\rm
  Aut}(S^2)\) is given by the zero-set of the normal Cauchy-Riemann operator
  near \(\mathbf{u}\).
The linearization of \(\bar{\partial}_\nu\) at \(\mathbf{u}\), denoted
  \(L_\nu\), is a first order elliptic differential operator of the
  following form
  \begin{align}\label{EQ_linearization}                                   
    L_{\nu_{\mathbf{u}}} = \bar{\partial} + a,
    \end{align}
  where \(a\in \Omega^{0,1}\big({\rm End}_\mathbb{R}
  (\nu_{\mathbf{u}})\big)\).
In preparation for a later result, we now claim the following.

\begin{lemma}[only one zero]
  \label{LEM_only_one_zero}
  \hfill\\
Let \(\mathbf{u}\in \widetilde{\mathcal{M}}\), and let
  \(L_{\nu_{\mathbf{u}}}\) be the linearization of
  \(\bar{\partial}_{\nu}\) at \(\mathbf{u}\) as above.
Let \(0\neq \sigma\in {\rm ker}(L_{\nu_{\mathbf{u}}})\).  
Then 
\begin{align*}                                                            
  1 = \#\{z\in S^2: \sigma(z) =0\}.
  \end{align*}
\end{lemma}
%
\begin{proof}
First observe that \(c_1(\nu_{\mathbf{u}}) = 1\), where
  \(c_1(\nu_{\mathbf{u}})\) is the first Chern number of the normal bundle
  \(\nu_{\mathbf{u}}\) over \(\mathbf{u}\); that is, it is the algebraic
  count of zeros of a generic section of \(\nu\).
Consequently
  \begin{align*}                                                          
    1 \leq \#\{z\in S^2: \sigma(z) =0\}.
    \end{align*}
Next we claim the zeros of \(\sigma\) are isolated and each contributes
  positively to \(c_1(\nu_{\mathbf{u}})\).
Indeed, this follows from the form \(L_{\nu_{\mathbf{u}}}\) takes,
  specifically equation (\ref{EQ_linearization}), together with the Carleman
  similarity principle.\footnote{See Theorem 2.3.5 of \cite{MS}.}
From this we conclude
  \begin{align*}                                                          
    1 \geq \#\{z\in S^2: \sigma(z) =0\}.
    \end{align*}
The desired result is immediate.  
This completes the proof of Lemma \ref{LEM_only_one_zero}.

\end{proof}

Returning to our discussion of the linearized operator, we recall that
  \(L_{\nu_{\mathbf{u}}}\) is Fredholm for suitable choices of Banach
  spaces; for example H\"{o}lder spaces \(\mathcal{C}^{k, \alpha}\) or
  Sobolev spaces \(W^{k, p}\) with \(k\geq 1\) and \(p>2\).
Moreover the index of \(L_\nu\) at a curve \(\mathbf{u}=(u, S, j, W, J,
  \emptyset, \emptyset)\), is given by
  \begin{align*}                                                          
    {\rm Ind}(L_{\nu_{\mathbf{u}}}) = 2\big( c_1(\nu_{\mathbf{u}}) + 1 -
    {\rm Genus}(S)\big)
    \end{align*}
  where \(c_1(\nu_{\mathbf{u}})\) is the first Chern number of the normal
  bundle \(\nu_{\mathbf{u}}\) over \(\mathbf{u}\), or equivalently if
  \(\mathbf{u}\) is embedded, then
  \begin{align*}                                                          
    c_1(\nu_{\mathbf{u}}) = \mathbf{u} \cdot \mathbf{u}, 
    \end{align*}
  where \(\mathbf{u} \cdot \mathbf{u}\) is the self-intersection number of
  \(\mathbf{u}\).
One of the main results of \cite{HLS} is that if \(c_1(\nu_{\mathbf{u}})
  \geq 2{\rm Genus}(S) - 1\) then \(L_{\nu_\mathbf{u}}\) is surjective.
As a consequence of Proposition \ref{PROP_embeddedness}, each
  \(\mathbf{u}\in \widetilde{\mathcal{M}}\) is embedded with \({\rm
  Genus}(S) = 0\) so that indeed \(c_1(\nu_{\mathbf{u}}) = \mathbf{u}\cdot
  \mathbf{u} = 1\) and hence \(L_{\nu_{\mathbf{u}}}\) is surjective with
  \({\rm Ind}(L_\nu) = 4\).
By the implicit function theorem on Banach spaces it follows that a
  neighborhood of \([\mathbf{u}]\in \widetilde{\mathcal{M}}/{\rm Aut}(S^2)\)
  is a manifold of dimension \(4\).
Moreover, for each integer \(k\geq 1\) there exists a convex open
  neighborhood \(\mathcal{O} \subset {\rm ker}(L_{\nu_{\mathbf{u}}})\) of
  the zero section and a smooth embedding \(E:\mathcal{O}\to
  \mathcal{C}^{k, \alpha}(\nu_{\mathbf{u}})\) with the following
  properties.
\begin{enumerate}[($\mathcal{F}$1)]                                       
  \item \label{EN_FF1}
  For each \(\sigma\in \mathcal{O}\), \(E(\sigma) \in
  \mathcal{C}^\infty(\nu_{\mathbf{u}})\); that is, \(E(\sigma)\) is a
  smooth section of the normal bundle \(\nu_{\mathbf{u}}\) over
  \(\mathbf{u}\).
  \item \label{EN_FF2}
  The linearization of \(E\) at the zero section, denoted by \(TE_0\),
  satisfies \(T E_0(\sigma) = \sigma\) for all \(\sigma\in  {\rm
  ker}(L_{\nu_{\mathbf{u}}})\).
  \item \label{EN_FF3}
  For each \(\sigma\in \mathcal{O}\), there exists a pseudoholomorphic
  curve 
  \begin{align*}                                                          
    \mathbf{u}_\sigma = (u_\sigma, S^2, j, \overline{W}, J, \emptyset,
    \emptyset)
    \end{align*}
  and a (not necessarily holomorphic) diffeomorphism
  \(\psi_\sigma: S^2\to S^2\) for which
  \begin{align*}                                                          
    u_\sigma \circ \psi_\sigma = \exp\big(E(\sigma)\big).
    \end{align*}
  Moreover, for a continuous path \([0,1]\to \mathcal{O}\) denoted by
  \(\tau\mapsto \sigma_\tau\), the \(\psi_{\sigma_\tau}\) can be found so
  that the map \([0, 1]\to \mathcal{C}^\infty(S^2, \overline{W})\) given
  by 
  \begin{align*}                                                            
    \tau\mapsto u_{\sigma_\tau}:= {\rm exp}_{u\circ
    \psi_{\sigma_\tau}^{-1}}\big(E(\sigma_\tau)\big)
    \end{align*}
  is continuous.
  \item \label{EN_FF4}
  The map 
  \begin{align*}                                                          
    &F\colon \mathcal{O}\times S^2\to \overline{W}
    \\
    &F(\sigma, z) = \exp_{u(z)}\big(E(\sigma)\big)
    \end{align*}
  is \(\mathcal{C}^\infty\) smooth.
  \end{enumerate}

\begin{remark}[paths of reparametrizations] 
  \label{REM_paths_of_reparametrizations}
  \hfill\\
To see the validity of the second part of property
  (\(\mathcal{F}\)\ref{EN_FF3}), one first observes that after choosing
  three distinct points \(\{z_1, z_2, z_3\}\subset S^2\) and \(\sigma\in
  \mathcal{O}\), the diffeomorphism \(\psi_\sigma\) is uniquely determined
  by requiring \(\psi_\sigma(z_i)=z_i\) for \(i\in \{1, 2, 3\}\).
Given \(\tau\mapsto \sigma_\tau\), this then uniquely determines the map
  \(\tau\mapsto \psi_{\sigma_\tau}\), for our choice of \(z_i\).
Continuity of the map \(\tau\mapsto u_{\sigma_\tau} := {\rm
  exp}_{u\circ \psi_{\sigma_\tau}^{-1}}\big(E(\sigma_\tau)\big)\) then follows
  from the continuity of the map \(\tau\mapsto \psi_{\sigma_\tau}\), which
  essentially follows from Gromov compactness; here it may be helpful to
  recall that \(\tau \mapsto j_\tau : = (u_{\sigma_\tau}\circ
  \psi_{\sigma_\tau})^* J\) is a continuous map into the space of smooth
  sections \(\Gamma\big({\rm End}(TS^2)\big)\).
\end{remark}
%

With these facts recalled, we are now prepared to prove the following.

\begin{lemma}[curve through nearby points]
  \label{LEM_curve_through_nearby_points}
  \hfill\\
Let \(\mathbf{W} = (\overline{W}, \bar{\omega}, \overline{J}, \bar{g},
  \bar{a}, \partial_{\bar{a}}, \epsilon)\) and \(\widetilde{\mathcal{M}}\)
  be as above, and let \(\mathbf{u}\in \widetilde{\mathcal{M}}\).
Fix \(z_0\in S^2\).
Then there exists \(\delta>0\) with the property that for each \(q\in
  \mathcal{B}_\delta\big(u(z_0)\big)\), there exists a continuous map
  \(h:[0,1]\to \widetilde{\mathcal{M}}\) for which \(h(0) = \mathbf{u}\),
  and \(h(1) = \mathbf{u}_1 = (u_1, S^2, j, \overline{W}, \overline{J},
  \emptyset, \emptyset)\) with \(q\in u_1(S^2)\).
\end{lemma}
%
\begin{proof}
We begin by letting \(V_{z_0}\) be the fiber of \(\nu_{\mathbf{u}}\) over
  the point \(z_0\in S^2\).  
Next we claim that for each point \(v\in V_{z_0}\) there exists
  \(\sigma\in {\rm ker}(L_{\nu_{\mathbf{u}}})\) such that \(\sigma(z_0) =
  v\).
To see this, suppose not. 
Then there exists a vector subspace \(Q\subset {\rm
  ker}(L_{\nu_{\mathbf{u}}})\) of dimension at least three for which
  \(\sigma\in Q\) implies \(\sigma(z_0) = 0\).
Because \(Q\) is at least three dimensional, it follows that there exists
  \(z_1\in S^2\setminus \{z_0\}\) and \(0\neq \sigma \in Q\) for which
  \(\sigma(z_1) = 0 = \sigma(z_0)\).
However this contradicts Lemma \ref{LEM_only_one_zero}.
This contradiction then establishes that indeed, for each point \(v\in
  V_{z_0}\), there exists \(\sigma\in {\rm ker} (L_{\nu_{\mathbf{u}}})\) for
  which \(\sigma(z_0) = v\).

In light of this observation, we choose \(\check{\sigma},
  \hat{\sigma}\in \mathcal{O} \) with \(\mathcal{O}\subset {\rm
  ker}(L_{\nu_{\mathbf{u}}})\) as above, so that \(V_{z_0} = {\rm
  Span}\big(\check{\sigma}(z_0), \hat{\sigma}(z_0)\big)\).
We then fix local coordinates \((s, t)\) centered at \(z_0\in S\) in a
  neighborhood \(\mathcal{U}\subset S\) and we then define the map
  \begin{align*}                                                          
    &\mathcal{F}\colon \mathcal{U}\times \mathcal{D}\to \overline{W}
    \\
    &\mathcal{F}(s,t, x, y) = {\rm
    exp}_{u(s,t)}\big(E\big(x\check{\sigma} + y\hat{\sigma})\big)
    \end{align*}
  where \(\mathcal{D} = \{(x, y)\in \mathbb{R}^2: x^2 +y^2 < 1\}\).
By property (\(\mathcal{F}\)\ref{EN_FF4}) the map \(\mathcal{F}\) is
  smooth.
Moreover, by property (\(\mathcal{F}\)2) the linearization
  \(T\mathcal{F}(0, 0, 0, 0)\) is surjective.
Consequently, there exists \(\delta>0\) such that
  \(\mathcal{B}_\delta(u(z_0))\subset \mathcal{F}(\mathcal{U}\times
  \mathcal{D})\).
Letting \(q\in \mathcal{B}_\delta(u(z_0))\), there exists \((s, t, x,
  y)\in \mathcal{U}\times \mathcal{D}\) so that \(\mathcal{F}(s, t, x, y) =
  q\).
Define \(\sigma:= x\check{\sigma} + y\hat{\sigma}\).
By convexity of \(\mathcal{O}\), it follows that \(\tau \sigma \in
  \mathcal{O}\) for all \(\tau\in [0,1]\).
Consequently, by property (\(\mathcal{F}\)\ref{EN_FF3}), the continuous
  path \([0, 1]\to \mathcal{O}\) given by \(\tau\mapsto \sigma_\tau: =
  \tau\sigma\) gives rise to a continuous map
  \begin{align*}                                                          
    &h\colon [0, 1]\to \widetilde{\mathcal{M}}
    \\
    &h(\tau)= \mathbf{u}_{\sigma_\tau} = (u_{\sigma_\tau}, S^2, j,
    \overline{W}, \overline{J}, \emptyset, \emptyset)
    \end{align*}
  for which \(h(0) = \mathbf{u}\), and \(h(1) = \mathbf{u}_{\sigma}\) with
  \(q\in u_{\sigma}(S^2)\).
This completes the proof of Lemma \ref{LEM_curve_through_nearby_points}.
\end{proof}

\noindent{\bf Step 3:} \emph{The moduli space $\mathcal{M}$ extends into
the negative end of $\overline{W}$}.\\

Here we aim to prove the following proposition.

\begin{proposition}[curves fall completely]
  \label{PROP_curves_fall_complete}
  \hfill \\
Let \(\mathbf{W} = (\overline{W}, \bar{\omega}, \overline{J}, \bar{g},
  \bar{a}, \partial_{\bar{a}}, \epsilon)\) and \(\mathcal{M}\) be as above.
Then for each \(a_0 \leq -1\), there exists \(z_0\in S^2\) and
  \(\mathbf{u}\in \mathcal{M}\) for which 
  \begin{align*}                                                          
    a_0 = a\circ u(z_0) = \inf_{z\in S^2} a\circ u(z).
    \end{align*}
In other words, the images of the curves in \(\mathcal{M}\) extend as far
  down as we like into \((-\infty, -1)\times M^- \subset \overline{W}\).
\end{proposition}
%
\begin{proof}

In order to proceed, we will need the following result.

\begin{lemma}[bounded depth implies bounded area]
  \label{LEM_bounded_depth_implies_bounded_area}
  \hfill\\
Let \(\mathbf{W} = (\overline{W}, \bar{\omega}, \overline{J}, \bar{g},
  \bar{a}, \partial_{\bar{a}}, \epsilon)\) and \(\mathcal{M}\) be as above.
For each \(a_0\leq -1\), there exists a \(C = C(a_0,
  \mathbf{W}, C_{\bar{\mathbf{u}}_0} )\)  with the following
  property.
For each \(\mathbf{u}\in \mathcal{M}\) for which \(u(S^2)\subset \{p \in
  \overline{W}: \bar{a}(p) > a_0\}\), we have
  \begin{align*}                                                          
    {\rm Area}_{u^*\bar{g}} (S^2) \leq C.
    \end{align*}
\end{lemma}
%
\begin{proof}
We begin by fixing \(\delta\) so that  \(-1+\frac{1}{8}\epsilon\leq \delta
  \leq -1+\frac{1}{4}\epsilon\) so that \(\delta\) is a regular value of
  \(a\circ u\).
We then observe that 
  \begin{align*}                                                          
    {\rm Area}_{u^*\bar{g}}(S^2) = {\rm Area}_{u^*\bar{g}}(S^+) +{\rm
    Area}_{u^*\bar{g}}(S^-)
    \end{align*}
  where 
  \begin{align*}                                                          
    S^+ = \{\zeta \in S^2 : \bar{a}\circ u (\zeta) \geq \delta\}
    \qquad\text{and}\qquad 
    S^- =  \{\zeta \in S^2 : \bar{a}\circ u (\zeta) \leq \delta\}.
    \end{align*}
Recall our definition of cylindrical end and core of \(\overline{W}\) are
  adapted from equations (\ref{EQ_core_W}) and (\ref{EQ_cyl_W}) to our
  case in which
  \(M^+=\emptyset\) as follows:
  \begin{align*}                                                          
    {\rm Core}(\overline{W}) = \{\bar{a} \geq -1 +{\textstyle
    \frac{1}{8}\epsilon}\}\qquad\qquad {\rm Cyl}^-(\overline{W})= \{
    \bar{a} < -1+{\textstyle \frac{1}{4}}\epsilon \}.
    \end{align*}
From this we immediately see that 
  \begin{align*}                                                          
    u(S^+) \subset {\rm Core}(\overline{W})\qquad\text{and}\qquad
    u(S^-)\subset {\rm Cyl}^-(\overline{W}).
    \end{align*}
Letting \(C_\theta = C_\theta(\mathbf{W})\) denote the constant
  guaranteed by the final part of Remark
  \ref{REM_structureal_observations}, we see immediately from equation
  (\ref{EQ_omega_coercive}) that
  \begin{align*}                                                          
    {\rm Area}_{u^*g}(S^+) \leq C_\theta \int_{S^+}u^*\bar{\omega} 
    \leq C_\theta \int_{S^2} u^*\bar{\omega}
    \leq C_\theta \int_{S^2} \bar{u}_0^*\bar{\omega}
    = \pi C_\theta.
    \end{align*}
Somewhat similarly, we have \(u(S^-)\subset {\rm Cyl}^-(\overline{W})\),
  and by Remark \ref{REM_structureal_observations}, we see that \({\rm
  Cyl}^-(\overline{W})\) has the structure of a realized Hamiltonian
  homotopy with suitably adapted almost Hermitian structure.
Consequently, Theorem \ref{THM_area_bounds_homotopy} below guarantees
  the existence of a constant \(C_A = C_A(a_0, \mathbf{W}
  )\) for which
  \begin{align*}                                                          
    {\rm Area}_{u^*\bar{g}}(S^-)\leq C_A.
    \end{align*}
Combining these two inequalities yields
  \begin{align*}                                                          
    {\rm Area}_{u^*\bar{g}}(S^2) \leq C_A + C_\theta =: C,
    \end{align*}
  which completes the proof of Lemma
  \ref{LEM_bounded_depth_implies_bounded_area}.
\end{proof}

Before proceeding with the proof of Proposition
  \ref{PROP_curves_fall_complete}, we state Theorem
  \ref{THM_area_bounds_homotopy}. 
The proof is provided in Section \ref{SEC_proof_of_exp_area_bounds}.

\setcounter{CounterSectionAreaHomotopy}{\value{section}}
\setcounter{CounterTheoremAreaHomotopy}{\value{theorem}}

\begin{theorem}[area bounds in realized Hamiltonian homotopy]\hfill \\
Fix positive constants \(C_H>0\), \(r>0\), and \(E_0>0\). 
Then there exists a constant \(C_A=C_A(C_H, r, E_0)\) with the following
  significance.
Let \(\mathcal{I}\times M, (\hat{\lambda}, \hat{\omega}))\) denote a
  realized Hamiltonian homotopy in the sense of Definition
  \ref{DEF_hamiltonian_homotopy}, and let \((J, g)\) be an adapted almost
  Hermitian structure in the sense of Definition
  \ref{DEF_adapted_structures_Ham_homotopy} with
  \begin{align*}                                                            
    C_{\mathbf{H}}:= \sup_{q\in \mathcal{I} \times M}
    \|d\hat{\lambda}_q\|_g \leq C_H.
    \end{align*}
For each proper pseudoholomorphic
  map \(u\colon S\to \mathcal{I}_r\times M\), where 
\begin{align*}                                                            
  \mathcal{I}_r=(a_0 -r, a_0 +r)\subset \mathcal{I}
  \end{align*}
  for which \(\partial S = \emptyset\),
  \(u^{-1}(\{a_0\}\times M)=\emptyset\), and
  \begin{equation*}                                                       
    \int_S u^*\omega \leq E_0 <\infty,
    \end{equation*}
  the following also holds:
  \begin{equation*}                                                       
    {\rm Area}_{u^*g}(S) = \int_{S} u^* (da \wedge \hat{\lambda} +
    \hat{\omega}) \leq C_A.
    \end{equation*} 
Additionally, for any \([ a_0, a_1] \subset \mathcal{I}\) and any compact
  pseudoholomorphic map \(u: S\to [a_0,a_1]\times M\) for which \(a_0\) and
  \(a_1\) are regular values of \(a\circ u\) and
  \(u^{-1}\big(\{a_0,a_1\}\times M\big) = \partial S\), the following also
  hold:
  \begin{align*}                                                          
    \int_{\Gamma_{a_0}} u^*\lambda\leq
    \Big(C_{H} E_0+ \int_{\Gamma_{a_1}}
    u^*\lambda\Big) e^{C_H (a_1-a_0)},
    \end{align*}
  and
  \begin{align*}                                                          
    \int_{\Gamma_{a_1}} u^*\lambda\leq
    \Big(C_{H} E_0+ \int_{\Gamma_{a_0}}
    u^*\lambda\Big) e^{C_H (a_1-a_0)},
    \end{align*}
  where \( \Gamma_{a_i} = (a\circ u)^{-1}(a_i) \) for \(i \in \{0, 1\}\).
Similarly, for 
  \begin{align*}                                                          
    \ell: = \min_{i\in \{0, 1\}} \Big\{ \int_{\Gamma_{a_i}} u^*\lambda
    \Big\}
    \end{align*}
  we have
  \begin{align*}                                                            
    {\rm Area}_{u^*g} (S) \leq (C_H^{-1} \ell +E_0 ) (e^{C_H (a_1-a_0)}
    -1)+E_0.
    \end{align*}
\end{theorem}
%

With Lemma \ref{LEM_bounded_depth_implies_bounded_area} established, we
  can now complete the proof of Proposition
  \ref{PROP_curves_fall_complete}.
Indeed, we do this by contradiction, and hence begin by assuming that
  Proposition \ref{PROP_curves_fall_complete} is false.
In this case, Lemma \ref{LEM_bounded_depth_implies_bounded_area}
  guarantees that the curves in \(\mathcal{M}\) have uniformly bounded
  area.
We define the set \(\mathcal{U}\subset \overline{W}\) by the following:
  \begin{align*}                                                          
    \mathcal{U}= \{\bar{q}\in \overline{W}: \exists \; \mathbf{u}\in
    \mathcal{M}\; \text{s.t.}\; u(z) = \bar{q}\}.
    \end{align*}
\noindent{\bf Claim 1:} The set \(\mathcal{U}\) is open.\\

This follows immediately from Lemma \ref{LEM_curve_through_nearby_points}. \\

\noindent{\bf Claim 2:} The set \(\mathcal{U}\) is closed.\\

To see this, we take a sequence \(\bar{q}_k\to \bar{q}\) with
  \(\{\bar{q}_k\}_{k\in \mathbb{N}} \subset \mathcal{U}\).
Then there exist \(\mathbf{u}_k\in \mathcal{M} \) with \(\bar{q}_k\in
  u_k(S^2)\), which have uniformly bounded area and genus. 
By Theorem \ref{THM_target_local_gromov_compactness}, a sub-sequence
  converges to the stable curve
  \begin{align*}                                                          
    \mathbf{u}= \big(u, S, j, \overline{W}, \overline{J}, \emptyset , D\big)
    \end{align*}
  with \(\bar{q}\in u(S)\) and \(S = \sqcup_{i=1}^{n+1} S^2\), for some
  \(n\geq 0\).\\

\noindent\emph{Case I:} \(\overline{\Sigma}\subset u(S)\).  \\
\indent In this case there must exist a connected component \(S_0\subset S\) for
  which \(u(S_0)= \overline{\Sigma}\).
If \(S = S_0\), then \(D=\emptyset\), and hence \(\mathbf{u}\in
  \mathcal{M}\), and we are done, so assume \(S\neq S_0\).
In this case, we denote any remaining connected components by \(S_1,
  \ldots, S_n\).
By stability of \(\mathbf{u}\), we may re-order the \(S_i\) so that
  \(u:S_1\to \overline{W}\) is not a constant map and \(u(S_1)\cap
  \overline{\Sigma}\neq \emptyset\).
However, in a neighborhood of \(\overline{\Sigma}\), the two-form
  \(\bar{\omega}\) is symplectic and evaluates positively on
  \(\overline{J}\)-complex lines.
Consequently \(\int_{S_1}u^*\bar{\omega} >0\), and because
  \(\bar{\omega}\) evaluates non-negatively on \(\overline{J}\)-complex
  lines in general, it follows that
  \begin{align*}                                                          
    \int_{S} u^*\bar{\omega} > \int_{S^2} \bar{u}_0^*\bar{\omega}
    \end{align*}
  which is impossible since \(\mathbf{u}\) and \(\bar{\mathbf{u}}_0\)
  represent the same homology class.
This contradiction establishes that \(S=S_0\), and hence \(\mathbf{u}\in
  \mathcal{M}\).\\

\noindent\emph{Case II:} \(\overline{\Sigma}\not\subset u(S)\).  \\
\indent By positivity of intersections, there exists exactly one connected
  component \(S_0\) of \(S\) for which \(u(S_0)\cap \overline{\Sigma} \neq
  \emptyset\).
As before, if \(S_0=S\) then we are done, so we consider the case that
  \(S\) has other connected components which we denote \(S_1, \ldots, S_n\).
Next we note that because \(\bar{\omega}\) evaluates non-negatively on
  \(\overline{J}\)-complex lines, it follows that 
\begin{align*}                                                            
  \int_{S_i} u^*\bar{\omega} \geq 0\qquad \text{for }i\in \{0, \ldots n\}.
  \end{align*}
As before, it follows from stability of \(\mathbf{u}\) that we may reorder
  the \(S_i\) so that \(u: S_1\to \overline{W}\) is non-constant.
Observe that by unique continuation, we must have \(\int_{S_1}
  u^*\bar{\omega}>0\).
However, because \(\bar{\omega} = \Psi^* \tilde{\omega}\), and because
  \(\int_{S_i} u^*\bar{\omega} >0\) for \(i\in \{0, 1\}\), it follows that
  \(\int_{S_i}(\Psi\circ u)^*\tilde{\omega} \geq 1\) for \(i\in \{0,1\}\),
  and hence
  \begin{align*}                                                          
    1 &= \frac{1}{\pi}\int_{S}  u^* \bar{\omega}
    \\
    &=\sum_{i=0}^n \frac{1}{\pi}\int_{S_i}  u^* \bar{\omega}
    \\
    &\geq \frac{1}{\pi}\int_{S_0}  u^*
    \bar{\omega}+\frac{1}{\pi}\int_{S_1}  u^* \bar{\omega}
    \\
    &=\frac{1}{\pi}\int_{S_0}  (\Psi \circ u)^*
    \tilde{\omega}+\frac{1}{\pi}\int_{S_1}  (\Psi\circ u)^* \tilde{\omega}
    \\
    &\geq 2.
    \end{align*}.
This contradiction establishes that \(S=S_0\), and hence \(\mathbf{u} \in
  \mathcal{M}\).
This completes Claim 2.
At this point we realize that \(\mathcal{U}\) is both open and closed, and
  hence must equal \(\overline{W}\), which is impossible. 
This contradiction then completes the proof of Proposition
  \ref{PROP_curves_fall_complete}.
\end{proof}

\noindent{\bf Step 4:} \emph{An area estimate}.\\

Here we prove the following.

\begin{lemma}[ad hoc area estimate]
  \label{LEM_ad_hoc_area_estimate}
  \hfill \\
There exists a \(C=C(\mathbf{W})>0\) with the following significance.
Let \(\mathbf{u}\in \mathcal{M}\). 
Define 
  \begin{align*}                                                          
    \widetilde{S} = (a \circ u)^{-1}(\widetilde{\mathcal{I}})
    \end{align*}
  where \(\mathcal{I}\) is the interval \(\widetilde{\mathcal{I}}=(-2, -1)\).
Then 
\begin{align*}                                                            
  {\rm Area}_{u^*g} (\widetilde{S}) \leq C.
  \end{align*}
\end{lemma}
%
\begin{proof}
For notational convenience, we define the interval \(\mathcal{I}' =
  \big(\frac{1}{8}\epsilon -1 , \frac{1}{4}\epsilon -1\big)\).
We then define the region
  \begin{align*}                                                          
    \widehat{W} = \big\{\bar{q} \in \overline{W} : \bar{a}(\bar{q}) \in
    \mathcal{I}'\big\}.
    \end{align*}
We then recall that because \(\mathbf{W}\) is an extended symplectic
  cobordism in the sense of Definition
  \ref{DEF_symplectic_extended_cobordism}, it follows that on
  \(\overline{W}\) we have
  \begin{align*}                                                          
    \bar{\omega} = \Big( \beta'(\bar{a}) (d\bar{a}\wedge \lambda^-)\Big) +
    \Big(\omega^- + \beta(\bar{a}) d\lambda^- \Big)
    \end{align*}
  where we define the positive constants \(c_\beta\) and \(c_\beta'\) by 
  \begin{align*}                                                          
    0 < c_\beta := \inf_{\bar{a}\in \mathcal{I}'} \beta(\bar{a}) 
    \qquad\text{and}\qquad
    0 < c_\beta' := \inf_{\bar{a}\in \mathcal{I}'} \beta'(\bar{a}) .
    \end{align*}
Because \(\overline{J}\) preserves the kernel of each of \(d \bar{a}\wedge
  \lambda^-\) and \(\omega^- +\beta d\lambda^-\), and because each of
  these two-forms evaluates non-negatively on \(\overline{J}\)-complex
  lines, the following holds:
  \begin{align*}                                                          
    \int_{u^{-1}(\widehat{W})} u^* (d\bar{a}\wedge \lambda^-)
    &\leq 
    c_\beta'^{-1} \int_{u^{-1}(\widehat{W})} u^* (\beta' d\bar{a}\wedge
    \lambda^-)
    \\
    &\leq 
    c_\beta'^{-1} \int_{u^{-1}(\widehat{W})} u^* (\beta' d\bar{a}\wedge
    \lambda^-)  + u^* (\omega^- + \beta d\lambda^-)
    \\
    &\leq 
    c_\beta'^{-1} \int_{u^{-1}(\widehat{W})} u^*\bar{\omega}
    \\
    &\leq c_\beta'^{-1} \pi.
    \end{align*}

However, by the co-area formula\footnote{See Lemma \ref{LEM_coarea_lambda}
  below for a precise statement and proof of the required version of the
  co-area formula.} we also have
  \begin{align*}                                                          
    \int_{u^{-1}(\widehat{W})} u^*(d\bar{a}\wedge \lambda^-) =
    \int_{\mathcal{I}'} \Big(\int_{\Gamma_t} u^*\lambda^-\Big) \; dt, 
    \end{align*}
  where \(\Gamma_t = (a\circ u)^{-1}(t)\) for each regular value \(t\) of
  \(a\circ u\).
Combining the above two observations then guarantees the existence a regular
  value \(t_0\) of \(a\circ u\) satisfying \(\frac{1}{8}\epsilon -1 < t_0<
  \frac{1}{4}\epsilon - 1\), and with the property that
  \begin{align*}                                                          
    \int_{\Gamma_{t_0}} u^*\lambda^- \leq \frac{16\pi }{\epsilon c_\beta'}.
    \end{align*}
In particular, \(\int_{\Gamma_{t_0}} u^*\lambda^-\) is bounded in terms of
  the geometry of \(\mathbf{W}\).
However, we then note that \(\widehat{W}\subset {\rm
  Cyl}^-(\overline{W})\), and by Remark \ref{REM_structureal_observations}
  we recall that \({\rm Cyl}^-(\overline{W})\) has the structure of a
  realized Hamiltonian homotopy.
Consequently Theorem \ref{THM_area_bounds_homotopy} applies, which
  guarantees the existence of a constant \(C = C(\mathbf{W})>0 \) so that
  \begin{align*}                                                          
    {\rm Area}_{u^*g}(\widetilde{S})
    \leq {\rm Area}_{u^*g}\big( (a\circ u)^{-1} \big((-2, t_0)\big)\big)
    \leq C.
    \end{align*}
This is the desired inequality which proves Lemma
  \ref{LEM_ad_hoc_area_estimate}.
\end{proof}

\noindent{\bf Step 5:} \emph{Trimming curves and applying the workhorse
  theorem}.\\

In order to complete the proof of Theorem \ref{THM_main_result}, we will
  apply Theorem \ref{THM_existence} to a collection of curves which we now
  construct from curves in \(\mathcal{M}\).
The rough idea is to carefully trim curves from \(\mathcal{M}\) so that we
  may regard the resulting compact curves with boundary as having images in
  the translation invariant region of \({\rm Cyl}^-(\overline{W})\) in a
  manner that Theorem \ref{THM_existence} applies.
After reviewing the hypotheses of Theorem \ref{THM_existence}, the main
  concern becomes how to trim the curves so the boundary of the domains
  have images in \((-2, -1)\times M^- \subset {\rm Cyl}^-(\overline{W})\),
  and so that the number of boundary components stays bounded.
To that end, we will need the following result.

\begin{lemma}[bounds on number of boundary components]
  \label{LEM_bounds_on_number_of_boundary_components}
  \hfill\\
Let \((W, J, g)\) be a compact almost Hermitian manifold with smooth
  boundary, and let \((J_k, g_k)\) be a sequence of almost Hermitian
  structures which converge in \(\mathcal{C}^\infty(W)\) to \((J, g)\).
Let \(I\) be an index set, possibly uncountable, and denote the interior
  of \(W\) by \(W^0 := {\rm Int} (W)\).
Suppose there exists a constant \(C>0\), and a set of stable proper
  pseudoholomorphic curves
  \begin{equation*}                                                       
    \mathbf{u}_{k, \iota} =(u_{k, \iota}, S_{k, \iota}, j_{k, \iota}, W^0,
    J_k, \mu_{k, \iota}, D_{k, \iota})
    \end{equation*}
  which satisfy
  \begin{enumerate}                                                       
    \item 
    \({\rm Area}_{u_{k, \iota}^* g_k}(S_{k,\iota}) < C\)
    \item 
    \({\rm Genus}(S_{k, \iota}) < C\) 
    \item
    \(\# (\mu_{k, \iota} \cup D_{k, \iota})\leq C \)
  \end{enumerate}
Then for each sufficiently small \(\delta>0\), there exists another
  constant \(C'=C'(\delta)>0\) with the following property.
For each \((k, \iota)\in \mathbb{N}\times I\), there exists a compact
  two-dimensional submanifold (possibly with smooth boundary)
  \(\widetilde{S}_{k, \iota}\subset S_{k, \iota}\) with the property that 
  \begin{equation*}                                                       
    \sup_{\zeta \in S_{k, \iota}\setminus \widetilde{S}_{k, \iota}} {\rm
    dist}_g\big(u_{k, \iota}(\zeta), \partial W\big) \leq \delta
    \end{equation*}
  and 
  \begin{equation*}                                                       
    \#\pi_0(\partial \widetilde{S}_{k, \iota}) \leq C';
    \end{equation*}
  here \(\#\pi_0(X)\) denotes the number of connected components of \(X\).
\end{lemma}
%
\begin{proof}
We suppose the lemma is not true and aim to derive a contradiction.
To that end, there must exist a sequence \(\ell\mapsto (k_\ell ,
  \iota_\ell)\) with the property that for each compact two-dimensional
  submanifold (possibly with smooth boundary) \(\widetilde{S}_{k_\ell,
  \iota_\ell}\subset S_{k_\ell, \iota_\ell}\) that satisfies
  \begin{equation*}                                                       
    \sup_{\zeta \in S_{k_\ell, \iota_\ell}\setminus \widetilde{S}_{k_\ell,
    \iota_\ell}} {\rm dist}_g\big(u_{k_\ell, \iota_\ell}(\zeta), \partial
    W\big) \leq \delta
    \end{equation*}
  also satisfies 
  \begin{equation*}                                                       
    \#\pi_0(\partial \widetilde{S}_{k_\ell, \iota_\ell}) \geq \ell.
    \end{equation*}
Without loss of generality, we may assume that the map \(\ell\mapsto
  k_\ell\) is either strictly monotonically increasing or else constant.
Here we shall assume \(\ell\mapsto k_\ell\) is strictly monotonic, and
  leave trivial modifications for the constant case to the reader.
Next, for notational convenience, we define a sequence of
  pseudoholomorphic curves by the following.
\begin{equation*}                                                         
  (\hat{u}_\ell, \widehat{S}_\ell, \hat{j}_\ell, W^0, \widehat{J}_\ell,
  \hat{\mu}_\ell, \widehat{D}_\ell):=(u_{k_\ell, \iota_\ell}, S_{k_\ell,
  \iota_\ell}, j_{k_\ell, \iota_\ell}, W^0, J_{k_\ell}, \mu_{k_\ell,
  \iota_\ell}, D_{k_\ell, \iota_\ell})
  \end{equation*}

We now observe that by assumption, this sequence of pseudoholomorphic
  curves is stable and proper in \(W^0\), and \(\widehat{J}_\ell\to J\) in
  \(\mathcal{C}^\infty(W)\), and they have uniformly bounded area, genus,
  number of special points.
We then define the compact set 
  \begin{equation*}                                                       
    \mathcal{K}:= \Big\{ p\in W: {\rm dist}_g(p, \partial W) \geq
    \delta\Big\},
    \end{equation*}
  and apply Theorem \ref{THM_target_local_gromov_compactness},
  which guarantees that after passing to a subsequence (still denoted with
  subscripts \(\ell\)), there exist compact manifolds (possibly with
  smooth boundary) denoted by \(\widetilde{S}_\ell\subset
  \widehat{S}_\ell\) such that
  \begin{equation*}                                                       
    \hat{u}_\ell(\widehat{S}_\ell\setminus \widetilde{S}_\ell)\subset
    \Big\{p \in W: {\rm dist}_g(p, \partial W) < \delta\Big\}
    \end{equation*}
  and the curves
  \begin{equation*}                                                       
    (\hat{u}_\ell, \widetilde{S}_\ell, \hat{j}_\ell, W^0, \widehat{J}_\ell,
    \hat{\mu}_\ell, \widehat{D}_\ell)
    \end{equation*}
  converge in a Gromov sense.
In particular, for all sufficiently large \(\ell\), we have
  \(\#\pi_0(\partial \widetilde{S}_\ell) = n\) for all sufficiently large
  \(\ell\).
But since \(\widetilde{S}_\ell \subset \widehat{S}_\ell = S_{k_\ell,
  \iota_\ell}\), we must have \(\#\pi_0(\partial \widetilde{S}_\ell) \geq
  \ell\) have the desired contradiction.
This completes the proof of Lemma
  \ref{LEM_bounds_on_number_of_boundary_components}
\end{proof}

We are now prepared to complete the proof of Theorem
  \ref{THM_main_result}.
First, for each real number \(b< 0\), we fix \(\hat{\mathbf{u}}^b\in
  \mathcal{M}\) such that
  \begin{align*}                                                          
    \hat{\mathbf{u}}^b = \big(\hat{u}^b, S^2, \hat{j}^b, \overline{W},
    \overline{J}, \emptyset, \emptyset\big),
    \end{align*}
  and
  \begin{align*}                                                          
    \inf_{z\in S^2} a \circ \hat{u}^b(z) = a\circ \hat{u}^b(z_b) = b -2.
    \end{align*}
Next, define the manifold
\begin{align*}                                                            
  \check{W} := \big\{\bar{q}\in \overline{W}: - {\textstyle \frac{19}{10}}
  \leq \bar{a}(\bar{q}) \leq -{\textstyle \frac{11}{10}}\big\},
  \end{align*}
its interior
\begin{align*}                                                            
  \check{W}^0: =\big\{\bar{q}\in \overline{W}: - {\textstyle
  \frac{19}{10}}  <  \bar{a}(\bar{q}) < -{\textstyle \frac{11}{10}}\big\},
  \end{align*}
the surfaces
\begin{align*}                                                            
  \check{S}^b = \{z\in S^2: \hat{u}^b(z) \in \check{W}^0 \},
  \end{align*}
and the pseudoholomorphic curves
\begin{align*}                                                            
  \check{\mathbf{u}}^b= (\check{u}^b, \check{S}^b, \check{j}^b, \check{W},
  \overline{J}, \emptyset, \emptyset),
  \end{align*}
  where \(\check{u}^b = \hat{u}^b\big|_{\check{S}^b}\) and \(\check{j}^b =
  \hat{j}^b\big|_{\check{S}^b}\).
We then apply Lemma \ref{LEM_bounds_on_number_of_boundary_components} to
  the curves \(\check{\mathbf{u}}^b\) in \((\check{W}, \overline{J},
  \bar{g})\) with \(\delta < \frac{1}{10}\) to obtain the compact surfaces
  with boundary denoted:
\begin{align*}                                                            
  \widetilde{\check{S}}^b \subset \check{S}^b \subset S^2.
  \end{align*}
Finally, we define the compact surfaces with boundary denoted by \(S^b\)
  to be the connected component of
  \begin{align*}                                                          
    \widetilde{\check{S}}^b \cup  (u^b)^{-1}\Big( \{\bar{q}\in
    \overline{W} : \bar{a}(\bar{q}) < -{\textstyle \frac{3}{2}}\}\Big)
    \end{align*}
  which contains a point \(z_b\in S^2\) so that 
  \begin{align*}                                                          
    \inf_{z\in S^2} a \circ \hat{u}^b(z) = a\circ \hat{u}^b(z_b) = b -2.
    \end{align*}
By construction, we then have that each \(S^b\) is compact, connected,
  with \(\hat{u}^b(\partial S^b) \subset (-2, -1) \times M^- \subset {\rm
  Cyl}^-(\overline{W})\), and
  \begin{align*}                                                          
    \sup_{b < 0} \# \pi_0(\partial S^b) < \infty.
    \end{align*}
Next, we let 
  \begin{align*}                                                          
    \mathbf{u}_k^b = \mathbf{u}^b = \big(u^b, S^b, j^b, (-\infty,
    1)\times M^-, \overline{J}, \emptyset, \emptyset\big).
    \end{align*}
We also require that \(j^b = \hat{j}^b \big|_{S^b}\)
  and define
  \begin{align*}                                                          
    u^b = {\rm Sh}_{-2}\circ (\Phi^-)^{-1} \circ \hat{u}^b,
    \end{align*}
  where \(\Phi^-: (-\infty, -1+\frac{1}{4}\epsilon)\times M^- \to {\rm
  Cyl}^-(\overline{W})\) is the diffeomorphism guaranteed by Remark
  \ref{REM_structureal_observations}, and \({\rm Sh}_{-2}\colon
  \mathbb{R}\times M^- \to \mathbb{R}\times M^-\) is the shift map given
  by \({\rm Sh}_{-2}(a, p) = (a+2, p)\).

With our curves \(\mathbf{u}_k^b = \mathbf{u}^b\) defined, we now collect
  the properties they have.
  \begin{enumerate}[(P1)]                                                 
    \item
    each \(S^b=|S^b|\) is connected 
    \item 
    \(\mathbf{u}_k^b\) is compact and \(u^b(\partial S^b)\subset
    (0,1)\times M^- \),
    \item 
    \(\inf_{\zeta\in S^b} a\circ u^b(\zeta) = b\)
    \item 
    there exists a continuous path \(\alpha:[0,1]\to |S^b|=S^b\) satisfying 
    \begin{equation*}                                                     
      a\circ u^b\circ \alpha(0) = b \qquad\text{and}\qquad \alpha(1)\in
      \partial S^b
      \end{equation*}
    \item 
    \({\rm Genus}(S^b)=0\)
    \item 
    \(\int_{S^b}(u^b)^*\omega^- \leq \pi\) 
    \item 
    \(\#D^b=0\)
    \item 
    The number of connected components of \(\partial S_k^b\) is uniformly bounded
    \end{enumerate}
Moreover, by Proposition \ref{PROP_embeddedness}, it follows that for any
  \(b, b'< 0\) with \(b\neq b'\) we have
  \begin{align*}                                                          
    \#\big(u^b(S^b)\cap u^{b'}(S^{b'})\big) \leq 1.
    \end{align*}
From this we see that the hypotheses of Theorem \ref{THM_existence} are
  satisfied, and hence we conclude the existence of a closed  set \(\Xi
  \subset M^-\) satisfying \(\emptyset \neq \Xi \neq M^-\) which is
  invariant under the flow of the Hamiltonian vector field \(X_{\eta^-}\).
Recalling the Hamiltonian \(H:\mathbb{R}^4\to \mathbb{R}\) given in they
  hypotheses of Theorem \ref{THM_main_result}, we note that by construction
  the Hamiltonian flow on \(M^-\) is conjugated to the flow on on
  \(H^{-1}(0)\), and hence we conclude that the Hamiltonian flow on
  \(H^{-1}(0)\) is not minimal.
This completes the proof of Theorem \ref{THM_main_result}.

\end{proof}

\subsection{Proof of Theorem
  \ref{THM_second_main_result}}\label{SEC_second_main_result} \hfill\\
We are now prepared to prove the second main dynamical result.
We begin with a few preliminaries.
In what follows, we let \(\mathcal{D}\) denote the closed disk in the
  complex plane, given by
  \begin{align*}                                                          
    \mathcal{D}= \{(s, t) \in \mathbb{R}^2 : s^2 + t^2 \leq 1\}.
    \end{align*}

\begin{definition}[contact type]
  \label{DEF_contact_type}
  \hfill \\
Let \((M , \eta)\) be a framed Hamiltonian manifold with \(\eta =
  (\lambda, \omega)\).
We say \((M, \eta)\) is \emph{contact type} provided that \(\omega=
  d\lambda\).
\end{definition}
%

\begin{definition}[tight/overtwisted]
  \label{DEF_tight_ot}
  \hfill\\
Let \((M, \eta)\) be a three-dimensional framed Hamiltonian manifold of
  contact type.
We say \((M, \eta)\) is \emph{overtwisted} provided there exists an
  embedding \(\phi: \mathcal{D}\to M\) so that the one form
  \(\phi^*\lambda\) has \(\{0\}\cup \partial \mathcal{D}\) as its zero set.
If no such embedding exists, then we call \((M, \eta)\) \emph{tight}.
\end{definition}
%

\setcounter{CurrentSection}{\value{section}}
\setcounter{CurrentTheorem}{\value{theorem}}
\setcounter{section}{\value{CounterSectionTheoremTwo}}
\setcounter{theorem}{\value{CounterTheoremTheoremTwo}}
\begin{theorem}[second main dynamical result]\hfill \\
Let \((M^\pm, \eta^\pm)\) be a pair of three-dimensional framed
  Hamiltonian manifolds, and let \((\widetilde{W}, \tilde{\omega})\) be a
  symplectic cobordism from \((M^+, \eta^+)\) to \((M^-, \eta^-)\) in the
  sense of Definition \ref{DEF_symplectic_cobordism}.
Suppose that \(\tilde{\omega}\) is exact, \(M^-\) is connected, and that
  \((M^+, \eta^+)\) is contact type and has a connected component \(M'\)
  which is either \(S^3\), overtwisted, or there exists an embedded
  \(S^2\) in \(M'\subset \partial \widetilde{W}\) which is homotopically
  nontrivial in \(\widetilde{W}\).
Then the flow of the Hamiltonian vector field \(X_{\eta^-}\) on \(M^-\) is
  not minimal.
\end{theorem}
%
\begin{proof}
\setcounter{section}{\value{CurrentSection}}
\setcounter{theorem}{\value{CurrentTheorem}}
We begin by noting that the core arguments of the proof of Theorem
  \ref{THM_second_main_result} are identical to those in Theorem
  \ref{THM_main_result}, with some minor modifications from \cite{H93}.
As such, our argument here will be brief. 

The first key observation is to see that for an almost Hermitian structure
  adapted to a framed Hamiltonian manifold which is contact type, the
  function \(a\circ u: S\to \mathbb{R}\times M\) has a maximum principle
  whenever \(u:S\to \mathbb{R}\times M\) is a pseudoholomorphic map.
That is, \(a\circ u\) can have no interior local maxima.
To see this, observe that 
  \begin{align*}                                                          
    \big(\Delta (a\circ u)\big) ds\wedge dt
    &=
    \big((a\circ u)_{ss} + (a\circ u)_{tt}\big)ds \wedge dt
    \\
    &= d\big( (a\circ u)_s dt - (a\circ u)_t ds\big)
    \\
    &= - d\big(d(a\circ u) \circ j\big)
    \\
    &=  - d\big( da ( du \circ j) \big)
    \\
    &= -d\big( da (J\circ du)\big)
    \\
    &=d\big(  \lambda (du)\big)
    \\
    &=u^* d\lambda  
    \\
    &= u^* \omega
    \\
    &\geq 0.
    \end{align*}
Following \cite{H93}, we then break the problem into three cases.\\

\noindent{\bf Case I:} \emph{The connected component \((M', \lambda^+)\) is
  tight \(S^3\).}

In this case, it follows from deep work of Eliashberg (see \cite{E1}  and
  \cite{E2}) that up to diffeomorphism there exists a unique positive
  tight contact structure on \(S^3\), and moreover  there exists a smooth
  embedding
  \(\phi\colon M' \to \mathbb{R}^4\) for which
  \begin{align*}                                                          
    \lambda^+\big|_{M'} = \phi^*(x_1 dy_1 + x_2 dy_2).
    \end{align*}
Additionally, \(\phi(M')\) is the boundary of a compact 
  star-shaped\footnote{
    A compact set \(\mathcal{K}\subset \mathbb{R}^4\) is said to be
    \emph{star-shaped} provided that for each \((x_1, y_1, x_2, y_2)\in
    \mathcal{K}\) and each \(\tau\in [0, 1]\), one also has \((\tau x_1,
    \tau y_1, \tau x_2, \tau y_2)\in \mathcal{K}\).
    }
  set \(\mathcal{O}\subset \mathbb{R}^4\) with \(\mathcal{O}\)
  diffeomorphic to a compact four-ball.

Following the construction in the proof of Theorem \ref{THM_main_result},
  it then becomes possible to build a symplectic cobordism
  obtained by symplectically capping off \(M'\subset
  \partial\widetilde{W}\) by \(\mathbb{C}P^2\setminus \mathcal{O}\).
The resulting manifold, denoted \((\check{W}, \check{\omega})\), is then a
  symplectic cobordism from \((M^+\setminus M', \eta^+)\) to \((M^-,
  \eta^-)\).
One can then find an almost Hermitian structure \((\widetilde{J},
  \tilde{g})\) on \(\check{W}\) for which \(\widetilde{J}\) is adapted to
  \(\check{\omega}\) and for which  there exists an embedded
  pseudoholomorphic sphere \(u\colon S^2 \to \check{W}\) which has the
  same properties as \(\bar{\mathbf{u}}_0\) from the proof of Theorem
  \ref{THM_main_result}.
That is, one considers the moduli space \(\mathcal{M}\) of non-nodal
  curves which are homotopic (through non-nodal pseudoholomorphic curves) to
  this special curve.
Curves in this moduli space are cut out transversely and pairwise
  intersect at exactly one point.
By the same means as in the proof Theorem \ref{THM_main_result}, one shows
  that if this family of curves is  contained in a compact region, then the
  area is uniformly bounded.

By positivity of intersections and exactness of \(\bar{\omega}\), bubbling
  is impossible, and hence one can show that the set of points in the
  extension of \(\check{W}\) which are in the image of a curve in
  \(\mathcal{M}\) is both open and closed if the curves stay in a compact
  region.
This contradiction establishes that the curves must escape into a
  cylindrical end of the extended cobordism, but the maximum principle
  prevents them from escaping into the positive end.
Thus the curves must extend all the way down into the single negative end
  of the extension of \(\check{W}\), while each still intersects the
  initial curve.
Trimming the curves as in the proof of Theorem \ref{THM_main_result} then
  yields a sequence of curves to which Theorem \ref{THM_existence} applies,
  and hence the non-minimality of the flow of \(X_{\eta^-}\) on \(M^-\) is
  established.\\

\noindent{\bf Case II: }\emph{The manifold \((M', \eta^+)\) is overtwisted.}

This case relies more heavily on input from \cite{H93}.  
Begin by letting the tuple \(\mathbf{W} = (\overline{W}, \bar{\omega},
  \overline{J}, \bar{g}, \bar{a}, \partial_{\bar{a}}, \epsilon)\) denote the
  extension associated to \((\widetilde{W}, \tilde{\omega})\).
Letting \(\phi\colon \mathcal{D}\to M'\) be the overtwisted disk guaranteed
  to exist, we lift this to an embedding \(\tilde{\phi}: \mathcal{D}\to
  \mathbb{R}^+\times M' \) via \(\tilde{\phi}(s,t) = \big(10, \phi(s,t)\big)
  \in \mathbb{R}\times M'\).
Letting \(\Phi^+\colon (1-\frac{1}{4}\epsilon, \infty)\times M^+ \to {\rm
  Cyl}^+(\overline{W})\) denote the embedding guaranteed by Remark
  \ref{REM_structureal_observations}, we then define the embedded disk
  \begin{align*}                                                          
    \Sigma := \Phi^+(\tilde{\phi}(\mathcal{D})).
    \end{align*}
Define the point \(e:= \Phi^+(\tilde{\phi}(0))\in \Sigma\). 
Then, following \cite{H93}, one constructs a family of pseudoholomorphic
  curves of the form
  \begin{align*}                                                          
    u\colon \mathcal{D}\to \overline{W}\qquad\text{with}\qquad
    u\colon\partial \mathcal{D}\to \Sigma
    \end{align*}
  so that \(u(\partial \mathcal{D})\) is transverse to \(T\Sigma \cap {\rm
  ker}\lambda^+\) and winds around \(e\) precisely once.
As is shown in \cite{H93}, such curves are the zero set of a smooth
  non-linear Fredholm section the linearization of which is always
  surjective.
Additionally, the curves are pairwise disjoint, and the images of their
  boundaries locally foliate \(\Sigma\).
An additional crucial fact is that the boundaries \(u(\partial
  \mathcal{D})\) always stay transverse to the characteristic foliation
  given by the integral curves of \(T\Sigma\cap {\rm ker}\lambda^+\), and
  hence the boundaries \(u(\partial \mathcal{D})\) must always stay disjoint
  from \(\partial \Sigma\).

At this point, we follow the script from Case I and from the proof of
  Theorem \ref{THM_main_result}.
We let \(\mathcal{M}\) denote the moduli space of such curves, and we note
  that if there exists a compact set \(\mathcal{K} \subset \overline{W}\)
  which contains the images of all the curves in \(\mathcal{M}\), then the
  set of points in \(\Sigma\) which are in the image of \(u\big|_{\partial
  \mathcal{D}}\) for \(u\in \mathcal{M}\) is both open and closed, which
  is impossible.
Here again we are making use of the fact that the existence of such a
  \(\mathcal{K}\) guarantees a uniform area bound as before, which then
  guarantees Gromov convergence, which establishes closedness.
Again, the maximum principle prevents curves from escaping into \({\rm
  Cyl}^+(\overline{W})\), so the curves must instead escape into \({\rm
  Cyl}^-(\overline{W})\), and again the curves can be trimmed so that
  Theorem \ref{THM_second_main_result} applies, and again the flow of
  \(X_{\eta^-}\) on \(M^-\) is not minimal.\\

\noindent{\bf Case III: }\emph{There exists embedded \(S^2\subset M'\)
  which is homotopically nontrivial in \(\widetilde{W}\).}

In this case we again follow \cite{H93} rather closely.
In particular, by assumption, there exists an embedded sphere in \(M'\)
  which when included into \(\widetilde{W}\) is homotopically nontrivial.
We also may assume that \(M'\) is tight, since the overtwisted case has
  already been established.
Consequently, we may perturb this sphere, keeping it embedded, so that
  there exist precisely two points \(\{e^+, e^-\}\subset \Sigma\) at which
  we have \(T\Sigma = {\rm ker}\lambda^+\), and all integral curves of
  \(T\Sigma\cap {\rm ker}\lambda^+\) have \(e^+\) as one end point and
  \(e^-\) as the other.
As in Case II, we lift this sphere into a level set \(\Sigma\subset
  \{10\}\times M'\subset \overline{W}\).
As in the proof of Case II, we then construct a family of
  pseudoholomorphic disks with boundary in the sphere \(\Sigma\), each
  winding around \(e^\pm\) exactly once.
These curves again have similar properties, like being cut out
  transversely, and they are pairwise disjoint and have boundaries which
  locally foliate \(\Sigma\).
Once again one shows that if the images of the curves in this moduli space
  are contained in some compact set, then the set of points in \(\Sigma\)
  which are in the image of the boundaries of curves in this moduli space is
  both open and closed in \(\Sigma\), and hence are all of \(\Sigma\).
However, if this is the case, then, as in \cite{H93}, one can use the
  moduli space of curves to show that \(\Sigma\) is homotopically trivial,
  which is impossible.
Consequently, the images of the curves cannot stay in a compact region, so
  they must escape out into a cylindrical end of \(\overline{W}\), and by
  the maximum principle it cannot be the positive end.
Again, curves escape into the negative end, are pairwise disjoint, with a
  suitable area bound to obtain the appropriate trimmings to apply the
  workhorse theorem and again the flow of \(X_{\eta^-}\) on \(M^-\) is not
  minimal.

\end{proof}

\section{Supporting Proofs}\label{SEC_supporting_proofs}

In this section we prove the main foundational results about feral curves
  which are needed to prove Theorem \ref{THM_main_result} and Theorem
  \ref{THM_second_main_result}.
Each of the following sections is dedicated to precisely one of the proofs
  of Theorem \ref{THM_area_bounds}  through Theorem \ref{THM_existence}.

\subsection{Proof of Theorem \ref{THM_area_bounds}: Exponential Area
  Bounds}\label{SEC_proof_of_exp_area_bounds}\hfill\\

The main purpose of this section is to prove Theorem
  \ref{THM_area_bounds} as well as several important generalizations.
Our first step will be to give a brief overview of the structure of the
  proofs, while indicating the methods used to overcome certain obstacles.
Throughout this section, we will assume that \(M\) is a closed manifold
  equipped with a framed Hamiltonian structure \(\eta=(\lambda, \omega)\),
  and that \((J, g)\) is an \(\eta\)-adapted almost Hermitian structure
  on \(\mathbb{R}\times M\).

\subsubsection{The Rough Sketch}
Before proceeding into some of the technical (and tedious but elementary)
  details, we first provide the core idea in a model scenario.  
Afterwards, we describe how to generalize.
To that end, we first need some definitions.
We begin by assuming \(u:S\to \mathbb{R}\times M\) is a pseudoholomorphic
  map, with the following properties.
\begin{enumerate}                                                         
  \item \(u(S)\subset [0, r]\times M\) for some fixed positive \(r>0\)
  \item \(\partial S = u^{-1}(\{0, r\}\times M)\)
  \item \(\{\zeta\in S: d(a\circ u)_\zeta = 0\} = \emptyset\), where \(a\)
  is the symplectization coordinate on \(\mathbb{R}\times M\).
  \end{enumerate}
Geometrically then, we should think of \(S\) as an annulus (or a finite
  union of annuli), and \(u\) maps one boundary component to \(\{0\}\times
  M\), and it maps the other boundary component to \(\{r\}\times M\).
Furthermore, since the set of critical points of the function \(a\circ u:
  S\to [0, r]\)  is empty, we know that any gradient trajectory in the lower
  boundary component \(u^{-1}(\{0\}\times M)\) will terminate in the upper
  boundary component \(u^{-1}(\{r\}\times M)\).
This is a fact we will heavily exploit.

Next, for each \(x,y\in [0, r]\) with \(x< y\) we define 
  \begin{equation*}                                                         
    S_x^y:= \{\zeta\in S: x \leq a\circ u(\zeta) \leq y\},
    \end{equation*}
  and
  \begin{equation*}                                                         
  \alpha:= -(u^* da)\circ j = u^*\lambda,
  \end{equation*} 
  as well as the functions
  \begin{equation*}                                                         
    h(s):=\int_{(a\circ u)^{-1}(s)}\alpha \qquad\text{and}\qquad G(s) =
    \int_{S_0^s} u^*\omega,
    \end{equation*}
  and the Riemannian metric
  \begin{equation*}                                                         
    \gamma:= u^*g.
    \end{equation*}
We note that \(h\), \(G\), and \(\gamma\)  are smooth.

With these definitions in place, we next recall a few linear algebra and
  calculus facts.
The first is that there exists a large constant \(C>0\) which depends on
  ambient geometry but not on the map \(u\) for which
  \begin{equation*}                                                       
    \|d\alpha \|_\gamma \leq C.
    \end{equation*}
This is readily seen here by recalling that \(\alpha=u^*\lambda\), and
  \(\gamma=u^*g\), so that \(\|d\alpha\|_\gamma = \|u^* d\lambda\|_{u^* g}
  \leq \|d\lambda\|_g \leq C\).
Next is the fact that given a two-dimensional oriented Riemannian
  manifold, like \((S, \gamma)\), there exists a corresponding
  two-dimensional Hausdorff measure \(d\mu_\gamma^2\); similarly for
  other dimensions.
Furthermore, because \((J,g)\) is suitably adapted to \(\eta\), and
  because \(u\) is pseudoholomorphic, we find
  \begin{equation*}                                                         
    {\rm Area}_\gamma(S) = \int_S d\mu_\gamma^2 = \int_S u^* da \wedge
    \alpha + u^* \omega,
    \end{equation*}
  with similar statements for subdomains in \(S\).
Finally, the following result is an immediate application of the co-area
  formula (see Lemma \ref{PROP_coarea_body} below), or a suitably applied
  change of coordinates.
\begin{equation*}                                                         
  \int_{S_x^y} (u^* da)\wedge \alpha= \int_x^y   \Big(\int_{(a\circ
  u)^{-1}(s)} \alpha \Big) ds
  \end{equation*}

With these preliminaries established, we now establish our principle
  differential inequality.
\begin{align*}                                                          
  |h'(s)|&=\Big|\lim_{\epsilon\to 0^+}
    \epsilon^{-1}\big(h(s+\epsilon)-h(s)\big)\Big|\\
  &=
    \lim_{\epsilon\to 0^+} \epsilon^{-1}\Big|\int_{(a\circ
    u)^{-1}( s+ \epsilon)}\alpha - \int_{(a\circ
    u)^{-1}( s)}\alpha\Big|\\
  &=
    \lim_{\epsilon\to 0^+}
    \epsilon^{-1}\Big|\int_{S_{s}^{s+\epsilon}}d
    \alpha\Big|\\
  &\leq
    \lim_{\epsilon\to 0^+}
    \epsilon^{-1}\int_{S_{s}^{s+\epsilon}} \|d
    \alpha\|_{\gamma}d\mu_{\gamma}^2 \\
  &\leq 
    \lim_{\epsilon\to 0^+} \epsilon^{-1}C
    \int_{S_{s}^{s+\epsilon}}
    d\mu_{\gamma}^2 \\
  &=  
    C \lim_{\epsilon\to 0^+} \epsilon^{-1}
    \int_{S_{s}^{s+\epsilon}} \big((u^*da)
    \wedge \alpha + u^*\omega\big)       \\
  &=
    C \Big(\lim_{\epsilon\to 0^+}\epsilon^{-1}
    \int_{S_{s}^{s+\epsilon}} (u^*da)\wedge
    \alpha + \lim_{\epsilon\to 0^+}\epsilon^{-1}
    \int_{S_{s}^{s+\epsilon}}u^*\omega\Big)\\
  &= 
    C\big( \int_{(a\circ u)^{-1}(s)} \alpha
    +G'(s)\big)\\
  &= 
    C \big(h(s) +  G'(s)\big)
  \end{align*}
Or to put it succinctly and in a more useable form, 
  \begin{equation*}                                                       
      h'(s)\leq C \big(h(s) + G'(s)\big),
    \end{equation*}
  where \(C\) depends on ambient geometry, but not on the map \(u\).
Integrating up, we find
  \begin{align*}                                                          
    h(s)&\leq h(0) + C \int_0^s h(t) dt + C\big(G(s) - G(0)\big)\\
    &\leq \big(h(0) + C\int_S u^* \omega\big) + C \int_0^s h(t) dt
    \end{align*}
By Gronwall's inequality (see Lemma \ref{LEM_gronwall} below), we then
  have
  \begin{equation*}                                                       
    h(s)\leq \big( h(0) + C\int_S u^*\omega\big)  e^{Cs},
    \end{equation*}
  or rewriting it, making use of the definition of \(h\), and the fact
  that \(\alpha=u^*\lambda\), we have
  \begin{equation*}                                                       
    \int_{(a\circ u)^{-1}(s)} u^*\lambda \leq \Big(\int_{(a\circ
    u)^{-1}(0)} u^*\lambda + C \int_S u^*\omega \Big) e^{Cs}.
    \end{equation*}
In essence, this is precisely the desired inequality which establishes
  Theorem \ref{THM_area_bounds}.
To emphasize the key characteristics, the above inequality says that if we
  consider the function \(s\mapsto \int_{(a\circ u)^{-1}(s)} u^*\lambda\),
  then the function is bounded from above by \(s\mapsto A e^{Cs}\), where
  \(C\) depends only on ambient geometry, and  \(A\) is bounded in terms
  of ambient geometry constant \(C\),  the \(\omega\)-energy (which is
  always a priori bounded),  and \(\int_{(a\circ u)^{-1}(0)} u^*\lambda\).
Essentially then, we have an exponential bound on the growth of the
  function \(s\mapsto \int_{(a\circ u)^{-1}(s) }u^*\lambda\).
To obtain a similar exponential bound on the area, it is sufficient to
  recall that
  \begin{align*}                                                          
    {\rm Area}_{u^* g} (S_0^r) &= \int_{S_0^r} u^* (da \wedge \lambda +
    \omega)
    \\
    &=\int_{S_0^r} u^* da \wedge \lambda +\int_{S_0^r} u^* \omega
    \\
    &=\int_0^r \Big(\int_{(a\circ u)^{-1}(t)} u^*\lambda \big) dt +
    \int_{S_0^r} u^*\omega
    \end{align*}
    and then employ our previous exponential growth estimate for
    \(s\mapsto \int_{(a\circ u)^{-1}(s)} u^*\lambda\).
We note that while all of the above estimates assume that \(s\in [0, r]\),
  as similar construction and analysis establish the case that \(s\in [-r,
  0]\).

With this principle estimate established, we can then generalize as appropriate.  
First, to move from annuli to more general surfaces, one observes
  that Sard's theorem guarantees that for a more general domain \(S\), the
  set of regular values of \(a\circ u\) has full measure, and it must be
  open since the set of critical points is closed.
A bit of elementary measure theory then lets us approximate the set of
  regular values from the inside by a finite set of compact intervals on
  which the desired estimate holds.
Making use of the fact that \(\omega\) evaluates non-negatively on
  \(J\)-complex lines and \(\int_S u^*\omega <\infty\), and some
  elementary real analysis then allows us to conclude the desired
  inequality for the more general surface.
We carry out these details below, but for the moment we sketch further
  generalizations.

Already, such an exponential growth bound on area is rather useful, however
  there are two more related results which prove to be quite important,
  and each essentially stems from the fact that gradient-flow type
  coordinates are more useful to us than holomorphic coordinates.
More specifically, one can construct local coordinates \((s, t)\) on
  \(S\), with the property that \(a\circ u(s,t) = t\), and the map
  \(t\mapsto (s_0, t)\in S\) is contained in an integral curve of the vector
  field \(\nabla (a \circ u)\).
One can then ask if our exponential growth bound on \(\int u^*\lambda\)
  and the area holds on such rectangular patches \((s,t)\in [0, b]\times [0,
  r]\) of pseudoholomorphic curve.
As it turns out, the answer is yes, essentially because \(0 = -(u^*da)(j
  \nabla(a\circ u)) = \alpha (\nabla (a\circ u)) = u^*\lambda(\nabla(a\circ
  u))\).
Indeed, replacing \(S_x^y\) with 
  \begin{equation*}                                                       
    \widetilde{S}_x^y:=\{\zeta\in S_x^y : 0 \leq s(\zeta) \leq b\}
    \end{equation*}
  where \((s,t)\) are rectangular gradient-like coordinates as above, we
  see the entire argument carries over unchanged, including
  \begin{equation*}                                                       
    \int_{(a\circ u)^{-1}(s+\epsilon)} \alpha - \int_{(a\circ
    u)^{-1}(s+\epsilon)} \alpha 
    =
    \int_{\widetilde{S}_s^{s+\epsilon}} d\alpha.
    \end{equation*}
This latter equality holds precisely because \(\alpha(\nabla(a\circ
  u))=0\), and this guarantees that there are no contributions to Stokes'
  theorem coming from the gradient-like ``sides'' of our pseudoholomorphic
  rectangle.
Essentially then, this establishes exponential area growth for certain
  pseudoholomorphic rectangles, which we define more precisely as
  \emph{tracts of pseudoholomorphic curves} in Definition
  \ref{DEF_tract_of_perturbed_J_map} below.

The final generalization of our exponential growth bound is less
  enlightening and more a necessary evil.
The issue is that in later sections, we will need to study portions of a
  pseudoholomorphic curve restricted to \(S_x^{x+\epsilon}\) where \(x\) and
  \(x+\epsilon\) are  regular values of \(a\circ u\).
Moreover, we would like to claim that if \(\int_{(a\circ u)^{-1}(x)} u^*
  \lambda\) is very large, and if \(\int_{S_x^{x+\epsilon}} u^*\omega\) is
  very small, then most of the gradient trajectories which start along the
  set \((a\circ u)^{-1}(x)\) terminate at a point in \((a\circ
  u)^{-1}(x+\epsilon)\).
As it turns out, this is not difficult to establish in the special case
  that the function \(a\circ u\) is Morse, but it appears to be
  intractable in the general case.
This forces us into the position that we must establish the desired
  exponential growth bound on area for \emph{perturbed} pseudoholomorphic
  curves -- that is, for curves which are no longer pseudoholomorphic.
Worse still, the \emph{type} of perturbation, and specifically its precise
  \emph{size}, will be important for later estimates, so we must establish
  the desired area estimate for all perturbed pseudoholomorphic curves for
  which the perturbation is small in a very explicit manner.
In turn, this seems to force us to establish a number of rather elementary
  estimates via rather tedious but elementary means, and this takes up a bulk
  of the Section \ref{SEC_proof_of_exp_area_bounds}.
The upshot however, is that we establish the desired estimates for
  perturbed curves, which is crucial for later results.
The remainder of Section \ref{SEC_proof_of_exp_area_bounds} is then
  devoted to making these above sketches rigorous.

\subsubsection{Definitions and Elementary Estimates}
\begin{definition}[perturbed pseudoholomorphic map]
  \label{DEF_perturbed_J_map}
  \hfill\\
Let \((M, \eta)\) be a framed Hamiltonian manifold with
  \(\eta=(\lambda,\omega)\), and let \((J, g)\) be an \(\eta\)-adapted
  almost Hermitian structure on \(\mathbb{R}\times M\).
A \emph{perturbed pseudoholomorphic map} consists of the tuple
  \((\tilde{u}, \tilde{\jmath}, f, u, S, j)\) where
  \begin{enumerate}[(p1)]                                                 
    \item\label{EN_p1} \(u:(S, j)\to (\mathbb{R}\times M, J)\) is a
      generally immersed pseudoholomorphic map, which is possibly
      non-compact,
    \item\label{EN_p2} \(f:S\to \mathbb{R}\) is a smooth function,
    \item\label{EN_p3} the support of \(f\) is compact and satisfies
      \({\rm supp}(f)\subset S\setminus (\partial S\cup \mathcal{Z})\),
      where
      \begin{equation*}                                                   
	\mathcal{Z}=\{\zeta\in S: Tu(\zeta)=0\}
      \end{equation*}
    \item\label{EN_p4} \(\tilde{u}(\zeta) = \exp_{u(\zeta)}^g(f (\zeta)
      \partial_a)\), where \(\exp^g\) is the exponential map associated
      to the Riemannian metric \(g\), and \(\partial_a\) is the coordinate
      vector field associated to the coordinate \(a\in \mathbb{R}\),
    \item\label{EN_p5} \(\tilde{\jmath}\) is a smooth almost complex
      structure on \(S\) which induces the same orientation as \(j\),
    \item\label{EN_p6} on the complement of \({\rm supp}(f)\) we
      have \(j=\tilde{\jmath}\), and elsewhere \(\tilde{\jmath}\)
      is uniquely determined by requiring that \(\tilde{\jmath}\) is
      a \(\tilde{u}^*g\)-isometry.
  \end{enumerate}
\end{definition}
%

Geometrically then, a perturbed pseudoholomorphic map is obtained by
  nudging an honestly pseudoholomorphic map a bit in the symplectization
  direction.
We require this modification to be away from the boundary of \(S\) and
  critical points of \(u\), and in practice it will occur only in a small
  neighborhood of the critical points of \(a\circ u\).
We adapt the almost complex structure on the domain so that our new
  perturbed map is an isometry, but not pseudoholomorphic.
For notational convenience and ease of exposition, rather than write the
  full tuple \((\tilde{u}, \tilde{\jmath}, f, u, S, j)\) to specify a
  perturbed pseudoholomorphic map, we will instead say: Let \((\tilde{u}, S,
  \tilde{\jmath})\) be an \(f\)-perturbation of a pseudoholomorphic map
  \((u, S, j)\).

\begin{definition}[tract of perturbed pseudoholomorphic map]
  \label{DEF_tract_of_perturbed_J_map}
  \hfill \\
Let \((M, \eta)\) be a framed Hamiltonian manifold with
  \(\eta=(\lambda,\omega)\), and let \((J, g)\) be an \(\eta\)-adapted
  almost Hermitian structure on \(\mathbb{R}\times M\).
A \emph{tract of perturbed pseudoholomorphic map} consists of the tuple
  \((\tilde{u}, \widetilde{S}, \tilde{\jmath}, f, u, S, j)\) where
  \begin{enumerate}                                                       
    \item \((\tilde{u}, \tilde{\jmath}, f, u, S, j)\) is a perturbed
      pseudoholomorphic map,
    \item \(\widetilde{S}\subset S\) is a smooth real two dimensional
      non-empty manifold, possibly with boundary, possibly with corners,
      and possibly non-compact,
    \item the restriction \(\tilde{u}:\widetilde{S} \to \mathbb{R}\times
      M\) is a proper\footnote{By proper, we mean that for each
      compact set \(\mathcal{K}\subset \mathbb{R}\times M\) the set
      \(\tilde{u}^{-1}(\mathcal{K})\) is compact.} map satisfying
      \begin{equation*}                                                   
	\{\zeta \in  S: d(a\circ \tilde{u})_\zeta = 0\}\cap \partial
	\widetilde{S}=\emptyset,
      \end{equation*}
    \item the boundary of \(\widetilde{S}\) decomposes as \(\partial
      \widetilde{S} = \partial_0 \widetilde{S}\cup \partial_1
      \widetilde{S}\) where
      \begin{enumerate}                                                   
	\item the set \(\partial_0 \widetilde{S}\cap
	  \partial_1\widetilde{S}\) is finite
	\item along \(\partial_1 \widetilde{S}\) the vector field
	  \(\widetilde{\nabla}(a\circ \tilde{u})\) is tangent to
	  \(\partial_1 \widetilde{S}\); here \(\widetilde{\nabla}\)
	  is the gradient computed with respect to the metric
	  \(\tilde{\gamma}:=\tilde{u}^* g\),
	\item the restriction of the map \(a\circ \tilde{u}\) to
	  each connected component of \(\partial_0 \widetilde{S}\) is a
	  constant map.
      \end{enumerate}
  \end{enumerate}
\end{definition}
%
Geometrically, a tract of perturbed pseudoholomorphic map is a perturbed
  pseudoholomorphic map with boundary and corners, with the property that
  the boundary is piecewise smooth, and each smooth portion is either a
  level set of \((a\circ u\) or else an integral curve of
  \(\widetilde{\nabla} a\circ \tilde{u}\).
We denote the level-set type boundary by \(\partial_0 \widetilde{S}\), and
  we denote the gradient-line type boundary by \(\partial_1 \widetilde{S}\).
The corners of the boundary are those points where the two types of
  boundary intersect.

\begin{definition}[the characteristic $\alpha$-foliation:
  $\mathcal{F}^\alpha$]
  \label{DEF_char_alpha_fol}
  \hfill\\
Let \(S\) be a real \(2\)-dimensional manifold, possibly with boundary,
  possibly with corners, and possibly non-compact.
Suppose \(\alpha\in \Omega^1(S)\) is a smooth one-form on \(S\). 
Then we define the characteristic \(\alpha\)-foliation,
  \(\mathcal{F}^\alpha\subset TS\), by
  \begin{align*}
    \mathcal{F}^\alpha &=\bigcup_{\zeta\in S}
      \mathcal{F}_\zeta^\alpha\qquad\text{where } \mathcal{F}_\zeta^\alpha
      \subset T_\zeta S\text{ is given by}\\
    \mathcal{F}_\zeta^\alpha &= 
      \begin{cases}
	{\rm ker}\;\alpha_\zeta &\text{if }{\rm dim}({\rm
	  ker}\;\alpha_\zeta) =1\\
	0 &\text{otherwise}.
      \end{cases} 
  \end{align*}
\end{definition}
%

\begin{lemma}[characteristic $\tilde{\alpha}$-foliation is
  gradient]
  \label{LEM_char_fol_grad}
  \hfill\\
Let \((M, \eta=(\lambda, \omega))\) be a framed Hamiltonian manifold,
  let \((J, g)\) be an \(\eta\)-adapted almost Hermitian structure
  on \(\mathbb{R}\times M\), let \((\tilde{u}, S, \tilde{\jmath}) \) be an
  \(f\)-perturbation of a pseudoholomorphic map \((u, S, j)\) Definition
  \ref{DEF_perturbed_J_map}, and let \(\tilde{\alpha}=-d(a\circ
  \tilde{u})\circ \tilde{\jmath}\) be as above.
Then 
  \begin{equation}\label{EQ_char_fol_grad_1}                              
    \{\zeta\in \widetilde{S}: {\rm
    dim}\;(\mathcal{F}_\zeta^{\tilde{\alpha}})=0\}=
    \{\zeta\in \widetilde{S}: d(a\circ u)_\zeta=0\},
    \end{equation}
  and \(\widetilde{\nabla}(a\circ \tilde{u})\), when thought of
  as a subset of \(T\widetilde{S}\) and computed with respect to the
  metric \(\tilde{\gamma}=\tilde{u}^*g\) satisfies the property
  \begin{equation}\label{EQ_char_fol_grad_2}                              
    \widetilde{\nabla}(a\circ \tilde{u}) \subset
    \mathcal{F}^{\tilde{\alpha}}.
    \end{equation}
\end{lemma}
%
\begin{proof}
We first observe that \(\tilde{\alpha}\) is a smooth one-form on a
  two-manifold, and hence \({\rm dim}({\rm ker}\, \tilde{\alpha}_\zeta)\in
  \{1, 2\}\) .
The set of points where this dimension is two, is precisely the
  set of points where \(\tilde{\alpha}\) is zero -- however, \(0=
  \tilde{\alpha}_\zeta= d(a\circ \tilde{u})_\zeta\circ \tilde{\jmath}\),
  and since \(\tilde{\jmath}\) is a \(\tilde{\gamma}\)-isometry, we
  see that the set where this dimension is two is precisely the set of
  critical points of \(a\circ \tilde{u}\).
Equation (\ref{EQ_char_fol_grad_1}) follows immediately.

To establish (\ref{EQ_char_fol_grad_2}), we observe that 
  \begin{equation*}                                                       
    \tilde{\alpha}\big(\widetilde{\nabla}(a\circ \tilde{u})\big) =
    -d(a\circ \tilde{u})\big(\tilde{\jmath} \widetilde{\nabla}(a\circ
    \tilde{u})\big) = 0
  \end{equation*}
  where to obtain the second equality we have used the fact
  that the almost complex structure \(\tilde{\jmath}\) is a
  \(\tilde{\gamma}\)-isometry and \(\widetilde{\nabla}\) is the gradient
  with respect to \(\tilde{\gamma}\).
\end{proof}

Given an \(f\)-perturbation of a pseudoholomorphic map, it will be
  important to have certain properties of the metric
  \(\tilde{\gamma}=\tilde{u}^*g\) estimated in terms of properties of the
  metric \(\gamma=u^*g\).
As such, we have the following. 

\begin{lemma}[$\tilde{\gamma}$-estimates]
  \label{LEM_gamma_tilde_estimates}
  \hfill \\
Let \(0<\epsilon < 2^{-24}\), let \((M, \eta)\) be a framed Hamiltonian
  manifold with \(\eta=(\lambda,\omega)\), and let \((J, g)\) be an
  \(\eta\)-adapted almost Hermitian structure on \(W:=\mathbb{R}\times M\),
  and let \((u, S, j)\) be a generally immersed pseudoholomorphic map.
Let \(f\) be a smooth function, and let  \(\tilde{\mathbf{u}}=(\tilde{u},
  S, \tilde{\jmath})\) be an \(f\)-perturbation of \((u, S, j)\).
Suppose further that
  \begin{equation*}                                                       
      \|df\|_\gamma+ \|\nabla df\|_\gamma \leq \frac{\epsilon}{2^{11}
      (1+\|B_u\|_\gamma)}
    \end{equation*}
  here \(\nabla\) denotes covariant differentiation with respect to
  the Levi-Civita connection associated to the metric \(\gamma=u^*g\),
   \(B_u\) denotes the second fundamental form of \(u:S\to
   \mathbb{R}\times M\) as given in Definition
   \ref{DEF_second_fundamental_form}, and finally by \(\|df\|_\gamma\),
   \(\|\nabla df\|_\gamma\), and \(\|B_u\|_\gamma\) we respectively mean
   the \(L^\infty\) norm of each over the support of \(f\). Recall that
   ${\mathcal Z}$ is the set of singular points of $u$, i.e. \({\mathcal
   Z}=\{\zeta\in S\ |\ Tu(\zeta)=0\}\).
\emph{Then} for any vector fields \(Y, Z\in \Gamma(TS\to
  (S\setminus\mathcal{Z}))\) we have
  \begin{equation}\label{EQ_gamma_tilde_estimate_1}                      
      \big| \langle Y, Z\rangle_\gamma - \langle Y,
      Z\rangle_{\tilde{\gamma}} \big| \leq \epsilon\|Y\|_{\gamma}
      \|Z\|_{\gamma}
    \end{equation}
  and
  \begin{equation}\label{EQ_gamma_tilde_estimate_2}                      
      \|\nabla_Y Z - \widetilde{\nabla}_Y Z \|_\gamma \leq
      \epsilon\|Y\|_{\gamma} \|Z\|_{\gamma}.
    \end{equation}
And for each one-form \(\alpha\) on \(S\) we have
  \begin{equation}\label{EQ_gamma_tilde_estimate_3}                      
    \|\nabla_Y \alpha - \widetilde{\nabla}_Y \alpha \|_\gamma \leq
    \epsilon\|Y\|_{\gamma} \|\alpha\|_{\gamma}.
    \end{equation}
\end{lemma}
%
\begin{proof}
Since the inequalities are trivially true for any \(\zeta\in S\setminus
  {\rm supp}(f)\), we begin by fixing a point \(\zeta_0\in {\rm supp}(f)\),
  and  letting \((y^1, y^2)\) denote \(\gamma\)-normal geodesic coordinates
  centered at \(\zeta_0\).
Next, we let \((x^1, \ldots, x^{2m -1})\) denote normal geodesic
  coordinates on \(M\) associated to metric \(\lambda\otimes \lambda +
  \omega\circ (\mathds{1}\times J\pi_\xi)\) and centered at the point
  \(u(\zeta_0)\); here \(\pi_\xi\) is the projection along the line bundle
  \({\rm ker}\, \omega\to M\) to the hyperplane distribution \(\xi:={\rm
  ker}\, \lambda\).
Define the coordinate \(x^0:= a\) where \(a\) is the usual
  symplectization coordinate on \(\mathbb{R}\).
We simplify the following computation using the abbreviation
  \(\tilde{u}^i=x^i\circ \tilde{u}\) and \(\tilde{u}_{, i}= Tu\cdot
  \partial_{y^i}\), and then
  \begin{align*}                                                            
      \tilde{\gamma}_{k\ell} &=\langle \tilde{u}_{,k},
	\tilde{u}_{,\ell}\rangle_g\\
      &=
	g_{ij}\tilde{u}_{,k}^i \tilde{u}_{,\ell}^j\\
      &=
	g_{ij}(u_{, k}^i + \delta^{0i}f_{, k})(u_{, \ell}^j +
	\delta^{0j}f_{, \ell})\\
      &=
	g_{ij}u_{, k}^i u_{, \ell}^j +
	g_{ij}\delta^{0i}f_{, k}u_{, \ell}^j +
	g_{ij} u_{, k}^i  \delta^{0j}f_{, \ell} +
	g_{ij} \delta^{0i}f_{, k} \delta^{0j}f_{, \ell}\\
      &=
	g_{ij}u_{, k}^i u_{, \ell}^j +
	g_{0j}f_{, k}u_{, \ell}^j +
	g_{i0} u_{, k}^i  f_{, \ell} +
	g_{00} f_{, k} f_{, \ell}\\
      &=
	g_{ij}u_{, k}^i u_{, \ell}^j +
	f_{, k}u_{, \ell}^0 +
	u_{, k}^0  f_{, \ell} +
	f_{, k} f_{, \ell}\\
      &=\gamma_{k \ell} +f_{, k}u_{, \ell}^0 +
	u_{, k}^0  f_{, \ell} +
	f_{, k} f_{, \ell}; 
    \end{align*}
  where throughout the above computation \(g_{ij}\) is evaluated at
  \(\tilde{u}(y^1, y^2)\), and \(g_{ij}u_{, k}^i u_{, \ell}^j=\gamma_{k
  \ell}\) because the \(g_{ij}\) are independent of the \(x^0\)
  coordinate; that is, the metric \(g\) is independent of translation
  in the symplectization direction.
To be clear, by definition it is true that
  \begin{align*}                                                            
    g_{ij}(u(y^1, y^2))u^i_{,k}(y^1,y^2)
    u^j_{,l}(y^1,y^2)=\gamma_{kl}(y^1,y^2),
    \end{align*}
however we are claiming
  \begin{align*}                                                            
    g_{ij}(\tilde{u}(y^1, y^2))u^i_{,k}(y^1,y^2)
    u^j_{,l}(y^1,y^2)=\gamma_{kl}(y^1,y^2),
    \end{align*}
which is only true because \(\tilde{u}(y^1,y^2) = u(y^1,y^2) +
  \big(f(y^1,y^2), 0, \ldots, 0\big)\) and \(g\) is invariant under
  \(\mathbb{R}\)-shifts.
Furthermore, since \((y^1, y^2)\) are \(\gamma\)-normal geodesic
  coordinates, it follows that\footnote{See for instance equation
  (\ref{EQ_christoffel_prop_2}).} \(\gamma_{ij, k}(\zeta_0)=0\).
Consequently
\begin{equation}\label{EQ_christoffel_term_estimate}                     
    \tilde{\gamma}_{k\ell, n}(\zeta_0) =
    (f_{, kn}u_{, \ell}^0 +
    f_{, k}u_{, \ell n}^0 +
    u_{, kn}^0  f_{, \ell} +
    u_{, k}^0  f_{, \ell n} +
    f_{, kn } f_{, \ell} +
    f_{, k} f_{, \ell n}) \big|_{\zeta_0}
  \end{equation}
It will be convenient to evaluate a few functions at \(\zeta_0\); making
  use of the fact that \((y^1, y^2)\) is a \(\gamma\)-normal geodesic
  coordinate system, together with Lemma \ref{LEM_g_and_Gamma_properties}
  and the fact that \(\|dx^0\|_g = 1\) the following inequalities are
  straightforward to verify.
\begin{align}                                                             
    |f_{, k}(\zeta_0)| 
    &\leq \|df\|_{\gamma}\label{EQ_messy_1}\\
      |f_{,k\ell} (\zeta_0)|&\leq\|\nabla df\|_{\gamma}\label{EQ_messy_2}\\
      |u_{,k}^0(\zeta_0)|
    &\leq 1.\label{EQ_messy_3}
  \end{align}
Now we make use of some elementary results from Riemannian geometry which
  are elaborated upon in Appendix  \ref{SEC_riemannian}, specifically Lemma
  \ref{LEM_g_and_Gamma_properties}, Corollary
  \ref{COR_g_and_Gamma_properties}, and Definition
  \ref{DEF_second_fundamental_form}, and we establish the following.
\begin{align}                                                             
    |u_{, k\ell }^0(\zeta_0)|
    &=\big|\partial_{y^\ell} \big(dx^0(u_{,
      k})\big)\big|_{\zeta_0}\big|\notag\\\
    &\leq |(\nabla_{u_{,\ell}} dx^0)(u_{, k})| +
      |dx^0(\nabla_{u_{,\ell}} u_{, k})|\notag\\
    &\leq  |(\nabla_{u_{,\ell}} dx^0)(u_{, k})| +
      \|\nabla_{u_{,\ell}} u_{, k}\|_g\notag\\
    &\leq   \|(\nabla_{u_{,\ell}} u_{, k})^\top\|_g +
      \|(\nabla_{u_{,\ell}} u_{, k})^\bot\|_g\notag \\
    &=\|\nabla_{\partial_{y^\ell}} \partial_{y^k}\|_{u^*g} + \|B_u(u_{,
      k}, u_{, \ell})\|_g \notag\\
    &\leq \|B_u\|_g. \label{EQ_ukl_estimate}
  \end{align}
With this inequality established, we let \(Y=Y^i \partial_{y^i}\) and
  \(Z=Z^i\partial_{y^i}\), and establish inequality
  (\ref{EQ_gamma_tilde_estimate_1}).
The following functions, forms, vector fields, etc, are all assumed to
  be evaluated at \(\zeta_0\).
\begin{align*}                                                            
    \big| \langle Y, Z \rangle_{\tilde{\gamma}} - \langle Y, Z
      \rangle_{\gamma} \big|
    &= 
      \big|\tilde{\gamma}_{k\ell} Y^k Z^\ell -  \gamma_{k \ell} Y^k
      Z^\ell \big| \\
    &\leq 
      \big(|f_{, k}u_{, \ell}^0| + |u_{, k}^0  f_{, \ell}| + |f_{, k}
      f_{, \ell}|\big)|Y^k| |Z^\ell|\\
    &\leq 
      4\big(2\|df\|_\gamma + \|df\|_\gamma^2 \big)\|Y\|_\gamma
      \|Z\|_\gamma\\
    &\leq \epsilon \|Y\|_\gamma \|Z\|_\gamma;
  \end{align*}
  this establishes inequality (\ref{EQ_gamma_tilde_estimate_1}). 
To establish inequality (\ref{EQ_gamma_tilde_estimate_2}), we make
  use of the formula for the Christoffel symbols given in equation
  (\ref{EQ_christoffel}) as well as the fact that the Christoffel symbols
  vanish at the center of normal geodesic coordinates to estimate the
  following; again all functions, forms, etc are assumed to be evaluated
  at \(\zeta_0\).
\begin{align}                                                             
    \|\nabla_Y Z - \widetilde{\nabla}_Y Z\|_\gamma 
    &= 
      \|Y^i Z^j \widetilde{\Gamma}_{ij}^k
      \partial_{y^k}\|_{\gamma}\notag\\
    &\leq 
      2^3 \|Y\|_\gamma \|Z\|_\gamma \max_{i, j, k}
      |\widetilde{\Gamma}_{ij}^k|\notag\\
    &= 
      2^3 \|Y\|_\gamma \|Z\|_\gamma \max_{i, j, k} \big|
      \tilde{\gamma}^{k\ell} \big(\tilde{\gamma}_{i\ell, j} +
      \tilde{\gamma}_{j\ell, i} - \tilde{\gamma}_{ij, \ell}  \big)
      \big|\notag\\
    &\leq 
      2^6 \|Y\|_\gamma \|Z\|_\gamma \big( \max_{i, j}
      |\tilde{\gamma}^{ij}|  \big) \big(\max_{i, j, k}
      |\tilde{\gamma}_{ij, k}|\big).\label{EQ_messy_6}
  \end{align}
To continue, we make use of inequality (\ref{EQ_gamma_tilde_estimate_1}),
  which guarantees that \(|\delta_{ij} - \tilde{\gamma}_{ij}|\leq
  \epsilon\), and estimate
  \begin{align}                                                           
    \max_{i, j} |\tilde{\gamma}^{ij}| 
    &\leq 
      \frac{\max_{i, j}
      |\tilde{\gamma}_{ij}|}{\tilde{\gamma}_{11}\tilde{\gamma}_{22}
      - \tilde{\gamma}_{12}\tilde{\gamma}_{21}}\notag\\
    &\leq 
      \frac{1+\epsilon}{(1-\epsilon)^2 -\epsilon^2}\notag\\
    &\leq 
      2^2.\label{EQ_messy_5}
    \end{align}
By combining inequalities (\ref{EQ_christoffel_term_estimate})  -
  (\ref{EQ_ukl_estimate}) we have
    \begin{align*}                                                        
      |\tilde{\gamma}_{ij, k}(\zeta_0)| 
      &\leq 
	2 \|\nabla df\|_\gamma +2  \|df\|_\gamma\, \|B_u\|_\gamma +
	2 \|df\|_\gamma\, \|\nabla df\|_\gamma \\
      &\leq 
	2\Big( \frac{\epsilon}{2^{11} (\|B_u\|_\gamma+1)} + \frac{\epsilon
	\|B_u\|_\gamma}{2^{11} (\|B_u\|_\gamma+1)}+ \frac{\epsilon}{2^{11}
	(\|B_u\|_\gamma+1)}\Big)\\
      &\leq \frac{\epsilon}{2^8}.
    \end{align*}  
Combining this with inequalities (\ref{EQ_messy_6}) and (\ref{EQ_messy_5})
  then yields
  \begin{equation*}                                                       
      \|\nabla_Y Z - \widetilde{\nabla}_Y Z\|_\gamma \leq \frac{2^6 \cdot
      2^2 \epsilon }{2^8}\|Y\|_\gamma \|Z\|_\gamma\leq
      \epsilon\|Y\|_\gamma \|Z\|_\gamma.
    \end{equation*}
This establishes inequality (\ref{EQ_gamma_tilde_estimate_2}). 
Finally, recall (e.g. from Section \ref{SEC_riemannian}) that
  \(\widetilde{\nabla}_{\partial_{y^k}}dy^i=-\widetilde{\Gamma}_{ki}^\ell
  dy^\ell \) and at \(\zeta_0\) we have \(\nabla_{\partial_{y^k}}dy^i=0\).
Consequently, if \(\alpha\in \Omega^1(S)\) is a one-form written in
  coordinates as \(\alpha=\alpha_\ell dy^\ell\), then at \(\zeta_0\)
  we have
  \begin{align*}                                                          
      \|\nabla_Y \alpha - \widetilde{\nabla}_Y \alpha\|_\gamma &= \|Y^i
      \alpha_j \widetilde{\Gamma}_{ij}^k d x^k\|_{\gamma}\leq \epsilon\|
      Y\|_{\gamma} \|\alpha\|_{\gamma}
    \end{align*}
  as above. 
This establishes inequality (\ref{EQ_gamma_tilde_estimate_3}), and
  completes the proof of Lemma \ref{LEM_gamma_tilde_estimates}.
\end{proof}
Given a framed Hamiltonian manifold \((M, \eta)\) and an
  \(f\)-perturbation \((\tilde{u}, S, \tilde{\jmath})\) of a generally
  immersed pseudoholomorphic map \((u, S, j)\) in \(\mathbb{R}\times M\)
  with \((J, g)\) an \(\eta\)-adapted almost Hermitian structure, we
  define a one-form \(\tilde{\alpha}\) by the following.
  \begin{equation}\label{EQ_def_alpha_tilde}                              
    \tilde{\alpha}\in \Omega^1(S)\qquad\text{by}\qquad\tilde{\alpha}:=
      -(\tilde{u}^*da)\circ \tilde{\jmath}
    \end{equation}

\begin{lemma}[$d\tilde{\alpha}$ estimate]
  \label{LEM_d_alpha_tilde}
  \hfill\\
Let \(\tilde{\alpha}\) be the one-form defined by equation
  (\ref{EQ_def_alpha_tilde}).
Then 
  \begin{equation}\label{EQ_d_tilde_alpha_1}                              
    d\tilde{\alpha} = \big(\widetilde{\Delta} (a\circ \tilde{u})\big)
      {\rm vol}_{\tilde{\gamma}}
    \end{equation}
where \(\widetilde{\Delta}\) is the Laplace-Beltrami operator given by
  \begin{align*}                                                          
    \widetilde{\Delta} f 
    &= 
      {\rm tr} (\widetilde{\nabla} df)\\
    &= 
      \tilde{\gamma}^{ij}\Big(\widetilde{\nabla}_{\partial_{x^i}}(
      \widetilde{\nabla}_{\partial_{x^j}} f) -
      \widetilde{\nabla}_{\widetilde{\nabla}_{\partial_{x^i}}
      \partial_{x^j}}f\Big)
    \end{align*}
  and \({\rm vol}_{\tilde{\gamma}}\) is the volume form\footnote{See
  Section \ref{SEC_riemannian} for more details, specifically equation
  (\ref{EQ_hausdorff_measure}).} associated to \(\tilde{\gamma}\)
  and the orientation on \(S\).
Consequently, the following pointwise equality holds 
\begin{equation}\label{EQ_d_tilde_alpha_2}                                
  \|d\tilde{\alpha}\|_{\tilde{\gamma}}= |\widetilde{\Delta}(a\circ
    \tilde{u})|.
  \end{equation}
\end{lemma}
%
\begin{proof}
Fix \(\zeta_0\in S\), and let \((\tilde{s},
  \tilde{t})\) denote \(\tilde{\jmath}\)-holomorphic
  coordinates centered at \(\zeta_0\) for which
  \(\|\partial_{\tilde{s}}\|_{\tilde{\gamma}}= 1
  =\|\partial_{\tilde{t}}\|_{\tilde{\gamma}}\) at \(\zeta_0\).
Then in these coordinates we have
  \begin{align*}                                                          
    \tilde{\alpha} 
    &= 
      \tilde{\alpha}(\partial_{\tilde{s}})\, d\tilde{s} +
      \tilde{\alpha}(\partial_{\tilde{t}})\, d\tilde{t}\\
    &= 
      -da(T\tilde{u}\cdot \tilde{\jmath}\cdot \partial_{\tilde{s}})\,
      d\tilde{s}   - da(T\tilde{u}\cdot \tilde{\jmath} \cdot
      \partial_{\tilde{t}})\, d\tilde{t}\\
    &= 
      -da(T\tilde{u}\cdot  \partial_{\tilde{t}})\, d\tilde{s}   +
      da(T\tilde{u}\cdot  \partial_{\tilde{s}})\, d\tilde{t}\\
    &=
      -(a\circ \tilde{u})_{\tilde{t}}\, d\tilde{s} + (a\circ
      \tilde{u})_{\tilde{s}}\, d\tilde{t}.
    \end{align*}
Consequently
  \begin{align*}                                                          
    d\tilde{\alpha}&=\big( (a\circ \tilde{u})_{\tilde{s}\tilde{s}} +
      (a\circ \tilde{u})_{\tilde{t}\tilde{t}}\big)\, d\tilde{s}\wedge
      d\tilde{t}.
    \end{align*}
Next we note
  \begin{align*}                                                          
    (a\circ \tilde{u})_{\tilde{s}\tilde{s}}
    &=
      \widetilde{\nabla}_{\partial_{\tilde{s}}} (a\circ
      \tilde{u})_{\tilde{s}}=\widetilde{\nabla}_{\partial_{\tilde{s}}}\big(
      d(a\circ \tilde{u}) (\partial_{\tilde{s}})\big)\\
    &= 
      \big(\widetilde{\nabla}_{\partial_{\tilde{s}}}
      d(a\circ u)\big)(\partial_{\tilde{s}}) + d(a\circ
      \tilde{u})\big(\widetilde{\nabla}_{\partial_{\tilde{s}}}
      \partial_{\tilde{s}}\big).
    \end{align*}
  and thus
  \begin{align*}                                                          
    d\tilde{\alpha}&=
      \Big(\big(\widetilde{\nabla}_{\partial_{\tilde{s}}}
      d(a\circ u)\big)(\partial_{\tilde{s}})
      +\big(\widetilde{\nabla}_{\partial_{\tilde{t}}} d(a\circ
      u)\big)(\partial_{\tilde{t}}) \Big) \, d\tilde{s}\wedge
      d\tilde{t}\\
    &\qquad+ 
      \Big(d(a\circ
      \tilde{u})\big(\widetilde{\nabla}_{\partial_{\tilde{s}}}
      \partial_{\tilde{s}} + \widetilde{\nabla}_{\partial_{\tilde{t}}}
      \partial_{\tilde{t}}\big)\Big)\, d\tilde{s}\wedge d\tilde{t}\\
    &=
      \Big(\big(\widetilde{\nabla}_{\partial_{\tilde{s}}}
      d(a\circ \tilde{u})\big)(\partial_{\tilde{s}})
      +\big(\widetilde{\nabla}_{\partial_{\tilde{t}}} d(a\circ
      \tilde{u})\big)(\partial_{\tilde{t}}) \Big) \, d\tilde{s}\wedge
      d\tilde{t}
    \end{align*}
  where the final equality follows from Lemma
  \ref{LEM_complex_coord_property} below.
Finally, we evaluate this equality at \(\zeta_0\) and
  make use of the fact that at \(\zeta_0\) the ordered pair
  \((\partial_{\tilde{s}}, \partial_{\tilde{t}})\) is a positively oriented
  \(\tilde{\gamma}\)-orthonormal basis to conclude that indeed
  \begin{equation*}                                                       
    d\tilde{\alpha} = \big(\widetilde{\Delta} (a\circ \tilde{u})\big)
    {\rm vol}_{\tilde{\gamma}}.
    \end{equation*}
This establishes equation (\ref{EQ_d_tilde_alpha_1}); equation
  (\ref{EQ_d_tilde_alpha_2}) follows immediately.
\end{proof}
%

\begin{lemma}[conformal coordinates property]
  \label{LEM_complex_coord_property}
  \hfill \\
Let \(S\) be a smooth real \(2\)-dimensional manifold, equipped with
  an almost Hermitian structure \((j, \gamma)\).
Suppose further that \(\nabla\) denotes covariant differentiation
  with respect to the Levi-Civita connection associated to the metric
  \(\gamma\), and \((s, t)\) are local coordinates for which \(j\partial_s
  = \partial_t\).
Then 
  \begin{equation*}                                                       
      \nabla_{\partial_s} \partial_s + \nabla_{\partial_t}\partial_t = 0.
    \end{equation*}
\end{lemma}
%
\begin{proof}
Let \((x^1, x^2) = (s, t)\), so that \(\gamma=\gamma_{ij}\,dx^i\otimes
  dx^j\).
Observe that because \((j, \gamma)\) is almost Hermitian, we have
  \(\langle \partial_s, \partial_t \rangle_\gamma = \langle \partial_s,
  j\partial_s\rangle_\gamma =0\) and \(\langle \partial_s, \partial_s
  \rangle_\gamma = \langle j\partial_s , j \partial_s \rangle_\gamma =
  \langle \partial_t, \partial_t\rangle_\gamma\).
Consequently 
  \begin{equation*}                                                       
      \gamma = h\, ds^2 + h\, dt^2,
    \end{equation*}
  where \(h\) is a smooth positive function depending on \(s\) and \(t\). 
Let \(\Gamma_{ij}^k\) be denote the Christoffel symbols associated to
  the Levi-Civita connection corresponding to this metric.
Recall these are given by 
  \begin{equation*}                                                       
      \Gamma_{ij}^k=\frac{1}{2}\gamma^{k\ell}\big(\gamma_{i\ell, j} + \gamma_{j\ell,
      i} - \gamma_{ij, \ell}) \qquad\text{where}\qquad \gamma_{ij, k} =
      \frac{\partial}{\partial x^k} \gamma_{ij}.
    \end{equation*}
We compute
  \begin{align*}                                                          
      \nabla_{\partial_s} \partial_s + \nabla_{\partial_t}\partial_t 
      &= \big(\Gamma_{11}^1\partial_s + \Gamma_{11}^2\partial_t\big)
	+ \big(\Gamma_{22}^1\partial_s + \Gamma_{22}^2\partial_t\big) \\
      &=(\Gamma_{11}^1 + \Gamma_{22}^1)\partial_s + (\Gamma_{11}^2 +
	\Gamma_{22}^2)\partial_t;
    \end{align*}
  however,
  \begin{align*}                                                          
      \Gamma_{11}^1
      &=
	\frac{1}{2}\gamma^{11}(\gamma_{11, 1} + \gamma_{11,1} -
	\gamma_{11,1})=\frac{1}{2h}\frac{\partial}{\partial x^1} h\\
      \Gamma_{22}^1 &=
	\frac{1}{2}\gamma^{11}(\gamma_{21, 2} + \gamma_{21, 2} - \gamma_{22,
	1})=-\frac{1}{2h}\frac{\partial}{\partial x^1} h\\
      \Gamma_{11}^2 &=
	\frac{1}{2}\gamma^{22}(\gamma_{12, 1} + \gamma_{12, 1} - \gamma_{11,
	2})=-\frac{1}{2h}\frac{\partial}{\partial x^2} h\\
      \Gamma_{22}^2 &=
	\frac{1}{2}\gamma^{22}(\gamma_{22, 2} + \gamma_{22,2} -
	\gamma_{22,2})=\frac{1}{2h}\frac{\partial}{\partial x^2} h
    \end{align*}
The desired result is immediate.
\end{proof}
%

\begin{lemma}[coercive estimate]
  \label{LEM_linear_alg_coercive}
  \hfill\\
Let \(0<\epsilon <2^{-24} \), let \((M, \eta=(\lambda, \omega))\) be
  a framed Hamiltonian manifold, let \((J, g)\) be an \(\eta\)-adapted
  almost Hermitian structure on \(W:=\mathbb{R}\times M\), and let  \(u:(S,
  j)\to ( W, J)\) be a \(J\)-holomorphic map which is also an immersion.
Suppose further that \(\tilde{\mathbf{u}}=(\tilde{u}, S,
  \tilde{\jmath})\) is an \(f\)-perturbation of \((u, S, j)\), where
  \(f:S\to \mathbb{R}\) is a smooth function satisfying
  \begin{equation*}                                                       
      \|df\|_{\gamma}+ \|\nabla df\|_{\gamma} \leq \frac{\epsilon}{2^{11}
      (1+\|B_{u}\|_{\gamma})};
    \end{equation*}
  here \(\nabla\) denotes covariant differentiation with respect to the
  Levi-Civita connection associated to the metric \(\gamma=u^*g\),
  \(B_{u}\) denotes the second fundamental form of \(u:S\to
  \mathbb{R}\times M\) as given in Definition
  \ref{DEF_second_fundamental_form}, and finally by \(\|df\|_{\gamma}\),
  \(\|\nabla df\|_{\gamma}\), and \(\|B_{u}\|_{\gamma}\) we respectively
  mean the \(L^\infty\) norm of each.
\emph{Then}, for any \(\tilde{\gamma}\)-unit vector \(\tau\in TS\)
  we have
  \begin{equation*}
      \frac{1}{2}\leq \big((\tilde{u}^* da)\wedge \tilde{\alpha} +
      \tilde{u}^*\omega\big)(\tau, \tilde{\jmath}\tau).
    \end{equation*}
\end{lemma}
%
\begin{proof}
To begin, we recall that as a consequence of Lemma
  \ref{LEM_gamma_tilde_estimates}, we have
  \begin{equation*}                                                       
      (1-2\epsilon)\|\tau\|_{\tilde{\gamma}}\leq
      \frac{1}{(1+\epsilon)^{\frac{1}{2}}}
      \|\tau\|_{\tilde{\gamma}} \leq \|\tau\|_{\gamma}\leq
      \frac{1}{(1-\epsilon)^{\frac{1}{2}}} \|\tau\|_{\tilde{\gamma}}
      \leq (1+2\epsilon)\|\tau\|_{\tilde{\gamma}}
    \end{equation*}
  for any \(\tau\in TS\); here we have made use of the elementary
  inequalities \((1-\epsilon)^{-\frac{1}{2}}\leq 1+2\epsilon\) and
  \(1-2\epsilon \leq (1+\epsilon)^{-\frac{1}{2}}\) which hold for
  \(0\leq \epsilon \leq \frac{1}{10}\).
We henceforth assume \(\|\tau\|_{\tilde{\gamma}}=1\). 
Our first task is to estimate \(\|j-\tilde{\jmath}\|_{\gamma}\). 
To that end, we recall that \(j\) and \(\tilde{\jmath}\) are the
  almost complex structures uniquely determined by the respective metrics
  \(\gamma\) and \(\tilde{\gamma}\) and the orientation on \(S\).
Consequently, we can write 
  \begin{equation*}                                                       
      j\tau = \|\tau\|_{\gamma}\frac{ \tilde{\jmath}\tau- \frac{\langle
      \tilde{\jmath} \tau, \tau\rangle_{\gamma} }{\|\tau \|_{\gamma}^2}
      \tau}{ \| \tilde{\jmath}\tau- \frac{\langle \tilde{\jmath} \tau,
      \tau\rangle_{\gamma} }{\|\tau \|_{\gamma}^2} \tau \|_{\gamma} }.
    \end{equation*}
We can then write
  \begin{align}\label{EQ_j_tilde_j}                                      
      \|(j-\tilde{\jmath})\tau \|_{\gamma}
      &= \| (K-1)\tilde{\jmath}\tau - KL\tau\|_{\gamma}.
    \end{align}
  where
  \begin{equation*}                                                       
      L=\frac{\langle \tilde{\jmath}\tau, \tau\rangle_{\gamma}
      }{\|\tau\|_{\gamma}^2} \qquad\text{and}\qquad K=\frac{
      \|\tau\|_{\gamma}} {\| \tilde{\jmath}\tau- L\tau \|_{\gamma}}.
    \end{equation*}
We now estimate the relevant terms.
\begin{align*}
    |L|&=\frac{|\langle \tilde{\jmath}\tau,
    \tau\rangle_{\gamma}|}{\|\tau\|_{\gamma}^2} 
    \leq \frac{\epsilon
    \|\tilde{\jmath}\tau\|_\gamma}{\|\tau\|_\gamma} \leq
    \epsilon\frac{1+2\epsilon}{1-2\epsilon} \leq 2\epsilon.
  \end{align*}
\begin{equation*}
    |K|\leq \frac{\|\tau\|_\gamma}{\|\tilde{\jmath} \tau\|_\gamma -
    |L|\|\tau\|_\gamma} \leq  \frac{1+2\epsilon}{(1-2\epsilon) -2\epsilon
    (1+2\epsilon)  } \leq 2.
  \end{equation*}
\begin{align*}                                                            
    \big| |K|-1 \big| 
    &= 
      \Big| \frac{\|\tau\|_\gamma - \|\tilde{\jmath}\tau -
      L\tau\|_\gamma}{  \|\tilde{\jmath}\tau - L\tau\|_\gamma} \Big|\\
    &\leq
      \frac{  \big|\|\tau\|_\gamma-1 \big|  + \big| \|\tilde{\jmath}\tau
      - L\tau\|_\gamma-1  \big|}{ (1-2\epsilon) -2\epsilon (1+2\epsilon)
      }\\
    &\leq\frac{ 2\epsilon + (1+2\epsilon)^2  - (1-2\epsilon -2\epsilon
      (1+2\epsilon)) }{ (1-2\epsilon) -2\epsilon (1+2\epsilon) }\\
    &\leq 20\epsilon.
  \end{align*}
Combining these estimates with equation (\ref{EQ_j_tilde_j}) immediately
  yields
  \begin{equation*}                                                       
      \|(j-\tilde{\jmath})\tau \|_{\gamma} \leq \big(|K-1| + |K|\,
      |L|\big) \|\tau\|_\gamma \leq 25\epsilon(1+2\epsilon),
    \end{equation*}
and hence
  \begin{equation*}                                                       
      \|j-\tilde{\jmath} \|_{\gamma} \leq 25\epsilon(1+2\epsilon)^2\leq
      100\epsilon.
    \end{equation*}
Now making use of the fact that \(\omega(\partial_a, \cdot) =0\),
  and the fact that
  \begin{equation*}                                                       
      \tilde{u}(\zeta) = \exp_{u(\zeta)}^{g}(f(\zeta) \partial_a)
    \end{equation*}
we find
  \begin{align*}                                                          
      |\tilde{u}^*\omega(\tau, \tilde{\jmath}\tau) - u^*\omega(\tau,
	j\tau)|&=|u^*\omega(\tau, (\tilde{\jmath}-j)\tau)|\\
      &\leq 
	\|u^*\omega\|_{\gamma} \|\tau\|_{\gamma}^2
	\|j-\tilde{\jmath}\|_{\gamma}\\
      &\leq 
	100\epsilon (1+2\epsilon)^2\|\omega\|_g\\
      &\leq 
	400\epsilon\\ 
      &\leq
	\frac{1}{10}; 
    \end{align*}
  here we have made use of the easily verified estimate \(\|\omega\|_g
  \leq 1\).
To obtain a similar estimate for the \(\tilde{u}^*da\wedge
  \tilde{\alpha}\) term, we first observe that
  \begin{align*}                                                          
      (\tilde{u}^* da)
      &\wedge 
	(\tilde{u}^*da \circ \tilde{\jmath})(\tau, \tilde{\jmath}\tau)\\
      &= 
	(df + u^* da)\wedge \big(df \circ \tilde{\jmath}+u^*da
	\circ j+u^*da \circ (\tilde{\jmath}- j)\big) (\tau,
	(\tilde{\jmath}-j)\tau)\\
      &\quad+ 
	(df + u^* da)\wedge \big(df \circ \tilde{\jmath}+u^*da \circ
	j+u^*da \circ (\tilde{\jmath}-j )\big) (\tau, j\tau)\\
    \end{align*}
From this it follows that
  \begin{align*}
      \big|(&\tilde{u}^* da)\wedge (\tilde{u}^*da \circ
	\tilde{\jmath})(\tau, \tilde{\jmath}\tau)
	-(u^* da)\wedge (u^*da \circ j)(\tau, j\tau) |\\
      &\leq 
	\big(\|df\|_{\gamma} + \|u^*
	da\|_{\gamma}\big) \big(\|df\|_{\gamma}
	\|\tilde{\jmath}\|_{\gamma} + \|u^*da\|_{\gamma}+
	\|u^*da\|_{\gamma}\|\tilde{\jmath}-j\|_{\gamma}
	\big)\|\tilde{\jmath}-j\|_{\gamma} \|\tau\|_{\gamma}^2\\
      &\quad+ 
	\|df\|_{\gamma}\big(\|df\|_{\gamma} \|\tilde{\jmath}\|_{\gamma}
	+\|u^*da\|_{\gamma}\|\tilde{\jmath}\|_{\gamma}\big)
	\|\tau\|_{\gamma}^2\\
      &\quad+ 
	\|u^*da\|_{\gamma}\big( \|df\|_{\gamma}
	\|\tilde{\jmath}\|_{\gamma} +
	\|u^*da\|_{\gamma}\|\tilde{\jmath}-j\|_{\gamma} \big)
	\|\tau\|_{\gamma}^2\\
      &\leq 
	\big(\epsilon + 1\big) \big(\epsilon(1+100\epsilon)+ 1+
	100\epsilon\big)100\epsilon(1+2\epsilon)^2\\
      &\quad+ 
	\epsilon\big(\epsilon (1+100\epsilon) +(1+100\epsilon)\big)
	(1+2\epsilon)^2\\
      &\quad+ 
	\big( \epsilon (1+100\epsilon)+ 100\epsilon \big)
	(1+2\epsilon)^2\\
      &\leq 
	1000\epsilon\\
      &\leq \frac{1}{10}.
    \end{align*}
Combining this with our previous estimate for \(|\tilde{u}^*\omega(\tau,
  \tilde{\jmath}\tau) - u^*\omega(\tau, j \tau)|\) then yields
  \begin{align*}                                                          
      \big|\big((\tilde{u}^* da)\wedge \tilde{\alpha} 
      &+ 
	\tilde{u}^*\omega\big)(\tau, \tilde{\jmath} \tau) - 1\big|\\
      &\qquad=
	\big|\big((\tilde{u}^* da)\wedge \tilde{\alpha} +
	\tilde{u}^*\omega\big)(\tau, \tilde{\jmath} \tau) - \big((u^*
	da)\wedge  \alpha + u^*\omega\big)(\tau, j \tau) \big|\\
      &\qquad\leq \frac{1}{2}.
    \end{align*}
where we have written
  \begin{equation*}                                                       
      \tilde{\alpha}=-\tilde{u}^*da \circ
      \tilde{\jmath}\qquad\text{and}\qquad\alpha=-u^*da \circ
      j=u^*\lambda.
    \end{equation*}
This is the desired estimate, and hence completes the proof of Lemma
  \ref{LEM_linear_alg_coercive}.
\end{proof}

Before proceeding, we must make some quick estimates of \(d\lambda\) and
  establish some geometric constants.
This is the purpose of Lemma \ref{LEM_dlambda_bounds} and Definition
  \ref{DEF_ambient_geometry_constant}, below.

\begin{lemma}[$d\lambda$ bounds]
  \label{LEM_dlambda_bounds}
  \hfill\\
Let \((M, \eta)\) be a framed Hamiltonian manifold with  \(\eta=(\lambda,
  \omega)\), let \((J, g)\) be an \(\eta\)-adapted almost Hermitian
  structure on \(\mathbb{R}\times M\),  and fix \(Y\in T(\mathbb{R}\times
  M)\). 
Then
  \begin{equation*}                                                       
    |d\lambda(Y, JY)| \leq C_0 (da\wedge \lambda+\omega)(Y, JY)
    \end{equation*}
  where 
\begin{align}\label{EQ_def_C}                                             
  C_0 =\|i_{X_\eta} d\lambda\big|_\xi\|_g  +\|d\lambda\big|_{\xi}\|_g
  \end{align}
  where
  \begin{equation*}                                                       
    \|i_{X_\eta} d\lambda\big|_\xi\|_g =
    \sup_{\substack{ Y\in \xi \\ Y\neq 0}}\frac{|d\lambda(
    X_\eta, Y)|}{(\omega(Y, JY))^{\frac{1}{2}}} 
    \end{equation*}
  and
  \begin{equation*} 
    \|d\lambda\big|_{\xi}\|_g=\sup_{\substack{
    Y\in \xi \\ Y\neq 0}} \frac{|d\lambda (Y, JY)|}{\omega(Y,
    JY)}.
    \end{equation*} 
\end{lemma}
%
\begin{proof}
First, write \(Y=Y_1+Y_2\) with \(Y_1\in {\rm Span}(\partial_a, X_\eta)\)
  and  \(Y_2\in \xi\).
Then
  \begin{align*}                                                           
 |d\lambda(Y, JY)| &\leq 
      |d\lambda(Y_1, JY_1)| + |d\lambda(Y_1, JY_2)| + |d\lambda(Y_2,
      J Y_1)| + |d\lambda(Y_2, JY_2)| \\
    &=
      |d\lambda(Y_1, JY_2)| + |d\lambda(Y_2, J Y_1)| + |d\lambda(Y_2,
      JY_2)|\\
    &=
      |d\lambda(\lambda(Y)X_{\eta} ,J Y_2)| + |d\lambda(Y_2, \lambda(J
      Y)X_{\eta})| + |d\lambda(Y_2, J Y_2)|\\
    &=
      |\lambda(Y)| |d\lambda(X_{\eta} , J Y_2)| + |da(Y)| |d\lambda(Y_2,
      X_{\eta})| + |d\lambda(Y_2, JY_2)|\\
    &\leq 
      |\lambda(Y)| c_0 \big(\omega(Y_2 , JY_2)\big)^{\frac{1}{2}}
      + |da(Y)| c_0\big(\omega(Y_2 , JY_2)\big)^{\frac{1}{2}}+c_1
      \omega(Y_2, JY_2) \\
    &= 
      |\lambda(Y)| c_0 \big(\omega(Y , JY)\big)^{\frac{1}{2}} + |da(Y)|
      c_0\big(\omega(Y , JY)\big)^{\frac{1}{2}}+c_1 \omega(Y, JY) \\
    &\leq 
      {\textstyle \frac{1}{2}}c_0\big(|\lambda(Y)|^2 + |da(Y)|^2\big) +
      (c_0+c_1)\omega(Y, JY)\\
    &=
      {\textstyle \frac{1}{2}}c_0 (da\wedge \lambda)(Y, JY) +
      (c_0+c_1)\omega(Y, JY)\\
    &\leq
      C_0 (da\wedge \lambda+  \omega)(Y, JY) ,
    \end{align*} 
  where to obtain the final equality we have made use of Lemma
  \ref{LEM_J_diff}.  
This is the desired estimate.
\end{proof}
%

\begin{remark}[deviation from stable]
  \label{REM_deviation_from_stable}
  \hfill\\
We note that the proof of Lemma \ref{LEM_dlambda_bounds} immediately
  establishes the more precise estimate
  \begin{equation*}                                                       
    |d\lambda(Y, JY)| \leq {\textstyle \frac{1}{2}}c_0 (da\wedge
    \lambda)(Y, JY) + (c_0+c_1)\omega(Y, JY)
    \end{equation*}
  where \(c_0 = \|i_{X_\eta} d\lambda\big|_{\xi}\|_g\) and
  \(c_1=\|d\lambda \big|_{\xi}\|_g\).
We do not use this additional precision in this manuscript, however it
  highlights the fact that \(c_0\) becomes a means to measure the degree to
  which \((\lambda, \omega)\) fails to be a \emph{stable}
  Hamiltonian structure.
That is, \(c_0=0\) if and only if \((\lambda, \omega)\) is a stable
  Hamiltonian structure, and in some sense, the larger \(c_0\) is the
  further our Hamiltonian structure deviates from being stable.
\end{remark}
%

\begin{definition}[ambient geometry constant]
  \label{DEF_ambient_geometry_constant}
  \hfill\\
For each tuple \(\mathbf{h}=(M, \lambda, \omega, J, g)\), where \(M\) is a
  closed odd dimensional manifold, \(\eta=(\lambda, \omega)\) is a framed
  Hamiltonian structure on \(M\), and \((J, g)\) is an \(\eta\)-adapted
  almost Hermitian structure on \(\mathbb{R}\times M\), we define the
  following finite number:
  \begin{equation*}                                                       
    C_{\mathbf{h}}:=2\big(10 + {\rm max}(1, c_0+c_1)\big)
    \end{equation*}
  where
  \begin{equation*}                                                       
    c_0=\|i_{X_\eta} d\lambda\big|_\xi\|_g
    \qquad\text{and}\qquad
    c_1=\|d\lambda\big|_{\xi}\|_g
    \end{equation*}
  as above in Lemma \ref{LEM_dlambda_bounds}.
\end{definition}
%

We now proceed with our final elementary estimate.

\begin{lemma}[$d\tilde{\alpha}$ bounds]
  \label{LEM_d_alpha_tilde_bound}
  \hfill\\
Let \((M, \eta=(\lambda, \omega))\) be a framed Hamiltonian manifold, let
  \((J, g)\) be an \(\eta\)-adapted almost Hermitian structure on
  \(\mathbb{R}\times M\), and let \((\tilde{u}, \tilde{\jmath}, f, u, S,
  j)\) be a perturbed pseudoholomorphic map.
Let 
  \begin{equation*}                                                       
      0<\epsilon < \min(2^{-24}, (1+\sup_{\zeta\in
      {\rm supp}(f)}\|B_u(\zeta)\|_{{\gamma}})^{-1}).
    \end{equation*}
Suppose further that
  \begin{equation*}                                                       
      \|df\|_{{\gamma}}+ \|\nabla df\|_{{\gamma}} \leq
      \frac{\epsilon}{2^{11}(1+\|B_{{u}}\|_{{\gamma}})},
    \end{equation*}
  where \(\|df\|_{{\gamma}}\), \(\|\nabla df\|_{{\gamma}} \), and
  \(\|B_{{u}}\|_{{\gamma}}\) are the \(L^\infty\) norms over the support
  of \(f\).
\emph{Then}
  \begin{equation*}                                                       
      \sup_{\zeta\in S} \|d\tilde{\alpha}_\zeta\|_{\tilde{\gamma}}\leq
      {\textstyle \frac{1}{2}}C_{\mathbf{h}}
    \end{equation*}
  where \(C_{\mathbf{h}}\) is the ambient geometry constant associated to
  \(\mathbf{h}=(M, \lambda, \omega, J, g) \) and as provided in Definition
  \ref{DEF_ambient_geometry_constant}.
\end{lemma}
%

\begin{proof}
First we note that for \(\zeta\notin {\rm supp}(f)\),
  we have \(\tilde{u}=u\) and \(\tilde{\alpha}=u^*\lambda\),
  and then it follows from Lemma \ref{LEM_dlambda_bounds} that
\(\|d\tilde{\alpha}_\zeta\|_{\gamma} = \|u^*d\lambda_\zeta\|_{\gamma}
  \leq \frac{1}{2}C_{\mathbf{h}}\).
As such we will assume for the remainder of the proof that \(\zeta\in
  {\rm supp}(f)\), and for notational clarity we remove \(\zeta\) from
  the notation.

In light of Lemma \ref{LEM_d_alpha_tilde}, we have
  \begin{equation*}                                                       
      \|d\tilde{\alpha}\|_{\tilde{\gamma}}=|\widetilde{\Delta}(a\circ
      \tilde{u})|\qquad\text{and}\qquad\|d\alpha\|_{\gamma}=|\Delta(a\circ
      u)|.
    \end{equation*}

In order to estimate further, fix \(\zeta_0\in S\), and choose
  \({\gamma}\)-orthonormal coordinates \(({y}^1, {y}^2)\) and
  \(\tilde{\gamma}\)-orthonormal coordinates \((\tilde{y}_1, \tilde{y}_2)\)
  with each centered at \(\zeta_0\).
We choose these coordinates so that at \(\zeta_0\) we
  have \(\partial_{{y}^1}\wedge \partial_{\tilde{y}^1}= 0 =
  \partial_{{y}^2}\wedge \partial_{\tilde{y}^2}\).
Such coordinates can be constructed by fixing an auxiliary
  \({\gamma}\)-orthonormal basis \((\partial_{z^1}, \partial_{z^2})\)
  of  \(T_{\zeta_0}S\), observing that the matrix
  \begin{equation*}                                                       
    \left(
    \begin{matrix}
      \langle \partial_{z^1}, \partial_{z^1} \rangle_{\tilde{\gamma}}
      & \langle \partial_{z^2}, \partial_{z^1} \rangle_{\tilde{\gamma}}\\
      \langle \partial_{z^1}, \partial_{z^2} \rangle_{\tilde{\gamma}}
      & \langle \partial_{z^2}, \partial_{z^2} \rangle_{\tilde{\gamma}}
    \end{matrix}
    \right)
  \end{equation*} 
is symmetric, and hence has an orthonormal eigen-basis
  \((\begin{smallmatrix} c_{11}\\ c_{12}\end{smallmatrix}),
  (\begin{smallmatrix} c_{21}\\ c_{22}\end{smallmatrix}) \);
  defining \(\partial_{{y}^i} = c_{ik}\partial_{z^k}\) yields
  a \(\gamma\)-orthonormal basis of \(T_{\zeta_0} S\), and defining
  \(\partial_{\tilde{y}^i}=
  \|\partial_{{y}^i}\|_{\tilde{\gamma}}^{-1}\partial_{{y}^i}\) yields
  a \(\tilde{\gamma}\)-orthonormal basis of \(T_{\zeta_0} S\).
The coordinates \(({y}^1, {y}^2)\) and \((\tilde{y}^1, \tilde{y}^2)\)
  are respectively the \({\gamma}\) and \(\tilde{\gamma}\)
  normal geodesic coordinates respectively associated to the frames
  \((\partial_{{y}^1}, \partial_{{y}^2})\) and \((\partial_{\tilde{y}^1},
  \partial_{\tilde{y}^2})\).
For future use, we also note 
\begin{align*}                                                            
  \big| 1 - \|\partial_{y^i}\|_{\tilde{\gamma}}^2 \big| \leq \epsilon
  \end{align*}
  by Lemma \ref{LEM_gamma_tilde_estimates}, and hence  
  \begin{align}\label{EQ_norm_change_estimate}                            
    \|\partial_{y^i}\|_{\tilde{\gamma}}^{-2} \leq 2\qquad\text{and} \qquad
    \big| 1 - \|\partial_{y^i}\|_{\tilde{\gamma}}^{-2} \big| \leq
    2\epsilon.
    \end{align}
Next, we write 
\begin{align*}                                                            
  \tilde{\gamma} = \tilde{\gamma}_{ik} d\tilde{y}^i\otimes d\tilde{y}^k
  \qquad\text{and}\qquad\gamma = \gamma_{ik} dy^i \otimes dy^k
  \end{align*}
  so that by evaluating at \(\zeta_0\) we have \(\tilde{\gamma}
  =\delta_{ik} d\tilde{y}^i\otimes d\tilde{y}^k\) and \(\gamma =
  \delta_{ik} dy^i \otimes dy^k \) as well as
  \begin{align*}                                                          
      \widetilde{\Delta} (a\circ \tilde{u}) 
      &= 
	{\rm tr}(\widetilde{\nabla}\big( d (a\circ \tilde{u}))\big)\\
      &=
	\tilde{\gamma}^{ik} \big(\widetilde{\nabla}_{\partial_{\tilde{y}^i}}
	d (a\circ \tilde{u})\big) (\partial_{\tilde{y}^k}) - d (a\circ
	\tilde{u})( \widetilde{\nabla}_{\partial_{\tilde{y}^i}}
	\partial_{\tilde{y}^k}) \big)\\
      &=
	\big(\widetilde{\nabla}_{\partial_{\tilde{y}^1}}
	d(a\circ \tilde{u})\big) (\partial_{\tilde{y}^1}) +
	\big(\widetilde{\nabla}_{\partial_{\tilde{y}^2}}
	d(a\circ \tilde{u})\big) (\partial_{\tilde{y}^2})\\
      &=
	\big(\widetilde{\nabla}_{\partial_{\tilde{y}^1}}
	d(a\circ{u})\big) (\partial_{\tilde{y}^1}) +
	\big(\widetilde{\nabla}_{\partial_{\tilde{y}^2}}
	d(a\circ {u})\big) (\partial_{\tilde{y}^2})\\
      &\qquad
	+\big(\widetilde{\nabla}_{\partial_{\tilde{y}^1}}
	df \big) (\partial_{\tilde{y}^1}) +
	\big(\widetilde{\nabla}_{\partial_{\tilde{y}^2}}
	df \big) (\partial_{\tilde{y}^2})
    \end{align*}
  and similarly
  \begin{align*}                                                          
      \Delta (a\circ u) 
      &=
	\big(\nabla_{\partial_{y^1}}
	d(a\circ{u})\big) (\partial_{y^1}) +
	\big(\nabla_{\partial_{y^2}}
	d(a\circ {u})\big) (\partial_{y^2}).
      \end{align*}
We now estimate
  \begin{align*}                                                          
    \big| 
    \Delta (a\circ u) - \widetilde{\Delta} (a\circ \tilde{u})  
    \big| 
    &\leq 
    \big|	
    \big(\widetilde{\nabla}_{\partial_{\tilde{y}^1}} d(a\circ{u})\big)
    (\partial_{\tilde{y}^1}) - \big(\nabla_{\partial_{y^1}}
    d(a\circ{u})\big) (\partial_{y^1})
    \big|\\
    &\quad+
    \big|
    \big(\widetilde{\nabla}_{\partial_{\tilde{y}^2}} d(a\circ {u})\big)
    (\partial_{\tilde{y}^2}) - \big(\nabla_{\partial_{y^2}} d(a\circ
    {u})\big) (\partial_{y^2})
    \big|\\
    &\quad+
    \big|
    \big(\widetilde{\nabla}_{\partial_{\tilde{y}^1}} df \big)
    (\partial_{\tilde{y}^1})
    \big|\\
    &\quad+
    \big|
    \big(\widetilde{\nabla}_{\partial_{\tilde{y}^2}} df \big)
    (\partial_{\tilde{y}^2})
    \big|.
    \end{align*}
For the moment, let us write \(c_i =
  \|\partial_{y^i}\|_{\tilde{\gamma}}^{-1}\) so that
  \(\partial_{\tilde{y}^i}=c_i\partial_{y^i}  \)
  \begin{align*}                                                          
    \big|	
    \big(\widetilde{\nabla}_{\partial_{\tilde{y}^1}} d(a\circ{u})\big)
    (\partial_{\tilde{y}^1}) &- \big(\nabla_{\partial_{y^1}}
    d(a\circ{u})\big) (\partial_{y^1})
    \big|\\
    &=
    \big|	
    \big(c_1^2\widetilde{\nabla}_{\partial_{y^1}} d(a\circ{u})\big)
    (\partial_{y^1}) - \big(\nabla_{\partial_{y^1}}
    d(a\circ{u})\big) (\partial_{y^1})
    \big|
    \\
    &\leq 
    \big|	
    \big(c_1^2\widetilde{\nabla}_{\partial_{y^1}} d(a\circ{u})\big)
    (\partial_{y^1}) - c_1^2\big(\nabla_{\partial_{y^1}}
    d(a\circ{u})\big) (\partial_{y^1})
    \big|
    \\
    &\quad+  
    \big|	
    \big(c_1^2\nabla_{\partial_{y^1}} d(a\circ{u})\big)
    (\partial_{y^1}) - \big(\nabla_{\partial_{y^1}}
    d(a\circ{u})\big) (\partial_{y^1})
    \big|
    \\
    &\leq 
    c_1^2 \big|	
    \big(\widetilde{\nabla}_{\partial_{y^1}} d(a\circ{u})\big)
    (\partial_{y^1}) - \big(\nabla_{\partial_{y^1}}
    d(a\circ{u})\big) (\partial_{y^1})
    \big|
    \\
    &\quad+  
    |1-c_1^2|\cdot \big|	
    \big(\nabla_{\partial_{y^1}} d(a\circ{u})\big)
    (\partial_{y^1}) 
    \big|
    \\
    &\leq 
    c_1^2 \epsilon \|\partial_{y^1}\|_{\gamma}^2 +
    |1-c_1^2|\cdot \big|	
    \big(\nabla_{\partial_{y^1}} d(a\circ{u})\big)
    (\partial_{y^1}) 
    \big|
    \\
    &\leq 
    c_1^2 
    \epsilon + |1-c_1^2|\cdot\|B\|_{\gamma}
    \\
    &\leq 
    2 \epsilon +
    2\epsilon\|B\|_{\gamma}
    \\
    &\leq 3,
    \end{align*}
  where we have made use of the fact that 
  \begin{align*}                                                          
    \big|	
    \big(\nabla_{\partial_{y^1}} d(a\circ{u})\big) (\partial_{y^1}) 
    \big|
    &=
    \big|	
    \nabla_{\partial_{y^1}} \big(d(a\circ{u}) (\partial_{y^1}) \big) 
    -
    \big(d(a\circ{u})\big) (\nabla_{\partial_{y^1}} \partial_{y^1}) 
    \big|
    \\
    &=
    \big|	
    \nabla_{\partial_{y^1}} \big(d(a\circ{u}) (\partial_{y^1}) \big)
    \big|
    \\
    &=
    \big|	
    \nabla_{\partial_{y^1}} \big(da(u_{,1})  \big) 
    \big|
    \\
    &=
    \big|	
    \nabla_{u_{,l}} \big(da(u_{,1})  \big) 
    \big|
    \\
    &=
    \big|	
    \big(\nabla_{u_{,1}} da\big)(u_{,1})
    +
    da(\nabla_{u_{,1}} u_{,1})
    \big|
    \\
    &=
    \big|	
    da(\nabla_{u_{,1}} u_{,1})
    \big|
    \\
    &=
    \big|	
    da\big((\nabla_{u_{,1}} u_{,1})^\bot\big)
    +
    da\big((\nabla_{u_{,1}} u_{,1})^\top\big)
    \big|
    \\
    &=
    \big|	
    da\big((\nabla_{u_{,1}} u_{,1})^\bot\big)
    \big|
    \\
    &=
    \big|	
    da\big(B(u_{,1}, u_{,1})\big)
    \big|
    \\
    &\leq \|B\|_\gamma
    \end{align*}
    where we have used Corollary \ref{COR_da_is_parallel} from Appendix
    \ref{SEC_riemannian} which guarantees that \(\nabla da =0\).
We note that a similar estimate establishes that
  \begin{align*}                                                          
    \big|	
    \big(\widetilde{\nabla}_{\partial_{\tilde{y}^2}} d(a\circ{u})\big)
    (\partial_{\tilde{y}^2}) &- \big(\nabla_{\partial_{y^2}}
    d(a\circ{u})\big) (\partial_{y^2})
    \big|\leq 3.
    \end{align*}

Next we note that 
  \begin{align*}                                                          
    \big| \big(\widetilde{\nabla}_{\partial_{\tilde{y}^1}}
    df\big)(\partial_{\tilde{y}^1})\big| 
    &=c_1^2\big| \big(\widetilde{\nabla}_{\partial_{y^1}}
    df\big)(\partial_{y^1})\big| 
    \\
    &\leq
    c_1^2\big| \big(\nabla_{\partial_{y^1}}
    df\big)(\partial_{y^1})\big| + c_1^2\big|
    \big(\widetilde{\nabla}_{\partial_{y^1}} df\big)(\partial_{y^1}) -
    \big(\nabla_{\partial_{y^1}} df\big)(\partial_{y^1}) \big|
    \\
    &\leq c_1^2 \|\nabla df\|_{\gamma} + \epsilon c_1^2 \| df\|_\gamma 
    \\
    &\leq 1.
    \end{align*}
Similarly
  \begin{align*}                                                          
    \big| \big(\widetilde{\nabla}_{\partial_{\tilde{y}^2}}
    df\big)(\partial_{\tilde{y}^2})\big| \leq 1.
    \end{align*}
We conclude from these estimates that
  \begin{align*}                                                          
    \big| \Delta (a\circ u) - \widetilde{\Delta} (a\circ \tilde{u})  \big|
    \leq 10.
    \end{align*}

Combining these inequalities, we then find   \begin{align*}                                                          
      \|d\tilde{\alpha}\|_{\tilde{\gamma}} 
      &= 
	|\widetilde{\Delta}(a\circ \tilde{u})|& \text{by equation }
	(\ref{EQ_d_tilde_alpha_2}),\\
      &\leq 
	|{\Delta}(a\circ {u})|+10& \text{by above inequalities,}\\
      &=
	 \|-d(({u}^*da)\circ j)\|_{{\gamma}} + 10  & \text{by equation }
	(\ref{EQ_d_tilde_alpha_2}), \\
      &=
	\|-d({u}^*(da\circ J))\|_{\gamma}+ 10& \text{since }{u}\text{
	is }J\text{-holomorphic,}\\
      &=
	\|{u}^* d\lambda\|_{{\gamma}} + 10  & \text{by Lemma }
	\ref{LEM_J_diff},  \\
      &\leq 
	\textstyle{\frac{1}{2}}C_{\mathbf{h}}&\text{by Lemma
	\ref{LEM_dlambda_bounds}}.
    \end{align*}
This is the desired estimate, and hence we have completed the proof of
  Lemma \ref{LEM_d_alpha_tilde_bound}.
\end{proof}
%


\subsubsection{Core proofs}
With the elementary estimates established, we now move on to proving the
  main technical results, specifically Theorem
  \ref{THM_area_bound_estimate}, from which Theorem
  \ref{THM_area_bounds} follows as an immediate corollary.
We begin with a special case of the co-area formula.

\begin{lemma} [co-area formula with $\tilde{\alpha}$]
  \label{LEM_coarea_lambda}
  \hfill\\
Let \((W, g)\) be a smooth Riemannian manifold, and suppose \(S\) is a two
  dimensional manifold equipped with a smooth almost complex structure
  \(\tilde{\jmath}\).
To be clear, we require \(\partial S = \emptyset\).
Suppose \(\tilde{u}\colon S\to W\) satisfies \(\tilde{u}^*g(x, y) =
  \tilde{u}^*g(\tilde{\jmath} x, \tilde{\jmath} y)\) for all \(x,y\in
  TS\).
Let \(a\colon W\to \mathbb{R}\) be a smooth function, and
  \(\tilde{\alpha}\) the one-form on \(S\) defined by
  \begin{align*}                                                          
    \tilde{\alpha} = - (\tilde{u}^*da)\circ \tilde{\jmath} .
    \end{align*}
Finally, we assume \(a\circ\tilde{u}(S)\subset [a_0, a_1] \).
Then 
  \begin{equation*}                                                       
    \int_{S}( \tilde{u}^*da)\wedge \tilde{\alpha} =\int_{a_0}^{a_1}
    \Big(  \int_{(a\circ \tilde{u})^{-1}(t)\setminus \mathcal{X}}
    \tilde{\alpha}\Big)\, dt,
    \end{equation*}
  where \(\mathcal{X}:=\{\zeta\in S: d(a\circ \tilde{u})_\zeta = 0\}\).
\end{lemma}
\begin{proof}
We begin by defining \(\widetilde{S}:=S\setminus\mathcal{X}\) and making
  a few observations.
First,  \(\mathcal{X}\) is closed and hence \(\widetilde{S}\subset S\)
  is open,  and therefore it carries the structure of a smooth manifold.
Second, by definition we have \((\tilde{u}^*da) \wedge
  \tilde{\alpha}\big|_{\mathcal{X}}\equiv 0\), so
  \begin{equation*}                                                       
    \int_{S}(\tilde{u}^*da) \wedge \tilde{\alpha}= \int_{\widetilde{S}}
    (\tilde{u}^*da) \wedge \tilde{\alpha}.
    \end{equation*}
Observe that \(\tilde{u}:\widetilde{S}\to \mathbb{R}\times M\)
  is an immersion, and hence may be equipped with the metric
  \(\tilde{\gamma}=\tilde{u}^*g\).
The almost complex structure \(\tilde{\jmath}\) on \(S\) induces an
  orientation on \(\widetilde{S}\), and hence we have
  \begin{equation}\label{EQ_int_a_lambda_A}                               
    \int_{\widetilde{S}} (\tilde{u}^*da) \wedge \tilde{\alpha}=
    \int_{\widetilde{S}} \big( (\tilde{u}^*da)\wedge
    \tilde{\alpha}\big)(\nu, \tau) d\mu_{\tilde{\gamma}}^2,
    \end{equation}
  where \((\nu, \tau)\) is a positively oriented
  \(\tilde{\gamma}\)-orthonormal frame field, and
  \(d\mu_{\tilde{\gamma}}^2\) is the volume form on \(\widetilde{S}\)
  associated to the metric \(\tilde{\gamma}\); see Section
  \ref{SEC_riemannian} for further details.
Equation (\ref{EQ_int_a_lambda_A}) holds for arbitrary orthonormal frame
  field \((\nu, \tau)\), however we shall  henceforth make use of the
  following particular frame.
  \begin{equation*}                                                       
    \nu:= \frac{\widetilde{\nabla} (a\circ
    \tilde{u})}{\|\widetilde{\nabla} a\circ
    \tilde{u}\|_{\tilde{\gamma}}}\qquad\text{and}\qquad \tau:=
    \tilde{\jmath}\nu.
    \end{equation*}
Making use of the fact that \(\tilde{\jmath}\) is a
  \(\tilde{\gamma}\)-isometry and an almost complex structure, it is
  straightforward to verify the following.
\begin{align*}                                                            
  \tilde{\alpha}(\nu) &= 0 = \tilde{u}^*da(\tau) \\
  0 &< \tilde{u}^*da (\nu) = \tilde{\alpha}(\tau)\\
  1 &= \|\tau\|_{\tilde{\gamma}}^2 =\|\nu\|_{\tilde{\gamma}}^2 
  \end{align*}
Also,
  \begin{equation}\label{EQ_gradient_relationship_A_A}                      
    \|\widetilde{\nabla} (a\circ \tilde{u})\|_{\tilde{\gamma}} =
    \sup_{\substack{ x\in T_\zeta S\\ \|x\|_{\tilde{\gamma}}=1}}
    d(a\circ \tilde{u})(x) =\sup_{\substack{ x\in T_\zeta S\\
    \|x\|_{\tilde{\gamma}}=1}} da(T\tilde{u}\cdot x) =\tilde{u}^*da(\nu),
    \end{equation}
  and
  \begin{equation}\label{EQ_taut_nu_2_A_A}                                   
    \|\tilde{\alpha}\|_{\tilde{\gamma}}
    = \tilde{\alpha}(\tau)=-\tilde{u}^*da
    (\tilde{\jmath}\tilde{\jmath}\nu)=
    \|\widetilde{\nabla}(a\circ\tilde{u})\|_{\tilde{\gamma}}.
    \end{equation} 
With \((\nu, \tau)\) defined as such, we have the following.  
\begin{align*}                                                            
  \int_{\widetilde{S}} \big( (\tilde{u}^*da)\wedge
    \tilde{\alpha}\big)(\nu, \tau) d\mu_{\tilde{\gamma}}^2
  &=
    \int_{\widetilde{S}} da(Tu\cdot\nu )  \tilde{\alpha}(\tau)
    d\mu_{\tilde{\gamma}}^2\\
  &=
    \int_{\widetilde{S}} \|\widetilde{\nabla} a\circ
    \tilde{u}\|_{\tilde{\gamma}}^2 d\mu_{\tilde{\gamma}}^2\\
  \end{align*}
Next we recall the following version of the co-area formula. 
A proof is provided in Section \ref{SEC_co-area}.
\begin{proposition}[The co-area formula]
  \label{PROP_coarea_body}
  \hfill\\
Let \((S, \gamma)\) be a \(\mathcal{C}^1\) oriented Riemannian manifold
  of dimension two; we allow that \(S\) need not be complete\footnote{That
  is, there may exist Cauchy sequences, with respect to \(g\), which do
  not converge in \(S\).}.
Suppose that \(\beta:S\to [a, b]\subset\mathbb{R}\) is a
  \(\mathcal{C}^1\) function without critical points.  Let \(f:S\to [0,
  \infty)\) be a measurable function  with respect to \(d\mu_\gamma^2\).
Then
  \begin{equation}
    \int_S f\|\nabla \beta\|_\gamma \, d\mu_\gamma^2 = \int_a^b \Big(
    \int_{\beta^{-1}(t)} f \, d\mu_\gamma^1\Big) dt
    \end{equation}
  where \(\nabla\beta\) is the gradient of \(\beta\) computed with
  respect to the metric \(\gamma\).
\end{proposition}

We employ this result on \(\widetilde{S}\) where \(\gamma
  = \tilde{\gamma}\), \(\beta=a\circ \tilde{u}\), and
  \(f=\|\widetilde{\nabla} (a\circ \tilde{u})\|_{\tilde{\gamma}}\)
  to obtain
  \begin{align*}                                                          
    \int_{\widetilde{S}}\|\widetilde{\nabla} (a\circ
      \tilde{u})\|_{\tilde{\gamma}}^2 d\mu_{\tilde{\gamma}}^2
    &=
      \int_{a_0}^{a_1} \Big( \int_{(a\circ \tilde{u})^{-1}(t)\setminus
      \mathcal{X}}  \|\widetilde{\nabla} (a\circ
      \tilde{u})\|_{\tilde{\gamma}}\, d\mu_{\tilde{\gamma}}^1\Big) dt\\
    &= 
      \int_{a_0}^{a_1} \Big( \int_{(a\circ \tilde{u})^{-1}(t)\setminus
      \mathcal{X} }  \tilde{\alpha}(\tau)\, d\mu_{\tilde{\gamma}}^1\Big)
      dt\\
    &=
      \int_{a_0}^{a_1} \Big( \int_{(a\circ \tilde{u})^{-1}(t)\setminus
      \mathcal{X}}  \tilde{\alpha}\Big) dt,
    \end{align*}
  and hence by combining equalities we have
  \begin{equation*}                                                       
    \int_{S}(\tilde{u}^*da) \wedge \tilde{\alpha}=\int_{a_0}^{a_1}
    \Big( \int_{(a\circ \tilde{u})^{-1}(t)\setminus \mathcal{X}}
    \tilde{\alpha}\Big) dt,
    \end{equation*}
  which is the desired equality.
This completes the proof of Lemma \ref{LEM_coarea_lambda}.
\end{proof}

We are now prepared to state and proof the main result of
  this section.
While rather technical in its statement, it is applicable throughout the
  remainder of this manuscript without need of generalization.
A more accessible corollary is stated immediately afterwards.

\begin{theorem}[area bound estimate]
  \label{THM_area_bound_estimate}
  \hfill \\
Let \((M, \eta=(\lambda, \omega))\) be a framed Hamiltonian manifold, let
  \((J, g)\) be an \(\eta\)-adapted almost Hermitian structure on
  \(\mathbb{R}\times M\), let  \(C_{\mathbf{h}}\) be the ambient geometry
  constant given in Definition \ref{DEF_ambient_geometry_constant}, and
  let \(E_0>0\) be a positive constant,
Then for each \(r>0\) and tract\footnote{
  Here we mean a tract of perturbed pseudoholomorphic map in the sense of
  Definition \ref{DEF_tract_of_perturbed_J_map}.
  } 
  of perturbed pseudoholomorphic map \((\tilde{u}, \widetilde{S},
  \tilde{\jmath}, f, u, S, j)\), satisfying
  \begin{enumerate}                                                       
  \item   
    \begin{equation*}                                                     
      \|df\|_{{\gamma}}+ \|\nabla df\|_{{\gamma}} \leq
      \frac{\epsilon}{2^{11}(1+\|B_{{u}}\|_{{\gamma}})},
      \end{equation*}
  where 
  \begin{equation*}                                                       
      0<\epsilon < \min(2^{-24}, (1+\sup_{\zeta\in
      {\rm supp}(f)}\|B_u(\zeta)\|_{{\gamma}})^{-1}),
    \end{equation*}
  and \(\|df\|_{{\gamma}}\), \(\|\nabla df\|_{{\gamma}} \), and
  \(\|B_{{u}}\|_{{\gamma}}\) are the \(L^\infty\) norms over the support
  of \(f\),
  \item \(a\circ\tilde{u}(\widetilde{S})\subset [0, r]\),
    \item \((0, r)\cap a\circ \tilde{u}(\partial_0 \widetilde{S})
      = \emptyset\),
    \item \(\int_{\widetilde{S}}u^*\omega \leq E_0\),
    \item \(0\) and \(r\) are regular values of \(a\circ \tilde{u}\),
    \end{enumerate}
  we have
  \begin{equation}\label{EQ_lambda_growth_inequality_1}                   
    \int_{(a\circ \tilde{u})^{-1}(r)} \tilde{\alpha}\leq
    \Big(C_{\mathbf{h}} E_0+ \int_{(a\circ \tilde{u})^{-1}(0)}
    \tilde{\alpha}\Big) e^{C_{\mathbf{h}} r},
    \end{equation}
  and
  \begin{equation}\label{EQ_area_growth_inequality}                       
    \int_{\widetilde{S}} (\tilde{u}^*da) \wedge \tilde{\alpha}+
    \tilde{u}^*\omega\leq \Big( C_{\mathbf{h}}^{-1} \int_{(a\circ
    \tilde{u})^{-1}(0)} \tilde{\alpha} +E_0 \Big) \big(e^{ C_{\mathbf{h}}
    r }-1) + E_0.
    \end{equation}
Similarly for each \(r>0\) and tract of perturbed pseudoholomorphic map,
  \((\tilde{u}, \widetilde{S}, \tilde{\jmath}, f, u, S, j)\), for which
  \begin{enumerate}                                                       
    \item \(a\circ\tilde{u}(\widetilde{S})\subset [-r, 0]\),
    \item \((-r, 0)\cap a\circ \tilde{u}(\partial_0 \widetilde{S})
      = \emptyset\),
    \item \(\int_{\widetilde{S}}u^*\omega \leq E_0\),
    \item \(0\) and \(-r\) are regular values of \(a\circ \tilde{u}\),
    \end{enumerate}
  we have
  \begin{equation}\label{EQ_lambda_growth_inequality_2}
    \int_{(a\circ \tilde{u})^{-1}(-r)} \tilde{\alpha}\leq
    \Big(C_{\mathbf{h}} E_0+ \int_{(a\circ \tilde{u})^{-1}(0)}
    \tilde{\alpha}\Big) e^{ C_{\mathbf{h}} r},
    \end{equation}
  and inequality (\ref{EQ_area_growth_inequality}) again holds.  
\end{theorem}
%
\begin{proof}
Observe that the above problem has two distinct cases: the positive
  case and the negative case; we refer to each as such.
We also pause to recall some of the geometry involved.
First, because \((\tilde{u}, \widetilde{S}, \tilde{\jmath}, f, u, S,
  j)\) is a tract of perturbed pseudoholomorphic curve in the sense of
  Definition \ref{DEF_tract_of_perturbed_J_map}, it follows that
  \(\tilde{u}:\widetilde{S}\to \mathbb{R}\times M\) is proper, and because
  \(M\) is compact and \(\tilde{u}(\widetilde{S})\subset [-r, r]\times M\),
  it follows that \(\widetilde{S}\) is compact.
Second, the boundary of \(\widetilde{S}\), if it is not empty, is
  piecewise smooth, and can be written as the union of two sets, namely
  \(\partial_0 \widetilde{S}\)  and \(\partial_1 \widetilde{S}\) where the
  connected components of the former are  level sets of \(a\circ
  \tilde{u}\), the latter are   integral curves of
  \(\widetilde{\nabla}(a\circ \tilde{u})\), and the set
  \(\partial_0\widetilde{S} \cap \partial_1\widetilde{S}\) is finite and
  consists of precisely those non-smooth points of \(\partial
  \widetilde{S}\).

To proceed with the proof, we begin by fixing \(\delta>0\). 
For the positive case, we define \(\mathcal{R}^+\subset
  [0, r]\) to be the set of regular values of the function
  \(a\circ\tilde{u}:\widetilde{S}\to \mathbb{R}\).
In the negative case we define \(\mathcal{R}^-\subset[0,
  r]\) to be the set of regular values of the function
  \(-a\circ\tilde{u}:\widetilde{S}\to \mathbb{R}\).
Depending on the case, we then define the following functions.
\begin{equation*}                                                         
  h^\pm: \mathcal{R}^\pm\subset [0, r]\to [0, \infty)\quad\text{given
  by}\quad h^\pm(s):=\int_{(a\circ \tilde{u})^{-1}(\pm s)}\tilde{\alpha}.
  \end{equation*}
Note that \(\mathcal{R}^\pm\) are relatively open subsets of \([0,
  r]\), and by Sard's theorem they each are of full measure in \([0, r]\).
Recall that any open set of \(\mathbb{R}\) can be written as the
  countable union of disjoint open intervals, and hence we may write
  \(\mathcal{R}^\pm = \cup_{k\in\mathbb{N}_0} \mathcal{O}_k^\pm\) with
  the \(\mathcal{O}_k^\pm\) relatively open and pairwise disjoint.
Henceforth, we will only consider the case that \(\mathcal{O}_k^\pm\neq
  \emptyset\) for all \(k\in \mathbb{N}_0\).
More generally, one can always assume \(\mathcal{O}_k^\pm\) is never
  the empty set, however in such a case there may only be finitely many
  such open sets \(\mathcal{O}_k^\pm\);  the proof in the finite case
  however is easily adapted from the infinite case.

Without loss of generality, we may re-index the
  \(\{\mathcal{O}_k^\pm\}_{k\in{\mathbb{N}}}\) to guarantee that  \(0\in
  \mathcal{O}_0^\pm\), and \(r\in \mathcal{O}_1^\pm\).
Next, using the fact that each \(\mathcal{O}_k^\pm\) is a non-empty
  open interval, we may choose a sequence of finite sets of closed
  intervals \(\{\mathcal{I}_{0, n}^\pm, \mathcal{I}_{1, n}^\pm, \ldots,
  \mathcal{I}_{n, n}^\pm \}_{n\in \mathbb{N}}\) with the following
  properties.
\begin{enumerate}[($\mathcal{I}$1)]                                       
  \item\label{EN_I1} \(\mathcal{I}_{k, n}^\pm = [a_{k, n}^\pm, b_{k,
    n}^\pm]\)
  \item\label{EN_I2} \(0= a_{0, n}^\pm <  b_{0, n}^\pm <  a_{1, n}^\pm <
    b_{1, n}^\pm  < \cdots <  a_{n, n}^\pm  <  b_{n, n}^\pm=r\)
  \item\label{EN_I4} for each \(n\in \mathbb{N}\) we have \(\cup_{k=0}^n
    \mathcal{I}_{k, n}^\pm\subset   \cup_{k=0}^{n+1} \mathcal{I}_{k,
    n+1}^\pm  \)
  \item\label{EN_I5} 
    \begin{align*}                                                        
      r=\lim_{n\to\infty}\sum_{k=0}^n |b_{k, n}^\pm - a_{k, n}^\pm|
      =\lim_{n\to\infty}\sum_{k=0}^n (b_{k, n}^\pm - a_{k, n}^\pm).
      \end{align*}
  \end{enumerate}

Recall that because \(\widetilde{S}\) is compact, it follows that
  \(\int_{\widetilde{S}} (\tilde{u}^*da)\wedge \tilde{\alpha}< \infty\).
By Lemma \ref{LEM_coarea_lambda}, we then find that in the positive case
  we have
  \begin{equation*}                                                       
    \int_{0}^{r}\Big( \int_{(a\circ \tilde{u})^{-1}(s)\setminus
      \mathcal{X}} \tilde{\alpha}\Big)\, ds
    =
      \int_{\widetilde{S}\setminus \partial \widetilde{S} }(
      \tilde{u}^*da) \wedge \tilde{\alpha}
    = 
      \int_{\widetilde{S}} (\tilde{u}^*da) \wedge \tilde{\alpha}
      < \infty,
    \end{equation*}
  where \(\mathcal{X}=\{\zeta\in \widetilde{S}: d(a\circ \tilde{u})_\zeta
  =0\}\); here we have also made use of the fact that \(\partial
  \widetilde{S}\subset \widetilde{S}\) has zero measure -- this follows
  from the fact that \(\partial \widetilde{S}\subset \widetilde{S}\)
  is a piecewise smooth embedded
  submanifold of codimension one. 
Also, the orientation on  \((a\circ \tilde{u})^{-1}(s)\setminus
  \mathcal{X}\) is such that \(\tilde{\jmath}\widetilde{\nabla}(a\circ
  \tilde{u})\) is a positive frame field.  Similarly in the negative case
  we have
  \begin{equation*}                                                       
    \int_{-r}^{0}\Big( \int_{(a\circ \tilde{u})^{-1}(s)\setminus
      \mathcal{X}} \tilde{\alpha}\Big)\, ds
    =
      \int_{\widetilde{S}\setminus \partial \widetilde{S} }(
      \tilde{u}^*da) \wedge \tilde{\alpha}
    = 
      \int_{\widetilde{S}} (\tilde{u}^*da) \wedge \tilde{\alpha}
      < \infty.
    \end{equation*}
Recall that \(\tilde{\alpha}=-\tilde{u}^*da\circ\tilde{\jmath}\)
  and \((v, \tilde{\jmath}v)\) is a positively oriented basis for
  \(v\neq 0\), and hence \(\tilde{u}^*da\wedge \tilde{\alpha}\) is a
  non-negative function multiple of the area form on \(\widetilde{S}\)
  associated to \(\tilde{\gamma}\), and hence the functions
  \(\tilde{h}^\pm(s) :=\int_{(a\circ \tilde{u})^{-1}(\pm s)\setminus
  \mathcal{X}}\tilde{\alpha}\) are integrable.
Recall a consequence of the dominated convergence theorem:  if
  \(\{A_k\}_{k\in \mathbb{N}}\) is a sequence of sets \(A_k\subset [-r,
  r]\) satisfying \(A_{k+1}\subset A_k\) for all \(k\in \mathbb{N}\)
  and the measure of the \(A_k\) tends to zero as \(k\to \infty\), then
  \(\int_{A_k}\tilde{h}(s) ds \to 0 \).
Again employing Lemma \ref{LEM_coarea_lambda} together with this latter
  application of the dominated convergence theorem we find that
  \begin{equation*}                                                       
    0=\lim_{n\to \infty} \sum_{k=0}^{n-1}\int_{\widetilde{S}_{b_{k,
    n}^\pm}^{a_{k+1, n }^\pm}} (\tilde{u}^*da) \wedge \tilde{\alpha},
    \end{equation*}
  where 
  \begin{equation}\label{EQ_Sab}                                          
    \widetilde{S}_{a_0}^{a_1}= \{\zeta\in \widetilde{S}: a_0 \leq a\circ
    \tilde{u}(\zeta)\leq a_1\}.
    \end{equation}

Consequently, there exists an \(n\in \mathbb{N}\) with the property that
  \begin{equation}\label{EQ_small_integral}                               
    C_{\mathbf{h}} e^{C_{\mathbf{h}} r} \sum_{k=0}^{n-1}\Big(
    \int_{\widetilde{S}_{b_{k,
    n}^\pm}^{a_{k+1, n }^\pm}} (\tilde{u}^* da)
    \wedge\tilde{\alpha}\Big)< \delta
    \end{equation}
  with \(\delta>0\) defined at the start of this proof. 
For the remainder of the proof we shall assume \(n\) is fixed
  sufficiently large so that (\ref{EQ_small_integral})  holds.

We now aim to study the growth rate of \(h^\pm\). 
Recalling the definition of \(\mathcal{R}^\pm\) and the fact that
  the \(\mathcal{R}^\pm\) are open, we see that \(h^\pm\) is smooth on
  \(\mathcal{R}^\pm\), and hence we will estimate \(|(h^\pm)'|\).
To do this, it will be convenient to have first made the following
  definitions.
\begin{equation*}                                                         
  G^+:\mathcal{R}^+\to [0, \infty)\qquad\text{given by}\qquad
  G^+(s):=\int_{S_{0}^s} \tilde{u}^*\omega.
  \end{equation*}
\begin{equation*}                                                         
  G^-:\mathcal{R}^-\to [0, \infty)\qquad\text{given by}\qquad
  G^-(s):=\int_{S_{-s}^0} \tilde{u}^*\omega.
  \end{equation*}

Using the fact that \(\omega\) evaluates non-negatively on complex lines,
  it follows that \(u^*\omega\) evaluates non-negatively on positively
  oriented bases; then by definition of \(\tilde{u}\), particularly
  property (p\ref{EN_p4}) of Definition \ref{DEF_perturbed_J_map},
  together with the fact that \(\omega(\partial_a, \cdot) \equiv 0\),
  it follows that \(u^*\omega=\tilde{u}^*\omega\); these two results
  together then show that the \(G^\pm\) are monotone increasing, and
  since they are differentiable, we must have \((G^\pm)'\geq 0\).
Recalling that the \(\mathcal{R}^\pm\) are open, we assume \(s\in
  \mathcal{R}^\pm\setminus \{r\}\); then
  \begin{align*}                                                          
    |(h^+)'(s)|&=\Big|\lim_{\epsilon\to 0^+}
      \epsilon^{-1}\big(h^+(s+\epsilon)-h^+(s)\big)\Big|\\
    &=
      \lim_{\epsilon\to 0^+} \epsilon^{-1}\Big|\int_{(a\circ
      \tilde{u})^{-1}( s+ \epsilon)}\tilde{\alpha} - \int_{(a\circ
      \tilde{u})^{-1}( s)}\tilde{\alpha}\Big|\\
    &=
      \lim_{\epsilon\to 0^+}
      \epsilon^{-1}\Big|\int_{\widetilde{S}_{s}^{s+\epsilon}}d
      \tilde{\alpha}\Big|\\
    &=
      \lim_{\epsilon\to 0^+}
      \epsilon^{-1}\Big|\int_{\widetilde{S}_{s}^{s+\epsilon}\setminus              
        \mathcal{X}}d \tilde{\alpha}\Big|\\
    &\leq
      \lim_{\epsilon\to 0^+}
      \epsilon^{-1}\int_{\widetilde{S}_{s}^{s+\epsilon}\setminus
      \mathcal{X}} \|d
      \tilde{\alpha}\|_{\tilde{\gamma}}d\mu_{\tilde{\gamma}}^2 \\
    &\leq 
      \lim_{\epsilon\to 0^+}
      \epsilon^{-1}{\textstyle\frac{1}{2}}C_{\mathbf{h}}
      \int_{\widetilde{S}_{s}^{s+\epsilon}\setminus \mathcal{X}}
      d\mu_{\tilde{\gamma}}^2 \\
    &\leq  
      {\textstyle\frac{1}{2}}C_{\mathbf{h}}  \lim_{\epsilon\to 0^+}
      \epsilon^{-1} \int_{\widetilde{S}_{s}^{s+\epsilon}}
      2\big((\tilde{u}^*da) \wedge \tilde{\alpha} + \tilde{u}^*\omega\big)
      \\
    &=
      C_{\mathbf{h}}\Big(\lim_{\epsilon\to 0^+}\epsilon^{-1}
      \int_{\widetilde{S}_{s}^{s+\epsilon}} (\tilde{u}^*da)\wedge
      \tilde{\alpha} + \lim_{\epsilon\to 0^+}\epsilon^{-1}
      \int_{\widetilde{S}_{s}^{s+\epsilon}}\tilde{u}^*\omega\Big)\\
    &= 
      C_{\mathbf{h}}\Big( \int_{(a\circ u)^{-1}(s)} \tilde{\alpha}
      +(G^+)'(s)\Big)\\
    &= 
      C_{\mathbf{h}} \big(h^+(s) +  (G^+)'(s)\big), 
    \end{align*} 
  where to obtain the third equality we have made use of Stokes'
  theorem and Lemma \ref{LEM_char_fol_grad}, to obtain the second
  inequality we have employed Lemma \ref{LEM_d_alpha_tilde_bound}, to
  obtain the third inequality we have employed  Lemma
  \ref{LEM_linear_alg_coercive}, and to obtain the sixth equality we have
  employed Lemma \ref{LEM_coarea_lambda}.
A similar computation shows that
  \begin{equation*}                                                       
    |(h^-)'(s)|\leq C_{\mathbf{h}}\big(h^-(s) +  (G^-)'(s)\big).
    \end{equation*}
Summarizing, we have shown that for \(s\in \mathcal{R}^\pm\), we have
  the following differential inequalities.
  \begin{equation}                                                        
    \label{EQ_dh_est}(h^\pm)'(s)\leq C_{\mathbf{h}} \big(h^\pm(s) +
    (G^\pm)'(s)\big)\\
    \end{equation}
Assume that \(a_{k, n}^\pm\) and \(b_{k, n}^\pm\) are as in
  (\(\mathcal{I}\)\ref{EN_I1}), so that \([a_{k, n}^\pm, b_{k,
  n}^\pm]\subset \mathcal{R}^\pm\); then integrate inequality
  (\ref{EQ_dh_est}) on \([a_{k, n}^\pm, s]\subset [a_{k, n}^\pm, b_{k,
  n}^\pm]\) to obtain the following.
  \begin{align*}                                                          
    h^\pm(s)&\leq h^\pm(a_{k, n}^\pm)+C_{\mathbf{h}}\int_{a_{k, n}^\pm}^s
      h^\pm(t)\,
      dt  + C_{\mathbf{h}} \big( G^+(s) - G^+(a_{k, n}^+)\big)\\
    &\leq 
      h^\pm(a_{k, n}^\pm)+C_{\mathbf{h}}\int_{\widetilde{S}_{a_{k,
      n}^\pm}^{b_{k, n}^\pm}}u^*\omega +C_{\mathbf{h}}\int_{a_{k,
      n}^\pm}^s h^\pm(t)\, dt,
    \end{align*}
  or in short,
  \begin{equation}\label{EQ_gronwall_est1}                                
    h^\pm(s) \leq h^\pm(a_{k,
    n}^\pm)+C_{\mathbf{h}}\int_{\widetilde{S}_{a_{k,
    n}^\pm}^{b_{k, n}^\pm}}u^*\omega +C_{\mathbf{h}}\int_{a_{k, n}^\pm}^s
    h^\pm(t)\, dt,
    \end{equation}
  for \(s\in [a_{k, n}^+, b_{k, n}^+]\).
In order to obtain a sharper estimate on \(h^\pm\), we now employ
  Gronwall's inequality.

\begin{lemma}[Gronwall's inequality]
  \label{LEM_gronwall}
  \hfill\\
Assume that for \(t_0< t_1\), the functions \(\phi, \psi:[t_0, t_1]\to
  [0, \infty)\) are continuous.
Suppose further that \(\delta_1>0\) and \(\delta_3>0\) are positive
  constants, and the following estimate is satisfied for all \(t\in
  [t_0, t_1]\)
  \begin{equation*}                                                       
    \phi(t)\leq \delta_1 \int_{t_0}^t\psi(t)\phi(t) ds + \delta_3.
    \end{equation*}
Then for each \(t\in[t_0, t_1]\) the following estimate also holds.
  \begin{equation*}
    \phi(t)\leq \delta_3 e^{\delta_1 \int_{t_0}^t \psi(s) ds}
    \end{equation*}
\end{lemma}
%
\begin{proof}
See Section 1.3 of \cite{Ve}.
\end{proof}

Applying Lemma \ref{LEM_gronwall} to inequality (\ref{EQ_gronwall_est1})
  then yields
  \begin{align}                                                           
    h^\pm(s)&\leq \Big( h^\pm(a_{k, n}^\pm)+C_{\mathbf{h}}E_{a_{k,
      n}^\pm}^{b_{k, n}^\pm}\Big) e^{ C_{\mathbf{h}}( s-a_{k,
      n}^\pm)}\notag\\
    &=
      h^\pm(a_{k, n}^\pm)e^{ C_{\mathbf{h}}( s-a_{k, n}^\pm)}
      +C_{\mathbf{h}}E_{a_{k, n}^\pm}^{b_{k, n}^\pm} e^{ C_{\mathbf{h}}(
      s-a_{k, n}^\pm)} \label{EQ_f_estimate_1}
    \end{align}
  where we have abbreviated
  \begin{equation*}                                                       
    E_{a_0}^{a_1}= \int_{\widetilde{S}_{a_0}^{a_1}}u^*\omega.
    \end{equation*}
Assume \( k< n\) and estimate \(h(a_{k+1, n}^\pm)\) as follows, again
  making use of Lemma \ref{LEM_char_fol_grad} and Lemma
  \ref{LEM_d_alpha_tilde_bound}, we find
  \begin{align*}                                                          
    h^\pm(a_{k+1, n}^\pm)&= 
      h^\pm(b_{k, n}^\pm) + \int_{\widetilde{S}_{b_{k, n}^\pm}^{a_{k+1,
      n}^\pm}}d \tilde{\alpha}\\
    &\leq 
      h^\pm(b_{k, n}^\pm) + {\textstyle
      \frac{1}{2}}C_{\mathbf{h}}\int_{\widetilde{S}_{b_{k,
      n}^\pm}^{a_{k+1, n}^\pm}} d\mu_{\tilde{\gamma}}^2\\
    &\leq 
      h^\pm(b_{k, n}^\pm) + C_{\mathbf{h}}\int_{\widetilde{S}_{b_{k,
      n}^\pm}^{a_{k+1, n}^\pm}}  (\tilde{u}^*da)\wedge \tilde{\alpha}
      + C_{\mathbf{h}}\int_{\widetilde{S}_{b_{k, n}^\pm}^{a_{k+1, n}^\pm}}
      u^* \omega\\
    &=
      h^\pm(b_{k, n}^\pm) +C_{\mathbf{h}}\int_{\widetilde{S}_{b_{k,
      n}^\pm}^{a_{k+1, n}^\pm}}(\tilde{u}^*da)\wedge \tilde{\alpha} +
      C_{\mathbf{h}}E_{b_{k, n}^\pm}^{a_{k+1, n}^\pm}.
    \end{align*}
Making use of estimate (\ref{EQ_f_estimate_1}), we have
  \begin{equation}                                                        
    h^\pm(b_{k, n}^\pm)\leq h^\pm(a_{k, n}^\pm)e^{ C_{\mathbf{h}}( b_{k,
    n}^\pm-a_{k, n}^\pm)} +C_{\mathbf{h}}E_{a_{k, n}^\pm}^{b_{k, n}^\pm}
    e^{ C_{\mathbf{h}}( b_{k, n}^\pm-a_{k, n}^\pm)}
    \end{equation}
  \begin{equation*}                                                       
    h^\pm(a_{k+1, n}^\pm)\leq h^\pm(b_{k, n}^\pm)
    +C_{\mathbf{h}}\int_{\widetilde{S}_{b_{k, n}^\pm}^{a_{k+1,
    n}^\pm}}(\tilde{u}^*da)\wedge \tilde{\alpha} + C_{\mathbf{h}}E_{b_{k,
    n}^\pm}^{a_{k+1, n}^\pm}.
    \end{equation*}
  and thus
  \begin{align*}                                                          
    h^\pm(a_{k+1, n}^\pm)&\leq
      \Big(  h^\pm(a_{k, n}^+)e^{ C_{\mathbf{h}}( b_{k, n}^\pm-a_{k,
      n}^\pm)} +C_{\mathbf{h}}E_{a_{k, n}^\pm}^{b_{k, n}^\pm} e^{
      C_{\mathbf{h}}( b_{k, n}^\pm -a_{k, n}^\pm)}\Big)\\
    &\qquad  
      +C_{\mathbf{h}}\int_{\widetilde{S}_{b_{k, n}^\pm}^{a_{k+1, n}^\pm
      }}(\tilde{u}^*da)\wedge \tilde{\alpha} + C_{\mathbf{h}}E_{b_{k,
      n}^\pm }^{a_{k+1, n}^\pm }\\
    &\leq  
      h^\pm (a_{k, n}^\pm)e^{C_{\mathbf{h}}( a_{k+1, n}^\pm -a_{k,
      n}^\pm)} +C_{\mathbf{h}} E_{a_{k, n}^\pm }^{a_{k+1, n}^\pm } e^{
      C_{\mathbf{h}}( a_{k+1, n}^\pm -a_{k, n}^\pm )}\\
    &\qquad  
      +C_{\mathbf{h}}\int_{\widetilde{S}_{b_{k, n}^\pm}^{a_{k+1,
      n}^\pm}}(\tilde{u}^*da)\wedge \tilde{\alpha}.
    \end{align*}
Iterating this estimate backwards in \(k\) from \(n-1\) to \(0\)
  then yields
  \begin{align*}                                                          
    h^\pm (a_{n, n}^\pm)&\leq h^\pm(a_{0, n}^\pm)e^{ C_{\mathbf{h}}( a_{n,
      n}^+-a_{0, n}^\pm)}+C_{\mathbf{h}}E_{a_{0,
      n}^\pm}^{a_{n, n}^\pm} e^{ C_{\mathbf{h}}( a_{n, n}^\pm -a_{0,
      n}^\pm)}\\
    &\qquad
      +C_{\mathbf{h}} e^{C_{\mathbf{h}} (a_{n, n}^+ - a_{0, n}^\pm)}
      \sum_{\ell = 1}^{n} \int_{\widetilde{S}_{b_{\ell-1,
      n}^\pm}^{a_{\ell, n}^\pm}} (\tilde{u}^*da )\wedge \tilde{\alpha}.
    \end{align*}
We recall that by construction \(a_{0, n}^\pm = 0\) and \(b_{n, n}^\pm =
  r\), so that with final application of estimate (\ref{EQ_f_estimate_1})
  we find
  \begin{align*}                                                          
    h^\pm(r)&\leq 
      \Big(h^\pm(0)+C_{\mathbf{h}} \int_{\widetilde{S}} u^*\omega \Big)
      e^{ C_{\mathbf{h}} r} +C_{\mathbf{h}} e^{C_{\mathbf{h}} r }
      \sum_{\ell = 1}^{n} \int_{\widetilde{S}_{b_{\ell-1, n}^\pm
      }^{a_{\ell, n}^\pm }}( \tilde{u}^*da) \wedge
      \tilde{\alpha}\\
    &\leq 
      \Big(h^\pm(0)+C_{\mathbf{h}}\int_{\widetilde{S}} u^*\omega \Big) e^{
      C_{\mathbf{h}} r } + \delta.
    \end{align*}
Since \(\delta>0\) was arbitrary, we conclude that
  \begin{equation*}                                                       
    h^\pm(r)=\int_{(a\circ \tilde{u})^{-1}(\pm r)}\tilde{\alpha} \leq
    \Big( \int_{(a\circ u)^{-1}(0)} \tilde{\alpha} +C_{\mathbf{h}}\int_S
    u^*\omega \Big) e^{C_{\mathbf{h}} r },
    \end{equation*}
which establishes inequalities (\ref{EQ_lambda_growth_inequality_1})
  and (\ref{EQ_lambda_growth_inequality_2}).

To achieve inequality (\ref{EQ_area_growth_inequality}) in both the
  positive and negative case, we essentially employ Lemma
  \ref{LEM_coarea_lambda} and observe that the functions \(s\mapsto
  \int_{(a\circ u)^{-1}(\pm s)} \tilde{\alpha}\) is defined almost
  everywhere (specifically on \(\mathcal{R}^\pm\) which has full measure,
  and on \(\mathcal{R}^\pm\) it agrees with \(h^\pm\)); the result is then
  obtained  by integrating the estimates just obtained for \(h^\pm(r)\).
Indeed, in the positive case we have 
  \begin{align*}                                                          
    \int_{\widetilde{S}} (\tilde{u}^*da) \wedge \tilde{\alpha}+
      \tilde{u}^*\omega
    &= 
      \int_{\widetilde{S}} (\tilde{u}^*da)\wedge \tilde{\alpha} +
      \int_{\widetilde{S}} u^*\omega\\
    &= 
      \int_{0}^{r} \Big( \int_{(a\circ \tilde{u})^{-1}(s)\setminus
      \mathcal{X}}\tilde{\alpha} \Big)\, ds + \int_{\widetilde{S}}u^*\omega
      \\
    &= 
      \int_{0}^{r} h^+(s)\, ds + \int_{\widetilde{S}}u^*\omega\\
    &\leq 
      \int_{0}^{r} \Big( h^+(0) +C_{\mathbf{h}} E_0 \Big) e^{
      C_{\mathbf{h}} s } \, ds + \int_{\widetilde{S}}u^*\omega\\
    & \leq \Big( C_{\mathbf{h}}^{-1} h^+(0)+E_0 \Big) \big(e^{
      C_{\mathbf{h}} r }-1) + E_0,
    \end{align*} 
  which establishes inequality (\ref{EQ_area_growth_inequality}) in
  the positive case; the negative case is established similarly.
This completes the proof of Theorem \ref{THM_area_bound_estimate}.
\end{proof}

To complete Section \ref{SEC_proof_of_exp_area_bounds}, it remains to
  prove Theorem \ref{THM_area_bounds}.
This will follow as an immediate corollary to the following result.

\begin{proposition}[area bounds more carefully]
  \label{PROP_area_bounds_carefully}
  \hfill \\
Fix positive constants \(C_H>0\), \(r>0\), and \(E_0>0\). 
Then there exists a constant \(C_A=C_A(C_H, r, E_0)\)  with the following
  significance.
For each closed odd-dimensional manifold \(M\) equipped with the quadruple
  \(\mathbf{h}=(J, g, \lambda, \omega)\) where \(\eta:=(\lambda, \omega)\)
  is a Hamiltonian structure on \(M\) and \((J, g)\) is an \(\eta\)-adapted
  almost Hermitian structure on \(\mathbb{R}\times M\) with the property
  that \(C_{\mathbf{h}}\leq C_H\), where \(C_{\mathbf{h}}\) is the ambient
  geometry constant established in Definition
  \ref{DEF_ambient_geometry_constant}, and for each proper pseudoholomorphic
  map \(u\colon S\to \mathcal{I}\times M\), where \(\mathcal{I}=(a_0 -r,
  a_0 +r)\), for which \(\partial S = \emptyset\), \(u^{-1}(\{a_0\}\times
  M)=\emptyset\), and
  \begin{equation*}                                                       
    \int_S u^*\omega \leq E_0 <\infty,
    \end{equation*}
  the following also holds:
  \begin{equation*}                                                       
    \int_{S} u^* (da \wedge \lambda + \omega) \leq C_A.
    \end{equation*} 
\end{proposition}
%
\begin{proof}
Fix \(M\), \(\mathbf{h}\), and \(u\) as in the hypotheses.
By translation invariance of \(J\), \(g\), \(\lambda\), and \(\omega\), we
  may assume with out loss of generality that \(a_0=0\).
Observe that since \((a\circ u)^{-1}(0)=\emptyset\) it follows that \(0\)
  is a regular value of \(a\circ u\).
Next, fix a sequence of \(\epsilon_k >0\) so that \(\epsilon_k\to 0\) as
  \(k\to \infty\), and with the additional property that \(\pm
  \epsilon_k\) are regular values of \(a\circ u\) for all \(k\in
  \mathbb{N}\); note that such choice is possible by Sard's theorem.
Then we have
  \begin{align*}                                                          
    \int_{S} u^*(da\wedge \lambda + \omega)&= \lim_{k\to \infty}
    \int_{S_{-r+\epsilon_k}^{r-\epsilon_k}} u^*(da\wedge \lambda + \omega)
    \end{align*}
  where \(S_{-r+\epsilon_k}^{r-\epsilon_k}\) is defined as in
  (\ref{EQ_Sab}).
We now intend to employ Theorem \ref{THM_area_bound_estimate} in the
  case that the perturbed pseudoholomorphic map is given by \((\tilde{u},
  \widetilde{S}, \tilde{\jmath}, f, u, S, j)=(u, S_0^{r-\epsilon_k}, j, 0,
  u, S, j)\).
In this case \(\tilde{\alpha}=u^*\lambda\), so that by inequality
  (\ref{EQ_area_growth_inequality}) we have
  \begin{align*}                                                          
    \int_{S_{0}^{r-\epsilon_k}} u^*(da\wedge \lambda + \omega)&\leq 
      E_0 e^{C_{\mathbf{h}}(r-\epsilon_k)} \\
    &\leq 
      E_0 e^{C_{\mathbf{h}}r} \\
    &=:
      {\textstyle \frac{1}{2}}C_A.
    \end{align*}
The same inequality holds for \(S_{-r+\epsilon_k}^0\) instead of
  \(S_0^{r-\epsilon_k}\), and summing the two inequalities then yields the
  desired result.
\end{proof}

As remarked above, Theorem \ref{THM_area_bounds} follows from
  Proposition \ref{PROP_area_bounds_carefully} as an immediate corollary.

\subsubsection{An area estimate in realized Hamiltonian homotopies}
Here we prove a slight modification of Proposition
  \ref{PROP_area_bounds_carefully}, for the case of realized Hamiltonian
  homotopies in the sense of Definition \ref{DEF_hamiltonian_homotopy}.

\setcounter{CurrentSection}{\value{section}}
\setcounter{CurrentTheorem}{\value{theorem}}
\setcounter{section}{\value{CounterSectionAreaHomotopy}}
\setcounter{theorem}{\value{CounterTheoremAreaHomotopy}}
\begin{theorem}[area bounds in realized Hamiltonian homotopy]
  \label{THM_area_bounds_homotopy}
  \hfill \\
Fix positive constants \(C_H>0\), \(r>0\), and \(E_0>0\). 
Then there exists a constant \(C_A=C_A(C_H, r, E_0)\) with the following
  significance.
Let \(\mathcal{I}\times M, (\hat{\lambda}, \hat{\omega}))\) denote a
  realized Hamiltonian homotopy in the sense of Definition
  \ref{DEF_hamiltonian_homotopy}, and let \((J, g)\) be an adapted almost
  Hermitian structure in the sense of Definition
  \ref{DEF_adapted_structures_Ham_homotopy} with
  \begin{align*}                                                          
    C_{\mathbf{H}}:= \sup_{q\in \mathcal{I} \times M}
    \|d\hat{\lambda}_q\|_g \leq C_H.
    \end{align*}
For each proper pseudoholomorphic
  map \(u\colon S\to \mathcal{I}_r\times M\), where 
\begin{align*}                                                            
  \mathcal{I}_r=(a_0 -r, a_0 +r)\subset \mathcal{I}
  \end{align*}
  for which \(\partial S = \emptyset\),
  \(u^{-1}(\{a_0\}\times M)=\emptyset\), and
  \begin{equation*}                                                       
    \int_S u^*\omega \leq E_0 <\infty,
    \end{equation*}
  the following also holds:
  \begin{equation*}                                                       
    {\rm Area}_{u^*g}(S) = \int_{S} u^* (da \wedge \hat{\lambda} +
    \hat{\omega}) \leq C_A.
    \end{equation*} 
Additionally, for any \([ a_0, a_1] \subset \mathcal{I}\) and any compact
  pseudoholomorphic map \(u: S\to [a_0,a_1]\times M\) for which \(a_0\) and
  \(a_1\) are regular values of \(a\circ u\) and
  \(u^{-1}\big(\{a_0,a_1\}\times M\big) = \partial S\), the following also
  hold:
  \begin{align*}                                                          
    \int_{\Gamma_{a_0}} u^*\lambda\leq
    \Big(C_{H} E_0+ \int_{\Gamma_{a_1}}
    u^*\lambda\Big) e^{C_H (a_1-a_0)},
    \end{align*}
  and
  \begin{align*}                                                          
    \int_{\Gamma_{a_1}} u^*\lambda\leq
    \Big(C_{H} E_0+ \int_{\Gamma_{a_0}}
    u^*\lambda\Big) e^{C_H (a_1-a_0)},
    \end{align*}
  where \( \Gamma_{a_i} = (a\circ u)^{-1}(a_i) \) for \(i \in \{0, 1\}\).
Similarly, for 
  \begin{align*}                                                          
    \ell: = \min_{i\in \{0, 1\}} \Big\{ \int_{\Gamma_{a_i}} u^*\lambda
    \Big\}
    \end{align*}
  we have
  \begin{align*}                                                          
    {\rm Area}_{u^*g} (S) \leq (C_H^{-1} \ell +E_0 ) (e^{C_H (a_1-a_0)}
    -1)+E_0.
    \end{align*}
\end{theorem}
%

\begin{proof}
\setcounter{section}{\value{CurrentSection}}
\setcounter{theorem}{\value{CurrentTheorem}}

The proof is essentially identical to the proof of Theorem
  \ref{THM_area_bound_estimate} and Proposition
  \ref{PROP_area_bounds_carefully}, except with the following essentially
  typographical changes.
In Theorem \ref{THM_area_bound_estimate}, the perturbed map
  \(\tilde{u}\) simply becomes the unperturbed map \(u\); or equivalently
  the perturbing function \(f\equiv 0\).
Similarly, the perturbed almost complex structure \(\tilde{\jmath}\) on
  the domain \(S\) is simply the unperturbed \(j\).
All references to \(\lambda\) and \(\omega\) should respectively be
  replaced with references to \(\hat{\lambda}\) and \(\hat{\omega}\).
The one-form \(\tilde{\alpha}\) is nothing more than
\begin{align*}                                                            
  \tilde{\alpha} = u^*\hat{\lambda}.
  \end{align*}
Instances of \(\frac{1}{2}C_{\mathbf{h}}\) should be replaced with
  \(C_{\mathbf{H}}\).
The proof then goes through unchanged.

\end{proof}

\subsection{Proof of Theorem \ref{THM_energy_threshold}: $\omega$-Energy
  Threshold}

The primary purpose of this section is to prove Theorem
  \ref{THM_energy_threshold}, which we restate below.

\setcounter{CurrentSection}{\value{section}}
\setcounter{CurrentTheorem}{\value{theorem}}
\setcounter{section}{\value{CounterSectionEnergyThreshold}}
\setcounter{theorem}{\value{CounterTheoremEnergyThreshold}}
\begin{theorem}[$\omega$-energy threshold]\hfill\\
Let \((M, \eta=(\lambda, \omega))\) be a compact framed Hamiltonian
  manifold, and let \((J, g)\) be an \(\eta\)-adapted almost Hermitian
  structure on \(\mathbb{R}\times M\).
Also, fix positive constants \(r>0\), and \(C_g>0\). 
Then there exists a positive constant \(0<\hbar=\hbar(M, \eta, J, g,
  r, C_g)\) with the following significance.
Let \(\{\mathbf{h}_k\}_{k\in\mathbb{N}}\) be a sequence of quadruples
  \((J_k, g_k, \lambda_k, \omega_k)\) with the property that each
  \(\eta_k=(\lambda_k, \omega_k)\) is a Hamiltonian structure on \(M\),
  and each \((J_k, g_k)\) is an \(\eta_k\)-adapted almost Hermitian
  structure on \(\mathbb{R}\times M\), and suppose that
  \begin{align*}                                                          
    (J_k, g_k, \lambda_k, \omega_k) \to (J, g, \lambda, \omega)
    \qquad\text{in }\mathcal{C}^\infty\text{ as }k\to \infty.
    \end{align*}
Furthermore, fix \(a_0\in \mathbb{R}\), and let \(u_k\colon S_k\to
  \mathbb{R}\times M\) be a sequence of compact connected generally
  immersed pseudoholomorphic maps which satisfy the following conditions:
  \begin{enumerate}[($\hbar$1)]                                           
      \item either \(a\circ u_k(S_k)\subset [a_0,
        \infty)\) or  \(a\circ u_k(S_k)\subset (-\infty, a_0]\) for all
        \(k\in \mathbb{N}\)
      \item \({\rm Genus}(S_k)\leq C_g\)
      \item \(a\circ u_k(\partial S_k)\cap [a_0-r, a_0+r]
	= \emptyset\)
      \item \(a_0\in a\circ u_k (S_k)\).
    \end{enumerate}
Then for all sufficiently large \(k\in \mathbb{N}\) we have
  \begin{equation*}                                                       
      \int_{S_k} u_k^*\omega_k \geq \hbar.
    \end{equation*}
\end{theorem}
%
\begin{proof}
\setcounter{section}{\value{CurrentSection}}
\setcounter{theorem}{\value{CurrentTheorem}}
For clarity of proof, we will first argue the case in which
\begin{align*}                                                            
  \mathbf{h}_k=(J_k, g_k, \lambda_k, \omega_k) = (J, g, \lambda, \omega)
  \end{align*}
  for all \(k\in \mathbb{N}\).
We argue by contradiction, so suppose not.
Then there exists a sequence \((u_k, S_k, j_k)\) of
   pseudoholomorphic maps satisfying properties (\(\hbar\)\ref{EN_hbar1})
   - (\(\hbar\)\ref{EN_hbar4}) with the property that
  \begin{align*}                                                          
    \int_{S_k}u_k^*\omega \to 0.
    \end{align*}
For simplicity we will assume that \(a\circ u_k(S_k)\subset [a_0,
   \infty)\); the other case is argued identically.
Making use of translation invariance of \(J\), we will also assume
   without loss of generality that \(a_0=0\).
Next, for each \(k\in \mathbb{N}\), fix a regular value \(a_k'\) of
   \(a\circ u_k\) with the property that  \(a_k'\in (\frac{3}{4}r, r) \).
Then define \(\widehat{S}_k\) to be \((a\circ u_k)^{-1}([0, a_k'])\),
  which we observe has non-trivial intersection with  \((a\circ
  u_k)^{-1}(0)\).
By Theorem \ref{THM_area_bounds}, it follows that there exists a constant
  \(C_A>0\) such that
  \begin{align*}                                                          
    {\rm Area}_{u_k^*g}(\widehat{S}_k) = \int_{\widehat{S}_k}u_k^*(da
    \wedge \lambda +\omega ) \leq C_A.
    \end{align*}
Consequently, the sequence of pseudoholomorphic maps \((u_k, \widehat{S}_k,
  j_k)\) has uniformly bounded area, uniformly bounded genus,  and each has
  compact domain \(\widehat{S}_k\) and is generally immersed, and
  furthermore they each satisfy \((a\circ u_k)(\widehat{S}_k)=[0, a_k']\) and
 \(u_k(\partial \widehat{S}_k) \cap [0, \frac{3}{4}r]\times M=\emptyset\).
We conclude from Theorem \ref{THM_target_local_gromov_compactness},
  namely target-local Gromov compactness, that there exists \(a''\in
  (\frac{1}{2}r, \frac{3}{4}r)\) with the property that after passing to a
  subsequence, the pseudoholomorphic curves \((u_k, \widetilde{S}_k,
  j_k)\) defined by \(\widetilde{S}_k:= (a\circ u_k)^{-1}([0 ,a''])\)
  converge in a Gromov sense to a nodal pseudoholomorphic curve which has
  the property that its image both intersects \(\{0\}\times M\)
  non-trivially and is also contained in \([0, a'']\times M\), and that
  the image of the boundary of the limit Riemann surface is contained in
  \(\{a''\}\times M\).

As a consequence of the definition of Gromov convergence, we may
   choose connected components of \(\widetilde{S}_k\) (still denoted as
   \(\widetilde{S}_k\)) with the property that \(u_k(\widetilde{S}_k)\cap
   \{0\}\times M\neq \emptyset\) for all \(k\in \mathbb{N}\) and such
   that the sequence of generally immersed pseudoholomorphic curves \((u_k,
   \widetilde{S}_k, j_k)\) converges in a Gromov sense to a nodal
   pseudoholomorphic curve \((u, S, j, \mathcal{D})\) with the property
   that \(u(S)\) is connected and again \(u(\partial S)\subset
   \{a''\}\times M\).
For any \(\zeta\in S\), define \(S^\zeta\) to be the connected component
   of \(S\) which contains \(\zeta\).
We claim that there must exist \(\zeta\in S\) such that \(a\circ u(\zeta)
   = 0\) and \(u:S^\zeta\to \mathbb{R}\times M\) is not a constant map.
To see this, we suppose not and derive a contradiction.
Indeed, note that there must exist \(\zeta\in S\) such that \(a\circ
   u(\zeta) = 0\), and \(u(S)\) is connected, so that if it is the case
   that \(u:S^\zeta\to \mathbb{R}\times M\) is constant for every such
   \(\zeta\) then it must be the case that \(u:S\to \mathbb{R}\times M\)
   is a constant map.
However, if \(u:S\to \mathbb{R}\times M\) is constant, then the sequence
   \((u_k, \widetilde{S}_k, j_k)\) is converging to a constant map,
   which is only possible if for all sufficiently large \(k\) we have
   \(u_k:\widetilde{S}_k\to \mathbb{R}\times M\) is a constant map,
   which contradicts the fact that the \((u_k, \widetilde{S}_k, j_k)\)
   are each generally immersed.
This contradiction establishes that there does indeed exist \(\zeta\in
   S\) such that \(a\circ u(\zeta)=0\) and \(u:S^\zeta\to \mathbb{R}\times
   M\) is not a constant map.

At this point we observe that \(u:S^\zeta\to \mathbb{R}\times M\)
   is not a constant, and since it is connected it must instead be
   generally immersed.
Furthermore, we have \(a\circ u (\zeta) = 0\), \(u(S^\zeta)\subset [0,
   a'']\times M \),  \(u(\partial S^\zeta)\subset \{a''\}\times M\), and
   \(u(S^\zeta)\) is connected.
Since \(\int_{\widetilde{S}_k}u_k^*\omega\to 0\), we must also have
   \(\int_{S^\zeta}u^*\omega = 0\), and since \(\omega\) evaluates
   non-negatively on \(J\)-complex lines, we conclude that there exists
   a trajectory \(\beta:\mathbb{R}\to M\) of the Hamiltonian vector field
   \(X_\eta\) such that \(u(S)\subset [0, a'']\times \beta(\mathbb{R})\).
Note that \((s, t)\mapsto (s, \beta(t))\) is a holomorphic map, and from
   this it follows that \(a\circ u:S^\zeta\to \mathbb{R}\) is a harmonic
   function, and hence can have no interior minima unless \(a\circ u:
   S^\zeta\to \mathbb{R}\) is a constant map.
An elementary fact from complex variables shows that if \(a\circ u:
   S^\zeta\to \mathbb{R}\) is a constant map, then \(u:S^\zeta\to [0,
   a'']\times \beta(\mathbb{R})\subset \mathbb{R}\times M\) is a constant
   map, which is impossible since \((u, S, j)\) is generally immersed.
Consequently, \(a\circ u:S^\zeta\to \mathbb{R}\) can have no interior
   minima, however from the above properties of \((u, S, j)\) we see that
   \(a\circ u\) achieves its absolute minimum at \(\zeta\in S\setminus
   \partial S\).
This is the desired contradiction which completes our proof in the case
  that \((J_k, g_k, \lambda_k, \omega_k) = (J, g, \lambda, \omega)\) for all
  \(k\in \mathbb{N}\).

To complete the proof Theorem \ref{THM_energy_threshold} it remains to
  consider the more general case in which \((J_k, g_k, \lambda_k,
  \omega_k) \to (J, g, \lambda, \omega)\).
We note that in this case the proof is identical with a single
  modification: Rather than citing Theorem \ref{THM_area_bounds} to
  obtain local area bounds, we instead cite the more general Proposition
  \ref{PROP_area_bounds_carefully}; see below.
Note that the hypotheses of Proposition \ref{PROP_area_bounds_carefully}
  are satisfied precisely because \((J_k, g_k, \lambda_k, \omega_k) \to (J,
  g, \lambda, \omega)\) in \(\mathcal{C}^\infty\).
\end{proof}

\subsection{Proof of Theorem \ref{THM_local_local_area_bound}:
  Asymptotic Connected-Local Area Bound}
  \label{SEC_proof_of_local_local}\hfill\\

The main purpose of this section is to prove Theorem
  \ref{THM_local_local_area_bound}, which roughly states that for each
  given feral curve, there exists a large compact set in the symplectization
  with the property that outside that compact set, the area of each
  connected component of the portion of the image of that curve contained
  in ball of some small radius is universally bounded.
Indeed, the bound is \(1\).
We note that the large compact set will depend upon the curve, but the
  radius of the ball does not.
Indeed, the radius depends only on the ambient geometry of the
  symplectization; specifically the framed Hamiltonian structure and the
  almost Hermitian structure.

Unfortunately, the proof of this result is rather long and complicated, so
  we give the basic idea here, and then later we will elaborate further.
Consider a feral curve, and for some very large \(a_0>0\) we consider the
  portion of curve defined by \(u\colon \widetilde{S} \to \mathbb{R}\times
  M\) where
  \begin{align*}                                                          
    \widetilde{S} = u^{-1}\big([a_0 -\epsilon, a_0+\epsilon]\big)
    \end{align*}
  for some very small \(\epsilon>0\).
Thinking of \(\epsilon\) as incredibly small and fixed, but \(a_0\) large,
  generic, and replaced with a larger value if needed, we then are inclined
  to think of \(u\colon \widetilde{S}\to \mathbb{R}\times M\) as a ribbon
  of pseudoholomorphic curve which is very thin, but also very, very, long,
  so that the area is quite large, and the area gets larger without bound
  as \(a_0\to \infty\).
For each \(\zeta_0\in S\) we then define
  \(\widetilde{S}_{\frac{1}{2}\epsilon}(\zeta_0)\) to be the connected
  component of \(u^{-1}(\mathcal{B}_{\frac{1}{2}\epsilon}(u(\zeta_0))\)
  which contains \(\zeta_0\), where \(\mathcal{B}_{r}(p)\) denotes the
  metric ball in \(\mathbb{R}\times M\) centered at \(p\) and of radius
  \(r\).
The essential question to ask next is then: For each \(\zeta\in (a\circ
  u)^{-1}(a_0)\), can the area of
  \(\widetilde{S}_{\frac{1}{2}\epsilon}(\zeta_0)\) be arbitrarily large by
  making \(a_0\) larger if needed?
We will show that the answer is no, and that the idea is to cut the ribbon
  into tracts of pseudoholomorphic curves of modest length and small height
  so that the area of each such tract is modest and so that
  \(\widetilde{S}_{\frac{1}{2}\epsilon}(\zeta)\) is completely contained
  in one of the tracts.
Put another way, we aim to partition the lower ribbon boundary \((a\circ
  u)^{-1}(a_0-\epsilon)\) so that the length of each of the corresponding
  intervals is small, and so that the end point of each interval can be
  flowed up to the top ribbon boundary \((a\circ u)^{-1}(a_0+\epsilon)\)
  via the gradient flow.
We then cut our ribbon along these gradient flow lines.
Assuming the \(\omega\)-energy of the ribbon is small, which can always be
  guaranteed by making \(a_0\) sufficiently large, it is elementary to show
  the tracts of curve extending from partition intervals have modest area,
  so the argument establishing Theorem
  \ref{THM_local_local_area_bound} is complete if we can show that
  \(\widetilde{S}_{\frac{1}{2}\epsilon}(\zeta)\) is contained in one such
  tract.

Of course there are a variety of obstacles to be overcome, and the bulk of
  these are related to guaranteeing the existence of the desired intervals,
  specifically establishing that most points in the lower ribbon
  boundary can be gradient flowed to the top ribbon boundary.
Indeed, although we have described it as a ribbon, it may be the case that
  our portion of curve is in fact is the disjoint union of disks, annuli,
  pairs of pants, or other more topologically complicated surfaces.
Furthermore one must deal with the possibility that \(u\colon
  \widetilde{S}\to \mathbb{R}\times M\) may not be immersed, \(a\circ
  u\colon \widetilde{S}\to \mathbb{R}\times M\)  may not be Morse, and that
  many flow lines initiating in the lower ribbon boundary have terminal
  points which are local interior maxima.
All of these issues are addressed, however significant preliminaries are
  necessary.
Indeed, in Section \ref{SEC_strip_estimates} we provide Definition
  \ref{DEF_perturbed_J_strip} which is the basic tool used throughout
  Section \ref{SEC_proof_of_local_local}, and we establish a number of
  important properties.
In Section \ref{SEC_miscellany} we establish a few miscellaneous
  results which will be referenced in the main proof.
Finally, in Section \ref{SEC_core_proof_local_local} we will provide the
  complete proof of Theorem \ref{THM_local_local_area_bound}, however
  the bulk of the technical work is established in Proposition
  \ref{PROP_local_local_area_bound_2}, which is essentially a special case
  of Theorem \ref{THM_local_local_area_bound} and is also proved in
  this section.

\subsubsection{Strip Estimates}\label{SEC_strip_estimates}

The purpose of this section is to provide the notion of a perturbed
  pseudoholomorphic strip, see Definition \ref{DEF_perturbed_J_strip}
  below, and to establish a few technical properties, which will be used in
  later sections.

Here and throughout, we let \((M, \eta)\) be a framed Hamiltonian manifold
  with \(\eta=(\lambda, \omega)\), and let \((J, g)\) be an
  \(\eta\)-adapted almost Hermitian structure on the symplectization
  \(\mathbb{R}\times M\).
As in the previous section, specifically in equation
  (\ref{EQ_def_alpha_tilde}), associated to any perturbed
  pseudoholomorphic map \((\tilde{u}, \tilde{\jmath}, f, u, S, j)\), we
  define the smooth one-form
  \begin{equation*}
    \tilde{\alpha}:= -(\tilde{u}^*da)\circ \tilde{\jmath}.
    \end{equation*}

\begin{definition}[perturbed pseudoholomorphic strip]
  \label{DEF_perturbed_J_strip}
  \hfill \\
Let \((M, \eta)\) be a framed Hamiltonian manifold, \((J, g)\) an
  \(\eta\)-adapted almost Hermitian structure on \(\mathbb{R}\times M\).
A perturbed pseudoholomorphic strip consists of the tuple \((\tilde{u},
  \widetilde{S}, \tilde{\jmath}, f, u, S, j)\), where
  \begin{itemize}                                                       
    \item \((\tilde{u}, \tilde{\jmath}, f, u, S, j)\) is a perturbed
      pseudoholomorphic map in the sense of Definition
      \ref{DEF_perturbed_J_map},
    \item \(\widetilde{S}\subset S\) is a compact manifold with boundary
      and corners, it is homeomorphic to a disk, and it is determined by
      the data \((p, \mathcal{I}, h^-, h^+)\) where 
      \begin{enumerate}[(e1)]                                             
	\item\label{EN_e1} \(\mathcal{I}\subset \mathbb{R}\) is an
	  closed interval of finite length
	\item\label{EN_e2} \(p:\mathcal{I}\to S\) is a smooth map for
	  which \(\tilde{u}\circ p:\mathcal{I}\to \mathbb{R}\times M\)
	  is an embedding, \(a\circ \tilde{u}\circ p: \mathcal{I}\to
	  \mathbb{R}\) is the constant map \(a\circ \tilde{u}\circ p=a_0\)
	  , and \(p^* \tilde{\alpha}=dt\) where \(t\) is the coordinate
	  on \(\mathcal{I}\) induced from \(\mathcal{I}\subset
	  \mathbb{R}\)
	\item\label{EN_e3} \( h^\pm :\mathcal{I}\to \mathbb{R}\) are
	  \(\mathcal{C}^1\) functions for which \(h^-< h^+\), and for each
	  \(t\in \mathcal{I}\) there exists a map
	  \begin{equation*}                                               
	    q^t:\Big[\min\big(0, h^-(t)\big), \max\big(0,
	    h^+(t)\big)\Big]\to S
	    \end{equation*}
	  satisfying 
	  \begin{equation*}                                               
	    \frac{d}{d s}q^t(s) = \frac{\widetilde{\nabla}
	    (a\circ \tilde{u})\big(q^t(s)\big)}{\|\widetilde{\nabla}
	    (a\circ \tilde{u})\big(q^t(s)\big)\|_{\tilde{\gamma}}^2}
	    \quad\text{and}\quad q^t(0)= p(t),
	    \end{equation*} 
	  in which case \(\widetilde{S}\) is given by
	  \begin{equation*}                                               
	    \widetilde{S}= \bigcup_{t\in\mathcal{I}} q^t\big( [h^-(t),
	    h^+(t)]\big);
	    \end{equation*}
	  here \(\tilde{\gamma}=\tilde{u}^*g\) and \(\widetilde{\nabla}\)
	  denotes the gradient with respect to the metric
	  \(\tilde{u}^*g\).
	  In the case \(h^-\) and \(h^+\) are constant functions,
	  we say \((\tilde{u}, \widetilde{S}, \tilde{\jmath})\) is a
	  \emph{rectangular} \(f\)-perturbed pseudoholomorphic strip.
	\end{enumerate}
    \end{itemize}
Given \((\tilde{u}, \tilde{\jmath}, f, u, S, j)\), we will refer to
  \((\tilde{u}, \widetilde{S}, \tilde{\jmath}, f, u, S, j)\), as the
  perturbed pseudoholomorphic strip determined by the data \((p,
  \mathcal{I}, h^-, h^+)\).
\end{definition}
%

\begin{remark}[strips vs tracts]
  \label{REM_strips_vs_tracts}
  \hfill\\
It is worth pointing out the differences between  perturbed
  pseudoholomorphic strips and tracts of perturbed pseudoholomorphic maps.
Ignoring the perturbation momentarily, we note that rectangular strips are
  in fact special cases of tracts of pseudoholomorphic maps, with the
  additional property that \(S\) is rectangular (that is, homeomorphic to a
  compact disk, and the boundary is piecewise smooth with four non-smooth
  points), and \(a\circ u:S\to \mathbb{R}\) has no critical points.
The more general notion of a pseudoholomorphic strip then allows for the
  possibility that \(a\circ u\) restricted to \(\partial_0 S\) need not be a
  constant map, and hence is no longer a tract of pseudoholomorphic map.
Regarding perturbations, the main difference is that for tracts of
  perturbed pseudoholomorphic maps, the support of the perturbation must be
  contained in the interior whereas for strips the support may overlap with
  the boundary.
\end{remark}
%

The following result establishes two important facts.  
First, given a perturbed pseudoholomorphic map, one can construct a
  perturbed pseudoholomorphic strip from less stringent data \((p,
  \mathcal{I}, h^-, h^+)\) than that given in Definition
  \ref{DEF_perturbed_J_strip}.
Second, associated to each perturbed pseudoholomorphic strip are
  coordinates \((s, t)\) which have a variety of properties.
We make these facts precise with the following.

\begin{lemma}[strip reparametrization]
  \label{LEM_J_strip_reparam}
  \hfill\\
Let \((M, \eta)\) be a framed Hamiltonian manifold, let \((J, g)\) be
  an \(\eta\)-adapted almost Hermitian structure on \(\mathbb{R}\times
  M\), and let \((\tilde{u}, \tilde{\jmath}, f, u, S, j)\) be a perturbed
  pseudoholomorphic map.
Suppose the \(4\)-tuple \((\hat{p}, \widehat{\mathcal{I}}, \hat{h}^-,
  \hat{h}^+)\) consists of the following data.
  \begin{enumerate}                                                       
    \item \(\widehat{\mathcal{I}}\) is an closed interval of finite
      length
    \item \(\hat{p}:\widehat{\mathcal{I}}\to S\) is a smooth map with
      image which is contained in a level set of \(a\circ \tilde{u}\)
      and disjoint from the critical points of \(a\circ \tilde{u}\),
    \item \(\hat{h}^-\) and \(\hat{h}^+\) satisfy property (e\ref{EN_f3})
      of Definition \ref{DEF_perturbed_J_strip}.
    \end{enumerate}
\emph{Then} there exists an interval \(\mathcal{I}\subset\mathbb{R}\),
  and diffeomorphism \(\psi:\mathcal{I}\to \widehat{\mathcal{I}}\) with
  the property that \((p, \mathcal{I}, h^-, h^+):=(\hat{p}\circ \psi,
  \mathcal{I}, \hat{h}^-\circ\psi, \hat{h}^+\circ\psi)\) satisfy properties
  (e\ref{EN_e1}) - (e\ref{EN_f3}) of Definition
  \ref{DEF_perturbed_J_strip}, and hence define a perturbed
  pseudoholomorphic strip \((\tilde{u}, \widetilde{S}, \tilde{\jmath}, f,
  u, S, j)\).
Moreover the diffeomorphism 
  \begin{align*}                                                          
    &\phi: \{(s, t)\in \mathbb{R}^2: t\in \mathcal{I}\text{ and }h^-(t)<
      s < h^+(t)\}\to \widetilde{S}\\
    &\phi(s, t) = q^t(s)
    \end{align*}
  satisfies the following properties.
\begin{enumerate}[(h1)]                                                   
  \item\label{EN_h1} \(a\circ \tilde{u}\circ \phi (s, t) = a_0+ s\)
  \item\label{EN_h2} \((\tilde{u}\circ \phi)^*g = \hat{\gamma} =
    \hat{\gamma}_{11}\, ds\otimes ds + \hat{\gamma}_{22}\,dt\otimes dt\)
    for smooth  functions \(\hat{\gamma}_{kk}=\hat{\gamma}_{kk}(s, t)\)
    and \(\hat{\gamma}_{11}\geq 1\)
  \item\label{EN_h3} \(\phi^*\tilde{\alpha}= \ell(s, t)\, dt\),
    with \(0<\ell\leq \hat{\gamma}_{22}^{\frac{1}{2}}\)  and \(\ell(0,
    t)\equiv 1\)
  \item\label{EN_h4} \((\phi^*j)\partial_s =\tau(s, t)\partial_t\)
    with \(\tau>0\).
  \end{enumerate}
Consequently, each perturbed pseudoholomorphic strip can be given
  coordinates \((s, t)\) determined by the equation \((s(\zeta), t(\zeta))
  = \phi^{-1}(\zeta)\).
\end{lemma}
%
\begin{proof}
We begin by observing that if there were \(t_0\in \widehat{\mathcal{I}}\)
  for which \(\hat{p}^*\tilde{\alpha}(t_0)=0\), then \(\hat{p}(t_0)\)
  is a critical point of \(a\circ \tilde{u}\).
To see this, recall that since \(a\circ \tilde{u}\circ \hat{p}={\rm
  const}\), it follows that \(d(a\circ\tilde{u})(T\hat{p}\cdot
  \partial_t)=0\); but then if \(0=\hat{p}^*\tilde{\alpha}(t_0)\), then
  \begin{equation*}                                                       
    0=\hat{p}^*\tilde{\alpha}(t_0)= -d(a\circ
      \tilde{u})(\tilde{\jmath}\cdot T\hat{p}\cdot \partial_t)\big|_{t_0},
    \end{equation*}
  so that \(d(a\circ \tilde{u})(\hat{p}(t_0))=0\), which is impossible
  since the image of \(\hat{p}\) is disjoint from the critical points of
  \(d(a\circ \tilde{u})\) by assumption.
After possibly precomposing with an orientation reversing diffeomorphism
  of \(\widehat{\mathcal{I}}\), we may consequently assume that
  \(\hat{p}^*\tilde{\alpha}(\partial_t)>0\).
We now define the interval \(\mathcal{I}\) by
  \begin{equation*}                                                       
    \mathcal{I}:= \Big[ 0,
      \int_{\widehat{\mathcal{I}}}\hat{p}^*\tilde{\alpha}\Big] \subset
      \mathbb{R}
    \end{equation*}
  and define the map \(\psi\) by
  \begin{equation*}                                                       
    \psi:\mathcal{I} \to \widehat{\mathcal{I}}\quad\text{by}\quad
      \psi(t):=F^{-1}(t) \qquad\text{where}\qquad F(t):=\int_0^t
      \tilde{\alpha}_{\hat{p}(\hat{t})}(\hat{p}'(\hat{t})) d \hat{t};
    \end{equation*}
here and above the subscripts denote the points of evaluation. 
To see that \(\psi\) is the desired diffeomorphism, it is sufficient
  to show that
  \begin{equation}\label{EQ_psi_prop_1}                                   
    \tilde{\alpha} \big((\hat{p}\circ \psi)'(t)\big)\equiv 1.
    \end{equation}
To establish this equality, we define the function
  \begin{equation*}                                                       
    G(t)=\tilde{\alpha}_{\hat{p}(t)}(\hat{p}'(t))
    \end{equation*}
so that \(F'=G\) and \(\psi=F^{-1}\).  
Next recall that \(t=F(F^{-1}(t))\), and differentiating we find
  \((F^{-1})'(t)=\frac{1}{F'(F^{-1}(t))}\), and hence
  \begin{align*}                                                          
    \tilde{\alpha}_{\hat{p}\circ \psi(t)}\big(\hat{p}_{\psi(t)}'\cdot
      \psi'(t)\big)&= \tilde{\alpha}_{\hat{p}\circ
      \psi(t)}\Big(\hat{p}_{\psi(t)}'\cdot \frac{1}{F'(F^{-1}(t))}
      \Big)\\
    &=\tilde{\alpha}_{\hat{p}\circ \psi(t)}\Big(\hat{p}_{\psi(t)}'\cdot
      \frac{1}{G(\psi(t))} \Big)\\
    &=1.
    \end{align*}
This verifies equation (\ref{EQ_psi_prop_1}). 
With \(\psi\) established,  it is straightforward to show that the
  tuple \((p, \mathcal{I}, h^-, h^+):=(\hat{p}\circ \psi, \mathcal{I},
  \hat{h}^-\circ\psi, \hat{h}^+\circ\psi)\) satisfies properties
  (e\ref{EN_e1}) - (e\ref{EN_f3}) of Definition
  \ref{DEF_perturbed_J_strip}, and hence all that remains is to establish
  properties (h\ref{EN_h1}) - (h\ref{EN_h4}).

Next we note that property (h\ref{EN_h4}) follows from properties
  (h\ref{EN_h1}) - (h\ref{EN_h3}), and property (h\ref{EN_h1}) follows
  immediately from the definition of the diffeomorphism \(\phi\).

To prove property (h\ref{EN_h2}),  we let subscripts denote partial
  differentiation, and then
  \begin{align*}                                                          
    \hat{\gamma}(\partial_s, \partial_t) &=g(T\tilde{u} \cdot \phi_s ,
      T\tilde{u} \cdot \phi_t).
    \end{align*}
However by construction, for each fixed \(s_0\) the map  \(t\mapsto
  \phi(s_0, t)\) is contained in a level set of \(a\circ \tilde{u}\), and
  the vector field \(\phi_s\) is parallel to \(\widetilde{\nabla}(a\circ
  \tilde{u})\).
The gradient is orthogonal to level sets, and hence \(\phi_t\) and
  \(\phi_s\) are \(\tilde{\gamma}\)-orthogonal.
Consequently \(\partial_s\) and \(\partial_t\) are
  \(\hat{\gamma}\)-orthogonal, and \(\hat{\gamma}=\hat{\gamma}_{11}ds^2 +
  \hat{\gamma}_{22}dt^2\) as claimed.
To see \(\hat{\gamma}_{11}\geq 1\),  we first note:
  \begin{align*}                                                          
    \hat{\gamma}_{11}=\|\phi_s\|_{\tilde{\gamma}}^2=(da ( T\tilde{u}
      \cdot \phi_s)\big)^2  + \big(\lambda(T\tilde{u} \cdot \phi_s)
      \big)^2 + \omega(T\tilde{u} \cdot\phi_s, J\cdot T\tilde{u}\cdot
      \phi_s).
    \end{align*}
By Lemma \ref{LEM_J_diff}, we see that \(\omega(T\tilde{u}\cdot \phi_s,
  J \cdot T\tilde{u}\cdot \phi_s)\geq 0\).
Also recall that 
  \begin{equation*}                                                       
    \phi_s =  \frac{\widetilde{\nabla}(a\circ
      \tilde{u})}{\|\widetilde{\nabla}(a\circ
      \tilde{u})\|_{\tilde{\gamma}}^2},
    \end{equation*}
  and hence 
  \begin{align*}                                                          
    \hat{\gamma}_{11} \geq (da ( T\tilde{u} \cdot \phi_s)\big)^2 =
      \Big(\frac{d(a \circ \tilde{u}) (\widetilde{\nabla}(a\circ
      \tilde{u}))}{\|\widetilde{\nabla}(a\circ
      \tilde{u})\|_{\tilde{\gamma}}^2}\Big)^2 =1.
    \end{align*}  
This establishes property (h\ref{EN_h2}).

Finally, to establish property (h\ref{EN_h3}), observe
  \begin{align*}
    \phi^*\tilde{\alpha}(\partial_s)&= 
      -d(a\circ \tilde{u})\big(\tilde{\jmath}\widetilde{\nabla}(a\circ
      \tilde{u})\big)/\|\widetilde{\nabla}(a\circ
      \tilde{u})\|_{\tilde{\gamma}}^2=0
      \end{align*}
  since \(\tilde{\jmath}\) is an almost complex structure and a
  \(\tilde{\gamma}\)-isometry.
Consequently \(\phi^*\tilde{\alpha}=\ell(s, t) dt\).
Note that 
  \begin{equation*}                                                       
    \ell(0, t) =\tilde{\alpha}(\phi_t(0, t))= \tilde{\alpha}(p'(t))=1
    \end{equation*}
  by definition of \(\phi\) and property (e\ref{EN_f2}) of Definition
  \ref{DEF_perturbed_J_strip}.
We also note that \(\ell\) vanishes precisely at the critical points of
  \(a\circ \tilde{u}\), and by construction none of these are contained
  in the region \(\widetilde{S}\) by assumption.

To complete the proof of property (h\ref{EN_h3}) all that remains is
  to show that \(\ell\leq \hat{\gamma}_{22}^{\frac{1}{2}}\).
We estimate as follows. 
\begin{align*}                                                            
  \ell^2 &= 
    \big(\tilde{\alpha}(\phi_t)\big)^2\\
  &= \big( d(a\circ \tilde{u}) (\tilde{\jmath} \phi_t)\big)^2\\
  &\leq \|d(a\circ \tilde{u})\|_{\tilde{\gamma}}^2
    \|\tilde{\jmath}\phi_t\|_{\tilde{\gamma}}^2\\
  &=\|da \|_{g}^2 \|\phi_t\|_{\tilde{\gamma}}^2\\
  &= \|\phi_t\|_{\tilde{\gamma}}^2\\
  &= \hat{\gamma}_{22}.
  \end{align*}
Since \(\ell\) and \(\gamma_{22}\) are both positive functions, the
  desired result is then immediate.
This proves property (h\ref{EN_h3}), and hence completes the proof of
  Lemma \ref{LEM_J_strip_reparam}.
\end{proof}
%

\begin{remark}[strip coordinates]
  \label{REM_strip_coordinates}
  \hfill\\
One immediate consequence of Lemma \ref{LEM_J_strip_reparam} above, is
  that it guarantees the existence of the carefully defined structure of a
  perturbed pseudoholomorphic strip from very little, but necessary,
  geometric data.
A second consequence is that it immediately guarantees coordinates
  \((s,t)\) on a perturbed pseudoholomorphic strip which turn out to be
  quite useful.
In particular, in light of properties (h\ref{EN_h2}) and (h\ref{EN_h3})
  above,  we will henceforth assume
  \begin{align*}                                                          
    &\tilde{u}^*g = \tilde{\gamma} = \tilde{\gamma}_{11} \,ds \otimes ds +
      \tilde{\gamma}_{22} \, dt \otimes dt  \\
    &\tilde{\gamma}_{11}(s,t) \geq 1\\
    &\tilde{\gamma}_{22}(s,t) \geq \ell(s,t)
    \end{align*}
  where
  \begin{align*}                                                          
    \tilde{\alpha} = \ell(s,t) dt.
    \end{align*}
Later we will attempt to provide a uniform lower bound for \(\ell\); for
  the moment though, we only know that this function is smooth and positive.
\end{remark}
%

Lemma \ref{LEM_general_strip_estimate} below can be thought of as an
  analog of Theorem \ref{THM_area_bound_estimate} for perturbed
  pseudoholomorphic strips which are not too tall.
Roughly speaking, it guarantees that if the integral of \(u^*\lambda\)
  (or, more precisely, the integral of \(-\tilde{u}^* da \circ
  \tilde{\jmath}\) along the top boundary is more than twice the integral
  along the bottom boundary, the strip must capture some \(\omega\)-energy
  proportional to this difference.
We note that in what follows we adapt our notation from Definition
  \ref{DEF_tract_of_perturbed_J_map} in which \(\partial \widetilde{S} =
  \partial_0 \widetilde{S}\cup \partial_1\widetilde{S}\), where \(\partial_0
  \widetilde{S}\) is the ``top'' and ``bottom'' boundaries of of
  \(\widetilde{S}\), and \(\partial_1 \widetilde{S}\) is the ``side''
  boundary of \(\widetilde{S}\).

\begin{lemma}[general strip estimate]
  \label{LEM_general_strip_estimate}
  \hfill\\
Let \((M, \eta)\) be a framed Hamiltonian manifold, let \((J, g)\) be
  an \(\eta\)-adapted almost Hermitian structure on \(\mathbb{R}\times
  M\), let \(C_{\mathbf h}\) be the associated ambient geometry constant
  established in Definition \ref{DEF_ambient_geometry_constant}, and let
  \((\tilde{u}, \widetilde{S}, \tilde{\jmath}, f, u, S, j)\) be a
  perturbed \(J\)-strip determined by the data \((p, \mathcal{I}, h^-,
  h^+)\), with
  \begin{equation}\label{EQ_h_bounds}                                     
    -\frac{\ln 2}{2 C_{\mathbf h}}\leq h^-\leq 0 \leq h^+\leq \frac{\ln
    2}{2 C_{\mathbf h}}.
    \end{equation}
With \((s, t)\) the coordinates on \(\widetilde{S}\) as guaranteed by
  Lemma \ref{LEM_J_strip_reparam},  we define
  \begin{equation*}                                                       
    \partial_0^\pm \widetilde{S}:=\{(s, t)\in \partial_0 \widetilde{S}:
    \pm s>0\},
    \end{equation*}
  with orientation such that \(\tilde{\alpha}\) defines a positive volume
  form on each.
Then 
  \begin{align*}                                                          
    \int_{\partial_0^+\widetilde{S} } \tilde{\alpha} -2
      \int_{\partial_0^-\widetilde{S}} \tilde{\alpha}
    &\leq 
      2C_{\mathbf h} \int_{\widetilde{S}}u^*\omega,
    \end{align*}
  and similarly
  \begin{align*}                                                          
    \int_{\partial_0^-\widetilde{S} } \tilde{\alpha} -2
      \int_{\partial_0^+\widetilde{S}} \tilde{\alpha}
    &\leq 
      2C_{\mathbf h} \int_{\widetilde{S}}u^*\omega.
    \end{align*}
\end{lemma}
%
\begin{proof}
Without loss of generality, \(\mathcal{I}=[0, b]\). 
For each \(k\in \mathbb{N}_0\) and \(1\leq \ell \leq 2^k\) define
  the interval
  \begin{equation*}                                                       
    \mathcal{I}_{k, \ell}:=\big[{\textstyle \frac{\ell -1}{2^k}}b,
    {\textstyle \frac{\ell}{2^k}}b\big].
    \end{equation*}
Similarly, define \(h_{k, \ell}^+:=\inf_{t\in \mathcal{I}_{k, \ell}}
  h^+(t)\) and \(h_{k, \ell}^-:=\sup_{t\in \mathcal{I}_{k, \ell}} h^-(t)\).
By definition of integrability, we have
  \begin{equation*}                                                       
    \int_{\{(s, t)\in \widetilde{S}: s\geq 0 \}}d\tilde{\alpha}=
    \lim_{k\to \infty} \sum_{\ell =1}^{2^k} \int_{[0, h_{k,
    \ell}^+]\times \mathcal{I}_{k, \ell}} d\tilde{\alpha}
    \end{equation*}
  and
  \begin{equation*}                                                       
    \int_{\{(s, t)\in \widetilde{S}: s\leq 0 \}}d\tilde{\alpha}=
    \lim_{k\to \infty} \sum_{\ell =1}^{2^k} \int_{[h_{k, \ell}^-,
    0]\times \mathcal{I}_{k, \ell}} d\tilde{\alpha},
    \end{equation*}
  and hence by Stokes' theorem and the fact that
  \(\tilde{\alpha}(\partial_s)=0\), we have
  \begin{equation*}                                                       
    \int_{\partial_0^+ \widetilde{S}} \alpha =\lim_{k\to \infty}
    \sum_{\ell =1}^{2^k} \int_{\{h_{k, \ell}^+\}\times \mathcal{I}_{k,
    \ell}} \tilde{\alpha},
    \end{equation*}
  and
  \begin{equation*}                                                       
    \int_{\partial_0^- \widetilde{S}} \alpha =\lim_{k\to \infty}
    \sum_{\ell =1}^{2^k} \int_{\{h_{k, \ell}^-\}\times \mathcal{I}_{k,
    \ell}} \tilde{\alpha}.
    \end{equation*}
At this point we observe that restricting \(\tilde{u}\) to each
  \([h_{k, \ell}^-, h_{k, \ell}^+]\times \mathcal{I}_{k, \ell}\)
  defines a rectangular perturbed pseudoholomorphic strip, and hence by
  Theorem
  \ref{THM_area_bound_estimate} that
  \begin{equation}                                                        
    \int_{\{h_{k, \ell}^+\}\times \mathcal{I}_{k, \ell}}
    \tilde{\alpha}\leq \Big(C_{\mathbf h}\int_{[h_{k, \ell}^-, h_{k,
    \ell}^+]\times\mathcal{I}_{k, \ell}}u^*\omega + \int_{\{h_{k,
    \ell}^-\}\times \mathcal{I}_{k, \ell}} \tilde{\alpha}\Big)
    e^{C_{\mathbf h} (h_{k, \ell}^+-h_{k, \ell}^-)}.
    \end{equation}  
Making use of the inequalities (\ref{EQ_h_bounds}) and rearranging,
  we find
  \begin{equation}                                                        
    \int_{\{h_{k, \ell}^+\}\times \mathcal{I}_{k, \ell}} \tilde{\alpha} -
    2\int_{\{h_{k, \ell}^-\}\times \mathcal{I}_{k, \ell}} \tilde{\alpha}
    \leq 2 C_{\mathbf h}\int_{[h_{k, \ell}^-, h_{k,
    \ell}^+]\times\mathcal{I}_{k, \ell}}u^*\omega.
    \end{equation}
Summing and passing to the limit then yields
  \begin{align*}                                                          
  \int_{\partial_0^+\widetilde{S} } \tilde{\alpha} -2
    \int_{\partial_0^-\widetilde{S}} \tilde{\alpha}
  &=
    \lim_{k\to\infty}\sum_{\ell=1}^{2^k}\int_{\{h_{k,
    \ell}^+\}\times \mathcal{I}_{k, \ell}} \tilde{\alpha} -2
    \lim_{k\to\infty}\sum_{\ell=1}^{2^k} \int_{\{h_{k, \ell}^-\}\times
    \mathcal{I}_{k, \ell}} \tilde{\alpha}\\
  &\leq
    \lim_{k\to \infty} \sum_{\ell=1}^{2^k} 2C_{\mathbf h} \int_{[h_{k,
    \ell}^-, h_{k, \ell}^+]\times\mathcal{I}_{k, \ell}}u^*\omega\\
  &\leq 
    2C_{\mathbf h} \int_{\widetilde{S}}u^*\omega.
  \end{align*}
This is the desired estimate, and hence completes the proof of Lemma
  \ref{LEM_general_strip_estimate}.
\end{proof}

In light of Remark \ref{REM_strip_coordinates}, and with the aid of Lemma
  \ref{LEM_general_strip_estimate}, our next task is to attempt to obtain a
  uniform lower bound on the smooth positive function \(\ell\) defined by
  the property that \(\tilde{\alpha} = \ell(s,t) dt\), since doing so
  would be quite beneficial for later estimates.
Unfortunately, such a uniform pointwise estimate is not true, but rather
  the desired pointwise estimate holds everywhere except on a set which is
  small in a controlled manner.
We make this precise with Lemma \ref{LEM_lambda_shrinkage} below.

\begin{lemma}[$\lambda$-shrinkage]
  \label{LEM_lambda_shrinkage}
  \hfill\\
Let \((M, \eta)\) be a framed Hamiltonian manifold and \((J, g)\) be an
  \(\eta\)-adapted almost Hermitian structure on \(\mathbb{R}\times M\),
  and let \((\tilde{u}, \widetilde{S}, \tilde{\jmath}, f, u, S, j)\) be
  a perturbed pseudoholomorphic strip determined by the data \((p,
  \mathcal{I}, h^-, h^+)\), with coordinates \((s, t)\) as guaranteed by
  Lemma \ref{LEM_J_strip_reparam}.
Let \(C_{\mathbf h}\) be the associated ambient geometry constant
  established in Definition \ref{DEF_ambient_geometry_constant}, 
Suppose that for each \(t\in \mathcal{I}\) we have
  \begin{equation*}
    - \frac{\ln 2}{2 C_{\mathbf{h}}}  \leq h^-(t) \leq 0 \leq h^+(t) \leq
    \frac{\ln 2}{2 C_{\mathbf{h}}}.
    \end{equation*}
Define the set 
  \begin{equation*}
    \mathcal{K}:=\big\{t\in \mathcal{I}: {\textstyle \frac{1}{8}}
    \geq \inf_{h^-(t)\leq s\leq h^+(t)} \ell(s, t)\big\}
    \end{equation*}
  where \(\tilde{\alpha} = \ell(s, t)\, dt\). 
Then
  \begin{equation*}
    4 C_{\mathbf{h}} \int_{\widetilde{S}}u^*\omega\geq \mu(\mathcal{K})
    \end{equation*}
  where \(\mu\) is the Lebesgue measure associated to the coordinate
  \(t\in \mathcal{I}\).
\end{lemma}
%
\begin{proof}
We begin by observing that \(\mathcal{K}\) is compact. 
Next we define
  \begin{equation*}                                                       
    \mathcal{A}:=\{(s, t)\in \widetilde{S}: \ell(s, t) \leq {\textstyle
    \frac{1}{8}}\}.
    \end{equation*}
Next, for each \(\zeta= (s_0, t_0)\in \mathcal{A}\) we define some
  quantities, which we also describe geometrically below:
  \begin{align*}                                                          
    \sigma(\zeta)&= \sigma(s_0, t_0)=
    \begin{cases}
      \sup \big\{\hat{s}\in [0, h^+(t_0)]:{\displaystyle \inf_{0\leq
	s\leq \hat{s}} }\{\ell(s, t_0)\}\geq {\textstyle
	\frac{1}{6}}\big\} &\text{if }s_0> 0\\
      \inf \big\{\hat{s}\in [h^-(t_0), 0]:{\displaystyle
	\inf_{\hat{s}\leq s \leq 0}  \{\ell(s, t_0)\}}\geq {\textstyle
	\frac{1}{6}}\big\} &\text{if }s_0<0
      \end{cases} \\
    x(\zeta)&= x(s_0, t_0) = \inf \big\{\tau\in \mathcal{I}: \tau\leq
      t_0 \text{ and } \sup_{\tau\leq t \leq t_0}\{\ell(\sigma(\zeta),
      t)\}\leq {\textstyle \frac{1}{4}}  \big\}\\ 
    y(\zeta) &= y(s_0, t_0) = \sup \big\{\tau\in \mathcal{I}: \tau\geq
      t_0 \text{ and } \sup_{t_0\leq t \leq \tau}\{\ell(\sigma(\zeta),
      t)\}\leq {\textstyle \frac{1}{4}}  \big\}.
    \end{align*}

We take a moment to describe these functions in a more geometric context.
First, recall that in the coordinates \((s,t)\), the \(s\) coordinate can
  be thought of as the symplectization coordinate, and the \(t\) coordinate
  measuring movement within a symplectization level set.
Suppose we are given some \((s_0, t_0)\in \mathcal{A}\) with \(s_0>0\),
  and consider the coordinate path \(t=t_0\), which by construction is an
  integral curve of \(\widetilde{\nabla} (a\circ \tilde{u})\).
In particular, we restrict the smooth function \(\ell\) to this path, and
  note that \(\ell(0, t_0)=1\) by construction, so that as we increase \(s\)
  from zero, there must be a first value for which
  \(\ell(s,t_0)=\frac{1}{6}\).
This first such \(s\) is then defined to be \(\sigma(s_0, t_0)\), provided
  \(s_0\) is positive, which we have indeed assumed.
A similar construction holds for the case that \(s_0\) is negative.

To understand the functions \(x\) and \(y\), we can consider the point
  \((s_0, t_0)\in \mathcal{A}\), again assuming \(s_0>0\).
Recall that by definition of \(\mathcal{A}\), we have \(\ell(s_0, t_0)
  \leq \frac{1}{8}\).
To the point \((s_0, t_0)\) we will associate the point \((\sigma(s_0,
  t_0), t_0)\), and we recall that \(\ell(\sigma(s_0, t_0), t_0) =
  \frac{1}{6}\) by construction, so that necessarily 
  \begin{align}\label{EQ_technical1}                                      
    \sigma(s_0, t_0) < s_0\leq h^+(t_0).
    \end{align}
We then we aim to construct a path of the form \(\{\sigma(s_0,t_0)\}\times
  I\) on which \(\ell \leq \frac{1}{4}\); here \(I\) is an interval.
This is possible since \(\ell\) is continuous and \(\ell( \sigma(s_0,
  t_0), t_0)=\frac{1}{6}<\frac{1}{4}\).
Thus we define the functions \(x\) and \(y\) to respectively be the
  smallest and largest possible values for which \(I:=[x(s_0,t_0), y(s_0,
  t_0)]\) has the property that on \(\{\sigma(s_0,t_0)\}\times I\) we have
  \(\ell\leq \frac{1}{4}\).
Moreover in light of inequality (\ref{EQ_technical1}) it also follows that
  \begin{align*}                                                          
    x(s_0,t_0) < t_0 < y(s_0, t_0)
    \end{align*}
 whenever \(t_0\in \mathcal{I}\setminus \partial \mathcal{I}\), and for all
  \(t_0\in \mathcal{I}\) we have \(x(s_0,t_0) < y(s_0, t_0)\).

With the functions \(x\) and \(y\) defined and understood, we can now
  define the following collection of (relatively) open subsets
  \(\mathcal{O}_\zeta \subset\mathcal{I}\) for each \(\zeta= (s_0, t_0)\in
  \mathcal{A}\).
\begin{equation*}                                                         
  \mathcal{O}_\zeta:=
  \begin{cases}
    \big[x(\zeta), y(\zeta)\big) &\text{if }x(\zeta)=t_0 \\
    \big(x(\zeta), y(\zeta)\big] &\text{if }y(\zeta)=t_0\\
    \big(x(\zeta), y(\zeta)\big) &\text{otherwise.}
    \end{cases}
  \end{equation*}
Observe that since \(\ell\) is continuous and \(\widetilde{S}\) is
  compact, it follows that if \(t_0\in \mathcal{K}\), then there exists
  \(s_0\in (h^-(t_0), h^+(t_0))\) such that \((s_0, t_0)\in \mathcal{A}\),
  and by construction \(t_0\in \mathcal{O}_{\zeta}\).
Consequently \(\{\mathcal{O}_{\zeta}\}_{\zeta\in \mathcal{K}}\) is an
  open cover of \(\mathcal{K}\), which as previously mentioned, is compact.
It follows that there exists a finite set \(\{\zeta_1^-, \ldots,
  \zeta_{k^-}^-,\zeta_1^+, \ldots, \zeta_{k^+}^+\}\subset\mathcal{A}\)
  such that
  \begin{equation*}                                                       
    \mathcal{K}\subset \big(\bigcup_{i=1}^{k^-}
    \mathcal{O}_{\zeta_i^-}\big)\;\cup\; \big( \bigcup_{i=1}^{k^+}
    \mathcal{O}_{\zeta_{i}^+}  \big),
    \end{equation*}
  and for which \(\zeta_i^\pm=(s_i^\pm, t_i^\pm)\) with \(s_i^+\geq 0\)
  and \(s_i^-<0\).
The pairwise intersection of these open intervals may be open and
  nonempty, so we refine this set of intervals by the following inductive
  procedure.
\begin{align*}                                                            
  \tilde{\mathcal{U}}_1^- &:= \mathcal{O}_{\zeta_1^-}\\
  \tilde{\mathcal{U}}_i^- &:=\mathcal{O}_{\zeta_i^-}\setminus
    \cup_{i'=1}^{i-1} \overline{\tilde{\mathcal{U}}_{i'}^-}\\
  \tilde{\mathcal{U}}_1^+&:=\mathcal{O}_{\zeta_1^+}\setminus
    \cup_{i'=1}^{k^-} \overline{\tilde{\mathcal{U}}_{i'}^-}\\
  \tilde{\mathcal{U}}_i^+&:=\mathcal{O}_{\zeta_i^+}\setminus\Big(
    \big(\cup_{i'=1}^{k^-} \overline{\tilde{\mathcal{U}}_{i'}^-}\big) \cup
    \big(\cup_{i'=1}^{i-1}
    \overline{\tilde{\mathcal{U}}_{i'}^+}\big)\Big),
  \end{align*}
  where the bar denotes closure. 
Observe that each \(\tilde{\mathcal{U}}_i^-\) and
  \(\tilde{\mathcal{U}}_i^+\) is the union of finitely many disjoint open
  intervals, and hence we may further write
  \begin{equation*}                                                       
    \tilde{\mathcal{U}}_i^- = \bigcup_{i'=1}^{m_i^-}\mathcal{U}_{i,
    i'}^-\qquad\text{and}\qquad\tilde{\mathcal{U}}_i^+ =
    \bigcup_{i'=1}^{m_i^+}\mathcal{U}_{i, i'}^+,
    \end{equation*}
  where each \(\mathcal{U}_{i, i'}^-\) and \(\mathcal{U}_{i, i'}^+\)
  is an open interval.

Finally we define the following finite sets of products of closed
  intervals.
\begin{align*}                                                            
  \mathcal{S}_{i, i'}^+&:=
    [0, s_{i, i'}^+]\times \overline{\mathcal{U}}_{i,
    i'}^+\quad\text{for}\quad 1\leq i\leq k^+\;\text{with }s_{i,
    i'}^+:=\sigma(\zeta_i^+)\geq 0, \; \text{and}\; 1\leq i'\leq m_i^+\\
  \mathcal{S}_{i, i'}^-&:=
    [ s_{i, i'}^-, 0]\times \overline{\mathcal{U}}_{i,
    i'}^-\quad\text{for}\quad 1\leq i\leq k^-\;\text{with }s_{i,
    i'}^-:=\sigma(\zeta_i^-)< 0, \; \text{and}\; 1\leq i'\leq m_i^-\\
  \end{align*}
For later computational clarity, it will be convenient to reindex these
  sets by
  \begin{equation*}                                                       
    \nu^\pm: \{1, \ldots, m^\pm:=\sum_{i=1}^{k^\pm} m_i^\pm\}\to \{(1,
    1), \ldots, (1, m_1^\pm), (2, 1), \ldots, (2, m_2^\pm), \ldots,
    (k, m_k^\pm)\}
    \end{equation*}
  so that
  \begin{equation*}                                                       
    \mathcal{S}_{\nu^+}^+ = [0, s_{\nu^+}^+]\times
    \overline{\mathcal{U}}_{\nu^+}^+ \quad\text{for}\quad \nu^+\in \{1,
    \ldots, m^+\}
    \end{equation*}
  and
  \begin{equation*}                                                       
    \mathcal{S}_{\nu^-}^- = [s_{\nu^-}^-, 0]\times
    \overline{\mathcal{U}}_{\nu^-}^- \quad\text{for}\quad \nu^-\in \{1,
    \ldots, m^-\}.
    \end{equation*}
For convenience, define  
  \begin{equation*}                                                       
    \overline{\mathcal{U}}:= \big(\cup_{\nu^-=1}^{m^-}
    \overline{\mathcal{U}}_{\nu^-}^-\big)\bigcup
    \big(\cup_{\nu^+=1}^{m^+} \overline{\mathcal{U}}_{\nu^+}^+\big) ,
    \end{equation*}
  so that \(\{0\}\times\overline{\mathcal{U}}=\big((\cup_{\nu^-=1}^{m^-}
  \mathcal{S}_{\nu^-}^-)\cup ( \cup_{\nu^+=1}^{m^+} \mathcal{S}_{\nu^+}^+
  )\big)\cap \big(\{0\}\times \mathcal{I}\big)\). The following facts
  are then straightforward to verify.
\begin{enumerate}[(S1)]                                                   
  \item\label{EN_S1}  
    \(\mathcal{K}\subset \overline{\mathcal{U}}\)
  \item\label{EN_S2} 
    \begin{equation*}                                                     
      \inf_{\substack{ t\in \mathcal{I}\setminus \overline{\mathcal{U}} \\
      h^-(t)\leq s\leq h^+(t)} } \{\ell(s, t)\}\geq {\textstyle
      \frac{1}{8}}
      \end{equation*} 
  \item\label{EN_S3} 
    for each \(\nu^\pm \in \{1, \ldots, m^\pm\}\) we have 
    \begin{equation*}
      \sup_{t\in \overline{\mathcal{U}}_{\nu^\pm}^\pm}
      \ell(s_{\nu^\pm}^\pm, t)\leq {\textstyle \frac{1}{4}}
      \end{equation*} 
  \item\label{EN_S4} 
    the pairwise intersection of elements of the set \(\{ S_1^-, \ldots,
    S_{m^-}^-, S_1^+, \ldots, S_{m^+}^+ \}\) have empty interior in
    \(\widetilde{S}\).
  \end{enumerate} 
Observe that since \(\tilde{\alpha}= \ell\, dt\), and since \(\ell(0,
  t)\equiv 1\), we may employ property (S\ref{EN_S3}) to estimate
  \begin{align*}                                                          
    \int_{\{s_{\nu^\pm}^\pm\}\times \overline{\mathcal{U}}_{\nu^\pm}^\pm}
      \tilde{\alpha}
    &= 
      \int_{\inf \overline{\mathcal{U}}_{\nu^\pm}^\pm}^{\sup
      \overline{\mathcal{U}}_{\nu^\pm}^\pm} \ell(s_{\nu^\pm}^\pm, t)
      dt \\
    &\leq
      {\textstyle \frac{1}{4}} \int_{\inf
      \overline{\mathcal{U}}_{\nu^\pm}^\pm}^{\sup
      \overline{\mathcal{U}}_{\nu^\pm}^\pm} 1\;dt \\
    &=
      {\textstyle \frac{1}{4}}\int_{\{0\}\times
      \overline{\mathcal{U}}_{\nu^\pm}^\pm}\tilde{\alpha}.
    \end{align*}
Observe that each \((\tilde{u}, \mathcal{S}_{\nu^\pm}^\pm ,
  \tilde{\jmath}, f, u, S, j)\) is a rectangular perturbed
  pseudoholomorphic strip determined by the data \((p,
  \overline{\mathcal{U}}_{\nu^\pm}^\pm
  , h_{\nu^\pm}^-, h_{\nu^\pm}^+)\) where \(h_{\nu^+}^-=0\),
  \(h_{\nu^+}^+=s_{\nu^+}^+\), \(h_{\nu^-}^-=s_{\nu^-}\), and
  \(h_{\nu^-}^+=0\).
Moreover, we have 
  \begin{align*}                                                          
    |s_{\nu^\pm}^\pm|\leq \sup_t |h^\pm(t)|\leq \frac{\ln 2}{
    2C_{\mathbf{h}}}
    \end{align*}
  and
  \begin{align*}                                                          
    \int_{\{s_{\nu^\pm}^\pm\}\times \overline{\mathcal{U}}_{\nu^\pm}^\pm}
      \tilde{\alpha} \leq {\textstyle \frac{1}{4}}\int_{\{0\}\times
      \overline{\mathcal{U}}_{\nu^\pm}^\pm}\tilde{\alpha},
    \end{align*}
  and thus letting \(\mu\) denote the Lebesgue measure associated to the
  coordinate \(t\), we employ Lemma \ref{LEM_general_strip_estimate} to
  estimates
\begin{align*}                                                            
  {\textstyle \frac{1}{2}}\mu(\overline{\mathcal{U}}_{\nu^\pm}^\pm) 
  &=
  {\textstyle\frac{1}{2}}\int_{\{0\}\times
  \overline{\mathcal{U}}_{\nu^\pm}^\pm}\tilde{\alpha}
  \\
  &= \int_{\{0\}\times \overline{\mathcal{U}}_{\nu^\pm}^\pm}\tilde{\alpha}
  -{\textstyle\frac{1}{2}}\int_{\{0\}\times
  \overline{\mathcal{U}}_{\nu^\pm}^\pm}\tilde{\alpha}
  \\
  &\leq \int_{\{0\}\times
  \overline{\mathcal{U}}_{\nu^\pm}^\pm}\tilde{\alpha} - 2
  \int_{\{s_{\nu^\pm}^\pm\}\times
  \overline{\mathcal{U}}_{\nu^\pm}^\pm}\tilde{\alpha}
  \\
  &\leq 2C_{\mathbf{h}} \int_{ \mathcal{S}_{\nu^\pm}^\pm}
  \tilde{u}^*\omega.
  \end{align*}
Or more concisely,
  \begin{align*}                                                          
    \mu(\overline{\mathcal{U}}_{\nu^\pm}^\pm) \leq 4C_{\mathbf{h}}\int_{
    \mathcal{S}_{\nu^\pm}^\pm} \tilde{u}^*\omega.
    \end{align*}
Using property (S\ref{EN_S4}) and the fact that \(\omega\) evaluates
  non-negatively on complex lines, it follows that
  \begin{align*}
    4C_{\mathbf{h}}\int_{\widetilde{S}}u^*\omega &\geq
    4C_{\mathbf{h}}\sum_{\nu^\pm=1}^{m^\pm}
      \int_{\mathcal{S}_{\nu^\pm}^\pm} u^*\omega \geq 
      \sum_{\nu^\pm=1}^{m^\pm}\mu
      (\overline{\mathcal{U}}_{\nu^\pm}^\pm)\\
    &=
      \mu(\cup_{\nu^\pm=1}^{m^\pm}
      \overline{\mathcal{U}}_{\nu^\pm}^\pm)= 
      \mu(\overline{\mathcal{U}})\\
    &\geq 
      \mu(\mathcal{K}),
    \end{align*}
  where the final equality follows from property (S\ref{EN_S1}).
\end{proof}

We finish Section \ref{SEC_strip_estimates} with Lemma
  \ref{LEM_modest_length_flow_lines} below, which roughly states that if a
  pseudoholomorphic strip is not too ``tall,'' and the \(\omega\)-energy is
  less than or equal to the ``height'' times the ``width'' then there must
  exist a gradient trajectory going from bottom to top with length with is
  not too long (relative to the height of the strip).

\begin{lemma}[modest length flow lines]
  \label{LEM_modest_length_flow_lines}
  \hfill\\
Let \((M, \eta)\) be a framed Hamiltonian manifold and \((J, g)\) be
  an \(\eta\)-adapted almost Hermitian structure on \(\mathbb{R}\times M\).
Let \(C_{\mathbf h}\) be the associated ambient geometry constant
  established in Definition \ref{DEF_ambient_geometry_constant}. 
Let \((\tilde{u}, \tilde{\jmath}, f, u, S, j)\) be a perturbed
  pseudoholomorphic map, and fix \(\epsilon\in \mathbb{R}\) such that
  \begin{equation*}                                                       
      0<\epsilon < \min(2^{-24}, (1+\sup_{\zeta\in
      {\rm supp}(f)}\|B_u(\zeta)\|_{{\gamma}})^{-1}).
    \end{equation*}
Suppose further that
  \begin{equation*}                                                       
      \|df\|_{{\gamma}}+ \|\nabla df\|_{{\gamma}} \leq
      \frac{\epsilon}{2^{11}(1+\|B_{{u}}\|_{{\gamma}})},
    \end{equation*}
  where \(\|df\|_{{\gamma}}\), \(\|\nabla df\|_{{\gamma}} \), and
  \(\|B_{{u}}\|_{{\gamma}}\) are the \(L^\infty\) norms over the support
  of \(f\).
Then for any finite set of rectangular perturbed pseudoholomorphic strips,
  denoted by \(\{(\tilde{u}_k, \widetilde{S}_k, \tilde{\jmath}_k, f, u, S,
  j)\}_{k=1}^n\), satisfying
  \begin{enumerate}[(R1)]                                                 
    \item\label{EN_R1} 
      \(a_0 = \inf_{\zeta\in \widetilde{S}_k} a\circ
      \tilde{u}_k(\zeta)\), independent of \(k\)
    \item\label{EN_R2} 
      \(a_1 = \sup_{\zeta\in \widetilde{S}_k} a\circ
      \tilde{u}_k(\zeta)\), independent of \(k\)
    \item\label{EQ_R3} 
      \(a_1-a_0\leq\frac{1}{8C_{\mathbf{h}}}\)
    \item\label{EN_R4} 
      \(\sum_{k=1}^n\int_{\widetilde{S}_k} u_k^*\omega \leq
      (a_1-a_0)\sum_{k=1}^n\int_{(a\circ \tilde{u}_k)^{-1}(a_0)}
      \tilde{\alpha},\)
    \item \label{EN_R5}
      \(\widetilde{S}_k \cap \widetilde{S}_{k'} = \emptyset\) for \(k \neq
      k'\)
    \end{enumerate}
  there exists \(k\in \{1, \ldots, n\}\) and a solution to the
  differential equation
  \begin{equation*}                                                       
    q: [0, s_0]\to  S_k \quad q'(s) = \widetilde{\nabla}(a\circ
    \tilde{u}_k)\big(q(s)\big)\quad a(\tilde{u}_k(q(0)))=a_0\quad
    a(\tilde{u}_k(q(s_0)))=a_1
    \end{equation*}
  for which 
  \begin{equation*}                                                       
    {\rm length}_{\tilde{\gamma}}\big(q([0, s_0])\big)\leq 2^7 (a_1-
    a_0).
    \end{equation*}
\end{lemma}
%
\begin{proof}
For convenience, we let \((\tilde{u}, \widetilde{S}, \tilde{\jmath}, f,
  u, S, j)\) denote the union of the rectangular perturbed
  pseudoholomorphic strips \(\{(\tilde{u}_k, \widetilde{S}_k,
  \tilde{\jmath}_k, f, u, S, j)\}_{k=1}^n\) so that
  \(\widetilde{S}=\cup_{k=1}^n \widetilde{S}_k\) and
  \(\tilde{u}\big|_{\widetilde{S}_k} = \tilde{u}_k\), and similarly for
  \(\tilde{\jmath}\).
Next we define the constants 
  \begin{equation*}                                                       
    c_2:=a_1-a_0\qquad\text{and}\qquad c_3:=\int_{(a\circ
    \tilde{u})^{-1}(a_0)}\tilde{\alpha},
    \end{equation*}
  and equip \(\widetilde{S}\) with coordinates \((s, t)\) via
  Lemma \ref{LEM_J_strip_reparam} so that using these coordinates to
  parameterize \(\widetilde{S}\) we have \(\widetilde{S}=[0, c_2]\times
  \mathcal{I}\) where \(\mathcal{I}\subset \mathbb{R}\) is the union
  of finitely many pairwise disjoint closed intervals with total \({\rm
  length}(\mathcal{I}) = c_3\).
Recall that another consequence of Lemma \ref{LEM_J_strip_reparam},
  specifically property (h\ref{EN_h2}), is that
  \begin{equation*}                                                       
    \tilde{\gamma}=\tilde{u}^*g =
    \tilde{\gamma}_{11}\,ds^2+\tilde{\gamma}_{22}\, dt^2.
    \end{equation*}
Consequently, to prove Lemma \ref{LEM_modest_length_flow_lines}, it is
  sufficient to prove
  \begin{equation}\label{EQ_length_est5}                                  
    c_4:=\inf_{t\in \mathcal{I}} \int_0^{c_2}
    \tilde{\gamma}_{11}^{\frac{1}{2}}(s, t)\, ds \leq 2^7 c_2.
    \end{equation}
Next we define the closed set \(\mathcal{K}\) similarly to the way it
  was defined in Lemma \ref{LEM_lambda_shrinkage}:
  \begin{equation*}                                                       
    \mathcal{K}:=\big\{t\in \mathcal{I}: {\textstyle \frac{1}{8}}
    \geq \inf_{0\leq s\leq c_2} \ell(s, t)\big\}
    \end{equation*}
  where \(\tilde{\alpha}=\ell(s, t)\, dt\).
Recall that as a consequence of Lemma \ref{LEM_lambda_shrinkage} we
  have \(\mu(\mathcal{K})\leq 4C_{\mathbf{h}}\int_{\widetilde{S}}
  u^*\omega\), where \(\mu\) denotes the Lebesgue measure on
  \(\mathcal{I}\) associated to the coordinate \(t\).
We then make the following estimate.
\begin{align}                                                             
  {\rm Area}_{\tilde{\gamma}}(\widetilde{S}) 
  &= 
    {\rm Area}_{\tilde{\gamma}}\big([0,c_2]\times\mathcal{I}\big)\notag\\
  &= 
    \int_{\mathcal{I}}\int_0^{c_2} \big(
    \tilde{\gamma}_{11}\tilde{\gamma}_{22}\big)^{\frac{1}{2}} \,
    dsdt\notag\\
  &\geq 
    \int_{\mathcal{I}}\int_0^{c_2}  \tilde{\gamma}_{11}^{\frac{1}{2}}
    \ell  \, dsdt\notag\\
  &\geq 
    \int_{\mathcal{I}\setminus \mathcal{K}} \int_0^{c_2}
    \tilde{\gamma}_{11}^{\frac{1}{2}} \ell  \, dsdt\notag\\
  &\geq 
    {\textstyle \frac{1}{8}} \int_{\mathcal{I}\setminus \mathcal{K}}
    \int_0^{c_2}\tilde{\gamma}_{11}^{\frac{1}{2}} \, dsdt\notag\\
  &\geq 
    {\textstyle \frac{1}{8}} c_4 \int_{\mathcal{I}\setminus \mathcal{K}}
    \, dt\notag\\
  &= 
    {\textstyle \frac{1}{8}} c_4\big(c_3 - \mu(\mathcal{K})\big)\notag\\
  &\geq 
    {\textstyle \frac{1}{8}} c_4\big(c_3 -
    4C_{\mathbf{h}}\int_{\widetilde{S}}
    u^*\omega\big)\label{EQ_area_est_1};
  \end{align}
  where to obtain the second equality we have employed equation
  (\ref{EQ_hausdorff_measure}) from Section \ref{SEC_riemannian}
  which expresses the Hausdorff measure associated to a Riemannian
  metric in local coordinates, and we have used the fact that
  \(\gamma_{21}=\gamma_{12}=0\); to obtain the first inequality we
  have made use of the fact that \(\gamma_{22}^{\frac{1}{2}}\geq \ell\)
  which was established in property (h\ref{EN_h3}) of Lemma
  \ref{LEM_J_strip_reparam}; the third inequality makes use of the
  definition of \(\mathcal{K}\); the fourth inequality follows from the
  definition of \(c_4\); and the final equality employs Lemma
  \ref{LEM_lambda_shrinkage}.
Note, we also have the following estimate.
\begin{align*}                                                            
  \int_{\widetilde{S}} d\big((a\circ
    \tilde{u}-a_1)\tilde{\alpha}\big)+u^*\omega
  &=
    \Big(\int_{\widetilde{S}} (\tilde{u}^*da )\wedge
    \tilde{\alpha}  +\tilde{u}^*\omega\Big) + \int_{\widetilde{S}}
    (a\circ\tilde{u}-a_1) d\tilde{\alpha} \\
  &\geq
    {\textstyle \frac{1}{2}}{\rm Area}_{\tilde{\gamma}} (\widetilde{S})
    + \int_{\widetilde{S}} (a\circ\tilde{u}-a_1) d\tilde{\alpha} \\
  &\geq 
    {\textstyle \frac{1}{2}}{\rm Area}_{\tilde{\gamma}} (\widetilde{S})
    - c_2 C_{\mathbf{h}} {\rm Area}_{\tilde{\gamma}}(\widetilde{S}) \\
  &=
    ({\textstyle \frac{1}{2}}-c_2 C_{\mathbf{h}}){\rm
    Area}_{\tilde{\gamma}} (\widetilde{S}) ,
  \end{align*}
  where the first inequality follows from Lemma 
  \ref{LEM_linear_alg_coercive}, and the second inequality follows from
  Lemma \ref{LEM_d_alpha_tilde_bound}.
We now recall that \(\partial\widetilde{S} = \partial_0 \widetilde{S}\cup
  \partial_1\widetilde{S}\) with \(a\circ \tilde{u}(\partial_0
  \widetilde{S})=\{a_0, a_1\}\) and that \(\partial_1 \widetilde{S}\)
  consists of integral curves of \(\widetilde{\nabla}(a\circ
  \tilde{u})\). We also recall Lemma \ref{LEM_char_fol_grad} which
  guarantees that \(\tilde{\alpha}(\widetilde{\nabla}(a\circ
  \tilde{u}))\equiv 0 \), and hence
  \begin{align*}                                                          
    \int_{\widetilde{S}}  d\big((a\circ \tilde{u}
      -a_1)\tilde{\alpha}\big)
    &= 
      \int_{(a\circ \tilde{u})^{-1}(a_1)} (a\circ \tilde{u}-a_1)
      \tilde{\alpha}-\int_{(a\circ \tilde{u})^{-1}(a_0)} (a\circ
      \tilde{u}-a_1) \tilde{\alpha}\\
    &\qquad  
      +\int_{\partial_1 S}(a\circ \tilde{u}-a_1) \tilde{\alpha} \\
    &=
      0+c_2\int_{(a\circ \tilde{u})^{-1}(a_0)}\tilde{\alpha} + 0\\
    &=
      c_2 c_3.
    \end{align*}
Combining the above inequalities then yields the following.
\begin{equation*}                                                         
  c_2 c_3 + \int_{\widetilde{S}} u^*\omega\geq ({\textstyle
  \frac{1}{2}}-c_2 C_{\mathbf{h}}){\rm Area}_{\tilde{\gamma}}
  (\widetilde{S})
  \end{equation*}
Combining this with inequality (\ref{EQ_area_est_1}), then yields
  the following.
\begin{equation}\label{EQ_area_est_3}                                     
  c_2 c_3 + \int_{\widetilde{S}} u^*\omega \geq {\textstyle
  \frac{1}{8}}({\textstyle \frac{1}{2}}-c_2 C_{\mathbf{h}}) c_4(c_3-
  4C_{\mathbf{h}}\int_{\widetilde{S}} u^*\omega).
  \end{equation}
Finally, we recall our assumptions (R\ref{EQ_R3}) and (R\ref{EN_R4}),
  which can be restated as
  \begin{equation*}                                                       
    c_2= a_1-a_0\leq
    \frac{1}{8C_{\mathbf{h}}}\qquad\text{and}\qquad\int_{\widetilde{S}}
    u^*\omega \leq c_2 c_3.
    \end{equation*}
From these it is elementary to establish the following.
\begin{align*}                                                            
  2 c_2 c_3&\geq 
    c_2 c_3 + \int_{\widetilde{S}} u^*\omega\\
  \frac{1}{2}-c_2 C_{\mathbf{h}}&\geq 
    \frac{1}{4}\\
  c_3-4C_{\mathbf{h}} \int_{\widetilde{S}} u^*\omega &\geq 
    \frac{1}{2}c_3.
  \end{align*}
We now combine these inequalities with (\ref{EQ_area_est_3}) to obtain 
  \begin{equation*}                                                      
    c_4 \leq 2^7 c_2 
    \end{equation*} 
  which is the desired inequality as stated in (\ref{EQ_length_est5}).
This completes the proof of Lemma \ref{LEM_modest_length_flow_lines}.
\end{proof}

\subsubsection{Some preliminary miscellany}\label{SEC_miscellany}

The purpose of this Section is to establish a few miscellaneous results to
  be referenced later in the proof of Theorem
  \ref{THM_local_local_area_bound}.
Firstly, these consist of the notion of a \((\delta,\epsilon)\)-tame
  perturbation of a pseudoholomorphic curve, see Definition
  \ref{DEF_epsilon_tame} below, which essentially provides a certain class
  of perturbations which are sufficiently small so that a variety of
  estimates hold automatically, and then we establish that such
  perturbations exist in sufficient abundance; see Lemma
  \ref{LEM_tame_perturbations}.
Secondly, we also establish that in a very particular measure theoretic
  sense, tangent planes of pseudoholomorphic curves with small
  \(\omega\)-energy have tangent planes which are usually almost vertical;
  see Lemma \ref{LEM_Q_has_small_measure} below.
And finally, we establish the existence of a small geometric constant
  \(r_0\) which will be made use of extensively in Section
  \ref{SEC_core_proof_local_local}; see Lemma \ref{LEM_small_radius}.

\begin{definition}[$(\delta,\epsilon)$-tame perturbations]
  \label{DEF_epsilon_tame}
  \hfill\\
Let \((M, \eta)\) be a framed Hamiltonian manifold, \((J, g)\) be an
  \(\eta\)-adapted almost Hermitian structure on the symplectization
  \(\mathbb{R}\times M\), let \((\tilde{u}, \tilde{\jmath}, f, u, S, j)\)
  be a perturbed pseudoholomorphic map, and let \(\delta, \epsilon>0\).
We say \((\tilde{u}, \tilde{\jmath}, f, u, S, j)\) is a \((\delta,
  \epsilon)\)-tame perturbed pseudoholomorphic map provided the following
  hold, where ${\mathcal Z}=\{\zeta\in S\ |\ Tu(\zeta)=0\}$.
\begin{enumerate}[(d1)]                                                   
  \item\label{EN_d1} 
    \begin{equation*}                                                     
      \delta < {\textstyle \frac{1}{10}}\min\Big({\rm dist}_\gamma({\rm
      Crit}_{a\circ u}, \partial S) , \min_{\substack{\zeta_0, \zeta_1\in
      \mathcal{Z}\\\zeta_0\neq \zeta_1}}{\rm dist}_\gamma(\zeta_0,
      \zeta_1), \; {\rm dist}_\gamma(\mathcal{Z}, \partial S)\Big)
      \end{equation*}
  \item \label{EN_d2} 
    \(\epsilon<\min(2^{-24}, \frac{1}{1+C_B})\)
  \item\label{EN_d3} 
    the restricted map \(f:S\setminus\{\zeta\in S: {\rm
    dist}_\gamma(\zeta, \mathcal{Z})< \delta\}:\to \mathbb{R}\) is Morse
  \item\label{EN_d4} 
    \begin{equation*}                                                     
      \sup_{\zeta\in \Omega} |f(\zeta)|+\sup_{\zeta\in
      \Omega}\|df(\zeta)\|_{\gamma}+ \sup_{\zeta\in \Omega}\|\nabla
      df(\zeta)\|_{\gamma} \leq \frac{\epsilon}{2^{11}(1+C_B)}
      \end{equation*}
\end{enumerate}
  where \(\Omega:={\rm supp}(f)\), \({\rm Crit}_{a\circ u}\) is the set of
  critical points of \(a\circ u: S\to \mathbb{R}\), \(\gamma=u^*g\),
  \(\nabla\) is covariant differentiation with respect to the Levi-Civita
  connection associated to the metric \(\gamma\),  \(B_u\) is the second
  fundamental form associated to \(u\) as recalled in Definition
  \ref{DEF_second_fundamental_form} of Section \ref{SEC_riemannian}, and
  \begin{equation*}                                                       
    C_B:=\sup \{ \|B_u(\zeta)\|_{\gamma}: {\rm dist}_\gamma(\zeta,
    \mathcal{Z})\geq {\textstyle \frac{1}{2}}\delta \}.
    \end{equation*}
\end{definition}
%

\begin{remark}[feature of being $(\delta,\epsilon)$-tame]
  \label{REM_feature_of_de_perturbations}
  \hfill\\
A key feature of an \((\delta, \epsilon)\)-tame perturbed
  pseudoholomorphic map is that the \(f\) and \(\epsilon\) always satisfy
  the hypotheses of Lemma \ref{LEM_gamma_tilde_estimates}.
\end{remark}
%

\begin{lemma}[existence of tame perturbations]
  \label{LEM_tame_perturbations}
  \hfill\\
Let \((M, \eta)\) be a framed Hamiltonian manifold, let \((J, g)\) be
  an \(\eta\)-adapted almost Hermitian structure on the symplectization
  \(\mathbb{R}\times M\), let \((u, S, j)\) be a generally immersed
  pseudoholomorphic map, and let \(\delta>0\) satisfy
  \begin{equation*}                                                       
    \delta < {\textstyle \frac{1}{10}}\min\Big({\rm dist}_\gamma({\rm
    Crit}_{a\circ u}, \partial S) , \min_{\substack{\zeta_0, \zeta_1\in
    \mathcal{Z}\\\zeta_0\neq \zeta_1}}{\rm dist}_\gamma(\zeta_0,
    \zeta_1), \;  {\rm dist}_\gamma(\mathcal{Z}, \partial S)\Big).
    \end{equation*}
Then for each \(\epsilon>0\) satisfying \(\epsilon<\min(2^{-24},
  \frac{1}{1+C_B})\), where
  \begin{equation*}                                                       
    C_B:=\sup \big\{ \|B_u(\zeta)\|_{\gamma}: {\rm dist}_\gamma(\zeta,
    \mathcal{Z})\geq {\textstyle \frac{1}{2}}\delta \big\},
    \end{equation*}
  and \(B_u\) is the second fundamental form of \(u\), there
  exists a smooth map \(f:S\to \mathbb{R}\) for which \((\tilde{u},
  \tilde{\jmath}, f, u, S, j)\) is an \((\delta, \epsilon)\)-tame
  perturbed pseudoholomorphic map in the sense of Definition
  \ref{DEF_epsilon_tame}.
\end{lemma}
%
\begin{proof}
Let \(\epsilon'=\frac{\epsilon}{2^{11}(1 +C_B)}\), let \(h=a\circ u\),
  and apply Lemma \ref{LEM_f_perturbation} from Section 
  \ref{SEC_tame_perturbations}.
\end{proof}

In order to proceed with later proofs, we would like to establish
  that for any given feral curve, outside some large compact set, the curve
  is usually immersed, and the tangent planes are usually close to being
  parallel to \({\rm span}(\partial_a, X_\eta)\).
Our first pass at making this precise is Lemma
  \ref{LEM_Q_has_small_measure} below.
Here the idea is that a tangent plane at a point \(\zeta\) is close to
  being tangent to \({\rm span}(\partial_a, X_\eta)\) if and only if
  \(\|u^*\lambda_\zeta\|_{u^*g}\) is nearly \(1\).
Thus we are interested in the measure of the set of symplectization level
  sets on which there are not many points with \(\|u^*\lambda\|_{u^*g}
  <\theta \), for some specified value \(\theta\in (0,1)\).
Moreover, by ``not too many'' we mean that the Hausdorff \(1\)-measure

  \begin{align*}                                                            
    \mu_{u^*g}^1\big(\{\zeta \in (a\circ u)^{-1}(t): \|(u^* \lambda)_{\zeta}\|_{u^*
    g} < \theta \} \big)
    \end{align*}
  should be smaller than some specified number \(\delta>0\).
Finally, in general the measure of such level sets might of course be
  quite large, however Lemma \ref{LEM_Q_has_small_measure} below essentially
  states that it cannot be too large provided that the \(\omega\)-energy is
  rather small; or more precisely, that for fixed \(\delta\) and \(\theta\),
  the measure of such level sets is bounded in terms of the
  \(\omega\)-energy.
Thus for a feral curve, which has finite \(\omega\)-energy, it should
  follow that outside a large compact set, the curve is usually immersed
  with tangent planes usually close to being parallel to \({\rm
  span}(\partial_a, X_\eta)\).
This is now made precise with Lemma \ref{LEM_Q_has_small_measure} below.

\begin{lemma}[tangent planes usually near vertical]
  \label{LEM_Q_has_small_measure}
  \hfill\\
Let \((M, \eta)\) be a framed Hamiltonian manifold, and \((J, g)\) an
  \(\eta\)-adapted almost complex structure on \(\mathbb{R}\times M\).
Suppose further that \((u, S, j)\) is a compact pseudoholomorphic curve,
  possibly with boundary, with image in \(\mathbb{R}\times M\), which
  satisfies the following conditions.
\begin{enumerate}                                                         
  \item 
    \(\int_S u^*\omega \leq E_0<\infty\) 
  \item 
    \(\{\zeta\in S: d(a\circ u)(\zeta) = 0\}\cap \partial S = \emptyset\)
  \item 
    \(u(\partial S)\subset \{a_0, a_1\}\) and \(a_1 = \sup\{a\circ u
  (S)\}\) and \(a_0 = \inf\{a\circ u (S)\}\).
  \end{enumerate}   
With \(\mathcal{I}:=[a_0, a_1]\), \(\mathcal{R}_u\) defined to be
  the regular values of \(a\circ u\), and for each \(\theta\in (0, 1)\)
  and each \(\delta>0\), we define the following set
  \begin{equation*}                                                       
    \mathcal{Q}_{u, \theta, \delta}:=\Big\{t\in
    \mathcal{R}_u:\mu_{u^*g}^1\big(\{\zeta \in (a\circ u)^{-1}(t):
    \|(u^* \lambda)_{\zeta}\|_{u^* g} < \theta \} \big) > \delta \Big\}
    \end{equation*}
Then
  \begin{equation*}
    \mu(\mathcal{Q}_{u, \theta, \delta}) \leq
    \frac{E_0}{\delta(1-\theta^2)}< \infty.
    \end{equation*}
Here, \(\mu\) is the Lebesgue measure associated to the coordinate \(a\)
  on \(\mathbb{R}\).
\end{lemma}
%
\begin{proof}
For convenience, for each \(t\in \mathcal{R}_u\) and \(\theta\in (0,
  1)\) we will define
  \begin{equation*}                                                       
    \Gamma_t := \{\zeta\in S: a\circ u(\zeta) =t\} \qquad\text{and}\qquad
    S^\theta:=\{\zeta\in S: \|u^*\lambda_\zeta \| < \theta\} .
    \end{equation*}
Consequently, we may write
  \begin{equation*}                                                       
    \mathcal{Q}_{u, \theta, \delta}:=\big\{t\in
    \mathcal{R}_u:\mu_{u^*g}^1(\Gamma_t\cap S^\theta ) > \delta \big\}.
    \end{equation*}
Next we define the tangent vector fields \(\nu \) and \(\tau\) by
  \begin{equation*}                                                       
    \nu:=\frac{\nabla (a\circ u)}{\|\nabla (a\circ
    u)\|_\gamma}\qquad\text{and}\qquad \tau:= j\nu
    \end{equation*}
  where \(\gamma=u^*g\), and \(g=da\otimes da + \lambda\otimes\lambda +
  \omega(\cdot, J\cdot)\).
It is straightforward to verify the following properties,
  \begin{align*}                                                          
    0&=u^*\lambda(\nu) = u^*da(\tau)\\
    0&<u^*da(\nu)=u^*\lambda(\tau)\\
    1&=\|\tau\|_{\gamma}^2 = \|\nu\|_{\gamma}^2,
    \end{align*}
  from which one can deduce that
  \begin{equation*}                                                       
    0<da(Tu\cdot \nu) =\|\nabla (a\circ u)\|_{u^*g} \leq 1,\qquad
    \|u^*\lambda\|_{u^*g}=\lambda(Tu\cdot \tau)
    \end{equation*}
  and
  \begin{equation*}                                                       
    1 = (\lambda(Tu\cdot \tau))^2 + \omega(Tu\cdot \nu, Tu\cdot \tau).
    \end{equation*}
From these we may estimate the measure of \(\mathcal{Q}_{u, \theta,
   \delta}\) as follows.
\begin{align*}                                                            
  \mu(\mathcal{Q}_{u, \theta, \delta}) &= \int_{\mathcal{Q}_{u, \theta,
    \delta}} 1\; dt\\
  &=
    \int_{\{t\in \mathcal{R}_u:\mu_{u^*g}^1(\Gamma_t\cap S^\theta ) >
    \delta\}} 1\; dt\\
  &=
    \delta^{-1}\int_{\{t\in \mathcal{R}_u:\mu_{u^*g}^1(\Gamma_t\cap
    S^\theta ) > \delta\}} \delta\; dt\\
  &\leq 
    \delta^{-1}\int_{\{t\in \mathcal{R}_u: \mu_{u^*g}^1(\Gamma_t\cap
    S^\theta) >\delta \}} \big(\mu_{u^*g}^1(\Gamma_t\cap S^\theta)\big)\;
    dt\\
  &\leq 
    \delta^{-1}\int_{\mathcal{I}}\big(\mu_{u^*g}^1(\Gamma_t\cap
    S^\theta)\big)\; dt\\
  &=
    \delta^{-1}\int_{\mathcal{I}} \int_{ \Gamma_t \cap S^\theta} 1\;
    d\mu_{u^*g}^1 \;dt\\
  &=
    (\delta(1-\theta^2))^{-1}\int_{\mathcal{I}} \int_{ \Gamma_t \cap
    S^\theta} (1-\theta^2)\; d\mu_{u^*g}^1 \;dt\\
  &\leq 
    (\delta(1-\theta^2))^{-1}\int_{\mathcal{I}} \int_{ \Gamma_t \cap
    S^\theta} \frac{1-\theta^2}{da(Tu\cdot \nu)}\; d\mu_{u^*g}^1 \;dt\\
  &\leq 
    (\delta(1-\theta^2))^{-1}\int_{\mathcal{I}} \int_{ \Gamma_t \cap
    S^\theta} \frac{1-\|u^*\lambda\|_{u^*g}^2}{da(Tu\cdot \nu)}\;
    d\mu_{u^*g}^1 \;dt\\
  &= 
    (\delta(1-\theta^2))^{-1}\int_{\mathcal{I}} \int_{ \Gamma_t \cap
    S^\theta} \frac{1-(\lambda(Tu\cdot \tau))^2}{da(Tu\cdot \nu)}\;
    d\mu_{u^*g}^1 \;dt\\
  &= 
    (\delta(1-\theta^2))^{-1}\int_{\mathcal{I}} \int_{ \Gamma_t \cap
    S^\theta} \frac{\omega(Tu\cdot \nu, Tu\cdot \tau)}{da(Tu\cdot \nu)}\;
    d\mu_{u^*g}^1 \;dt\\
  &= 
    (\delta(1-\theta^2))^{-1}\int_{\mathcal{I}} \int_{ \Gamma_t \cap
    S^\theta} \frac{\omega(Tu\cdot \nu, Tu\cdot \tau)}{\|\nabla(a\circ
    u)\|_{u^*g}}\; d\mu_{u^*g}^1 \;dt\\
  &\leq 
    (\delta(1-\theta^2))^{-1}\int_{\mathcal{I}} \int_{ \Gamma_t
    } \frac{\omega(Tu\cdot \nu, Tu\cdot \tau)}{\|\nabla(a\circ
    u)\|_{u^*g}}\; d\mu_{u^*g}^1 \;dt\\
  &= 
    (\delta(1-\theta^2))^{-1}\int_S \omega(Tu\cdot \tau, Tu\cdot \nu)
    d\mu_{u^*g}^2 \\
  &= 
    (\delta(1-\theta^2))^{-1}\int_S u^*\omega \\
  &\leq 
    \frac{E_0}{\delta(1-\theta^2)},
  \end{align*}
  where to achieve the second to last equality we have made use of
  Proposition \ref{PROP_coarea_body} with \(f=\omega(Tu\cdot \nu,
  Tu\cdot \tau)\).
\end{proof}

In order to state Proposition \ref{PROP_local_local_area_bound_2} below
   concisely, it will be useful to have the following definition at our
   disposal.

\begin{definition}[connected component $S_{\rho}(\zeta)$]
  \label{DEF_local_local_component}
  \hfill\\
Let \((W, g)\) be a Riemannian manifold with bounded
  geometry,\footnote{Recall that a Riemannian manifold is said to have
  bounded geometry provided the sectional curvature is uniformly bounded
  from above and below and the injectivity radius of the manifold is
  positive.}, let \(S\) be manifold, and let \(u:S\to W\) be a smooth map.
For each \(\rho>0\) and each \(\zeta\in S\), we
  define \(S_\rho(\zeta)\) to be the connected component of
  \(u^{-1}\big(\mathcal{B}_\rho(u(\zeta))\big)\) containing \(\zeta\);
  here for each \(p\in W\), the set  \(\mathcal{B}_\rho(p)\subset W\)
  is the metric ball of radius \(\rho\) centered at \(p\).
\end{definition}
%

Before proceeding, we need an additional geometric constant, namely
  \(r_0\), the existence of which is guaranteed by the following lemma.

\begin{lemma}[small radius $r_0$]
  \label{LEM_small_radius}
  \hfill\\
Let \((M, g)\) be a smooth Riemannian manifold of bounded geometry, and
  let \(\lambda\in \Omega^1(M)\) be a smooth one-form with the property
  that for each point \(p\in M\) we have
  \begin{equation*}                                                       
    \sup_{0\neq \tau\in T_p M} \frac{\lambda(\tau)}{\|\tau\|_g}=1
    \end{equation*}
Then there exists a positive real number \(r_0=r_0(M, g, \lambda)\leq
  \frac{1}{100}\) with the following significance.
For each smooth unit speed immersion \(\tilde{q}:[0, T]\to M\) which
  satisfies the following conditions
  \begin{enumerate}                                                       
    \item 
      \(\lambda(\tilde{q}'(t))>0\)
    \item 
      \(r_0\leq \int_{\tilde{q}} \lambda \leq 10 r_0\)
    \item  
      \(\mu_{\tilde{q}^*g}^1(\{t\in [0, T]:
      \lambda(\tilde{q}'(t))<\frac{1}{2}\}) \leq  r_0\)
    \end{enumerate}
  also satisfies
  \begin{equation*}                                                       
    {\rm dist}_g\big(\tilde{q}(0), \tilde{q}(T)\big) \geq
    {\textstyle\frac{1}{2}}r_0.
    \end{equation*}
\end{lemma}
%
\begin{proof}
We begin by letting \(\rho= \min(1, {\rm inj}(M))\) where \({\rm
  inj}(M)\) is  the injectivity radius of \(M\) with respect to \(g\).
For each \(p\in M\) we let \(\mathcal{B}_\rho(p)\) denote the metric
  ball of radius \(\rho\) centered at \(p\).
Recall that for each point \(p\in M\) and each orthonormal
  frame for \(T_pM\) one may define geodesic normal coordinates on
  \(\mathcal{B}_\rho(p)\) which are centered at \(p\).
We will denote such coordinates as \(x=(x^1, \ldots, x^m)\),
  in which case we can express the metric as \(g=\sum_{i,
  j=1}^m g_{ij}(x)dx^i\otimes dx^j\), and our one-form as
  \(\lambda=\sum_{i=1}^m\lambda_i(x)dx^i\).
Note that for each \(p\in M\) there exists an orthonormal frame of \(T_p
  M\) such that the associated geodesic normal coordinates have the property
  that \(\lambda_{i}(p) = \delta_{1,i}dx^1\), where \(\delta_{1,i}\)
  is the Kronecker delta.
We then fix \(r_0\in (0, \frac{\rho}{100})\) sufficiently small so that
  for any \(p\in M\) and any such orthonormal frame, we have
  \begin{equation}\label{EQ_distance_est}                                 
    \sup_{y\in \mathcal{B}_{100r_0}(p)} \|dx_y^1 - \lambda_y\|_g\leq
    \frac{1}{100}.
    \end{equation}
Next we consider an immersion \(\tilde{q}:[0, T]\to M\) which satisfies
  the above hypotheses of the lemma.
We now need to estimate the length of the path \(\tilde{q}\) in terms of \(r_0\).
To that end, we have:
  \begin{align*}                                                          
  {\rm length}(\tilde{q})&=T\\
    &=
      \int_0^T \|\tilde{q}'(t)\|_g \; dt\\
    &=
      \int_{\{t\in [0, T]: \lambda(\tilde{q}'(t))<
      {\frac{1}{2}}\}}\|\tilde{q}'(t)\|_g\; dt +
      \int_{\{t\in [0, T]: 1\geq \lambda(\tilde{q}'(t))\geq
      {\frac{1}{2}}\}}\|\tilde{q}'(t)\|_g\; dt\\
    &=\mu_{\tilde{q}^*g}^1\big(\{t\in [0, T]: \lambda(\tilde{q}'(t))<
      {\textstyle \frac{1}{2}}\}\big) + \int_{\{t\in [0, T]: 1\geq
      \lambda(\tilde{q}'(t))\geq \frac{1}{2}\}}1 \; dt\\
    &\leq 
      r_0 +2 \int_{\{t\in [0, T]: 1\geq
      \lambda(\tilde{q}'(t))\geq\frac{1}{2}\}}\lambda(\tilde{q}'(t))
      \; dt\\
    &\leq 
      r_0 +2\int_0^T\lambda(\tilde{q}'(t)) \; dt\\
    &\leq   
      21r_0.
    \end{align*}
As a consequence of this estimate, we see that the image of \(\tilde{q}\)
  is contained in the metric ball of radius \(100r_0\) centered at
  \(\tilde{q}(0)\), and hence inequality (\ref{EQ_distance_est}) holds
  along the image of \(\tilde{q}\).
As such, we take geodesic normal coordinates centered at
  \(p:=\tilde{q}(0)\) as above so that \(\lambda_{i}(p) =
  \delta_{1,i}dx^1\), and we estimate as follows.
  \begin{align*}                                                          
    x^1\big(\tilde{q}(T)\big) 
    &= 
      \int_0^T \frac{d}{dt}\big(x^1(\tilde{q}(t)\big)\; dt\\
    &= 
      \int_0^T dx^1(\tilde{q}'(t))\; dt\\
    &= 
      \int_0^T \lambda(\tilde{q}'(t))\; dt +  \int_0^T \big(dx^1
      -\lambda\big)(\tilde{q}'(t))\; dt\\
    &\geq 
      \int_{\tilde{q}} \lambda - {\textstyle \frac{1}{100}}T\\
    &\geq 
      r_0 - {\textstyle\frac{21}{100}} r_0\\
    &\geq 
      {\textstyle\frac{1}{2}}r_0.
    \end{align*}
Since 
  \begin{align*}                                                          
    {\rm dist}_{g}\big(\tilde{q}(T), \tilde{q}(0)\big) 
    &= 
      \Big(\sum_{i=1}^m\big(x^i(\tilde{q}(T))\big)^2\Big)^{\frac{1}{2}}\\
    &\geq 
      x^1(\tilde{q}(T))\\
    &\geq 
      {\textstyle\frac{1}{2}} r_0,
    \end{align*}
  the desired result is immediate. 
\end{proof}

\subsubsection{The core proof}
  \label{SEC_core_proof_local_local}
Here we provide the complete proof of Theorem
  \ref{THM_local_local_area_bound}, which essentially states that for each
  generally immersed feral pseudoholomorphic curve, there exists a large
  compact set in the symplectization with the property that outside this
  compact set, the curve has uniformly bounded connected-local area.
That is, in a small\footnote{Here by ``small'' we mean small relative to
  the geometry of the ambient manifold and not small relative to the curve
  itself.} ball the area of each connected component of the portion of the
  curve contained in the ball has universally bounded area.
The first and most technical step towards proving Theorem
  \ref{THM_local_local_area_bound} is to prove Proposition
  \ref{PROP_local_local_area_bound_2}, which is a special case.
We accomplish this at present.

\begin{proposition}[connected-local area bound -- special
  case]
  \label{PROP_local_local_area_bound_2}
  \hfill\\
Let \((M, \eta=(\lambda, \omega))\) be a framed Hamiltonian manifold, and
  let \((J, g)\) be an \(\eta\)-adapted almost Hermitian structure on the
  symplectization \(\mathbb{R}\times M\).
Let \(r_0\leq \frac{1}{100}\) be the positive constant associated to
  \((M, g, \lambda)\) which is guaranteed by Lemma \ref{LEM_small_radius}.
Let \(r_1= 2^{-24} \min\big( C_{\mathbf{h}}^{-1} , r_0\big)
  \), where \(C_{\mathbf{h}}\) is the ambient geometry constant established
  in Definition \ref{DEF_ambient_geometry_constant}.
For each generally immersed pseudoholomorphic map \((u, S, j)\) satisfying
  the following conditions 
  \begin{enumerate}[(LL1)]                                                
    \item\label{EN_LL1} 
      \(S\) is homeomorphic to an annulus
    \item\label{EN_LL2} \(a\circ u(\partial S)=\{a_0, a_1\}\) with
      \(2^{-14} \min\big( C_{\mathbf{h}}^{-1} , r_0\big)\leq     a_1-a_0\)
    \item\label{EN_LL3} 
      \(\big\{\zeta\in S:a\circ u(\zeta)\in \{a_0, a_1\}\text{ and }d
      (a\circ u)(\zeta)=0\big\}= \emptyset\)
    \item\label{EN_LL4} 
      \(\sup_{\zeta\in S} a\circ u(\zeta)- \inf_{\zeta\in S}a\circ
      u(\zeta)\leq 2^{-11} \min\big(C_{\mathbf{h}}^{-1}, r_0\big)  \)
    \item\label{EN_LL5} 
      \(0< \int_{S} u^*\omega\leq
      r_0\big((a_1-a_0)^{-1}+10C_{\mathbf{h}}\big)^{-1} \)
    \item\label{EN_LL6} 
      \(\int_{(a\circ u)^{-1}(a_0) \cap\partial S}u^*\lambda \geq
      100r_0\)
    \item\label{EN_LL7}  
      \(\mu_{u^*g}^1\big(\{\zeta \in \partial S: a\circ u(\zeta)=a_0 \;
      \text{ and }\; \|(u^* \lambda)_{\zeta}\|_{u^* g} < {\textstyle\frac{1}{2}}\}
      ) \leq r_0 \),     
    \end{enumerate}
  also has the following property:   
For each \(\zeta\in S\) with 
  \begin{equation*}                                                       
    \big|a\circ u (\zeta) - {\textstyle \frac{1}{2}}(a_1+a_0)\big|
    \leq {\textstyle\frac{1}{4}}(a_1-a_0),
    \end{equation*}
  we also have
  \begin{equation}\label{EQ_key_area_estimate}
    {\rm Area}_{u^*g}\big(S_{r_1}(\zeta)\big) \leq 1.
    \end{equation} 
Here, as above, \(\mu_{u^*g}^1\) is the one-dimensional Hausdorff measure
  associated to the metric \(u^*g\).
\end{proposition}
%

We note that the statement and proof of Proposition
  \ref{PROP_local_local_area_bound_2} are each rather long, so we take a
  moment to clarify the former and outline the latter.
First, the hypotheses require that we are dealing with a compact
  pseudoholomorphic curve, homeomorphic to an annulus, with a ``top''
  boundary at the symplectization level set \(\{a_1\}\times M\), and
  ``bottom'' boundary at the symplectization level set \(\{a_0\}\times
  M\).
We allow that the interior points of the curve may lie either above
  \(\{a_1\}\times M\) or below \(\{a_0\}\times M\), however we demand that
  both \(a_0\) and \(a_1\) be regular values of \(a\circ u\), and if we
  define the ad hoc constant
  \begin{align*}                                                            
    C:= 2^{-11}{\rm min}(C_{\mathbf{h}}^{-1}, r_0)
    \end{align*}
  then we require 
  \begin{align*}                                                            
    \frac{1}{8} C \leq a_1 - a_0 \leq \sup_{\zeta\in S} a\circ u(\zeta)-
    \inf_{\zeta\in S}a\circ u(\zeta)\leq C.
    \end{align*}
Roughly then, both the height difference between the boundaries and the
  height difference between the absolute peak and absolute valley can
  neither be too large nor too small.
We also demand that the \(\omega\)-energy be rather small, the
  \(\lambda\)-integral along the bottom boundary be rather large, and the
  measure of those points in the bottom boundary for which the tangent
  planes are not close to \({\rm span}(\partial_a, X_\eta)\) is rather
  small.
After imposing all of these conditions, Proposition
  \ref{PROP_local_local_area_bound_2} then guarantees that the area of a
  connected component of the portion of the curve that lives in a ball of
  radius \(r_1\) which is centered near \(\frac{1}{2}(a_1+a_0)\) is
  uniformly bounded; indeed, the bound is simply \(1\).

Before outlining the proof, it is natural to ask how one is likely to find
  a curve which satisfies these conditions, so we sketch a candidate
  example.
Indeed, consider a feral curve which, for example, has an absolute minimum
  and no maximum, so it extends to \(\{+\infty\}\times M\).
For simplicity, we assume that for some sufficiently large and generic
  \(a_2\in \mathbb{R}\), the set \((a\circ u)^{-1}((a_2, \infty)\times M)\)
  is diffeomorphic to a cylinder \(\mathbb{R}\times S^1\).
Note that this simplifying condition is essentially what makes Proposition
  \ref{PROP_local_local_area_bound_2} only a \emph{special case} of
  Theorem \ref{THM_local_local_area_bound}. 
Given such a curve, one then considers values \(a_0\) and \(a_1\) which
  are very large, and which are regular values of \(a\circ u\).
One then defines a compact annular curve by restricting the domain of the
  feral curve to \((a\circ u)^{-1}([a_0, a_1]\times M)\), and then capping
  off excess boundary components with appropriate disks.\footnote{That
  such a capping procedure is possible is established later when needed.}
For this resulting curve, we see that by making \(a_0\) sufficiently
  large, condition (LL\ref{EN_LL5}) must be satisfied because feral curves
  have finite \(\omega\)-energy.
Condition (LL\ref{EN_LL6}) follows essentially because if the
  \(\lambda\)-integral along the bottom boundary did not get arbitrarily
  large, our feral curve would have finite Hofer-energy; thus we assume
  this is not the case so we must be able to find many large \(a_0\) for
  which condition (LL\ref{EN_LL6}) is satisfied.
One then finds \(a_0\) for which condition (LL\ref{EN_LL7}) holds by a
  judicious application of Lemma \ref{LEM_Q_has_small_measure}.

We now turn our attention to sketching the proof of Proposition
  \ref{PROP_local_local_area_bound_2}.
As a preliminary step, we give the overarching idea which motivates the
  proof.
Namely, the key conclusion is that the quantity \({\rm
  Area}_{u^*g}\big(S_{r_1}(\zeta)\big)\) is bounded by \emph{some} large
  universal constant. 
That this constant is \(1\) instead of \(10^{(10^{10})}\) is essentially
  irrelevant.
Also irrelevant to the main thrust of the argument is the fact that we
  have explicitly specified \(r_1\) in terms of geometric constants, and
  we have bound the \(\omega\)-energy \(\int_S u^*\omega\) in terms of
  geometric data.
Instead, the key idea is to consider the case that there exists a
  sequence\footnote{The subscripts denoting the index of the term in the
  sequence has been suppressed for notational clarity.} of such
  pseudoholomorphic annuli with the property that as one progresses
  through the sequence, one can find a point \(\zeta\) not near the
  boundary of the curve such that \({\rm
  Area}_{u^*g}\big(S_{r_1}(\zeta)\big)\to \infty\) while \(r_1\to 0\) and
  \(\int_S u^*\omega\to 0\).
For example, given a feral pseudoholomorphic curve, which necessarily has
  finite \(\omega\)-energy, suppose that no matter how small one fixes a
  radius \(r\), and no matter how large one fixes \(A\in \mathbb{R}^+\),
  one can always find a point \(\zeta\in S\) so that \(a\circ u(\zeta)
  \geq A\) and the area of the connected component of
  \(u^{-1}(B_r(u(\zeta)))\), that contains \(\zeta\), is as large as we
  like while the \(\omega\)-energy is as small as we like.
One then aims to derive a contradiction by finding a region of \(S\) that
  contains \(S_r(\zeta)\) but which has bounded area.
Indeed, much of the proof is focused on finding this region of
  \(S\) which provides the desired contradiction.
That we can specify certain quantities, like \(r_1\), the area bound, etc,
  in terms of geometric constants simply follows from taking some extra care
  with our estimates.

Let us now turn our attention to describing that region in \(S\) that
  contains \(S_{r_1}(\zeta)\), but which has the desired area bound.
As a first step, we impose some drastic simplifying assumptions to get at
  the core argument.
In particular, we begin by assuming that on our pseudoholomorphic annulus,
  there are no critical points of the function \(a\circ u\).
We weaken this assumption in a moment, however in this simplified case, we
  observe that every gradient trajectory of \(a\circ u\) initiating
  in \(\partial_0^- S\) will terminate in \(\partial_0^+ S\).
Geometrically then, all gradient flow lines extend from the bottom
  boundary to the top boundary of our annulus, without getting trapped at
  critical points.
We then consider a compact interval \(\mathcal{I}\subset \partial_0^- S\)
  which has small \(\tilde{\alpha}\)-measure; that is, suppose
  \(\int_{\mathcal{I}}\tilde{\alpha} = \int_{\mathcal{I}} u^*\lambda\) is
  small.
Then consider the pseudoholomorphic strip \(u\colon \Sigma \to
  \mathbb{R}\times M\) determined by \(\mathcal{I}\); that is, with
  \(\partial_0^- \Sigma = \mathcal{I}\), \(\partial_0^+ \Sigma\subset
  \partial_0^+S\), and with the other portions of the boundary given as
  gradient flow lines.
In this case, we have
  \begin{align*}                                                          
    {\rm Area}_{u^*g}(\Sigma) 
    &= 
    \int_{\Sigma} u^*da \wedge \tilde{\alpha} + u^*\omega
    \\
    &=\int_{\Sigma} d\big((a\circ u - a_1) \tilde{\alpha}\big) -
    \int_{\Sigma} (a\circ u - a_1) d\tilde{\alpha}+ \int_{\Sigma}
    u^*\omega
    \\
    &=(a_1-a_0)\int_{\mathcal{I}}\tilde{\alpha} -
    \int_{\Sigma} (a\circ u - a_1) d\tilde{\alpha}+ \int_{\Sigma}
    u^*\omega\\
    &\leq
    (a_1-a_0)\int_{\mathcal{I}}\tilde{\alpha} +
    (a_1-a_0)\|d\tilde{\alpha}\|{\rm Area}_{u^*g}(\Sigma)+ \int_{\Sigma}
    u^*\omega.
    \end{align*}
At this point, invoke Lemma \ref{LEM_d_alpha_tilde_bound} which bounds
  \(\|d\tilde{\alpha}\|\) in terms of the ambient geometry constant, and
  note that \(a_1-a_0\) is small, so that we obtain an estimate of the
  form
  \begin{align*}                                                          
    \frac{1}{2}{\rm Area}_{u^*g}(\Sigma) \leq (a_1-a_0)
    \int_{\mathcal{I}}\tilde{\alpha} + \int_\Sigma u^*\omega.
    \end{align*}
Here we recall that \((a_1-a_0)\) is small by our hypotheses, and so is 
  \(\int_{\Sigma} u^*\omega\), and \(\int_{\mathcal{I}}\tilde{\alpha}\) is
  small by assumption.
Thus the area of \(\Sigma\) is bounded, and the region is determined
  simply by choosing an interval \(\mathcal{I}\subset \partial_0^-S\).
The goal then becomes to show that \(S_{r_1}(\zeta)\subset \Sigma\) for
  some choice of \(\mathcal{I}\), which would essentially yield the
  desired bound.
We almost do this.
Instead, we partition \(\partial_0^-S\) into a bunch of small intervals
  so that \(\zeta\) is contained in exactly one of the corresponding strips
  \(\Sigma\).
We then show that \(S_{r_1}(\zeta)\) cannot intersect both gradient-flow
  boundary portions of any strip \(\Sigma\) associated to our partition.
This is achieved by a simple geodesic distance argument combined with
  Lemma \ref{LEM_small_radius}.
With this established, it then follows that \(S_{r_1}(\zeta)\) is
  contained in the union of three consecutive strips, and fails to have
  non-trivial intersection with the outer-most gradient-flow boundary
  portions.
Our previous area estimate applies, but in triple, and this is sufficient
  to obtain the desired area bound on \(S_{r_1}(\zeta)\).

Of course more generally, \(a\circ u\) may indeed have critical points, so
  we next consider the case that \(a\circ u\) is a Morse function.
In this case, we note that gradient trajectories that initiate at points
  in \(\partial_0^-S\) now terminate at either points in
  \(\partial_0^+S\), or else in critical points of \(a\circ u\) of Morse
  index either 1 or 2.
Note that there are only finitely many points in \(\partial_0^-S\) with
  gradient flow lines that limit to critical points of Morse index 1, but
  potentially a continuum which limit to critical points of Morse index 2.
Thus the goal is to show that the set of such points has small
  \(\tilde{\alpha}\) measure.
This follows essentially from Lemma \ref{LEM_general_strip_estimate},
  which guarantees that if the \(\tilde{\alpha}\)-measure of such points
  were not small, then neither would the \(\omega\)-energy, which in fact
  is small.
Knowing that in an \(\tilde{\alpha}\)-measure theoretic sense, most points
  in \(\partial_0^-S\) have gradient flow lines that limit to points in
  \(\partial_0^+S\), we can adapt our aforementioned argument to achieve
  the desired area bound.

Next we note that \(a\circ u\) need not be a Morse function, however if
  \(u\) is an immersion, then we may find a perturbed
  patch of pseudoholomorphic curve via a small perturbing
  function \(f\), so that for the resulting perturbed curve
  \((\tilde{u},\tilde{\jmath}, f, u, S, j)\), the function \(a\circ
  \tilde{u}\) is indeed Morse, and the previous arguments essentially
  hold.
Indeed, for \(f\) chosen suitably small enough, the \(\tilde{u}^*g\)-area
  bound on the patch yields the desired \(u^*g\)-area
  bound on that same patch, which is the goal.

Finally, we must worry about the case that \(u\) is not immersed. 
Unfortunately, our method to perturb the curve does not handle
  non-immersed points. 
However the assumptions of Theorem \ref{THM_local_local_area_bound}
  guarantee that there are only finitely many interior non-immersed points,
  and none on the boundary.
Thus our goal will be to first find certain small neighborhoods of the
  non-immersed points \(\mathcal{Z}\), and perturb the curve on the
  compliment of these neighborhoods so that in the larger region
  \(a\circ\tilde{u}\) is Morse.
We then show that the set of points in \(\partial_0^-S\) which are initial
  points of gradient flow lines of the function \(a\circ \tilde{u}\) which
  pass into the small neighborhood of \(\mathcal{Z}\) has
  \(\tilde{\alpha}\)-measure which is controlled by the \(\omega\)-energy of
  the given pseudoholomorphic annulus, and hence essentially small.
This procedure is similar to showing that the set of points limiting to
  local maxima has small \(\tilde{\alpha}\)-measure, but a touch more
  complicated.

It is perhaps worth mentioning that only at this point does it make
  considerable sense to have established estimates so explicitly in terms
  of geometric and universal constants.
The issue is that in order to establish Theorem
  \ref{THM_local_local_area_bound}, we must guarantee that if the
  \(\omega\)-energy of a curve is small and the \(\tilde{\alpha}\)-measure
  of the bottom boundary \(\partial_0^- S\) is large, then, in an
  \(\tilde{\alpha}\)-measure theoretic sense, most gradient trajectories
  starting in the bottom boundary \(\partial_0^- S\) end in the top boundary
  \(\partial_0^+S \).
Of course, this need not be true if \(a\circ u\) is not Morse, so we need
  to perturb our curve and we needed to show that our area and strip
  estimates hold for the perturbed curve; for example, this motivates
  Theorem \ref{THM_area_bound_estimate}, Lemma
  \ref{LEM_general_strip_estimate}, and Lemma \ref{LEM_lambda_shrinkage} to
  be established for \emph{perturbed} curves, instead of simply
  pseudoholomorphic curves.
Moreover, because our perturbation method does not extend across
  non-immersed points, there will be small regions of unperturbed curve
  and we need to establish that most gradient trajectories avoid these
  regions.
A complication however is that as we make the neighborhood of the critical
  points smaller, our perturbing function \(f\) must also change, which in
  turn changes which gradient trajectories enter the neighborhood.
Worse still, because the curvature of a curve may be unbounded in a
  neighborhood of a non-immersed point, we see that the \(\mathcal{C}^2\)
  norm of the function \(f\) must be made smaller as we shrink the
  neighborhood.
Consequently, less exacting care in obtaining estimates can easily lead to
  circular reasoning: the size of the neighborhood of the non-immersed
  points depends on the gradient flow lines, which depend on the perturbing
  function \(f\), which depends on the size of the neighborhood.
To avoid this circular logic, we are careful throughout this manuscript to
  make estimates and inequalities in terms of universal and geometric
  constants.
Consequently, our choice of neighborhood depends only on geometric
  constants associated to either the curve itself or the ambient geometry,
  and our perturbing function then depends upon the neighborhood, but not
  the other way around.
The upshot is that we avoid circular dependence, but the downside is the
  seemingly pedantic focus on precision\footnote{Of course, ``precision''
  is in the eye of the beholder, since many estimates can be substantially
  sharpened.  Indeed, we have made little effort to distinguish
  \(\frac{1}{2}\) from \(2^{-10}\) etc, however such additional precision
  seems to have little utility in regards to our results.} in the obtained
  inequalities.

Although the above sketch accurately characterizes the proof of
  Theorem \ref{THM_local_local_area_bound}, the actual proof will be
  implemented in somewhat reverse order.
Specifically, as follows:
  \begin{enumerate}[\hspace{30pt}Step 1.]                                 
    \item 
    Carefully define the neighborhoods of the non-immersed points, and
    then define an appropriate perturbation of the curve.
    \item 
    Show that the \(\tilde{\alpha}\)-measure of the initial points of
    gradient flow lines in \(\partial_0^- S\) which enter into the
    neighborhood of the non-immersed points is bounded in terms of an
    ambient geometry constant and the \(\omega\)-energy.
    \item 
    Show that the \(\tilde{\alpha}\)-measure of the initial points of
    gradient flow lines in \(\partial_0^- S\) which limit to local maxima
    of \(a\circ \tilde{u}\) is bounded in terms of an
    ambient geometry constant and the \(\omega\)-energy.
    \item 
    Approximate the set of points in \(\partial_0^-S\), that limit to
    points in \(\partial_0^+S\), from the inside by finitely many compact
    pairwise disjoint intervals.
    \item 
    Construct the desired partition, and associated patches of our curve.
    \item 
    Estimate the area of each of these patches, show that
    \(S_{r_1}(\zeta)\) is contained in the union of a consecutive triple
    of patches, and complete the proof.
    \end{enumerate}

This completes the outline of the proof, so that finally we turn our
  attention toward the actual proof of Proposition
  \ref{PROP_local_local_area_bound_2}.

\begin{proof}[Proof of Proposition \ref{PROP_local_local_area_bound_2}]\hfill\\

\noindent{\bf Step 1.}

We begin by recalling a previously used notation.
  \begin{align*}                                                          
    \partial_0^- S:= (\partial S)\cap (a\circ u)^{-1}(a_0)\\
    \partial_0^+ S:= (\partial S)\cap (a\circ u)^{-1}(a_1)
    \end{align*}
Next, we let \(\mathcal{Z}\subset S\setminus \partial S\) denote the
  set of non-immersed points of \(u\).
Recall that a consequence of \((u, S, j)\) being generally immersed is
  that \(\mathcal{Z}\) is finite.
It will also be convenient to fix \(\delta_7>0\) sufficiently small
  so that
  \begin{equation*}                                                       
    \delta_7<2^{-24} \min\Big({\rm dist}_{\gamma}(\partial S,
    \mathcal{Z}), \min_{\substack{z, z' \in \mathcal{Z}\\ z\neq
    z'}}\big({\rm dist}_\gamma (z, z')\big) , \; {\rm
    dist}_\gamma\big({\rm Crit}_{a\circ u}, \partial S\big)\Big).
    \end{equation*}
Here $\gamma=u^\ast g$.
It is well known (see for example Lemma 2.9 of \cite{Fish2})
  that for each \(z\in \mathcal{Z}\) there exists a local
  holomorphic chart \(\phi_z:\mathcal{O}(z)\to \mathcal{O}(0)\subset
  \mathbb{C}\simeq \mathbb{R}^2\), and geodesic normal coordinates
  \(\Phi_z:\mathcal{O}(u(z))\to \mathcal{O}(0)\subset \mathbb{C}^m\simeq
  \mathbb{R}^{2m}\), and \(2\leq k_z\in \mathbb{N}\), such that
  \(\phi_z(z)=0\), \(\Phi_z(u(z))=0\), and
  \begin{equation*}                                                       
    \Phi_z\circ u\circ \phi_z^{-1}(w) = (w^{k_z}, 0, \ldots, 0) + F_z(w)
    \end{equation*} 
  where \(F_z(w)=O(|w|^{k_z+1})\) and \(dF_z(w) = O(|w|^{k_z})\). 
For each \(z\in \mathcal{Z}\) we may locally define the function
  \(r_z:\mathcal{O}(z)\to \mathbb{R}\) by the following: \(r_z\circ
  \phi_z^{-1}(w) = |w|^{k_z}\), which is smooth everywhere it is defined,
  except possibly at \(w=0\) where it only must be continuous (or more
  specifically, \(\mathcal{C}^{0, \frac{1}{2}}\)).
For each \(z\in \mathcal{Z}\) we then define the sets 
  \begin{equation*}                                                       
    \mathcal{V}_z:=\{\zeta\in \mathcal{O}(z): r_z(\zeta)< \delta_6\}
    \end{equation*}
  where we have assumed that \(\delta_6>0\) is sufficiently small so
  that for each \(0<\delta\leq \delta_6<\delta_7\) we have 
  \begin{enumerate}                                                       
    \item 
      \(\{\zeta\in {\mathcal V}_z: r_z(\zeta)=\delta\}\cong S^1\)
    \item  
      \(\{\zeta\in {\mathcal V}_z: r_z(\zeta)<\delta\}\cap \partial
      S=\emptyset\)
    \item 
      \({\rm length}_\gamma\big( \{\zeta\in {\mathcal V}_z:
      r_z(\zeta)=\delta\}\big)
      \leq 4\pi \delta k_z\)
    \item 
      \(\mathcal{V}_z \cap \mathcal{V}_{z'} = \emptyset\) for each \(z,
      z'\in \mathcal{Z}\) with \(z\neq z'\)
    \item 
      \(\delta \,(\sum_{z\in \mathcal{Z}}k_z) \leq
      \frac{C_{\mathbf{h}}}{8\pi} \int_{S}u^*\omega \).
    \end{enumerate}
For ease of notation, we now define
  \begin{equation*}                                                       
    \mathcal{V}=\bigcup_{z\in \mathcal{Z}} \mathcal{V}_z
    \end{equation*}
   and we define the function
  \begin{equation*}                                                       
    r:\mathcal{V}\to \mathbb{R}\qquad\text{by}\qquad
    r\big|_{\mathcal{V}_z} = r_z.
    \end{equation*}
We now fix \(\delta>0\) so that
  \begin{equation*}                                                       
    \delta < {\textstyle \frac{1}{10}}\min\Big({\rm dist}_\gamma({\rm
    Crit}_{a\circ u}, \partial S) , \min_{\substack{z_0, z_1\in
    \mathcal{Z}\\z_0\neq z_1}}{\rm dist}_\gamma(z_0,
    z_1), \delta_6 \Big),
    \end{equation*}
  and
  \begin{equation}\label{EQ_containment}                                  
    \{\zeta \in S: {\rm dist}_\gamma(\zeta, \mathcal{Z}) \leq \delta
    \}\subset \bigcup_{z\in \mathcal{Z}}\{\zeta\in \mathcal{O}(z):
    r_z(\zeta)< \textstyle{\frac{1}{10}} \delta_6\}.
    \end{equation}
We fix \(\epsilon>0\) sufficiently small so that 
  \begin{equation}\label{EQ_epsilon_suff_small}                           
    \epsilon<2^{-24} \min\Big(\frac{1}{1+C_B}, \frac{1}{C_{\mathbf{h}}},
    r_0\Big)
    \end{equation}
  where \(r_0\) is the small radius guaranteed by Lemma
  \ref{LEM_small_radius}, \(C_{\mathbf{h}}\) is the ambient geometry
  constant given in Definition \ref{DEF_ambient_geometry_constant}, and
  \begin{equation*}                                                       
    C_B:=\sup \big\{ \|B_u(\zeta)\|_{\gamma}: {\rm dist}_\gamma(\zeta,
    \mathcal{Z})\geq {\textstyle \frac{1}{2}}\delta \big\},
    \end{equation*}
  and \(B_u\) is the second fundamental form of \(u\).  
We then let \(f:S\to \mathbb{R}\) be a smooth function for which
  \((\tilde{u}, \tilde{\jmath}, f, u, S, j)\) is an \((\delta,
  \epsilon)\)-tame perturbed pseudoholomorphic map in the sense of
  Definition \ref{DEF_epsilon_tame}; recall that the existence of such a
  perturbation is guaranteed by Lemma \ref{LEM_tame_perturbations}.

\begin{remark}[Morse failure]
  \label{REM_morse_failure}
  \hfill\\
By property (d\ref{EN_d3}) of  Definition \ref{DEF_epsilon_tame},
  the function \(a\circ \tilde{u}\)  will fail to be Morse only in \(
  \{\zeta\in S: {\rm dist}_\gamma(\mathcal{Z}, \zeta) \leq \delta\}\),
  and equation (\ref{EQ_containment}) then guarantees that this function
  only fails to be Morse inside \(\mathcal{B}= \{\zeta\in \mathcal{V}:
  r(\zeta)<{\textstyle \frac{1}{2}}\delta_6\} \).
\end{remark}
%

We will need to define several sets in terms of the following
  differential equation.
  \begin{equation}\label{EQ_gradient_flow_1}                              
    q: [0, T]\to S \qquad q'(s) = \widetilde{\nabla}(a\circ
    \tilde{u})\big(q(s)\big) \qquad q(0)\in \partial_0^-S
    \end{equation}
Here $\widetilde{\nabla}$ is the gradient with respect to
  \(\widetilde{\gamma}=\widetilde{u}^\ast g\); see also Definition
  \ref{DEF_tract_of_perturbed_J_map}.
In particular, we define the sets
  \begin{align}\label{EQ_ASB}                                             
    \mathcal{A}&:=\Big\{\zeta\in \mathcal{V} : \exists \text{ a solution
      to (\ref{EQ_gradient_flow_1}) s.t. } q(T)=\zeta \Big\}\\
    \mathcal{S}&:= \{\zeta\in \mathcal{V}: r(\zeta) ={\textstyle
      \frac{1}{2}}\delta_6\}\nonumber \\
    \mathcal{B}&:= \{\zeta\in \mathcal{V}: r(\zeta) <{\textstyle
      \frac{1}{2}}\delta_6\}\nonumber\\
    \mathcal{B}'&:=\mathcal{B}\cap\mathcal{A}\nonumber\\
    \mathcal{S}'&:=\mathcal{S}\cap\mathcal{A}.\nonumber
    \end{align}

\noindent{\bf Step 2.}

By conditions on \(\delta_6\), we see that \(\mathcal{S}\subset S\) is
  a finite set of pairwise disjoint embedded loops, which we equip with
  the subspace topology, and we  let \(\psi:\sqcup_{z\in \mathcal{Z}}
  S^1 \to \mathcal{S}\) denote a diffeomorphism.
By Lemma \ref{LEM_gamma_tilde_estimates}, the following estimate holds.
\begin{equation}\label{EQ_length_controlled_by_omega_energy}              
  {\rm length}_{\tilde{\gamma}}( \mathcal{S})\leq 2\cdot {\rm
  length}_{\gamma}(\mathcal{S}) \leq 4\pi \delta_6
  \sum_{z\in \mathcal{Z}} k_z\leq {\textstyle
  \frac{1}{2}}C_{\mathbf{h}}\int_S u^*\omega.
  \end{equation}
By existence, uniqueness, and continuous dependence upon initial
  conditions it follows that \(\mathcal{A}\) is open in \(S\), and
  consequently \(\mathcal{S}'\) is open in \(\mathcal{S}\).
Next we note that as a consequence of the definition of \(\mathcal{A}\),
  it follows that there is a well defined smooth map \(\pi\) given by
  \begin{align*}                                                          
    &\pi: \mathcal{A}\to \partial_0^-S\\
    &\pi(\zeta) = \zeta'\qquad\qquad\text{where there exists a solution
      to (\ref{EQ_gradient_flow_1})}\\
    &\qquad\qquad\qquad\qquad\text{ for which
      }q(0)=\zeta'\;\;\text{and}\;\; q(T)=\zeta.
    \end{align*}
By existence, uniqueness, and smooth dependence upon initial
  conditions, the map \(\pi\) is smooth.  
We define \(\mathcal{C}\subset \mathcal{S}'\) to be the set of critical
  points of the restricted map \(\pi: \mathcal{S}'\to \partial_0^- S\), and
  we define
  \begin{equation*}                                                       
    \mathcal{S}'':=\mathcal{S}'\setminus \mathcal{C}.
    \end{equation*} 
In other words, \(\mathcal{S}''\) consists of those points in
  \(\mathcal{S}\) which are hit by gradient trajectories extending from
  points in \(\partial_0^-S\) but which are not critical points of
  \(\pi\).
Let \({\rm cl}_{\mathcal{S}}(\mathcal{C})\) denote the closure in
  \(\mathcal{S}\) of the set \(\mathcal{C}\), and observe that
  \(\mathcal{S}'\cap {\rm cl}_{\mathcal{S}}(\mathcal{C}) =\mathcal{C}\), so
  that \(\mathcal{S}''\) is open in \(\mathcal{S}\).
Next, we note that \(\pi(\mathcal{C})\) has measure zero, and hence
  \(\partial_0^-S\setminus \pi(\mathcal{C})\) is non-empty, and thus we fix
  \(z'\in \partial_0^-S\setminus \pi(\mathcal{C})\).
Of course \(\{z'\}\) is closed in \(\partial_0^-S\), and hence
  \(\pi^{-1}(z')\) is closed in \(\mathcal{S}'\), and thus
  \(\mathcal{S}'\setminus (\mathcal{C}\cup \pi^{-1}(z'))\) is open in
  \(\mathcal{S}\).
As such, we define 
  \begin{align*}                                                          
    \mathcal{S}''':=\mathcal{S}'\setminus \big(\mathcal{C}\cup
    \pi^{-1}(z')\big)={\mathcal S}''\setminus \pi^{-1}(z'),
    \end{align*}
  which then must be open in \(\mathcal{S}\).

Since \(\mathcal{S}\) is diffeomorphic to the disjoint union of finitely
  many copies of \(S^1\), and \(\mathcal{S}'''\) is open in
  \(\mathcal{S}\), we conclude that \(\mathcal{S}'''\) is diffeomorphic to
  the countable disjoint union of pairwise disjoint open intervals and
  copies of \(S^1\).
We immediately note however that it must in fact be the countable disjoint
  union of \emph{intervals} with no copies of \(S^1\), essentially because
  we have removed \(\pi^{-1}(z')\) from \(\mathcal{S}''\) to obtain
  \(\mathcal{S}'''\), which forces each connected component of
  \(\mathcal{S}'''\) to be diffeomorphic to an interval.
Consequently, we write
  \begin{align*}                                                          
    \mathcal{S}'''=\cup_{k\in \mathbb{M}} \mathcal{I}_k
    \end{align*}
  where each \(\mathcal{I}_k\subset \mathcal{S}\) is diffeomorphic to an
  open interval, the \(\mathcal{I}_k\) are pairwise disjoint,
  the index set \(\mathbb{M}\) denotes either a finite set or else
  \(\mathbb{N}\) as appropriate, and \(\pi:\mathcal{I}_k \to
  \pi(\mathcal{I}_k)\subset \partial_0^-S\) is a diffeomorphism for each
  \(k\in \mathbb{M}\).
We note that by construction of the \(\mathcal{I}_k\), the one-form
  \(\tilde{\alpha}\) defines a one-dimensional volume form on each
  \(\mathcal{I}_k\).
However, for each \(k\in \mathbb{M}\), the map \(\pi:\mathcal{I}_k\to
  \pi(\mathcal{I}_k)\subset \partial_0^-S\) is a diffeomorphism, and hence
  we may define a second volume form \(\pi^* \tilde{\alpha}\) on each
  \(\mathcal{I}_k\).
This creates a dichotomy: for each \(k\in \mathbb{M}\), the orientations
  on \(\mathcal{I}_k\) induced from \(\tilde{\alpha}\) and
  \(\pi^*\tilde{\alpha}\) agree, or they do not.
Thus we write \(\{\mathcal{I}_k\}_{k\in
  \mathbb{M}}=\{\mathcal{I}_k^+\}_{k\in \mathbb{M}^+}\cup
  \{\mathcal{I}_k^-\}_{k\in \mathbb{M}^-}\) where
  \(\{\mathcal{I}_k^+\}_{k\in \mathbb{M}^+}\) denotes those intervals for
  which the orientations agree, and \(\{\mathcal{I}_k^-\}_{k\in
  \mathbb{M}^-}\) denotes those intervals for which the orientations
  disagree.
Obviously then,
  \begin{align*}                                                          
    \mathcal{S}'''=\Big(\bigcup_{k\in \mathbb{M}^+} \mathcal{I}_k^+\Big)
    \; \; \bigcup \; \; \Big(\bigcup_{k\in \mathbb{M}^-}
    \mathcal{I}_k^-\Big).
    \end{align*}

We now make the following claim.

\begin{lemma}[a technical containment]
  \label{LEM_technical_containment}
  \hfill\\
  \begin{align}\label{EQ_refinement}                                      
    \pi(\mathcal{C}) \cup \{z'\}\cup \pi(\cup_{k\in \mathbb{M}^+}
    \mathcal{I}_k^+)
    = 
    \pi(\mathcal{S}').
    \end{align}

\end{lemma}
%
\begin{proof}
To prove this lemma, we first let \(\zeta_1\in \mathcal{S}'\), so there
  exists gradient trajectory emanating from \(\zeta_0\in \partial_0^-S\)
  and terminating at \(\zeta_1\in \mathcal{S}'\).
We observe that there are four possible cases. \\

\noindent \emph{Case I.} \(\zeta_0=z'\).
In this case, \(\pi(\zeta_1)\in \{z'\}\).\\
 
\noindent \emph{Case II.} \(\zeta_0\neq z'\) and the vector
  \(\widetilde{\nabla}(a \circ \tilde{u})\) is tangent to \(\partial
  \mathcal{S}\) at \(\zeta_1\).
In this case \(\zeta_1\in \mathcal{C}\), and hence \(\pi(\zeta_1)\in
  \pi(\mathcal{C})\).\\

\noindent \emph{Case III.} \(\zeta_0\neq z'\) and the vector
  \(\widetilde{\nabla}(a \circ \tilde{u})\) is transverse to \(\partial
  \mathcal{S}\) at \(\zeta_1\) and inward pointing relative to
  \(\mathcal{B}\). 
It immediately follows that \(\zeta_1\in \cup_{k\in \mathbb{M}^-}
  \mathcal{I}_k^-\), and hence \(\pi(\zeta_1)\in \pi\big(\cup_{k\in
  \mathbb{M}^-} \mathcal{I}_k^-\big)\). \\

\noindent \emph{Case IV.} \(\zeta_0\neq z'\) and the vector
  \(\widetilde{\nabla}(a \circ \tilde{u})\) is transverse to \(\partial
  \mathcal{S}\) at \(\zeta_1\) and outward pointing relative to
  \(\mathcal{B}\). 
It immediately follows that \(\zeta_1\in \cup_{k\in \mathbb{M}^+}
  \mathcal{I}_k^+\), and hence \(\pi(\zeta_1)\in \pi\big(\cup_{k\in
  \mathbb{M}^+} \mathcal{I}_k^+\big)\). \\

\noindent We conclude that 
  \begin{align*}                                                          
    \pi(\mathcal{S}')
    = 
    \pi(\mathcal{C}) \cup \{z'\}\cup \pi(\cup_{k\in \mathbb{M}^+}
    \mathcal{I}_k^+) \cup \pi(\cup_{k\in \mathbb{M}^-} \mathcal{I}_k^-),
    \end{align*}
  and hence to establish equation (\ref{EQ_refinement}), it is sufficient
  to prove that
  \begin{align*}                                                          
    \pi\big(\cup_{k\in \mathbb{M}^-}\mathcal{I}_k^-\big)\setminus
    \pi(\mathcal{C}) \subset \pi\big(\cup_{k\in
    \mathbb{M}^+}\mathcal{I}_k^+\big).
    \end{align*}
Note however, that if \(\zeta_1\in \cup_{k\in
  \mathbb{M}^-}\mathcal{I}_k^-\setminus \pi^{-1}\circ \pi (\mathcal{C})\),
  then \(\widetilde{\nabla} (a\circ \tilde{u})\) is pointing outward
  relative to \(\mathcal{B}\) at \(\zeta_1\).
Or in other words, following the gradient flow \(\widetilde{\nabla}(a\circ
  \tilde{u})\) from \(\zeta_1\) for sufficiently small but \emph{negative}
  time, yields a point in \(\mathcal{B}\).
However, \({\rm cl}(\mathcal{B})\cap \partial_0^-S=\emptyset\), and hence
  we conclude that the gradient flow line initiating at \(\zeta_0\in
  \partial_0^-S\) and terminating at \(\zeta_1\subset \cup_{k\in
  \mathbb{M}^-}\mathcal{I}_k^-\setminus \pi^{-1}\circ \pi(\mathcal{C})\)
  must first intersect \(\mathcal{S}'\) in an inward pointing direction,
  and this intersection must be transverse.
It immediately follows that \(\pi(\zeta_1) \subset \pi(\cup_{k\in
  \mathbb{M}^+}\mathcal{I}_k^+)\), and hence
  \begin{align*}                                                          
    \pi\big(\cup_{k\in \mathbb{M}^-}\mathcal{I}_k^-\big)\setminus
    \pi(\mathcal{C}) \subset \pi\big(\cup_{k\in
    \mathbb{M}^+}\mathcal{I}_k^+\big),
    \end{align*}
  as required.
This completes the proof of Lemma \ref{LEM_technical_containment}.
\end{proof}

Next, we  choose \(n_0\in \mathbb{N} \) sufficiently large so that
  \begin{align}\label{EQ_I_and_alpha_estimate}                            
    \Big| \int_{\pi(\cup_{k\in \mathbb{M}^+} \mathcal{I}_k^+)}
    \tilde{\alpha} - \int_{\pi(\cup_{k=1}^{n_0} \mathcal{I}_k^+)}
    \tilde{\alpha} \Big| < {\textstyle \frac{1}{2}}C_{\mathbf{h}} \int_S
    u^*\omega
    \end{align}
  where the orientation on \(\cup_k \mathcal{I}_k^+\subset \partial_0^- S\)
  is such that \(\tilde{\alpha}\) is a volume form.
Recall that the \(\mathcal{I}_k^+\) are diffeomorphic to open intervals,
  and \(\pi:\mathcal{I}_k^+\to \partial_0^-S\) are diffeomorphisms with
  their images.
As such, we may find sets \(\{\mathcal{J}_k \}_{k=1}^{n_1}\)
  in \(\mathcal{S}'\) with the following properties:
\begin{enumerate}                                                         
  \item each \(\mathcal{J}_k\) is diffeomorphic to a compact interval
  \item \(\pi(\mathcal{J}_k)\cap \pi(\mathcal{J}_{k'})=\emptyset\) for
    \(k\neq k'\)
  \item for each \(k\in \{1, \ldots, n_1\}\) there exists \(k'\in \{1,
    \ldots, n_0\}\) such that \(\mathcal{J}_k\subset \mathcal{I}_{k'}^+\)
  \item the map \(\pi\colon\mathcal{J}_k\to \pi(\mathcal{J}_k)\subset
    \partial_0^-S\) is a diffeomorphism
  \item and finally,
  \begin{align}\label{EQ_I_J_comparison}                                  
    \Big|\int_{\cup_{k=1}^{n_0}\pi(\mathcal{I}_k^+)}\tilde{\alpha} -
    \int_{\cup_{k=1}^{n_1}\pi(\mathcal{J}_k)}\tilde{\alpha}\Big| <
    \frac{1}{2}C_{\mathbf{h}} \int_S u^*\omega.
    \end{align}
  \end{enumerate}
As a consequence of the existence of such \(\mathcal{J}_k\), we note that
  there exist perturbed pseudoholomorphic strips \((\tilde{u}_k,
  \widetilde{S}_k, \tilde{\jmath}, f, u, S, j)\) for which
  \begin{align*}                                                          
    \partial_0^- \widetilde{S}_k =
    \pi(\mathcal{J}_k)\qquad\text{and}\qquad \partial_0^+ \widetilde{S}_k
    = \mathcal{J}_k \subset \bigcup_{k\in \mathbb{M}^+} \mathcal{I}_k^+
    \subset \mathcal{S}.
    \end{align*}
With these perturbed pseudoholomorphic strips established, we are now able
to estimate as follows.
\begin{align}                                                             
  \Big|\int_{\pi(\mathcal{B}')}\tilde{\alpha}\Big| 
  &= 
    \int_{\pi(\mathcal{B}')}\tilde{\alpha}\notag\\
  &=
    \int_{\pi(\mathcal{S}'')}\tilde{\alpha}\qquad\qquad\qquad\qquad\text{(See
    Lemma \ref{LEM_alpha_integrals} below)}\notag\\
  &=
    \int_{\cup_{k\in \mathbb{M}^+}\pi(
    \mathcal{I}_k^+)}\tilde{\alpha}\notag\\
  &= 
    \Big(\int_{\cup_{ k\in \mathbb{M}^+}\pi(\mathcal{I}_k^+) }\tilde{\alpha}-
    \int_{\cup_{k=1}^{n_0} \pi(\mathcal{I}_k^+)}\tilde{\alpha}\Big) +
    \int_{\cup_{k=1}^{n_0} \pi(\mathcal{I}_k^+)}\tilde{\alpha}\notag\\
  &\leq
    {\textstyle \frac{1}{2}}C_{\mathbf{h}}\int_S u^*\omega +
    \int_{\cup_{k=1}^{n_0} \pi(\mathcal{I}_k^+)}\tilde{\alpha}\notag\\
  &= 
    {\textstyle \frac{1}{2}}C_{\mathbf{h}}\int_S u^*\omega +
    \Big(\int_{\cup_{ k=1}^{n_0}\pi(\mathcal{I}_k^+) }\tilde{\alpha}-
    \int_{\cup_{k=1}^{n_1} \pi(\mathcal{J}_k)}\tilde{\alpha}\Big) +
    \int_{\cup_{k=1}^{n_1} \pi(\mathcal{J}_k)}\tilde{\alpha}\notag\\
  &\leq 
    C_{\mathbf{h}}\int_S u^*\omega+  \int_{\cup_{k=1}^{n_1}
    \pi(\mathcal{J}_k)}\tilde{\alpha}\notag\\
  &=
    C_{\mathbf{h}}\int_S u^*\omega+ \Big( \int_{\cup_{k=1}^{n_1}
    \pi(\mathcal{J}_k)}\tilde{\alpha} -2\int_{\cup_{k=1}^{n_1}
    \mathcal{J}_k}\tilde{\alpha} \Big)  + 2\int_{\cup_{k=1}^{n_1}
    \mathcal{J}_k}\tilde{\alpha}\notag\\
  &\leq 
    C_{\mathbf{h}}\int_S u^*\omega +
    2C_{\mathbf{h}}\sum_{k=1}^{n_1}\int_{\widetilde{S}_k}
    \tilde{u}^*\omega+ 2\int_{\cup_{k=1}^{n_1}
    \mathcal{J}_k}\tilde{\alpha}\notag\\
  &\leq 
    3C_{\mathbf{h}}\int_{S} \tilde{u}^*\omega+ 2\int_{\cup_{k=1}^n
    \mathcal{J}_k}\tilde{\alpha}\notag\\
  &\leq 
    3C_{\mathbf{h}}\int_{S}\tilde{u}^*\omega
    +2\|\tilde{\alpha}\|_{L^\infty}
    {\rm length}_{\tilde{\gamma}}(\cup_{k=1}^n\mathcal{J}_k)\notag\\
  &\leq 
    3C_{\mathbf{h}}\int_{S}\tilde{u}^*\omega +2\,{\rm
    length}_{\tilde{\gamma}}(\cup_{k=1}^n\mathcal{J}_k)
    \label{EQ_tilde_alpha_est}\\
  &\leq 
    3C_{\mathbf{h}}\int_{S}\tilde{u}^*\omega +2\,{\rm
    length}_{\tilde{\gamma}}(\mathcal{S})\notag\\
  &\leq 
    4C_{\mathbf{h}} \int_{S}\tilde{u}^*\omega \notag
  \end{align}
  where to obtain the second equality we have made use of Lemma
  \ref{LEM_alpha_integrals} below, to obtain the third equality we have
  made use of the fact that \(\pi(\mathcal{C})\cup \{z'\}\) has measure
  zero together with Lemma \ref{LEM_technical_containment}, to obtain the
  first inequality we have made use of equation
  (\ref{EQ_I_and_alpha_estimate}), to obtain the second inequality we have
  employed equation (\ref{EQ_I_J_comparison}), to obtain the third
  inequality we have employed Lemma \ref{LEM_general_strip_estimate}, to
  obtain the inequality at (\ref{EQ_tilde_alpha_est}) we have used
  \begin{equation*}                                                       
    \|\tilde{\alpha}\|_{\tilde{\gamma}}=
    \|-\tilde{u}^*da\circ \tilde{\jmath}\|_{\tilde{u}^*g} =
    \|\tilde{u}^*da\|_{\tilde{u}^*g}\leq \|da\|_{g} = 1,
    \end{equation*}
  and to obtain the final inequality we have employed equation
  (\ref{EQ_length_controlled_by_omega_energy}).  
Note that the above inequality relies on the following equality.

\begin{lemma}[equality of $\tilde{\alpha}$ integrals]
  \label{LEM_alpha_integrals}
  \hfill\\
  \begin{equation*}                                                       
    \int_{\pi(\mathcal{B}')}\tilde{\alpha} = \int_{\pi(\mathcal{S}'')}
    \tilde{\alpha}
    \end{equation*}
\end{lemma}
%
\begin{proof}
Recall the definitions of ${\mathcal B}$ and ${\mathcal B}'$ from
  (\ref{EQ_ASB}).
First observe that each connected component of \(\mathcal{B}\) is
  homeomorphic to an open disk   which is disjoint from \(\partial_0^-S\),
  and \(\partial\mathcal{B} = \mathcal{S}\).
Also observe that \(\mathcal{B}'\subset\mathcal{B}\) and
  \(\pi:\mathcal{B}'\to \partial_0^- S\) is well defined.
It follows that for each \(\zeta\in \mathcal{B}'\) there exists a
  \(\zeta'\in \mathcal{S}'\) such that \(\pi(\zeta)=\pi(\zeta')\).
From this we conclude that
  \begin{equation*}                                                       
    \pi(\mathcal{B}')\subset \pi(\mathcal{S}').
    \end{equation*}
Next observe that the definition of \(\mathcal{S}''\) guarantees that
  if \(\zeta\in \mathcal{S}''\), then the gradient trajectory solving
  (\ref{EQ_gradient_flow_1}) intersects \(\mathcal{S}''\) transversely
  at \(\zeta\).
It follows that \(\pi(\mathcal{S}'')\subset\pi(\mathcal{B}')\), and
  hence we have
  \begin{equation*}                                                       
    \pi(\mathcal{S}'')\subset\pi(\mathcal{B}') \subset \pi(\mathcal{S}').
    \end{equation*}
We then recall that \(\mathcal{S}'' = \mathcal{S}'\setminus \mathcal{C}\)
  where \(\mathcal{C}\) is the set of critical points of the map
  \(\pi:\mathcal{S}'\to \partial_0^-S\), and hence by Sard's theorem we
  conclude that \(\pi(\mathcal{C})\) has Lebesgue measure zero, and thus
  we have
  \begin{equation*}                                                       
    \pi(\mathcal{S}')\setminus
    \pi(\mathcal{C})\subset\pi(\mathcal{S}'')\subset\pi(\mathcal{B}')
    \subset \pi(\mathcal{S}').
    \end{equation*}
Since \(\pi(\mathcal{C})\) has Lebesgue measure zero, it immediately
  follows that \(\mathcal{S}''\) and \(\mathcal{B}'\) differ by a set of
  measure zero and hence
  \begin{equation*}
    \int_{\pi(\mathcal{B}')}\tilde{\alpha} = \int_{\pi(\mathcal{S}'')}
    \tilde{\alpha}
    \end{equation*}
  which is the desired result. 
This completes the proof of Lemma \ref{LEM_alpha_integrals}.
\end{proof}

As a consequence, we have shown that
  \begin{equation}\label{EQ_measure_of_bad_set}                           
    \Big| \int_{\pi(\mathcal{B}')} \tilde{\alpha} \Big| \leq
    4C_{\mathbf{h}} \int_S u^*\omega.
    \end{equation}
Thus we have shown that the \(\tilde{\alpha}\)-measure of those points in
  \(\partial_0^-S\) with gradient flow lines that enter into our
  neighborhood of \(\mathcal{Z}\) is bounded in terms of the
  ambient geometry constant and \(\omega\)-energy, the latter of which is
  assumed to be small.\\

\noindent{\bf Step 3.}

Next, we recall Remark \ref{REM_morse_failure}, which observes that
  there there are only finitely many critical points of \(a\circ
  \tilde{u}\) in \(S\setminus\mathcal{B}\) and each will be
  non-degenerate.
As such, for \(k\in \{0, 1, 2\}\) we define the finite sets
  \begin{equation*}                                                       
    \mathcal{M}_k:=\{\zeta\in S\setminus \mathcal{B}: d(a\circ
    \tilde{u})(\zeta)=0\;\;\text{and}\;\;{\rm Index}_{Morse}(\zeta)=k\}.
    \end{equation*}
Note that \(\mathcal{M}_0\) consists of local minima and therefore
  there cannot exist solutions to the gradient equation
  \begin{equation}\label{EQ_half_infinite_gradient_equation}              
    q: [0, \infty)\to S \qquad q'(s) = \widetilde{\nabla}(a\circ
    \tilde{u})\big(q(s)\big) \qquad q(0)\in \partial_0^-S
    \end{equation}
  which limit to a point in \(\mathcal{M}_0\).
Also note that \(\mathcal{M}_1\) consists of finitely many
  (non-degenerate) saddle-points and hence there are only finitely many
  solutions to (\ref{EQ_half_infinite_gradient_equation}) which limit
  to a point in \(\mathcal{M}_1\); we denote the set of such initial
  conditions \(\mathcal{D}\).
It remains to consider those initial conditions in \(\partial_0^-S\)
  for which solutions to (\ref{EQ_half_infinite_gradient_equation})
  limit to points in \(\mathcal{M}_2\).
To that end, we fix \(\epsilon'>0\) sufficiently small so
  that  for each \(z\in \mathcal{M}_2\) the set 
  \(\{\zeta\in S: a\circ\tilde{u}(\zeta)=a\circ \tilde{u}(z)-\epsilon'\}\)
  contains a connected component,
  \(\mathcal{S}_z\), contained in a small neighborhood of \(z\)
  in which there exist local coordinates, \((s, t)\),  such that 
  \(a\circ\tilde{u}(s, t) = a\circ\tilde{u}(z) - s^2 - t^2\), 
  and furthermore \(\epsilon'\) has been chosen sufficiently small so that
  \begin{equation*}                                                       
    \sum_{z\in \mathcal{M}_2}{\rm length}_{\tilde{\gamma}}(\mathcal{S}_z)
    \leq {\textstyle \frac{1}{2}}C_{\mathbf{h}} \int_S u^*\omega.
    \end{equation*}
Observe that by construction, no trajectory initiating from
  \(\partial_0^-S\)  may limit to a point in \(\mathcal{M}_2\)
  without transversally intersecting \(\cup_{\zeta\in \mathcal{M}_2}
  \mathcal{S}_\zeta\).
Defining
  \begin{equation*}                                                       
    \mathcal{E}:=\Big\{\zeta\in \partial_0^- S: \exists \text{ a solution
    to (\ref{EQ_gradient_flow_1}) such that }q(0)=\zeta\text{ and }q(T)\in
    \cup_{z\in \mathcal{M}_2} \mathcal{S}_z  \Big\},
    \end{equation*}
  we note that \(\mathcal{E}\) is open, and hence we may find finitely
  many pair-wise disjoint closed intervals \(\mathcal{L}_k\subset
  \partial_0^- S\) with the property that
  \begin{align*}                                                          
    \Big|\int_{\mathcal{E}}\tilde{\alpha}\Big| 
    &= 
      \int_{\mathcal{E}}\tilde{\alpha}\\
    &\leq 
      \sum_k\int_{\mathcal{L}_k}\tilde{\alpha} +
      C_{\mathbf{h}}\int_{S}u^*\omega.
    \end{align*}
As above, for each \(\mathcal{L}_k\) one constructs a perturbed
  pseudoholomorphic strip, denoted \((\tilde{u}_k, \widetilde{S}_k,
  \tilde{\jmath}_k, f, u, S, j) \), with the property that \(\partial_0^-
  \widetilde{S}_k = \mathcal{L}_k \) and
  \(\partial_0^+\widetilde{S}_k\subset \cup_{z\in \mathcal{M}_2}
  \mathcal{S}_z\).
Also as above, this yields a similar estimate:
  \begin{align*}                                                          
    \sum_k\int_{\mathcal{L}_k}\tilde{\alpha}
    &=
      \sum_k\Big(\int_{\partial_0^- \widetilde{S}_k}\tilde{\alpha} -
      2\int_{\partial_0^+ \widetilde{S}_k}\tilde{\alpha}\Big) +\sum_k
      2\int_{\partial_0^+ \widetilde{S}_k}\tilde{\alpha}\\
    &\leq 
      2C_{\mathbf{h}}\sum_k\int_{\widetilde{S}_k}u^*\omega  +\sum_k
      2\int_{\partial_0^+ \widetilde{S}_k}\tilde{\alpha}\\
    &\leq 
      2C_{\mathbf{h}}\int_{S}u^*\omega+ 2 \sum_{z\in \mathcal{M}_2}{\rm
      length}_{\tilde{\gamma}}(\mathcal{S}_z)\\
    &\leq 
      3C_{\mathbf{h}}\int_{S}u^*\omega.
    \end{align*}
From this we conclude that
  \begin{equation}\label{EQ_measure_of_E}                                 
    \int_{\mathcal{E}}\tilde{\alpha}\leq 4C_{\mathbf{h}}
    \int_{S}u^*\omega.
    \end{equation}

Thus we have shown that the \(\tilde{\alpha}\)-measure of those points in 
  \(\partial_0^-S\) which have gradient flow lines that limit to local
  maxima of \(a\circ \tilde{u}\) is bounded in terms of the
  ambient geometry constant and the \(\omega\)-energy, the latter of which
  is assumed to be small.\\

\noindent{\bf Step 4.}

At this point we define the \(\mathcal{T}\subset \partial_0^- S\)
  to be the set of points for which there exists a solution
  \begin{equation}                                                        
    q: [0, T]\to S \qquad q'(s) = \widetilde{\nabla}(a\circ
    \tilde{u})\big(q(s)\big)
    \end{equation}
  for which
  \begin{enumerate}                                                       
    \item 
      \(q(0)\in \partial_0^-S\)
    \item 
      \(q(T)\in \partial_0^+S\)
    \item 
      \(q([0, T]) \bigcap \big(\cup_{z\in
      \mathcal{Z}}\{\zeta\in \mathcal{V}_z: r_z(\zeta)\leq
      \frac{1}{4}\delta_6\}\big)=\emptyset\).
    \end{enumerate}
We note that a consequence of Remark \ref{REM_morse_failure} and
  equation (\ref{EQ_containment}),  the function \(a\circ \tilde{u}\) is
  Morse on \(S\setminus \cup_{z\in \mathcal{Z}}\{\zeta\in \mathcal{V}_z:
  r_z(\zeta)\leq \frac{1}{4}\delta_6\}\).
As such, we conclude that \(\mathcal{T}\subset \partial_0^- S\) is open
  in \(\partial_0^- S\).
We then claim the following inequalities are true:
  \begin{equation*}                                                       
    \int_{(\partial_0^- S)\setminus \mathcal{T} }  \tilde{\alpha}
    \leq \int_{\pi(\mathcal{B}')\cup \mathcal{D}\cup\mathcal{E}}
    \tilde{\alpha}\leq 8C_{\mathbf{h}}\int_S u^*\omega.
    \end{equation*}
Observe that the first inequality follows from the fact that \(
  (\partial_0^- S)\setminus \mathcal{T} \subset \pi(\mathcal{B}' )\cup
  \mathcal{D}\cup\mathcal{E}\), and the second inequality follows from the
  fact that \(\mathcal{D}\) is finite together with inequalities
  (\ref{EQ_measure_of_bad_set}) and  (\ref{EQ_measure_of_E}).
Since \(\mathcal{T}\) is open in \(\partial_0^- S\),
  we note that there exist finitely many pairwise disjoint closed
  intervals \(\{\mathcal{T}_k\}_{k=1}^N\), each contained in
  \(\mathcal{T}\), such that
  \begin{equation}\label{EQ_Tk_estimate}                                  
    \sum_{k=1}^N\int_{\mathcal{T}_k}\tilde{\alpha} \geq
    \int_{\partial_0^- S} \tilde{\alpha} - 10C_{\mathbf{h}} \int_S
    u^*\omega.
    \end{equation}
We pause for a moment to highlight the utility of these strips. 
Roughly speaking, the proof of Theorem
  \ref{THM_local_local_area_bound} would be significantly simpler if every
  gradient trajectory initiating at a point in \(\partial_0^- S\) terminated
  at point in \(\partial_0^+S\).
Because this is not the case, the next best scenario would be for there to
  exist a finite set of disjoint closed intervals in \(\partial_0^-S\) with
  the property that their \(\tilde{\alpha}\)-measure was close to that of
  \(\partial_0^-S\), and with the property that every gradient trajectory
  starting in one of these intervals then terminated in \(\partial_0^+S\).
The intervals \(\{\mathcal{T}_k\}_{k=1}^N\) have precisely this property,
  and thus heuristically we should think of the associated perturbed
  pseudoholomorphic strips as taking the place of \(S\).\\

\noindent{\bf Step 5.}

We now claim the following.

\begin{lemma}[modest length gradient trajectories]
  \label{LEM_existence_of_modest_gradient_trajectory}
  \hfill \\
For each closed interval \(\mathcal{I}\subset \partial_0^- S\) satisfying
  \begin{equation*}                                                       
    \int_{\mathcal{I}}\tilde{\alpha} \geq
    \big((a_1-a_0)^{-1}+10C_{\mathbf{h}}\big)\int_S u^*\omega
    \end{equation*} 
  there exists a solution to
  \begin{equation}                                                        
    q: [0, T]\to S \qquad q'(s) = \widetilde{\nabla}(a\circ
    \tilde{u})\big(q(s)\big)
    \end{equation}
  such that \(q(0)\in \mathcal{I}\), \(q(T)\in \partial_0^+S\), and
  \begin{equation}\label{EQ_length_estimate_in_prop}                      
    {\rm length}_{\tilde{\gamma}}\big(q([0, T])\big)\leq 2^7(a_1-a_0).
    \end{equation}
\end{lemma}
%
\begin{proof}
We begin by observing
  \begin{align*}                                                          
    \int_{\mathcal{I}\cap (\cup_{k=1}^N \mathcal{T}_k)}\tilde{\alpha}
    &=
      \int_{\mathcal{I}}\tilde{\alpha} - \int_{\mathcal{I}\cap
      (\partial_0^-S \setminus \cup_{k=1}^N
      \mathcal{T}_k)}\tilde{\alpha}\\
    &\geq
      \int_{\mathcal{I}}\tilde{\alpha} - \int_{\partial_0^-S \setminus
      \cup_{k=1}^N \mathcal{T}_k}\tilde{\alpha}\\
    &=
      \int_{\mathcal{I}}\tilde{\alpha} - \int_{\partial_0^-S
      }\tilde{\alpha} + \sum_{k=1}^N\int_{\mathcal{T}_k}\tilde{\alpha}\\
    &\geq 
      \int_{\mathcal{I}}\tilde{\alpha} -
      10C_{\mathbf{h}}\int_{S}u^*\omega\\
    &\geq 
      \frac{1}{a_1-a_0}\int_Su^*\omega.
    \end{align*}
However \(\mathcal{I}\cap (\cup_{k=1}^N \mathcal{T}_k)\) is a finite
  union of closed intervals, so that by Lemma
  \ref{LEM_modest_length_flow_lines} and the inequality just established,
  it follows that there exists \(\zeta\in \mathcal{I}\) with the property
  that the gradient line extending from this point intersects
  \(\partial_0^+S\) in finite time and it satisfies the length estimate
  (\ref{EQ_length_estimate_in_prop}).
This completes the proof of Lemma
  \ref{LEM_existence_of_modest_gradient_trajectory}.  
\end{proof}
To continue, it will be convenient to define the following.
Given two points, \(\zeta_0, \zeta_1\in \partial_0^- S\), we define
  \(\mathcal{I}_{\zeta_0}^{\zeta_1}\subset \partial_0^-S\) to be the
  closed interval, oriented so that \(\tilde{\alpha}\) is a volume
  form on \(\mathcal{I}_{\zeta_0}^{\zeta_1}\), and such that \(\partial
  \mathcal{I}_{\zeta_0}^{\zeta_1} = \zeta_1 - \zeta_0\).
We now find a finite set of points \(\{\zeta_k\}_{k=1}^{2n}\subset
  \partial_0^-S\) with the following properties.
For each \(k\in \{1, \ldots, 2n\}\) we have
  \begin{equation*}                                                       
    r_0< \int_{\mathcal{I}_{\zeta_k}^{\zeta_{k+1}}}\tilde{\alpha} <2r_0,
    \end{equation*}
  where \(\zeta_{2n+1}:=\zeta_1\).
We also require that if \(\zeta_\ell\notin \{\zeta_k, \zeta_{k+1}\}\)
  then \(\zeta_\ell\notin \mathcal{I}_{\zeta_k}^{\zeta_{k+1}}\).
By Lemma \ref{LEM_existence_of_modest_gradient_trajectory}, it follows
  that for each \(k\in \{1, \ldots, 2n\}\) there exists \(\zeta_k^-\in
  \mathcal{I}_{\zeta_k}^{\zeta_{k+1}}\) with the property that there
  exists a solution to
  \begin{equation}                                                        
    q_k: [0, T_k]\to S \qquad q_k'(s) = \widetilde{\nabla}(a\circ
    \tilde{u})\big(q_k(s)\big)
    \end{equation}
  such that \(q_k(0)=\zeta_k^-\), \(q_k(T_k)\in \partial_0^+S\), and
  \begin{equation}\label{EQ_side_length_estimate}                         
    {\rm length}_{\tilde{\gamma}}\big(q_k([0, T_k])\big)\leq
    2^7(a_1-a_0).
    \end{equation}
For each \(k\in \{1, \ldots, n\}\) we then define
  \(z_k^-:=\zeta_{2k}^-\).
These points satisfy the property that  for each \(k\in \{1, \ldots,
  n\}\) we have
  \begin{equation}\label{EQ_some_integral_bounds}                         
    r_0\leq \int_{\mathcal{I}_{z_k^-}^{z_{k+1}^-}}\tilde{\alpha}
    \leq 6r_0,
    \end{equation}
  where for notational convenience we have used \(z_{n+1}^-=z_1^-\).
Denoting \(z_k^+:=q_{2k}(T_{2k})\in \partial_0^+S\), we now define
  \(\Sigma_k\subset S\) to be the surface uniquely determined by having
  boundary
  \begin{equation*}                                                       
    \partial\Sigma_k = q_{2k}([0, T_{2k}])\;\;\bigcup\;\; q_{2k+2}([0,
    T_{2k+2}])\;\;\bigcup\;\; \mathcal{I}_{z_k^-}^{z_{k+1}^-}\;\;\bigcup
    \;\;\mathcal{I}_{z_k^+}^{z_{k+1}^+}.
    \end{equation*}
Here we have abused notation a bit to write
  \(\mathcal{I}_{z_k^+}^{z_{k+1}^+}\subset \partial_0^+ S\), though its
  meaning should be clear form context.
For later use, we make the following definition.
  \begin{align}\label{EQ_Sigma_sides}                                     
    \partial_1^- \Sigma_k:= q_{2k}([0, T_{2k}])\qquad\text{and}\qquad
    \partial_1^+\Sigma_k: = q_{2k+2}([0, T_{2k+2}])
    \end{align}
Observe that the \(\{z_k^-\}_{k=1}^n\) satisfy the property that
  if \(z_\ell^-\notin \{z_k^-, z_{k+1}^-\}\) then \(z_\ell^-\notin
  \mathcal{I}_{z_k^-}^{z_{k+1}^-}\), and hence the \(\{\Sigma_k\}_{k=1}^n\)
  have the property that if \(k\neq \ell\) then \(\Sigma_k\cap
  \Sigma_\ell\) is either empty or consists of a single gradient
  trajectory.\\

\noindent{\bf Step 6.}

It will be convenient to estimate the area of \(\Sigma_k\) which is
  done as follows.
\begin{align*}                                                            
  {\textstyle\frac{1}{2}}{\rm Area}_{\tilde{\gamma}}(\Sigma_k)
  &\leq
    \int_{\Sigma_k} \tilde{u}^*da \wedge \tilde{\alpha}
    +\int_{\Sigma_k}\tilde{u}^*\omega\\
  &= 
    \int_{\Sigma_k}d\big( (a\circ \tilde{u}-a_1)\tilde{\alpha}\big)
    -\int_{\Sigma_k}(a\circ \tilde{u} - a_1)d\tilde{\alpha}
    +\int_{\Sigma_k}\tilde{u}^*\omega\\
  &=
    (a_1-a_0)\int_{\mathcal{I}_{z_k^-}^{z_{k+1}^-}} \tilde{\alpha}
    -\int_{\Sigma_k}(a\circ \tilde{u} - a_1)d\tilde{\alpha}
    +\int_{\Sigma_k}\tilde{u}^*\omega\\
  &\leq 
    (a_1-a_0)6r_0 + \|a\circ\tilde{u} - a_1\|_{L^\infty(\Sigma_k)}
    \|d\tilde{\alpha}\|_{L^\infty(\Sigma_k)} {\rm Area}_{\tilde{\gamma}}
    (\Sigma_k)+ {\textstyle \frac{1}{100}}\\
  &\leq 
    (a_1-a_0)6r_0 + \frac{1}{16C_{\mathbf{h}}}
    \frac{C_{\mathbf{h}}}{2}{\rm
    Area}_{\tilde{\gamma}} (\Sigma_k) + {\textstyle \frac{1}{100}}
  \end{align*}
The first inequality follows from Lemma \ref{LEM_linear_alg_coercive},
  and the final inequality employs Lemma \ref{LEM_d_alpha_tilde_bound}
  and property (LL\ref{EN_LL4}) from our assumptions in Proposition
  \ref{PROP_local_local_area_bound_2}.
Consequently, the desired area estimate is given as      
  \begin{equation*}                                                       
    {\rm Area}_{\tilde{\gamma}}(\Sigma_k)\leq (a_1-a_0)24r_0 +
    {\textstyle \frac{1}{25}}\leq {\textstyle \frac{1}{6}} .
    \end{equation*}
We are now prepared to finish the proof of Proposition
  \ref{PROP_local_local_area_bound_2}.
Indeed, let \(\zeta\in S\) so that 
  \begin{equation*}                                                       
    \big|a\circ u(\zeta)- {\textstyle \frac{1}{2}}(a_1 + a_0)\big|\leq
    {\textstyle\frac{1}{4}}(a_1-a_0),
    \end{equation*}
  and define \(\widetilde{S}_{r}(\zeta)\) to be the connected component
  of \(\tilde{u}^{-1}\big(\mathcal{B}_r(\tilde{u}(\zeta))\big)\) which
  contains \(\zeta\).
Here \(\mathcal{B}_r(p)\) is the open metric ball of radius \(r\)
  centered at \(p\in \mathbb{R}\times M\).
Recalling that \(r_1= 2^{-24} \min\big( C_{\mathbf{h}}^{-1} ,
  r_0\big) \leq 2^{-10}(a_1-a_0)\), we now claim the following.

\begin{lemma}[$\zeta$ cannot be close to both sides simultaneously]
  \label{LEM_small_cannot_intersect_both_one_boundaries}
  \hfill\\
It cannot be the case that \(\widetilde{S}_{4r_1}(\zeta)\cap \partial_1^-
  \Sigma_k\neq \emptyset \) and \(\widetilde{S}_{4r_1}(\zeta)\cap
  \partial_1^+ \Sigma_k\neq \emptyset \) for any \(k\in \{1, \ldots,
  n\}\); here the \(\partial_1^\pm \Sigma_k\) are defined in equation
  (\ref{EQ_Sigma_sides}).
\end{lemma}
%
We will prove Lemma \ref{LEM_small_cannot_intersect_both_one_boundaries}
  momentarily, but for now we make use of it to complete the proof of
  Proposition \ref{PROP_local_local_area_bound_2}.
Indeed,  as a consequence of Lemma
  \ref{LEM_small_cannot_intersect_both_one_boundaries} it must be the case
  that if \(\widetilde{S}_{4r_1}(\zeta)\cap \partial \Sigma_k\neq
  \emptyset\) then \(\widetilde{S}_{4r_1}(\zeta)\subset
  \Sigma_{k-1}\cup\Sigma_k\cup\Sigma_{k+1}\).
Consequently,
  \begin{equation}\label{EQ_Area_containtment_estimate}                    
    {\rm Area}_{\tilde{\gamma}} \big(\widetilde{S}_{4r_1}(\zeta)\big)\leq
    \frac{1}{2}.
    \end{equation}
Next, we note that since \(\tilde{u}\) is an \((\delta, \epsilon)\)-tame
  perturbation of \(u\) with \(\epsilon< r_1\) (recall inequality
  (\ref{EQ_epsilon_suff_small}) and Remark 
  \ref{REM_feature_of_de_perturbations}) it follows that for each \(z\in
  S_{r_1}(\zeta)\) we have \(\tilde{u}(z)\in \mathcal{B}_{r_1}(u(z))\).
Moreover since \(u(S_{r_1}(\zeta))\subset \mathcal{B}_{r_1}(u(\zeta))\)
  it then follows that  \(u(S_{r_1}(\zeta))\subset
  \mathcal{B}_{2r_1}(\tilde{u}(\zeta))\).
From this it follows that \(\tilde{u}(S_{r_1}(\zeta))\subset
  \mathcal{B}_{4r_1}(\tilde{u}(\zeta))\).
In other words we have shown that \( S_{r_1}(\zeta)\subset
  \tilde{u}^{-1}( \mathcal{B}_{4r_1}(\tilde{u}(\zeta)))\).
Note that by definition \( S_{r_1}(\zeta)\) is connected, and
  hence contained in a connected component of \(\tilde{u}^{-1}(
  \mathcal{B}_{4r_1}(\tilde{u}(\zeta)))\).
Also recall that by definition, \(\widetilde{S}_{4r_1}(\zeta)\) is the
  connected component of
  \(\tilde{u}^{-1}(\mathcal{B}_{4r_1}(\tilde{u}(\zeta)))\) containing
  \(\zeta\).
Consequently, to show that \(S_{r_1}(\zeta)\subset
  \widetilde{S}_{4r_1}(\zeta)\) it is sufficient to show that they have
  non-empty intersection, however this is obvious since they each contain
  \(\zeta\) by definition.
Thus we have shown
  \begin{equation*}                                                       
    S_{r_1}(\zeta)\subset 4\widetilde{S}_{r_1}(\zeta).
    \end{equation*}
Making use of this and equation (\ref{EQ_Area_containtment_estimate}),
  we have
  \begin{align*}                                                          
    \textstyle{\frac{1}{2}}&\geq 
      {\rm Area}_{\tilde{\gamma}}(\widetilde{S}_{4r_1}(\zeta))\\
    &\geq 
      {\rm Area}_{\tilde{\gamma}}(S_{r_1}(\zeta))\\
    &\geq 
      {\textstyle \frac{1}{2}}{\rm Area}_{\gamma}(S_{r_1}(\zeta)),
    \end{align*}
  where the final inequality follows from combining Lemma
  \ref{LEM_gamma_tilde_estimates} -- particularly inequality
  (\ref{EQ_gamma_tilde_estimate_1}) -- together with equation
  (\ref{EQ_hausdorff_measure})
  from Section \ref{SEC_co-area} which expresses the Hausdorff measure
  in terms of coordinates and a Riemannian metric.
Since \(\gamma=u^*g\), the above estimate can be restated as
  \begin{equation*}                                                       
    {\rm Area}_{\gamma}(S_{r_1}(\zeta))\leq 1,
    \end{equation*}
  which is also the desired inequality (\ref{EQ_key_area_estimate}).
Other than providing the proof of Lemma
  \ref{LEM_small_cannot_intersect_both_one_boundaries}, this completes the
  proof of Proposition
  \ref{PROP_local_local_area_bound_2}.
\end{proof}
In order to complete the proof of Proposition
  \ref{PROP_local_local_area_bound_2} it only remains to prove Lemma
  \ref{LEM_small_cannot_intersect_both_one_boundaries}, which we do at
  present.

\begin{proof}[Proof of Lemma
  \ref{LEM_small_cannot_intersect_both_one_boundaries}]
Because \(\partial_1^\pm\Sigma_k = q_{2k+1\pm 1}([0, T_{2k+1\pm 1}])\),
  and because of inequality (\ref{EQ_side_length_estimate}), it follows
  that
  \begin{equation*}                                                       
    \max\Big({\rm length}_{\tilde{\gamma}}(\partial_1^-\Sigma_k),
    {\rm length}_{\tilde{\gamma}}(\partial_1^+\Sigma_k)\Big) \leq
    2^7(a_1-a_0).
    \end{equation*}
In order to derive a contradiction, let us assume that
  \(\widetilde{S}_{4r_1}(\zeta)\cap \partial_1^- \Sigma_k\neq \emptyset \)
  and \(\widetilde{S}_{4r_1}(\zeta)\cap \partial_1^+ \Sigma_k\neq \emptyset
  \).
Consequently, there exists a piece-wise smooth path \(\beta\colon[0,
  1]\to\mathbb{R}\times M\) such that \(\beta(0)=\tilde{u}(z_k^-)=u(z_k^-)\),
 \(\beta(1)=\tilde{u}(z^-_{k+1})=u(z^-_{k+1})\), and
  \begin{equation}\label{EQ_another_length_estimate}                      
    {\rm length}_g(\beta([0, 1]))\leq 2^8(a_1-a_0) + 8r_1.
    \end{equation}
Indeed, this path is described by following
  \(\tilde{u}(\partial_1^-\Sigma_k)\) from \(\tilde{u}(z_k^-)\)
  to some point inside \(\mathcal{B}_{4r_1}(\tilde{u}(\zeta))\),
  following the unique geodesic to \(\tilde{u}(\zeta)\), following
  a geodesic in \(\mathcal{B}_{4r_1}(\tilde{u}(\zeta))\) to a
  point in \(\tilde{u}(\partial_1^+\Sigma_k)\), and then following
  \(\tilde{u}(\partial_1^+\Sigma_k)\) to \(\tilde{u}(z_{k+1}^-)\).
In light of inequality (\ref{EQ_another_length_estimate}), and
  the fact that \(r_1\leq 2^{-10} (a_1-a_0) \), we conclude that \({\rm
  length}_g(\beta([0, 1]))\leq 2^9(a_1-a_0)\), and hence
  \begin{equation}\label{EQ_inequality_124}                               
    {\rm dist}_g\big(\tilde{u}(z_k), \tilde{u}(z_{k+1})\big)\leq
    2^9(a_1-a_0).
    \end{equation}
We now parametrize \(\partial_0^-\Sigma_k\) by \(\phi:[0, T]\to
  \partial_0^-\Sigma_k\) so that \(\|\phi'\|_{\tilde{u}^*g}=1\) and
  \(\phi(0)=z_k^-\) and \(\phi(T)=z_{k+1}^-\).
Define \(\tilde{q}:=\tilde{u}\circ \phi\). 
Observe that \(\tilde{q}\) is a unit speed parametrization of a path
  between \(\tilde{u}(z_k^-)\) and \(\tilde{u}(z_{k+1}^-)\).
We also claim the following hold.
\begin{enumerate}                                                         
  \item 
    \(\lambda(\tilde{q}'(t))>0\) for all \(t\in [0, T]\),
  \item 
    \(r_0 \leq \int_{\tilde{q}}\lambda \leq 6r_0\)
  \item  
    \(\mu_{\tilde{q}^*g}^1(\{t\in [0, T]:
    \lambda(\tilde{q}'(t))<\frac{1}{2}\}) \leq  r_0\)
  \end{enumerate}
We take a moment to verify these properties. 
Recall that \(\partial_0^-\Sigma_k\subset \partial S\) and this is a
  connected component of the preimage of a regular value of the function
  \(a\circ u\).
Consequently we either have \(\lambda(\tilde{q}'(t))>0\) for all \(t\in
  [0, T]\) or we have \(\lambda(\tilde{q}'(t))<0\) for all \(t\in [0,
  T]\); inequality (\ref{EQ_some_integral_bounds}) then establishes the
  former holds.
The second property is simply a restatement of equation
  (\ref{EQ_some_integral_bounds}).
The third property follows from combining several observations, which
  we accomplish presently.
Because \(\tilde{u}\) is an \((\epsilon,
  \delta)\)-tame perturbation of \(u\), it follows that
  \(\tilde{u}\big|_{\partial S}=u\big|_{\partial S}\), and hence
  \(\tilde{u}\big|_{\partial_0^-\Sigma_k}=u\big|_{\partial_0^-\Sigma_k}\).
Because \(\phi\) is a \(\tilde{u}^*g\)-unit speed
  parametrization, and because of the first property, and because
  \(\tilde{u}(\partial_0^-\Sigma_k)=a_0\), it follows that pointwise we
  have \(\|u^*\lambda\|_{u^*g} = \lambda(\tilde{q}')\).
Combining these facts together with assumption \((LL\ref{EN_LL7})\)
  then yields
  \begin{align*}                                                          
    r_0
    &\geq
      \mu_{u^*g}^1\big(\{\zeta \in \partial S: a\circ u(\zeta)=a_0 \;
      \text{ and }\; \|u^* \lambda\|_{u^* g} < {\textstyle\frac{1}{2}}\}
      ) \\
    &\geq 
      \mu_{u^*g}^1\big(\{\zeta \in \partial_0^-\Sigma_k: \|u^*
      \lambda\|_{u^* g} < {\textstyle\frac{1}{2}}\} ) \\
    &=
      \mu_{u^*g}^1\big(\{t\in [0, T]: \lambda(\tilde{q}'(t)) <
      {\textstyle\frac{1}{2}}\} ),
    \end{align*}
  which is the claim of the third property. 
With these properties established, we now apply Lemma 
  \ref{LEM_small_radius}, which guarantees that
  \begin{equation}\label{EQ_inequality_123}                               
    {\rm dist}_{g}\big(\tilde{q}(0), \tilde{q}(T)\big)\geq
    {\textstyle\frac{1}{2}}r_0.
    \end{equation}
However,
  \begin{align*}                                                          
    {\textstyle\frac{1}{2}}r_0
    &\leq  
      {\rm dist}_{g}\big(\tilde{q}(0), \tilde{q}(T)\big) &\text{by
      }(\ref{EQ_inequality_123})\\
    &= 
      {\rm dist}_{g}\big(\tilde{u}(z_k^-),
      \tilde{u}(z_{k+1}^-)\big)&\text{since
      }\tilde{u}(z_k^-)=\tilde{q}(0)\text{ and
      }\tilde{u}(z_{k+1}^-)=\tilde{q}(T)\\
    &\leq 
      2^9(a_1-a_0)&\text{ by } (\ref{EQ_inequality_124})\\
    &\leq 
      2^9 2^{-11}r_0&\text{ by }(LL\ref{EN_LL4})\\
    &= 
      {\textstyle\frac{1}{4}}r_0
    \end{align*}
  which is the desired contradiction.  
This completes the proof of Lemma 
  \ref{LEM_small_cannot_intersect_both_one_boundaries}.
\end{proof}
%
\begin{lemma}[connected-local area bound for orbit cylinders]
  \label{LEM_local_local_orbit_cylinders}
  \hfill\\
Proposition \ref{PROP_local_local_area_bound_2} remains true when the
  assumption
  \begin{enumerate}[(LL1)]                                                
    \setcounter{enumi}{4}
    \item \(0< \int_{S} u^*\omega\leq
        r_0\big((a_1-a_0)^{-1}+10C_{\mathbf{h}}\big)^{-1} \)
    \end{enumerate}
  is weakened to the following
  \begin{enumerate}[(LL1')]                                               
  \setcounter{enumi}{4}
    \item \( \int_{S} u^*\omega\leq
        r_0\big((a_1-a_0)^{-1}+10C_{\mathbf{h}}\big)^{-1} \).
    \end{enumerate}
That is, we allow for the case that \(\int_S u^* \omega = 0\).
\end{lemma}
%
\begin{proof}
We begin by observing that we need only prove the case that \(\int_S u^*
  \omega = 0\).
Since \(u:S\to \mathbb{R}\times M\) is  pseudoholomorphic and \(\omega\)
  evaluates non-negatively on \(J\)-complex lines, it follows that in this
  case, there must exist a trajectory \(\gamma\colon\mathbb{R}\to M\) of the
  Hamiltonian vector field \(X_\eta\) with the property that \(u(S)\subset
  \mathbb{R}\times \gamma(\mathbb{R})\).

By property (LL\ref{EN_LL2}) we have \(a\circ u(\partial S) =
  \{a_0,a_1\}\), and by property (LL\ref{EN_LL3}) it follows that there are
  no critical points of \(a\circ u\) on \(\partial S\).
Consequently, \(u^*\lambda\big|_{\partial S}\) is non-vanishing, and thus
  there must exist 
  \begin{align*}                                                          
    0 < T = \int_{(a\circ u)^{-1}(a_1) \cap \partial S} u^*\lambda
    \end{align*}
  such that \(\gamma(0)=\gamma(T)\).
Moreover, associated to the covering map
  \begin{align*}                                                          
    &\Phi\colon \mathbb{R}\times (\mathbb{R}/T\mathbb{Z}) \to \mathbb{R}
      \times \gamma(\mathbb{R})\\
    &\Phi(s,t) = \big(s, \gamma(t)\big)
    \end{align*}
  there exists a lift \(\phi\colon S\to \mathbb{R}\times
  (\mathbb{R}/T\mathbb{Z})\) of \(u\colon S\to  \mathbb{R}\times
  \gamma(\mathbb{R}) \subset \mathbb{R}\times M\).
Equipping \(\mathbb{R}\times (\mathbb{R}/ T\mathbb{Z})\) with the almost
  complex structure \(J\partial_s = \partial_t\) makes \(\Phi\)
  pseudoholomorphic and hence \(\phi\) is pseudoholomorphic.
We then conclude from the maximum principle that \(\phi\) is an embedding
  of \(S\) into \([a_0, a_1]\times (\mathbb{R}/T\mathbb{Z})\).
Using \(\phi\) to pull back coordinates \((s,t)\), we then have \(u(s,t) =
  (s, \gamma(t))\).
Consequently for each \(\zeta\in S\) with 
  \begin{align*}                                                          
    \big|a\circ u(\zeta) - {\textstyle \frac{1}{2}}(a_1+a_0)\big| \leq
    {\textstyle \frac{1}{4}}(a_1 - a_0)
    \end{align*}
  we also have
  \begin{align*}                                                          
    {\rm Area}_{u^*g}\big(S_{r_1}(\zeta)\big) \leq  \pi r_1^2  \leq \pi
    \Big( 2^{-24} \frac{1}{100}\Big)^2 \leq 1,
    \end{align*}
  which is the desired conclusion and completes the proof of Lemma
  \ref{LEM_local_local_orbit_cylinders}.
\end{proof}

With the proof of Proposition \ref{PROP_local_local_area_bound_2}  and
  Lemma \ref{LEM_local_local_orbit_cylinders} established, we are now
  prepared to prove the main result of this section.

\setcounter{CurrentSection}{\value{section}}
\setcounter{CurrentTheorem}{\value{theorem}}
\setcounter{section}{\value{CounterSectionAsymptoticLocalLocalAreaBound}}
\setcounter{theorem}{\value{CounterTheoremAsymptoticLocalLocalAreaBound}}
\begin{theorem}[asymptotic connected-local area bound]\hfill\\
Let \((M, \eta)\) be a compact framed Hamiltonian manifold, and let \((J,
  g)\) be an \(\eta\)-adapted almost Hermitian structure on
  \(\mathbb{R}\times M\).
Then the positive constant \(r_1=r_1(M, \eta, J, g)\) guaranteed by
  Proposition \ref{PROP_local_local_area_bound_2} has the following
  additional significance.
For each generally immersed feral pseudoholomorphic curve \((u, S, j)\) in
  \(\mathbb{R}\times M\), there exists a compact set of the form
  \(K:=[-a_0, a_0]\times M\) with the property that for each \(\zeta\in
  S\) such that \(u(\zeta)\notin K\) we have
  \begin{equation*}                                                       
      {\rm Area}_{u^*g}\big(S_{r_1}(\zeta)\big)\leq 1;
    \end{equation*}
  here \(S_{r_1}(\zeta)\) is defined to be the connected component of
  \(u^{-1}(\mathcal{B}_{r_1}(u(\zeta)))\) containing \(\zeta\), and
  \(\mathcal{B}_{r_1}(p)\) is the open metric ball of radius \(r_1\)
  centered at the point \(p\in \mathbb{R}\times M\).
\end{theorem}
%
\begin{proof}
\setcounter{section}{\value{CurrentSection}}
\setcounter{theorem}{\value{CurrentTheorem}}

Before we begin, we note that our proof will primarily rely on Proposition
  \ref{PROP_local_local_area_bound_2} above to achieve the desired area bound.
In a sense, or goal is to show that outside a large compact set, a point
  in a feral curve lives inside an annulus which satisfies the hypotheses of
  Proposition \ref{PROP_local_local_area_bound_2}, and hence the desired
  area bound follows immediately.
As to be expected, the bulk of the work below is devoted to constructing
  such an annulus.

We begin by letting \(r_1\) denote the constant guaranteed by Proposition
  \ref{PROP_local_local_area_bound_2}, and we define
  \(\mathcal{R}^\pm\subset \mathbb{R}\) by 
  \begin{align*}                                                          
    {\mathcal R}^{\pm}=\{e\in {\mathbb R}\ |\ e\ \text{is regular for}\
    a\circ u\ \text{and}\ -a\circ u\}.
    \end{align*}
For notational convenience, we define
  \begin{align}\label{EQ_epsilon_not}                                     
    \epsilon_0:={\rm min}(C_{\mathbf{h}}^{-1}, r_0),
    \end{align}
  where \(C_{\mathbf{h}}\) is the ambient geometry constant given in
  Definition \ref{DEF_ambient_geometry_constant}, and \(r_0\) is the
  positive constant provided in Lemma \ref{LEM_small_radius}.
We also let \(\hbar=\hbar(M, \eta, J, g, 2^{-13}\epsilon_0, C_g)\) be the
  constant guaranteed by Theorem \ref{THM_energy_threshold} with the genus
  bound \(C_g := 0\).
For notational convenience, for each \(a_0\in \mathcal{R}^\pm\), we define
  \begin{equation*}                                                       
    S^{a_0} = (a\circ u)^{-1}\big((-a_0, a_0)\big)\qquad\text{and}\qquad
    \overline{S}^{a_0} = (a\circ u)^{-1}\big([-a_0, a_0]\big)
    \end{equation*}
Observe that there exists an \(a_0\in \mathcal{R}^\pm\) with the
  following properties.
\begin{enumerate}[(L1)]                                                   
  \item\label{EN_L1} 
    \(\int_{S\setminus S^{a_0}} u^*\omega \leq \min\big(\frac{1}{4},
    \frac{1}{2}\hbar, 2^{-19}\epsilon_0 r_0 , (2^{14}\epsilon_0^{-1}
    + 10 C_{\mathbf{h}})^{-1}\big)\)
  \item\label{EN_L2} 
    \({\rm Genus} (S) = {\rm Genus}(\overline{S}^{a_0})\)
  \item\label{EN_L3} 
    \(\#_{nc}(S\setminus S^{a_0})={\rm Punct}(S)\)
  \end{enumerate}
Recall the notion of \({\rm Punct}(S)\) is given by Definition
  \ref{DEF_generalized_punctures}, and \(\#_{nc}X\) denotes the number of
  path-connected components of \(X\) that are not compact.
Note that for any \(a_1\in \mathcal{R}^\pm\) for which \(a_1> a_0\),
  properties (L\ref{EN_L1}) - (L\ref{EN_L3}) hold even when \(S^{a_0}\)
  is replaced with \(S^{a_1}\).

Recall that our goal here is to show that for some sufficiently large
  \(c>0\), the portion of the pseudoholomorphic curve in the complement of
  \([-c, c]\times M\) satisfies the aforementioned connected-local area bound.
Strictly speaking, this breaks our problem up into two cases: the portion
  of curve in \((c, \infty)\times M\) and the portion in \((-\infty,
  -c)\times M\), however we shall henceforth only study the first case;
  the second is essentially identical.

Suppose \(a', b'\in \mathcal{R}^\pm\) so that \(a_0+\epsilon_0 < a' < b'\)
  and consider \(\overline{S}_{a'}^{b'}\), where here and throughout we
  use the notation
  \begin{align*}                                                          
    S_x^y:=(a\circ u)^{-1}((x,
    y))\qquad\text{and}\qquad\overline{S}_x^y:=(a\circ u)^{-1}([x, y]).
    \end{align*}
We characterize \(\partial \overline{S}_{a'}^{b'}\) by separating it into
  essential and inessential components.
More specifically we write \(\partial \overline{S}_{a'}^{b'} =
  \partial_{ess} \overline{S}_{a'}^{b'}\cup \partial_{ess}^\bot
  \overline{S}_{a'}^{b'}\) where \(\partial_{ess}
  \overline{S}_{a'}^{b'}\) consists of those connected components of
  \(\partial \overline{S}_{a'}^{b'}\)  which are contained in either
  non-compact connected components of \(S\setminus S_{a'}^{b'}\), or are
  contained in connected components of \(S\setminus (S_{a'}^{b'}\cup
  S_{-a_0}^{a_0})\) which have non-trivial intersection with both
  \(\overline{S}_{-a_0}^{a_0}\) and \(\overline{S}_{a'}^{b'}\).
More geometrically, if we think of \(\overline{S}_{-a_0}^{a_0}\) as being
  the core of \(S\), then the essential boundary components of
  \(\overline{S}_{a'}^{b'}\) are those which connect
  \(\overline{S}_{a'}^{b'}\) to the infinite positive end, or else they
  are boundary components of portions of curves which connect
  \(\overline{S}_{a'}^{b'}\) to the core of \(S\).

Observe that by definition of essential boundary components, we may
  cap off the inessential boundary components with the union of connected
  components of \(S\setminus (S_{-\infty}^{a_0} \cup S_{a'}^{b'})\)
  which satisfy the following conditions
  \begin{enumerate}                                                       
    \item 
      the connected component is compact
    \item 
      the connected component has empty intersection with
      \(\overline{S}_{-a_0}^{a_0}\)
    \end{enumerate} 
We will let \(\widetilde{S}_{a'}^{b'}\) denote the union of
  \(\overline{S}_{a'}^{b'}\) with the union of these specified capping
  components, so that \(\partial \widetilde{S}_{a'}^{b'}\) consists only
  of essential components.
Summarizing, we have constructed \(\widetilde{S}_{a'}^{b'}\) so that
  \begin{enumerate}[($\widetilde{S}$1)]                                   
    \item 
      \(\widetilde{S}_{a'}^{b'}\) is compact with \(a\circ u (\partial
      \widetilde{S}_{a'}^{b'})\subset \{a', b'\}\)
    \item 
      \(u(\widetilde{S}_{a'}^{b'})\subset(a_0, \infty)\times M\)
    \item 
      \(\partial \widetilde{S}_{a'}^{b'}\) is contained in the union
      of connected components of \(S\setminus (S_{-\infty}^{a_0}\cup
      S_{a'}^{b'})\) which are either non-compact, or have non-empty
      intersection with both \(\overline{S}_{-\infty}^{a_0}\) and
      \(\overline{S}_{a'}^{b'}\).
    \end{enumerate}
We now claim the following.

\begin{lemma}[short capping disks]
  \label{LEM_extension_is_small}
  \hfill\\
  \begin{equation*}                                                       
    \max \big(a'- \inf_{\zeta\in \widetilde{S}_{a'}^{b'}} a\circ
    u(\zeta), \sup_{\zeta\in \widetilde{S}_{a'}^{b'}} a\circ  u(\zeta)
    -b'\big) \leq 2^{-13}\epsilon_0.
    \end{equation*}
\end{lemma}
%

\begin{proof}
Suppose not.
For example, suppose
  \begin{equation} \label{EQ_extension_is_small}                          
    a'- \inf_{\zeta\in \widetilde{S}_{a'}^{b'}} a\circ u(\zeta)>
    2^{-13}\epsilon_0.
    \end{equation} 
Then there must exist a non-empty connected component \(\widehat{S}\)
  of \(S\setminus (\overline{S}_{-\infty}^{a_0}\cup S_{a'}^{\infty})\)
  with the following properties:
  \begin{enumerate}                                                       
    \item 
      \((u, \widehat{S}, j)\) is compact, connected, and generally
      immersed
    \item 
      \(a\circ u (S)\subset (a_0, \infty)  \)
    \item \
      \({\rm Genus}(\widehat{S}) =0 \)
    \item 
      \(a\circ u(\partial\widehat{S})=\{a_{min}+ c'\} \), 
    \end{enumerate}
   where \(a_{min}:=\inf_{\zeta\in \widehat{S}} a\circ u(\zeta)\),
   and \(c'> 2^{-13}\epsilon_0\).
We now apply Theorem \ref{THM_energy_threshold} with \(C_g=0\) and \(r =
  2^{-13}\epsilon_0\) to conclude
  \begin{equation}\label{EQ_omega_energy_contradiction}                   
    \int_{\widehat{S}}u^*\omega \geq \hbar, 
    \end{equation}
  where \(\hbar=\hbar(M, \eta, J, g, 2^{-13}\epsilon_0, 0)\).
However, equation (\ref{EQ_omega_energy_contradiction}) together with
  the fact that \(\widehat{S}\subset S\) and \(a\circ u(\widehat{S})\subset
  (a_0, \infty)\) contradicts the fact that \(a_0\) has been chosen so that
  \begin{equation*}                                                       
    0<\hbar \leq \int_{\widehat{S}}u^*\omega \leq \int_{S\setminus
    S_{-a_0}^{a_0}}u^*\omega \leq {\textstyle\frac{1}{2}}\hbar .
    \end{equation*} 
This shows that inequality (\ref{EQ_extension_is_small}) is impossible.
A similar argument shows that we must also have 
  \begin{equation*}                                                       
    \sup_{\zeta\in \widetilde{S}_{a'}^{b'}} a\circ u(\zeta)- b'\leq
    2^{-13}\epsilon_0.
    \end{equation*} 
This completes the proof of Lemma \ref{LEM_extension_is_small}
\end{proof}

In light of Lemma \ref{LEM_extension_is_small}, we conclude that another
  property of \(\widetilde{S}_{a'}^{b'}\) is the following.
\begin{enumerate}[($\widetilde{S}$1)]                                     
  \setcounter{enumi}{3}
    \item\label{EN_wS4} 
      \(\sup_{\zeta\in \widetilde{S}_{a'}^{b'}} a\circ u(\zeta)-
      \inf_{\zeta\in \widetilde{S}_{a'}^{b'}} a\circ u(\zeta)\leq b'-a'
      + 2^{-12}\epsilon_0\).
    \end{enumerate}

It will be useful to employ the following notation: If \(X\) is a
  topological space, then we let \({\rm Comp}(X)\) denote the set of
  connected components of \(X\).
It will also be useful to say that connected components \(L_1, L_2\in
  {\rm Comp}(\partial \overline{S}_{-a_0}^{a_0})\) are \emph{eventually
  connected} if there exists a connected component \(\check{S}\) of
  \(S\setminus S_{-a_0}^{a_0}\) for which \(L_1\cup L_2\subset \check{S}\).
We now note that because \(u:S\to \mathbb{R}\times M\) is a proper map
  and \(a_0\) is a regular value of both \(a\circ u\) and \(-a\circ u\),
  it follows that \({\rm Comp}(\partial \overline{S}_{-a_0}^{a_0})\)
  is finite.
Consequently, there exists \(a_1\in \mathbb{R}^+\) for which \(a_1>a_0\)
  and with the following property.
For each pair \(L_1, L_2\in {\rm
  Comp}(\partial\overline{S}_{-a_0}^{a_0})\) which are eventually
  connected, there exists a connected component \(\check{S}\)
  of \(\overline{S}_{-a_1}^{a_1}\setminus S_{-a_0}^{a_0}\) for which
  \(L_1\cup L_2\subset \check{S}\).
 
We henceforth consider \(\widetilde{S}_{a'}^{b'}\) with \(a'> 1+a_1\).
A consequence of this assumption is that each connected component of
  \(\widetilde{S}_{a'}^{b'}\) only has at most one bottom boundary
  component, and at most one top component.
We make this precise with the following two lemmas.

\begin{lemma}[bottom boundary is a circle or empty]
  \label{LEM_at_most_one_part_1}
  \hfill\\
Consider \(\widetilde{S}_{a'}^{b'}\) with \(a', b'\in \mathcal{R}^\pm\)
  and for which \(a'> 1+a_1\).
Then each connected component \(\check{S}\) of
  \(\widetilde{S}_{a'}^{b'}\) has the property that
  \begin{equation*}                                                       
    {\rm Comp}\big(\partial \check{S}\cap (a\circ u)^{-1}(a')\big)
    \leq  1.
    \end{equation*}
\end{lemma}
%
\begin{proof}
Suppose not; that is, suppose there exists a connected component
  \(\check{S}_1\) of \(\widetilde{S}_{a'}^{b'}\) for which
  \begin{equation*}                                                       
    {\rm Comp}\big(\partial \check{S}_1\cap (a\circ u)^{-1}(a')\big) \geq 2.
    \end{equation*}
Then by definition of \(a_1\) and the fact that \(a'>1 +a_1\), there
  exists \(c'\in \mathcal{R}^\pm\) for which \(a_0 < c' < a'\) and has
  the property that there exists a connected component \(\check{S}_2\)
  of \(\overline{S}_{c'}^{a'}\) for which
  \begin{equation*}                                                       
    \partial \check{S}_1\cap (a\circ u)^{-1}(a')\subset \check{S}_2.
    \end{equation*} 
We now define 
  \begin{align*}                                                          
    n_1^+ &= \#\big(\check{S}_1\cap (a\circ u)^{-1}(b')\big)\geq 0\\
    n_1^- &= \#\big(\check{S}_1\cap (a\circ u)^{-1}(a')\big)\geq 2\\
    n_2 &= \#\big(\partial \check{S}_2\big) - n_1^-\geq 1,
    \end{align*}
  where \(\#X\) denotes the number of connected components of \(X\).
Recall that for a compact two-dimensional surface \(S\) possibly with
  boundary, the Euler characteristic of \(S\) is given by
  \begin{equation*}                                                       
    \chi(S) = 2-2{\rm Genus}(S) - \#(\partial S),
    \end{equation*}
  and thus 
\begin{align*}                                                            
  \chi(\check{S}_1) &= 2 - 0 - (n_1^+ + n_1^-)\\
  \chi(\check{S}_2) &= 2 - 0 - (n_2 + n_1^-).
  \end{align*}
We now make two observations; first \(\check{S}_1\cap \check{S}_2=
  \partial \check{S}_1\cap (a\circ u)^{-1}(a')\).
Second, \(\check{S}_1 \cup \check{S}_2\subset S\) is a compact two
  dimensional surface with boundary, and which satisfies
  \begin{equation*}                                                       
    \#(\partial (\check{S}_1\cup\check{S}_2)) = n_1^+ + n_2.
    \end{equation*}
We then compute
  \begin{align*}                                                          
    \chi(\check{S}_1\cup \check{S}_2) 
    &= 
      \chi(\check{S}_1) + \chi(\check{S}_2)\\
    &= 
      \big(2 - 0 - (n_1^+ + n_1^-)\big) + \big(  2 - 0 - (n_2 +
      n_1^-)\big) \\
    &= 
      2 - 2 ( n_1^--1) - (n_1^+ +n_2)   \\
    &= 
      2 - 2( n_1^- -1) -\#(\partial (\check{S}_1\cup\check{S}_2))\\
    &= 
      2 - 2{\rm Genus}(\check{S}_1\cup \check{S}_2) -\#(\partial
      (\check{S}_1\cup\check{S}_2)),
    \end{align*}
   and conclude that \({\rm Genus}(\check{S}_1\cup \check{S}_2)= n_1^-
   - 1>0\).
However, from this it immediately follows that \({\rm
  Genus}(S_{-a_0}^{a_0}) < {\rm Genus}(S)\) which is impossible by the
  definition of \(a_0\) and genus super-additivity, namely Lemma
  \ref{LEM_genus_addition}.
This is the desired contradiction which proves Lemma
  \ref{LEM_at_most_one_part_1}.
\end{proof}
%

\begin{lemma}[top boundary is a circle or empty]
  \label{LEM_at_most_one_part_2}
  \hfill\\
Consider \(\widetilde{S}_{a'}^{b'}\) with \(a', b'\in \mathcal{R}^\pm\)
  and for which \(a'> 1+a_1\).
Then each connected component \(\check{S}\) of
  \(\widetilde{S}_{a'}^{b'}\) has the property that
  \begin{equation*}                                                       
    {\rm Comp}\big(\partial \check{S}\cap (a\circ u)^{-1}(b')\big)
    \leq  1.
    \end{equation*}
\end{lemma}
%
\begin{proof}
Suppose not.  
Then by Definition \ref{DEF_generalized_punctures} and Remark
  \ref{REM_punct} we have \({\rm Punct}(S_{-b'}^{b'})> {\rm
  Punct}(S_{-a_0'}^{a_0'})\) which contradicts the fact that \({\rm
  Punct}(S_{-a_0}^{a_0}) = {\rm Punct}(S)\) and the fact that
  \begin{equation*}                                                       
    S_{b_0}^{b_1}\subset S_{b_0'}^{b_1'}\qquad\text{implies}\qquad{\rm
    Punct}(S_{b_0}^{b_1}) \leq{\rm Punct}(S_{b_0'}^{b_1'}).
    \end{equation*}
This contradiction completes the proof.
\end{proof}

With the above topological preliminaries out of the way, we now turn our
  attention to more measure theoretic preliminaries.
Indeed, for the remainder of the proof fix \(\zeta \in S\) such that
  \(a\circ u(\zeta)\geq a_1+2\).
We also note that to prove Theorem \ref{THM_local_local_area_bound}
  we must establish the existence of a compact set \(K\subset
  \mathbb{R}\times M\), which we can now define explicitly as
  \begin{align*}                                                          
    K:= [-a_1-2, a_1+2]\times M.
    \end{align*}
Next, we consider surfaces \(S_{b_-''}^{b_+''}\)  and
  \(S_{a_-''}^{a_+''}\) where \(a_-'', a_+'', b_-'', b_+''\in
  \mathcal{R}^\pm\) and
  \begin{align*}                                                          
    a\circ u(\zeta) + 2^{-14}\epsilon_0-2^{-18}\epsilon_0 \leq\;\;\;
      & b_+''\; \leq a\circ u(\zeta) + 2^{-14}\epsilon_0 \\
    a\circ u(\zeta) + 2^{-15}\epsilon_0 \leq\;\;\;  & b_-''\; \leq
      a\circ u(\zeta) + 2^{-15}\epsilon_0 +2^{-18}\epsilon_0\\
    a\circ u(\zeta) - 2^{-15}\epsilon_0-2^{-18}\epsilon_0 \leq\;\;\;
      & a_+''\; \leq a\circ u(\zeta) - 2^{-15}\epsilon_0 \\
    a\circ u(\zeta) - 2^{-14}\epsilon_0 \leq\;\;\;  & a_-''\; \leq
      a\circ u(\zeta) - 2^{-14}\epsilon_0 +2^{-18}\epsilon_0.
    \end{align*}

Observe that by definition, we have 
  \begin{equation*}                                                       
    b_+'' - b_-'' \geq  2^{-14}\epsilon_0 - 2^{-18}\epsilon_0 -
    2^{-15}\epsilon_0 - 2^{-18}\epsilon_0\geq 2^{-16}\epsilon_0
    \end{equation*}
  and similarly
  \begin{equation}\label{EQ_a_minus_a}                                    
    a_+'' - a_-'' \geq 2^{-16}\epsilon_0.
    \end{equation}
Likewise, it is elementary to establish
  \begin{align}\label{EQ_a_and_b_estimates}                               
    2^{-14}\epsilon_0 \leq b_-'' - a_+''\qquad\text{and}\qquad b_+'' -
    a_-'' \leq 2^{-13}\epsilon_0.
    \end{align}
It is perhaps worth explicitly observing that
\begin{align*}                                                            
  a_-'' < a_+'' < a\circ u(\zeta) <  b_-'' < b_+''.
  \end{align*}
We define 
  \begin{equation*}            
    \mathcal{Q}_{u, \frac{1}{2},
    r_0}(S_{b_-''}^{b_+''}):=\Big\{t\in [b_-'', b_+'']\cap
    \mathcal{R}^\pm : \mu_{u^*g}^1\big(\{\zeta\in (a\circ
    u)^{-1}(t):\|(u^*\lambda)_{\zeta}\|_{u^*g}<
    {\textstyle\frac{1}{2}}\}\big)> r_0\Big\}
    \end{equation*}
  and 
  \begin{equation*}                                                       
    \mathcal{Q}_{u, \frac{1}{2},
    r_0}(S_{a_-''}^{a_+''}):=\Big\{t\in [a_-'', a_+'']\cap
    \mathcal{R}^\pm : \mu_{u^*g}^1\big(\{\zeta\in (a\circ
    u)^{-1}(t):\|(u^*\lambda)_{\zeta}\|_{u^*g}<
    {\textstyle\frac{1}{2}}\}\big)> r_0\Big\}.
    \end{equation*}
By construction we have \(\mathcal{Q}_{u, \frac{1}{2},
  r_0}(S_{a_-''}^{a_+''})\subset [a_-'', a_+'']\) and  \(\mathcal{Q}_{u,
  \frac{1}{2}, r_0}(S_{b_-''}^{b_+''})\subset [b_-'', b_+'']\).
However, as a consequence of Lemma \ref{LEM_Q_has_small_measure},
  we also have
  \begin{align*}                                                          
    \mu\big(\mathcal{Q}_{u, \frac{1}{2}, r_0}(S_{a_-''}^{a_+''})\big)
    &\leq
      \frac{1}{r_0(1-(\frac{1}{2})^2)}\int_{S_{a_-''}^{a_+''}}u^*\omega\\
    &=
      \frac{4}{3r_0}\int_{S_{a_-''}^{a_+''}}u^*\omega \\
    &\leq 
      \frac{4}{3r_0}\int_{S\setminus S_{-a_0}^{a_0}}u^*\omega\\
    &\leq 
      2^{-18}\epsilon_0,
    \end{align*} 
  where to obtain final inequality we have made use of property
  \((L\ref{EN_L1})\).
In other words, the subset
  \begin{equation*}                                                       
    \mathcal{Q}_{u, \frac{1}{2}, r_0}(S_{a_-''}^{a_+''})\subset [a_-'',
    a_+'']
    \end{equation*}
    satisfies
\begin{align*}                                                            
  \mu\big(\mathcal{Q}_{u, \frac{1}{2}, r_0}(S_{a_-''}^{a_+''})\big)
  \; \; \leq \; \;
  2^{-18}\epsilon_0
  \;\;<\;\; 
  2^{-16}\epsilon_0 
  \; \;\leq\; \;
  \mu\big([a_-'', a_+'']\big),
  \end{align*}
  where we have made use of inequality (\ref{EQ_a_minus_a}).
We conclude that there exists \(a'\in \mathcal{R}^\pm\cap [a_-'',
  a_+'']\setminus \mathcal{Q}_{u, \frac{1}{2}, r_0}(S_{a_-''}^{a_+''})\),
  and similarly there exists \(b'\in \mathcal{R}^\pm\cap [b_-'',
  b_+'']\setminus \mathcal{Q}_{u, \frac{1}{2}, r_0}(S_{b_-''}^{b_+''})\).
In other words, there exists \(a'\in [a_-'', a_+'']\) and \(b'\in
  [b_-'', b_+'']\) which are each regular values of \(a\circ u\), and
  \begin{equation}\label{EQ_meausre_estimate}                             
    \mu_{u^*g}^1\big(\{\zeta'\in (a\circ
    u)^{-1}(a'):\|u^*\lambda\|_{u^*g}<{\textstyle\frac{1}{2}}\}\big)\leq
    r_0
    \end{equation}
  and
  \begin{equation}\label{EQ_meausre_estimate_A}                           
    \mu_{u^*g}^1\big(\{\zeta'\in (a\circ
    u)^{-1}(b'):\|u^*\lambda\|_{u^*g}<{\textstyle\frac{1}{2}}\}\big)\leq
    r_0.
    \end{equation}
Importantly, we henceforth assume that \(a', b'\in \mathcal{R}^\pm \) have
  been fixed so that equation (\ref{EQ_meausre_estimate}) and equation
  (\ref{EQ_meausre_estimate_A}) are true, and so that
  \begin{align*}                                                          
    a_-''\; \leq\; a' \;\leq a_+''\; \leq\; a\circ u(\zeta)\; \leq \;
    b_-''\; \leq\;  b' \; \leq \; b_+''.
    \end{align*}
It then follows from equation (\ref{EQ_a_and_b_estimates}) that
  \begin{align*}                                                          
    2^{-14}\epsilon_0\; \; \leq\; \;  b_-'' - a_+'' \;\;\leq\;\; b' - a'
    \;\; \leq\;\; b_+'' - a_-''\;\; \leq\;\; 2^{-13}\epsilon_0,
    \end{align*}
  or for clarity,
  \begin{align}\label{EQ_b_minus_a}                                       
    2^{-14}\epsilon_0\; \leq\; b' - a'\;\leq \; 2^{-13}\epsilon_0.
    \end{align}

With these measure theoretic preliminaries out of the way, we can now
  complete the proof of Theorem \ref{THM_local_local_area_bound}.
Consider the surface \(S_{a'}^{b'}\), and more importantly,
  its extension \(\widetilde{S}_{a'}^{b'}\).
Furthermore, we let \(\widetilde{S}_{a'}^{b'}(\zeta)\) denote the
  connected component of \(\widetilde{S}_{a'}^{b'}\) containing \(\zeta\).
We list some properties of \(\widetilde{S}_{a'}^{b'}(\zeta)\) which
  have already been established. 
The properties (T1)--(T7) are listed in such a way that they can be
  compared with the hypotheses (LL1)--(LL7) of Proposition
  \ref{PROP_local_local_area_bound_2}.
We note that (T6) is still empty and will be filled during the discussion.
\begin{enumerate}[(T1)]                                                   
  \item\label{EN_T1_A} 
    \(\widetilde{S}_{a'}^{b'}(\zeta)\) is homeomorphic to either
    a sphere, a disk, or an annulus; this follows from Lemma
    \ref{LEM_at_most_one_part_1} and Lemma \ref{LEM_at_most_one_part_2},  
    which guarantee that \(\widetilde{S}_{a'}^{b'}(\zeta)\)
    has at most two boundary components, together with the fact that
    \(\widetilde{S}_{a'}^{b'}(\zeta)\subset S\setminus S_{-a_0}^{a_0}\),
    which, by definition of \(a_0\) and Lemma \ref{LEM_genus_addition},
    guarantees \({\rm Genus}(\widetilde{S}_{a'}^{b'}(\zeta))=0\).
  \item\label{EN_T2_A} 
    \(a\circ u(\partial \widetilde{S}_{a'}^{b'}(\zeta)) \subset \{a',
    b'\}\) with \(2^{-14}\epsilon_0 \leq b'- a'\); this follows from
    equation (\ref{EQ_b_minus_a}).
  \item\label{EN_T3} 
    \(\{\zeta'\in \widetilde{S}_{a'}^{b'}(\zeta): a\circ u(\zeta') \in
    \{a', b'\}\text{ and } d(a\circ u)(\zeta')=0\}=\emptyset\); this
    follows since \(a'\) and \(b'\) are regular values of \(a\circ u\).
  \item\label{EN_T4} 
    \(\sup_{\zeta'\in \widetilde{S}_{a'}^{b'}} a\circ u(\zeta') -
    \inf_{\zeta'\in \widetilde{S}_{a'}^{b'}} a\circ u(\zeta')\leq
    2^{-11}\epsilon_0\);  this follows from property
    (\(\widetilde{S}\)\ref{EN_wS4}) combined with inequality
    (\ref{EQ_b_minus_a}):
    \begin{align*}                                                        
      \sup - \inf 
      &\leq 
	b' - a' +2^{-12}\epsilon_0\\
      &\leq 
	2^{-13}\epsilon_0 + 2^{-12}\epsilon_0\\
      &\leq 
	2^{-11}\epsilon_0.
      \end{align*}
    which is the desired inequality.
  \item\label{EN_T5} 
    \(\int_{\widetilde{S}_{a'}^{b'}}u^*\omega \leq r_0\big(
    (b'-a')^{-1} +10 C_{\mathbf{h}})^{-1} \); this follows from the fact
    that \( \widetilde{S}_{a'}^{b'}(\zeta) \subset S\setminus
    S_{-a_0}^{a_0} \), the definition of \(a_0\), property (L\ref{EN_L1}),
    and the fact that \(b'- a' \geq 2^{-14}\epsilon_0\)
  \item\label{EN_T6} --- See the following discussion.
  \item\label{EN_T7} 
    \(\mu_{u^* g}^1\big(\{\zeta'\in \partial
    \widetilde{S}_{a'}^{b'}(\zeta): a\circ u(\zeta') = a'\text{
    and } \|u^*\lambda\|_{u*g} < \frac{1}{2}\}\big) \leq r_0\); this
    follows from our definition of \(a'\), and specifically equation
    (\ref{EQ_meausre_estimate}).
  \end{enumerate}
We also claim that 
  \begin{equation}\label{EQ_deep_interior}                                
    \big|a\circ u (\zeta) - {\textstyle\frac{1}{2}}(b'+a')\big| \leq
    {\textstyle\frac{1}{4}}(b'-a').
    \end{equation}
Before justifying inequality (\ref{EQ_deep_interior}) it may be helpful to
  recall that we have defined \(\epsilon_0=\min(C_{\mathbf{h}}^{-1},
  r_0)\).
We note that from properties (T\ref{EN_T1}) - (T\ref{EN_T7})
  and equation (\ref{EQ_deep_interior}), to apply Proposition
  \ref{PROP_local_local_area_bound_2} (or Lemma
  \ref{LEM_local_local_orbit_cylinders} in the case that
  \(\int_{\widetilde{S}_{a'}^{b'}}u^*\omega = 0\)), it is sufficient to
  establish that \(\widetilde{S}_{a'}^{b'}(\zeta)\) is an annulus with
  image of one boundary component in \(\{a'\}\times M\) and the other in
  \(\{b'\}\times M\), where the condition (T\ref{EN_T6}) given by 
  \begin{equation*}                                                       
    \int_{(a\circ u)^{-1}(a')\cap \partial \widetilde{S}_{a'}^{b'}}
    u^*\lambda \geq 100 r_0.
    \end{equation*}
We will do this momentarily, however at present we establish inequality
  (\ref{EQ_deep_interior}).
To that end, we begin by letting
  \begin{align*}                                                          
    b' &= a\circ u(\zeta) + 2^{-15}\epsilon_0 + \epsilon_b 
      &\text{ with }\;\;0\leq \epsilon_b \leq 2^{-15}\epsilon_0\\
    a' &= a\circ u(\zeta) - 2^{-15}\epsilon_0 - \epsilon_a 
      &\text{ with }\;\;0\leq \epsilon_a \leq 2^{-15}\epsilon_0
    \end{align*}
  and we define \(\epsilon_c:=\max(\epsilon_b, \epsilon_a)\).
We then observe that 
  \begin{equation*}                                                       
    |\epsilon_b -\epsilon_a| \leq \epsilon_c \leq  2^{-15}\epsilon_0 \leq
    2^{-15}\epsilon_0 + {\textstyle\frac{1}{2}}(\epsilon_b +\epsilon_a).
    \end{equation*}
  so that
  \begin{equation}\label{EQ_LL1}                                          
    {\textstyle\frac{1}{2}}|\epsilon_b -\epsilon_a| \leq
    {\textstyle\frac{1}{4}} \big(2^{-14}\epsilon_0 + (\epsilon_b
    +\epsilon_a)\big).
    \end{equation}
However, 
  \begin{equation}\label{EQ_LL2}                                          
    \big| a\circ u(\zeta) - {\textstyle\frac{1}{2}}(b'+a')\big| =
    {\textstyle\frac{1}{2}}|\epsilon_b -\epsilon_a|
    \end{equation}
  and
  \begin{equation}\label{EQ_LL3}                                          
    {\textstyle\frac{1}{4}}(b'-a') = {\textstyle\frac{1}{4}}(2^{-14}
    \epsilon_0 + \epsilon_b +\epsilon_a).
    \end{equation}
Combining equations (\ref{EQ_LL1}) - (\ref{EQ_LL3}) then establishes
  inequality (\ref{EQ_deep_interior}).
To complete the proof of Theorem \ref{THM_local_local_area_bound},
  we break the problem into cases.\\

\emph{Case I:} \(a\circ u(\partial \widetilde{S}_{a'}^{b'}(\zeta))
  \neq \{a', b'\}\).
In this case we will assume that \(a'\notin a\circ u(\partial
  \widetilde{S}_{a'}^{b'}(\zeta))\); the case that \(b'\notin a\circ
  u(\partial \widetilde{S}_{a'}^{b'}(\zeta))\) follows in identical
  fashion.
Next we define the surface \(\check{S}:=\widetilde{S}_{a'}^{b'}\cap
  (a\circ u)^{-1}[a' - 2^{-13}\epsilon_0, b']\).
We note that as a consequence of Lemma \ref{LEM_extension_is_small},
  \(\check{S}\) is indeed a smooth surface, possibly with smooth boundary,
  and \(a\circ u(\partial\check{S}) \subset \{ b'\}\).
Furthermore, 
  \begin{align*}                                                          
    \sup_{\zeta' \in \check{S}} a\circ u(\zeta') - \inf_{\zeta' \in
      \check{S}} a\circ u(\zeta') 
    &\leq 
      b'-a' + 2^{-13}\epsilon_0\\
    &\leq 
      2^{-12}\epsilon_0.
    \end{align*}
As a consequence of Theorem \ref{THM_area_bound_estimate}, we then have
  \begin{equation}\label{EQ_main_area_bound_1}                            
    {\rm Area}_{u^*g}(\check{S}) = \int_{\check{S}} u^*(da\wedge \lambda +
    \omega) \leq e^{C_{\mathbf{h}}
    2^{-12}\epsilon_0}\int_{\check{S}}u^*\omega \leq
    {\textstyle\frac{1}{2}}e^{C_{\mathbf{h}} 2^{-12}\epsilon_0}\leq 1.
    \end{equation}
Recall that \(r_1 = 2^{-24}\epsilon_0\),  and
  \(S_{r_1}(\zeta)\) is defined to be the connected component of
  \(u^{-1}(\mathcal{B}_{r_1}(u(\zeta)))\) containing \(\zeta\), where
  \(\mathcal{B}_{r_1}(p)\) is a metric ball in \(\mathbb{R}\times M\)
  of radius \(r_1\) centered at the point \(p\).
Writing \(a^\dag=a\circ u(\zeta)\), we then have
  \begin{align*}                                                          
    S_{r_1}(\zeta)\;\;
    &\subset\;\; 
      S_{a^\dag - 2^{-24}\epsilon_0}^{a^\dag +
      2^{-24}\epsilon_0}(\zeta)\\
    &\;\;\subset 
      S_{a^\dag - 2^{-15}\epsilon_0}^{a^\dag +
      2^{-15}\epsilon_0}(\zeta)\\
    &\;\;\subset 
      S_{a'}^{b'}(\zeta)\\
    &\;\;\subset 
      \check{S},
    \end{align*}
  where we have let \(S_{c_1}^{c_2}(\zeta)\) denote the connected
  component of \(S_{c_1}^{c_2}\) containing \(\zeta\).
Combining this containment with (\ref{EQ_main_area_bound_1}) then yields
  \begin{equation*}                                                       
    {\rm Area}_{u^*g}(S_{r_1}(\zeta))\leq 1, 
    \end{equation*}
  which is the desired inequality.

\emph{Case II:} \(a\circ u(\partial \widetilde{S}_{a'}^{b'}(\zeta))
  =\{a', b'\}\).
We break this into two further sub-cases.

\emph{Case IIa:} \(\int_{(a\circ u)^{-1}(a')\cap \partial
  \widetilde{S}_{a'}^{b'}} u^*\lambda \geq 100r_0\).
In this case we see that \(\widetilde{S}_{a'}^{b'} \) must be an annulus,
  with the image of one boundary component in \(\{a'\}\times M\) and the
  other in \(\{b'\}\times M\), furthermore by assumption \(\int_{(a\circ
  u)^{-1}(a')\cap \partial \widetilde{S}_{a'}^{b'}} u^*\lambda \geq
  100r_0\), and hence by the remarks immediately following the statements
  of properties (T\ref{EN_T1}) - (T\ref{EN_T7}), we may apply Proposition
  \ref{PROP_local_local_area_bound_2} (or Lemma
  \ref{LEM_local_local_orbit_cylinders} as appropriate), which precisely
  guarantees that
  \begin{equation*}                                                       
    {\rm Area}_{u^*g}(S_{r_1}(\zeta))\leq 1, 
    \end{equation*}
  which is the desired inequality.

\emph{Case IIb:} \(\int_{(a\circ u)^{-1}(a')\cap \partial
  \widetilde{S}_{a'}^{b'}} u^*\lambda \leq 100r_0\).
This case has more in similarity with Case I than Case IIa, in the sense
  that we will estimate area directly rather than invoke Proposition
  \ref{PROP_local_local_area_bound_2}.
We begin by claiming that \(S_{r_1}(\zeta)\subset S_{a'}^{b'}(\zeta)\);
  the proof is identical to that of Case I.
Moreover we have
  \begin{equation*}                                                       
    S_{r_1}(\zeta)\subset S_{a'}^{b'}(\zeta) \subset
    \overline{S}_{a'}^{b'}\cap  \widetilde{S}_{a'}^{b'}(\zeta) =:
    \check{S},
    \end{equation*}
We also define
  \begin{align*}                                                          
    \widehat{S}:=\overline{S}_{-\infty}^{a'}  \;\cap \;
    \widetilde{S}_{a'}^{b'}(\zeta)
    \end{align*}
In this way we have 
  \begin{align*}                                                          
    \overline{S}_{-\infty}^{b'}\cap \widetilde{S}_{a'}^{b'}(\zeta) =
    \check{S}\cup \widehat{S}\qquad\text{and}\qquad \partial \check{S}
    \cap \partial \widehat{S} = \partial \widehat{S},
    \end{align*}
  and in fact if we define 
  \begin{align*}                                                          
    \Lambda =  \partial \widetilde{S}_{a'}^{b'}(\zeta) \cap (a\circ
    u)^{-1}(a')
    \end{align*}
  we then have
  \begin{align*}                                                          
    \widetilde{S}_{a'}^{b'}(\zeta) \cap (a\circ u)^{-1}(a') \; \;= \; \;
    \partial \widehat{S}\; \cup\; \Lambda\qquad\text{and}\qquad
    \partial\widehat{S}\; \cap \;\Lambda = \emptyset.
    \end{align*}

With these preliminary definitions out of the way, we recall that by the
  hypotheses of Case IIb, we have
  \begin{align*}                                                          
    \int_{\Lambda} u^*\lambda \leq 100r_0,
    \end{align*}
  and we aim to estimate that area of \(\check{S}\), since
  \(S_{r_1}(\zeta) \subset \check{S}\).
Our technique will be to employ Theorem 
  \ref{THM_area_bound_estimate}, but first we must estimate the quantity
  \(\int_{\partial \widehat{S}\cup \Lambda} u^*\lambda\).
To that end, we employ Lemma \ref{LEM_extension_is_small} and Theorem
  \ref{THM_area_bound_estimate} in regards to
  \(\int\tilde{\alpha}\) estimates\footnote{Recall that
  \(\widetilde{\alpha} = -(\tilde{u}^*da)\circ \tilde{\jmath} \), and when
  the pseudoholomorphic map is unperturbed, we simply have
  \(\widetilde{\alpha} = -(u^*da)\circ j = u^*\lambda.\)}  to obtain,
  \begin{equation*}                                                       
    \int_{\partial \widehat{S}}u^*\lambda \leq C_{\mathbf{h}}
    e^{C_{\mathbf{h}} 2^{-13}\epsilon_0} \int_{\widehat{S}}u^*\omega.
    \end{equation*}
Consequently,
  \begin{align}                                                           
    \int_{(a\circ u)^{-1}(a')\cap \partial S_{a'}^{b'}(\zeta)} u^*\lambda
    &=
    \int_{\partial \widehat{S}} u^*\lambda + \int_{\Lambda}
    u^*\lambda\notag
    \\
    &\leq\notag
    C_{\mathbf{h}} e^{C_{\mathbf{h}} 2^{-13}\epsilon_0}
    \int_{\widehat{S}}u^*\omega+
    100r_0
    \\
    &\leq\notag
    C_{\mathbf{h}} e^{ 2^{-13}}
    \int_{\widehat{S}}u^*\omega+
    100r_0
    \\
    &\leq\label{EQ_boundary_integral_estimate}
    2C_{\mathbf{h}} 
    \int_{\widehat{S}}u^*\omega+
    100r_0
    \end{align}
  where the second inequality follows from equation
  (\ref{EQ_epsilon_not}), and the third inequality follows from the fact
  that \(e^{(2^{-13})} \leq 1 + \frac{1}{8} \leq 2\).
With this estimate in hand, we now apply Theorem
  \ref{THM_area_bound_estimate} to estimate the area of \(\check{S}\).
\begin{align*}                                                            
  {\rm Area}_{u^*g}(\check{S}) 
  &=\int_{\check{S}} u^*(da\wedge \lambda +\omega)
  \\
  &\leq
  \Big(C_{\mathbf{h}}^{-1} \int_{\Lambda \cup \partial \widehat{S}}
  u^*\lambda + \int_{\check{S}} u^*\omega \Big)
  \big(e^{C_{\mathbf{h}}(b'-a')} -1 \big) +\int_{\check{S}} u^*\omega
  \\
  &\leq
  \Big(C_{\mathbf{h}}^{-1} \int_{\Lambda \cup \partial \widehat{S}}
  u^*\lambda + \int_{\check{S}} u^*\omega \Big)
  \big(e^{C_{\mathbf{h}}2^{-13}\epsilon_0} -1 \big) +\int_{\check{S}}
  u^*\omega
  \\
  &\leq
  C_{\mathbf{h}}^{-1} \int_{\Lambda \cup \partial \widehat{S}}
  u^*\lambda + 2\int_{\check{S}} u^*\omega 
  \\
  &\leq
  C_{\mathbf{h}}^{-1} \Big(
    2C_{\mathbf{h}} \int_{\widehat{S}}u^*\omega+ 100r_0
  \Big) + 2\int_{\check{S}} u^*\omega 
  \\
  &\leq
  100r_0C_{\mathbf{h}}^{-1} + 2\int_{\widetilde{S}_{a'}^{b'}} u^*\omega
  \leq
  C_{\mathbf{h}}^{-1} + \frac{1}{2}\leq 
  \frac{1}{20} +\frac{1}{2}
  \leq 1,
  \end{align*}
  where the second inequality follows from equation (\ref{EQ_b_minus_a}),
  the third inequality follows from equation (\ref{EQ_epsilon_not}) and
  the fact that \(e^{(2^{-13})}\leq  2\), the fifth inequality the fact
  that \(\check{S}\cup \widehat{S} \subset \widetilde{S}_{a'}^{b'}\subset
  S\setminus S_{-a_0}^{a_0}\) and property (L\ref{EN_L1}), the sixth
  inequality follows from the fact that \(r_0 \leq \frac{1}{100}\) as
  guaranteed by Lemma \ref{LEM_small_radius} and property (L\ref{EN_L1})
  again, and the seventh inequality follows from the definition of the
  ambient geometry constant provided in Definition
  \ref{DEF_ambient_geometry_constant}.
Recall that we have already established that \(S_{r_1}(\zeta)\subset
   \check{S}\), and hence we have
\begin{equation*}                                                         
     {\rm Area}_{u^*g}(S_{r_1}(\zeta)) \leq   {\rm Area}_{u^*g}(\check{S})\leq 1,
   \end{equation*}
which is the desired estimate, and completes Case IIb.  

Since cases I, IIa, and IIb exhaust all possibilities, we conclude
   that
  \begin{align*}                                                            
    {\rm Area}_{u^*g}(S_{r_1}(\zeta))\leq 1
    \end{align*}
  for all \(\zeta\) such that \(a\circ u(\zeta) \geq a_1 +2\).
Recall we have defined the compact set \(K=[-a_1-2, a_1+2]\times M\), and
  hence this completes the proof Theorem
  \ref{THM_local_local_area_bound}.
\end{proof}

\subsection{Proof of Theorem \ref{THM_curv_bound}: Asymptotic
  Curvature Bound}
  \label{SEC_asymptotic_curvature_bound}\hfill \\

The main purpose of this section is to prove the following result.

\setcounter{CurrentSection}{\value{section}}
\setcounter{CurrentTheorem}{\value{theorem}}
\setcounter{section}{\value{CounterSectionAsymptoticCurvatureBound}}
\setcounter{theorem}{\value{CounterTheoremAsymptoticCurvatureBound}}
\begin{theorem}[asymptotic curvature bound]\hfill\\
Let \((M, \eta)\) be a compact framed Hamiltonian manifold, and let \((J,
  g)\) be an \(\eta\)-adapted almost Hermitian structure on
  \(\mathbb{R}\times M\).
For each feral pseudoholomorphic curve \(\mathbf{u}= (u, S, j,
  \mathbb{R}\times M, J, \mu, D)\), there exists a compact set of the form
  \(K:=[-a_2, a_2]\times M\), and positive constant \(C_\kappa=C_\kappa(M,
  \eta, J, g)\) with the following significance.
First, the restricted map
  \begin{equation*}                                                       
      u:S\setminus u^{-1}(K)\to \mathbb{R}\times M 
    \end{equation*}
  is an immersion. 
Second, for each \(\zeta\in S\setminus u^{-1}(K)\) we have 
  \begin{equation*}                                                       
      \|B_u(\zeta)\| \leq C_\kappa
    \end{equation*}
  where \(B_u(\zeta)\) is the second fundamental form of the immersion
  \(u\) evaluated at the point \(\zeta\).
\end{theorem}
%
\begin{proof}
\setcounter{section}{\value{CurrentSection}}
\setcounter{theorem}{\value{CurrentTheorem}}
Suppose not.
Then there exists a sequence of points \(\zeta_k\in S\)  with the property
  that \(|a\circ u(\zeta_k)|\to \infty\) for which either
  \(Tu_{\zeta_k}=0\) for all \(k\in \mathbb{N}\) or else
  \(\|B_u(\zeta_k)\|\to \infty\).
Without loss of generality, we will assume \(a\circ u(\zeta_k)\to
  \infty\) monotonically; the case \(a\circ u(\zeta_k)\to -\infty\) is
  essentially identical.
Recall that Theorem \ref{THM_local_local_area_bound} guarantees that
  there exists an \(a_0>0\) and an \(r_1=r_1(M, \eta, J, g)>0\) such that
  \begin{equation*}                                                       
    {\rm Area}_{u^*g}\big(S_{r_1}(\zeta)\big)\leq 1
    \end{equation*}
  for each \(\zeta\in S\) for which \(a\circ u(\zeta) \geq a_0\), and
  where \(S_{r_1}(\zeta)\) is the connected component of
  \(u^{-1}(\mathcal{B}_{r_1}(u(\zeta)))\) containing \(\zeta\); here
  \(\mathcal{B}_r(p)\) denotes an open metric ball in \(\mathbb{R}\times
  M\) of radius \(r_1\) centered at \(p\).
By increasing \(a_0\) if necessary, we may also assume that \(a_0\) and \(
  -a_0\) are regular values of \(a\circ u\),
  \begin{align*}                                                          
    {\rm Genus}(S) = {\rm Genus}\big( u^{-1}\big([-a_0, a_0]\times
    M\big)\big),
    \end{align*}
  and \({\rm Punct}(S)\) equals the number of non-compact
  path-connected components of the set \(S\setminus u^{-1}\big((-a_0,
  a_0)\times M\big)\).
With the abbreviation $S^{a_0}=u^{-1}\big((-a_0,a_0)\times M\big)$ we may also assume 
  \begin{align}\label{EQ_small_energy}                                   
    \int_{S\setminus S^{a_0}} u^*\omega \leq \min\Big(\frac{1}{4},
    \frac{1}{2}\hbar, 2^{-19}\epsilon_0 r_0 , (2^{14}\epsilon_0^{-1} + 10
    C_{\mathbf{h}})^{-1}\Big)
    \end{align}
  where \(\hbar=\hbar(M, \eta, J, g, 2^{-13}\epsilon_0, 0)\) is the
  constant guaranteed by Theorem \ref{THM_energy_threshold}, and
  \(\epsilon_0 = {\rm min}(C_{\mathbf{h}}^{-1}, r_0)\) as in equation
  (\ref{EQ_epsilon_not}), \(C_{\mathbf{h}}\) is the ambient geometry
  constant given in Definition \ref{DEF_ambient_geometry_constant}, and
  \(r_0\) is the positive constant provided in Lemma
  \ref{LEM_small_radius}.
Furthermore, by passing to a subsequence if necessary we may assume
  \(a\circ u(\zeta_1) > a_0+1\), \(u(\mu\cup D) \in [-a_0, a_0]\times M\),
  and that
  \begin{align}\label{EQ_large_gaps}                                      
    a\circ u(\zeta_{k+1}) - a \circ u(\zeta_k) \geq 10(1+r_1)
    \end{align}
  for all \(k\in \mathbb{N}\).
For notational convenience, we define
  \begin{equation*}                                                       
      S_k:=S_{r_1}(\zeta_k),
    \end{equation*}
  and we define the maps 
  \begin{equation}\label{EQ_no_nodes_6}                                   
      v_k:S_k\to [-1, 1]\times M\qquad\text{given by}\qquad v_k(\zeta):=
      {\rm Sh}_{a\circ u(\zeta_k) }\circ u(\zeta)
    \end{equation}
  where \({\rm Sh}_{(\cdot)}\) is the shift map 
  \begin{align*}                                                          
    &{\rm Sh}_{x}\colon \mathbb{R}\times M\to \mathbb{R}\times M\\
    &{\rm Sh}_{x}(a,p) = (a-x,p).
    \end{align*}
Next observe that by construction \(v_k(\zeta_k)\in \{0\}\times M\)
  for all \(k\) where \(M\) is compact.
We conclude that after passing to a further subsequence, still denoted
  with subscripts \(k\), we have convergence of the sequence of points
\begin{equation}\label{EQ_point_p}
  v_k(\zeta_k) \to p:=(0,p') \in \{0\}\times M,
  \end{equation}
  and
  \begin{equation*}                                                       
      \overline{\mathcal{B}}_{\frac{1}{2}r_1}(p) \subset
      \mathcal{B}_{r_1}(v_k(\zeta_k)),
    \end{equation*}
  where \(\overline{\mathcal{B}}_r(p)\) denotes the closed metric ball of
  radius \(r\) centered at \(p\).
For notational convenience we define \(W=(-1, 1)\times M\).
Next we observe that the sequence of pseudoholomorphic curves
  \((v_k, S_k, j_k, W, J, \emptyset, \emptyset )\) have uniformly bounded
  area, zero genus, \(\partial S_k = \emptyset\),
  \(v_k^{-1}(\overline{\mathcal{B}}_{\frac{1}{2}r_1}(p))\) is compact, and
  \(v_k(\zeta_k)\to p\).
We conclude from Theorem \ref{THM_target_local_gromov_compactness}
  (target-local Gromov compactness) that after passing to a subsequence,
  still denoted with subscripts \(k\), there exist compact surfaces with
  boundary \(\widetilde{S}_k\subset S_k\) with the property that
  \(v_k(S_k\setminus \widetilde{S}_k) \subset W\setminus
  \overline{\mathcal{B}}_{\frac{1}{4}r_1}(p)\) and with the property that
  the pseudoholomorphic curves
  \begin{align*}                                                          
    \tilde{\mathbf{v}}_k :=(\tilde{v}_k, \widetilde{S}_k, \tilde{j}_k, W,
    J, \emptyset, \emptyset)
    \end{align*}
  defined by \(\tilde{v}_k = v_k\big|_{\widetilde{S}_k}\) and
  \(\tilde{j}_k = j_k\big|_{\widetilde{S}_k}\), converge in a Gromov
  sense\footnote{See Definition \ref{DEF_gromov_convergence}.} to the
  pseudoholomorphic curve
  \begin{align*}                                                          
    \tilde{\mathbf{v}} :=(\tilde{v}, \widetilde{S}, \tilde{j}, W,
    J, \emptyset, \widetilde{D}).
    \end{align*}
In particular, there will exist decorations \(\tilde{r}\) for the nodal
  points \(\widetilde{D}\subset \widetilde{S}\) and diffeomorphisms
  \begin{align*}                                                          
    \phi_k\colon \widetilde{S}^{\widetilde{D}, \tilde{r}}\to
    \widetilde{S}_k
    \end{align*}
  for which \(\tilde{v}_k\circ \phi_k\to \tilde{v}\) in
  \(\mathcal{C}_{loc}^\infty(\widetilde{S}^{\widetilde{D},
  \tilde{r}}\setminus \cup_i \Gamma_i)\) where the \(\Gamma_i\) are the
  special circles obtained by blowing up the nodal points, and
  \(\tilde{v}_k\circ \phi_k\to \tilde{v}\) in
  \(\mathcal{C}^0(\widetilde{S}^{\widetilde{D},\tilde{r}})\).
As a final consequence of Theorem
  \ref{THM_target_local_gromov_compactness}, we note that \(\tilde{v}\) is
  an immersion along \(\partial \widetilde{S}\), and \(\tilde{v}(\partial
  \widetilde{S})\cap \overline{\mathcal{B}}_{\frac{1}{4}r_1}(p) =
  \emptyset\).
As a consequence of these facts, we see that \(p\in
  \tilde{v}(\widetilde{S})\).
Moreover, we define \(\hat{\zeta}_k\in \widetilde{S}\) so that
  \(\phi_k(\hat{\zeta}_k) = \zeta_k \in S\), and thus \(\tilde{v}_k\circ
  \phi_k(\hat{\zeta}_k)\to p\).
If needed, we then pass to a subsequence so that  \(\hat{\zeta}_k\to
  \hat{\zeta}_\infty\in \widetilde{S}\setminus \partial \widetilde{S}\).

In what follows, it will be convenient to have a bit more control over the
  \(\partial \widetilde{S}_k\).
As such, we choose a regular value \(r_2\in (0,
  \frac{1}{4}r_1]\) of the function
  \begin{align*}                                                          
    &\rho\colon \widetilde{S}\to \mathbb{R}
    \\
    &\rho(\zeta)= {\rm dist}_g\big(p, \tilde{v}(\zeta)\big),
    \end{align*}
  for which \(\tilde{v}(\widetilde{D}) \cap \partial
  \overline{\mathcal{B}}_{r_2}(p)=\emptyset\).
We then define \( \widehat{S} \subset
  \tilde{v}^{-1}(\overline{\mathcal{B}}_{r_2}(p)) \) to be the union of
  connected components of \(\tilde{v}^{-1}(\overline{\mathcal{B}}_{r_2}(p))
  \)  with the property that \(|\widehat{S}| := \widehat{S}/\sim\) is
  connected\footnote{
    Here \(\zeta\sim \zeta'\) for \(\zeta\neq \zeta'\) if and only if
    \(\{\zeta, \zeta'\}\subset \widetilde{D}\cap \widehat{S}\) forms a nodal
    pair.}
  and \(\zeta_\infty \in \widehat{S}\).
This allows us to define the pseudoholomorphic curve
  \begin{align*}                                                          
    \hat{\mathbf{v}} = \big(\hat{v},
    \widehat{S}, \hat{j}, W, J, \emptyset,
    \widehat{D}\big)
    \end{align*}
  where
  \begin{align*}                                                          
    \hat{v} = \tilde{v}\big|_{\widehat{S}},
    \qquad\hat{j}=\tilde{j}\big|_{\widehat{S}},\qquad
    \text{and}\qquad \widehat{D} = \widetilde{D}\cap
    \widehat{S}.
    \end{align*}
We then define 
  \begin{align*}                                                          
    \widehat{S}_k:=
    \phi_k\big(\widehat{S}^{\widehat{D},
    \tilde{\tilde{r}}}\big)\subset \widetilde{S}_k
    \end{align*}
  so that for the pseudoholomorphic curves
  \begin{align*}                                                          
    \hat{\mathbf{v}}_k = \big(\hat{v}_k,
    \widehat{S}_k, \hat{j}_k, W, J, \emptyset,
    \emptyset\big)
    \end{align*}
  defined from the \(\tilde{\mathbf{v}}_k\) via domain restriction, we
  have \(\hat{\mathbf{v}}_k\to \hat{\mathbf{v}}\) in a
  Gromov sense.
To proceed, we will need the following.

\begin{lemma}[properties of limit curve]
  \label{LEM_no_nodes}
  \hfill\\
The limit curve \((\hat{v}, \widehat{S}, \hat{j}, W, J, \emptyset,
  \widehat{D})\) is not nodal; that is, \(\widehat{D}=\emptyset\).
Moreover, the limit curve is generally immersed in the sense of Definition
  \ref{DEF_generally_immersed}.
\end{lemma}
%

We will postpone the proof of Lemma \ref{LEM_no_nodes} until a bit later
  because the proof distracts from the main argument.
For the moment then, we assume it is true, and hence \(\widehat{S}\) is
  connected.
Observe that as a consequence of Gromov convergence, Lemma
  \ref{LEM_no_nodes}, and our construction of the \(\hat{\mathbf{v}}\), it
  follows that 
  $$
  \hat{v}_k\circ \phi_k \to \hat{v}\ \ \text{in}\ \
  \mathcal{C}^\infty(\widehat{S}, [-1, 1]\times M).
  $$
Recall that as part of our argument to derive a contradiction,
  we have assumed that the \(\zeta_k\in S\)  have the property that
  \(|a\circ u(\zeta_k)|\to \infty\) and either \(Tu(\zeta_k)=0\) or else
  \(\|B_u(\zeta_k)\|\to \infty\).
We have also defined \(\hat{\zeta}_k \in \widehat{S}\) so that
  \(\phi_k(\hat{\zeta}_k) = \zeta_k\) and \(\hat{\zeta}_k\to
  \hat{\zeta}_\infty\).
As a consequence of our above definitions, we then have:
$$
\text{either}\ \ 
T\hat{v}_k(\zeta_k)=0\ \ \text{or else}\ \ \ \|B_{\hat{v}_k}(\zeta_k)\|\to
  \infty.
  $$
Note that in either case, we must have \(T\hat{v}(\hat{\zeta}_\infty)=0\).
Indeed, if \(T\hat{v}(\hat{\zeta}_\infty)\neq 0\), then
  \(\hat{v}\) is immersed in a neighborhood of \(\hat{\zeta}_\infty\), and
  hence \(\|T\hat{v}\|\) is bounded away from zero in a neighborhood of
  \(\hat{\zeta}_\infty\) and \(\|B_{\hat{v}}\|\) is bounded in a
  neighborhood of \(\hat{\zeta}_\infty\); making use of the fact that
  \(\hat{\zeta}_k\to \hat{\zeta}_\infty\) and \(\hat{v}_k\to \hat{v}\)
  in \(\mathcal{C}^\infty\) then would yield a contradiction.
Thus we have shown that \(T\hat{v}(\hat{\zeta}_\infty)=0\) as claimed.

Our next task is then to prove that in fact \(T\hat{v}(\hat{\zeta}_\infty)
  \neq 0\), which would then yield the desired contradiction to prove
  Theorem \ref{THM_curv_bound}.
To that end, recall that we have assumed that \(|a\circ u(\zeta_k)|\to
  \infty\), and by construction the \(\widehat{S}_k\subset S\) are all
  pairwise disjoint.
Moreover, because \(\int_S u^*\omega<\infty\), and \(\omega\) evaluates
  non-negatively on \(J\)-complex lines, it follows that
  \(\int_{\widehat{S}_k} \hat{v}_k^*\omega \to 0\), and hence
  \(\int_{\widehat{S}}\hat{v}^*\omega = 0\).
Also recall that as a consequence of Lemma \ref{LEM_no_nodes}, \(\hat{v}\)
  is generally immersed.
From these facts we conclude the following about the image of \(\hat{v}\):
  \begin{equation*}                                                       
    \hat{v}(\widehat{S}) \subset
    \overline{\mathcal{D}}:=\overline{\mathcal{B}}_{r_2}(p)\;\cap\;
    \big([-1, 1]\times \beta\big([-2r_2, 2 r_2]\big)\big)
    \end{equation*}
  where \(\beta\) is a solution to the differential equation  \(\beta'
  = X_{\eta}(\beta)\) with \((0, \beta(0))=p\).
We note that \(\overline{\mathcal{D}}\) is a holomorphically embedded
  disk.
As a consequence we can find a compact disk-like domain with smooth
  boundary \(\mathsf{D}\subset \mathbb{C}\), which satisfies \(0\in
  \mathsf{D}\setminus \partial \mathsf{D} \subset \mathbb{C}\), supporting
  a holomorphic diffeomorphism of the form
  \begin{equation*}                                                       
      \psi:\mathsf{D} \to \overline{\mathcal{D}}
      \qquad\text{given by}\qquad \psi(s, t) = \big(s, \beta(t)\big).
    \end{equation*}
We note that $\psi$  is also an isometric embedding with respect to the
  flat metric \(ds^2 +dt^2\) on \(\mathbb{C}\).
Recall by construction that 
  \begin{equation*}                                                       
    \hat{v}(\hat{\zeta}_\infty)=p\in \overline{\mathcal{D}}\qquad
    \text{with} \qquad T \hat{v}(\hat{\zeta}_\infty)= 0,
    \end{equation*}
  and hence
  \(\hat{v}:\widehat{S}\to \overline{\mathcal{D}}\) is a branched cover
  with \(\hat{\zeta}_\infty\) a branch point.
A consequence of target-local Gromov compactness, Theorem
  \ref{THM_target_local_gromov_compactness}, is that the map \(\hat{v}\) is
  immersed along \(\partial \widehat{S}\), and by Lemma 2.4.1 of \cite{MS}
  it follows that the set of critical points of the map \(\hat{v}\) is
  finite.
Because of the latter, we will assume \(r_2>0\) has been chosen
  sufficiently small so that \(\hat{\zeta}_\infty\) is the unique critical
  point of \(\hat{v}\colon \widehat{S}\to \overline{\mathcal{D}}\).
More specifically, we follow the trimming procedure to obtain
  \(\hat{\mathbf{v}}\) from \(\tilde{\mathbf{v}}\) but for which \(r_2\)
  chosen sufficiently small so as to meet our needs.
In either case, we do not introduce new notation to indicated this newly
  trimmed curve.
As a consequence of this construction, it is then an elementary
  exercise from complex variables to show that there exists
  complex coordinates \(z\) on \(\widehat{S}\) so that
  \begin{equation}\label{EQ_branch_degree}                                
    \psi^{-1}\circ \hat{v}(z) = z^{n_\dag}\;\;\text{with}\;\; n_\dag\geq
    2.
    \end{equation}
With this local patch of limit curve understood as a branched cover of a
  disk, we aim to use this structure and Gromov convergence, to back up in
  the sequence to study compact manifolds with boundary of the form
  \begin{equation*}                                                       
      \Sigma_k:= \{\zeta\in S: a\circ u(\zeta_k)-c_5\leq a\circ
      u(\zeta)\leq a\circ u(\zeta_k)+c_5\}.
    \end{equation*}
  for some small generic choice of \(c_5\).
Modulo the addition of some small ``inessential capping disks'' (defined
  below), and for \(n_\dag \geq 2\) as defined in equation
  (\ref{EQ_branch_degree}) we will show that there is a \(4n_\dag\)-gon
  neighborhood \(\Sigma_{k,0} \subset \Sigma_k\) of \(\hat{\zeta}_k\),
  which we will use to show \(\Sigma_k\) (or rather the capped surface
  \(\widetilde{\Sigma}_k\)) has negative Euler characteristic for all
  sufficiently large \(k\in \mathbb{N}\), and hence \(S\) has either
  infinitely many ends (which is impossible), or infinite genus (which is
  also impossible).
This will yield the desired contradiction.
It may be helpful to consider Figure \ref{FIG_four_n_gon}.

We now proceed with the details.
Recalling the point $p=(0,p')$ defined in (\ref{EQ_point_p}), we begin by
  defining the set \(E\subset \widehat{S}\) by
\begin{equation*}                                                         
    E:=\hat{v}^{-1}\big([0, 1]\times\{p'\}\big).
  \end{equation*}
\begin{figure}\label{FIG_four_n_gon}
  \includegraphics[scale=0.2]{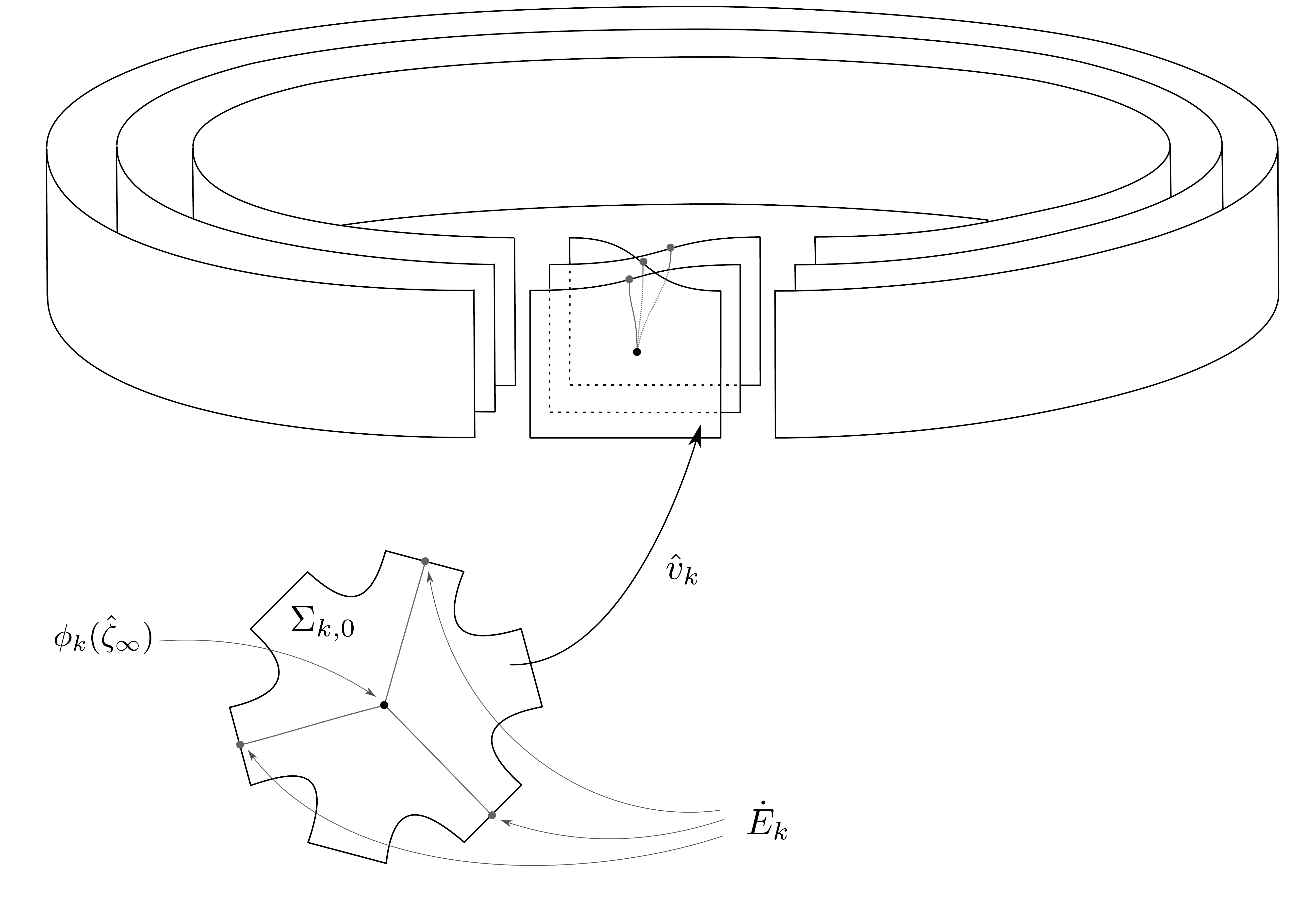}
  \caption{ 
    $\Sigma_{k,0}\subset \Sigma_k$. 
    Note that image of the  two-dimensional $\Sigma_{k,0}$ lies in a
      four-dimensional space (which is difficult to draw).
    Note that the important sequence $(\hat{\zeta}_k)$, which is not
      indicated in this figure, consists of points close to the points
      $\phi_k(\hat{\zeta}_k).$}
\end{figure}
We also fix \(c_5\in \mathbb{R}\) so that \(0< c_5 \leq \frac{1}{2}r_2\)
  with the property that the set \(\{a\circ u(\zeta_k) - c_5, a\circ
  u(\zeta_k) +c_5\}_{k\in \mathbb{N}}\) is contained in the set of regular
  values of the function \(a\circ u: S\to \mathbb{R}\).
Note that since \(c_5\neq 0\), we also have that \(\pm c_5\) are regular
  values \(a\circ \hat{v}:\widehat{S}\to \mathbb{R}\).
We then define the compact surfaces  with boundary
  \begin{equation*}                                                       
      \Sigma_k:= \{\zeta\in S: a\circ u(\zeta_k)-c_5\leq a\circ
      u(\zeta)\leq a\circ u(\zeta_k)+c_5\}.
    \end{equation*}
We also define important sub-surfaces of the \(\Sigma_k\) in the
  following manner.
Recall that \(\phi_k:\widehat{S}\to \widetilde{S}_k\subset S\), so we
  may regard the \(\phi_k\) as having image in \(S\).
Since we also have \(\Sigma_k\subset S\) by construction, we then define
  the sequence of sets \(\dot{E}_k\) by
  \begin{equation*}                                                       
    \dot{E}_k:= \phi_k(E)\cap \partial \Sigma_k.
    \end{equation*}
For all sufficiently large \(k\in\mathbb{N}\), we have by construction
  that the set \(\dot{E}_k\) consists of \(n_\dag\) points, where
  \(n_\dag\) is the natural number given in equation
  (\ref{EQ_branch_degree}).
We now define the manifolds \(\Xi_k:= \partial \Sigma_k\) and equip
  them with the metric \(\gamma_k= u^*g \big|_{\Xi_k}\).
Define \(\dot{F}_k \subset \Sigma_k\) to be the set of points
  given by \(\dot{F}_k:=\{\xi\in \Xi_k: {\rm dist}_{\gamma_k}(\xi,
  E_k)=\frac{1}{2}r_2\}\).
Observe that for all sufficiently large \(k\in
  \mathbb{N}\) the sets \(\dot{F}_k\) consist of \(2n_\dag\) points.
Define  \(L_k\subset \Sigma_k\subset S\) to be the (image of the)
  \(u^*g\)-gradient trajectories in \(\Sigma_k\) terminating 
  in \(\dot{F}_k\).
Define \(\Sigma_{k, 0}\) to be the closure of the connected component
  of \(\Sigma_k\setminus L_k\) which contains \(\zeta_k\).
Observe that by construction, for all \(k\in \mathbb{N}\) sufficiently
  large \(\Sigma_{k, 0}\) is a smooth manifold with piecewise smooth
  boundary and homeomorphic to a closed disk.
Furthermore, the single boundary component of \(\Sigma_{k,0}\) is
  comprised of \(4n_\dag\) smooth segments connected together at
  \(4n_\dag\) corners.
Moreover, \(2n_\dag\) of these smooth segments are \(u^*g\)-gradient flow
  lines of the function \(a\circ u\), and \(n_\dag\) segments are
  contained in the level set \((a\circ u)^{-1}(a\circ u(\zeta_k)+c_5)\),
  and \(n_\dag\) segments are contained in the level set \((a\circ
  u)^{-1}(a\circ u(\zeta_k)-c_5)\).
Later it will be useful to recall that for each connected component
  \(\Xi'\) of \((a\circ u)^{-1}(a\circ u(\zeta_k) \pm c_5)\) we have
  \begin{align}\label{EQ_integral_lambda_estimate_in_Sigma}               
    \int_{\Xi'} u^* \lambda \geq \textstyle{\frac{1}{2}} r_2
    \end{align}
  for all sufficiently large \(k\in \mathbb{N}\).

Next we define a set  \(\boldsymbol{\Delta}_k\) to consist of those
  compact connected components $\Delta'$ of the set \(S\setminus
  (\Sigma_k\setminus \partial \Sigma_k)\) which have empty intersection with
  \((a\circ u)^{-1}([-a_0, a_0])\).
We call these \emph{inessential caps}.
We note that as a consequence of equation (\ref{EQ_small_energy}),
  equation (\ref{EQ_large_gaps}), and Theorem \ref{THM_energy_threshold}, it
  follows that if \(\Delta'\in \boldsymbol{\Delta}_k\) with \(\Sigma_k \cap
  \Delta'\neq \emptyset\), then
  \begin{align}\label{EQ_oscillation}                                     
    a\circ u(\zeta_k) -c_5 - 1 \leq \inf_{\zeta\in \Delta'} a\circ
    u(\zeta) < \sup_{\zeta\in \Delta'} a\circ u(\zeta) \leq a\circ
    u(\zeta_k) +c_5 + 1.
    \end{align}
Consequently, each \(\Delta'\in\boldsymbol{\Delta}_k\) has non-empty
  intersection with at most one of the \(\Sigma_{k'}\), and thus we must
  have $k'=k$.
As in the proof of Theorem \ref{THM_local_local_area_bound}, we then
  define \(\widetilde{\Sigma}_k\) to be the union of \(\Sigma_k\) with all
  those elements of \(\boldsymbol{\Delta}_k\) which have non-empty
  intersection with \(\Sigma_k\).
We now claim the following.
                                                            
\begin{lemma}[inessential caps miss the $4n$-gon]
  \label{LEM_extensions_disjoint_from_4n_gon}
  \hfill \\
For all sufficiently large \(k\in \mathbb{N}\), we have
  \begin{equation*}                                                       
      \Sigma_{k, 0}\cap (\overline{\widetilde{\Sigma}_{k}\setminus
      \Sigma_k}) = \emptyset.
    \end{equation*}
In other words, when \(k\) is large enough, the \(4n_\dag\)-gons
  \(\Sigma_{k, 0}\) constructed above have empty intersection with the
  inessential caps added to the \(\Sigma_k\) to create
  \(\widetilde{\Sigma}_k\).
\end{lemma}
%

As above, we postpone the proof of Lemma 
  \ref{LEM_extensions_disjoint_from_4n_gon} for now and complete the proof
  of Theorem \ref{THM_curv_bound}.
To that end, we now claim the following.

\begin{lemma}[negative Euler characteristic]
  \label{LEM_neg_euler_char}
  \hfill\\
Letting \(\widetilde{\Sigma}_k(\zeta_k)\) denote the connected component
  of \(\widetilde{\Sigma}_k\) containing \(\zeta_k\), the following holds:
  \begin{equation*}                                                       
      \chi\big(\widetilde{\Sigma}_k(\zeta_k)\big) < 0
    \end{equation*}
  where \(\chi\) is the Euler characteristic.
\end{lemma}
%

Again, we postpone the proof of Lemma \ref{LEM_neg_euler_char} so as to
  complete the proof of Theorem \ref{THM_curv_bound}.
Recall the following terminology from the proof of Theorem
  \ref{THM_local_local_area_bound}.
We say that two connected components \(L_1\) and \(L_2\) of \((a\circ
  u)^{-1}(a_0)\) are \emph{eventually connected} provided that there exists
  a connected component \(\check{S}\) of \(S\setminus (a\circ
  u)^{-1}\big((-a_0, a_0)\big)\) for which \(L_1\cup L_2\subset
  \check{S}\).
Thus we fix \(a_1>a_0\) sufficiently large so that for each pair of
  connected components \(L_1\) and \(L_2\) of \((a\circ u)^{-1}(a_0)\) which
  are eventually connected, there exists a connected component of
  \((a\circ u)^{-1}\big([a_0, a_1]\big)\) which contains \(L_1\cup L_2\).
With \(a_1\) established as in the proof of Theorem
  \ref{THM_local_local_area_bound} (see after the proof of Lemma
  \ref{LEM_extension_is_small})  we now apply Lemma
  \ref{LEM_at_most_one_part_1} and Lemma \ref{LEM_at_most_one_part_2}
  which together guarantee that for all sufficiently large \(k\) we have
  that
  \begin{equation*}                                                       
      \#\big( \partial \widetilde{\Sigma}_k(\zeta_k)\big) \leq 2.
    \end{equation*}
That is, the number of connected components of \(\partial
  \widetilde{\Sigma}_k(\zeta_k)\) is at most two.
However by super-additivity of genus\footnote{See Lemma
  \ref{LEM_genus_addition}.}, we also have
  \begin{equation*}                                                       
      {\rm Genus}\big(\widetilde{\Sigma}_k(\zeta_k)\big) = 0
    \end{equation*}
  for all sufficiently large \(k\).
From these two observations, we deduce that
  \begin{equation*}                                                       
      \chi\big(\widetilde{\Sigma}_k(\zeta_k)\big) \geq 0,
    \end{equation*}
  but this contradicts Lemma \ref{LEM_neg_euler_char}.
This is the desired contradiction which completes the proof of
  Theorem \ref{THM_curv_bound} (modulo the proofs of Lemma
  \ref{LEM_no_nodes}, Lemma \ref{LEM_extensions_disjoint_from_4n_gon}, and
  Lemma \ref{LEM_neg_euler_char}).
\end{proof}

\begin{proof}[Proof of Lemma \ref{LEM_neg_euler_char}]
Recall that we must show that
  \(\chi\big(\widetilde{\Sigma}_k(\zeta_k)\big) < 0\) where
  \(\widetilde{\Sigma}_k(\zeta_k)\) is the connected component of
  \(\widetilde{\Sigma}_k\) containing \(\zeta_k\).
For the sake of notational convenience, we define
\begin{align*}                                                            
 \widetilde{\Sigma}_k':= \widetilde{\Sigma}_k(\zeta_k)\ \ \text{and}\ \
 \widetilde{\Sigma}_{k, 0}':={\Sigma}_{k, 0},
  \end{align*}
  and denote by  \(\widetilde{\Sigma}_{k, 1}', \ldots,
  \widetilde{\Sigma}_{k, n_k}'\) the connected components of
  \(\overline{\widetilde{\Sigma}_k'\setminus \widetilde{\Sigma}_{k,
  0}'}\).
We will need the following ad hoc definition.

\begin{definition}[surface with special boundary]
  \label{DEF_surface_with_boundary}
  \hfill \\
A surface with special boundary is a smooth compact real two-dimensional
  oriented manifold \(S\) with piece-wise smooth boundary and zero genus,
  which additionally has the following properties.
The boundary of \(S\) is the union of three sets denoted \(\partial_+
  S\), \(\partial_- S\), and \(\partial_1 S\) where
  \begin{enumerate}                                                       
      \item \(\partial_1 S\) is diffeomorphic to the disjoint union of
	finitely many compact intervals,
      \item each of \(\partial_- S\) and \(\partial_+ S\) is
	diffeomorphic to the disjoint union of finitely many compact
	intervals and circles \(\mathbb{R}/\mathbb{Z}\)
      \item \(\partial_- S\cap \partial_+ S = \emptyset\)
      \item each connected component of \(\partial_1 S\) intersects each
        of \(\partial_+ S\) and \(\partial_- S\) exactly once
      \item neither \(\partial_- S\) nor \(\partial_+ S\) is empty.
     \end{enumerate}
\end{definition}
%
It is worth noting that each of \(\widetilde{\Sigma}_{k, 0}',
  \widetilde{\Sigma}_{k, 1}', \ldots, \widetilde{\Sigma}_{k, n_k}'\) are
  surfaces with special boundary.
Next we need to understand the effect on the Euler characteristic of
  gluing such surfaces along their ``sides'' \(\partial_1 S\).
This is accomplished via the following.

\begin{lemma}[cuts increase Euler characteristic]
  \label{LEM_boundary_gluing}
  \hfill \\
Let \(S\) be a surface with special boundary as in Definition
  \ref{DEF_surface_with_boundary}.
Let \(L\subset S\) denote a smoothly embedded compact interval which
  transversely intersects each of \(\partial_- S\) and \(\partial_+ S\)
  precisely once and for which \(L\cap \partial_1 S= \emptyset\).
Let \(\Sigma\) denote the surface with special boundary obtained
  by cutting \(S\) along \(L\). 
More precisely, this means we consider \(S\setminus L\) as a Riemannian
  manifold with boundary, equipped with a metric (that is, a distance
  function) essentially defined as the length of the shortest path in
  \(S\setminus L\) connecting a pair of points, and then we define
  \(\Sigma\) to be the metric closure of \(S\setminus L\).
In this way, we have
  \begin{align*}                                                          
    S\neq \Sigma := \overline{S\setminus L} = (S\setminus L) \cup L_1\cup
    L_2
    \end{align*}
  where each \(L_i\) is diffeomorphic to \(L\), and \(L_1 \cup L_2\subset
  \partial_1 \Sigma\).
Then 
  \begin{equation*}                                                       
      \chi(\Sigma) = \chi(S) +1  
    \end{equation*}
  where \(\chi(M)\) is the Euler characteristic of \(M\).
\end{lemma}
 \begin{figure}[h]
\includegraphics[scale=0.42]{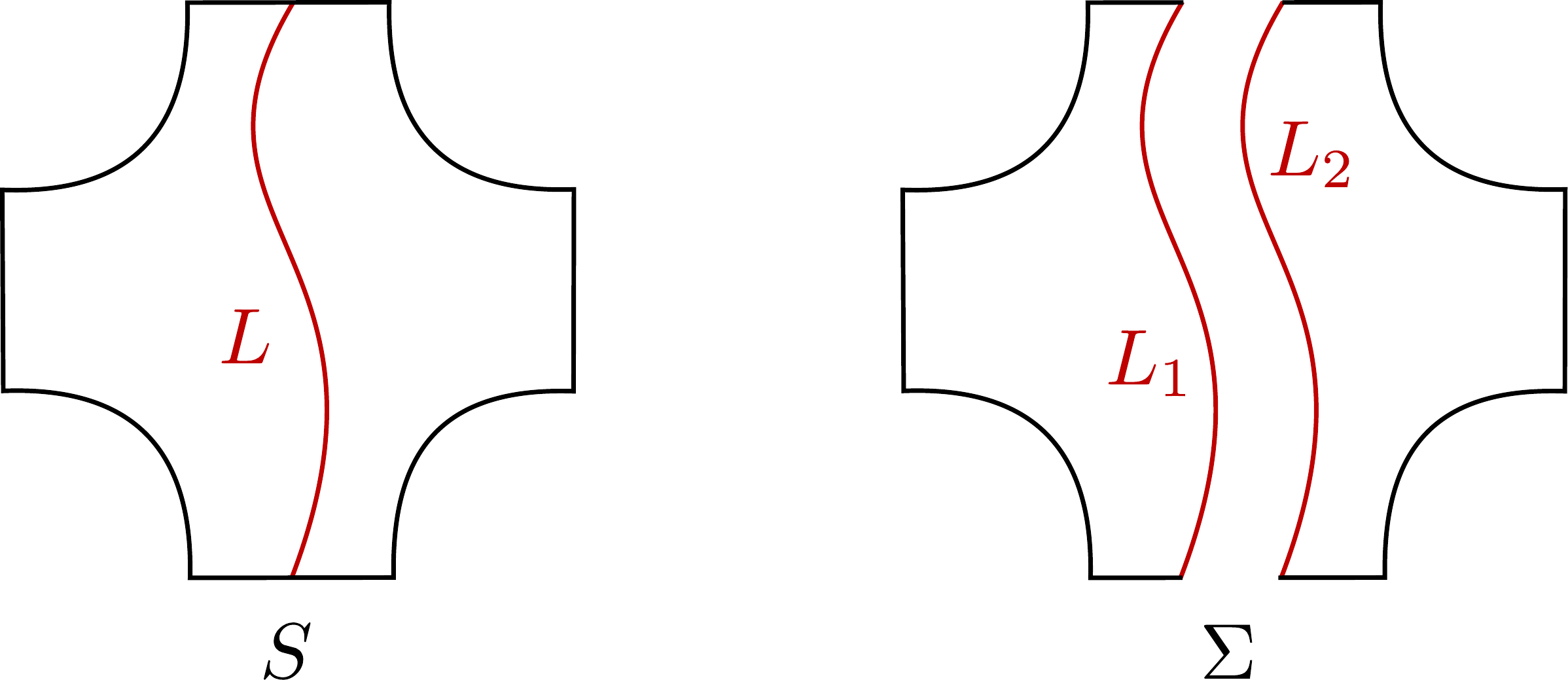}
\caption{Cutting $S$ to obtain $\Sigma$.}\label{FIG3}
\end{figure}
\begin{proof}
Recall the Gauss-Bonnet theorem for surfaces with boundary and corners,
  which states that
  \begin{equation*}                                                       
      \chi(S) = \frac{1}{2\pi}\Big(\int_S K_g\; d A + \int_{\partial S}
      \kappa_g ds + \sum_{i=1}^n \theta_i\Big)
    \end{equation*}
  where \(K_g\) is the Gaussian curvature, \(\kappa_g\) is the geodesic
  curvature, and the \(\theta_i\) are external angles at corners
  associated to a Riemannian metric \(g\).
To prove the lemma, first choose a metric on \(S\) for which \(L\),
  \(\partial_\pm S\),  and \(\partial_1 S\) are all geodesics, and then
  apply Gauss-Bonnet.
\end{proof}
We now observe that
  \begin{equation*}
      \widetilde{\Sigma}_k' = \widetilde{\Sigma}_{k, 0}'\cup
      \bigcup_{i=1}^{n_k} \widetilde{\Sigma}_{k, i}',
    \end{equation*}
  where each of the \(\widetilde{\Sigma}_{k, 0}', \ldots,
  \widetilde{\Sigma}_{k, n_k}'\) are special surfaces with boundary in
  the sense of Definition \ref{DEF_surface_with_boundary}.
Indeed, in this case we have
\begin{equation*}                                                         
    \partial_\pm\widetilde{\Sigma}_{k, i}' = \big(\partial
    \widetilde{\Sigma}_{k, i}'\big) \cap \Big( (a\circ u)^{-1}\big(a\circ
    u(\zeta_k) \pm c_5\big)\Big)
  \end{equation*}
for \(i\in \{0, \ldots, n_k\}\), and \(\partial_1 \widetilde{\Sigma}_{k,
  i}'\) consists of the remaining gradient-type boundary segments.
Moreover, we note that \(\widetilde{\Sigma}_k'\) can be obtained by
  gluing the \(\widetilde{\Sigma}_{k, 1}', \ldots, \widetilde{\Sigma}_{k,
  n_k}'\) components to  \(\widetilde{\Sigma}_{k, 0}'\) along appropriate
  gradient-type boundary segments.
Observe that by construction, we have \(\#(\partial_1
\widetilde{\Sigma}_{k, 0}') = 2n_\dag\) where \(n_\dag\geq 2\), and
\(n_k\leq n_\dag\).
Also note that 
  \begin{equation*}                                                       
      \chi(\widetilde{\Sigma}_{k, i}')\leq 1
    \end{equation*}
  for all \(k\in \mathbb{N}\) and \(i\in \{0, \ldots, n_k\}\).
We then apply Lemma \ref{LEM_boundary_gluing}, which guarantees the
  following
  \begin{align*}                                                          
      \chi(\widetilde{\Sigma}_k') &=\sum_{i=0}^{n_k}
	\chi(\widetilde{\Sigma}_{k, i}') - 2n_\dag\\
      &= 1 + \sum_{i=1}^{n_k} \chi(\widetilde{\Sigma}_{k, i}') - 2n_\dag\\
      &\leq 1 + n_k - 2n_\dag\\
      &\leq 1 - n_\dag\\
      &\leq -1.
    \end{align*}
This completes the proof of Lemma \ref{LEM_neg_euler_char}.  
\end{proof}

\begin{proof}[Proof of Lemma \ref{LEM_extensions_disjoint_from_4n_gon}]
Recall that we must show that for all sufficiently large \(k\in
  \mathbb{N}\), we have \(\Sigma_{k, 0}\cap
  (\overline{\widetilde{\Sigma}_{k}\setminus \Sigma_k}) = \emptyset\).
We note that as a consequence of our construction, specifically equation
  (\ref{EQ_oscillation}), we have
  \begin{equation*}                                                       
      \Big(\sup_{\zeta\in \widetilde{\Sigma}_k} a\circ u(\zeta)\Big)
      - \Big(\inf_{\zeta\in \widetilde{\Sigma}_k} a\circ u(\zeta)\Big)
      \leq 2(1+c_5)
    \end{equation*}
  and consequently the \(\widetilde{\Sigma}_k\) are all pairwise
  disjoint.
Since they are disjoint, we find that as \(k\to \infty\) we have
  \begin{equation*}                                                       
      \int_{\widetilde{\Sigma}_k}u^*\omega \to 0,
    \end{equation*}
  and hence another application of Theorem \ref{THM_energy_threshold},
  guarantees that 
  \begin{equation*}                                                       
      \Big(\sup_{\zeta\in \widetilde{\Sigma}_k} a\circ u(\zeta)\Big)
      - \Big(\inf_{\zeta\in \widetilde{\Sigma}_k} a\circ u(\zeta)\Big)
      \to 2c_5
    \end{equation*}
  and
  \begin{equation}\label{EQ_omega_to_zero}                                
	\int_{\widetilde{\Sigma}_k\setminus \Sigma_k}u^*\omega \to 0.
      \end{equation}
Now we note by construction that if \(\Sigma_{k, 0}\cap
  (\overline{\widetilde{\Sigma}_{k}\setminus \Sigma_k}) \neq \emptyset\),
  then \(\overline{\widetilde{\Sigma}_{k}\setminus \Sigma_k}\) and
  \(\Sigma_{k, 0}\) must overlap on a connected component of \(\Sigma_{k,
  0}\cap(a\circ u)^{-1}(\{a\circ u(\zeta_k)\pm c_5\})\subset \Sigma_{k,
  0}\).
Because the integral of \(\lambda\) along such components tend to
  \(r_2\), we can conclude that
  \begin{equation}\label{EQ_threshold_lambda_integral}                    
      \int_{\partial (\overline{\widetilde{\Sigma}_k\setminus \Sigma_k})
      }u^*\lambda \geq \textstyle{\frac{1}{2}} r_2
    \end{equation}
  for all sufficiently large \(k\); see for example equation
  (\ref{EQ_integral_lambda_estimate_in_Sigma}).
However, we then invoke Theorem \ref{THM_area_bound_estimate}, which
  guarantees that
  \begin{equation}\label{EQ_small_lambda_integral}                       
    \int_{\partial (\overline{\widetilde{\Sigma}_k\setminus
    \Sigma_k}) }u^*\lambda \leq \Big(C_{\mathbf{h}} \int_{
    \overline{\widetilde{\Sigma}_k\setminus \Sigma_k} }u^*\omega +
    0\Big) e^{C_{\mathbf{h}} \delta_k}
    \end{equation}
  where
  \begin{equation*}                                                       
      \delta_k:= \Big(\sup_{\zeta\in \widetilde{\Sigma}_k} a\circ
      u(\zeta)\Big) - \Big(\inf_{\zeta\in \widetilde{\Sigma}_k} a\circ
      u(\zeta)\Big) - 2c_5 \to 0.
    \end{equation*}
In light of equation (\ref{EQ_omega_to_zero}), we see that equation
  (\ref{EQ_small_lambda_integral}) contradicts equation 
  (\ref{EQ_threshold_lambda_integral}).
This contradiction then guarantees that indeed, 
  \begin{equation*}                                                       
      \Sigma_{k, 0}\cap (\overline{\widetilde{\Sigma}_{k}\setminus
      \Sigma_k}) = \emptyset
    \end{equation*}
  which completes the proof of Lemma 
  \ref{LEM_extensions_disjoint_from_4n_gon}.
\end{proof}

\begin{proof}[Proof of Lemma \ref{LEM_no_nodes}]
Recall that $W=(-1,1)\times M$ and that we must prove that the limit curve 
\begin{align*}                                                            
  \hat{\mathbf{v}}=(\hat{v}, \widehat{S}, \hat{j}, W, J,
  \emptyset, \widehat{D})
  \end{align*}
  is not nodal; that is, that \(\widehat{D}=\emptyset\). 
To that end, we suppose not, and we will derive a contradiction.
First however, we will need to briefly recall some facts about Gromov
  convergence.
In particular, the set of nodes is given by
  \(\widehat{D}=\{\overline{d}_1, \underline{d}_1, \ldots,
  \overline{d}_{n_d}, \underline{d}_{n_d}\}\), with \(\{\overline{d}_\nu,
  \underline{d}_\nu\}\subset \widehat{D}\) a nodal pair.
In particular, for each nodal pair  \(\{\overline{d}_\nu,
  \underline{d}_\nu\}\subset \widehat{D}\) we have
  \(\hat{v}(\overline{d}_\nu) = \hat{v}(\underline{d}_\nu)\).
Next we recall that \(\widehat{S}^{\widehat{D}}\) is defined to be the
  circle compactification of \(\widehat{S}\setminus \widehat{D}\) (or
  more specifically, an oriented blow-up at the points in
  \(\widehat{D}\)), and the newly added circles are denoted
  \(\overline{\Gamma}_\nu\) and \(\underline{\Gamma}_\nu\), which signifies
  that each circle \(\overline{\Gamma}_\nu\) is associated to a nodal point
  \(\overline{d}_\nu\) and similarly for \(\underline{\Gamma}_\nu\) and
  \(\underline{d}_\nu\).
The surface \(\widehat{S}^{\widehat{D}, \hat{r}}\) is then obtained
  by gluing pairs of circles \(\overline{\Gamma}_\nu\) and
  \(\underline{\Gamma}_\nu\) via the orientation reversing orthogonal maps
  \(r_\nu:\overline{\Gamma}_\nu\to \underline{\Gamma}_\nu\); here
  \(\hat{r}=\{r_1, \ldots, r_{n_d}\}\) is called a decoration.
It is useful to let \(\Gamma_\nu\subset \widehat{S}^{
  \widehat{D}, \hat{r}}\) denote the circle obtained by by gluing
  \(\overline{\Gamma}_\nu\) and \(\underline{\Gamma}_\nu\).
Also recall that the definition of Gromov convergence guarantees
  the existence of diffeomorphisms \(\phi_k:\widehat{S}^{ \widehat{D},
  \hat{r}}\to \widehat{S}_k\) with the property that
  \(\phi_k^*\hat{v}_k\to \hat{v}\) in \(\mathcal{C}^0\),
  \(\phi_k^*\hat{v}_k\to \hat{v}\) in
  \(\mathcal{C}_{loc}^\infty(\widehat{S}^{r, D}\setminus \cup_\nu
  \Gamma_\nu)\), \(\phi_k^*j_k\to j\) in
  \(\mathcal{C}_{loc}^\infty(\widehat{S}^{\widehat{D},
  \hat{r}}\setminus \cup_\nu \Gamma_\nu)\).  
In particular, this guarantees that there exists a sequence
  \(\epsilon_k\to 0\) with the property that
  \begin{equation}\label{EQ_no_nodes}                                     
      \phi_k^*\hat{v}_k(\Gamma_\nu)\subset
      \mathcal{B}_{\epsilon_k}(p_\nu),
    \end{equation}
  where
  \begin{equation*}                                                       
	p_\nu:=\hat{v}(\overline{d}_\nu)=\hat{v}(\underline{d}_\nu)
	=(a_\nu, q_\nu)\in W.  
	    \end{equation*}
We note that by the construction of \(\hat{v}\), the  \(p_{\nu}\) belong
  to \(W\); see the set-up before the initial statement of Lemma
  \ref{LEM_no_nodes}.

\begin{lemma}[some local properties]
  \label{LEM_some_local_properties}
  \hfill\\
Let \(\hat{\mathbf{v}}_k=(\hat{v}_k, \widehat{S}_k, \hat{j}_k, W,
  J, \emptyset, \emptyset)\) and \(\hat{\mathbf{v}}=(\hat{v},
  \widehat{S}, \hat{j}, W, J, \emptyset, \widehat{D})\) be as above
  with \(\hat{\mathbf{v}}_k\to \hat{\mathbf{v}}\) in a Gromov sense,
  and let \(\phi_k\colon \widehat{S}^{\widehat{D}, \hat{r}}\to
  \widehat{S}_k \) be the associated diffeomorphisms, and let
  \(\{\Gamma_1, \ldots, \Gamma_{n_d}\}\) be the collection of circles
  $\Gamma_{\nu}$ obtained by identifying
  $\overline{\Gamma}_{\nu}=\underline{\Gamma}_{\nu}$; see above.
Fix \(\Gamma\in \{\Gamma_1, \ldots , \Gamma_{n_d}\}\), and let
  \(\widehat{\Sigma}\) be the connected component of
  \(\widehat{S}^{\widehat{D},\hat{r}}\) containing \(\Gamma\).
Then \(\widehat{\Sigma}\setminus \Gamma\) is disconnected with connected
  components given by \(\widehat{\Sigma}_1\) and \(\widehat{\Sigma}_2\),
  and for all sufficiently large \(k\in \mathbb{N}\), there exists an
  \(\epsilon>0\), \(\zeta_1 \in \widehat{\Sigma}_1\), and \(\zeta_2\in
  \widehat{\Sigma}_2\) such that
  \begin{align} \label{EQ_goes_above}                                     
    a\circ \hat{v}_k\circ \phi_k(\zeta_i) - \sup_{\zeta \in \Gamma}
    a\circ \hat{v}_k\circ \phi_k(\zeta)  \geq \epsilon,
    \end{align}
  for each \(\zeta_i \in \{1, 2\}\).
\end{lemma}
%
\begin{proof}
Recall, for example from the proof of Lemma \ref{LEM_genus_addition}, 
  the removal of a loop from a surface either disconnects the surface or
  else reduces the genus.
However, by construction \({\rm Genus}(\widehat{S}_k) = 0\), so that the
  removal \(\phi_k(\Gamma)\) must disconnect \(\widehat{\Sigma}\) into
  \(\widehat{\Sigma}_1\) and \(\widehat{\Sigma}_2\) as required.
To proceed, we make the following claim.\\

\noindent{\bf Claim:} \(\partial \widehat{S}^{\widehat{D}, \hat{r}}
  \cap \widehat{\Sigma}_i \neq \emptyset\) for each \(i\in \{1, 2\}.\)\\

To see this, we first let \(\check{S}\) be a connected component of
  \(\widehat{S}^{\widehat{D},\hat{r}}\setminus \cup_\nu \Gamma_\nu
  \), and then observe that \(\hat{v}:\check{S}\to W\) is
  a pseudoholomorphic map, which is either a constant map or generally
  immersed.
Furthermore, because \(\int_S u^*\omega<\infty\) and because
  the \(\widehat{S}_k\subset S\) are disjoint, it follows that
  \(\int_{\widehat{S}_k}\hat{v}_k^*\omega \to 0\) and hence
  \(\int_{\check{S}}\hat{v}^*\omega=0\).
As a consequence of this, it follows that \(\hat{v}(\check{S})\)
  is contained in a patch of orbit cylinder.
By unique continuation\footnote{See Section 2.3 of \cite{MS}.} it then
  follows that \(\hat{v}(\check{S})\) is either a point in
  \(\mathcal{B}_{r_2}(p)\), or else it has nontrivial
  intersection with \(\partial \overline{\mathcal{B}}_{r_2}(p)\); and
  recall that \(\hat{v}^{-1}\big(\partial
  \overline{\mathcal{B}}_{r_2}(p\big)) = \partial \widehat{S}\).
Thus if \(\widehat{\Sigma}_i\cap \partial
  \widehat{S}^{\widehat{D},\hat{r}}=\emptyset\) for some \(i\in \{1, 2\}\)
  then we must have that \(\hat{v}\) restricted to any connected component
  of \((\widehat{S}^{\widehat{D},\hat{r}}\setminus \cup_k \Gamma_k)\cap
  \widehat{\Sigma}_i\) is a constant map.
However, letting \(\Sigma_i\) denote the image of
  \(\widehat{\Sigma}_i\) under the quotient map
  \(\widehat{S}^{\widehat{D},\hat{r}}\to S/(\overline{d}_k\sim
  \underline{d}_k)\), we see that \(\hat{v}\colon \Sigma_i\to W\) must
  give rise to a compact stable pseudoholomorphic curve, with no marked
  points, zero (arithmetic) genus, on which \(\hat{v}\) is constant on
  every component; but this is impossible.
We conclude that indeed, \(\widehat{\Sigma}_i\cap \partial
  \widehat{S}^{\widehat{D},\hat{r}}\neq \emptyset\).
This establishes the above claim.

To finish proving Lemma \ref{LEM_some_local_properties}, we let
  \(\{\underline{d},\bar{d}\}\) be the nodal pair associated to
  \(\Gamma\), and we let \((a',q')=\hat{v}(\underline{d}) =
  \hat{v}(\bar{d})\in W\).
With \(\hat{v}^{-1}\big(\partial \overline{\mathcal{B}}_{r_2}(p\big))
  = \partial \widehat{S}\),  \(\{\underline{d}, \bar{d}\}\cap \partial
  \widehat{S} = \emptyset\), \(\int_{\widehat{\Sigma}}\hat{v}^*\omega =
  0\), and because \(\widehat{\Sigma}_i\cap \partial
  \widehat{S}^{\widehat{D},\hat{r}}\neq \emptyset\)  for each \(i\in
  \{1,2\}\), it follows, since the images of the maps
  $\hat{v}|\widehat{\Sigma}_i$  are open  in an orbit cylinder, that there
  exists an \(\epsilon>0\) for which \((a'+2\epsilon, q') \in
  (\hat{v}(\widehat{\Sigma}_1)\cap \hat{v}(\widehat{\Sigma}_2))\setminus
  \partial \overline{\mathcal{B}}_{r_2}(p) \).  
Consequently, we define \(\zeta_1 \in \widehat{\Sigma}_1\) and
  \(\zeta_2\in \widehat{\Sigma}_2\) by fixing \(\zeta_1 \in
  \hat{v}^{-1}\big((a'+2\epsilon,q'))\cap \widehat{\Sigma}_1\) and
  \(\zeta_2\in \hat{v}^{-1}\big((a'+2\epsilon,q'))\cap
  \widehat{\Sigma}_2\).
Then by Gromov convergence, we have
  \begin{align*}                                                          
    \hat{v}_k\circ\phi_k(\zeta_i) \to (a'+2\epsilon, q')
    \end{align*}
  for each \(i\in \{1,2\}\).
Also as a consequence of Gromov convergence, there exist \(\epsilon_k\to
  0\) such that
  \begin{align*}                                                          
    \hat{v}_k \circ \phi_k(\Gamma)\subset \mathcal{B}_{\epsilon_k}(p')
    \end{align*}
  where \(p' = (a', q')= \hat{v}(\underline{d}) = \hat{v}(\bar{d})\).
Inequality (\ref{EQ_goes_above}) then follows immediately.
This completes the proof of Lemma \ref{LEM_some_local_properties}.
\end{proof}

We are now prepared to complete the proof of Lemma \ref{LEM_no_nodes}. 
 Indeed, as above we fix \(\Gamma\in \{\Gamma_1, \ldots, \Gamma_{n_d}\}\),
  and we will consider \(u^{-1}([a_0, \infty)\times M) \setminus
  \cup_{k=1}^\infty\phi_k(\Gamma)\). 
By construction, there exists sequences \(a_k\to \infty\) and
  \(\epsilon_k\to 0\) with \(\epsilon_k>0\) for which
  \begin{align*}                                                            
    u\circ \phi_k(z) \in [a_k-\epsilon_k,a_k+\epsilon_k]\times M
    \quad\text{for all }z\in \phi_k(\Gamma).
    \end{align*}
In view of (\ref{EQ_large_gaps}) we have the 
  inequality \(a_{k+1}-a_k\geq 10\) for large \(k\). 
Because \({\rm Punct}(S)<\infty\) and \({\rm Genus}(S)<\infty\), and
  because the \(\phi_k(\Gamma)\) are all pairwise disjoint, it follows that
  only finitely many connected components of \(u^{-1}([a_0,\infty)\times
  M)\setminus \cup_{k=1}^\infty\phi_k(\Gamma)\) have closure which is
  non-compact, and infinitely many which have compact closure.
We denote this infinite set of compact closures by
  \(\{\check{\Sigma}_{k'}\}_{k'\in \mathbb{N}}\), and observe that by
  construction it is the case that for each \(k'\in \mathbb{N}\) we have
  \(\partial \check{\Sigma}_{k'} \subset \cup_{k=1}^\infty
  \phi_k(\Gamma)\).
In fact, for each \(k'\in \mathbb{N}\) there exists a finite set
  \(F_{k'}\subset \mathbb{N}\) such that \(\partial \check{\Sigma}_{k'} =
  \cup_{k\in F_{k'}} \phi_k(\Gamma)\).
By virtue of the \(\check{\Sigma}_{k'}\) being compact and with boundary
  contained in \(\cup_{k=1}^\infty \phi_k(\Gamma)\), the application of
  Lemma \ref{LEM_some_local_properties} guarantees not only that the
  function \(a\circ u\) has an interior absolute maximum on each
  \(\check{\Sigma}_{k'}\), but also that for all sufficiently large \(k'\)
  the maximal value of \(a\circ u\) over \(\check{\Sigma}_{k'}\) is at
  least some uniform threshold amount larger than the maximal value of
  \(a\circ u\) along the boundary.
More precisely, there exists an \(\epsilon>0\) independent of \(k'\) such that
\begin{align*}                                                            
  \Big(\sup_{\zeta\in \check{\Sigma}_{k'}} a\circ u(\zeta)\Big) -
  \Big(\sup_{\zeta \in \partial \check{\Sigma}_{k'}} a \circ
  u(\zeta)\Big) \geq \epsilon.
  \end{align*}
Indeed, the uniformity of this inequality follows from Lemma
  \ref{LEM_some_local_properties}.
Now it  follows from Theorem \ref{THM_energy_threshold} that
  there exists an \(\hbar>0\) such that \(\int_{\check{\Sigma}_{k'}}u^*\omega
  \geq \hbar\) for all sufficiently large \(k'\in \mathbb{N}\), which then
  implies that \(\int_S u^*\omega = \infty\), which is impossible.
This is the desired contradiction which proves Lemma
  \ref{LEM_no_nodes}.
\end{proof}

\subsection{Proof of Theorem \ref{THM_existence}: Existence
  Workhorse}\label{SEC_workhorse} \hfill\\

This section is devoted to the proof of Theorem \ref{THM_existence}. 
The main argument is provided in Section
  \ref{SEC_existence_workhorse_proof}, however this relies on some
  preliminary notions established in Section
  \ref{SEC_workhorse_preliminaries}, and two technical results which are
  then proved in Section \ref{SEC_proof_of_feral_limit} and Section
  \ref{SEC_proof_of_bounded_intersections}.

\subsubsection{Preliminaries}\label{SEC_workhorse_preliminaries}
Recall that a gradient flow line of a smooth, but possibly degenerate,
  real-valued function defined on a closed manifold \(N\) need not have a
  unique point as its \(\omega\)-limit set.
That is to say, in general it may be the case that for a gradient flow
  line \(\gamma\colon \mathbb{R}\to N\) there exist sequences of real
  numbers \(t_k\to \infty\) and \(t_k'\to \infty\) for which
  \(\gamma(t_k)\to p\), \(\gamma(t_k')\to p'\), and \(p\neq p'\).
Nevertheless, both \(p\) and \(p'\) will be critical points of the
  associated function. 
This phenomenon is also known for finite energy pseudoholomorphic curves,
  see \cite{Siefring}.
Analogously, feral curves need not have unique limits, but
  nevertheless by passing to a subsequence one can extract the desired limit
  set which is indeed closed and invariant under the flow of \(X_\eta\).
Here we make this precise with Definition \ref{DEF_x_limit_set} and
  Proposition \ref{PROP_properties_of_x_limit_set} below.

\begin{definition}[$\mathbf{x}$-limit set]
  \label{DEF_x_limit_set}
  \hfill \\
Let \((M, \eta)\) be a closed framed Hamiltonian manifold, and let \((J,
  g)\) be an \(\eta\)-adapted almost Hermitian structure on
  \(\mathbb{R}\times M\).
Let \(\mathbf{u} = (u, S, j, W, J, \mu, D)\) be a feral curve in the sense
  of Definition \ref{DEF_feral_J_curve}.
For each \(x\in \mathbb{R}\), define 
  \begin{equation*}                                                       
    \widehat{\Xi}_{x} ={\rm Sh}_{x}\Big(((x-1, x+1)\times M )\cap
    u(S)\Big) \subset (-1,1)\times M,
    \end{equation*}
  where for each \(x\in \mathbb{R}\), the map \({\rm
  Sh}_{x}:\mathbb{R}\times M\to \mathbb{R}\times M\)  is the shift map
  defined by \({\rm Sh}_{x}(a,p) = (a-x,p)\).
Let \(\mathbf{x}= \{x_i\}_{i\in \mathbb{N}}\subset \mathbb{R}\) be a
  monotonic sequence  with either \(\lim_{i\to \infty}x_i= \infty\) or
  \(\lim_{i\to \infty} x_i= -\infty\).
We then define the \(\mathbf{x}\)-limit set of \(u\) to be the following:
  \begin{equation*}                                                       
    L_{\mathbf{x}}:=\bigcap_{k=1}^\infty {\rm cl}  \Big(
    \bigcup_{i=k}^\infty
    \widehat{\Xi}_{x_i}\Big) \subset (-1,1)\times M
    \end{equation*}
\end{definition}                                                                     
%

\begin{proposition}[properties of $\mathbf{x}$-limit set]
  \label{PROP_properties_of_x_limit_set}
  \hfill\\
Let \((M, \eta)\) be a closed framed Hamiltonian manifold, and let \((J,
  g)\) be an \(\eta\)-adapted almost Hermitian structure on
  \(\mathbb{R}\times M\).
Let \(\mathbf{u}=(u, S, j, W, J, \mu, D)\) be a feral curve and 
  \(\mathbf{x}= \{x_i\}_{i\in \mathbb{N}}\subset \mathbb{R}\) be a
  monotonic sequence with \(|x_i|\to \infty\).
\emph{Then} the \(\mathbf{x}\)-limit of \(u\) has the form \((-1, 1)\times
  \Xi \), where \(\Xi\subset M\) is a closed set which is invariant under
  the Hamiltonian flow of \(\eta\).
\end{proposition}
%
\begin{proof}
The main technical tool to prove this result will be the following.

\begin{lemma}[local invariance]
  \label{LEM_local_invariance}
  \hfill\\
Let \((M,\eta)\), \((J, g)\), \(\mathbf{u}\), and
  \(\mathbf{x}=\{x_i\}_{i\in \mathbb{N}}\) be as above in Proposition
  \ref{PROP_properties_of_x_limit_set}.
Then there exists an \(\epsilon_0>0\) with the following property.
If \((a_0, p_0)\in L_{\mathbf{x}}\), and if \(|\epsilon|< \epsilon_0\),
  then \((a_0+\epsilon, p_0) \in L_{\mathbf{x}}\) whenever
  \(|a_0+\epsilon|< 1\).
Similarly if \((a_0, p_0)\in L_{\mathbf{x}}\), and if \(|\epsilon|<
  \epsilon_0\), then \((a_0, \varphi_\eta^\epsilon(p_0))\in
  L_{\mathbf{x}}\), where \(\varphi_\eta^\epsilon\) is the time
  \(\epsilon\) flow of the Hamiltonian vector field associated to \(\eta\).
\end{lemma}
%
We will prove Lemma \ref{LEM_local_invariance} momentarily, however for
  the moment we use it to complete the proof of Proposition
  \ref{PROP_properties_of_x_limit_set}.
To that end, observe that Lemma  \ref{LEM_local_invariance} immediately
  establishes that \(L_{\mathbf{x}}=(-1, 1)\times \Xi\) with \(\Xi\)
  invariant under the Hamiltonian flow associated to \(\eta\).
Furthermore, by definition, \(L_{\mathbf{x}}\) is the 
  intersection of closed sets, and hence itself closed.
It is then elementary to deduce that \(\Xi\) is closed in \(M\).
This completes the proof of Proposition
  \ref{PROP_properties_of_x_limit_set}.
\end{proof}
\begin{proof}[Proof of Lemma \ref{LEM_local_invariance}]
We prove the case that \(x_i\to \infty\); the case that \(x_i\to -\infty\)
  is essentially the same.
To begin, we define \(\epsilon_0:= {\textstyle\frac{1}{4}}r_1\) where
  \(r_1=r_1(M, \eta, J, g)>0\) is the positive constant guaranteed by
  Theorem \ref{THM_local_local_area_bound} (asymptotic connected-local
  area bound).
We then suppose that  \((a_0,p_0)\in L_{\mathbf{x}}\), and
  \(q_0=\varphi_\eta^\epsilon(p_0)\) for some \(|\epsilon|< \epsilon_0 \).
By definition of \(L_{\mathbf{x}}\), there exists a sequence
  \(\{\zeta_k\}_{k\in \mathbb{N}}\subset  S\) and monotonic sequence
  \(\{i_k\}_{k\in \mathbb{N}}\subset \mathbb{N}\) with \(i_k\to \infty\)
  for which \({\rm Sh}_{x_{i_k}}\circ u(\zeta_k) \to (a_0, p_0)\).
Note that since \(i_k\to \infty\), we must have \(x_{i_k}\to \infty\).

We then define a sequence of pseudoholomorphic curves 
  \begin{align*}                                                          
    \mathbf{u}_k = (u_k, \widetilde{S}_k, j_k, \mathbb{R}\times M, J,
    \emptyset, \emptyset)
    \end{align*}
  where
  \begin{equation*}                                                       
      \widetilde{S}_k:= S_{r_1}(\zeta_k)\subset
      u^{-1}\big(\mathcal{B}_{r_1}(u(\zeta_k))\big)
      \subset S
      \end{equation*}
  is the connected component of
  \(u^{-1}\big(\mathcal{B}_{r_1}(u(\zeta_k))\big)\)  containing    \(\zeta_k\), and where \(j_k:=j\big|_{\widetilde{S}_k}\) and \(u_k:=
  {\rm Sh}_{x_{i_k}}\circ u\); here \(\mathcal{B}_r(p)\subset
  \mathbb{R}\times M\) denotes the open metric ball of radius \(r\)
  centered at \(p\).
Note that by construction we have \(\zeta_k\in \widetilde{S}_k\) for every
  \(k\in \mathbb{N} \), and \(u_k(\zeta_k)\to (a_0, p_0)\).
By construction, we may also apply Theorem
  \ref{THM_local_local_area_bound} (asymptotic connected-local area bound),
  which guarantees that
  \begin{equation*}                                                       
    {\rm Area}_{u_k^*g} (\widetilde{S}_k)\leq 1.
    \end{equation*}
Also recall that \(\mathbf{u}\) is a feral curve, and hence has finite
  genus, so by genus super-additivity\footnote{See Lemma
  \ref{LEM_genus_addition}.} and the fact that \(\widetilde{S}_k\subset
  S\), it follows that \({\rm Genus}(\widetilde{S}_k)\) is uniformly
  bounded in \(k\).
Also because \(\mathbf{u}\) is feral it follows that \(\#(\mu \cup
  D)<\infty\), and \(x_{i_k}\to \infty\) so that for all sufficiently
  large \(k\) we have \((\mu \cup D) \cap \widetilde{S}_k = \emptyset\),
  and hence the \(\mathbf{u}_k\) are stable for all sufficiently large
  \(k\).
With uniform area bounds, uniform genus bounds, and stability, it then
  follows from Theorem \ref{THM_target_local_gromov_compactness},
  target-local Gromov compactness, that after passing to a subsequence
  (still denoted with subscripts \(k\)), there exist compact Riemann
  surfaces with smooth boundary \(\widehat{S}_k\subset \widetilde{S}_k\)
  which satisfy the following properties.
  \begin{enumerate}                                                         
    \item 
    \(\zeta_k\in \widehat{S}_k\) for all \(i\), 
    \item 
    \( u_k(\zeta_k)\to (a_0, p_0)\),  
    \item 
    \(u_k(\partial \widehat{S}_k) \cap \mathcal{B}_{r_1/2}((a_0,p_0))
    =\emptyset\)
    \item 
    \((u_k, \widehat{S}_k, j_k, \mathbb{R}\times M, J, \emptyset,
    \emptyset)\to (u_\infty, \widehat{S}_\infty, j_\infty,
    \mathbb{R}\times M, J, \emptyset, D_\infty)\) in a Gromov
    sense, where the limit is a compact pseudoholomorphic curve with
    immersed boundary.
    \end{enumerate}
Furthermore, we note that because \(u:S\to \mathbb{R}\times M\) is a feral
  curve, it follows that \(\int_S u^*\omega< \infty\), and because
  \(\omega\) evaluates non-negatively on \(J\)-invariant planes and because
  \(a\circ u(\zeta_k)\geq x_{i_k} -1 \to \infty\), it follows that
  \(\int_{S_k} u_k^*\omega \to 0\), and hence \(u_\infty^*\omega=0\).
Recall that  \({\rm ker}\; \omega = {\rm Span} (\partial_a , X_\eta)\).
We conclude that \(u_\infty(\widehat{S}_\infty) \) is contained in
  \(\mathbb{R}\times \Gamma\) where \(\Gamma\) is the finite union of
  trajectories of the Hamiltonian vector field \(X_\eta\).

Letting \(\Gamma_{p_0}\) denote the Hamiltonian trajectory containing
  \(p_0\), we note from the fact that \(u_k(\zeta_k)\to (a_0, p_0)\) and
  by definition of Gromov convergence that there exists a connected
  component \(\widehat{S}_\infty'\subset \widehat{S}_\infty\) for which
  \(u_\infty(\widehat{S}_\infty')\subset \mathbb{R}\times \Gamma_{p_0}\).
Moreover, \((a_0, p_0)\in u_\infty(\widehat{S}_\infty')\subset
  \mathbb{R}\times \Gamma_p\) and \(u_\infty(\partial \widehat{S}_\infty')
  \cap \mathcal{B}_{r_1/2}((a_0, p_0)) = \emptyset\), from which it
  follows that for each \(q_0=\varphi_\eta^\epsilon(p_0)\) with
  \(|\epsilon|<\textstyle{\frac{1}{4}}r_1=\epsilon_0\) we have \((a_0,
  q_0) \in u_\infty(\widehat{S}_\infty')\).
But then by Gromov convergence, it follows that there exists a sequence
  \(\zeta_k' \in \widehat{S}_k\) such that \(u_k(\zeta_k')\to (a_0,
  q_0)\), and hence the sequence \(\zeta_k'\in S\) satisfies
  \(x_{i_k}-1 \leq a \circ u(\zeta_k') \leq x_{i_k}+1\) for all
  sufficiently large \(k\in \mathbb{N}\), so that \((a_0, q_0)\in
  L_{\mathbf{x}}\) as required.
A similar argument shows that \((a_0+\epsilon, p_0)\in L_{\mathbf{x}}\)
  whenever \(|\epsilon|< \epsilon_0\), and \(|a_0+\epsilon|< 1\). 
This completes the proof of Lemma \ref{LEM_local_invariance}. 
\end{proof}

\subsubsection{Main argument}\label{SEC_existence_workhorse_proof}
In what follows, it may be useful to review the notion of a marked nodal
  pseudoholomorphic curve as provided in Definition
  \ref{DEF_pseudholomorphic_curve}, as well as the notion of a marked nodal
  Riemann surface as provided in Definition \ref{DEF_nodal_riemann_surface}.
The latter specifically is expressed as \((S, j, \mu, D)\) where
  \(D=\{\underline{d}_1, \overline{d}_1, \underline{d}_2, \overline{d}_2,
  \ldots\}\) is the set of nodal points.
Furthermore, as discussed after Remark \ref{REM_nodal_noation}, a (marked)
  nodal Riemann surface gives rise to the topological space \(|S|\) obtained
  by identifying each point in \(D\) with its corresponding nodal pair; in
  other words \(|S| = S/ (\underline{d}_i \sim \overline{d}_i)\).
We are now prepared to re-state the result we aim to prove here. 

\setcounter{CurrentSection}{\value{section}}
\setcounter{CurrentTheorem}{\value{theorem}}
\setcounter{section}{\value{CounterSectionExistenceWorkhorse}}
\setcounter{theorem}{\value{CounterTheoremExistenceWorkhorse}}
\begin{theorem}[existence workhorse]
  \hfill \\
Let \((M, \eta)\) be a compact framed Hamiltonian manifold with \({\rm
  dim}(M) = 3\).
Let \(\{a_k\}_{k\in \mathbb{N}}\subset \mathbb{R}^-\) be a sequence for
  which \(a_k\to -\infty\) monotonically.
For each \(k\in \mathbb{N}\), let \((J_k, g_k)\) be a \(\eta\)-adapted
  almost complex structure on \(\mathbb{R}\times M\).
Suppose that there exists a positive constant \(C\geq 1\), and suppose that
  for each \(k\in \mathbb{N}\) and each \(b\in [a_k, 0] \) there exists a
  stable\footnote{
    Here we mean stable in the sense described in Definition \ref{DEF_stable}.
    } 
  unmarked but possibly nodal pseudoholomorphic curve
  \begin{equation*}                                                       
    \mathbf{u}_k^b=\big(u_k^b, S_k^b, j_k^b, (-\infty, 1)\times M, J_k,
    \emptyset, D_k^b \big)
    \end{equation*}
  with the following properties.
  \begin{enumerate}[(P1)]                                                 
    \item
    the topological space \(|S_k^b|\) is connected (implying (P4) below),
    \item 
    \(\mathbf{u}_k^b\) is compact and \(u_k^b(\partial S_k^b)\subset
    (0,1)\times M \),
    \item 
    \(\inf_{\zeta\in S_k^b} a\circ u_k^b(\zeta) = b\),
    \item 
  there exists a continuous path \(\alpha:[0,1]\to |S_k^b|\) satisfying 
    \begin{equation*}                                                     
      a\circ u_k^b\circ \alpha(0) = b \qquad\text{and}\qquad \alpha(1)\in
      \partial S_k^b,
      \end{equation*}
    \item 
    \({\rm Genus}(S_k^b)\leq C\),
    \item 
    \(\int_{S_k^b}(u_k^b)^*\omega \leq C\), 
    \item 
    \(\#D_k^b\leq C\),
    \item 
    the number of connected components of \(\partial S_k^b\) is bounded
    above by \(C\).
    \end{enumerate}
Furthermore, suppose that \(J_k\to \bar{J}\) in \(\mathcal{C}^\infty\), and for  
  each fixed \(k\), and each pair \(b, b'\in [a_k, 0]\) with \(b\neq b'\)
  we have\footnote{
    This is a geometric count of the intersection points of the images. 
    It does not involve multiplicities.
    }
  \begin{equation*}                                                       
    \#\big(u_k^b(S_k^b)\cap u_k^{b'}(S_k^{b'})\big)\leq C.
    \end{equation*}
\emph{Then} there exists a closed set \(\Xi\subset M\)  satisfying
  \(\emptyset \neq \Xi \neq M\) which is invariant under the flow of the
  Hamiltonian vector field \(X_{\eta}\).
\end{theorem}
%

\begin{figure}[h]
  \label{FIG4_jwf_alternate}
  \includegraphics[scale=0.3]{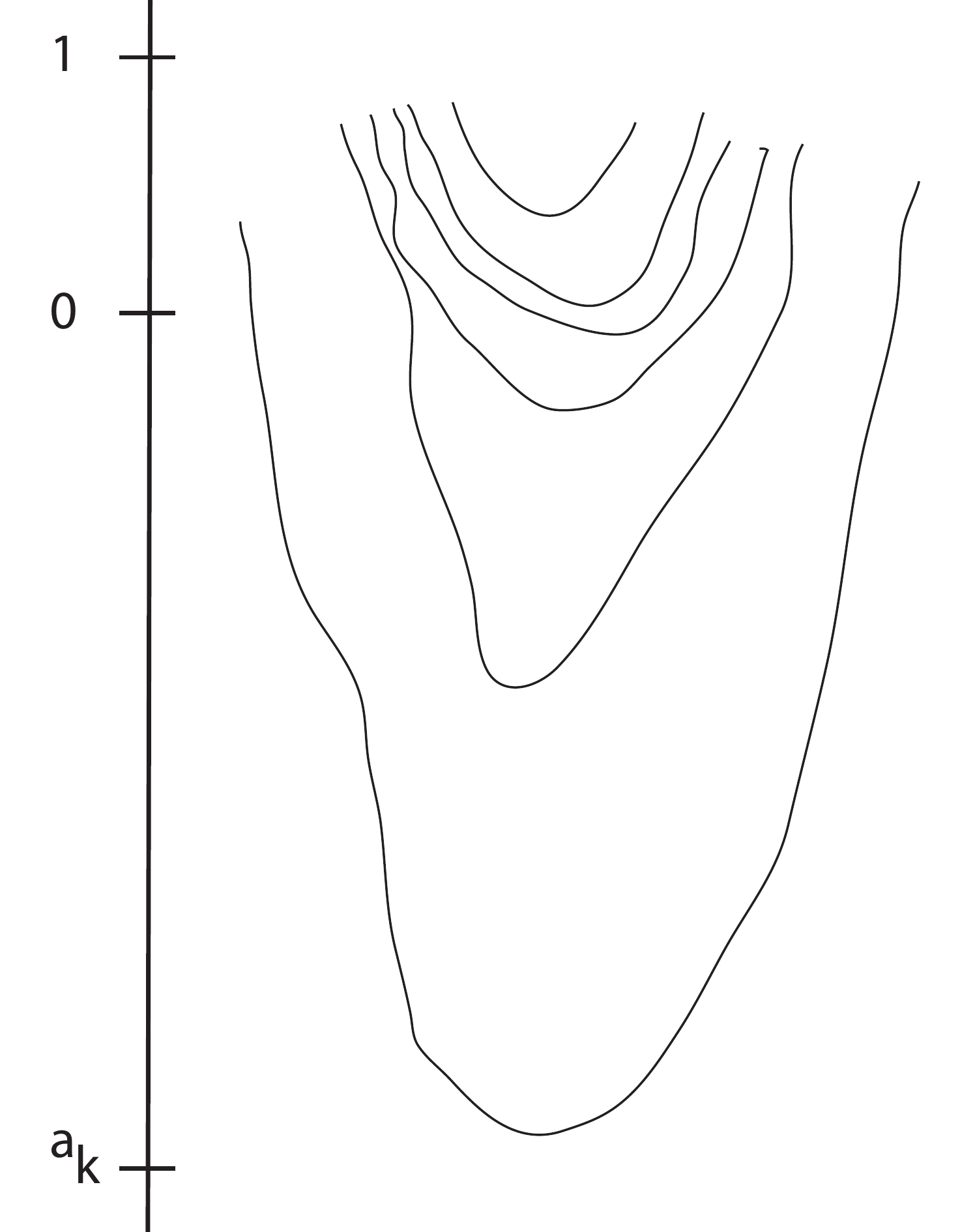}
  \caption{For every $k$ the figure shows a schematic family of
    $J_k$-holomorphic embedded mutually disjoint disks whose minimal
    ${\mathbb R}$-projection covers the interval $[a_k,0]$.}
\end{figure}
%
Before we prove this theorem, we illustrate the hypotheses with an
  example.
We consider for $(M,\eta)$ the manifold ${\mathbb R}\times M$  equipped
  with $(J_k,g_k)$.
Assume that for fixed $k$ there exists a family\footnote{
  It is not assumed to be  a continuous family!
  }  
  of embedded $J_k$-holomorphic disks with boundaries in $(0,1)\times M$.
We assume that any two different disks in the family do not intersect, the
  image of a disk in the family lies in $(-\infty,1)\times M$ and the set of
  minimum $a$-values  covers $[a_k,0]$.
Since the genus is $0$, the only additional assumption we need is a
  uniform $\omega$-energy bound.
Thus we assume that we have such a sequence of families, indexed by \(k\),
  with the additional property that \(a_k\to -\infty\) as \(k\to \infty\).
That is, as we progress through the sequence, the associated families
  extend more and more deeply into the negative end of \(\mathbb{R}\times
  M\).
Note that one can produce such a system of disks in the case in which we
  have an exact\footnote{
    One should be able to remove the exactness assumption, see
    Remark \ref{REM_removing_exactness}.
    } 
  symplectic cobordism from an overtwisted contact manifold \(M^+\) on
  top to \((M,\eta)\) on bottom.
One can then use Bishop's theorem on disk fillings to construct the
  families.
A complete discussion of Bishop's theorem can be found in
  \cite{Abbas-Hofer}.
Of course, these ideas must be combined with the constructions and
  estimates derived in the current manuscript.  
Allowing non-embedded curves and mutual intersections increases the
  complexity of the argument. 
However, the basic idea can be seen in the special example which we have
  just outlined.

\begin{proof}
\setcounter{section}{\value{CurrentSection}}
\setcounter{theorem}{\value{CurrentTheorem}}
We proceed via a proof by contradiction, and thus we begin by assuming
  Theorem \ref{THM_existence} is false.
Our first step in deriving a contradiction is then to fix \(c\geq 0\) and
  define a sequence of manifolds and surfaces via the following:
  \begin{align}                                                           
    &W_k:=\big({\textstyle \frac{1}{k}}-2, -a_k\big)\times M,\notag  
    \\
    &S_{k,c}:= S_k^{a_k+c} \label{EQ_Skc}
    \\
    &\widehat{S}_{k,c}:= ({\rm Sh}_{a_k}\circ u_k^{a_k+c})^{-1}(W_k)
    \subset S_{k,c}, \label{EQ_whSkc}
    \end{align}
  where for each \(x\in \mathbb{R}\) the shift map \({\rm Sh}_x\) is
  defined by
  \begin{align*}                                                          
    &{\rm Sh}_x:\mathbb{R}\times M\to \mathbb{R}\times M
    \\
    &{\rm Sh}_x(a_0, q_0) = (a_0 - x, q_0).
    \end{align*}
We then define pseudoholomorphic curves
\begin{equation}                                                          
  \mathbf{w}_{k,c} =\big(w_{k,c}, \widehat{S}_{k, c}, j_{k, c}, W_k, J_k,
  \emptyset, \widehat{D}_{k,c} \big)\label{EQ_wkc}
  \end{equation}
  where
  \begin{align}                                                           
    \widehat{D}_{k,c} &= D_k^{a_k+c}\cap \widehat{S}_{k, c}\notag
    \\
    j_{k,c} &= j_k^{a_k+c}\big|_{\widehat{S}_{k,c}}\notag
    \\
    w_{k,c} &= {\rm Sh}_{a_k}\circ u_k^{a_k+c}. \label{EQ_def_w}
    \end{align}
Observe that by definition, we also see that there exists a continuous
  path of the form \(\alpha:[0, 1]\to |\widehat{S}_{k, c}|\) satisfying 
  \begin{align}\label{EQ_spans}                                           
    a\circ w_{k, c} \circ \alpha(0) = c\qquad\text{and}\qquad a\circ w_{k,
    c}\circ \alpha(1) \geq -a_k = |a_k| \to \infty.
    \end{align}
Moreover, each \(w_{k,c}:\widehat{S}_{k,c} \to W_k =(\frac{1}{k}
  -2, -a_k)\times M \subset (-2, -a_k)\times M\) is a proper map with empty
  intersection with \((-\infty,0)\times M\).
To put this more geometrically, observe that \(\widehat{S}_{k,c} \subset
  S_{k,c}\setminus \partial S_{k,c} \) and therefore
  \(w_{k,c}:\widehat{S}_{k,c}\to W_k\) has a natural extension to
  \(w_{k,c}:S_{k,c}\to (-2,\infty)\times M\) via \(w_{k,c} = {\rm
  Sh}_{a_k}\circ u_k^{a_k,c}\).
Consequently we may regard the \emph{set-wise} boundary \(\partial
  \widehat{S}_{k,c} \subset S_{k,c} \setminus \partial S_{k,c}\), in which
  case we have
  \begin{align}\label{EQ_high_boundary}                                   
    \sup_{\zeta\in \partial \widehat{S}_{k, c} } a\circ w_{k, c}(\zeta) =
    \inf_{\zeta\in \partial \widehat{S}_{k, c} } a\circ w_{k, c}(\zeta)
    = -a_k = |a_k| \to \infty.
    \end{align}

For later use, we also define the nodal Riemann surface 
  \((S_{k,c}, j_{k,c}, D_{k,c})\) by letting \(S_{k, c}=S_k^{a_k+c}\) as
  above and letting \(D_{k,c}=D_k^{a_k+c}\cap S_{k,c}\).
Recall that because the \((J_k, g_k)\) are \(\eta\)-adapted almost
  Hermitian structures, it follows from Lemma
  \ref{LEM_eta_adapted_are_almost_Hermitian} that the \((W, J_k, g_k)\)
  are indeed almost Hermitian manifolds with the \(g_k\) expressible as
  \begin{equation*}                                                       
    g_k := da\otimes da + \lambda\otimes \lambda + \omega( \cdot, J_k
    \cdot).
    \end{equation*}
Also recall that by construction, the triples \((W_k, J_k, g_k)\)
  properly exhaust the almost Hermitian manifold  \((-2,\infty)\times M
  \subset \mathbb{R}\times M\) in the sense of Definition
  \ref{DEF_properly_exhausting_regions}.
Furthermore for each \(k\in \mathbb{N}\), the curve
  \(\mathbf{w}_{k,c}\) is a proper pseudoholomorphic curve in \((W_k,
  J_k)\), which can be included into \(\mathbb{R}\times M\).
Moreover, the symplectization coordinate of each \(\mathbf{w}_{k,c}\)
  has absolute minimum of \(c\); in other words, \(\inf_{\zeta\in S_{k,c}}
  a\circ w_{k,c}(\zeta) = c\).
We now apply Theorem  \ref{THM_area_bounds}, which guarantees the
  existence of a sequence of positive constants \(C_n\), with the property
  that for every \(k\geq n\in \mathbb{N}\) we have
  \begin{equation*}                                                       
    {\rm Area}_{g_k}\big(\widehat{S}_{k,c}^n\big) =
    \int_{\widehat{S}_{k,c}^n}w_{k,c}^*(da\wedge \lambda + \omega) \leq
    C_n
    \end{equation*}
  where
  \begin{equation*}                                                       
    \widehat{S}_{k,c}^n:= w_{k,c}^{-1}(W_n)\subset \widehat{S}_{k,c}.
    \end{equation*}
We then apply Theorem \ref{THM_exhaustive_gromov_compactness} (exhaustive
  Gromov compactness), which guarantees the existence of a proper stable
  nodal pseudoholomorphic curve without boundary
  \begin{equation*}                                                       
  (\bar{w}_c, \overline{S}_c, \bar{j}_c, \mathbb{R}\times M,
    \overline{J}, \emptyset , \overline{D}_c),
    \end{equation*} 
  to which a subsequence of the \({\bf w}_{k,c}\) converge.
We now make the following claim.

\setcounter{CounterSectionFeralLimitCurves}{\value{section}}
\setcounter{CounterLemmaFeralLimitCurves}{\value{lemma}}
\begin{proposition}[feral limit curves]
  \label{PROP_feral_limit_curves}
  \hfill\\
Let \(c\geq 0\) and let 
  \begin{equation*}                                                       
  \bar{\mathbf{w}}_c=(\bar{w}_c, \overline{S}_c, \bar{j}_c,
    \mathbb{R}\times M,
    \overline{J}, \emptyset , \overline{D}_c),
    \end{equation*}  
 be an exhaustive limit of some subsequence of the \(\mathbf{w}_{k, c}\).
\emph{Then} \(\bar{\mathbf{w}}_c\) is a feral pseudoholomorphic curve in the
  sense of Definition \ref{DEF_feral_J_curve}.
\end{proposition} 
%

We postpone the proof of Proposition \ref{PROP_feral_limit_curves}  until
  Section \ref{SEC_proof_of_feral_limit} below, and assume its validity for
  the moment in order to complete the proof of Theorem \ref{THM_existence}.
To that end, we still aim to derive a contradiction, and thus we will need
  the following result.

\setcounter{CounterSectionBoundedIntersections1}{\value{section}}
\setcounter{CounterLemmaBoundedIntersections1}{\value{lemma}}
\begin{lemma}[bounded transverse intersections]
  \label{LEM_bounded_transverse_intersections}
  \hfill\\
Consider non-negative numbers \(c, c'\geq 0\) with \(c'> c\), and let  
  \(k\mapsto \ell_k\in \mathbb{N}\) be a strictly increasing sequence for
  which \(\mathbf{w}_{\ell_k,c}\to \bar{\mathbf{w}}_c\) and
  \(\mathbf{w}_{\ell_k,c'}\to \bar{\mathbf{w}}_{c'}\) in an exhaustive sense.
Then the subset \(\mathcal{P}\subset \mathbb{R}\times M\) of transversal
  intersection points of the two curves, which is defined by
  \begin{align*}                                                          
    \mathcal{P}:=\big\{p\in \mathbb{R}\times M &: \text{ there exists }
    (\zeta, \zeta')\in \overline{S}_c\times \overline{S}_{c'} \;
    \text{such that}\;
    \\
    &\; \; \bar{w}_c(\zeta) = p = \bar{w}_{c'}(\zeta')\;  \text{ and }\;
    T\bar{w}_c(\zeta)\pitchfork T\bar{w}_{c'}(\zeta') \big\},
    \end{align*}
  satisfies\footnote{
    Note that this a  bound on the number of intersection points of the
    images and not the number of parametrizing pairs.
    }
  \begin{equation*}                                                       
    \#\mathcal{P} \leq C.
    \end{equation*}
\end{lemma}
%
As before, we will postpone this proof until Section
  \ref{SEC_proof_of_bounded_intersections} below, and in the meantime
  proceed with the proof of Theorem \ref{THM_existence}.
We now construct a sequence of feral curves in \(\mathbb{R}\times M\).
We start with the sequence of pseudoholomorphic curves
  given by \(\{\mathbf{w}_{\ell, 0}\}_{\ell\in \mathbb{N}}\), and pass to a
  subsequence so that this subsequences converge (in an exhaustive Gromov
  sense) to the feral limit curve \(\bar{\mathbf{w}}_0\).
We will need to keep track of the subsequence in \(\mathbb{N}\) which
  yields convergence, and thus we write
  \begin{equation*}                                                       
    \mathbf{w}_{k_\nu^0, 0} \to \bar{\mathbf{w}}_0\qquad\text{as}\qquad
    \nu\to \infty.
    \end{equation*}
We then consider the sequence of curves given by \(\{\mathbf{w}_{k_\nu^0,
  1}\}_{\nu\in \mathbb{N}}\).  
We pass to a further subsequence to the obtain exhaustive Gromov
  convergence
  \begin{equation*}                                                       
    \mathbf{w}_{k_\nu^1, 1} \to \bar{\mathbf{w}}_1.    
    \end{equation*}
We then consider the sequence of curves given by \(\{\mathbf{w}_{k_\nu^1,
  2}\}_{\nu\in \mathbb{N}}\), and we pass to a further subsequence to
  obtain exhaustive Gromov convergence
  \begin{equation*}                                                       
    \mathbf{w}_{k_\nu^2, 2} \to \bar{\mathbf{w}}_2. 
 \end{equation*}
In this way we pass to further and further subsequences and obtain a
  sequence of converging sequences:
  \begin{equation*}                                                       
    \mathbf{w}_{k_\nu^\ell, \ell} \to \bar{\mathbf{w}}_\ell\qquad\text{for
    each } \ell\in \mathbb{N}.
    \end{equation*}
We then pass to the diagonal subsequence, \(\bar{k}_\nu: = k_\nu^\nu\),
  which by definition has the property that
  \begin{equation*}                                                       
    \mathbf{w}_{\bar{k}_\nu, \ell}\to \bar{\mathbf{w}}_\ell\qquad\text{for
    each } \ell\in \mathbb{N}.
    \end{equation*}
Recalling our notation, we then have 
  \begin{equation*}                                                       
    \bar{\mathbf{w}}_\ell = (\bar{w}_\ell, \overline{S}_\ell,
    \bar{j}_\ell, \mathbb{R}\times M, \overline{J}, \emptyset,
    \overline{D}_\ell). 
    \end{equation*}
We introduce another sequence of pseudoholomorphic curves denoted
  \begin{equation*}                                                       
    \mathbf{v}_\ell  = \big(v_\ell, \Sigma_\ell, \jmath_\ell, (-1,
    1)\times M, \overline{J}, \emptyset, \Delta_\ell\big)
    \end{equation*}
  and defined by
  \begin{align*}                                                          
    &\Sigma_\ell:=\bar{w}_\ell^{-1}\big( (\ell-1, \ell+1)\times M\big)
    \\
    &v_\ell := {\rm Sh}_{\ell}\circ \bar{w}_\ell
    \\
    &\jmath_\ell := j_\ell\big|_{\Sigma_\ell}
    \\
    &\Delta_\ell:= \overline{D}_\ell \cap \Sigma_\ell .
    \end{align*}
By construction, the \(\mathbf{v}_\ell\) are proper curves in \((-1,
  1)\times M\), and they have uniformly bounded area and genus.
Consequently, by Theorem \ref{THM_target_local_gromov_compactness},
  target-local Gromov compactness, we can pass to a subsequence
  \(\mathbf{v}_{\ell_\nu}\) and find compact domains
  \(\widetilde{\Sigma}_\ell \subset \Sigma_\ell\) so that we have Gromov
  convergence
  \begin{equation*}                                                       
    \big(v_{\ell_\nu}, \widetilde{\Sigma}_{\ell_\nu}, \jmath_{\ell_\nu},
    (-1,1)\times M, \overline{J}, \emptyset,
    \widetilde{\Delta}_{\ell_\nu}\big) \to (v,\Sigma, \jmath, (-1,1)\times
    M, \overline{J}, \emptyset, \Delta)
    \end{equation*}
  as \(\nu\to \infty\); here 
  \(\widetilde{\Delta}_{\ell_\nu} = \widetilde{\Sigma}_{\ell_\nu}\cap
  \Delta_{\ell_\nu}\).
Also recall that these \(\widetilde{\Sigma}_{\ell_\nu}\) have the property
that
  \begin{equation*}                                                       
    v_{\ell_\nu}^{-1}\big([-{\textstyle\frac{1}{2}},
    {\textstyle\frac{1}{2}}]\times M\big) \subset
    \widetilde{\Sigma}_{\ell_\nu}.
    \end{equation*}
We also define the sets, \(\{\widehat{\Xi}_{\ell_\nu}\}_{\nu\in
  \mathbb{N}}\),  by
  \begin{equation*}                                                       
    \widehat{\Xi}_{\ell_\nu}:={\rm Sh}_{\ell_\nu} \Big(\big((\ell_\nu-1,
    \ell_\nu+1)\times M \big) \cap  \bar{w}_0 (\overline{S}_0) \Big).
    \end{equation*}
We now make the following observation. 
Let \(\ell_\nu'\) be a subsequence of \(\ell_\nu\); then for any such
  subsequence, the set
  \begin{equation*}                                                       
    \widehat{\Xi}:= \bigcap_{k=1}^\infty {\rm cl}\Big(\bigcup_{\nu =
    k}^\infty \widehat{\Xi}_{\ell_\nu'}\Big)
    \end{equation*}
  is the \(\mathbf{x}\)-limit set $L_{\mathbf{x}}$ of
  \(\bar{w}_0:\overline{S}_0\to \mathbb{R}\times M\) for
  \(\mathbf{x}=\{a_{\ell_\nu'}\}_{\nu \in \mathbb{N}}\), in the sense of
  Definition \ref{DEF_x_limit_set}.
By Proposition \ref{PROP_properties_of_x_limit_set}, we have
  \(L_{\mathbf{x}}=\widehat{\Xi}=(-1,1)\times \Xi \), where \(\Xi\subset
  M\) is a closed set which is invariant under the flow of the Hamiltonian
  vector field associated to \(\eta\).

At this point there are then three possible cases, where
  $\widehat{\Xi}=L_{\mathbf{x}}$: \vspace{8pt}

\noindent\emph{Case I:} \(\widehat{\Xi} = \emptyset\). \\
Note, however, that this case is impossible because the curve
  \(\bar{\mathbf{w}}_0\) is proper without boundary but not compact and
  has image contained in \([0, \infty)\times M\).
Indeed, properness follows as a consequence of \(\bar{\mathbf{w}}_0\)
  being feral, and it has image contained in \([0, \infty)\times M\) as a
  consequence of exhaustive compactness together with the fact that the
  approximating curves \(\mathbf{w}_{\bar{k}_\nu, 0}\) are a subsequence
  \(\{\mathbf{w}_{\ell, 0}\}_{\ell\in \mathbb{N}}\), and each
  \(\mathbf{w}_{\ell, 0}\) has image contained in \([0,\infty)\times M\) by
  definition.
To see that the curves are without boundary and not compact, it is
  sufficient to recall properties of the approximating curves,
  \(\{\mathbf{w}_{\bar{k}_\nu, 0}\}_{\nu\in \mathbb{N}}\subset
  \{\mathbf{w}_{\ell, 0}\}_{\ell\in \mathbb{N}}\), and specifically the
  properties expressed in equation (\ref{EQ_spans}) and  equation
  (\ref{EQ_high_boundary}) together with the definition of exhaustive
  compactness.

\vspace{8pt}
\noindent\emph{Case II:} \( \emptyset\neq \widehat{\Xi} \neq M\). \\
This case ruled out by our contradiction hypothesis. 
\vspace{8pt}

\noindent\emph{Case III:} \( \widehat{\Xi} = M\). \\
We assume this to be true for the remainder of our proof and seek to
  derive a contradiction.
In fact, our contradiction hypothesis allows us to assume much more,
  namely that for each subsequence \(\{\ell_\nu'\}_{\nu\in \mathbb{N}}\) of
  \(\{\ell_\nu\}_{\nu\in \mathbb{N}}\) we must have \(\widehat{\Xi} = M\)
  for the corresponding \(\mathbf{x}\)-limit set \(\widehat{\Xi}\).
To proceed, let us define
  \begin{align*}                                                          
    \mathbf{v}_{\ell_\nu} = \big(v_{\ell_\nu},
    \widetilde{\Sigma}_{\ell_\nu}, \jmath_{\ell_\nu}, (-1,1)\times M,
    \overline{J}, \emptyset, \widetilde{\Delta}_{\ell_\nu}\big)
    \end{align*}
  and
  \begin{align*}                                                          
  \mathbf{v}= (v,\Sigma, \jmath, (-1,1)\times M, \overline{J},\emptyset,
    \Delta),
    \end{align*}
  and recall from above that \(\mathbf{v}_{\ell_\nu} \to \mathbf{v}\) in a
  Gromov sense as \(\nu\to \infty\).
We then choose a finite set of points \(Z\subset \Sigma\setminus
  (\Delta \cup \partial \Sigma)\), with the property that each point in
  \(Z\) is an immersed point of \(v\), \(v(z)\neq v(z')\) for each
  \(z,z'\in Z\) with \(z\neq z'\), and for each \(z\in Z\) we have \(v^*
  \omega(z)\neq 0\), and \(\#Z > C\).
Such a set \(Z\) exists as a consequence of target-local Gromov compactness
  and properties of the approximating curves; specifically, \(a\circ v\)
  has an absolute minimum of \(0\), and \(\inf_{\zeta\in \partial
  \widetilde{\Sigma} } a\circ v(\zeta) \geq {\textstyle\frac{1}{2}}\).
We then let \(Q\subset \Sigma\) denote the union of pairwise disjoint 
  disk-like neighborhoods of the points in \(Z\), each of which contains
  precisely one element of \(Z\). 
We assume that these disk-like neighborhoods are chosen so small that
  \(v:Q\to \mathbb{R}\times M\) is an embedding.
By Gromov convergence, there exist exist maps
  \(\phi_{\ell_\nu}: Q \to \widetilde{\Sigma}_{\ell_\nu}\) for which
  \begin{equation}\label{EQ_final_eqation_1}                              
    v_{\ell_\nu} \circ  \phi_{\ell_\nu} \to v\qquad\text{in
    }\mathcal{C}^\infty(Q, \mathbb{R}\times M).
    \end{equation}
Next we note that by assumption (to derive a contradiction) it follows
  that after passing to a subsequence of the \(\ell_\nu\), denoted
  \(\ell_\nu'\), there exists a sequence of finite sets \(\{Z_\nu\}_{\nu\in
  \mathbb{N}} \) with \( Z_\nu\subset \overline{S}_0\) for each \(\nu\in
  \mathbb{N}\) with the property that \({\rm Sh}_{\ell_\nu'}\circ
  \bar{w}_0( Z_\nu) \to v(Z)\); it may be helpful to recall that
  \(\overline{S}_0\) is the domain of the curve \(\bar{\mathbf{w}}_0\).
We then let \(r_1>0\) be the positive constant guaranteed by Theorem
  \ref{THM_local_local_area_bound}, and we then define a sequence of open
  sets \(\{P_\nu\}_{\nu\in \mathbb{N}}\) with 
  \begin{equation*}                                                       
    P_\nu \subset  \bar{w}_0^{-1}\Big(\bigcup_{z\in Z_\nu}
    \mathcal{B}_{r_1}(\bar{w}_0(z))\Big) \subset\overline{S}_0,
    \end{equation*}
  and with the property that each connected component of each \(P_\nu\)
  has non-empty intersection with \(Z_\nu\).
Note that by shrinking \(r_1\) if necessary, we may assume that the
  \(\mathcal{B}_{r_1}\big(\bar{w}_0(z)\big)\) are pairwise disjoint, and
  hence each connected component of \(P_\nu\) contains exactly one element
  of \(Z_\nu\), and that for all sufficiently large \(\nu\in \mathbb{N}\)
  the number of connected components of \(P_\nu\) equals \(\#Z\).
We note that from Theorem \ref{THM_local_local_area_bound} it follows
  that \({\rm Area}_{\bar{w}_0^*g}(P_\nu)\) is uniformly bounded
  independent of \(\nu\), and the \(P_\nu\) have uniformly bounded genus
  since they are all subsets of \(\overline{S}_0\).
We conclude from Theorem \ref{THM_target_local_gromov_compactness},
  namely target-local Gromov compactness, that after passing to a
  subsequence, denoted with subscripts \(\ell_{\nu_k}'\), there exist
  compact Riemann surfaces with boundary \(\widetilde{P}_{\nu_k} \subset
  P_{\nu_k}\) with the property that
  \begin{align*}                                                          
    {\rm Sh}_{\ell_{\nu_k}'}\circ \bar{w}_0(\partial
    \widetilde{P}_{\nu_k})  \cap \bigcup_{z\in Z}
    \mathcal{B}_{\frac{1}{2}r_1}\big(v(z)\big) = \emptyset,
    \end{align*}
  while 
  \begin{align*}                                                          
    Z_{\nu_k}\subset \widetilde{P}_{\nu_k}\setminus \partial
    \widetilde{P}_{\nu_k} \qquad\text{and}\qquad {\rm
    Sh}_{\ell_{\nu_k}'}\circ \bar{w}_0(Z_{\nu_k})\to v(Z),
  \end{align*}
  and furthermore we can arrange to have Gromov convergence of the maps 
  \begin{equation*}                                                       
    ({\rm Sh}_{\ell_{\nu_k}'}\circ \bar{w}_0, \widetilde{P}_{\nu_k},
    \bar{j}_0, \mathbb{R}\times M, \overline{J}, \bar{\mu}_0\cap
    \widetilde{P}_{\nu_k}, \overline{D}_0\cap \widetilde{P}_{\nu_k}) \to
    (\check{w}, \widetilde{P}, \check{j}, \mathbb{R}\times M,
    \overline{J}, \check{\mu}, \check{D}).
    \end{equation*}
Importantly, by Theorem \ref{THM_curv_bound} (asymptotic
  curvature bound), the maps 
  \begin{align*}                                                          
    {\rm Sh}_{\ell_{\nu_k}'}\circ w_0 :\widetilde{P}_{\ell_{\nu_k}'}\to
    \mathbb{R}\times M
    \end{align*}
  are immersions with uniformly bounded curvature.
We conclude that that \(\check{D}=\emptyset\), and
  \(\check{w}:\widetilde{P}\to \mathbb{R}\times M\) is an immersion, and,
  by Gromov convergence, there exist embeddings \(\varphi_k:
  \widetilde{P}\to \widetilde{P}_{\nu_k}\) with the property that
  \begin{equation}\label{EQ_final_eqation_2}                              
    {\rm Sh}_{\ell_{\nu_k}'} \circ \bar{w}_0 \circ \varphi_k \to \check{w}
    \qquad\text{in }\mathcal{C}^\infty(\widetilde{P}, \mathbb{R}\times M).
    \end{equation}
Note that by shrinking \(r_1\) if necessary, we may assume in fact that 
  \(\check{w}\) is an embedding.
Moreover, by construction \(v(Z)\subset \check{w}(\widetilde{P})\), and,
  because \(\cup_{k\in \mathbb{N}} \widetilde{P}_{\nu_k}\subset
  \overline{S}_0 \) and \(\int_{\overline{S}_0}\bar{w}_0^* \omega<
  \infty\), it follows that \(\check{w}^*\omega \equiv 0\).
We conclude that \(\check{w}:\widetilde{P}\to \mathbb{R}\times M\) and 
  \(v:Q\to \mathbb{R}\times M\) transversally intersect at \(v(Z)\);
  indeed, this follows from the fact that \(\check{w}^*\omega \equiv 0\)
  but that \(v^*\omega \neq 0\).
Then by equation (\ref{EQ_final_eqation_1}) and equation
  (\ref{EQ_final_eqation_2}) and Lemma
  \ref{LEM_stability_of_transversal_intersections}, it follows that for
  all sufficiently large \(k\) the maps
  \begin{equation*}                                                       
    {\rm Sh}_{\ell_{\nu_k}'}\circ \bar{w}_0\circ \varphi_k
    :\widetilde{P}\to \mathbb{R}\times M
    \end{equation*}
  and
  \begin{equation*}                                                       
    v_{\ell_{\nu_k}'}\circ \phi_{\ell_{\nu_k}'}: Q\to \mathbb{R}\times M
    \end{equation*}
  intersect transversally at \(\#Z >C\) points.
However, by definition of the \(v_{\ell_{\nu_k}'}\) it follows that the
  maps
  \begin{equation*}                                                       
    {\rm Sh}_{\ell_{\nu_k}'}\circ \bar{w}_{\ell_{\nu_k}'}: 
    \overline{S}_{\ell_{\nu_k}'}\to \mathbb{R}\times M
    \end{equation*}
  and the maps 
  \begin{equation*}                                                       
    {\rm Sh}_{\ell_{\nu_k}'} \circ \bar{w}_0: \overline{S}_0\to
    \mathbb{R}\times M
    \end{equation*}
  intersect transversally at \(\#Z > C\) points.
And hence the maps \(\bar{w}_{\ell_{\nu_k}'}: \overline{S}_{\ell_{\nu_k}'}
  \to \mathbb{R}\times M \) and \(\bar{w}_0:\overline{S}_0\to
  \mathbb{R}\times M\) intersect transversally at \(\# Z >C\) points.
However, this contradicts Lemma
  \ref{LEM_bounded_transverse_intersections}.
This is the desired contradiction which completes the proof of 
  Theorem \ref{THM_existence}. 
\end{proof}

We finish Section \ref{SEC_existence_workhorse_proof} with a lemma which
  we use without proof. 

\begin{lemma}[stability of transversal intersections]
  \label{LEM_stability_of_transversal_intersections}
  \hfill\\
Let \(\Sigma\) and \(\dot{\Sigma}\) each be a compact manifold with
  boundary and diffeomorphic to the closed two-dimensional disk \(\{z\in
  \mathbb{C}: |z|\leq 1\}\).
Let \(W\) be a four dimensional manifold, and let 
  \begin{equation*}                                                       
    u:\Sigma\to W\qquad\text{and}\qquad\dot{u}:\dot{\Sigma}\to W
    \end{equation*}
  be embeddings, for which there exist \(\zeta\in \Sigma\setminus \partial
  \Sigma\) and \(\dot{\zeta}\in \dot{\Sigma}\setminus \partial
  \dot{\Sigma}\) with the  property that \(u(\zeta) =
  \dot{u}(\dot{\zeta})\) and \(Tu(\zeta)\pitchfork
  T\dot{u}(\dot{\zeta})\).
That is, \(u\) and \(\dot{u}\) intersect transversally at
  \(u(\zeta)=\dot{u}(\dot{\zeta})\).
\emph{Then} for any sequences \(\{u_k\}_{k\in \mathbb{N}}\) and
  \(\{\dot{u}_k\}_{k\in \mathbb{N}}\) for which \(u_k \to u\) and
  \(\dot{u}_k\to \dot{u}\) in \(\mathcal{C}^\infty\), it is the case that
  for all sufficiently large \(k\in \mathbb{N}\), there exist \(\zeta_k\in
  \Sigma\) and \(\dot{\zeta}_k\in \dot{\Sigma}\) with with the property that
  \(u_k\) and \(\dot{u}_k\) intersect transversally at \(u_k(\zeta_k) =
  \dot{u}_k(\dot{\zeta}_k)\).  
In addition we may assume that $\zeta_k\rightarrow \zeta$ and
  $\dot{\zeta}_k\rightarrow \dot{\zeta}$.
\end{lemma}
%

\subsubsection{Proof of Proposition \ref{PROP_feral_limit_curves}}
  \label{SEC_proof_of_feral_limit}

We begin with a restatement.

\setcounter{CurrentSection}{\value{section}}
\setcounter{CurrentLemma}{\value{lemma}}
\setcounter{section}{\value{CounterSectionFeralLimitCurves}}
\setcounter{lemma}{\value{CounterLemmaFeralLimitCurves}}
\begin{restatementproposition}[feral limit curves]\hfill\\
Let \(c\geq 0\) and let 
  \begin{equation*}                                                       
    \bar{\mathbf{w}}_c=(\bar{w}_c, \overline{S}_c, \bar{j}_c,
    \mathbb{R}\times M,
    \overline{J},\bar{\mu}_c, \overline{D}_c),
    \end{equation*}   
  be an exhaustive limit of some subsequence of the \(\mathbf{w}_{k, c}\).
  \emph{Then} \(\bar{\mathbf{w}}_c\) is a feral pseudoholomorphic curve
  with $\bar{\mu}_c=\emptyset$ in the sense of Definition
  \ref{DEF_feral_J_curve}.
\end{restatementproposition}
%
\setcounter{section}{\value{CurrentSection}}
\setcounter{lemma}{\value{CurrentLemma}}
As already mentioned before, the feral curve compactness theorem uses the
  \({\mathbb R}\)-action in a less systematic way then in the SFT
  compactness theory.
As previously explained the reason is that we would need a better
  understanding of the behavior of the ends in a `generic' situation to
  give a compactness theorem comparable to the SFT-compactness theory.  
At this point it is not even clear what the notion of `generic' has to be,
  or if even such a notion exists.
With the current notion of convergence there is in general loss of
  information which we shall describe by a few examples.
\begin{figure}[h]
  \includegraphics[scale=0.4]{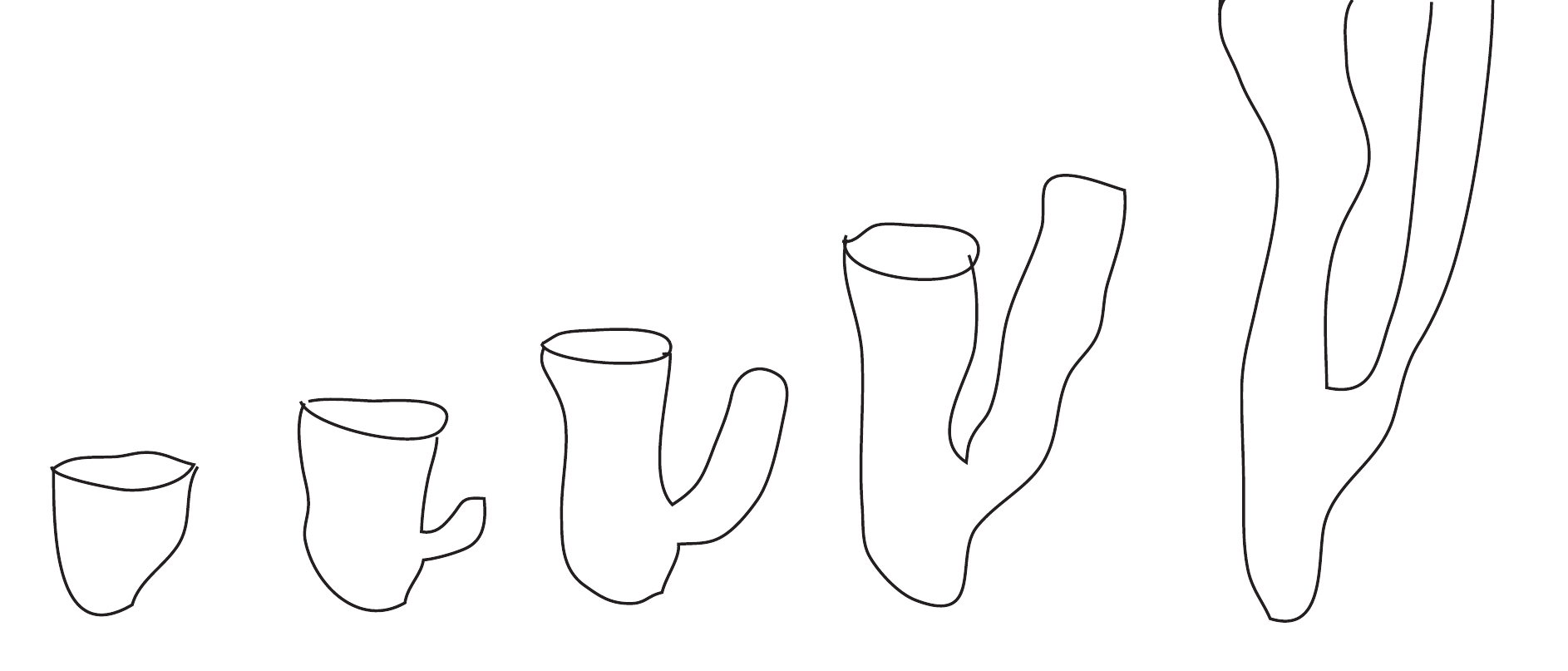}
  \caption{
    \label{FIG_disk_escapes}
    The figure shows a sequence of disks converging in an exhaustive
    Gromov compactness sense to a properly mapped two-punctured sphere.
    A cap flies away to $\infty$ creating a second end.
  }
\end{figure}
%
As a consequence of Theorem \ref{THM_energy_threshold}, it should be clear
  that it is impossible for an arbitrarily large number of caps to ``fly
  away to infinity'' provided we have a uniform $\omega$-energy bound.
Indeed, each such cap would remove at least an $\hbar>0$ of
  $\omega$-energy, leading to the absurd conclusion that the sequence of
  disks failed to have uniform \(\omega\)-energy bound.
Figure \ref{FIG_disk_escapes} show an example of one such disk escaping to
  infinity.
Similarly, Figure \ref{FIG_handle_escapes} shows how it is possible for a
  sequence of curves with genus one to limit to a once-punctured sphere
  because a handle escapes to infinity.
\begin{figure}[h]
  \includegraphics[scale=0.4]{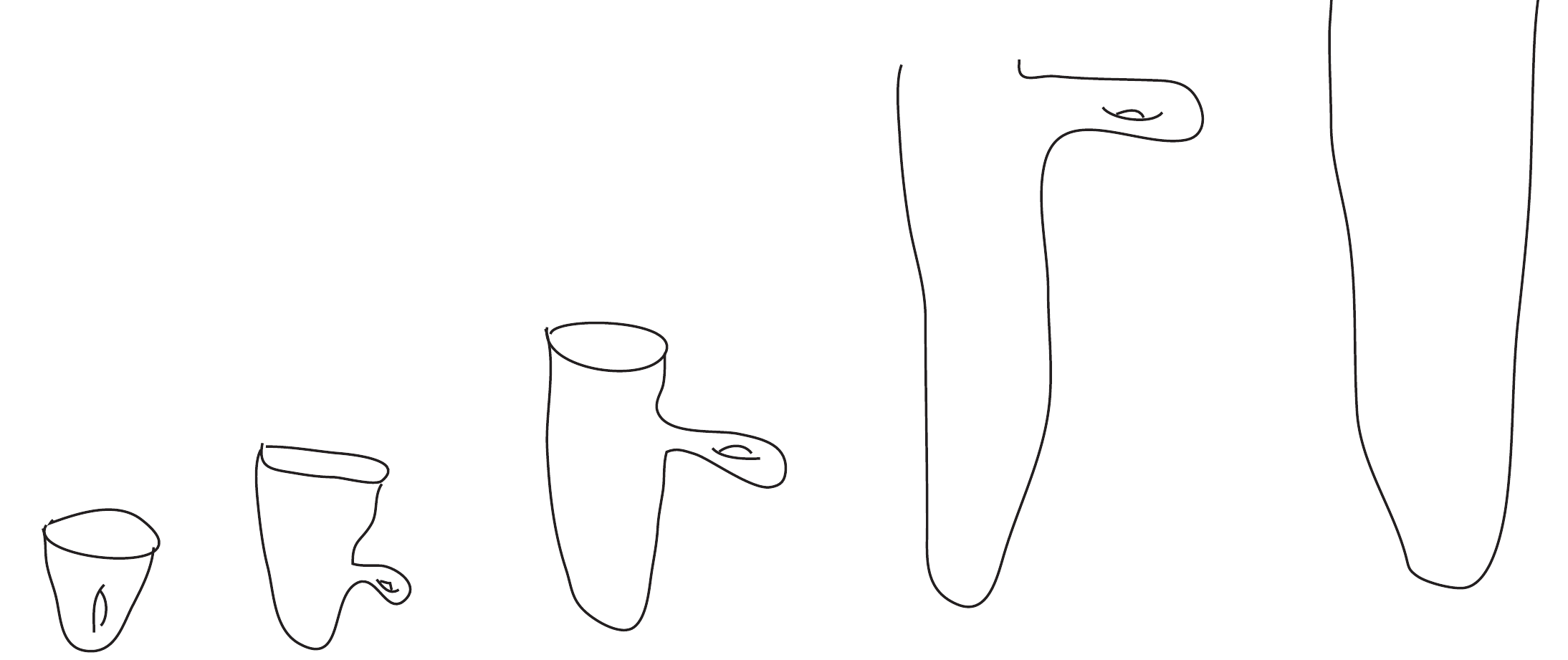}
  \caption{
  \label{FIG_handle_escapes}
    The figure shows a sequence of disks converging in an exhaustive
    Gromov compactness sense to a properly mapped one-punctured sphere
    shedding genus.
    }
\end{figure}
%
These are just two examples of what can happen, and below we provide a
  comprehensive discussion.
We leave it to the reader to imagine an example which lacks certain
  uniform topology bounds (like genus, connected components, etc) and hence
  can have a sequence of compact curves which develops infinitely many
  connected components or infinitely many nodal pairs.
Indeed, without topological bounds, the limit curve can get notably wild,
  however a key result of Proposition \ref{PROP_feral_limit_curves} is that
  the limit is a feral curve, and hence has bounded topology.
The reason for this is that producing ends or producing nodal pairs or
  other examples of infinite topology either requires approximating curves
  to have unbounded topology or unbounded \(\omega\)-energy, each of which
  are excluded by the hypotheses of Proposition
  \ref{PROP_feral_limit_curves}.
The proof of this result, takes some effort, which we now provide.
\begin{proof}
We begin by observing that as a result of Definition
  \ref{DEF_exhaustive_gromov_convergence} (exhaustive Gromov compactness)
  and properties of the \(\mathbf{u}_k^b\), it follows that
 \begin{enumerate}[(w1)]                                                 
    \item \label{EN_w1}
    \(\bar{w}_c:\overline{S}_c\to \mathbb{R}\times M\) is proper,
    \item \label{EN_w2}
    \(a\circ \bar{w}_c(\overline{S}_c) = [c, \infty)\), 
    \item \label{EN_w3}
    \(\int_{\overline{S}_c} \bar{w}_c^* \omega \leq C\)  
    \item \label{EN_w4}
    \({\rm Genus}(\overline{S}_c) \leq C\).
    \end{enumerate}
Thus, to establish that \(\bar{\mathbf{w}}_c\) is feral, it remains to
  establish that
  \begin{enumerate}[(F1)]                                                   
    \item  \label{EN_F1}
    \(\#\overline{D}_c < \infty\)
    \item  \label{EN_F2}
    \(\#\bar{\mu}_c =0\)
    \item  \label{EN_F3}
    \(\#\pi_0(\overline{S}_c)< \infty\)
    \item \label{EN_F4}
    \({\rm Punct}(\overline{S}_c) <\infty\)
    \end{enumerate}
  where \({\rm Punct}(S)\) is the number of generalized punctures
  (see Definition \ref{DEF_generalized_punctures}).
We note that \(\#\bar{\mu}_c=0\) since the \(\mathbf{w}_{k, c}\) had no
  marked points.
Establishing the remaining properties of a feral curve will take more
  effort than this, and so we first establish some notation.
For any topological space \(X\), we will let \(\pi_0(X)\) denote the set
  of connected components of \(X\), and we let \(\#\pi_0(X)\) denote the
  number of connected components of \(X\).
We now establish the finiteness of the number of nodal points and the
  number of connected components with Lemma
  \ref{LEM_bounded_connected_components} below.

\newcounter{CounterSectionBoundedConnectedComponents}
\newcounter{CounterLemmaBoundedConnectedComponents}
\setcounter{CounterSectionBoundedConnectedComponents}{\value{section}}
\setcounter{CounterLemmaBoundedConnectedComponents}{\value{lemma}}
\begin{lemma}[Some bounds on the limit curve]
  \label{LEM_bounded_connected_components}
  \hfill\\
For the pseudoholomorphic curve 
  \begin{equation*}                                                       
    \bar{\mathbf{w}}_c=(\bar{w}_c, \overline{S}_c, \bar{j}_c,
    \mathbb{R}\times M, \overline{J}, \emptyset, \overline{D}_c),  
    \end{equation*} 
  defined above, the following inequalities hold.
\begin{enumerate}                                                         
  \item 
  \({\rm Genus}_{arith}(\overline{S}_c,\bar{j}_c,\overline{D}_c)\leq 3C\)
  \item 
  \(\#\pi_0(\overline{S}_c)< 6(1  + \hbar^{-1} )C\)
  \item 
  \(\#\overline{D}_c < 18(1+\hbar^{-1})C \);
  \end{enumerate}
  where \({\rm Genus}_{arith}(\overline{S}_c, \bar{j}_c, \overline{D}_c)\)
  is the arithmetic genus as in Definition \ref{DEF_arithmetic_genus} and
  \(0 < \hbar = \hbar(M, \eta, \overline{J}, \bar{g}, 1, C)\) is the
  positive constant guaranteed by Theorem  \ref{THM_energy_threshold}.
\end{lemma}
%

We will postpone the proof of Lemma \ref{LEM_bounded_connected_components}
  until later; for now we continue with the proof of Proposition
  \ref{PROP_feral_limit_curves}, and to that end, all that remains is to
  establish that \({\rm Punct}(\overline{S}_c) < \infty\),  which we will
  prove by contradiction.
Thus, assuming \({\rm Punct}(\overline{S}_c) = \infty\), we make use of
  the following result.

\newcounter{CounterSectionImpossibleSubmanifold}
\newcounter{CounterLemmaImpossibleSubmanifold}
\setcounter{CounterSectionImpossibleSubmanifold}{\value{section}}
\setcounter{CounterLemmaImpossibleSubmanifold}{\value{lemma}}
\begin{lemma}[impossible submanifold]
  \label{LEM_impossible_submanifold}
  \hfill\\
Let \(c\geq 0\) and let 
  \begin{equation*}                                                       
    \bar{\mathbf{w}}_c=(\bar{w}_c, \overline{S}_c, \bar{j}_c,
    \mathbb{R}\times M,
    \overline{J},\emptyset,  \overline{D}_c),
    \end{equation*} 
  be an exhaustive limit of some subsequence of the \(\mathbf{w}_{k, c}\).
If \({\rm Punct}(\overline{S}_c)=\infty\) \emph{then} there exists a
  compact manifold with smooth boundary \(\Sigma\subset \overline{S}_c\)
  with the following properties.
\begin{enumerate}[(g1)]                                                   
  \item 
  \(\#\pi_0(\Sigma) = \#\pi_0(\overline{S}_c) \)
  \item \label{EN_g2}
  \(\#\pi_0(\partial \Sigma)\geq 12(1+\hbar^{-1})C\)
  \item \label{EN_g3}
  each connected component of \(\overline{S}_c\setminus (\Sigma\setminus
  \partial \Sigma)\)  is non-compact.
  \end{enumerate} 

\end{lemma}
%

Again we postpone the proof of Lemma \ref{LEM_impossible_submanifold}
  until after we have completed the proof of Proposition
  \ref{PROP_feral_limit_curves}.
We pause for a moment to collect the structure of our argument.
We are proving Proposition \ref{PROP_feral_limit_curves}, which amounts to
  establishing properties (F\ref{EN_F1}) - (F\ref{EN_F4}).
Property (F\ref{EN_F2}) was easily established, and properties
  (F\ref{EN_F1}) and (F\ref{EN_F3}) are established by Lemma
  \ref{LEM_bounded_connected_components}, although the proof is deferred
  until later.
All that remains is to prove property (F\ref{EN_F4}), which is that \({\rm
  Punct}(\overline{S}_c)<\infty\).
We will prove property (F\ref{EN_F4}) by contradiction, and hence assume
  \({\rm Punct}(\overline{S}_c)= \infty\), and as a consequence of this
  contradiction hypothesis, we can apply Lemma
  \ref{LEM_impossible_submanifold}, which will guarantee the existence of a
  compact submanifold with smooth boundary  \(\Sigma\subset
  \overline{S}_c\) with a number of implausible properties.
In particular, \(\Sigma\) will have a very large number of essential
  boundary components, and this is the feature that we will exploit in order
  to derive our desired contradiction, which will hence establish that
  indeed \({\rm Punct}(\overline{S}_c)< \infty\).
Thus, modulo the proofs of Lemma \ref{LEM_bounded_connected_components}
  and Lemma \ref{LEM_impossible_submanifold}, we will complete the proof of
  Proposition \ref{PROP_feral_limit_curves} by showing that although we
  have
  \begin{align*}                                                          
    \#\pi_0(\partial \Sigma) \geq 12 (1+\hbar^{-1})C,
    \end{align*}
  we must also have 
  \begin{align*}                                                          
      \#\pi_0(\partial \Sigma) \leq 11 (1+\hbar^{-1})C;
      \end{align*}
  this will be the desired contradiction.

Continuing on with the proof of Proposition \ref{PROP_feral_limit_curves},
  we have assumed that 
  $$
  {\rm Punct}(S) = \infty,
  $$
   and thus we may assume
  that the conclusions of Lemma \ref{LEM_impossible_submanifold} are true.
Consequently, we let \(\widetilde{S} = \Sigma_0\) be the surface guaranteed   
  by Lemma \ref{LEM_impossible_submanifold}, and we define
  \((\widetilde{S}, \tilde{j}, \widetilde{D})\) to be the compact nodal
  Riemann surface with boundary for which
  \(\tilde{j}:=j\big|_{\widetilde{S}}\) and \(\widetilde{D}:=D\cap
  \widetilde{S}\).

Next we recall that \(\bar{\mathbf{w}}_c\) is the exhaustive Gromov
  limit of a suitable subsequence of the curves \(\mathbf{w}_{k,c}\).
We further recall that the domain of the former is \((\overline{S}_c,
  \bar{j}_c, \overline{D}_c)\) and the domains of the latter are
  \((\widehat{S}_{k,c}, j_{k,c}, \widehat{D}_{k,c})\); see equation
  (\ref{EQ_wkc}).
Also recall from equations (\ref{EQ_Skc}) and (\ref{EQ_whSkc}) that we have
  defined the Riemann surfaces \((S_{k,c}, j_{k,c}, D_{k,c})\), which have
  the property that \(\widehat{S}_{k,c}\subset S_{k,c}\) and
  \(\widehat{D}_{k,c}\subset D_{k,c}\).
We then employ exhaustive Gromov compactness\footnote{See Theorem
  \ref{THM_exhaustive_gromov_compactness}.} to obtain decorations
  \(\widetilde{r}\), \(\hat{r}_{k, c}\), and \(r_{k,c}\) respectively for
  \((\widetilde{S}, \tilde{j}, \widetilde{D})\), \((\widehat{S}_{k, c},
  j_{k, c} , \widehat{D}_{k, c})\) in the sense of Definition
  \ref{DEF_decorated_nodal_riemann_surface}, and \((S_{k, c}, j_{k, c} ,
  D_{k, c})\), and we obtain embeddings
  \begin{equation*}                                                       
    \phi_k: \widetilde{S}^{\widetilde{D}, \tilde{r}} \to
    \widehat{S}_{k,c}^{\widehat{D}_{k,c},\hat{r}_{k,c}} \hookrightarrow
    S_{k,c}^{D_{k,c}, r_{k,c}},
    \end{equation*}
  for all sufficiently large \(k\in \mathbb{N}\).
We then fix some sufficiently large \(k\in \mathbb{N}\), and we define 
  \begin{equation*}                                                       
    \Sigma_0 = S_{k, c}^{D_{k,c}, r_{k,c}}, \qquad \Sigma_1 :=
    \phi_k(\widetilde{S}^{\widetilde{D}, \tilde{r}}) \subset
    \Sigma_0\qquad\text{and}\qquad \Sigma_2:={\rm cl}\big(\Sigma_0\setminus
    \Sigma_1).
    \end{equation*}
In particular, we will assume that \(k\in \mathbb{N}\) has been chosen
  sufficiently large so that for each connected component \(\Sigma'\) of
  \(\Sigma_2\) for which \(\partial \Sigma' \subset \partial \Sigma_1\) we
  have
  \begin{align*}                                                          
    \sup_{\zeta\in \Sigma'} a \circ u_k^{a_k+c}(\zeta) - \sup_{\zeta\in
    \partial \Sigma'} a \circ u_k^{a_k+c}(\zeta) \geq 1.
    \end{align*}
That \(k\in \mathbb{N}\) can be chosen sufficiently large to arrange this
  follows from property (g\ref{EN_g3}) together with the definition of
  exhaustive Gromov convergence; see Definition
  \ref{DEF_exhaustive_gromov_convergence}.
As a consequence of this inequality, we then immediately have the
  following.
\begin{lemma}[energy threshold acquired]
  \label{LEM_energy_threshold_acquired}
  \hfill\\
Let \(\Sigma'\) be a connected component of \(\Sigma_2\) for which
  \(\partial \Sigma'\subset \partial \Sigma_1\).
Then 
  \begin{align*}                                                          
    \int_{\Sigma'} (u_k^{a_k +c})^* \omega \geq \hbar,
    \end{align*}
  where \(\hbar = \hbar(M, \eta, \overline{J}, \bar{g}, 1, C)>0\) is the
  positive constant guaranteed by Theorem \ref{THM_energy_threshold}.
\end{lemma}
%
\begin{proof}
This follows immediately from Theorem \ref{THM_energy_threshold}.
\end{proof}
The reader should note that the situation described in the lemma is the
  phenomenon where a cap, perhaps with some universally bounded genus,
  flies away.
Each such occurrence takes at least an \(\hbar\)-amount of
  \(\omega\)-energy away.
If for the initial sequence of (compact) pseudoholomorphic curves the
  number of boundary components as well as the total \(\omega\)-energy is
  bounded, then the number of occurrences just described must be bounded.
Of course, we will need to establish the details in order to get better
  bounds on constants.
For now we now turn our attention to more topological estimates.

We note that by property (g\ref{EN_g2}) of Lemma
  \ref{LEM_impossible_submanifold} and the definition of
  \(\Sigma_1\), we have
\begin{equation}\label{EQ_Sig1_num_bdry_components}                       
  \#\pi_0(\partial \Sigma_1) \geq 12(1+\hbar^{-1})C.
  \end{equation}
We then recall that the Euler characteristic is additive, so that 
  \begin{equation}\label{EQ_euler_additivity}                             
    \chi(\Sigma_0)=\chi(\Sigma_1)+\chi(\Sigma_2) .
    \end{equation}
We also recall that the Euler characteristic is given by
  \begin{equation}\label{EQ_euler_char}                                   
    \chi(\Sigma_0) = 2 \# \pi_0(\Sigma_0) - 2{\rm Genus}(\Sigma_0) - \#
    \pi_0(\partial \Sigma_0),
    \end{equation}
  and similarly for \(\Sigma_1\) and \(\Sigma_2\).
Combining equations (\ref{EQ_euler_additivity}) and (\ref{EQ_euler_char}),
  we find
  \begin{align*}                                                          
    &2\#\pi_0(\Sigma_0) - 2g(\Sigma_0) - \# \pi_0(\partial \Sigma_0)
     \\
&= \chi(\Sigma_0)
    =
    \chi(\Sigma_1) + \chi(\Sigma_2)
    \\
    &=
    2\#\pi_0(\Sigma_1) - 2g(\Sigma_1) - \# \pi_0(\partial \Sigma_1)
    \\
    &\quad+ 2\#\pi_0(\Sigma_2) - 2g(\Sigma_2) -  \# \pi_0(\partial
    \Sigma_2)
    \\
    &=
    2\#\pi_0(\Sigma_1) - 2g(\Sigma_1) - \# \pi_0(\partial \Sigma_1)
    \\
    &\quad+ 2\#\pi_0(\Sigma_2) - 2g(\Sigma_2) - \big(\# \pi_0(\partial
    \Sigma_0)  + \# \pi_0(\partial \Sigma_1) \big),
    \end{align*}
  where we have simplified the notation by writing  \(g(\Sigma_0)= {\rm
  Genus}(\Sigma_0)\), and to obtain the final equality, we have made use
  of the following observation:
  \begin{equation*}                                                       
    \# \pi_0(\partial \Sigma_2) = \# \pi_0(\partial \Sigma_0)  + \#
    \pi_0(\partial \Sigma_1).
    \end{equation*}
After rearranging and simplifying, we have the following estimate.
  \begin{align}                                                           
  \# \pi_0(\partial \Sigma_1) &= g(\Sigma_0)  - g(\Sigma_1) - g(\Sigma_2)
    + \#\pi_0(\Sigma_1) + \#\pi_0(\Sigma_2) -  \#\pi_0(\Sigma_0)\notag\\
    &\leq g(\Sigma_0)  + \#\pi_0(\Sigma_1) + \#\pi_0(\Sigma_2)  
    \label{EQ_boundary_bound}
    \end{align}
Next we estimate the genus of \(\Sigma_0\), as follows:
  \begin{align}                                                           
    {\rm Genus}(\Sigma_0) 
    &= {\rm Genus}(S_{k,c}^{D_{k,c}, r_{k,c}}) \notag
    \\
    &= {\rm Genus}_{arith}(S_{k,c}, j_{k,c}, D_{k,c})\notag
    \\
    &=\#\pi_0(|S_{k,c}|) - \#\pi_0(S_{k,c}) + {\rm Genus}(S_{k,c}) +
    {\textstyle \frac{1}{2}} \# D_{k,c}\notag
    \\
    &\leq 
    \#\pi_0(|S_{k,c}|) + {\rm Genus}(S_{k,c}) + {\textstyle
    \frac{1}{2}} \# D_{k, c} \notag
    \\
    &\leq 3C, \label{EQ_genus_bound}
    \end{align}
  where to obtain the third equality we have employed Lemma A.1 from
  \cite{FH1}, and to obtain the final inequality, we have used the fact
  that \(S_{k, c} = S_k^{a_k +c}\), and hence by the assumptions of
  Theorem \ref{THM_existence} we have \(\#\pi_0(|S_{k,c}|) = 1 \leq C\),
  \({\rm Genus}(S_{k,c})\leq C\), and \({\textstyle \frac{1}{2}} \# D_{k,
  c} \leq \frac{1}{2} C \leq C\).
Combining inequality (\ref{EQ_boundary_bound}) with inequality
  (\ref{EQ_genus_bound}) then yields
  \begin{align}\label{EQ_component_bound_a}                               
    \# \pi_0(\partial \Sigma_1) \leq 3C +  \#\pi_0(\Sigma_1) +
    \#\pi_0(\Sigma_2).
    \end{align}
We can then estimate \(\#\pi_0(\Sigma_1)\) as follows. 
\begin{align*}                                                            
  \#\pi_0 (\Sigma_1) 
  &=\#\pi_0(\widetilde{S}^{\widetilde{D}, \tilde{r}} )
  &\text{by Definition of }\Sigma_1
  \\
  &\leq \#\pi_0(\widetilde{S})
  &\text{by properties of nodal curves}
  \\
  &= \#\pi_0(\overline{S}_c)
  &\text{by Lemma \ref{LEM_impossible_submanifold}}
  \\
  &\leq 6(1 +\hbar^{-1} )C
  &\text{by Lemma \ref{LEM_bounded_connected_components}}
  \end{align*}
Or in other words,
  \begin{equation}\label{EQ_component_bound_1}                            
    \#\pi_0 (\Sigma_1) \leq 6(1 +\hbar^{-1} )C.
    \end{equation}

To proceed further, we partition \(\Sigma_2\) into three disjoint sets
  denoted \(\Sigma_2^{bdry}\),  \(\Sigma_2^{int}\), and
  \(\Sigma_2^{const}\); here \(\Sigma_2^{const}\) consists of connected
  components of \(\Sigma_2\) on which the map \(u_k\) is constant,
  \(\Sigma_2^{bdry}\) consists of connected components of \(\Sigma_2\)
  which have non-trivial intersection with \(\partial \Sigma_0\), and we
  define \(\Sigma_2^{int}:=\Sigma_2\setminus(\Sigma_2^{const}\cup
  \Sigma_2^{bdry})\).

As a consequence of the fact that the number of connected components of
  the \(\partial S_k^b\) is uniformly bounded by \(C\), it follows from
  the definition of \(\Sigma_0\) and \(\Sigma_2^{bdry}\) that we must have
  \(\# \pi_0(\Sigma_2^{bdry})\leq C\).
Also, because the curves 
  \begin{equation*}                                                       
    \mathbf{u}_k^b=\big(u_k^b, S_k^b, j_k^b, (-\infty, 1)\times M, J_k,
    \emptyset, D_k^b \big)
    \end{equation*}
  are stable and without marked points, it follows that each connected
  component of \(\Sigma_2^{const}\) must contain a nodal point in  
  \(D_k^b\).  Recalling that \(\#D_k^b\leq C\)  it follows that
  \(\#\pi_0(\Sigma_2^{const}) \leq C\).
Combining these two inequalities then yields
  \begin{equation}\label{EQ_component_bound_2}                            
    \#\pi_0(\Sigma_2^{const})  + \#\pi_0(\Sigma_2^{bdry}) \leq 2C.
    \end{equation} 
Lastly we note that \(\Sigma_2^{int}\) consists of connected
  components on which \(u_k^b\) is non-constant, and \(\partial
  \Sigma_2^{int} \subset \partial \Sigma_1\), and hence by Lemma
  \ref{LEM_energy_threshold_acquired} and the assumption that our curves
  have \(\omega\)-energy bounded by \(C\), we have
  \begin{equation}\label{EQ_component_bound_3}                            
    \hbar\cdot  \#\pi_0( \Sigma_2^{int})\leq
    \int_{\Sigma_2^{int}}u_{k,c}^* \omega \leq C.   
    \end{equation}
Here we have abused notation somewhat since, strictly speaking,
  \(\Sigma_2^{int}\) is a circle-compactified surface rather
  than a domain of a pseudoholomorphic curve, however this can be made
  rigorous by removing the added-special circles from
  \(\Sigma_2^{int}\) in the above integral; in any case, the
  desired estimate \(\hbar\cdot  \# \pi_0(\Sigma_2^{int}) \leq C\)
  holds.
Combining inequalities (\ref{EQ_component_bound_1}),
  (\ref{EQ_component_bound_2}), and (\ref{EQ_component_bound_3}) with
  inequality (\ref{EQ_component_bound_a}) then yields
  \begin{align*}                                                          
    \# \pi_0(\partial \Sigma_1) &\leq 3C + 6(1+\hbar^{-1})C + 2C
    +\hbar^{-1}C
    \\
    &\leq 11(1+\hbar^{-1})C.
    \end{align*}
However, combining the above inequality with inequality
  (\ref{EQ_Sig1_num_bdry_components})
  \begin{equation*}                                                       
    12(1+\hbar^{-1})C\leq \#\pi_0(\partial \Sigma_1) \leq
    11(1+\hbar^{-1})C
    \end{equation*}
  which is the desired contradiction, which establishes that we must have
  \({\rm Punct}(\overline{S}_c)<\infty\).
Thus, modulo the proofs of Lemma \ref{LEM_bounded_connected_components}
  and Lemma \ref{LEM_impossible_submanifold}, we have completed the proof
  of Proposition \ref{PROP_feral_limit_curves}.
\end{proof}

We now turn our attention to the proof of Lemma
  \ref{LEM_bounded_connected_components}.
We begin with a restatement.

\setcounter{CurrentSection}{\value{section}}
\setcounter{CurrentLemma}{\value{lemma}}
\setcounter{section}{\value{CounterSectionBoundedConnectedComponents}}
\setcounter{lemma}{\value{CounterLemmaBoundedConnectedComponents}}
\begin{restatementlemma}[Some bounds on the limit curve]\hfill\\
For the pseudoholomorphic curve 
  \begin{equation*}                                                       
    \bar{\mathbf{w}}_c=(\bar{w}_c, \overline{S}_c, \bar{j}_c,
    \mathbb{R}\times M, \overline{J},\emptyset , \overline{D}_c),
    \end{equation*} 
  defined above, the following inequalities hold.
\begin{enumerate}                                                         
  \item 
  \({\rm Genus}_{arith}(\overline{S}_c,\bar{j}_c,\overline{D}_c)\leq 3C\)
  \item 
  \(\#\pi_0(\overline{S}_c)< 6(1  + \hbar^{-1} )C\)
  \item 
  \(\#\overline{D}_c <18(1+\hbar^{-1})C \);
  \end{enumerate}
  where \({\rm Genus}_{arith}(\overline{S}_c, \bar{j}_c, \overline{D}_c)\)
  is the arithmetic genus as in Definition \ref{DEF_arithmetic_genus} and
  \(0 < \hbar = \hbar(M, \eta, \overline{J}, \bar{g}, 1, C)\) is the
  positive constant guaranteed by Theorem  \ref{THM_energy_threshold}.
\end{restatementlemma}
%
\begin{proof}
\setcounter{section}{\value{CurrentSection}}
\setcounter{lemma}{\value{CurrentLemma}}
In an effort to simplify notation a bit, we will drop the subscripts
  \(c\), and write, for example, \(\bar{\mathbf{w}}\) and \(\overline{S}\)
  instead of \(\bar{\mathbf{w}}_c\) and \(\overline{S}_c\).

We begin by recalling Definition \ref{DEF_arithmetic_genus} which
  guarantees that
  \begin{align*}                                                          
    {\rm Genus}_{arith}(\overline{S}, j, \mu, D) = {\rm
    Genus}(\overline{S}^{D, r}).
    \end{align*}
Moreover, the genus of a non-compact surface is obtained as the limit of
  genera of an exhausting sequence of compact surfaces with boundary.
By genus super-additivity\footnote{See Lemma \ref{LEM_genus_addition}.},
  the definition of exhaustive Gromov compactness\footnote{See Definition
  \ref{DEF_exhaustive_gromov_convergence}.}, and properties of the
  \(S_k^b\), it follows that
  \begin{align}\label{EQ_bound_on_arith_genus_a}                          
    {\rm Genus}_{arith}(\overline{S},\bar{j}, \overline{D}) \leq \sup_{k,
    b} {\rm Genus}_{arith}(S_k^b,j_k^b, D_k^b)
    \end{align}
However, recall Lemma A.1 from \cite{FH1} which provides a formula
  for the arithmetic genus of a compact Riemann surface with boundary:
  \begin{align}\label{EQ_formula_arithmetic_genus}                        
    &{\rm Genus}_{arith}(\overline{S}, j, D)
    \\
    &\qquad=\#\pi_0(|\overline{S}|)-\#\pi_0(\overline{S}) +\Big(
    \sum_{\Sigma\in \pi_0(\overline{S})} {\rm
    Genus}(\Sigma) \Big) +{\textstyle \frac{1}{2}}\# D. \notag
    \end{align}
In light of the bounds we have on \({\rm Genus}(S_k^b)\) and \(\#D_k^b\)
  due to the hypotheses of Theorem \ref{THM_existence}, we immediately see
  that
  \begin{equation}\label{EQ_bound_on_arith_genus}                         
    {\rm Genus}_{arith}(\overline{S},\bar{j}, \overline{D}) \leq 1 +
    2C\leq 3C.
    \end{equation}
This establishes the first desired inequality; the next two will require a
  bit more effort.

We pause for a moment to highlight the difficulty in proving the second  
  desired inequality, namely that
  \begin{align*}                                                          
    \#\pi_0(\overline{S})< 6(1  + \hbar^{-1} )C.
    \end{align*}
If the map \(\bar{w}\) were non-constant on each connected
  component of \(\overline{S}\), then of course the estimate (in fact a
  better estimate) would follow quickly.
Thus the main difficulty is to establish a bound on the number of constant
  components.
Because \(\bar{\mathbf{w}}\) is stable, we could easily bound the number
  of constant components in terms of the number of nodal points, but we do
  not have an a priori bound on that either, since the number of nodal
  points can increase in the exhaustive Gromov limit of a sequence of
  curves.
Finally, it would also be easier to bound the number of constant
  components if we knew that either the non-constant components were compact
  or we knew that the number of nodal points was finite, however a priori
  we know neither of these.
As such, the path to obtaining the desired bound may not seem
  straightforward, even though the basic idea is; that is, we essentially
  aim to use the stability condition plus an energy threshold to bound the
  number of constant components in terms of \(\omega\)-energy.
This is the tack we take, and we return to the proof presently.

The next step is to define the set \(\mathcal{I}_{reg}\subset \mathbb{R}\)
  to be the intersection of the set \( [c,\infty)\setminus a\circ
  \bar{w}(\overline{D})\subset (c,\infty)\) with the set of regular values
  of the function \(a\circ \bar{w}: \overline{S}\to \mathbb{R}\).
We note that \(\mathcal{I}_{reg}\) is an open and dense subset of \((c,
  \infty)\).
Next, for each \(x\in \mathcal{I}_{reg}\) we define a compact nodal
  Riemann surface \((S^x, j^x, D^x)\) in the following manner.
First, we enumerate the connected components of
 \(\overline{S}\) by \(\overline{S}_k\), so that
  \(\overline{S}=\bigcup_{k=1}^\infty \overline{S}_k\).
Next, on each connected component 
\(\overline{S}_k\) we choose \(\zeta_k\in \overline{S}_k\) 
so that
  \begin{equation*}                                                       
    \inf_{\zeta\in S_k} a\circ \bar{w}(\zeta) = a\circ \bar{w}(\zeta_k).
    \end{equation*}
We denote the collection of these points by \(Z=\{\zeta_1, \zeta_2,
  \ldots\}\).
For each \(x\in \mathcal{I}_{reg}\) we then define
  \begin{equation*}                                                       
    \Sigma^x:=(a\circ \bar{w})^{-1}\big((-\infty, x]\big)
    \end{equation*}
  \begin{equation*}                                                       
    \mathcal{S}^x:= \big\{\Sigma\in \pi_0(\Sigma^x): Z\cap \Sigma \neq
    \emptyset \; \; \text{and}\; \; a\circ \bar{w} (Z\cap \Sigma) \leq x
    -1\ \big\}
    \end{equation*}
    \begin{figure}[h]
\includegraphics[scale=0.3]{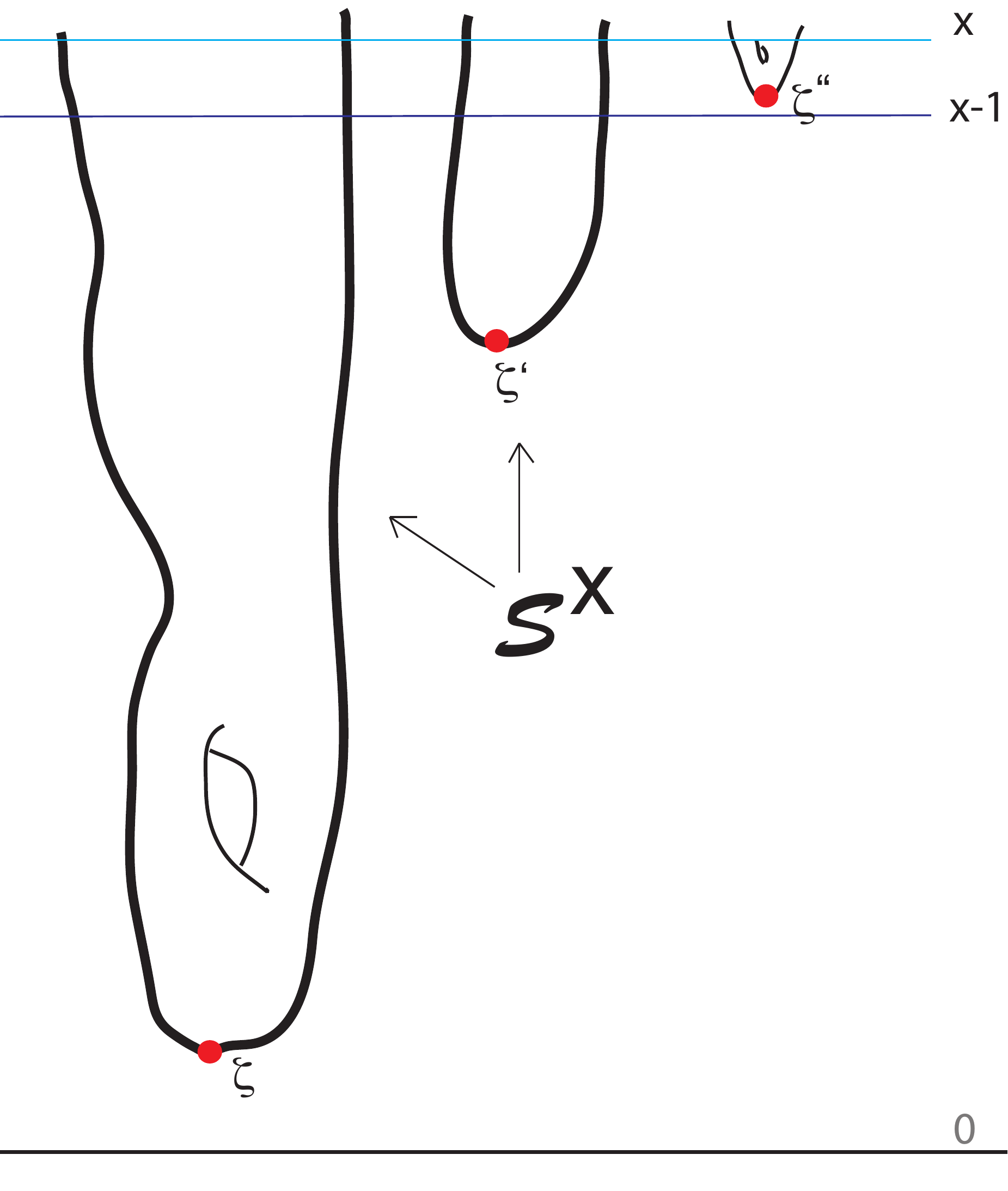}\ \ \ \includegraphics[scale=0.3]{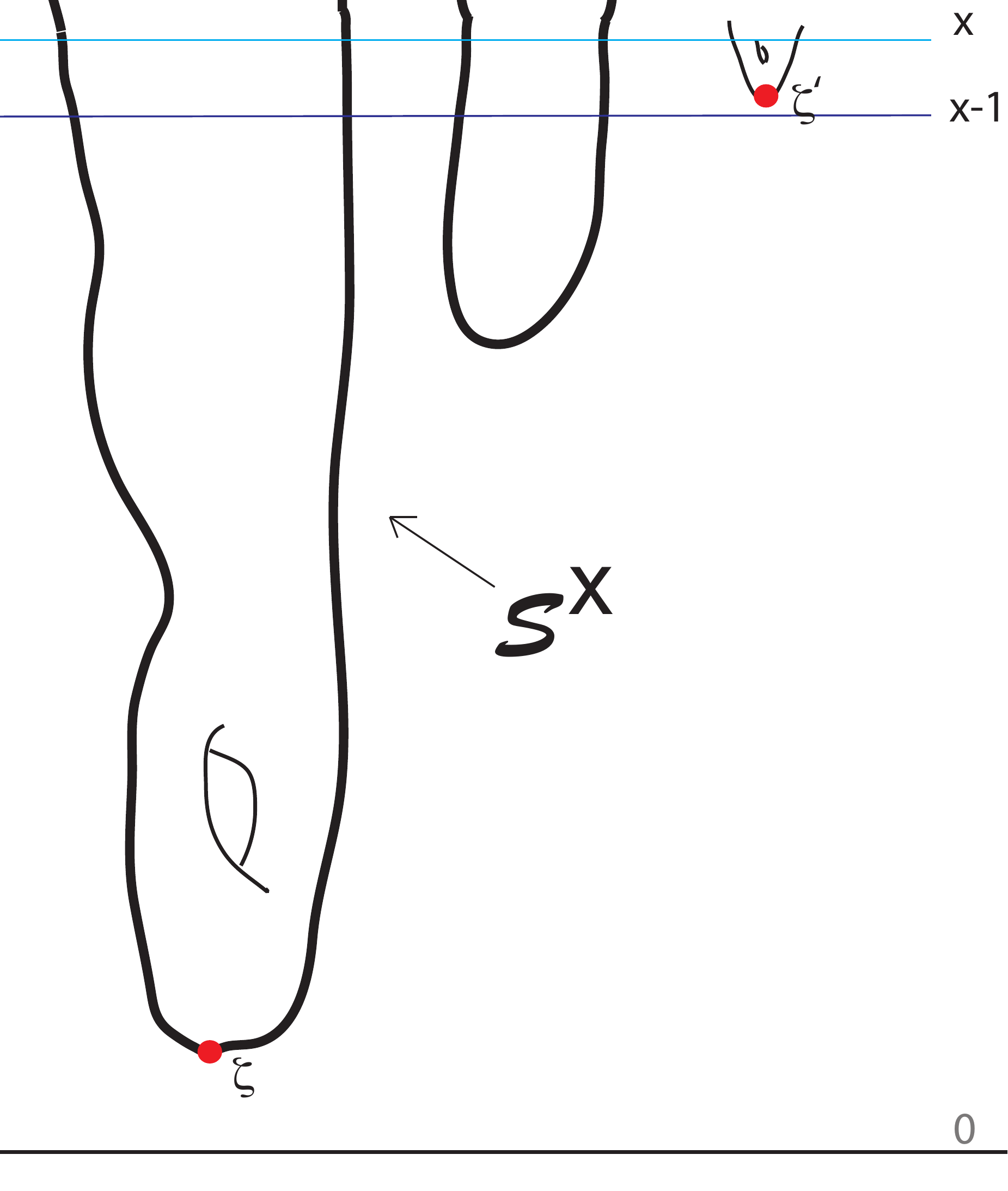}
\caption{Two examples of the set ${\mathcal S}^x$. In the left figure it
  has two elements and in the right figure one element. These components are
  obtained from the sets indicated by taking those points for which
  $a\circ\bar{w}$ takes a value not exceeding $x$.}
\label{FIG4}
\end{figure}
Observe that  ${\mathcal S}^x$ is a finite set.
It is worth pausing to describe this set \(\mathcal{S}^x\).
Indeed, this can be regarded as a set of ``essential'' connected
  components of \(\Sigma^x\), where by essential we mean those components
  which contain both a marker \(\zeta_k\) which identifies connected
  components of \(\overline{S}\), and those components on which the
  minimum value of \(a\circ \bar{w}\) differs from \(x\) (which will often
  be that maximal value of \(a\circ \bar{w}\)) by at least 1.
We will exploit these features momentarily, but we first must continue our
  definition of \(S^x\).

Next we aim to define a certain collection of Riemann surfaces which we
  denote \({\rm Stab}^x\).
To do this, we let \(2^{\mathcal{S}^x}\) denote the power set of
  \(\mathcal{S}^x\), we let \(\sigma:D\to D\) denote the involution
  satisfying \(\sigma(\underline{d}_i)=\overline{d}_i\) and
  \(\sigma(\overline{d}_i) = \underline{d}_i\) for each \(\underline{d}_i,
  \overline{d}_i\in D\)
and we say a triple \((\widetilde{\Sigma}, \tilde{j}, \widetilde{D})\) is
  \(\bar{\mathbf{w}}\)-stable provided it is a nodal Riemann surface with
  \(\widetilde{\Sigma}\subset \overline{S}\) \(\widetilde{D}\subset
  \overline{D}\), and for each connected component \(\Sigma\subset
  \widetilde{\Sigma}\) for which \(\bar{w}:\Sigma\to W\) is constant we
  have 
  \begin{equation*}                                                    
    2 {\rm Genus}(\Sigma) + \# (\widetilde{D}\cap \Sigma) \geq 3.
    \end{equation*}
We then define \({\rm Stab}^x\) via the following.
\begin{align*}                                                         
  {\rm Stab}^x :=\Big\{(\widetilde{\Sigma}, \tilde{j}, \widetilde{D}):
  \qquad &\widetilde{\Sigma} = \bigcup_{\Sigma\in \mathcal{A}} \Sigma \;
  \; \text{where} \; \mathcal{A} \in 2^{\mathcal{S}^x},  \; \; \;\;
  \tilde{j} = j\big|_{\widetilde{\Sigma}}
  \\
  &\widetilde{D} \subset \widetilde{\Sigma}\cap D\; \; \text{satisfies} \;
  \sigma(\widetilde{D}) = \widetilde{D}
  \\
  &\text{and} \; (\widetilde{\Sigma}, \tilde{j}, \widetilde{D}) \;
  \text{is }\mathbf{w}\text{-stable}\qquad\qquad \Big\}
  \end{align*}
Observe that ${\mathcal S}^x$ for $x\in (c,\infty)$ is nonempty. 
We introduce  a partial order on \({\rm Stab}^x\) by defining 
  \((\widetilde{\Sigma}_1, \tilde{j}_1, \widetilde{D}_1) \leq
  (\widetilde{\Sigma}_2, \tilde{j}_2, \widetilde{D}_2)\) if and only if
  \(\widetilde{\Sigma}_1 \subset \widetilde{\Sigma}_2\) and
  \(\widetilde{D}_1 \subset \widetilde{D}_2\).
Finally, we note that given two elements \((\widetilde{\Sigma}_1,
  \tilde{j}_1, \widetilde{D}_1),  (\widetilde{\Sigma}_2,
  \tilde{j}_2, \widetilde{D}_2)\in {\rm Stab}^x\) their union (in the
  obvious manner) is again in \({\rm Stab}^x\), and hence the partially
  ordered set \({\rm Stab}^x\) has a greatest element.
We define \((S^x, j^x, D^x)\) to be the greatest element of \({\rm Stab}^x\).
At this point, we have defined the compact  nodal Riemann surface \((S^x,
  j^x, D^x)\), which may have boundary.

The definition provided may seem convoluted, however it has a number of
  features we now state and will exploit momentarily.

First, we note that for each \(x,y\in \mathcal{I}_{reg}\) with \(x< y\),
  we have \(\# \pi_0(S^x) \leq \# \pi_0(S^y) \leq \# \pi_0(\overline{S})\),
  and \(\bigcup_{x\in \mathcal{I}_{reg}} S^x = \overline{S}\), from which
  we conclude that
  \begin{equation}\label{EQ_limit_of_num_of_con_comp}                     
    \lim_{x\to \infty} \#\pi_0(S^x) = \# \pi_0(S).
    \end{equation}
Similarly, for each \(x,y\in \mathcal{I}_{reg}\) with \(x<y\), we have
  \begin{equation*}                                                       
    {\rm Genus}_{arith}(S^x, j^x, D^x) \leq {\rm Genus}_{arith}(S^y, j^y,
    D^y),
    \end{equation*}
  and by definition of the arithmetic genus, we have
  \begin{equation}\label{EQ_limit_of_arith_genus}                         
    \lim_{x\to \infty} {\rm Genus}_{arith}(S^x, j^x, D^x) = {\rm
    Genus}_{arith}(\overline{S}, \bar{j}, \overline{D}).
    \end{equation}
    Second, we let \(\hbar=\hbar(M, \eta, \overline{J}, \bar{g}>0, 1, C)>0 \)
  be the positive constant guaranteed by Theorem
  \ref{THM_energy_threshold}, which has the property that on each
  connected component \(\Sigma\) of \(S^x\) on which \(\bar{w}\) is
  non-constant, we have
  \begin{equation*}                                                       
    \int_\Sigma \bar{w}^*\omega \geq \hbar.
    \end{equation*}
To make use of this property, we first decompose \(S^x\) into two sets,
  \(S_{const}^x\) and \(S_{nc}^x\), where \(S_{const}^x\) is the union of
  connected components on which \(\bar{w}\) is constant and \(S_{nc}^x=S^x
  \setminus S_{const}^x \), and we then recall that \(\omega\) evaluates
  non-negatively on \(\overline{J}\)-complex lines, so that by properties
  of the \(\mathbf{w}_k\) and exhaustive Gromov compactness, we have
  \begin{equation}\label{EQ_Snc_bound}                                    
    C \geq \int_{S_{nc}^x} \bar{w}^*\omega \geq \hbar \cdot
    \#\pi_0(S_{nc}^x).
    \end{equation}
Third, the \(\bar{\mathbf{w}}\)-stability condition guarantees that for
  each \(\Sigma \in S_{const}^x\) we have
  \begin{equation*}                                                       
    2 {\rm Genus}(\Sigma) + \# (D^x\cap \Sigma) \geq 3.
    \end{equation*}
To make use of this, it will be convenient to define 
  \begin{equation*}                                                       
    S_{const}^x(k):=\Big\{\Sigma^x\in \pi_0(S_{const}^x): {\rm
    Genus}(\Sigma^x) = k  \Big\},
    \end{equation*}
  in which case we can estimate:
  \begin{equation}\label{EQ_stability_inequality}                         
    3\#\pi_0\big(S_{const}^x(0)\big)  + \#\pi_0\big(S_{const}^x(1)\big)
    \leq \# D^x.
    \end{equation}
We are now prepared to complete the proof of Lemma
  \ref{LEM_bounded_connected_components}.
As above, we have a formula for the arithmetic genus of \((S^x, j^x,
  D^x)\) given by
  \begin{align}\label{EQ_formula_for_appx_arith_genus}                    
    &{\rm Genus}_{arith}(S^x, j^x, D^x)
    \\
    &\qquad=\#\pi_0(|S^x|)-\#\pi_0(S^x)
    +\Big( \sum_{\Sigma^x\in \pi_0(S^x)} {\rm Genus}(\Sigma^x) \Big)
    +{\textstyle \frac{1}{2}}\# D^x. \notag
    \end{align}
We then note that 
  \begin{equation*}                                                       
    \# \pi_0(S^x) =\#\pi_0(S_{nc}^x)  + \sum_{k=0}^\infty
    \#\pi_0\big(S_{const}^x(k)\big),
    \end{equation*}
  and we recall that 
  \begin{equation*}                                                       
    \sum_{\Sigma^x\in \pi_0(S^x)} {\rm Genus}(\Sigma^x) = {\rm
    Genus}(S_{nc}^x) +\sum_{k=1}^\infty k\cdot \#\pi_0\big(
    S_{const}^x(k)\big) \geq \sum_{k=1}^\infty k\cdot \#\pi_0\big(
    S_{const}^x(k)\big)
    \end{equation*}
  which is finite since \(\sup \{k\in \mathbb{N}: S_{const}^x(k)\neq
  \emptyset \} \leq {\rm Genus}(\overline{S}) \leq C\).

Combining the above two (in)equalities with inequality
  (\ref{EQ_stability_inequality}) and the formula for the arithmetic genus
  then yields the following.
  \begin{align*}                                                          
    &{\rm Genus}_{arith}(S^x, j^x, D^x)\\
    &\qquad=\#\pi_0(|S^x|)-\#\pi_0(S^x)
    +\Big( \sum_{\Sigma^x\in \pi_0(S^x)} {\rm Genus}(\Sigma^x) \Big)
    +{\textstyle \frac{1}{2}}\# D^x
    \\
    &\qquad\geq\#\pi_0(|S^x|)
    -\#\pi_0(S_{nc}^x)  - \sum_{k=0}^\infty
    \#\pi_0\big(S_{const}^x(k)\big)
    \\
    &\qquad\qquad+\sum_{k=1}^\infty k\cdot \#\pi_0\big(
    S_{const}^x(k)\big)
+{\textstyle \frac{3}{2}}\#\pi_0\big(S_{const}^x(0)\big)
    +{\textstyle \frac{1}{2}} \#\pi_0\big(S_{const}^x(1)\big)
    \\
    &\qquad=\#\pi_0(|S^x|)-\#\pi_0(S_{nc}^x)
 +\sum_{k=1}^\infty (k-1)\cdot \#\pi_0\big(
    S_{const}^x(k)\big)
    \\
    &\qquad\qquad+{\textstyle \frac{1}{2}}\#\pi_0\big(S_{const}^x(0)\big)
    +{\textstyle \frac{1}{2}} \#\pi_0\big(S_{const}^x(1)\big)
    \\
    &\qquad\geq \#\pi_0(|S^x|)-\#\pi_0(S_{nc}^x) +{\textstyle\frac{1}{2}}
    \#\pi_0(S_{const}^x)
    \\
    &\qquad\geq -\#\pi_0(S_{nc}^x) +{\textstyle\frac{1}{2}}
    \#\pi_0(S_{const}^x).
    \end{align*}
Or in other words, 
  \begin{equation*}                                                       
    2{\rm Genus}_{arith}(S^x, j^x, D^x) + 2\#\pi_0(S_{nc}^x) \geq
    \#\pi_0(S_{const}^x),
    \end{equation*}
  and thus
  \begin{equation*}                                                       
    \#\pi_0(S^x)\leq 2{\rm Genus}_{arith}(S^x, j^x, D^x) +
    3\#\pi_0(S_{nc}^x) .
    \end{equation*}
Next, we recall equations (\ref{EQ_limit_of_num_of_con_comp}),
  (\ref{EQ_limit_of_arith_genus}) and (\ref{EQ_Snc_bound}), which
  guarantee the following
  \begin{align*}                                                          
    \#\pi_0(\overline{S}) &= \lim_{x\to \infty} \#\pi_0(S^x)\\
    &\leq \lim_{x\to \infty}2{\rm Genus}_{arith}(S^x,
    j^x, D^x) + 3\hbar^{-1}  C
    \\
    &= 2{\rm Genus}_{arith}(\overline{S},
    \bar{j}, \overline{D}) + 3\hbar^{-1}  C
    \\
    &\leq 2(1+2C) + \hbar^{-1} 3C,
    \\
    &\leq 6(1 + \hbar^{-1})C
    \end{align*}
  where to obtain the second inequality we have employed inequality
  (\ref{EQ_bound_on_arith_genus}).
This establishes the desired bound on the number of connected components of
  \(\overline{S}\), and proves the second part of the conclusions of Lemma
  \ref{LEM_bounded_connected_components}.

To establish the third part of Lemma
  \ref{LEM_bounded_connected_components}, we recall equation
  (\ref{EQ_formula_for_appx_arith_genus}), which states the following.
  \begin{align*}                                                           
    &{\rm Genus}_{arith}(S^x, j^x, D^x)
    \\
    &\qquad=\#\pi_0(|S^x|)-\#\pi_0(S^x)
    +\Big( \sum_{\Sigma^x\in \pi_0(S^x)} {\rm Genus}(\Sigma^x) \Big)
    +{\textstyle \frac{1}{2}}\# D^x 
    \end{align*}
Solve for \(\#D^x\) and pass to the limit as \(x\to \infty\) in
  \(\mathcal{I}_{reg}\) to obtain the following:
  \begin{align*}                                                          
    \# \overline{D}  &= 2{\rm Genus}_{arith}(\overline{S}, \bar{j},
    \overline{D}) -2 \#\pi_0(|\overline{S}|)+2 \#\pi_0(\overline{S}) -2
    \Big( \sum_{\overline{\Sigma}\in \pi_0(\overline{S})} {\rm
    Genus}(\overline{\Sigma}) \Big)
    \\
    &\leq 2{\rm Genus}_{arith}(\overline{S}, \bar{j},
    \overline{D}) +2 \#\pi_0(\overline{S}) 
    \\
    &\leq 6C + 12(1+\hbar^{-1})C
    \\
    &\leq 18(1+\hbar^{-1})C
    \end{align*}
This is the desired estimate, which then completes the proof of Lemma
  \ref{LEM_bounded_connected_components}.
\end{proof}

At this point we note that we have proved Proposition
  \ref{PROP_feral_limit_curves} modulo only the proof of Lemma
  \ref{LEM_impossible_submanifold}, and so we turn our attention to that.
First however, it will be important to define a procedure called a
  \emph{cut}.
We make the definition precise below.

\begin{definition}[cut]
  \label{DEF_cut}
  \hfill \\
Let \(u:S\to \mathbb{R}\times M\) be a proper pseudoholomorphic map. 
Letting \(\mathcal{I}_{reg}\) denote the regular values of \(a\circ u\),
  and assuming \(a\circ u(z)\in \mathcal{I}_{reg}\), we define \({\rm
  cut}_z(S)\) by first defining \(\Gamma_z\) to be the connected component
  of \((a\circ u)^{-1}\big(a\circ u(z)\big)\) containing \(z\), and we let
  \({\rm cut}_z(S)\) be the surface obtained by gluing in two disjoint
  copies of \(\Gamma_z\) into \(S\setminus \Gamma_z \).
\end{definition}
For example, suppose  \(u:\mathbb{R}\times S^1\to  \mathbb{R}\times M\) is
  a proper pseudoholomorphic cylinder for which the function \(a\circ u\) has
  no critical points, then for each \(z\in \mathbb{R}\times S^1\) the
  surface \({\rm cut}_z(\mathbb{R}\times S^1)\) is diffeomorphic to the
  disjoint union of \((-\infty,0] \times S^1 \) and \([0,\infty)\times S^1\).

We now recall what we will prove.

\setcounter{CurrentSection}{\value{section}}
\setcounter{CurrentLemma}{\value{lemma}}
\setcounter{section}{\value{CounterSectionImpossibleSubmanifold}}
\setcounter{lemma}{\value{CounterLemmaImpossibleSubmanifold}}
\begin{restatementlemma}[impossible submanifold]\hfill\\
Let \(c\geq 0\) and let 
  \begin{equation*}                                                       
    \bar{\mathbf{w}}_c=(\bar{w}_c, \overline{S}_c, \bar{j}_c,
    \mathbb{R}\times M,
    \overline{J},\emptyset, \overline{D}_c),
    \end{equation*} 
  be an exhaustive limit of some subsequence of the \(\mathbf{w}_{k, c}\).
If \({\rm Punct}(\overline{S}_c)=\infty\) \emph{then} there exists a
  compact manifold with smooth boundary \(\Sigma\subset \overline{S}_c\)
  with the following properties.
\begin{enumerate}[(g1)]                                                   
  \item 
  \(\#\pi_0(\Sigma) = \#\pi_0(\overline{S}_c) \)
  \item 
  \(\#\pi_0(\partial \Sigma)\geq 12(1+\hbar^{-1})C\)
  \item
  each connected component of \(\overline{S}_c\setminus (\Sigma\setminus
  \partial \Sigma)\)  is non-compact.
  \end{enumerate} 
\end{restatementlemma}
%
\begin{proof}
\setcounter{section}{\value{CurrentSection}}
\setcounter{lemma}{\value{CurrentLemma}}
As in the proof of Lemma \ref{LEM_bounded_connected_components}, we will
  attempt to simplify notation a bit by dropping the subscripts \(c\), and
  writing, for example, \(\bar{\mathbf{w}}\) and \(\overline{S}\) instead of
  \(\bar{\mathbf{w}}_c\) and \(\overline{S}_c\).

Our first step is to put precisely one special point, \(\zeta_k\), on each
  connected component of \(\overline{S}\).
We denote the set of such points \(Z=\{\zeta_1, \ldots, \zeta_n\}\), and
  note that this set is finite as a consequence of Lemma
  \ref{LEM_bounded_connected_components}.
By assumption we have \({\rm Punct}(\bar{\mathbf{w}}) = \infty\), so it
  follows that there exists \(x_0>0\) with the property that for each
  \(x>x_0\), the number of non-compact connected  components of
  \(\overline{S}\setminus(a\circ \bar{w})^{-1}((-\infty, x))\) is greater
  than or equal to \(12(1+\hbar^{-1})C\).
To make use of this, we first define \(\mathcal{I}_{reg}\) to be the
  intersection of the sets \(\mathbb{R}\setminus a\circ
  \bar{w}(\overline{D})\) and the set of regular values of the function
  \(a\circ \bar{w}:\overline{S}\to \mathbb{R}\).
We then choose \(x_0\) sufficiently large so that for each \(x \in
  \mathcal{I}_{reg}\) with \(x> x_0\), and for \(\Sigma^x:= (a\circ
  \bar{w})^{-1}((-\infty,x])  \) we have
  \begin{enumerate}                                                       
    \item 
    \(Z\cup \overline{D}\subset \Sigma^x\),
    \item 
    \({\rm Genus}(\Sigma^x) = {\rm Genus}(\overline{S})\),
    \item 
    each compact connected component of \(\overline{S}\) is contained
    in \(\Sigma^x\),
    \item
    the number of non-compact connected components of
    \(\overline{S}\setminus (\Sigma^x\setminus \partial \Sigma^x)\) is
    greater than \(12(1+\hbar^{-1})C\). 
    \end{enumerate}
We note that the existence of such a \(x_0\) relies both on the validity
  of Lemma \ref{LEM_bounded_connected_components} and the assumption that
  \({\rm Punct}(\bar{\mathbf{w}})=\infty\). 
We henceforth assume \(x\in \mathcal{I}_{reg}\) with \(x>x_0\) has been
  fixed.
We also note an important property, namely that as a consequence of the
  fact that \({\rm Genus}(\Sigma^x) = {\rm Genus}(\overline{S})\), it
  follows that any embedded loop removed from \(\overline{S}\setminus
  \Sigma^x\) disconnects the surface \(\overline{S}\); this follows from
  genus super-additivity and an straightforward Euler characteristic
  argument.
 \begin{figure}[h]
\includegraphics[scale=0.5]{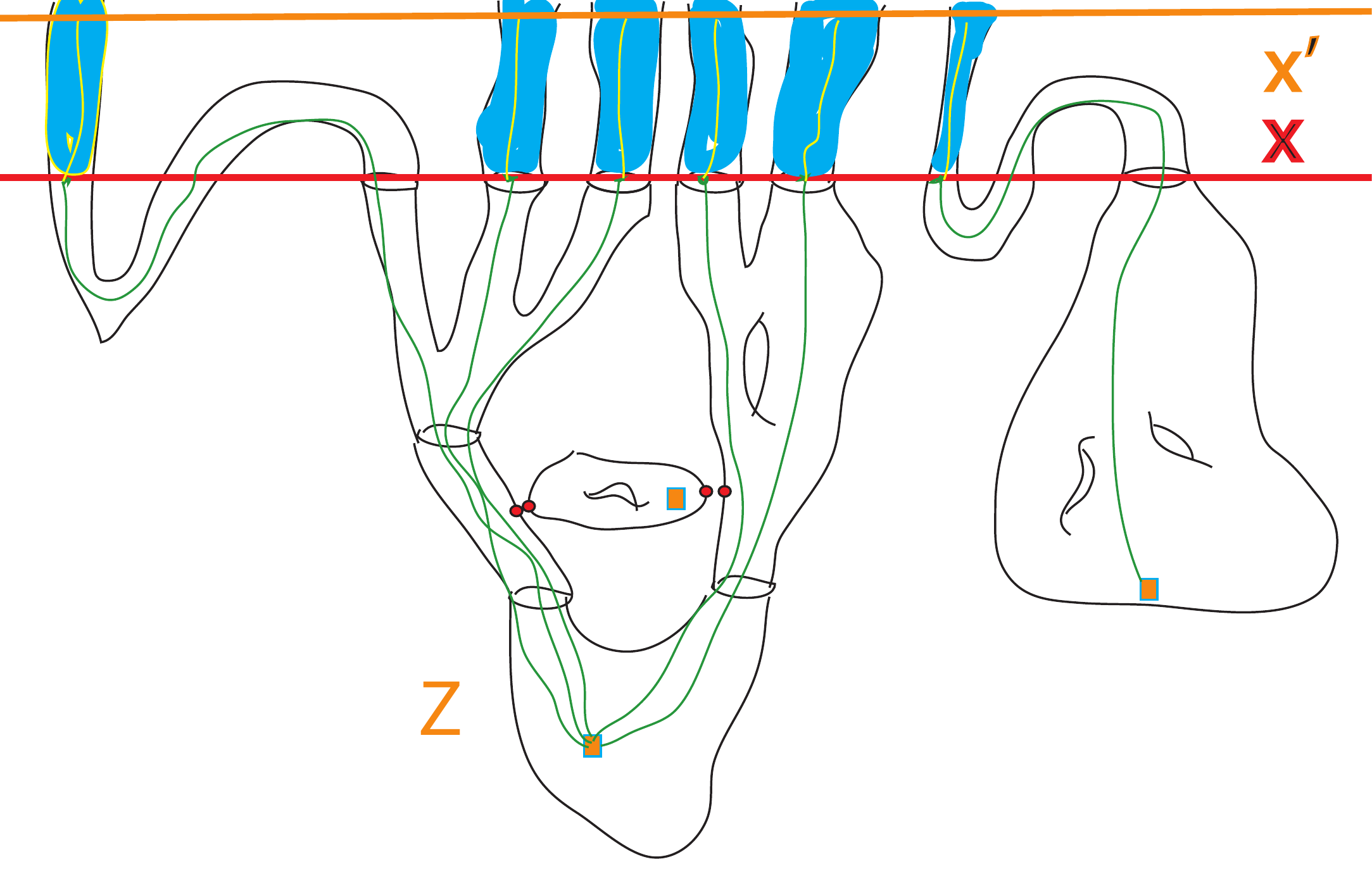}
\caption{The figure illustrates the construction. It shows $x$, $x'$ and
  the extended curves $\gamma_k$.
  The actual situation can be in generally much wilder. In our case we
  have above $x'$ only non-compact components
  (not shown), i.e. $E_1'$,...,$E'_m$. A later figure will show additional
  possible features.}
\label{FIG40_jwf_alternate}
\end{figure}

Next, we enumerate the set of non-compact connected components of
  \(\overline{S}\setminus (\Sigma^x\setminus \partial \Sigma^x)\) as
  \(E_1, \ldots, E_m\) with
  \begin{equation*}                                                       
    m \geq 12(1+\hbar^{-1})C.
    \end{equation*}
Also, for each \(k\in \{1,\ldots, m\}\) we choose a continuous path
  \(\gamma_k:[0,1] \to \overline{S},\) each with the property that
  \(\gamma_k(0)\in Z \) and \(\gamma_k(1)\in \partial E_k\).
At this point, we fix a \(x' \in \mathcal{I}_{reg}\) with \(x'>x\) with
  the additional property that
  \begin{equation*}                                                       
    \bigcup_{k=1}^m \gamma_k\big([0, 1]\big) \subset \Sigma^{x'}.
    \end{equation*}
     \begin{figure}[h]
\includegraphics[scale=0.25]{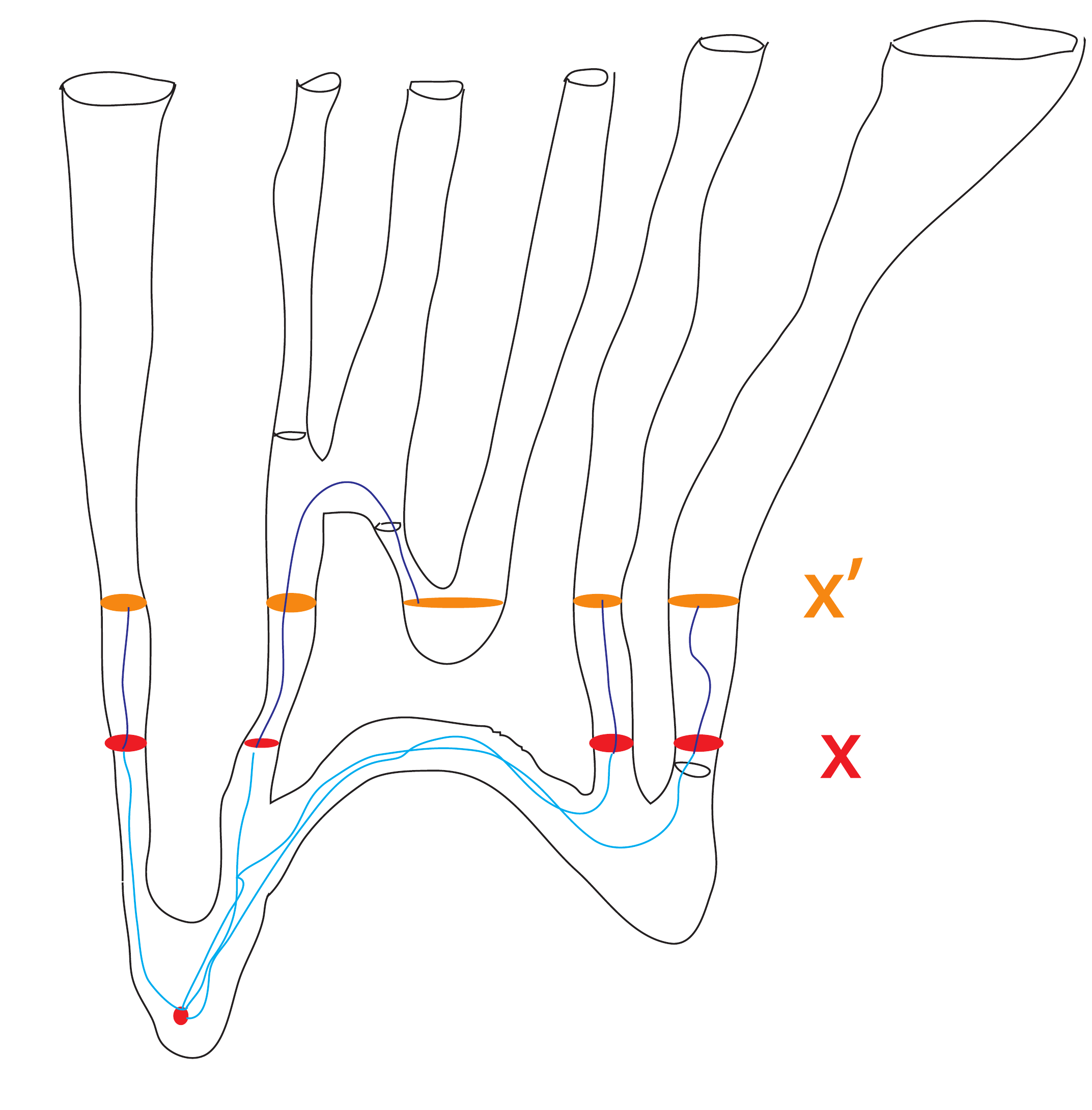}\ \ \ \includegraphics[scale=0.25]{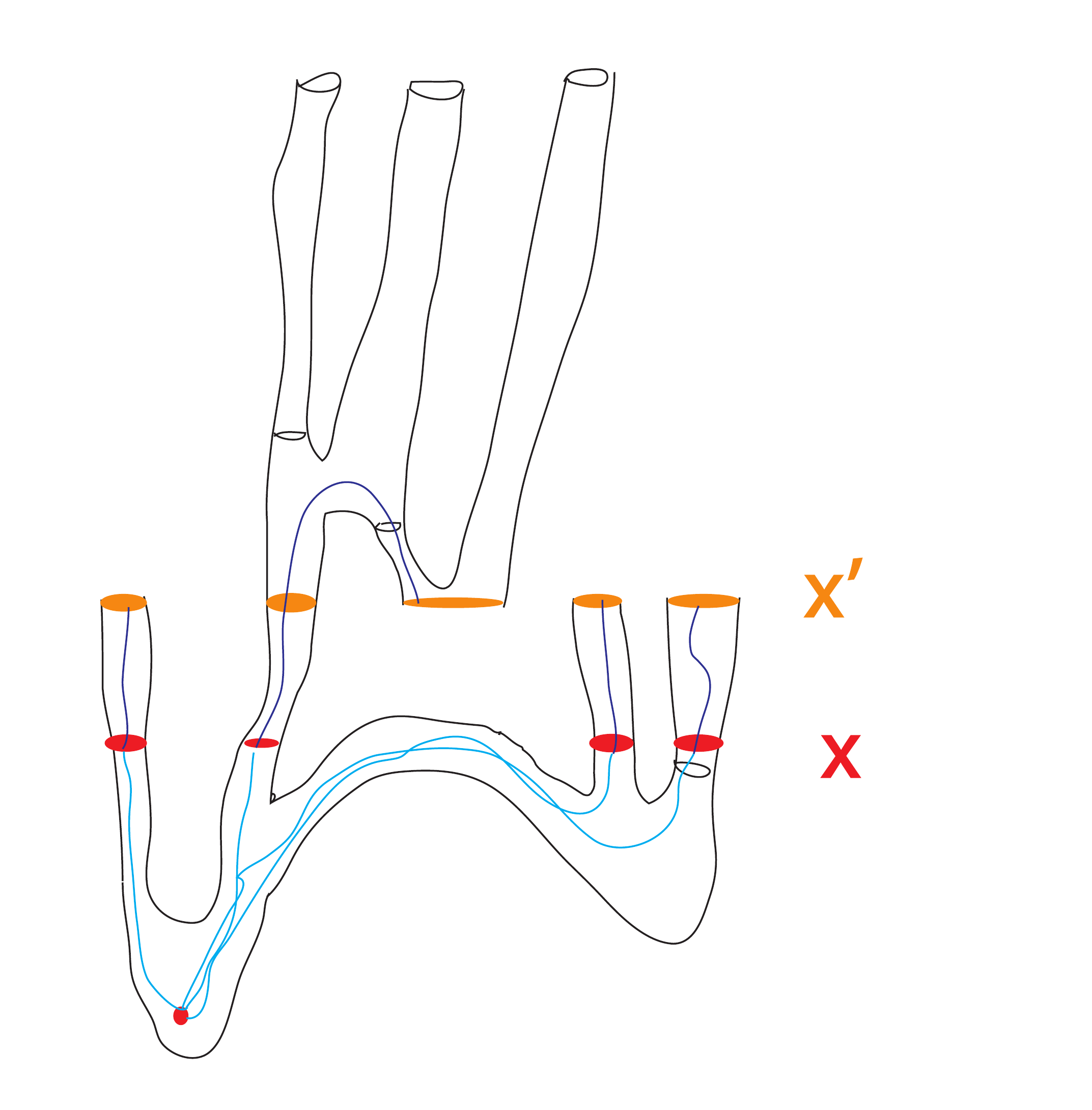}\\
\includegraphics[scale=0.25]{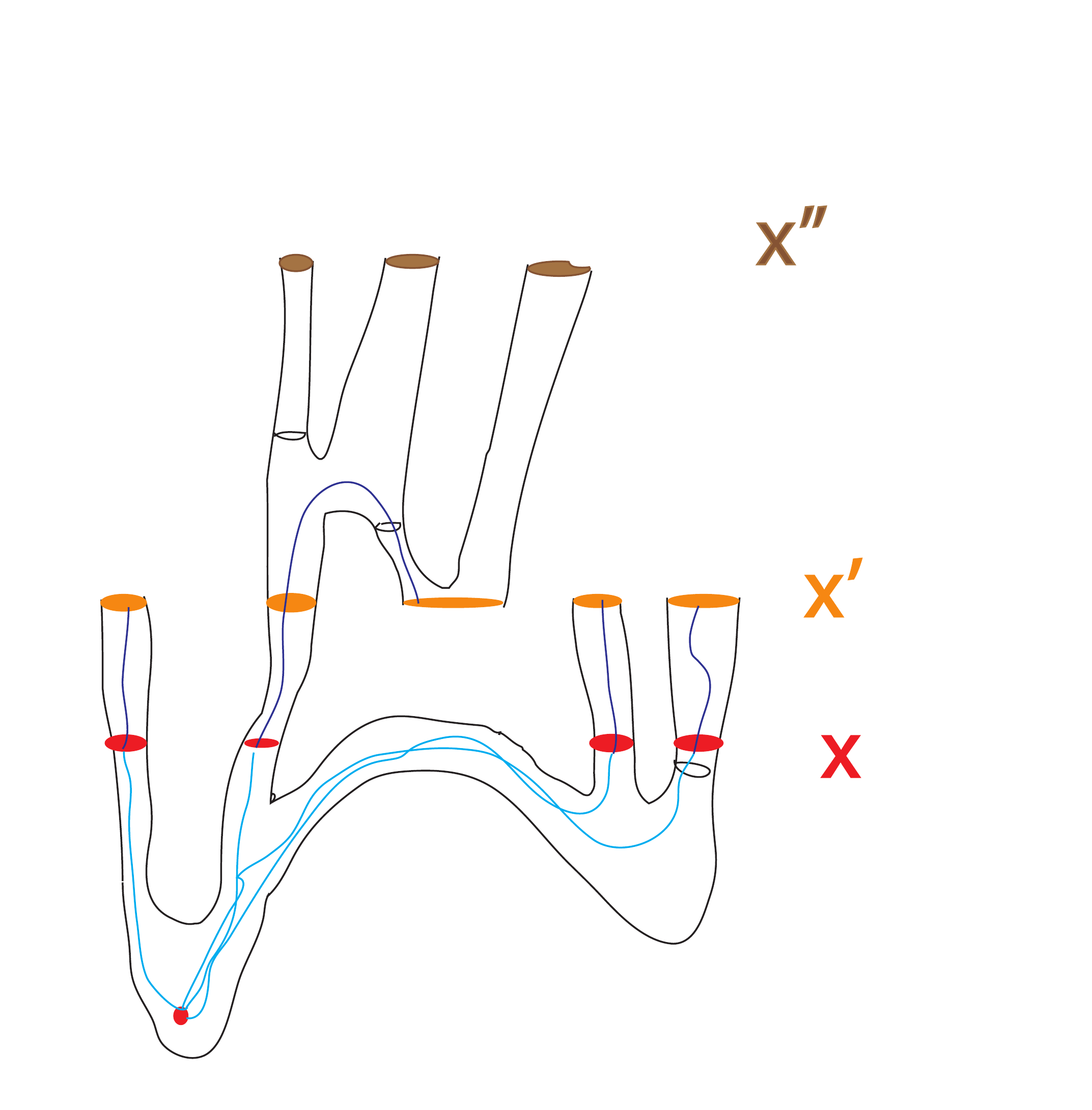}
\caption{Left: In this case $Z$ consists of one point and we have four
non-compact components $E_1,..,E_4$. We also have four $E_1',.,E'_4$.
Right: The set $\widetilde{\Sigma}$, which we note is not compact. This
set has already a lot of desirable properties. Below: The desired set
$\Sigma$ is obtained by trimming it further.} 
\label{FIG40}
\end{figure}
    
Next, we enumerate the non-compact ends of \(\overline{S}\setminus
  (\Sigma^{x'}\setminus \partial \Sigma^{x'}\big) \) via \(E_1', E_2',
  \ldots, E_{m'}'\).
We also extend each \(\gamma_k:[0,1]\to \overline{S}\) to continuous
  \(\gamma_k:[0, 2]\to \overline{S} \) so that
  \begin{equation}\label{EQ_gamma_properties_1}                           
    \gamma_k(1,2) \subset \overline{S}\setminus \Sigma^x
    \end{equation}
  and for each \(k\in \{1,\ldots, m\}\) we have
  \begin{equation}\label{EQ_gamma_properties_2}                           
    \gamma_k(t) \in \cup_{i=1}^{m'} E_i'\qquad\qquad\text{if and only
    if}\qquad\qquad t=2.
    \end{equation}
We then obtain a new surface, denoted by \(\widetilde{\Sigma}\), by
  cutting \(\overline{S}\) at the circles associated to the points
  \(\gamma_1(2), \ldots, \gamma_m(2)\), and defining
  \(\widetilde{\Sigma}\) to be the union of the connected components of
  the cut surface which have non-empty intersection with \(Z\).

We pause for a moment to consider the properties of the surface
  \(\widetilde{\Sigma}\), since it is close to the surface we seek.
To that end, we first observe that \(\partial \widetilde{\Sigma}\subset
  \cup_{k=1}^{m'} \partial E_k'\); this follows from equations
  (\ref{EQ_gamma_properties_1}) and (\ref{EQ_gamma_properties_2}).

Second, we claim that \(\#\pi_0(\widetilde{\Sigma}) = n =
  \#\pi_0(\overline{S})\).
To see this, first note that by definition \(Z\subset
  \widetilde{\Sigma}\subset \overline{S}\), and each element of \(Z\) lies
  on a different connected component of \(\overline{S}\), and hence
  \(\#\pi_0(\widetilde{\Sigma}) \geq \#\pi_0(\overline{S}) = n\); and
  because each connected component of \(\widetilde{\Sigma}\) must contain
  a point in \(Z\), the opposite inequality must hold as well.

Third, we claim that \(\#\pi_0(\partial \widetilde{\Sigma}) = m \geq
  12(1+\hbar^{-1})C\).
To establish this, it is  important to observe that
  \(\cup_{k=1}^m\gamma_k([0, 2])\subset \widetilde{\Sigma}\).
To see this, recall that equation (\ref{EQ_gamma_properties_2}) guarantees
  that for each \(k\in \{1, \ldots, m\}\), we have \(\gamma_k(t) \in
  \cup_{i=1}^{m'}E_i\) if and only if \(t=2\); furthermore, since
  \(\partial\widetilde{\Sigma}\subset \cup_{k=1}^{m'} \partial E_k' \), it
  follows that each \(\gamma_k([0, 2])\) is contained in a connected
  component of the surface obtained by cutting \(\overline{S}\) at the
  circles associated to the points \(\gamma_1(2), \ldots, \gamma_m(2)\).
However, because \(\gamma_k(0)\in Z\) for \(k\in \{1, \ldots, m\}\), and
  \(Z\subset \widetilde{\Sigma}\), it follows that indeed, \(\cup_{k=1}^m
  \gamma_k([0, 2]) \subset \widetilde{\Sigma}\).
We can now prove that \(\#\pi_0(\partial \widetilde{\Sigma})= m\).
To see this, recall that by construction, for each \(k\in \{1, \ldots,
  m\}\), the point \(\gamma_k(1)\) is an element of a connected component of
  \(\overline{S}\setminus (\Sigma^x\setminus \partial \Sigma^x)\), and
  moreover no two such points \(\gamma_k(1)\) and \(\gamma_{k'}(1)\) are
  contained in the same connected component of \(\overline{S}\setminus
  (\Sigma^x\setminus \partial \Sigma^x)\).
Furthermore, by equation (\ref{EQ_gamma_properties_1}) we have
  \(\gamma_k((1,2)) \subset \overline{S}\setminus \Sigma^x\), and since
  \(\gamma_k(2)\in \partial \widetilde{\Sigma}\) for each \(k\in \{1,
  \ldots, m\}\), it follows that \(\#\pi_0(\partial \widetilde{\Sigma})\geq
  m\).
The equality \(\#\pi_0(\partial \widetilde{\Sigma})=m\)  follows from the
  fact that each connected component of \(\partial \widetilde{\Sigma}\)
  contains one point of the form \(\gamma_k(2)\).

Fourth, and finally, we claim that each connected component of
  \(\overline{S}\setminus(\widetilde{\Sigma}\setminus \partial
  \widetilde{\Sigma})\) is non-compact.
To establish this, let us define \(\widehat{\Sigma}\) to be the surface
  obtained by cutting \(\overline{S}\) at the circles associated to the
  points \(\gamma_1(2), \ldots, \gamma_m(2)\); recall that
  \(\widetilde{\Sigma}\) is then defined to be the union of the connected
  components of \(\widehat{\Sigma}\) which have non-empty intersection with
  \(Z\).
Consequently, suppose \(\Sigma\) is a compact connected component of
  \(\widehat{\Sigma}\) .
There are two cases to consider.
In the first case, \(\partial \Sigma = \emptyset\), in which case it
  follows that \(\Sigma\cap Z \neq \emptyset\) by definition of \(Z\).
In the second case, \(\partial \Sigma \neq \emptyset\), it follows that
  \(\partial \Sigma\) has non-trivial intersection with the set
  \(\{\gamma_1(2), \gamma_2(2), \ldots, \gamma_m(2)\}\).
It then follows from the fact that \(\partial\widetilde{\Sigma}\subset
  \cup_{k=1}^{m'} \partial E_k'\), that either \(\Sigma\subset
  \widetilde{\Sigma}\) or else \(\Sigma\subset \cup_{k=1}^{m'} E_k'\).
However, since each connected component of \(\cup_{k=1}^{m'} E_k'\) is
  non-compact and \(\Sigma\) is compact, it follows that we must have
  \(\Sigma\subset \widetilde{\Sigma}\).
Thus whenever \(\Sigma\) is a compact connected component of
  \(\widehat{\Sigma}\), we have \(\Sigma\subset \widetilde{\Sigma}.\)
Summarizing, we have constructed a surface \(\widetilde{\Sigma}\subset
  \overline{S}\) with the properties
  \begin{enumerate}                                                       
    \item
    \(\partial \widetilde{\Sigma} \subset \cup_{k=1}^{m'} \partial E_k'\)
    \item 
    \(\#\pi_0(\widetilde{\Sigma}) = n = \#\pi_0(\overline{S})\)
    \item 
    \(\#\pi_0(\partial \widetilde{\Sigma}) = m\geq 12(1+\hbar^{-1})C\).
    \item 
    each connected component of
    \(\overline{S}\setminus(\widetilde{\Sigma}\setminus \partial
    \widetilde{\Sigma})\) is non-compact.
    \end{enumerate}

Observe that we would have found the desired Riemann surface if only
  \(\widetilde{\Sigma}\) had been compact.
Since \(\widetilde{\Sigma}\) need not be compact, we must trim it further
  to obtain the desired surface.
To that end, we fix, \(x''\in \mathcal{I}_{reg}\) with \(x''> x'\). 
We then define \(E''\) to be the union of the interiors of the non-compact
  connected components of \(\widetilde{\Sigma} \cap (a\circ w)^{-1}([x'',
  \infty))\), and we define \(\Sigma\) to be the set of all points in
  \(p\in \widetilde{\Sigma}\) for which there exists a continuous path in
  \(\widetilde{\Sigma}\setminus E''\) from \(x\) to \(Z\).

We now establish the required properties. 
First we note that \(\widetilde{\Sigma}\setminus E''\) is compact, and
  \(\Sigma\subset \widetilde{\Sigma}\setminus E''\) is closed, so that
  \(\Sigma\) is indeed compact.
Again, every connected component of \(\Sigma\) is path-connected to \(Z\),
  and hence \(\#\pi_0(\Sigma) =n  = \#\pi_0(\overline{S}) \).
Also by construction \(\partial \widetilde{\Sigma}\subset \partial
  \Sigma\), and hence 
  \begin{equation*}                                                       
    \#\pi_0(\partial \Sigma)\geq \#\pi_0(\partial\widetilde{ \Sigma}) = m
    \geq 12(1+\hbar^{-1})C.
    \end{equation*}
Finally, we claim that each connected component of \(\overline{S}\setminus
  (\Sigma\setminus \partial \Sigma)\)  is non-compact. 
To see this, we first note that each connected component of \(\overline{S}
\setminus (\widetilde{\Sigma}\setminus\partial \widetilde{\Sigma})\) is a
  connected component of \(\overline{S}\setminus (\Sigma\setminus \partial
  \Sigma)\), and we have already established that each of these is
  non-compact.
Thus it is sufficient to show that the connected components of
  \(\widetilde{\Sigma}\setminus (\Sigma\setminus \partial \Sigma)\) are
  non-compact.
Observe that any connected component of \(\widetilde{\Sigma}\setminus
  (\Sigma\setminus \partial \Sigma)\) having nontrivial intersection with
  \(E''\) must be non-compact.
However, by definition of \(\Sigma\), any point \(p\in
  \widetilde{\Sigma}\setminus (\Sigma\setminus \partial \Sigma)\) has the
  property that every path connecting \(p\) to \(Z\) will intersect
  \(E''\).
In other words, the connected component of \(\widetilde{\Sigma}\setminus
  (\Sigma\setminus \partial \Sigma)\) which contains such a \(p\) must
  also contain a connected component of \(E''\), and hence must be
  non-compact.
This establishes all the required properties of \(\Sigma\), and hence
  completes the proof of Lemma \ref{LEM_impossible_submanifold}.
\end{proof}

We now observe that we have completed the proof of Proposition
  \ref{PROP_feral_limit_curves}, including all dependencies.

\subsubsection{Proof of Lemma
  \ref{LEM_bounded_transverse_intersections}}
  \label{SEC_proof_of_bounded_intersections}

\setcounter{CurrentSection}{\value{section}}
\setcounter{CurrentLemma}{\value{lemma}}
\setcounter{section}{\value{CounterSectionBoundedIntersections1}}
\setcounter{lemma}{\value{CounterLemmaBoundedIntersections1}}
\begin{restatementlemma}[bounded transverse intersections]\hfill\\
Consider non-negative numbers \(c, c'\geq 0\) with \(c'> c\), and let
  \(k\mapsto \ell_k\in \mathbb{N}\) be a strictly increasing sequence for
  which \(\mathbf{w}_{\ell_k,c}\to \bar{\mathbf{w}}_c\) and
  \(\mathbf{w}_{\ell_k,c'}\to \bar{\mathbf{w}}_c'\) in an exhaustive sense.
Then the subset \(\mathcal{P}\subset \mathbb{R}\times M\) of transversal
  intersection points of the two curves, which is defined by
  \begin{align*}                                                          
    \mathcal{P}:=\big\{p\in \mathbb{R}\times M &: \text{ there exists }
    (\zeta, \zeta')\in \overline{S}_c\times \overline{S}_{c'} \; \text{such
    that}\;
    \\
    &\; \; \bar{w}_c(\zeta) = p = \bar{w}_{c'}(\zeta')\;  \text{ and }\;
    T\bar{w}_c(\zeta)\pitchfork T\bar{w}_{c'}(\zeta') \big\},
    \end{align*}
  satisfies 
  \begin{equation*}                                                       
    \#\mathcal{P} \leq C.
    \end{equation*}
\end{restatementlemma}
%
\begin{proof}
\setcounter{section}{\value{CurrentSection}}
\setcounter{lemma}{\value{CurrentLemma}}
We will proceed via a proof by contradiction, and assume that
  \(\#\mathcal{P} > C\).
Consequently there exist distinct \(p_1, \ldots, p_n\in \mathbb{R}\times
  M\) with \(n> C\), and there exist \(\zeta_1,\ldots, \zeta_n\in
  \overline{S}\) and \(\zeta_1', \ldots, \zeta_n'\in \overline{S}'\) with
  the property that 
\begin{equation*}                                                         
  \bar{w}_c(\zeta_k) = p_k = \bar{w}_{c'}(\zeta_k') \qquad\text{and}\qquad
  T\bar{w}_c(\zeta_k)\pitchfork T\bar{w}_{c'}(\zeta_k')
  \end{equation*}
  for each \(k\in \{1, \ldots, n\}\).
Because the \(p_1, \ldots, p_n\) are distinct and
  because \(\bar{w}_c\) and \(\bar{w}_{c'}\) are respectively immersions at
  the points \(\zeta_k\) and \(\zeta_k'\) for each \(k\in \{1, \ldots, n\}\)
  it follows that we may find closed disks \(\Delta_1,\ldots,
  \Delta_n\subset \overline{S}_c\) and \(\Delta_1', \ldots,
  \Delta_n'\subset \overline{S}_{c'}\) which are pairwise disjoint, and
  satisfy \(\zeta_k\in \Delta_k\setminus \partial \Delta_k\) and
  \(\zeta_k'\in \Delta_k'\setminus \partial \Delta_k'\), and for which the
  maps
  \begin{equation*}                                                       
    \bar{w}_c: \bigcup_{k=1}^n \Delta_k \to \mathbb{R}\times M
    \qquad\text{and}\qquad \bar{w}_{c'}: \bigcup_{k=1}^n \Delta_k' \to
    \mathbb{R}\times M
    \end{equation*}
  are embeddings.

We then note as a consequence of the definition of exhaustive Gromov
  compactness, there exist, for all sufficiently large \(k\in \mathbb{N}\),
  embeddings
  \begin{equation*}                                                       
    \phi_{\ell_k}: \bigcup_{\nu=1}^n\Delta_\nu \to S_{\ell_k,c}
    \qquad\text{and}\qquad \phi_{\ell_k}': \bigcup_{\nu=1}^n\Delta_\nu'
    \to S_{\ell_k,c'}
    \end{equation*}
  with the property that the maps
  \begin{equation*}                                                       
    w_{\ell_k,c}\circ \phi_{\ell_k}: \bigcup_{\nu=1}^n\Delta_\nu \to
    \mathbb{R}\times M \qquad\text{and}\qquad w_{\ell_k,
    c'}\circ\phi_{\ell_k}': \bigcup_{\nu=1}^n\Delta_\nu' \to
    \mathbb{R}\times M
    \end{equation*}
  respectively converge in \(\mathcal{C}^\infty\) to 
  \begin{equation*}                                                       
    \bar{w}_{c}: \bigcup_{\nu=1}^n\Delta_\nu \to \mathbb{R}\times M
    \qquad\text{and}\qquad \bar{w}_{c'}: \bigcup_{\nu=1}^n\Delta_\nu' \to
    \mathbb{R}\times M.
    \end{equation*}
Then by Lemma \ref{LEM_stability_of_transversal_intersections}, it follows
  that for all sufficiently large \(k\in \mathbb{N}\) there exist distinct
  \begin{equation*}                                                       
    z_1,\ldots, z_n \in \phi_{\ell_k}\Big(\bigcup_{\nu=1}^n\Delta_\nu\Big)
    \subset S_{\ell_k,c} \qquad\text{and}\qquad z_1',\ldots, z_n' \in
    \phi_{\ell_k}'\Big(\bigcup_{\nu=1}^n\Delta_\nu'\Big) \subset
    S_{\ell_k,c'}
    \end{equation*}
  for which \(w_{\ell_k, c}(z_\nu) = w_{\ell_k,c'}(z_\nu')\) for \(\nu
  \in \{1,\ldots, n\}\).
Recall equation (\ref{EQ_def_w}) which guarantees
  \begin{equation*}                                                       
    w_{k,c} = {\rm Sh}_{a_k}\circ u_k^{a_k+c}
    \end{equation*}
  where \({\rm Sh}_x\) is the shift map \({\rm Sh}_x(a,p)= (a-x,p)\), and
  the \(u_k^{a_k+c}\) are of the curves \(u_k^b\) specified in the
  hypotheses of Theorem \ref{THM_existence}.
Consequently, 
  \begin{equation*}                                                       
   {\rm Sh}_{a_{\ell_k}}\circ u_{\ell_k}^{a_{\ell_k}+c}(z_\nu) =
   w_{\ell_k, c}(z_\nu) = w_{\ell_k,c'}(z_\nu') = {\rm
   Sh}_{a_{\ell_k}}\circ u_{\ell_k}^{a_{\ell_k}+c'}(z_\nu')
   \end{equation*}
  and hence 
  \begin{equation*}                                                       
    u_{\ell_k}^{a_{\ell_k}+c}(z_\nu) =
    u_{\ell_k}^{a_{\ell_k}+c'}(z_\nu')
    \end{equation*}
  for \(\nu\in \{1, \ldots, n\}\) with \(n >C\).
Recall equation (\ref{EQ_Skc}) which guarantees that 
  \(S_{\ell_k, c} = S_{\ell_k}^{a_{\ell_k} +c}\) and \(S_{\ell_k, c'} =
  S_{\ell_k}^{a_{\ell_k} +c'}\), and hence we conclude that
  \begin{equation*}                                                       
  \#\big(u_{\ell_k}^{a_{\ell_k}+c}(S_{\ell_k}^{a_{\ell_k} +c})\cap
  u_{\ell_k}^{a_{\ell_k}+c'}(S_{\ell_k}^{a_{\ell_k}+ c'})\big)  > C,
    \end{equation*}
  for all sufficiently large \(k\in \mathbb{N}\).
However this contradicts the hypothesis of Theorem \ref{THM_existence}
  which states
  \begin{equation*}                                                       
    \# \big(u_k^b(S_k^b) \cap u_k^{b'}(S_k^{b'})\big) \leq C
    \end{equation*}
  for all \(b, b'\geq 0\) and \(k\in \mathbb{N}\).
This is the contradiction we have sought, and hence the proof of Lemma
  \ref{LEM_bounded_transverse_intersections}.
\end{proof}

\appendix

\addtocontents{toc}{\protect\setcounter{tocdepth}{1}}
\section{Minor Miscellanea}\label{SEC_minor_miscellanea}
\subsection{Riemannian Recollections}\label{SEC_riemannian}

Let \((M, g)\) be a Riemannian manifold of dimension \(n\).  
Recall that the metric \(g\) uniquely determines a torsion-free metric
  connection, called the Levi-Civta connection.
We denote the associated covariant derivative by \(\nabla\). 
That is, \(\nabla\) satisfies
  \begin{equation*}
    \nabla_X Y - \nabla_Y X = [X, Y] \quad\text{and}\quad \nabla \langle
    X, Y\rangle_g = \langle\nabla X, Y\rangle_g +\langle X, \nabla
    Y\rangle_g.
    \end{equation*} 
Let \(p\in M\), and let \(\mathbf{x}=(x^1, \ldots, x^n)\) be local
  coordinates near \(p\) so that \(\mathbf{x}(p)=0\in \mathbb{R}^n\).
We express \(g\) in local coordinates by the following. 
  \begin{equation*}
    g = g_{ij}\, dx^i\otimes dx^j
    \end{equation*}
Here, and throughout, we employ Einstein's  notation for summing over
  repeated indices.
We also uniquely define \(n^2\) functions \(g^{ij}\) by the equations
  \begin{equation*}
    g^{i\ell}g_{\ell j} = \delta_j^i
    \end{equation*} 
  with \(\delta_j^i\) the Kronecker delta.  
In this case we may express \(\nabla\) in local coordinates as 
  \begin{equation*}
    \nabla_{X^i\partial_{x^i}} (Y^j \partial_{x^j}) = X^i
    dY^j(\partial_{x^i})\partial_{x^j} + X^i Y^j \Gamma_{ij}^k
    \partial_{x^k},
    \end{equation*}
  where \(\Gamma_{ij}^k\) are the Christoffel symbols, which are given by
  \begin{equation}\label{EQ_christoffel}
    \Gamma_{ij}^k = {\textstyle \frac{1}{2}} g^{k\ell}\big(g_{i\ell, j} +
    g_{j\ell, i} - g_{i j, \ell}\big)
    \end{equation}
  where \(g_{ij, k}= \frac{\partial}{\partial x^k}g_{ij}\).  
It is worth noting that
  \begin{align*}
    0 &= \nabla_{\partial_{x^k}} (\delta_j^i)
    \\
    &= \nabla_{\partial_{x^k}} (dx^i(\partial_{x^j}))
    \\
    &= (\nabla_{\partial_{x^k}} dx^i)(\partial_{x^j}) + dx^i
    (\nabla_{\partial_{x^k}}\partial_{x^j})
    \\
    &= (\nabla_{\partial_{x^k}} dx^i)(\partial_{x^j}) + dx^i
    (\Gamma_{kj}^\ell \partial_{x^\ell})
    \\
    &= (\nabla_{\partial_{x^k}} dx^i)(\partial_{x^j}) + \Gamma_{kj}^i 
    \end{align*}
  from which we conclude that 
  \begin{equation*}
    \nabla_{\partial_{x^k}}dx^i=-\Gamma_{ki}^\ell dx^\ell.
    \end{equation*}

Recall that given a point \(p\in M\) and a (sufficiently small) vector
  \(Z\in T_p M\), there exists unique geodesic emanating from \(p\) with
  initial velocity \(Z\).
That is to say, there exists a unique solution \(\gamma:[0, 1]\to M\) to
  the initial value problem
  \begin{equation}\label{EQ_geodesic}
    \nabla_{\dot{\gamma}(t)} \dot{\gamma}(t) =
    0\qquad\text{and}\qquad\gamma(0)=p,\quad\gamma'(0)=Z.
    \end{equation}
The exponential map associated to \(g\), denoted \(\exp_p^g:T_p M \to M\),
  is defined by \(\exp_p^g(Z)=\gamma(1)\) where \(\gamma\) solves the
  differential equation (\ref{EQ_geodesic}).
Furthermore, given an orthonormal basis \((Z_1, \ldots, Z_n)\) of
  \(T_p M\), we define normal geodesic polar coordinates \(\mathbf{x}=(x^1,
  \ldots, x^n)\) near \(p\) by the following
  \begin{equation*}
    x^i(q)=\big\langle (\exp_p^g)^{-1}(q), Z_i \big\rangle_g.
    \end{equation*} 
Recall that in these normal geodesic coordinates, we have the following
  \begin{equation}\label{EQ_christoffel_prop_1}
    g_{ij}(p) = \delta_{ij}\qquad\text{and}\qquad \Gamma_{ij}^k(p) = 0.
    \end{equation} 
Furthermore, we also have
  \begin{equation}\label{EQ_christoffel_prop_2}
    \frac{\partial}{\partial x^k} g_{ij}(p) = 0.
    \end{equation}
To see that equation (\ref{EQ_christoffel_prop_2}) holds, we simply compute
  \begin{align*}
    \frac{\partial}{\partial x^k} g_{ij} &= \nabla_{\partial_{x^k}}
    \langle \partial_{x^i}, \partial_{x^j}\rangle_g
    \\
    & = \langle \nabla_{\partial_{x^k}} \partial_{x^i},
    \partial_{x^j}\rangle_g +\langle \partial_{x^i},
    \nabla_{\partial_{x^k}} \partial_{x^j}\rangle_g
    \\
    & = \langle \Gamma_{ki}^\ell\partial_{x^\ell}, \partial_{x^j}\rangle_g
    +\langle \partial_{x^i}, \Gamma_{kj}^\ell\partial_{x^\ell}\rangle_g
    \end{align*} 
  evaluating at \(p\) and employing  equation
  (\ref{EQ_christoffel_prop_1}) then establishes equation
  (\ref{EQ_christoffel_prop_2}).

We now consider \(\mathbb{R}\times M\) with \((M, g)\) as above, a
  coordinate \(a\) on \(\mathbb{R}\), and the metric \(\bar{g}:=da\otimes da
  + g\).
Fix a point \((p_0, p_1)\in \mathbb{R}\times M\) and fix \((Z_1, Z_2,
  \ldots, Z_n)\) an orthonormal basis of \(T_{p_1} M\), and let \((x^1,
  \ldots, x^n)\) denote the associated normal geodesic coordinates defined
  near \(p_1\).
We extend these to coordinates \((x^0, x^1, \ldots, x^n)\) by taking
  \(x^0=a\). 
We now claim the following.

\begin{lemma}[properties of $g$ and $\Gamma$]
  \label{LEM_g_and_Gamma_properties}
  \hfill\\
In the coordinates \((x^0, \ldots, x^n)\) established above, the following
  hold.
\begin{equation}\label{EQ_symp_normal_coords}
  \bar{g}_{ij}(p_1, p_2) = \delta_{ij}\qquad\text{and}\qquad
  \overline{\Gamma}_{ij}^k(p_0, p_1)=0,
  \end{equation}
  where \(\overline{\Gamma}_{ij}^k\) are the Christoffel symbols
  associated to \(\bar{g}\) in the coordinates \((x^0, \ldots, x^n)\).
Moreover, in these local coordinates, we have 
  \begin{equation}\label{EQ_translation}
  \frac{\partial}{\partial x^0} \bar{g}_{ij} = 0 =
  \frac{\partial}{\partial x^0} \overline{\Gamma}_{ij}^k.
  \end{equation}
\end{lemma}
%
\begin{proof}
For ease of notation, we write \(\bar{p}=(p^0, p^1)\).  
Begin by observing that whenever \(0\notin \{i, j, k\}\) we have
  \(\bar{g}_{ij}(\bar{p})=\delta_{ij}\) and
  \(\overline{\Gamma}_{ij}^k(\bar{p})=0\).
This follows from the fact that the \((x^1, \ldots, x^n)\) for normal
  geodesic coordinates associated to \(g\).
Next observe that by definition of \(\bar{g}\) we have \(\bar{g}_{00}=1\),
  and because \(\bar{g}\) is a product metric, it follows that \(\bar{g}_{0
  j}=0=\bar{g}_{j 0}\) for \(j\in \{1, \ldots, n\}\).  
These results establish the first equality in equation
  (\ref{EQ_symp_normal_coords}).

We next aim to prove that \(\overline{\Gamma}_{ij}^k(\bar{p})=0\) when
  \(0\in \{i, j, k\}\).
To that end, first observe that \(\bar{g}^{i0} =
  \bar{g}^{0i}=\delta_{i0}\), \(\bar{g}^{ij} = g^{ij}\) whenever \(0\notin
  \{i, j\}\), and all \(\bar{g}_{ij}\) are independent of \(x^0\).
This latter fact together with the formula for the Christoffel symbols
  given in equation (\ref{EQ_christoffel}) then guarantee equation
  (\ref{EQ_translation}).
Furthermore, the term \((g_{j\ell, i} + g_{i\ell, j} - g_{ij, \ell})\) in
  equation (\ref{EQ_christoffel}) vanishes whenever \(0\in \{i, j, \ell\}\).
If \(0\notin \{i, j, k\}\), then in particular \(k\neq 0\) so that
  \begin{equation*}
    \overline{\Gamma}_{ij}^k = \sum_{\ell=0}^n{\textstyle \frac{1}{2}}
    g^{k\ell}\big(g_{i\ell, j} + g_{j\ell, i} - g_{i j, \ell}\big) =
    \sum_{\ell=1}^n{\textstyle \frac{1}{2}} g^{k\ell}\big(g_{i\ell, j} +
    g_{j\ell, i} - g_{i j, \ell}\big), 
    \end{equation*}
  and hence for \(0\notin \{i, j, k\}\) we have
  \(\overline{\Gamma}_{ij}^k(p_0, p_1)= \Gamma_{ij}^k(p_1) = 0\).
We conclude that for arbitrary \((i, j, k)\) we have
  \(\overline{\Gamma}_{ij}^k(p_0, p_1)=0\). 
This completes the proof of Lemma \ref{LEM_g_and_Gamma_properties}.
\end{proof} 

\begin{corollary}[properties of $g$ and $\Gamma$]
  \label{COR_g_and_Gamma_properties}
  \hfill\\
In the coordinates as above, we have
  \begin{equation*}
    \overline{\nabla}_{\partial_{x^i}} \partial_{x^j} \big|_{(p_0, p_1)} =
    0\qquad\text{and}\qquad \overline{\nabla}_{\partial{x^i}}
    dx^j\big|_{(p_0, p_1)} = 0.
    \end{equation*}
\end{corollary}
%
\begin{proof}
The first equality follows from the fact that in our coordinate system,
  \(\overline{\nabla}_{\partial_{x^i}} \partial_{x^j} =
  \overline{\Gamma}_{ij}^k \partial_{x^k}\) and the
  \(\overline{\Gamma}_{ij}^k\) vanish at \((p_0, p_1)\) by Lemma
  \ref{LEM_g_and_Gamma_properties} above.
To prove the second equality, we covariantly differentiate the equality
  \(\delta_j^i=dx^i(\partial_{x^j})\) to find
  \begin{align*}
    0&= (\overline{\nabla}_{\partial_{x^k}} dx^i) (\partial_{x^j}) +
    dx^i(\overline{\nabla}_{\partial_{x^k}} \partial_{x^j}).
    \end{align*}
By our previous results, the second term vanishes when evaluated at
  \((p_0, p_1)\).
Since the above equality holds for each \(k\in \{0, \ldots , n-1\}\), and
    \(\{\partial_{x^k}\}_{k\in \{0, \ldots, n-1\}}\) forms a basis of
  \(T_{(p_0, p_1)} (\mathbb{R}\times M)\), we see that indeed
  \(\overline{\nabla}_{\partial_{x^i}}  dx^j\big|_{(p_0, p_1)}=0\) as
  claimed.  
\end{proof}

\begin{corollary}[$\partial_a$ and $da$ are parallel]
  \label{COR_da_is_parallel}
  \hfill\\
Let \((M, g)\) be a Riemannian manifold, and consider the manifold
  \(\mathbb{R}\times M\) equipped with the Riemannian metric \(\bar{g} =
  da \otimes da + g\) where \(a\) is the coordinate on \(\mathbb{R}\).
Then for the Levi-Civita connection \(\overline{\nabla}\) associated to
  \(\bar{g}\) on \(\mathbb{R}\times M\), we have
  \begin{align*}                                                          
    \overline{\nabla} da =0 \qquad\text{and}\qquad \overline{\nabla}
    \partial_a = 0.
    \end{align*}
\end{corollary}
%
\begin{proof}
Let \((x_0, x_1, \ldots, x_n)\) be coordinates as above and let \(V = v^i
  \partial_{x^i}\) be an arbitrary smooth vector field.
Then
  \begin{align*}                                                          
    \overline{\nabla}_V \partial_a = v^i \overline{\nabla}_{\partial_x^i}
    \partial_{x^0}
    = v^i \overline{\Gamma}_{i0}^k \partial_{x^k},
  \end{align*}
where 
  \begin{align*}                                                          
    \Gamma_{ij}^k = {\textstyle \frac{1}{2}} g^{k\ell}\big(g_{i\ell, j} +
    g_{j\ell, i} - g_{i j, \ell}\big),
    \end{align*}
  and more importantly
  \begin{align*}                                                          
    \Gamma_{i0}^k &= {\textstyle \frac{1}{2}} g^{k\ell}\big(g_{i\ell, 0} +
    g_{0\ell, i} - g_{i 0, \ell}\big),
    \\
    &={\textstyle \frac{1}{2}} g^{k\ell}\big(g_{0\ell, i} -
    g_{i 0, \ell}\big),
    \\
    &={\textstyle \frac{1}{2}} g^{k\ell}\big(g_{0\ell, i} 
    \big),
    \\
    &=0,
    \end{align*}
  where to obtain the second inequality we note that the metric
  \(\bar{g}\) is \(\mathbb{R}\)-invariant and hence
  \(\frac{\partial}{\partial x_0} g_{ij} =0 \); to obtain the third
  equality we have used that \(\bar{g}_{i0} = \delta_{i0}\); the fourth
  equality follows similarly.
This establishes that \(\overline{\nabla}\partial_a = 0\).
As noted above, we also have
  \begin{equation*}
    \nabla_{\partial_{x^k}}dx^i=-\Gamma_{ki}^\ell dx^\ell,
    \end{equation*}
  which then establishes that \(\overline{\nabla}da =0\).
\end{proof}
%

\begin{definition}[second fundamental form $B$]
  \label{DEF_second_fundamental_form} 
  \hfill\\
Let \(u:S\to M\) denote an immersion into a Riemannian manifold \((M,
  g)\).
Then the \emph{second fundamental form} associated to \(u\) and \(g\) is
  denoted \(B_u \in \Gamma(u^*(T^*M\otimes T^*M \otimes TM))\), and is
  defined by
  \begin{equation*}
    B_u(X, Y):= (\nabla_X Y)^\bot
    \end{equation*} 
  where \(X, Y\in \Gamma(u^*(T M))\), \(\nabla\) denotes covariant
  differentiation with respect to the Levi-Civita connection associated to
  \(g\), and \(Z\mapsto Z^\bot\) denotes orthogonal projection to the
  normal bundle of \(u(S)\).
\end{definition}
%

\subsection{Carefully Formulating the Co-Area Formula}\label{SEC_co-area}

Given an oriented \(m\)-dimensional Riemannian manifold \((M, g)\), there
  exists a canonical volume form given by \(*(1)\), where \(*\) is the Hodge
  \(*\)-operator; we shall denote this form \(d\mu_g^m\).
Recall that in local coordinates \((x_1, \ldots, x_m)\) with
  \(\partial_{x_1}, \ldots, \partial_{x_m}\) a positive basis, we have
  \begin{equation}\label{EQ_hausdorff_measure}
    d\mu_g^m = \sqrt{{\rm det} (g_{ij})}\, dx_1\wedge \cdots\wedge dx_m;
    \end{equation}
  here \(g=g_{ij}dx_i\otimes dx_j\). 
Although \(d\mu_g^m\) is a volume form, we may regard it as a measure via
  integration: \(d\mu_g^m(\mathcal{O}):=\int_{\mathcal{O}} d\mu_g^m\).
We will abuse notation by letting \(d\mu_g^m\) denote both the measure and
  volume form referring to each as needed.

Before continuing on to establish some useful results, we make two last
  observations.
First, if \(S\) is oriented and \(u:S\to M\) is an immersion, then \((S,
  u^*g)\) is an oriented Riemannian manifold.
Second, if \((S, g)\) is an oriented Riemannian manifold of dimension
  \(k\), and \(\omega\) is a differentiable \(k\)-form on \(S\), then we
  have the following.
\begin{equation*}
  \int_S \omega = \int_S \omega(e_1, \ldots, e_k) d\mu_{u^*g}^k
  \end{equation*} 
  where \((e_1, \ldots, e_k)\) forms a positive \(u^*g\)-orthonormal frame.  
This is straightforward to verify. 

\begin{proposition}[The co-area formula]\label{PROP_co-area}\hfill\\
Let \((S, g)\) be a \(\mathcal{C}^1\) oriented Riemannian manifold of
  dimension two; we allow that \(S\) need not be complete\footnote{That is,
  there may exist Cauchy sequences, with respect to \(g\), which do not
  converge in \(S\).}.
Suppose that \(\beta:S\to [a, b]\subset\mathbb{R}\) is a \(\mathcal{C}^1\)
  function without critical points.
Let \(f:S\to [0, \infty)\) be a measurable function  with respect to
  \(d\mu_g^2\).  
Then
  \begin{equation}\label{EQ_coarea0}
  \int_S f\|\nabla \beta\|_g \, d\mu_g^2 = \int_a^b \Big(
  \int_{\beta^{-1}(t)} f \, d\mu_g^1\Big) dt
  \end{equation}
  where \(\nabla\beta\) is the gradient of \(\beta\) computed with respect
  to the metric \(g\).
\end{proposition}
%
\begin{proof}
We begin by defining two vector fields \(\vec{x}=(\nabla \beta)/\|\nabla
  \beta\|^2\) and \(\vec{y}\) which is uniquely defined by the three
  conditions: \(\|\vec{y}\|_g=1\), \(\langle \vec{x}, \vec{y}\rangle_g =
  0\), and \(\{\vec{x}, \vec{y}\}\) is a positive basis.
The flow of \(\vec{y}\) preserves \(\beta\) since it is orthogonal to
  \(\nabla \beta\), and \(d\beta (\vec{x})=1\).
Let \(\varphi_{\vec{y}}^t\) and \(\varphi_{\vec{x}}^s\) respectively
  denote the time \(t\) flow of \(\vec{y}\) and the time \(s\) flow of
  \(\vec{x}\).
For each \(\zeta\in S\), we then define the map
  \(\Phi_\zeta:\tilde{\mathcal{O}}_\zeta:= (-\epsilon_\zeta,
  \epsilon_\zeta)\times (-\epsilon_\zeta, \epsilon_\zeta)\to S\) to be
  \(\Phi_\zeta(s, t) = \varphi_{\vec{x}}^s(\varphi_{\vec{y}}^t(\zeta))\),
  with \(\epsilon_\zeta>0\) chosen sufficiently small so that \(\Phi_\zeta\)
  is a diffeomorphism with its image, which we denote \(\mathcal{O}_\zeta\).
Observe that  \(\beta\circ \Phi_\zeta^{-1}(x_1, x_2) = x_1\) by
  construction.

We now define functions \(\tilde{f}:= f\circ \Phi_\zeta\) and
  \(\tilde{\beta}:=\beta\circ \Phi_\zeta\), and the metric \(\tilde{g}:=
  \Phi_\zeta^* g\).
In these local coordinates, we write \(\tilde{g}=\tilde{g}_{ij}dx_i\otimes
  dx_j\), and then
  \begin{equation*}
    d\mu_{\tilde{g}}^2 = \sqrt{{\rm det}\, (\tilde{g}_{ij})}\, dx_1\wedge
    dx_2 = \sqrt{{\rm det}\, (\tilde{g}_{ij})}\,d\tilde{\beta} \wedge dx_2
    .
    \end{equation*}
Similarly, the volume form on the level sets of \(\tilde{\beta}\) are then
  given as
  \begin{equation*}
    d\mu_{\tilde{g}}^1 = \sqrt{\tilde{g}_{22}}\, dx_2
    \end{equation*}
Making use of the fact that in our special case
  \(\tilde{\mathcal{O}}_\zeta\) is a product space, we may employ Tonelli's
  theorem to obtain
  \begin{align*}
    \int_{\tilde{\mathcal{O}_\zeta}} \tilde{f} \|\nabla
    \tilde{\beta}\|_{\tilde{g}} d\mu_{\tilde{g}}^2 &=
    \int_{\tilde{\mathcal{O}_\zeta}} \tilde{f} \|\nabla
    \tilde{\beta}\|_{\tilde{g}} \sqrt{{\rm det}\, (\tilde{g}_{ij})}\,
    dx_1\wedge dx_2
    \\
    &= \int_{(-a_\zeta, b_\zeta)\times (-\epsilon_\zeta, \delta_\zeta)}
    \tilde{f} \|\nabla \tilde{\beta}\|_{\tilde{g}} \sqrt{{\rm det}\,
    (\tilde{g}_{ij})}\, dx_1 dx_2
    \\
    &= \int_{-a_\zeta}^{b_\zeta} \Big(
    \int_{-\epsilon_\zeta}^{\delta_\zeta} \tilde{f} \|\nabla
    \tilde{\beta}\|_{\tilde{g}} \sqrt{{\rm det}\, (\tilde{g}_{ij})}\,
    dx_2\Big) dx_1
    \\
    &= \int_{-a_\zeta}^{b_\zeta} \Big( \int_{\tilde{\beta}^{-1}(t)}
    \frac{\tilde{f} \|\nabla \tilde{\beta}\|_{\tilde{g}} \sqrt{{\rm det}\,
    (\tilde{g}_{ij})}}{\sqrt{\tilde{g}_{22}}}\, d\mu_{\tilde{g}}^1 \Big)
    dt
    \end{align*}

We now claim the following. 
  \begin{equation}\label{EQ_norm_db}
    \|\nabla \tilde{\beta}\|_{\tilde{g}} =\frac{\sqrt{\tilde{g}_{22}}}
    {\sqrt{{\rm det}\, (\tilde{g}_{ij})}}
    \end{equation}
To see this is true we first note that it is sufficient to work pointwise. 
Next, we let  \(v\) be a \(\tilde{g}\)-unit vector orthogonal to the level
  sets of \(\tilde{\beta}\) (i.e. orthogonal to the sets \(\{x_1=const\}\)).
It is elementary to show that \(v\) can be written as 
  \begin{equation*}
    v= \big(\tilde{g}_{22} \partial_{x_1} - \tilde{g}_{12}
    \partial_{x_2}\big)/\big(\sqrt{\tilde{g}_{22}}\sqrt{{\rm det}\,
    (\tilde{g}_{ij})}\big).
    \end{equation*}
We then compute
  \begin{equation*}
    \|\nabla \tilde{\beta}\|_{\tilde{g}}^2 = \big(d\tilde{\beta}(v)\big)^2
    = \big( dx_1(v)\big)^2
    =\frac{\tilde{g}_{22}^2}{\tilde{g}_{22} {\rm det}\, (\tilde{g}_{ij})} 
    =\frac{\tilde{g}_{22}}{ {\rm det}\, (\tilde{g}_{ij})} 
    \end{equation*}
  and equation (\ref{EQ_norm_db}) is established. Consequently we have
  established
  \begin{align}\label{EQ_coarea_basic}
    \int_{\mathcal{O}_\zeta} f \|\nabla \beta\|_g d\mu_g^2 &=
    \int_{\tilde{\mathcal{O}}_\zeta} \tilde{f} \|\nabla
    \tilde{\beta}\|_{\tilde{g}} d\mu_{\tilde{g}}^2
    \\
    &= \int_{-a_\zeta}^{b_\zeta} \Big( \int_{\tilde{\beta}^{-1}(t)}
    \tilde{f} \, d\mu_{\tilde{g}}^1 \Big) dt\notag
    \\
    &=\int_{\beta(\zeta)-a_\zeta}^{\beta(\zeta)+b_\zeta} \Big(
    \int_{\mathcal{O}_\zeta\cap\beta^{-1}(t)} f \, d\mu_g^1 \Big) dt \notag.
    \end{align}

We now prove the more general case by a partition of unity argument.  
First, for each point \(\zeta\in S\), we let \(\mathcal{O}_\zeta\) denote
  the open set containing \(\zeta\) constructed above, and we let
  \(\Phi_\zeta:\tilde{\mathcal{O}}_\zeta \to \mathcal{O}_\zeta\) denote the
  associated diffeomorphism.
These diffeomorphisms show \(S\) is locally compact. 
Since \(S\) is a manifold, it is second countable and Hausdorff; together
  with being locally compact this guarantees \(S\) is paracompact and
  Hausdorff, and hence the open cover \(\{\mathcal{O}_\zeta\}_{\zeta\in S}\)
  admits a subordinate partition of unity \(\{\rho_\alpha:S \to [0,
  1]\}_{\alpha\in I}\).
That is, there is an index set \(I\) and an open cover
  \(\{\mathcal{U}_\alpha\}_{\alpha\in I}\) of \(S\), and there exist
  functions \(\{\rho_\alpha\}_{\alpha\in I}\) with the property that
  \begin{itemize}
    \item 
    \({\rm supp}( \rho_\alpha)\subset \mathcal{U}_\alpha \subset
    \mathcal{O}_{\zeta_\alpha}\),
    \item 
    for each \(\zeta\in S\) we have \(\#\{\alpha \in I: \zeta\in
    \mathcal{U}_\alpha \}<\infty\),
    \item 
    \(\sum_{\alpha\in I} \rho_\alpha = 1\).
    \end{itemize}

Now, making use of the partition of unity, equation
  (\ref{EQ_coarea_basic}), and the monotone convergence theorem to pass
  limits through integrals, we find the following.
\begin{align*}
  \int_S f\|\nabla \beta\|_g \, d\mu_g^2 &= \int_S \sum_{\alpha\in I}
  \rho_\alpha f \|\nabla \beta\|_g d\mu_g^2 =\sum_{\alpha\in I} \int_S
  \rho_\alpha f \|\nabla \beta\|_g d\mu_g^2
  \\
  &=\sum_{\alpha\in I} \int_{\mathcal{O}_{\zeta_\alpha}} \rho_\alpha f
  \|\nabla \beta\|_g d\mu_g^2 =\sum_{\alpha\in I}
  \int_{\beta(\zeta_\alpha)-a_{\zeta_\alpha}}^{\beta(\zeta_\alpha)+b_{\zeta_\alpha}}
  \Big(\int_{\mathcal{O}_{\zeta_\alpha}\cap \beta^{-1}(t)} \rho_\alpha f
  \, d\mu_g^1\Big) dt
  \\
  &=\sum_{\alpha\in I} \int_a^b \Big(\int_{\beta^{-1}(t)} \rho_\alpha f \,
  d\mu_g^1\Big) dt = \int_a^b \Big(\int_{\beta^{-1}(t)} \sum_{\alpha\in
  I}\rho_\alpha f \, d\mu_g^1\Big) dt
  \\
  &= \int_a^b \Big(\int_{\beta^{-1}(t)} f \, d\mu_g^1\Big) dt .
  \end{align*}
This is the desired result, and this completes the proof of Proposition
  \ref{PROP_co-area}.
\end{proof}

\subsection{Typically Tame Perturbations}\label{SEC_tame_perturbations}
The purpose of this section is to prove Lemma \ref{LEM_f_perturbation}
  below, which is the lemma which essentially proves the existence of tame
  perturbations via Lemma \ref{LEM_tame_perturbations}.
In order to prove the main result here, it will be useful to have the
  following definition established.

\begin{definition}[generally-Riemannian metric]
 \label{DEF_generally_Riemannian_metric}
  \hfill\\
On a manifold \(S\), which is smooth and may have boundary and corners, we
  call the pair \((\mathcal{Z}, \gamma)\) a \emph{generally-Riemannian
  metric} provided \(\mathcal{Z}\subset S\setminus \partial S\) is finite,
  \(\gamma\) is a Riemannian metric on \(S\setminus \mathcal{Z}\), and
  \(\gamma\) vanishes on \(\mathcal{Z}\).
\end{definition}
%

\begin{remark}[generally-Riemannian metrics yield distances]
  \label{REM_dist}
  \hfill\\
Although a generally-Riemannian metric is not, strictly speaking, a
  Riemannian metric, it nevertheless induces a distance function defined by
  the following.
\begin{equation*}                                                         
  {\rm dist}_\gamma(\zeta_0, \zeta_1):= \inf \Big\{ \int_0^1
  \gamma\big(\alpha'(t), \alpha'(t)\big)^{\frac{1}{2}} \, dt :\alpha\in
  \mathcal{C}^1([0, 1], S)\text{ and } \alpha(i)=\zeta_i\Big\}.
  \end{equation*}
\end{remark}
%

\begin{lemma}[sufficienty small perturbations]
  \label{LEM_f_perturbation}
  \hfill\\
Let \(\epsilon', \delta>0\) be small positive constants.  
Let \(S\) be a compact real two-dimensional manifold possibly with
  boundary and possibly with corners.
Suppose further that \(S\) is equipped with the following data.
\begin{enumerate}
  \item 
  \(h:S\to \mathbb{R}\) a smooth function satisfying \(\{\zeta\in \partial
  S: dh(\zeta) =0\}=\emptyset\),
  \item 
  \((\mathcal{Z}, \gamma)\) a generally-Riemannian metric.
  \end{enumerate}
Suppose \(\delta\) satisfies  
  \begin{equation*}                                                       
    \delta < {\textstyle \frac{1}{10}} \min \Big( {\rm dist}_\gamma
    (\mathcal{Z}, \partial S), \min_{\substack{\zeta_0, \; \zeta_1\in
    \mathcal{Z}\\ \zeta_0\neq \zeta_1}} {\rm dist}_{\gamma}(\zeta_0,
    \zeta_1), \; {\rm dist}_{\gamma}\big(\{\zeta\in S: dh(\zeta)=0\},
    \partial S\big) \Big).
    \end{equation*}
Define the sets 
  \begin{align*}                                                          
    \mathcal{U}_{\frac{1}{2}\delta} &= \{\zeta\in S: {\rm
    dist}_{\gamma}(\zeta, \mathcal{Z})<{\textstyle \frac{1}{2}}\delta\}
    \\
    \mathcal{U}_{\delta} &= \{\zeta\in S: {\rm dist}_{\gamma}(\zeta,
    \mathcal{Z})<\delta\}
    \\
    \mathcal{V}_{\delta} &= \big\{\zeta\in S: {\rm
    dist}_{\gamma}\big(\zeta, \{z\in S: dh(z)=0\}\big)<\delta\big\}.
    \end{align*}
\emph{Then} there exists a function \(f\in \mathcal{C}^\infty(S)\)
  satisfying the following conditions
  \begin{enumerate}[(f1)]                                                 
    \item \label{EN_f1} 
    \({\rm supp}(f)\subset \mathcal{V}_{\delta}\setminus
    \mathcal{U}_{\frac{1}{2}\delta}\)  with \(\mathcal{V}_{\delta}\cap
    \partial S = \emptyset\),
    \item \label{EN_f2}
    \begin{equation*}
      \sup_{\zeta \in \Omega}|f(\zeta)|+\sup_{\zeta\in \Omega}
      \|df(\zeta)\|_\gamma + \sup_{\zeta\in \Omega}\|\nabla
      df(\zeta)\|_\gamma< \epsilon'
      \end{equation*}
      where \(\Omega={\rm supp}(f)\), and \(\nabla\) denotes covariant
      differentiation associated to the Levi-Civita connection
      corresponding to \(\gamma\),
    \item\label{EN_f3} 
    on \(S\setminus \mathcal{U}_{\delta}\) the function  \(h+f\) is Morse;
    that is, the Hessian \(\nabla d(h+f)\) at critical points of \(h+f\)
    is non-degenerate.
    \end{enumerate}
\end{lemma}
%
\begin{proof}
We begin by regarding \(dh\) as a section of the cotangent bundle \(T^*
  S\), so the zeros of \(dh\) are precisely the critical points of \(h\),
  denoted by
  \begin{align*}
    {\rm Crit}_h = \{\zeta\in S: dh(\zeta)=0\}.
    \end{align*}
For each \(z\in {\rm Crit}_h\setminus \mathcal{Z}\), we may regard the
  linear map
  \begin{align*}
    &A_z:T_z S\to T_z^* S
    \\
    &Y\mapsto \nabla_Y dh\big|_z
    \end{align*}
  as the linearization of the principal part\footnote{By ``linearization
  of the principal part'' we mean the following.
Given a vector bundle \(\mathcal{E}\to \mathcal{B}\) with a connection
  \(T\mathcal{E} = V\mathcal{E}\oplus H\mathcal{E}\),  which for our
  purposes will always be the Levi-Civita connection,  together with a
  continuously differentiable section \(\sigma: \mathcal{B}\to
  \mathcal{E}\), then the linearization of the principal part of
  \(\sigma\) is defined to be \({\rm pr}_V\circ T\sigma \) where \({\rm
  pr}_V: T\mathcal{E}\to V\mathcal{E}\) is the projection to vertical
  sub-bundle associated to the connection.} of the section \(dh\in
  \Gamma(T^*S \to S)\) at the point \(z\).
Because \(T_z S\) and \(T_z^* S \) have the same dimension, we see that
  \(z\) is a non-degenerate critical point of \(h\) if and only if \(A_z\)
  has trivial kernel.
For each \(z\in{\rm Crit}_h\setminus \mathcal{Z}\) we can define
  \(\gamma\)-geodesic coordinates \((x_z^1, x_z^2)\)  centered at \(z\),
  with the additional property that if \(A_z\) has nontrivial kernel, then
  \(A_z(\partial_{x_z^1})= 0\).
Next, for each \(z\in {\rm Crit}_h\setminus \mathcal{Z}\) we define the
  number \(n_z\in \{0, 1, 2\}\) by \(n_z:={\rm dim}\, {\rm ker}(A_z)\),
  and we define the neighborhood
  \begin{equation*}
    \mathcal{W}_z:=\big\{\zeta\in S: {\rm dist}_\gamma(\zeta, z)<
    {\textstyle \frac{1}{2}}{\rm min}\big(\delta_1, {\rm
    inj}^\gamma(z)\big)\big\},
    \end{equation*}
  where \({\rm inj}^\gamma(z)\) is the injectivity radius associated to
  \(\gamma\) at \(z\in S\).
Letting \({\rm pr}_1:S\times \mathbb{R}^{n_z}\to S\) denote the canonical
  projection to the first factor, we define the section \(\sigma_z\in
  \Gamma({\rm pr}_1^* T^* S\to S\times \mathbb{R}^{n_z})\) by
  \begin{align*}
    \sigma_z(\zeta, \mathbf{s}) =
    \begin{cases}
    dh(\zeta) & \text{if }\zeta\notin \mathcal{W}_z \text{ or } n_z=0
    \\
    dh(\zeta) + \sum_{i=1}^{n_z} s^i d( x_z^i \beta_{z})(\zeta)&
    \text{otherwise}
    \end{cases}
    \end{align*}
  where \((x_z^1, x_z^2)\) are the coordinates established above, and
  \(\beta_z\) is a smooth cut-off function with \(\beta_z(\zeta)=1\) in a
  neighborhood of \(\zeta=z\) and
  \begin{equation}\label{EQ_f_support}                                     
    {\rm supp}(\beta_z)\subset \big\{\zeta\in S: {\rm dist}_\gamma(\zeta ,
    z) < {\textstyle \frac{1}{4}}\min(\delta_1, {\rm inj}^\gamma(z))\big\}.
    \end{equation}
By construction the section \(\sigma_z\) is  transverse to the zero
  section at the point \((z, 0)\) with \(z\in {\rm Crit}_h\setminus
  \mathcal{Z}\), and hence the linearization of the principal part of
  \(\sigma_z\) is surjective at \((z, 0)\).
Consequently, there exists an open set \(\mathcal{O}_z\subset S\)
  containing the  point \(z\) with the additional property that for each
  \((w, 0)\in \mathcal{O}_z\times \mathbb{R}^{n_z}\)  the linearization of
  the principal part of \(\sigma_z\) at \((w, 0)\) is surjective.
We now repeat this construction for each \(z\in {\rm Crit}_h\setminus
  \mathcal{Z}\).
Note that \({\rm Crit}_h\setminus \mathcal{U}_{\delta}\) is compact, and
  the collection \(\{\mathcal{O}_z\}_{z\in {\rm Crit}_h\setminus
  \mathcal{Z}}\) is an open cover of \({\rm Crit}_h\setminus
  \mathcal{U}_{\delta}\), and hence may be reduced to a finite sub-cover of
  \({\rm Crit}_h\setminus \mathcal{U}_{\delta}\), which we denote by
  \(\{\mathcal{O}_{z_1}, \ldots, \mathcal{O}_{z_m}\}\).
Let \(\{w_1, \ldots, w_{\ell}\}\subset \{z_1, \ldots, z_m\}\) be the
  subset for which \(A_{w_k}\) fails to be surjective, and define the
  section \(\tilde{\sigma}\) by
  \begin{align*}                                                          
    &\tilde{\sigma}\in\Gamma\big({\rm pr}^*T^* S \to
    S\times\mathbb{R}^{n_{w_1}}\times \cdots \times
    \mathbb{R}^{n_{w_{\ell}}}\big)
    \\
    &\tilde{\sigma}(\zeta, \mathbf{s}) = dh(\zeta) + \sum_{j=1}^{\ell}
    \sum_{i=1}^{n_{w_j}} s^{j, i} d( x_{w_j}^i \beta_{w_j})(\zeta)
    \end{align*}
For ease of notation, let us re-index and rewrite the above as
  \begin{align*}                                                          
    &\tilde{\sigma}\in \Gamma\big({\rm pr}^*T^*S\to S\times
    \mathbb{R}^{\tilde{m}}\big)
    \\
    &\tilde{\sigma}(\zeta, \mathbf{s}) = dh(\zeta) + \sum_{i=1}^{\tilde{m}}
    s^i d\tilde{f}_i (\zeta).
    \end{align*}
By construction, the linearization of the principal part of the section
  \(\tilde{\sigma}\) is surjective over the set \(\mathcal{O}\times \{0\}\)
  where \(\mathcal{O}=\cup_{i=1}^m \mathcal{O}_{z_i}\).
Consequently, there exists an open set
  \(\widetilde{\mathcal{O}}=\mathcal{O}\times
  \{|\mathbf{s}|<\hat{\epsilon}\}\subset S\times \mathbb{R}^{\tilde{m}}\)
  which has the following two important properties.
\begin{enumerate}[(T1)]
  \item \label{EN_T1} 
  the linearization of the principal part of \(\tilde{\sigma}\) is
  surjective at every point in \(\widetilde{\mathcal{O}}\)
  \item \label{EN_T2} 
  if \((\zeta, \mathbf{s})\in S\times \{|\mathbf{s}|< \hat{\epsilon}\}\)
  solves \(\tilde{\sigma}(\zeta, \mathbf{s})=0\), then  \(\zeta\in
  \mathcal{O}\cup \mathcal{U}_{\delta}\)
  \end{enumerate}
The first property follows essentially because surjectivity is an open
  condition; the second property follows because \(dh\) is non-vanishing on
  the compact set \(S\setminus (\mathcal{O}\cup \mathcal{U}_{\delta})\), so
  that \(\|dh\|\) attains a non-zero minimum on this set, and hence for all
  \(\mathbf{s}\) sufficiently close to \(0\) we have
  \begin{equation*}
    \sup_{\zeta\in S} \|\sum_{i=1}^{\tilde{m}} s^i d\tilde{f}^i
    (\zeta)\|_{\gamma} < \inf_{\zeta\in S\setminus (\mathcal{O}\cup
    \mathcal{U}_{\delta})} \|dh(\zeta)\|.
    \end{equation*}
As a consequence of property (T\ref{EN_T1}) and the implicit function
  theorem, the set \(\mathcal{B}:=\widetilde{\mathcal{O}}\cap
  \tilde{\sigma}^{-1}(0)\subset S\times \mathbb{R}^{\tilde{m}}\) is a smooth
  manifold of dimension \(\tilde{m}\).  
Fix \((\zeta, \mathbf{s})\in \mathcal{B}\), and note that the associated
  tangent fiber of \(\mathcal{B}\) is given by
  \begin{align*}
    T_{(\zeta, \mathbf{s})} \mathcal{B} &= \{(v, \hat{s}^1, \ldots,
    \hat{s}^{\tilde{m}})\in T_\zeta S \times
    \mathbb{R}^{\tilde{m}}:0=\nabla_v dh + \sum_{i=1}^{\tilde{m}}s^i
    \nabla_vd \tilde{f}^i +\hat{s}^i d\tilde{f}^i  \}.
    \end{align*}
We denote the following vector spaces \(X=T_\zeta S\),  \(Y=T_\zeta^* S\),
  and \(Z=\mathbb{R}^{\tilde{m}}\), and we define the following linear
  maps.
\begin{equation*}                                                         
  D:X\to Y\qquad\text{by}\qquad D(v) =  \nabla_v dh
  +\sum_{i=1}^{\tilde{m}}  s^i \nabla_v d\tilde{f}^i
  \end{equation*}
\begin{equation*}                                                         
  L:Z \to Y\qquad\text{by}\qquad L(\hat{s}^1, \ldots,
  \hat{s}^{\tilde{m}})=\sum_{i=1}^{\tilde{m}} \hat{s}^i\, d\tilde{f}^i
  \end{equation*}
Consequently, we may express 
  \begin{align*}                                                          
    T_{(\zeta, \mathbf{s})} \mathcal{B} ={\rm ker}\, (D\oplus L).
    \end{align*}
Finally, we define the projection
  \begin{equation*}                                                       
    \Pi : {\rm ker}\, (D\oplus L) \to Z\qquad\text{by}\qquad \Pi(v,
    \hat{s}^1, \ldots, \hat{s}^i)=  (\hat{s}^1, \ldots, \hat{s}^i).
    \end{equation*}
At this point we note that \(D\oplus L\) is the linearization of the
  principal part of \(\tilde{\sigma}\) at the point \((\zeta, \mathbf{s})\),
  which by construction is surjective, and hence \(D\oplus L\) is onto.
By Lemma \ref{LEM_A36} below, it follows that \(D\) is surjective if and
  only if \(\Pi\) is surjective.

At this point, our aim is to show that there exist many choices of
  \(\mathbf{s}\in \mathbb{R}^{\tilde{m}}\) with the property that whenever
  \(\tilde{\sigma}(\zeta, \mathbf{s})=0\), we also have that \(\Pi\) is
  surjective.
To that end, consider \({\rm pr}_2: \mathcal{B}\subset S\times
  \mathbb{R}^{\tilde{m}}\to \mathbb{R}^{\tilde{m}}\) the canonical
  projection to the second factor, and observe that this map is smooth, and
  \(T{\rm pr}_2 = \Pi\).
By Sard's theorem, the regular values of \({\rm pr}_2\) have full measure,
  and hence there exists a sequence \(\{\mathbf{s}_k\}_{k\in \mathbb{N}}\)
  in \(\mathbb{R}^{\tilde{m}}\) satisfying \(\mathbf{s}_k\to 0\) and each
  \(\mathbf{s}_k\) is a regular value of \({\rm pr}_2\).
For each such \(\mathbf{s}_k\) and every \((\zeta, \mathbf{s}_k)\in
  \mathcal{B}\) we then have that \(D\) is surjective.
That is to say, for each fixed such  \(\mathbf{s}_k=(s_k^1, \ldots,
  s_k^{\tilde{m}})\), and each \(\zeta\in S\) which is a zero of the section
  \(d(h+f_k):=d(h+\sum_i s_k^i\tilde{f}^i)\in \Gamma(T^*S\to S)\), the
  linearization of the principal part of this section  is surjective.
In other words, for each such \(\mathbf{s}_k\), the section \(d(h+\sum_i
  s_k^i\tilde{f}^i)\) is transverse to the zero-section for all
  \(\zeta\in\mathcal{O}\); by property (T\ref{EN_T2}) the only zeros of
  \(d(h+f_k)\) lie in \(\mathcal{O}\cup\mathcal{U}_{\delta}\), and hence all
  critical point of \(h+f_k\) in \(S\setminus \mathcal{U}_{\delta}\) are
  non-degenerate.
This establishes property (f\ref{EN_f3}) for any \(f=f_k\).
Next we note that \(f_k\to 0\) in \(\mathcal{C}^\infty\) so that property
  (f\ref{EN_f2}) holds for any sufficiently large \(k\).
Finally property (f\ref{EN_f1}) follows from the definition of the \(f_k\)
  -- specifically the support of the cut-off functions \(\beta_z\)
  established in equation (\ref{EQ_f_support}).
This completes the proof of Lemma \ref{LEM_f_perturbation}
\end{proof}
%

\begin{lemma}[Lemma A.3.6, \cite{MS}]
  \label{LEM_A36}
  \hfill\\
Assume \(D:X\to Y\) is a Fredholm operator and \(L:Z\to Y\) is a bounded
  linear operator such that \(D\oplus L:X\oplus Z\to Y\) is onto.
Then \(D\oplus L \) has a right inverse. 
Moreover, the projection \(\Pi: {\rm ker}(D\oplus L)\to Z\) is a Fredholm
  operator with \({\rm ker}\, \Pi\cong {\rm ker}\, D\) and \({\rm coker}\,
  \Pi\cong {\rm coker}\, D\), and hence \({\rm index}\, \Pi = {\rm index}\,
  D\).
\end{lemma}
%

\bibliography{bibliography}{}

\begin{bibdiv}
\begin{biblist}

\bib{Abbas-Hofer}{book}{
      author={Abbas, Casim},
      author={Hofer, Helmut},
       title={Holomorphic curves and global questions in contact geometry},
      series={Birkh\"auser Advanced Texts},
   publisher={Birkh\"auser, Basel},
        date={2019},
        note={to appear},
}

\bib{AI}{article}{
      author={Asaoka, Masayuki},
      author={Irie, Kei},
       title={A {$C^\infty$} closing lemma for {H}amiltonian diffeomorphisms of
  closed surfaces},
        date={2016},
        ISSN={1016-443X},
     journal={Geom. Funct. Anal.},
      volume={26},
      number={5},
       pages={1245\ndash 1254},
         url={https://doi.org/10.1007/s00039-016-0386-3},
      review={\MR{3568031}},
}

\bib{B}{incollection}{
      author={Bennequin, Daniel},
       title={Entrelacements et \'equations de {P}faff},
        date={1983},
   booktitle={Third {S}chnepfenried geometry conference, {V}ol. 1
  ({S}chnepfenried, 1982)},
      series={Ast\'erisque},
      volume={107},
   publisher={Soc. Math. France, Paris},
       pages={87\ndash 161},
      review={\MR{753131}},
}

\bib{BEHWZ}{article}{
      author={Bourgeois, F.},
      author={Eliashberg, Y.},
      author={Hofer, H.},
      author={Wysocki, K.},
      author={Zehnder, E.},
       title={Compactness results in symplectic field theory},
        date={2003},
        ISSN={1465-3060},
     journal={Geom. Topol.},
      volume={7},
       pages={799\ndash 888},
         url={http://dx.doi.org/10.2140/gt.2003.7.799},
      review={\MR{2026549 (2004m:53152)}},
}

\bib{AMS}{book}{
      editor={Browder, Felix~E.},
       title={Mathematical developments arising from {H}ilbert problems},
      series={Proceedings of Symposia in Pure Mathematics, Vol. XXVIII},
   publisher={American Mathematical Society, Providence, R. I.},
        date={1976},
      review={\MR{0419125 (54 \#7158)}},
}

\bib{CieM2005}{article}{
      author={Cieliebak, K.},
      author={Mohnke, K.},
       title={Compactness for punctured holomorphic curves},
        date={2005},
        ISSN={1527-5256},
     journal={J. Symplectic Geom.},
      volume={3},
      number={4},
       pages={589\ndash 654},
         url={http://projecteuclid.org/euclid.jsg/1154467631},
        note={Conference on Symplectic Topology},
      review={\MR{2235856}},
}

\bib{CHR}{article}{
      author={Cristofaro-Gardiner, Daniel},
      author={Hutchings, Michael},
      author={Ramos, Vinicius Gripp~Barros},
       title={The asymptotics of {ECH} capacities},
        date={2015},
        ISSN={0020-9910},
     journal={Invent. Math.},
      volume={199},
      number={1},
       pages={187\ndash 214},
         url={https://doi.org/10.1007/s00222-014-0510-7},
      review={\MR{3294959}},
}

\bib{E1}{article}{
      author={Eliashberg, Y.},
       title={Classification of overtwisted contact structures on
  {$3$}-manifolds},
        date={1989},
        ISSN={0020-9910},
     journal={Invent. Math.},
      volume={98},
      number={3},
       pages={623\ndash 637},
         url={http://dx.doi.org/10.1007/BF01393840},
      review={\MR{1022310}},
}

\bib{EGH}{article}{
      author={Eliashberg, Y.},
      author={Givental, A.},
      author={Hofer, H.},
       title={Introduction to symplectic field theory},
        date={2000},
        ISSN={1016-443X},
     journal={Geom. Funct. Anal.},
      number={Special Volume, Part II},
       pages={560\ndash 673},
         url={http://dx.doi.org/10.1007/978-3-0346-0425-3_4},
        note={GAFA 2000 (Tel Aviv, 1999)},
      review={\MR{1826267 (2002e:53136)}},
}

\bib{E2}{article}{
      author={Eliashberg, Yakov},
       title={Contact {$3$}-manifolds twenty years since {J}. {M}artinet's
  work},
        date={1992},
        ISSN={0373-0956},
     journal={Ann. Inst. Fourier (Grenoble)},
      volume={42},
      number={1-2},
       pages={165\ndash 192},
         url={http://www.numdam.org/item?id=AIF_1992__42_1-2_165_0},
      review={\MR{1162559}},
}

\bib{FK}{article}{
      author={Fayad, Bassam},
      author={Katok, Anatole},
       title={Analytic uniquely ergodic volume preserving maps on odd spheres},
        date={2014},
        ISSN={0010-2571},
     journal={Comment. Math. Helv.},
      volume={89},
      number={4},
       pages={963\ndash 977},
         url={http://dx.doi.org/10.4171/CMH/341},
      review={\MR{3284302}},
}

\bib{Fish2}{article}{
      author={Fish, Joel~W.},
       title={Target-local {G}romov compactness},
        date={2011},
        ISSN={1465-3060},
     journal={Geom. Topol.},
      volume={15},
      number={2},
       pages={765\ndash 826},
         url={http://dx.doi.org/10.2140/gt.2011.15.765},
      review={\MR{2800366 (2012f:53192)}},
}

\bib{FH1}{article}{
      author={{Fish}, Joel~W.},
      author={{Hofer}, Helmut},
       title={Exhaustive gromov compactness for pseudoholomorphic curves},
        date={2018Nov},
     journal={ArXiv e-prints},
       pages={arXiv:1811.09321},
      eprint={1811.09321},
}

\bib{Floer}{incollection}{
      author={Floer, A.},
       title={Holomorphic curves and a {M}orse theory for fixed points of exact
  symplectomorphisms},
        date={1987},
   booktitle={Aspects dynamiques et topologiques des groupes infinis de
  transformation de la m\'ecanique ({L}yon, 1986)},
      series={Travaux en Cours},
      volume={25},
   publisher={Hermann, Paris},
       pages={49\ndash 60},
      review={\MR{906896 (88m:58053)}},
}

\bib{Ginzburg:1995}{article}{
      author={Ginzburg, Viktor~L.},
       title={An embedding {$S^{2n-1}\to {\bf R}^{2n}$}, {$2n-1\geq 7$}, whose
  {H}amiltonian flow has no periodic trajectories},
        date={1995},
        ISSN={1073-7928},
     journal={Internat. Math. Res. Notices},
      number={2},
       pages={83\ndash 97},
         url={https://doi.org/10.1155/S1073792895000079},
      review={\MR{1317645}},
}

\bib{Ginzburg1997}{article}{
      author={Ginzburg, Viktor~L.},
       title={A smooth counterexample to the {H}amiltonian {S}eifert conjecture
  in {$\mathbb{R}^6$}},
        date={1997},
        ISSN={1073-7928},
     journal={Internat. Math. Res. Notices},
      number={13},
       pages={641\ndash 650},
         url={https://doi.org/10.1155/S1073792897000421},
      review={\MR{1459629}},
}

\bib{GG}{article}{
      author={Ginzburg, Viktor~L.},
      author={G{\"u}rel, Ba{\c{s}}ak~Z.},
       title={A {$C\sp 2$}-smooth counterexample to the {H}amiltonian {S}eifert
  conjecture in {$\mathbb R\sp 4$}},
        date={2003},
        ISSN={0003-486X},
     journal={Ann. of Math. (2)},
      volume={158},
      number={3},
       pages={953\ndash 976},
         url={http://dx.doi.org/10.4007/annals.2003.158.953},
      review={\MR{2031857 (2005e:37133)}},
}

\bib{GN}{article}{
      author={Ginzburg, Viktor~L.},
      author={Niche, C\'{e}sar~J.},
       title={A remark on unique ergodicity and the contact type condition},
        date={2015},
        ISSN={0003-889X},
     journal={Arch. Math. (Basel)},
      volume={105},
      number={6},
       pages={585\ndash 592},
         url={https://doi.org/10.1007/s00013-015-0832-8},
      review={\MR{3422862}},
}

\bib{Gr}{article}{
      author={Gromov, M.},
       title={Pseudoholomorphic curves in symplectic manifolds},
        date={1985},
        ISSN={0020-9910},
     journal={Invent. Math.},
      volume={82},
      number={2},
       pages={307\ndash 347},
         url={http://dx.doi.org/10.1007/BF01388806},
      review={\MR{809718 (87j:53053)}},
}

\bib{Harr}{article}{
      author={Harrison, J.},
       title={{$C\sp 2$} counterexamples to the {S}eifert conjecture},
        date={1988},
        ISSN={0040-9383},
     journal={Topology},
      volume={27},
      number={3},
       pages={249\ndash 278},
         url={http://dx.doi.org/10.1016/0040-9383(88)90009-2},
      review={\MR{963630 (89m:58167)}},
}

\bib{Herman}{article}{
      author={Herman, Michael},
       title={Some open problems in dynamical systems},
        date={1998},
     journal={Proceedings of the {I}nternational {C}ongress of
  {M}athematicians},
        note={Sections 10--19, Held in Berlin, August 18--27, 1998, Doc. Math.
  {{\bf{1}}998}, Extra Vol. II},
}

\bib{Herman:1999}{incollection}{
      author={Herman, Michel~R.},
       title={Examples of compact hypersurfaces in {${\bf R}^{2p},\ 2p\ge6$},
  with no periodic orbits},
        date={1999},
   booktitle={Hamiltonian systems with three or more degrees of freedom
  ({S}'{A}gar\'{o}, 1995)},
      series={NATO Adv. Sci. Inst. Ser. C Math. Phys. Sci.},
      volume={533},
   publisher={Kluwer Acad. Publ., Dordrecht},
       pages={126},
      review={\MR{1720888}},
}

\bib{H93}{article}{
      author={Hofer, Helmut},
       title={Pseudoholomorphic curves in symplectizations with applications to
  the {W}einstein conjecture in dimension three},
        date={1993},
        ISSN={0020-9910},
     journal={Invent. Math.},
      volume={114},
      number={3},
       pages={515\ndash 563},
         url={http://dx.doi.org/10.1007/BF01232679},
      review={\MR{1244912 (94j:58064)}},
}

\bib{HLS}{article}{
      author={Hofer, Helmut},
      author={Lizan, V\'eronique},
      author={Sikorav, Jean-Claude},
       title={On genericity for holomorphic curves in four-dimensional
  almost-complex manifolds},
        date={1997},
        ISSN={1050-6926},
     journal={J. Geom. Anal.},
      volume={7},
      number={1},
       pages={149\ndash 159},
         url={https://doi.org/10.1007/BF02921708},
      review={\MR{1630789}},
}

\bib{HutchLec}{incollection}{
      author={Hutchings, Michael},
       title={Lecture notes on embedded contact homology},
        date={2014},
   booktitle={Contact and symplectic topology},
      series={Bolyai Soc. Math. Stud.},
      volume={26},
   publisher={J\'{a}nos Bolyai Math. Soc., Budapest},
       pages={389\ndash 484},
         url={https://doi.org/10.1007/978-3-319-02036-5_9},
      review={\MR{3220947}},
}

\bib{Kup}{article}{
      author={Kuperberg, Greg},
       title={A volume-preserving counterexample to the {S}eifert conjecture},
        date={1996},
        ISSN={0010-2571},
     journal={Comment. Math. Helv.},
      volume={71},
      number={1},
       pages={70\ndash 97},
         url={http://dx.doi.org/10.1007/BF02566410},
      review={\MR{1371679 (96m:58199)}},
}

\bib{KK1}{article}{
      author={Kuperberg, Greg},
      author={Kuperberg, Krystyna},
       title={Generalized counterexamples to the {S}eifert conjecture},
        date={1996},
        ISSN={0003-486X},
     journal={Ann. of Math. (2)},
      volume={143},
      number={3},
       pages={547\ndash 576},
         url={http://dx.doi.org/10.2307/2118536},
      review={\MR{1394969 (97k:57031a)}},
}

\bib{KK2}{article}{
      author={Kuperberg, Krystyna},
       title={A smooth counterexample to the {S}eifert conjecture},
        date={1994},
        ISSN={0003-486X},
     journal={Ann. of Math. (2)},
      volume={140},
      number={3},
       pages={723\ndash 732},
         url={http://dx.doi.org/10.2307/2118623},
      review={\MR{1307902 (95g:57040)}},
}

\bib{MS2}{book}{
      author={McDuff, Dusa},
      author={Salamon, Dietmar},
       title={Introduction to symplectic topology},
     edition={Second},
      series={Oxford Mathematical Monographs},
   publisher={The Clarendon Press, Oxford University Press, New York},
        date={1998},
        ISBN={0-19-850451-9},
      review={\MR{1698616}},
}

\bib{MS}{book}{
      author={McDuff, Dusa},
      author={Salamon, Dietmar},
       title={{$J$}-holomorphic curves and symplectic topology},
     edition={Second},
      series={American Mathematical Society Colloquium Publications},
   publisher={American Mathematical Society, Providence, RI},
        date={2012},
      volume={52},
        ISBN={978-0-8218-8746-2},
      review={\MR{2954391}},
}

\bib{Rab}{article}{
      author={Rabinowitz, Paul~H.},
       title={Periodic solutions of a {H}amiltonian system on a prescribed
  energy surface},
        date={1979},
        ISSN={0022-0396},
     journal={J. Differential Equations},
      volume={33},
      number={3},
       pages={336\ndash 352},
         url={http://dx.doi.org/10.1016/0022-0396(79)90069-X},
      review={\MR{543703}},
}

\bib{Schw}{article}{
      author={Schweitzer, Paul~A.},
       title={Counterexamples to the {S}eifert conjecture and opening closed
  leaves of foliations},
        date={1974},
        ISSN={0003-486X},
     journal={Ann. of Math. (2)},
      volume={100},
       pages={386\ndash 400},
      review={\MR{0356086 (50 \#8557)}},
}

\bib{Siefring}{article}{
      author={Siefring, Richard},
       title={Finite-energy pseudoholomorphic planes with multiple asymptotic
  limits},
        date={2017},
        ISSN={0025-5831},
     journal={Math. Ann.},
      volume={368},
      number={1-2},
       pages={367\ndash 390},
         url={https://doi.org/10.1007/s00208-016-1478-y},
      review={\MR{3651577}},
}

\bib{Smale}{article}{
      author={Smale, S.},
       title={Mathematical problems for the next century},
        date={2000},
        ISSN={1138-8927},
     journal={Gac. R. Soc. Mat. Esp.},
      volume={3},
      number={3},
       pages={413\ndash 434},
        note={Translated from Math. Intelligencer {{\bf{2}}0} (1998), no. 2,
  7--15 [ MR1631413 (99h:01033)] by M. J. Alc{\'o}n},
      review={\MR{1819266}},
}

\bib{Taubes1998}{article}{
      author={Taubes, Clifford~Henry},
       title={The structure of pseudo-holomorphic subvarieties for a degenerate
  almost complex structure and symplectic form on {$S^1\times B^3$}},
        date={1998},
        ISSN={1465-3060},
     journal={Geom. Topol.},
      volume={2},
       pages={221\ndash 332},
         url={https://doi.org/10.2140/gt.1998.2.221},
      review={\MR{1658028}},
}

\bib{Taubes-Erg}{article}{
      author={Taubes, Clifford~Henry},
       title={An observation concerning uniquely ergodic vector fields on
  3-manifolds},
        date={2009},
        ISSN={1935-2565},
     journal={J. G\"okova Geom. Topol. GGT},
      volume={3},
       pages={9\ndash 21},
      review={\MR{2595753 (2011e:57051)}},
}

\bib{Tromba}{book}{
      author={Tromba, Anthony~J.},
       title={Teichm\"uller theory in {R}iemannian geometry},
      series={Lectures in Mathematics ETH Z\"urich},
   publisher={Birkh\"auser Verlag, Basel},
        date={1992},
        note={Lecture notes prepared by Jochen Denzler},
}

\bib{Ve}{book}{
      author={Verhulst, Ferdinand},
       title={Nonlinear differential equations and dynamical systems},
      series={Universitext},
   publisher={Springer-Verlag, Berlin},
        date={1990},
        ISBN={3-540-50628-4},
         url={http://dx.doi.org/10.1007/978-3-642-97149-5},
        note={Translated from the Dutch},
      review={\MR{1036522}},
}

\bib{WendlSFT}{article}{
      author={Wendl, Chris},
       title={Lectures on {S}ymplectic {F}ield {T}heory},
     journal={To Appear: EMS Series of Lectures in Mathematics, 2019},
}

\end{biblist}
\end{bibdiv}
\bibliographystyle{plain}

\end{document}